\documentclass[english]{smfart}
\usepackage{logmotpackage}

\makeindex
\makeindex[name=notation,title=Index of Notation]

\begin{document}

\author{Federico Binda}
\address{Dipartimento di Matematica F.~Enriques,  Universit\`a degli Studi di Milano\\ Via Cesare Saldini 50, 20133 Milano, Italy}
\email{federico.binda@unimi.it}

\author{Doosung Park}
\address{Institut f\"ur Mathematik, Universit\"at Z\"urich\\ Winterthurerstr. 190, 8057 Z\"urich, Switzerland}
\email{doosung.park@math.uzh.ch}

\author{Paul Arne {\O}stv{\ae}r}
\address{Department of Mathematics, University of Oslo, Niels Henrik Abels hus, Moltke Moes vei 35, 0851 Oslo, Norway}
\address{Dipartimento di Matematica F.~Enriques,  Universit\`a degli Studi di Milano\\ Via Cesare Saldini 50, 20133 Milano, Italy}
\email{paularne@math.uio.no} \email{paul.oestvaer@unimi.it}

\title{Triangulated categories of logarithmic motives over a field}

\begin{abstract}
In this work, we develop a theory of motives for logarithmic schemes over fields in the sense of Fontaine, Illusie, and Kato. 
Our construction is based on the notion of finite log correspondences, the dividing Nisnevich topology on log schemes, 
and the basic idea of parameterizing homotopies by $\boxx$;
the projective line with respect to its compactifying logarithmic structure at infinity. 
We show that Hodge cohomology of log schemes is a $\boxx$-invariant theory that is representable in 
the category of logarithmic motives.
The category of logarithmic motives is closely related to Voevodsky's category of motives and $\A^{1}$-invariant theories.
Assuming resolution of singularities, 
we identify the latter with the full subcategory comprised of $\A^{1}$-local objects in the category of logarithmic motives. 
Fundamental properties such as $\boxx$-homotopy invariance, 
Mayer-Vietoris for coverings, 
and the analogs of the Gysin sequence and the Thom space isomorphism witness the robustness of the setup.
\end{abstract}

\maketitle

\setcounter{tocdepth}{2}
\tableofcontents


\toplevelfalse

\section{Introduction}

\subsection{Overview}
Logarithmic structures and log geometry developed by Deligne, Faltings, Fontaine, Illusie, Kato, Tsuji, 
and many others \cite{Ogu} can be thought of as an enlargement of algebraic geometry.
Having its origin in the theory of logarithmic differentials,
log schemes make the notion of schemes with a boundary precise and have implications, inter alia, towards moduli spaces, deformation theory, and $p$-adic Hodge theory.
The primary purpose of this work is to build a theory of motives in the setting of logarithmic algebraic geometry.
\vspace{0.1in}

Depending on the point of view, there are several desiderata for a theory of motives of log schemes.
For starters, one should find a way to linearize log schemes such that the values of standard cohomology theories,
e.g.,
crystalline, de Rham-Witt, and Hodge cohomology,
occur as hom-groups in the same linear tensor category.
One can hope that log motives give rise to new invariants for logarithmic schemes and a better understanding of ``old'' invariants for usual schemes.
For sanity checks, one should have available various realization functors for log motives and their regulator maps.
By equipping schemes with trivial log structures, one should be able to recover motives in the sense of algebraic geometry.
Finally,
an advanced understanding of the basic properties should involve the formalism of six operations.
In this work, we begin this exploration by constructing the triangulated category of effective log motives $\ldmeff$, 
extending Voevodsky's category $\dmeff$, 
where $\Lambda$ is a commutative unital ring.
\vspace{0.1in}

There have been several constructions of triangulated categories of motives in algebraic geometry,
including Levine \cite{LevineMixedMotives} and Voevodsky \cite{MR1764202}, 
which, together with motivic homotopy theory, have led to Voevodsky's spectacular proofs of the Milnor and Bloch-Kato conjectures \cite{voevodsky-milnor}, \cite{voevodsky-BK}.
The fact that the affine line parametrizes homotopies restricts the theory to $\A^{1}$-homotopy invariant phenomena.
Since there are many nontrivial \'etale $\Z/p$-coverings of the affine line, 
where $p$ is the characteristic of the base field $k$,
this restriction limits the reach of the theory.
Indeed, 
the derived category $\mathbf{DM}_{\text{\'et}}(k,\Z/p)$ of \'etale motives with mod-$p$ coefficients is trivial \cite{AyoICM}. 
\vspace{0.1in}

In Voevodsky's construction of motivic categories, one of the first problems is to find an algebraic avatar for the unit interval $[0,1]\subset \mathbb{R}$ in topology. 
In our approach to log motives, we encounter the same problem. 
One option is to proceed with the affine line $\A^{1}$ equipped with the trivial log structure, 
see \cite{Howellthesis}, \cite{ParThesis}.
There is, 
however, 
another less obvious choice: in the log geometric setting, it is natural to compactify the affine line,
i.e.,
replace it by the projective line $\P^{1}$ pointed at infinity, denoted by $\boxx= (\P^1,\{\infty\})$. 
A priori this is not a unit interval in the traditional sense since the multiplication map on $\A^{1}$ extends only to a rational map 
$$
\P^{1}\times\P^{1}\dashrightarrow\P^{1}.
$$
This pronounced difference is also noticeable in the theory of algebraic cycles with moduli conditions \cite{BS}. 
\vspace{0.1in}

Let us write $lSm/k$ for the category of fine and saturated log schemes that are log smooth and separated over a fixed field $k$ equipped with the trivial log structure.
To carry out our constructions, we turn $lSm/k$ into a Grothendieck site using the {\it dividing Nisnevich topology}. 
The coverings in this topology arise from distinguished squares of different origins:
{\it dividing squares},
typically involving a log blow-up or the star subdivision of a monoid,
and {\it strict Nisnevich squares}, 
an adaption of the Nisnevich topology to log geometry.
Conceptually one can view the strict Nisnevich coverings as a witness for the scheme-theoretic data and the dividing coverings as a witness for the log structures.
The properties of these topologies are used throughout the proofs of our results.
\vspace{0.1in}

Another fundamental ingredient in the theory of log motives is the notion of {\it finite log correspondences} in $lSm/k$. 
With this in hand, we define the additive category $lCor/k$ of finite log correspondences over $k$, 
having the same objects as $lSm/k$ and with morphisms given by finite log correspondences.
There is a faithful embedding $lSm/k\hookrightarrow lCor/k$ induced by taking the graph of morphisms in $lSm/k$.
Passing to the underlying schemes yields a functor to the category $Cor/k$ of finite correspondences defined by Suslin-Voevodsky \cite{VSF}.
Moreover, 
the compositions in $lCor/k$ and $Cor/k$ are compatible in the obvious sense. 
\vspace{0.1in}

Let $\Lambda$ be a commutative unital ring. 
A {\it presheaf of $\Lambda$-modules with log transfers} on $lSm/k$ is a contravariant additive functor from $lCor/k$ to $\Lambda$-modules.
By the linear setup, every presheaf equipped with log transfers determines ``wrong way'' maps between $\Lambda$-modules parametrized by finite log correspondences.
As in the case of usual schemes, the structure sheaf $\mathcal{O}$ and the sheaf of global units $\mathcal{O}^{\times}$ are basic examples of presheaves with log transfers.
For any log scheme $(X,\mathcal{M}_{X})$ in $lSm/k$ the group completion $\mathcal{M}^{gp}$ is also a presheaf with log transfers.
\vspace{0.1in}

\begin{df}\label{A.0.2}
Suppose $Z=n_1Z_1+\cdots+n_rZ_r$ is an effective Cartier divisor on a scheme $X$ over $k$.
Let $I_i$ be the invertible sheaf of ideals defining $n_iZ_i$ for each irreducible component $Z_i$.
Then $(X,Z)$ is an fs log scheme over $k$ whose log structure is given by the Deligne-Faltings structure \cite[\S III.1.7]{Ogu} associated to the inclusions $I_i\rightarrow \mathcal{O}_X$,
$i=1,\ldots,r$.
In particular,
for the projective line, 
we set
\begin{equation}
\label{equation:box}
\boxx:=(\P^1,\{\infty\})\in lSm/k.
\end{equation}
For most part of this paper we us
e the multiplicity one case $n_1=\cdots=n_r=1$.
In this case the log structure of $(X,Z)$ is the compactifying log structure \cite[\S III.1.6]{Ogu} associated to the open immersion $X-Z\rightarrow X$.
\end{df}
\begin{exm}
\label{exm:box}
Following Definition \ref{A.0.2} in the case of a normal strict crossing divisor, we set 
\[
\boxx^{n}
=
((\P^1)^{n},\underbrace{\infty\times\P^{1}\times\cdots\times\P^{1}}_{n} 
+ 
\underbrace{\P^{1}\times\infty\times\P^{1}\times\cdots\times\P^{1}}_{n} 
+ 
\underbrace{\P^{1}\times\cdots\times\P^{1}\times\infty}_{n}).
\]
This is the $n$th self product of $\boxx$ for the standard monoidal structure on $lSm/k$.  
The notation means that on $(\P^1)^{n}$ we consider the strict normal crossing divisor given by setting the coordinate $z_{i}=\infty$ for $i=1,\ldots n$. 
\end{exm}
 
We note there are canonically defined morphisms 
$$
(\A^1,\emptyset)
\to
\boxx
\to
(\P^1,\emptyset)
$$
between distinctly different objects in $lSm/k$.
We emphasize that $\boxx$ has a nontrivial log structure.
\vspace{0.1in}

The triangulated category of effective log motives of $k$ with $\Lambda$-coefficients
$$
\ldmeff
$$ 
is obtained from the category of chain complexes of presheaves with log transfers on $lSm/k$ by imposing descent for all dividing Nisnevich coverings, 
and homotopy invariance with respect to $\boxx$. 
To every fs log scheme $X\in lSm/k$ we can associate the motive 
$$
M(X)\in \ldmeff.
$$
This construction is our main object of study in this work.
\vspace{0.1in}

Notwithstanding the technicalities,
the next four statements are straightforward consequences of the construction of $\ldmeff$.
\vspace{0.1in}

\begin{itemize}
\item {\rm (Monoidal structure)} For every $X,Y\in lSm/k$ there is a naturally induced isomorphism of log motives 
$$
M(X\times Y)\cong M(X)\otimes M(Y).
$$
\item {\rm (Homotopy invariance)} For every $X\in lSm/k$ there is a naturally induced isomorphism of log motives 
$$
M(X\times \boxx)\cong M(X).
$$
\item {\rm (Mayer-Vietoris)} For every strict Nisnevich distinguished square in $lSm/k$
\[
\begin{tikzcd}
Y'\arrow[d]\arrow[r]&Y\arrow[d]\\X'\arrow[r]&X
\end{tikzcd}
\]
there is a naturally induced homotopy cartesian square of log motives
\[
\begin{tikzcd}
M(Y')\arrow[d]\arrow[r]&M(Y)\arrow[d]\\
M(X')\arrow[r]&M(X).
\end{tikzcd}
\]
\item {\rm (Log modification)}
Every log modification $f:Y\rightarrow X$ of fs log schemes log smooth over $k$ in the sense of Kato \cite{FKato} induces an isomorphism of log motives 
$$
M(f)\colon M(Y)\cong M(X).
$$
\end{itemize}
\vspace{0.1in}

To an fs log scheme $X\in lSm/k$ and a vector bundle $\xi\colon \cE\to X$ we associate the Thom motive 
$$
MTh_{X}(\cE)
\in
\ldmeff
$$ 
via the blow-up of $\cE$ along its $0$-section.
We refer to Definitions \ref{A.5.7}, \ref{df:logvectorbundle}, and \ref{A.3.22} for precise statements.
The construction of Thom motives turns out to be an important one for several reasons.
One should note that the Betti realization of the log motivic Thom space of $\xi\colon\cE\to X$ is homotopy equivalent to the quotient of the unit disk bundle by the unit sphere bundle for the Betti realization of $\xi\colon\cE\to X$.
In the presence of a Euclidean metric, the latter is one formulation of Thom spaces in topology.
\vspace{0.1in}

We employ motivic Thom spaces of  vector bundles to show a Gysin isomorphism in Theorem \ref{A.3.36}.
A closely related result is the Gysin triangle for strict normal crossing divisors.
\vspace{0.1in}

\begin{itemize}
\item {\rm (Gysin triangle)} 
Suppose $Z,Z_1,\ldots,Z_r$ are smooth divisors forming a strict normal crossing divisor on $X\in Sm/k$, 
and set $Y:=(X,Z_1+\cdots+Z_r)$.
Let $E$ be the exceptional divisor of the blow-up $B_Z Y$ of $Y$ along $Z$, and let $N_ZY$ denote the normal bundle of $Z$ in $Y$.
Then there is a distinguished triangle 
\[
M(B_Z Y,E)\rightarrow M(Y)\rightarrow MTh(N_Z Y)\rightarrow M(B_Z Y,E)[1]
\]
of log motives in $\ldmeff$.
\end{itemize}
\vspace{0.1in}

In Theorem \ref{A.3.13} we show the following blow-up exact triangle.
\vspace{0.1in}

\begin{itemize}
\item {\rm (Blow-up triangle)} Let $X$ be a smooth scheme over $k$, and let $X'$ be the blow-up of $X$ along a smooth center $Z$.
Then there is a distinguished triangle 
\[
M(Z\times_X X')\rightarrow M(X')\oplus M(Z)\rightarrow M(X)\rightarrow M(Z\times_X X')[1]
\]
of log motives in $\ldmeff$.
\end{itemize}
\vspace{0.1in}

In $\A^1$-homotopy theory, 
the affine space $\A^{n}$ can be viewed as a disk. 
Similarly, 
$\boxx$ in Definition \ref{A.0.2} and more generally $\boxx^{n}$ in Example \ref{exm:box} can be viewed as a disk in our setting.
For our purposes,
however, 
it is also natural to consider the pair of projective spaces $(\P^{n},\P^{n-1})$, 
where $\P^{n-1}$ is considered as a hyperplane in $\P^{n}$. 
\vspace{0.1in}

In Proposition \ref{A.6.1} we prove the following fundamental invariance for log motives.
\vspace{0.1in}

\begin{itemize}
\item {\rm ($(\P^{n},\P^{n-1})$-invariance)} For every $X\in lSm/k$ there is a naturally induced isomorphism of log motives 
$$
M(X\times (\P^{n},\P^{n-1}))\cong M(X).
$$
\end{itemize}
\vspace{0.1in}

Following the convention in \cite{MR1764202} we set 
$$
\Lambda(0):=M(\Spec k)=\Lambda,
$$ 
and define the Tate object $\Lambda(1)$ to be the shifted cone
\[
\Lambda(1)
:=
M(\Spec k\stackrel{i_0}\rightarrow \P^1)[-2]
\]
in $\ldmeff$. 
Here $i_0:\Spec k \rightarrow \P^1$ is the $0$-section, and both schemes are equipped with a trivial log structure.
Moreover,
for $n\in\N$, 
we form the derived tensor product
\[\Lambda(n)
:=
\underbrace{\Lambda(1)\otimes \cdots \otimes \Lambda(1)}_{n}.
\]
More generally,
the $n$th Tate twist of an object $M\in\ldmeff$ is defined as
\[
M(n)
:=
M\otimes \Lambda(n).
\]
\vspace{0.1in}

In Theorem \ref{A.6.2} we conclude a projective bundle theorem for the projectivization of  vector bundles.
We refer the reader to Definition \ref{A.3.8} for the precise meaning of resolution of singularities over fields as in \cite{MR2289519}.
\vspace{0.1in}

\begin{itemize}
\item {\rm (Projective bundle theorem)} Assume that $k$ admits resolution of singularities.
Let $\cE$ be a vector bundle of rank $n+1$ over $X\in lSm/k$.
In $\ldmeff$ there is a canonical isomorphism
\[
M(\P(\cE))
\cong
\bigoplus_{i=0}^{n} M(X)(i)[2i].
\]
\end{itemize}
\vspace{0.1in}

As a consequence of the projective bundle theorem, we can relate Thom motives to Tate twists via the Thom isomorphism, 
see Theorem \ref{A.6.7}.
\vspace{0.1in}

\begin{itemize}
\item {\rm (Thom isomorphism)} 
Assume that $k$ admits resolution of singularities.
Let $\cE$ be a  vector bundle of rank $n$ over $X\in lSm/k$.
In $\ldmeff$, there is a canonical isomorphism
\[
MTh(\cE)
\cong
M(X)(n)[2n].
\]
\end{itemize}
\vspace{0.1in}

Associated with any log structure on a scheme, there is a corresponding notion of a boundary enabling us to think about log schemes as analogous to manifolds with boundaries.
In our setting, the boundary corresponds to an effective Cartier divisor on the underlying scheme.

\begin{df}\label{A.0.1}
For an fs log scheme $X$ over $k$, let $\partial X$ denote the set of points of $X$ with a nontrivial log structure.
That is,
$x\in\partial X$ if the stalk $\overline{\cM}_{X,x}$ of the characteristic monoid \index[notation]{deltaX @ $\partial X$}
$$
\overline{\cM}_{X}:=\cM_{X}/\cM_{X}^{\ast}
$$ 
at $x\in X$ is nontrivial.
\end{df}

We note that $\partial X$ is a closed subset of $X$ according to \cite[Proposition III.1.2.8]{Ogu} so that $X-\partial X$ is an open subscheme of $X$.
If $X\in lSm/k$, then $X-\partial X$ is smooth over $k$ and $\partial X$ is an effective Cartier divisor on $X$.
Theorem \ref{A.3.12} shows the following fundamental admissible blow-up property telling us how the motive of $X$ depends on the boundary $\partial X$.
\vspace{0.1in}

\begin{itemize}
\item {\rm (Admissible blow-ups)} Assume that $k$ admits resolution of singularities.
Let $f:Y\rightarrow X$ be a proper morphism of fs log schemes that are log smooth over $k$.
If the naturally induced morphism 
$$
Y-\partial Y\rightarrow X-\partial X
$$ 
is an isomorphism of $k$-schemes,
then there is a naturally induced isomorphism 
\[
M(Y)\rightarrow M(X)
\]
of log motives in $\ldmeff$.
\end{itemize}
\vspace{0.1in}

In Section \ref{section:cvvmotives} we compare $\ldmeff$ with Voevodsky's triangulated category of effective motives $\dmeff$ introduced in \cite{MR1764202} and reviewed in \cite{MR2181839}.
For proper smooth schemes, we identify particular hom-objects in $\ldmeff$ and $\dmeff$.
More precisely,
Theorem \ref{A.4.10} shows the following useful comparison result.
\vspace{0.1in}

\begin{itemize}
\item {\rm (Comparison with Voevodsky's derived category of effective motives)}  Assume that $k$ admits resolution of singularities.
Let $X$ and $Y$ be fs log schemes that are log smooth over $k$.
If $X$ is proper over $k$,
then for every $i\in \Z$, there is a naturally induced isomorphism of abelian groups 
\[
\hom_{\ldmeff}(M(Y)[i],M(X))
\cong
\hom_{\dmeff}(M(Y-\partial Y)[i],M(X-\partial X)).
\]
\end{itemize}
\vspace{0.1in}

As an immediate application we can identify the endomorphism ring of the unit object in $\ldmeff$ with the coefficients
\begin{equation}
\label{equation:endomorphismringofunit}
\hom_{\ldmeff}(\Lambda(0),\Lambda(0))
\cong
\Lambda.
\end{equation}
\vspace{0.1in}

Assuming $k$ admits resolution of singularities, we show there is an adjoint functor pair
\[
\omega_\sharp:
\ldmeff
\rightleftarrows
\dmeff:
R\omega^*.
\]

Theorem \ref{A.4.14} shows that the right adjoint functor $R\omega^*$ is fully faithful.
However, 
as witnessed by the sheaf of logarithmic differentials $\Omega^{i}_{/k}$, 
the functor $R\omega^*$ is not essentially surjective.
\vspace{0.1in}

While $\dmeff$ admits a fully faithful embedding into $\ldmeff$, 
the latter also contains interesting invariants that are not $\A^1$-invariant.
The latter is one way of justifying our choice of $\boxx$ as the preferred unit interval in $\ldmeff$.
\vspace{0.1in}

To identify the essential image of $R\omega^*$ we follow \cite{MR1764202} and declare that 
$$
\cF\in\ldmeff
$$
is {\it $\A^1$-local} if for every $X\in lSm/k$ and $i\in\Z$ the projection $X\times \A^1\rightarrow X$ induces an isomorphism of abelian groups 
\[
\hom_{\ldmeff}(M(X)[i],\cF)
\rightarrow
\hom_{\ldmeff}(M(X\times \A^1)[i],\cF).
\]
Theorem \ref{A.4.22} shows that every $\A^1$-local effective log motive is in the essential image of $R\omega^*$.
That is, 
Voevodsky's category of derived motives $\dmeff$ is equivalent to the {\it $\A^1$-localization} of $\ldmeff$.
\vspace{0.1in}

Let $\ldmeffprop$ be the smallest triangulated subcategory of $\ldmeff$ that is closed under small sums and contains every motive $M(X)$,
where $X\in lSm/k$ is proper as a scheme over $\Spec{k}$.
We identity $\ldmeffprop$ with the essential image of $R\omega^*$ under the assumption that $k$ admits resolution of singularities.
This shows there is an equivalence of triangulated categories
\begin{equation}\label{eq:equivpropintro}
\ldmeffprop\simeq \dmeff,
\end{equation}
which gives another log geometric description of Voevodsky's category $\dmeff$. Note that the equivalence \eqref{eq:equivpropintro} \emph{does not hold} with integral coefficients if we replace the dividing Nisnevich topology with its \'etale counterpart. See Remark \ref{Etalemot.5} and \ref{ssec:specul} below for a discussion about this.
\vspace{0.1in}

\subsection{Outline of the paper} In what follows, we give a more detailed overview of our work.

\subsubsection{Correspondences on logarithmic schemes} 
In Section \ref{sec:logtransfers} we introduce the notion of \emph{finite log correspondences}.
Our approach is fairly classical:
informally, 
an elementary finite log correspondence $Z$ from $X$ to $Y$ (smooth log schemes over $k$) is the datum of an elementary finite correspondence in the sense of Suslin-Voevodsky 
between the underlying subschemes $\ul{X}$ and $\ul{Y}$ that is equipped with the minimal log structure necessary to ensure the projection $Z\to X$ is a strict morphism. 
However, 
to get a morphism to the product $X\times Y$ in the category of log schemes, we require the existence of a morphism from the normalization $Z^N$ to $Y$ as part of the data.
The said condition on the normalization of an elementary finite log correspondence is reminiscent of the modulus condition on correspondences discussed in \cite{KSY2}.
We work with the normalizations to establish a well-defined associative composition of finite log correspondences and turn $lCor/k$ into a category. 
In turn, 
this requires a certain amount of technicalities involving solid log schemes; 
a notion with strong permanence inherited from $X$ and $Y$ by the log structure on the normalization $Z^N$ of the elementary finite log correspondence.  
We refer the reader to Section \ref{sec:logtransfers} (more specifically Definition \ref{A.5.2}) for details.

\subsubsection{cd structures on log schemes and sheaves with log transfers}
Suppose $X$ is a log scheme with log structure given by an open immersion $j\colon U\to X$ corresponding to the complement of an effective Cartier divisor $\partial X$.
Then the motive of $X$ is expected to help us understand the cohomology of $U$.
To achieve this, we need to impose a suitable invariance of the motives under the choice of the ``compactification'' $X$.
The solution we offer involves the notion of \emph{a dividing cover}, 
i.e., 
a surjective proper log \'etale monomorphisms $f\colon Y\to X$ (see Definition \ref{A.5.14} and \cite{Par}). 

Intuitively, we can think of such morphisms as blow-ups with their center in the boundary $\partial X$. 
We also consider strict Nisnevich covers, 
i.e., 
Nisnevich covers coming from the underlying schemes, 
These two types of coverings conspire into a cd-structure in the sense of \cite{Vcdtop}, 
and we call the associated topology the \emph{dividing Nisnevich topology}. 
As discussed in Section \ref{sec:topologies}, the dividing Nisnevich cd-structure is not bounded. 
For this reason, we need to generalize Voevodsky's results to \emph{quasi-bounded density structures}, 
which is the central technical notion of the section (see Definition \ref{A.9.65}). 
Theorem \ref{A.10.4} proves that for any quasi-bounded, regular, and complete cd structures on a category $\mathcal{C}$ the associated cohomology groups vanish above the 
\emph{density dimension} of any object. 
Corollary \ref{A.10.10} allows us to define a suitable descent structure in the sense of \cite{CD09}.
\vspace{0.1in}

In Section \ref{sec.sheavestransfer} we study how finite log correspondences relate to topologies. 
Our main findings can be summarized in the following combination of Proposition \ref{A.8.7}, Proposition \ref{A.5.22}, Proposition \ref{A.8.8} and Theorem \ref{A.5.23} 
(see below for a reminder on Kummer \'etale maps).

\begin{thm}
\label{thm-intro}
Let $t$ be one of the following topologies on $lSm/k$: strict or dividing Nisnevich, strict or dividing \'etale, Kummer \'etale, or log \'etale.  
\begin{itemize}
\item[(i)] 
The topology $t$ is compatible with log transfers on $lSm/k$. 
\item[(ii)]
The category $\mathbf{Shv}_{t}^{\rm ltr}(k,\Lambda)$ is a Grothendieck abelian category.
\item[(iii)] 
For every complex of $t$-sheaves $\mathcal{F}$ with log transfers and every $X\in lSm/k$, 
there is a canonical isomorphism
\[
\hom_{\Deri(\Shvltrtkl)}(a_t^*\Zltr(X),\cF[i])
\cong 
\bH_t^i(X,\cF).
\]
\end{itemize}
\end{thm}
The dividing topology on $lSm/k$ is not subcanonical (like the $h$-topology on $Sm/k$).
However, 
there is an algebraic cycle description of the sections of the sheafification of the representable presheaf $a_{t}^*\Zltr(X)$ of any $X$ (here $t$ is one of the dividing topologies on $lSm/k$). 
To that end, 
we introduce the notion of \emph{dividing log correspondences} between fs log schemes in Definition \ref{A.5.47}, 
which up to log modifications can be viewed as finite log correspondences in the sense of Definition \ref{A.5.2}. 
In fact, 
it turns out that the resulting categories of sheaves are equivalent; see Proposition \ref{A.5.56}. 
\vspace{0.1in}

Dividing log correspondences allow us to further restrict the class of log schemes that is needed to build $\ldmeff$. 
For $t=dNis$, $d\acute{e}t$, $l\acute{e}t$, the results in Section \ref{Subsec::Sheaves.SmlSm} give an equivalence of categories
\begin{equation}
\label{eq-smlsmintro} 
\iota^*
\colon
\Shv_{t}^{\rm ltr}(k,\Lambda)
\xrightarrow{\simeq}
\Shv_{t}^{\rm ltr}(SmlSm/k,\Lambda). 
\end{equation}
Here $\iota\colon SmlSm/k\hookrightarrow lSm/k$ is the full subcategory of $lSm/k$ consisting of fs log schemes of the form $(X,\partial X)$, 
where $X$ is smooth over $k$ and $\partial X$ is a strict normal crossing divisor on $X$ (see Lemma \ref{lem::SmlSm}). 
This gives a way of computing dividing cohomology groups; see Theorem \ref{Div.4}.
\begin{thm}
Let $\cF$ be a bounded below complex of strict Nisnevich (resp.\ strict \'etale, resp.\ Kummer \'etale) sheaves on $SmlSm/k$.
Then for every $X\in SmlSm/k$ and integer $i\in \Z$ there is an isomorphism
\begin{align*}
\bH_{dNis}^i(X,a_{dNis}^*\cF)&\cong \colimit_{Y\in X_{div}^{Sm}}\bH_{sNis}^i(Y,\cF)
\\
\text{(resp.\ }\bH_{d\acute{e}t}^i(X,a_{d\acute{e}t}^*\cF)&\cong \colimit_{Y\in X_{div}^{Sm}}\bH_{s\acute{e}t}^i(Y,\cF),
\\
\text{resp.\ }\bH_{l\acute{e}t}^i(X,a_{l\acute{e}t}^*\cF)&\cong \colimit_{Y\in X_{div}^{Sm}}\bH_{k\acute{e}t}^i(Y,\cF)\text{)}.
\end{align*}
\end{thm}
We remark that the equivalence \eqref{eq-smlsmintro} provides a canonical extension to $lSm/k$ of any sheaf with log transfers a priori defined on $SmlSm/k$.

\subsubsection{Motivic categories and properties of motives}
Defining $\ldmeff$ is straightforward at this point. 
The local objects for the $\boxx$-descent model structure are precisely the complexes of strictly $\boxx$-invariant $t$-sheaves with log transfer (especially for $t=dNis$), 
and the properties established in Sections \ref{sec:topologies}-\ref{sec.constructldmeff} automatically give the first set of axioms 
($\boxx$-invariance, Mayer-Vietoris, log modification invariance). 
We present next a leitfaden for the less transparent properties.
\vspace{0.1in}

First, we extend the log modification invariance of motives to blow-ups along smooth centers contained in $\partial X$. 
This is the content of Section \ref{ssec:invariancesmoothblowups}, 
and it is required for the following reason. 
Like any other smooth curve, 
the affine line $\A^1$ affords a unique compactification to an object in $lSm/k$, namely $\boxx$. 
In higher dimensions this construction is no longer canonical.
For example, 
in dimension $2$, 
the product $\A^1\times \A^1 = \A^2$ can be compactified to $\boxx\times \boxx = \boxx^2$, or to $(\P^2, \P^1)$, where the  $\P^1$ is viewed as a line at infinity. 

The construction of $\ldmeff$ makes the first choice automatically contractible, but from the construction, it is unclear what happens to $(\P^2, \P^1)$. 
Inspired by the properties of log differentials observed in \cite[\S 6.2]{BS}, 
see also Section \ref{PnPn-1invariancediff}, 
it is reasonable to expect that $(\P^2, \P^1)$ is contractible too.
In order to compare $\boxx^2$ and $(\P^2, \P^1)$, 
we consider the blow-up $X$ of $\P^2$ at a point contained in its boundary, $\P^1$. 
This latter object can more easily be compared with the blow-up of $\P^1\times \P^1$ at the point $(\infty, \infty)$. 
The map $X\to \P^2$, however, is not a log modification. 
Our crucial observation is that a combination of Mayer-Vietoris, $\boxx$-invariance, and dividing-invariance properties are, in fact, enough to establish that this map is an equivalence in $\ldmeff$. 
See Theorem \ref{A.3.7} and Theorem \ref{A.3.13} for precise statements. 
A key point of the proof is to deal with the case of surfaces; see Proposition \ref{A.3.39}. 
This finally allow us to prove the $(\P^n, \P^{n-1})$-invariance property of motives and also the blow-up formula in Proposition \ref{A.6.1} and Theorem \ref{A.3.13}, 
respectively.
\vspace{0.1in}

Section \ref{ssec:ThomMotives} discusses Thom motives, 
i.e., 
the motives of Thom spaces of vector bundles $\mathcal{E}$ in the logarithmic setting. 
To get the correct homotopy type, the definition we use involves the blow-up of $\mathcal{E}$ along the zero section (and it is therefore quite different from the corresponding notion in Voevodsky's category).  
The critical technical result here is Proposition \ref{A.3.42}, in which we show the existence of a canonical isomorphism 
\[
MTh_{X_1}(\cE_1)\otimes MTh_{X_2}(\cE_2)\rightarrow MTh_{X_1\times_S X_2}(\cE_1\times_S \cE_2).
\]
for vector bundles $\cE_i$ on $X_i$. 
This part is quite involved in the log setting compared to the $\A^1$-invariant setting. 
\vspace{0.1in}

Having Thom motives at our disposal, we finally prove the Gysin isomorphism and the existence of Gysin triangles in Theorem \ref{A.3.36}.
Assuming resolution of singularities, we show an even finer invariance under admissible blow-ups in Section \ref{ssec:invadmblowrefined}. 
Our comparison with Voevodsky's category of motives uses the latter, from which we deduce the projective bundle formula in Theorem \ref{A.6.2} as well as the Thom isomorphism Theorem \ref{A.6.7}.
\vspace{0.1in}

In Section \ref{section:calculus} we review the construction of a $Sing$ functor, 
inspired by the Suslin complex, 
for a symmetric monoidal category $\mathcal{A}$ equipped with a weak interval object $I$. 
This seeming detour is necessary since $\boxx$ is not an interval in the sense of \cite{MV} for $lCor/k$: 
the graph of the multiplication map $\mu\colon \A^1\times \A^1\to \A^1$ does not extend to an admissible correspondence in our sense. 

Our solution, 
see also \cite[\S 4]{AdditiveHomotopy}, 
is to enlarge the category using calculus of fractions $\mathcal{A}[T^{-1}]$ for a suitable class of maps $T$ to include the multiplication for $I$, 
build the singular functor there, 
and then restrict back to the original category. 
Applied to the category of chain complexes $\mathbf{C}(\mathbf{Psh}^{\rm ltr}(k))$, 
the said construction produces a $\boxx$-local object, 
albeit not very computationally-friendly.  
Nevertheless, it plays a key role in the comparison between $\ldmeff$ and $\dmeff$, see Propositions \ref{A.4.12} and \ref{A.4.29}.
 \vspace{0.1in}

By ignoring transfer structures altogether, we present a simpler construction $\mathbf{logDA}^{\rm eff}(k,\Lambda)$ that is a direct analogue of Ayoub's $\mathbf{DA}^{\rm eff}(k, \Lambda)$ 
(see \cite{Ayo07} or \cite{AyoICM} for a short overview). 
However, 
using correspondences allows us to prove much stronger results, including the comparison with Voevodsky's $\dmeff$ and the existence of Gysin triangles.

\subsubsection{Applications: Hodge sheaves and cyclic homology} 
Prototype examples of objects that are not available in Voevodsky's derived category of motives are the additive group scheme $\mathbb{G}_{a}$ and the differentials $\Omega^i_{-/k}$ 
(or their absolute counterpart). 
We show in Section \ref{sec:Hodge} that this problem disappears in the log setting (see Corollary \ref{Pn-invariance of Log differentials}, Theorem \ref{thm:Omega_log_transfer}).

\begin{thm}
Suppose $k$ is a perfect field and $X\in lSm/k$. 
Let $t$ be one of the topologies on log schemes in Theorem \ref{thm-intro} or the Zariski topology. 
Then the assignment 
$$
X\mapsto \Omega^j_{X/k}
$$ 
extends to a $t$-sheaf with log transfers. 
Moreover, for every $i$, $j$, $n\geq 0$, 
there is an isomorphism
\[
H_{t}^i(X\times (\P^n,\P^{n-1}),\Omega_{X\times (\P^n,\P^{n-1})}^j)\cong H_{t}^i(X,\Omega_X^j).
\]
In particular, the sheaves $\Omega^j_{-/k}$ are strictly $\boxx$-invariant $t$-sheaves with log transfers.
\end{thm}

The transfer structure is obtained by extending R{\"u}lling-Chatzistamatiou's work \cite{ChatzistamatiouRullingANT} to the logarithmic setting. 
By extending from $SmlSm/k$ to $lSm/k$ via Construction \ref{SmlSm.3},
we conclude there exists a natural isomorphism
\begin{equation}
\label{eq.introhodge}
\hom_{\ldmeff}(M(X),\Omega^j_{-/k}[i]) 
\cong H^i_{Zar}(\ul{X},\Omega^j_{X/k}).
\end{equation}
Representing Hodge cohomology in $\ldmeff$ refines the analogous result for de Rham cohomology in $\mathbf{DM}^{\rm eff}(k)$; see \cite{lecomtewach}. 
\vspace{0.1in}

When the characteristic of the base field $k$ is zero, 
combining \eqref{eq.introhodge} with the Hochschild-Kostant-Rosenberg theorem yields an isomorphism for every $X\in Sm/k$
\[ 
HC_n(X) 
\cong 
\hom_{{\mathbf{logDM}^{\rm eff}(k, \mathbb{Q})}}(M(X), \bigoplus_{p\in \Z} \Omega^{\leq 2p}_{-/k}[2p-n]).
\]
This shows representability of cyclic homology in ${\mathbf{logDM}^{\rm eff}(k, \mathbb{Q})}$, 
see Section \ref{ssec-Hodgecyclic}, 
an example of a non-$\A^{1}$-invariant theory on $Sm/k$ and thus out of reach of $\mathbf{DM}^{\rm eff}(k)$.
We plan to address the question of representing the cohomology of Hodge-Witt sheaves in characteristic $p>0$ in later work. 

\subsubsection{Appendices}
Besides reminding the reader of some necessary background material in log geometry, the main aim of Appendix A is to minimize confusion in other chapters. 
Appendix A can be consulted if the reader doubts notation or the precise form of basic log geometric notions. 
Appendix B briefly introduces model structures on chain complexes and thereby introduces notation employed elsewhere in the text.

\subsection{Outlook and future developments}
We conclude the introduction with some remarks and conjectures that put our theory in a broader framework.

\subsubsection{Relationship with reciprocity sheaves}
The idea of extending Voevodsky's framework to encompass non-$\A^{1}$-invariant phenomena have been investigated by many authors in the past decade, 
starting with the work of Bloch-Esnault \cite{BEAdditiveChow}, \cite{BEAdditive} on additive Chow groups and Chow groups with modulus. 
\vspace{0.1in}

The theory of \emph{reciprocity sheaves} by Kahn-Saito-Yamazaki \cite{KSY}, 
and the closely related theory of \emph{modulus sheaves} by Kahn-Miyazaki-Saito-Yamazaki \cite{KSY2}, \cite{KSY3}, 
has provided a unified framework for both $\mathbb{A}^1$-invariant objects as well as interesting non-$\mathbb{A}^1$-invariant ones. 
The category $\mathbf{RSC}_{\rm Nis}$ of (Nisnevich) reciprocity sheaves is a full subcategory of $\Shv^{\rm tr}_{Nis}(k, \mathbb{Z})$ that contains $\mathbf{HI}_{\rm Nis}$, 
i.e., 
$\mathbb{A}^1$-invariant Nisnevich sheaves with transfers but also the additive group scheme $\mathbb{G}_a$, 
the sheaf of absolute K{\"a}hler differentials $\Omega^i$, 
and the de Rham-Witt sheaves $W_m\Omega^i$. 
The connection between $\mathbf{RSC}_{\rm Nis}$ and our theory has recently been provided by Saito \cite{SaitoRSClog}.

\begin{thm}[Saito]
\label{saitothmintro} 
There exists a fully faithful exact functor
\[
\mathcal{L}og
\colon 
\mathbf{RSC}_{\rm Nis} \to \Shv_{dNis}^{\rm ltr}(k, \mathbb{Z})
\]
such that $\mathcal{L}og(F)$ is strictly $\boxx$-invariant for every $F\in \mathbf{RSC}_{\rm Nis}$. 
Moreover, for each $X\in Sm/k$, there is a natural isomorphism
\[
H^i_{Nis}(X, F_X) \cong \hom_{\ldmeff}(M(X),\mathcal{L}og(F)[i]).
\]
\end{thm}

The isomorphism shows that Nisnevich cohomology of reciprocity sheaves is representable in $\ldmeff$, 
giving log motivic interpretations (at least for reduced modulus) of properties such as the Gysin exact sequence, 
the projective bundle formula (assuming resolution of singularities),
the blow-up formula, 
and the $(\P^n, \P^{n-1})$-invariance of the cohomology of reciprocity sheaves provided in \cite{BRS}.
This applies to $(\P^n, \P^{n-1})$-invariance and the existence of pushforward in Hodge cohomology; see Section \ref{sec:Hodge}.
At least for $X\in SmlSm/k$, this is a consequence of Theorem \ref{saitothmintro}, and the corresponding formulas in \cite{BRS}.
We view this as an essential step in connecting the theory of motives with modulus to logarithmic motives. 

A different incarnation of the category of reciprocity sheaves (again assuming resolution of singularities) have been recently introduced in \cite{BindaMerici}.  Thanks to Saito's theorem, together with a purity theorem for $\boxx$-local complexes of sheaves with log transfers \cite[Thm. 5.10]{BindaMerici}, the category of log-reciprocity sheaves contains the category $\mathbf{RSC}_{\rm Nis}$ as a full subcategory.

This is in turn essential to develop an analogue of the theory of higher Albanese sheaves of Ayoub,  Barbieri-Viale and Kahn, see \cite{BindaMericiSaito}, thanks to the existence of a \emph{homotopy $t$-structure} (analogue of the Morel-Voevodsky $t$-structure) on the category of log motives by \cite[Thm. 5.7]{BindaMerici}.

\subsubsection{Conjectures and speculations} \label{ssec:specul}
Suppose that $\Lambda$ is an $N$-torsion ring, 
where $N>0$ is coprime to the exponential characteristic of the ground field $k$. 
Voevodsky proved the following version of Suslin's rigidity theorem for the small \'etale site of $k$, see \cite[Theorem 9.35]{MVW}.

\begin{thm} 
The morphism of sites $\pi\colon k_{\acute{e}t} \to (Sm/k)_{\acute{e}t}$ induces an equivalence
\[
\pi_\sharp \colon \mathbf{D}(k_{\acute{e}t},\Lambda)\rightarrow \dmeffet.
\]
\end{thm}  
There are many other instances of rigidity theorems comparing small and big \'etale sites in the motivic context. 
See \cite{AyoubEtale} for $\mathbf{DA}_{\acute{e}t}(S, \Lambda)$, 
\cite{CDEtale} for motives with transfer, 
and \cite{BambVezz} for the rigid analytic setting.
\vspace{0.1in}

Recall that a morphism of fs log schemes $f\colon X\to Y$ is called Kummer \'etale if it is exact and log \'etale, 
see Proposition \ref{kummer-is-logetandstrict}. 
A typical example is given by $\A^1_k\to \A^1_k$, $t\mapsto t^n$, 
where $n\geq 1$ and $(n, {\rm char}(k))=1$. 
Here $\A^1_k$ is endowed with the log structure given by the origin, written $\A_\N = \Spec{k[\N]}$ in the following.  
Proposition \ref{ketcomp.4} shows that every locally constant log \'etale sheaf on $lSm/k$, 
see Definition \ref{ketcomp.3}, 
admits a unique transfer structure. 
This allow us to define a functor
\begin{equation}
\label{eq:conjrig}
\eta_\sharp
\colon 
\mathbf{D}(\Shv(k_{\acute{e}t},\Lambda))
\rightarrow 
\ldmefflet,
\end{equation}
which is easily seen to be fully faithful. 
On the right-hand side we consider the log \'etale topology, 
i.e., 
the smallest Grothendieck topology finer than the Kummer \'etale topology and the dividing topology 
(on the left-hand side there are no non-trivial dividing covers on the small site, so it makes no difference). 
Conjecture \ref{conjrigidity} is an analog of the Suslin-Voevodsky rigidity theorem.

\begin{conj} 
The functor \eqref{eq:conjrig} is an equivalence.
\end{conj}

More generally, 
if $\Lambda$ is a torsion ring of exponent invertible in $k$, 
we expect there are equivalences 
\[
\ldmeffet\leftarrow \ldmeffetprop\rightarrow \dmeffet.
\]
Conjecture \ref{ketcomp.7} predicts a similar result in the log \'etale topology. 
This expectation is justified by examples of strictly $\boxx$-invariant complexes of sheaves that are not strictly $\A^1$-invariant, 
such as $\Omega^i_{-}$ and $W_m\Omega^i_{-}$, 
are nullified under our assumption on $\Lambda$. 
This is similar in spirit to \cite[Theorem 3.5]{tor-div-rec} and \cite{MiyazakiCI}.
\vspace{0.1in}

Let us discuss the case when the ground field $k$ has positive characteristic $p>0$.  
Then the category $\mathbf{DM}_{\acute{e}t}(k, \mathbb{Z})$ of \'etale motives has trivial $p$-torsion and is in fact a $\mathbb{Z}[1/p]$-linear category. 
On the other hand,
the \'etale (or Lichtenbaum) motivic cohomology groups with $\Z/p^r$-coefficients are not $\mathbb{A}^1$-invariant, 
so they cannot be hom-groups in one of Voevodsky's categories.

Currently, there is no interpretation of the category of \emph{integral} \'etale motivic complexes in terms of algebraic cycles.
Milne-Ramachandran \cite{MRpreprint} proposed a construction of $\mathbf{DM}^{\rm gm}_{\acute{e}t}(k, \mathbb{Z})$ as a dg-pullback of $\mathbf{DM}_{\acute{e}t}(k, \mathbb{Z}[1/p])$ 
and $\mathbf{D}^b_c(R)$, 
the derived category of coherent complexes of graded modules over the Raynaud ring $R$, along the rigid realisation functor. 
The category of \'etale \emph{log} motives constructed in this work can be a candidate for a cycle-theoretic incarnation of Milne-Ramachandran's category 
$\mathbf{DM}^{\rm gm}_{\acute{e}t}(k, \mathbb{Z})$. 
To support this belief, 
note that $\mathbf{logDM}_{d\acute{e}t, prop}^{\rm eff}(k, \Z/p^r)$ is nontrivial because $\Z/p\not\cong 0$, 
while its $\mathbb{A}^1$-invariant counterpart is trivial because of the Artin-Schreier sequence. 
It would be interesting to compare $\mathbf{logDM}_{d\acute{e}t, prop}^{\rm eff}(k, \Z/p^r)$ with $\mathbf{D}^b_c(R, \Z/p^r)$ via a log-crystalline realisation functor.

\subsection{Notations and terminology}
Every log scheme in this paper is equipped with a Zariski log structure unless otherwise stated.
Recall that for every fs log scheme $X$ with an \'etale log structure,
there exists a log blow-up $Y\rightarrow X$ such that $Y$ has a Zariski log structure thanks to work by Niziol \cite[Theorem 5.10]{Niz}.
Since motives are invariant under log blow-ups,
this shows that the restriction to fs log schemes with Zariski log structures is no loss of generality.
\vspace{0.1in}

We employ the following notation (see the Index for a more comprehensive list).
\vspace{0.1in}

\begin{tabular}{l|l}
$k$ & (perfect) field\\
$Sm/S$ & category of smooth and separated schemes of finite type over \\ & a scheme $S$\\
$lSm/S$ & category of log smooth and separated fs log schemes of finite \\ & type over an fs log scheme $S$\\
$lSch/S$ & category of separated fs log schemes of finite type over an fs\\ &  log scheme $S$\\
$\LCor$ & category of finite log correspondences over $k$\\
$SmlSm/S$ & the full subcategory of $lSm/S$ consisting of $X$ such that $\underline{X}$ is \\ & smooth over $\underline{S}$\\ 
$\Lambda$ & commutative unital ring \\
$\Lambda\text{-}\mathbf{Mod}$ & category of $\Lambda$-modules \\
$\cA$ & abelian category \\
$\Co(\cA)$ & category of {unbounded} chain complexes of $\cA$ \\
${\bf K}(\cA)$ & $\Co(\cA)$ modulo chain homotopy equivalences\\
$\Deri(\cA)$ & derived category of $\cA$ \\
\end{tabular}

\vspace{0.1in}

\subsection*{Acknowledgements} 
We thank Joseph Ayoub for reading a preliminary version of this paper and providing helpful suggestions.
We also thank Shuji Saito for his interest in this work, a detailed reading of the manuscript, 
and for sharing the first version of \cite{SaitoRSClog}. 
We would also like to thank Kay R{\"u}lling and  Alberto Merici for many interesting conversations on the subject of this paper.
Finally, we are grateful to the referees for their detailed comments leading to improvements of this text.

It is a pleasure to acknowledge the generous hospitality and support from Institut Mittag-Leffler in Djursholm (2017), 
University of Oslo (2017-2019), 
and Forschungsinstitut f{\"u}r Mathematik at ETH Zurich (2019) during the various stages of this project.
Binda was partially supported by the SBF 1085 ``Higher Invariants'' at the University of Regensburg and the PRIN ``Geometric, 
Algebraic and Analytic Methods in Arithmetic'' at MIUR. 
{\O}stv{\ae}r was partially supported by the Friedrich Wilhelm Bessel Research Award from the Humboldt Foundation, 
a Nelder Visiting Fellowship from Imperial College London, 
the Professor Ingerid Dal and sister Ulrikke Greve Dals prize for excellent research in the humanities, 
and a Guest Professorship under the auspices of The Radboud Excellence Initiative.
The authors gratefully acknowledge support by the RCN Frontier Research Group Project no.~250399 ``Motivic Hopf Equations,''
and the Centre for Advanced Study at the Norwegian Academy of Science and Letters in Oslo, 
Norway, which funded and hosted our research project “Motivic Geometry” during the 2020/21 academic year,
\newpage

\section{Finite logarithmic correspondences} \label{sec:logtransfers} 
\subsection{Definition of finite log correspondences} 
This section introduces the notion of finite log correspondences relative to any base field $k$ equipped with its trivial log structure.
\footnote{An alternate notion of finite correspondences between log schemes, free of any normalization constraints, was considered by K. R{\"u}lling in \cite{Rulling-logcor}. 
We thank him for sharing his note with us.}
Our approach is an adaptation to the log setting of the original definition of finite correspondences due to Suslin and Voevodsky \cite{VSF}, \cite{MVW}.
As for usual schemes, one can intuitively regard finite log correspondences as multivalued functions, but some subtleties are arising involving the log structure.
Building on this concept, we turn $lSm/k$ into an additive category $\LCor$ of finite log correspondences over $k$. 
The objects of $\LCor$ are fine and saturated log schemes that are log smooth over $k$, 
i.e., 
the same objects as in $lSm/k$.  
As morphisms, we take the free abelian group generated by elementary log correspondences (see \cite[Lecture 1]{MVW}).
\vspace{0.1in}

If $X\in lSm/k$ we write $\underline{X}$ for the corresponding underlying scheme.

\begin{df}
\label{A.5.2}
For $X,Y\in lSm/k$, an {\it elementary log correspondence}\index{correspondence!elementary log} \index{correspondence!finite log} $Z$ from $X$ to $Y$ consists of 
\begin{enumerate}
\item[(i)] an integral closed subscheme $\underline{Z}$ of $\underline{X}\times \underline{Y}$ that is finite and surjective over a connected component of $\underline{X}$, and
\item[(ii)] a morphism $Z^N\rightarrow Y$ of fs log schemes, 
where $Z^N$ denotes the fs log scheme whose underlying scheme is the normalization of $\underline{Z}$ and whose log structure is induced by the one on  $X$.
More precisely, if $p:\underline{Z^N}\rightarrow \underline{X}$ denotes the induced scheme morphism, then the log structure $\cM_{Z^N}$ is given as $p_{log}^*\cM_X$.
\end{enumerate}

A {\it finite log correspondence} from $X$ to $Y$ is a formal sum $\Sigma n_{i}Z_{i}$ of elementary log correspondences $Z_{i}$ from $X$ to $Y$.
Let $\lCor_k(X,Y)$ denote the free abelian group of finite log correspondences from $X$ to $Y$. 
\end{df}

The above definitions coincide with the ones in \cite{VSF} and \cite{MVW} when $X$ and $Y$ have trivial log structures. 
We shall omit the subscript $k$ and write $\lCor(X,Y)$ when no confusion seems likely to arise. 
Note that if $X=\amalg_{i\in I} X_i$ is a decomposition into connected components, 
then there is a direct sum decomposition\index[notation]{lCor(X,Y)@$\lCor(X,Y)$}
\[
\lCor(X,Y)
=
\bigoplus_{i\in I} \lCor(X_i,Y).
\]

\begin{rmk}
\label{A.5.29}
\begin{enumerate}
\item[(i)]
With reference to Definition \ref{A.5.2}, 
let $p:\underline{Z^N}\rightarrow \underline{X}$ and $q:\underline{Z^N}\rightarrow \underline{Y}$ denote the induced scheme morphisms.
Then condition (ii) is equivalent to the existence of a morphism of log structures
$$
q_{log}^*\cM_Y\rightarrow p_{log}^*\cM_X
$$ 
on the underlying scheme $\underline{Z^N}$ (see Definition \ref{A.9.13}).
Indeed, to give a morphism $Z^N\rightarrow Y$, we only need to specify a morphism $q_{log}^*\cM_Y\rightarrow \cM_{Z^N}$.
We have that $\cM_{Z^N}=p_{log}^*\cM_X$.
\item[(ii)] Our notion of an elementary log correspondence involves considering a log structure on the normalization of the underlying elementary correspondence between the underlying schemes. 
The reason for this extra complication will be explained below in the construction of compositions of finite log correspondences.
\item[(iii)] Let $V\in \lCor(X,Y)$ be an elementary log correspondence. 
Then condition (ii) in Definition \ref{A.5.2} implies the inclusion $V\cap ((X-\partial X)\times Y) \subset V\cap (X\times (Y- \partial Y))$. 
\item[(iv)] Suppose $F/k$ is a finite separable field extension. 
Then for every $U\in lSm/F$, we have $U\times_k Y\cong U\times_F Y_F$, where $Y_F:=\Spec F\times_k Y$.
Hence there is a base change isomorphism
\[
\lCor_k(U,X)
\cong 
\lCor_F(U,X_F).
\]
\end{enumerate}
\end{rmk}

\begin{exm}
\label{A.5.31}
\begin{enumerate}
\item[(1)] Suppose $X$ has a trivial log structure. 
Then $Z^N$ has a trivial log structure, and $Z^N\rightarrow Y$ factors through $Y-\partial Y$. 
Hence there is an isomorphism of abelian groups
\[
\lCor(X,Y)
\cong 
\Cor(X,Y-\partial Y).
\] 
In particular, 
$\lCor(\Spec k,Y)$ can be identified with the group of $0$-cycles on the open complement $Y-\partial Y$.
\item[(2)]   
If $Y$ has a trivial log structure, then the condition (ii) is automatically satisfied. 
It follows that $\lCor(X,Y)\cong\Cor(\underline{X},Y)$.
In particular, $\lCor(X,\Spec k)$ is the free group generated by the irreducible components of $X$.
\item[(3)]   
By combining (1) and (2) we see that if $X$ and $Y$ have trivial log structures, then $\lCor(X,Y)\cong\Cor(X,Y)$. 
\item[(4)] 
Let $i:Z\rightarrow X\times Y$ be a closed immersion of fs log schemes, 
i.e., 
$\underline{Z}$ is a closed subscheme of $\underline{X}\times \underline{Y}$ and the induced homomorphism $i_{log}^*(\cM_{X\times Y})\rightarrow \cM_Z$ is surjective.
In addition, 
assume that $Z$ is strict (i.e., $p_{log}^*(\cM_X)\cong \cM_Z$, where $p$ is the projection from $X\times Y$ onto $X$), finite, and surjective over a component of $X$. 
Then the induced morphism $Z\rightarrow Y$ gives rise to a morphism $Z^N\rightarrow Y$.
Hence we can consider $Z$ as a finite log correspondence from $X$ to $Y$.
\end{enumerate}
\end{exm}

The reader might object that our definition of log correspondence is too restrictive since we require the underlying subscheme $\ul{Z}$ of any elementary log correspondence to be finite over $\ul{X}$ 
(and not, for example, only over $X-\partial X$, cf. \cite{KSY2}). We  ``correct'' this by introducing a suitable topology on log schemes, the dividing topology. 
By sheafifying the representable presheaf $\lCor(-,X)$ for the dividing topology, 
we can consider correspondences that are finite over $X- \partial X$ and become finite globally after a suitable log modification. 
This will be discussed in Section \ref{subsec:divdinglogcor}.

\subsection{Solid log schemes}
\label{ss:sls}
For the notion of finite log correspondences to have any meaning, we need to define a sensible composition of finite log correspondences $V\in \lCor(X,Y)$ and $W\in \lCor(Y,Z)$, where $X,Y,Z\in lSm/k$.  
To that end, it turns out that we need to construct log structures on normalizations of cycles on $\underline{X}\times \underline{Z}$. 
This can be achieved using cycles on $\underline{X}\times\underline{Y}\times \underline{Z}$ provided that $X$, $Y$, and $Z$ are solid log schemes,
as defined below.  

\begin{df}
\label{A.9.1}
Let $X$ be a coherent log scheme over $k$. 
We say that $X$ is {\it solid} \index{log scheme!solid}if for any point $x\in X$, 
the induced map
\[
\Spec {\cO_{X,x}}
\rightarrow 
\Spec {\cM_{X,x}}
\]
is surjective.
\end{df}

According to \cite[Proposition III.1.10.8]{Ogu} this definition of solidness is equivalent to the one given in \cite[Definition III.1.10.1(2)]{Ogu}.

\begin{lem}
\label{A.9.2}
Let $f:Y\rightarrow X$ be a strict open morphism of coherent log schemes. If $X$ is solid, then $Y$ is solid.
\end{lem}
\begin{proof}
Consider the induced commutative diagram
\[
\begin{tikzcd}
\Spec{\cO_{Y,y}}\arrow[r]\arrow[d]&\Spec{\cM_{Y,y}}\arrow[d]\\
\Spec{\cO_{X,x}}\arrow[r]&\Spec{\cM_{X,x}}
\end{tikzcd}\]
of topological spaces where $y\in Y$ is a point and $x:=f(y)$. 
The lower horizontal arrow is surjective since $X$ is solid. 
The left vertical arrow is surjective by \cite[IV.1.10.3]{EGA} since $f$ is open.
Moreover, 
the right vertical arrow is an isomorphism since $f$ is strict. 
Thus the upper horizontal arrow is surjective.
\end{proof}

\begin{lem}
\label{A.9.6}
Let $X$ and $Y$ be log schemes, and let $g:\underline{X}\rightarrow \underline{Y}$ be a morphism of schemes. 
If the log structure homomorphisms $\alpha_X:\cM_X\rightarrow \cO_X$ and $\alpha_Y:\cM_Y \rightarrow \cO_Y$ are injective, 
then there is at most one morphism $f:X\rightarrow Y$ of log schemes such that $\underline{f}=g$.
\end{lem}
\begin{proof}
Such a morphism $f$ consists of $\underline{f}:\underline{X}\rightarrow \underline{Y}$ and a commutative diagram
\[
\begin{tikzcd}
\cM_Y\arrow[d,"\alpha_Y"]\arrow[r,"f^\flat"]&{\underline{f}}_*\cM_X\arrow[d,"{{\underline{f}}_*\alpha_X}"]\\
\cO_Y\arrow[r,"{\underline{f}}^\sharp"]&{\underline{f}}_*\cO_X
\end{tikzcd}
\]
of sheaves of monoids on $X$. 
The vertical arrows are injective by assumption. 
Thus there are at most one possible $f^\flat$ if $\underline{f}=g$.
\end{proof}

\begin{lem}
\label{A.9.9}
Let $\theta:P\rightarrow Q$ be an exact homomorphism of integral monoids. 
If $\theta^{\rm gp}$ is an isomorphism, then $\theta$ is an isomorphism.
\end{lem}
\begin{proof}
Since $\theta$ is exact, the induced monoid homomorphism $P\rightarrow P^{\rm gp}\times_{Q^{\rm gp}}Q$ is an isomorphism. 
Thus $\theta$ is an isomorphism since $\theta^{\rm gp}$ is an isomorphism.
\end{proof}

\begin{lem}
\label{A.9.4}
Suppose that $X$ is a solid fs log scheme. 
Then the log structure homomorphism $\alpha_X:\cM_X\rightarrow \cO_X$ is injective.
\end{lem}
\begin{proof}
Assume for contradiction that $\alpha_X$ is not injective. 
Then there exists a point $x$ of $X$ such that $\alpha_{X,x}:\cM_{X,x}\rightarrow \cO_{X,x}$ is not injective.
We choose different elements $p$ and $p'$ of $\cM_{X,x}$ such that $\alpha_{X,x}(p)=\alpha_{X,x}(p')$. 
Let $P$ be the submonoid of $\cM_{X,x}^{\rm gp}$ generated by $\cM_{X,x}$, $p-p'$, and $p'-p$. 
Note that $\alpha_{X,x}$ factors through $P$.

Since $X$ is solid, 
the induced map $\Spec {\cO_{X,x}}\rightarrow \Spec {\cM_{X,x}}$ is surjective. 
Thus the induced map $\Spec P\rightarrow \Spec {\cM_{X,x}}$ is surjective. 
Then by \cite[Proposition I.4.2.2]{Ogu}, the inclusion $\cM_{X,x}\rightarrow P$ is exact, 
which is an isomorphism by Lemma \ref{A.9.9}. 
This contradicts the assumption that $p$ and $p'$ are different elements.
\end{proof}

\begin{lem}\label{A.9.12}
Let $K$ be a field of characteristic $p>0$. 
If an element $x$ of the multiplicative group of units $K^*$ has order $r$, then $p$ does not divide $r$.
\end{lem}
\begin{proof}
If $r=pr'$, then $(x^{r'}-1)^p=0$ and hence $x^{r'}=1$. 
This contradicts the assumption that $x$ has order $r$.
\end{proof}

\begin{lem}\label{A.9.3}
Let $X$ be an fs log scheme log smooth over $k$. Then $X$ is solid.
\end{lem}
\begin{proof}
Since $\overline{\cM}_{X,x}^\gp$ is torsion free for every point $x\in X$ by \cite[Proposition I.1.3.5(2)]{Ogu}, the torsion subgroups of $\cO_{X,x}^*$ and $\cM_{X,x}^\gp$ are isomorphic.
The question is Zariski local on $X$, 
so we may assume that $X$ has a neat fs chart $f:X\rightarrow \A_P$ at a point $x\in X$ such that $f$ is smooth by \cite[Theorem IV.3.3.1]{Ogu} and Lemma \ref{A.9.12}. 
Then $f$ is an open dominant morphism. 
Let $F$ be a face of $P$, 
and let $Y$ be the corresponding strict closed subscheme of $\A_P$. 
Then $f(X)\cap Y$ is a dense open subset of $Y$, so the image of $f$ contains the generic point $y$ of $Y$.

The prime ideal $P-F$ of $P$ is equal to $g(y)$ for the induced map $g:\A_P\rightarrow \Spec P$. 
Since $y$ is also in the image of $f$, $P-F$ is in the image of the composite $gf$. 
It follows that $gf$ is surjective.
\end{proof}

\begin{lem}
\label{A.9.18}
Let $X$ be an fs log scheme log smooth over $k$. Then $X$ is normal.
\end{lem}
\begin{proof}
The question is Zariski local on $X$, so we may assume that $X$ has an fs chart $f:X\rightarrow \A_P$ such that $f$ is \'etale by \cite[Theorem IV.3.3.1]{Ogu}. 
Since $\A_P$ is normal by \cite[Proposition I.3.4.1(2)]{Ogu}, $X$ is normal due to \cite[IV.11.3.13]{EGA}.
\end{proof}

\begin{lem}
\label{A.9.17}
Let $f:X\rightarrow Y$ be a strict finite and surjective morphism of fs log schemes. If $Y$ is log smooth over $k$, then $X$ is solid.
\end{lem}
\begin{proof}
According to Lemmas \ref{A.9.2} and \ref{A.9.3}, it suffices to show that $f$ is an open morphism. 
Note that $f$ is equidimensional by assumption.
Since $Y$ is normal by Lemma \ref{A.9.18}, $f$ is universally open by Chevalley's criterion (\cite[IV.14.4.4]{EGA}).
\end{proof}

\begin{lem}\label{A.9.14}
Let $X$ be an fs log scheme with a neat fs chart $P\rightarrow \cO_X$ at a point $x\in X$. 
Then $\cM_{X,x}\cong \cO_{X,x}^* \oplus P$.
\end{lem}
\begin{proof}
Consider the exact sequence of monoids
\[
0\rightarrow \cO_{X,x}^*\rightarrow \cM_{X,x}\rightarrow \overline{\cM}_{X,x}\rightarrow 0.
\]
By assumption, $\overline{\cM}_{X,x}\cong P$, and the homomorphism $\cM_{X,x}\rightarrow \overline{\cM}_{X,x}$ admits a section $P\rightarrow \cM_{X,x}$.
\end{proof}

\begin{lem}
\label{A.9.11}
Let $f:X\rightarrow Y$ be finite dominant morphism of integral schemes, and let $\alpha_Y:\cM_Y\rightarrow \cO_Y$ be an fs log structure on $Y$. 
Assume that $\alpha_Y$ is injective. 
If $Y$ is normal, then the unit of the adjunction
\[
\cM_Y\rightarrow f_*^{log}f_{log}^*\cM_Y
\]
is an isomorphism.
\end{lem}
\begin{proof}
The question is Zariski local on $Y$, so we may assume that $Y$ is a local affine scheme $\Spec B$ and that $\cM_Y(Y)\cong B^*\oplus P$ by Lemma \ref{A.9.14}.
Since $f$ is a dominant affine morphism, $\underline{f}=\Spec g$ for some injective homomorphism $g:B\rightarrow A$ of rings. 
We only need to show that the induced commutative diagram
\begin{equation}\label{A.9.11.1}\begin{tikzcd}
P\oplus B^*\arrow[r]\arrow[d]&B\arrow[d,"g"]\\
P\oplus A^*\arrow[r]&A
\end{tikzcd}\end{equation}
of monoids is cartesian.
\vspace{0.1in}

To that end, let $(p,a)$ be an element of $P\oplus A^*$ such that $pa\in B$. 
Since $\alpha_Y$ is injective we see that $p\neq 0$ in $B$. 
Thus $a,a^{-1}\in {\rm Frac}(B)$, and hence $a,a^{-1}\in B$ since $B$ is normal and $f$ is finite. 
We conclude that $a\in B^*$ and the assertion follows.
\end{proof}

\begin{exm}
\label{A.9.16}
Let us give an example showing that the conclusion in Lemma \ref{A.9.11} is false when $Y$ is non-normal. 
We set $X=\Spec A$, $Y=\Spec B$, where 
\[
B:=k[x,y]/(y^2-x^3)_{(x-1)} \text{ and } A:=k[t]_{(t^2-1)}.
\]
Set $f=\Spec g$, where $g:B\rightarrow A$ is the inclusion mapping $x$ to $t^2$ and $y$ to $t^3$. 

Let $\alpha:\N\oplus B^*\rightarrow B$ be the log structure mapping $(n,f)$ to $x^nf$.
Then $x(1+t)$ is in the image of the induced homomorphism $\N\oplus A^*\rightarrow A$. 
But since $x(1+t)=x+y\in B$ and $1+t\notin B^*$, the diagram \eqref{A.9.11.1} of monoids is not cartesian.
\end{exm}

\subsection{Compositions of finite log correspondences}
Using the observations in Section \ref{ss:sls} we are now ready to define the composition of finite log correspondences. 
We begin with the following Lemma.

\begin{lem}\label{A.5.9}
Let $X$ and $Y$ be fs log schemes log smooth over $k$, and suppose that $\underline{Z}$ is a closed subscheme of $\underline{X}\times \underline{Y}$. 
Then there exists at most one elementary log correspondence from $X$ to $Y$ whose underlying scheme is $\underline{Z}$. 
\end{lem}
\begin{proof}
Let $Z^N$ be the fs log scheme whose underlying scheme is the normalization $\underline{Z}^N$ of $\underline{Z}$.
We equip it with the log structure induced from $X$ via the canonical map $\underline{Z}^N\to \underline{Z}\to \underline{X}$, 
as in Definition \ref{A.5.2} (ii).  
Let $q:\underline{Z^N}\rightarrow \underline{Y}$ be the induced morphism of schemes. 
We need to show that there is at most one morphism $r\colon Z^N\rightarrow Y$ of fs log schemes such that $\underline{r}=q$. 
Lemma \ref{A.9.3} shows that $Y$ is solid. 
Since $Z^N$ is finite over $X$, we see that $Z^N$ is solid on account of Lemma \ref{A.9.17}. 
To conclude we apply Lemmas \ref{A.9.6} and \ref{A.9.4}.
\end{proof}

\begin{lem}
\label{A.5.10}
Let $X$ and $Y$ be fs log schemes log smooth over $k$. 
Then the homomorphism
\[
(-)^\circ
\colon 
\lCor(X,Y)
\rightarrow \
Cor(X-\partial X,Y-\partial Y)
\]
of abelian groups mapping $V\in \lCor(X,Y)$ to $V^\circ :=V-\partial V$ is injective.
\end{lem}
\begin{proof}
We may assume $X$ and $Y$ are integral. 
Let 
\[
V
=
n_1V_1+\cdots+n_rV_r\in \lCor(X,Y) 
\]
be a finite log correspondence with $r>0$ where $n_i\neq 0$ for every $i$ and the $V_i$'s are distinct. 
If $V-\partial V=0$, 
then $V_i-\partial V_i=V_j-\partial V_j$ for some $i\neq j$. 
Since the closures of $V_i-\partial V_i$ and $V_j-\partial V_j$ in $\underline{X}\times \underline{Y}$ are equal to $\underline{V_i}$ and $\underline{V_j}$ respectively, 
we have that $\underline{V_i}=\underline{V_j}$. 
By uniqueness, 
Lemma \ref{A.5.9} implies $V_i=V_j$.
Since this contradicts our assumption, we conclude that $V-\partial V\neq 0$.
\end{proof}

\begin{lem}
\label{lem:composition} For every $X,Y,Z\in lSm/k$ and finite log correspondences\index{correspondence!composition} $\alpha\in \lCor(X,Y)$, $\beta\in \lCor(Y,Z)$, 
there exists a well-defined composition $\beta\circ \alpha\in\lCor(X,Z)$. 
The assignment 
$$
(\alpha,\beta)\mapsto \beta\circ \alpha
$$ 
is associative. 
Under the morphism $(-)^\circ$ of Lemma \ref{A.5.10}, 
the composite $\beta\circ \alpha$ is sent to the corresponding composite of finite correspondences.
\end{lem}
\begin{proof}
We may assume $\alpha=[V]$ and $\beta=[W]$ are elementary log correspondences. 
In that case, we can form the finite correspondences
\[
V-\partial V
\in 
\Cor(X-\partial X,Y-\partial Y) \text{ and } W-\partial W\in \Cor(Y-\partial Y,Z-\partial Z).
\]
The intersection product
\[
(V-\partial V\times Z)\cdot (X\times W-\partial W)
\]
can be expressed as a formal sum
\[
n_1U_1+\cdots+n_r U_r
\]
of reduced cycles in $X-\partial X\times Y-\partial Y\times Z-\partial Z$.
Let $\underline{T_i}$ be the closure of $U_i$ in $\underline{X}\times \underline{Y}\times \underline{Z}$. 
By Remark \ref{A.5.29} (iv), 
the irreducible components of 
\[
(\underline{V}\times \underline{Z})\cap (\underline{X}\times \underline{W})
\]
are precisely the closures of the irreducible components of 
\[
(V-\partial V \times Z- \partial Z) \cap (X-\partial X \times W-\partial W).
\]

We claim that $\underline{T_i}$ is finite and surjective over a component of $\underline{X}$. 
To that end, we may assume $X$, $Y$, and $Z$ are irreducible. 
Since $Y$ is an fs log scheme, log smooth over $k$, $\underline{Y}$ is normal by Lemma \ref{A.9.18}.   
Thus by \cite[Lemma 1.6]{MVW}, $\underline{T_i}$ is finite and surjective over $\underline{V}$, and so $\underline{T_i}$ is finite and surjective over $\underline{X}$.  
In particular, 
we can consider
\[
[\underline{W}]\circ [\underline{V}] = \sum_{i=1}^r n_i p_*[\underline{T_i}],
\]
where $p\colon X\times Y\times Z\to X\times Z$ is the projection, and 
\[
p_*[\underline{T_i}] = [k(\underline{T_i}): k(p(\underline{T_i}))] \cdot p(\underline{T_i}) =: d_i \cdot p(\underline{T_i}) 
\]
is the corresponding push forward. 
Note that $p(\underline{T_i})$ is finite and surjective over a component of $\underline{X}$ by \cite[Lemma 1.4]{MVW}.  
It remains to verify that $p(\underline{T_i})$ is admissible in the sense of Definition \ref{A.5.2}, 
which is a condition on the log structure.
\vspace{0.1in}

To simplify our notation, we write $\underline{R_i}$ for the image $p(\underline{T_i})$ in $\underline{X}\times\underline{Z}$. 
Let $R_i^N$ denote the fs log scheme whose underlying scheme is the normalization of $\underline{R_i}$ and whose log structure is induced by $X$,  and let $\ul{T_i^N}$ denote the nomalization of $\ul{T_i}$.
There is an induced commutative diagram of schemes
\[
\begin{tikzcd}
&\ul{T_i^N}\arrow[ld,"r_1"']\arrow[d,"r_2"]\arrow[rd,"r_3"]\\
\underline{V^N}\arrow[d,"p_1"']\arrow[rd,"q_1"',very near start]&\underline{R_i^N}\arrow[ld,"p_2"',very near start,crossing over]\arrow[rd,"q_2", very near start] 
&\underline{W^N}\arrow[ld,"p_3",crossing over, very near start]\arrow[d,"q_3"]\\
\underline{X}&\underline{Y}&\underline{Z}.
\end{tikzcd}
\]
By Remark \ref{A.5.29} (i), part (ii) in Definition \ref{A.5.2} of elementary log correspondences is tantamount to the existence of structure morphisms
\[
\alpha_{12}
:
(q_1)_{log}^*\cM_Y
\rightarrow 
(p_1)_{log}^*\cM_X
\text{ and }
\alpha_{23}:(q_3)_{log}^*\cM_Z
\rightarrow 
(p_3)_{log}^*\cM_Y.
\]
The composite morphism $\alpha_{123}$ given by 
\[(q_3r_3)_{log}^*\cM_Z
\stackrel{\alpha_{23}}\longrightarrow 
(p_3r_3)_{log}^*\cM_Y
\stackrel{\sim}\longrightarrow 
(q_1r_1)_{log}^*\cM_Y
\stackrel{\alpha_{12}}\longrightarrow 
(p_1r_1)_{log}^*\cM_X\]
can be rewritten as
\[
\alpha_{123} 
\colon 
(q_2r_2)_{log}^* \cM_Z
\to 
(p_2 r_2)_{log}^* \cM_X.
\]
Since $\underline{R^N}$ is finite and surjective over a component of $\underline{X}$, 
and $X$ is log smooth over $k$ by assumption, 
$R^N$ is solid by Lemma \ref{A.9.17}. 
Thus by Lemma \ref{A.9.4}, the log structure map $\alpha_{R_i^N}$ is injective.
Since $\underline{R_i}^N$ is normal, we can apply Lemma \ref{A.9.11} to deduce there exists a unique morphism
\[
\alpha_{13}
\colon 
(q_2)_{log}^*\cM_Z
\rightarrow 
(p_2)_{log}^*\cM_X
\]
of log structures on $\underline{R_i^N}$ such that $(r_2)_{log}^*\alpha_{13}=\alpha_{123}$. 
Thus the pair $(R_i^N,\alpha_{13})$ forms an elementary log correspondence $R_i$ from $X$ to $Z$.
\vspace{0.1in}
    
Owing to the above, we obtain the finite log correspondence
\[
[W]\circ [V]
:=
\sum_{i=1}^r n_i d_i[R_i] 
\in 
\lCor(X,Z).
\]
By construction, there is a commutative diagram
\begin{equation}
\label{equation1:lem:composition}
\begin{tikzcd}
\lCor(Y,Z)\times \lCor(X,Y)\arrow[d]\arrow[r,"(-)\circ(-)"]&\lCor(X,Z)\arrow[d]\\
\Cor(Y-\partial Y,Z-\partial Z)\times \Cor(X-\partial X,Y-\partial Y)\arrow[r,"(-)\circ(-)"]&\Cor(X-\partial X,Z-\partial Z)
\end{tikzcd}
\end{equation}
where the vertical morphisms are induced by the homomorphism $(-)^\circ$ in Lemma \ref{A.5.10}. 
Since the vertical morphisms are injective by Lemma \ref{A.5.10},
the associative law for the composition follows from the case of finite correspondences (for references see \cite[Lecture 1]{MVW}).
\end{proof}

For every $X\in lSm/k$, the \emph{identity} ${\rm id}_X\in \lCor(X,Y)$  is the finite log correspondence associated with the diagonal embedding $X\rightarrow X\times X$.

\begin{lem}
For every $X,Y\in lSm/k$ and finite log correspondence $\alpha\in \lCor(X,Y)$ we have that
\[
{\rm id}_Y \circ \alpha
=
\alpha
=
\alpha \circ {\rm id}_X.
\]
\end{lem}
\begin{proof}
Owing to Lemma \ref{A.5.10} we reduce to the case when $X=X-\partial X$ and $Y=Y-\partial Y$, i.e., $X$ and $Y$ are usual schemes, and we can appeal to \cite[Lecture 1]{MVW}.
\end{proof}
  
\begin{df}
\label{A.5.3}\index[notation]{lCor @ $\LCor$}
Let $\LCor$ denote the category with objects $lSm/k$ and morphisms finite log correspondences between fs log schemes log smooth over $k$.
Composition of morphisms are defined in Lemma \ref{lem:composition}. 
It is an additive category with zero object $\emptyset$ (the empty log scheme) and coproduct $\amalg$ (disjoint union).
\end{df}

\begin{rmk}
\label{A.5.35}
\begin{enumerate}
\item[(i)] For $X,Y,Z\in SmlSm/k$, the composition 
\[
\Cor(Y,X)\times \Cor(Z,Y)\rightarrow \Cor(X,Z)
\] 
in \cite[\S 1]{MVW} coincides with the above composition.
  
\item[(ii)] If $f\colon X\to Y$ is a morphism in $lSm/k$, its graph $\Gamma_f$ defines an element of $\lCor(X,Y)$. 
Note that $\underline{X}$ is normal by Lemma \ref{A.9.18}, and $\underline{\Gamma_f}$ is isomorphic to $\underline{X}$ via the projection morphism.

\item[(iii)] If $W\in \lCor(Y,Z)$ is an elementary log correspondence, then $W\circ \Gamma_f$ is obtained by the compositions
\[
\underline{X}\times_{\underline{Y}}\underline{W}\rightarrow \underline{X}\times_{\underline{Y}}(\underline{Y}\times \underline{Z})
\cong 
\underline{X}\times \underline{Z},\; (X\times_Y W^N)^N\rightarrow X\times_Y W^N\rightarrow  X\times Z.
\]

\item[(iv)] If $V\in \lCor(U,X)$ is an elementary log correspondence, then $\Gamma_f\circ V$ is obtained by the compositions
\[
\underline{V}\rightarrow \underline{U}\times \underline{X}\rightarrow \underline{U}\times \underline{Y},\;
V^N\rightarrow U\times X\rightarrow U\times Y.
\]
\end{enumerate}
\end{rmk}

\begin{df}\index[notation]{underl @ $\ul{X}$}
\label{A.5.16}
For $X,Y\in SmlSm/k$, there is a canonical homomorphism
\[
\underline{(-)}:\lCor(X,Y)\rightarrow \Cor(\underline{X},\underline{Y})
\]
mapping any elementary log correspondence $V$ to the elementary correspondence $\underline{V}$. 
This homomorphism is well defined since the underlying closed subscheme $\ul{V}$ of an elementary log correspondence $V$ is finite and surjective over (a component of) $\ul{X}$ by Definition \ref{A.5.2}(i).
\end{df}

\begin{lem}
\label{A.5.17}
For $X,Y,Z\in SmlSm/k$, there is a commutative diagram
\[
\begin{tikzcd}
\lCor(Y,Z)\times \lCor(X,Y)\arrow[r,"(-)\circ (-)"]\arrow[d,"\underline{(-)}\circ \underline{(-)}"']& \lCor(X,Z)\arrow[d,"\underline{(-)}"]\\
\Cor(\underline{Y},\underline{Z})\times \Cor(\underline{X},\underline{Y})\arrow[r,"(-)\circ (-)"]&\Cor(\underline{X},\underline{Z}).
\end{tikzcd}
\]
\end{lem}
\begin{proof}
Let $u:X-\partial X\rightarrow X$ and $v:Y-\partial Y\rightarrow Y$ be the structure open immersion.
Since $\Cor(-,\underline{Z})$ is an \'etale sheaf by \cite[Lemma 6.2]{MVW},
\[
(-)\circ \Gamma_{\underline{u}}:\Cor(\underline{X},\underline{Z})\rightarrow \Cor(X-\partial X,\underline{Z})
\]
is injective.
Hence it suffices to check the commutativity of
\[
\begin{tikzcd}
\lCor(Y,Z)\times \lCor(X,Y)\arrow[r,"(-)\circ (-)"]\arrow[d,"\underline{(-)}\times (\underline{(-)}\circ \Gamma_{\underline{u}})"']& \lCor(X,Z)\arrow[d,"\underline{(-)}\circ \Gamma_{\underline{u}}"]\\
\Cor(\underline{Y},\underline{Z})\times \Cor(X-\partial X,\underline{Y})\arrow[r,"(-)\circ (-)"]&\Cor(X-\partial X,\underline{Z}).
\end{tikzcd}
\]

Due to the commutativity of
\[
\begin{tikzcd}[column sep=large]
\lCor(Y,Z)\times \lCor(X,Y)\arrow[d,"\underline{(-)}\times (\underline{(-)}\circ \Gamma_{\underline{u}})"']\arrow[r,"(\underline{(-)}\circ \Gamma_{\underline{v}})\times (-)^\circ"]
&\Cor(Y-\partial Y,\underline{Z})\times \Cor(X-\partial X,Y-\partial Y)\arrow[d,"(-)\circ (-)"]\\
\Cor(\underline{Y},\underline{Z})\times \Cor(X-\partial X,\underline{Y})\arrow[r,"(-)\circ (-)"]&\Cor(X-\partial X,\underline{Z}),
\end{tikzcd}
\]
it suffices to check the commutativity of
\[
\begin{tikzcd}
\lCor(Y,Z)\times \lCor(X,Y)\arrow[r,"(-)\circ (-)"]\arrow[d,"(\underline{(-)}\circ \Gamma_{\underline{v}})\times (-)^\circ"']& \lCor(X,Z)\arrow[d,"\underline{(-)}\circ \Gamma_{\underline{u}}"]\\
\Cor(Y-\partial Y,\underline{Z})\times \Cor(X-\partial X,Y-\partial Y)\arrow[r,"(-)\circ (-)"]&\Cor(X-\partial X,\underline{Z}).
\end{tikzcd}
\]
Since the left (resp.\ right) vertical morphism factors through
\[
\Cor(Y-\partial Y,Z-\partial Z)\times \Cor(X-\partial X,Y-\partial Y)
\textrm{ (resp.\ }\Cor(X-\partial X,Z-\partial Z)),
\]
we are done by the commutativity of \eqref{equation1:lem:composition}.
\end{proof}

We conclude this section by defining the monoidal structure on $lCor/k$. 
\begin{df} \index[notation]{XtimesY @ $X\otimes Y$}
\label{A.5.36}
The tensor product of $X,Y\in\lCor/k$ is given by the fiber product of fs log schemes
\[
X\otimes Y 
:= 
X\times_k Y.
\]
If $V$ is an elementary log correspondence from $X$ to $X'$ and $W$ is an elementary log correspondence from $Y$ to $Y'$, 
then $V\times_k W$ defines an elementary log correspondence from $X\otimes Y$ to $X'\otimes Y'$.
The tensor product defines a symmetric monoidal structure on $lCor/k$.
\end{df}
\newpage

\section{Topologies on fine saturated logarithmic schemes}\label{sec:topologies}

We want our construction of motives for fs log schemes to depend on intrinsic local data for logarithmic structures. 
To achieve this, we need to provide an appropriate Grothendieck topology on our ground category categories $lSm/S$ and $lSch/S$, 
where throughout this section, $S$ is a finite-dimensional noetherian base fs log scheme.
To establish the basic desiderata for log motives, we are basically forced to consider a suitable version of the Nisnevich topology for logarithmic schemes called the \emph{strict Nisnevich topology}. 
We also introduce a new topology in $lSch/S$, which we call the \emph{dividing topology}.
The coverings in the dividing topology include log modifications in the sense of  F.\ Kato \cite{FKato}.
The dividing topology arises from a cd-structure in the sense of Voevodsky \cite{Vcdtop} called the \emph{dividing cd-structure}. 
\vspace{0.1in}

The general theory of cd-structures was developed by Voevodsky \cite{Vcdtop}. Suppose that $S$ is a scheme.
Primary examples include the Zariski and Nisnevich topologies on $Sm/S$, the cdh-topology on $Sch/S$, and their generalizations to algebraic stacks. 
To prove basic properties about the induced topology, it suffices to check whether a given cd-structure is bounded, complete, and regular. 
Unfortunately, the dividing cd-structure does not satisfy the boundedness condition, 
but instead, we consider the weaker condition of being \emph{quasi-bounded}. 
For this reason, we revisit Voevodsky's work in \cite{Vcdtop} to ensure that the set-up applies to fs log schemes.

\subsection{cd-structures on fs log schemes}
\label{sec::cdlog}
Recall the following definition from \cite{Vcdtop}.
\begin{df}[{\cite[Definition 2.1]{Vcdtop}}]\label{A.8.25}
A {\it cd-structure} $P$ on a category \index{cd-structure} $\cC$ with an initial object is a collection of commutative squares in $\cC$ 
\begin{equation}
\label{A.8.13.1}
Q
=
\begin{tikzcd}
Y'\arrow[d,"f'"']\arrow[r,"g'"]&Y\arrow[d,"f"]\\
X'\arrow[r,"g"]&X
\end{tikzcd}
\end{equation}
that are stable under isomorphisms. 
The squares in the collection $P$ are called the \emph{distinguished squares of $P$}. 
The topology $t_P$ \index[notation]{tp @ $t_P$}associated with $P$ is the smallest Grothendieck topology such that $Y\amalg X'$ is a covering sieve of $X$ for every distinguished square $Q$ in $P$, 
and the empty family is a covering sieve of the initial object. 
\end{df}

\begin{df}[{\cite[Definitions 2.2, 2.3]{Vcdtop}}]
Let $P$ be a cd-structure on a category $\cC$ with an initial object. \index{simple cover}
The class $S_P$ of {\it simple covers}\index{simple cover} is the smallest class of families of morphisms $\{U_i\rightarrow X\}_{i\in I}$ satisfying the following two conditions.
\begin{enumerate}
\item[(i)] Any isomorphism is in $S_P$.
\item[(ii)] For any distinguished square of the form \eqref{A.8.13.1}, if $\{p_i':X_i'\rightarrow X'\}_{i\in I}$ and $\{q_j:Y_j\rightarrow Y\}_{j\in J}$ are in $S_P$, then $\{g\circ p_i'\}_{i\in I}\cup \{f\circ q_j\}_{j\in J}$ is in $S_P$.
\end{enumerate}
We say that $P$ is {\it complete} if any covering sieve admits a refinement by a simple cover.
\end{df}\index{cd-structure!complete}

\begin{df}[{\cite[Definitions 2.10]{Vcdtop}}]
Let $P$ be a cd-structure on a category $\cC$ with an initial object. 
Let $L(T)$ denote the $t_P$-sheaf \index{tp @ $t_P$-sheaf}associated with the presheaf represented by $T\in \cC$.
We say that $P$ is {\it regular} \index{cd-structure!regular}if the following conditions are satisfied.
\begin{enumerate}
\item[(i)] For any distinguished square $Q$ of the form \eqref{A.8.13.1}, $Q$ is cartesian, $g$ is a monomorphism.
\item[(ii)] 
The morphism of $t_P$-sheaves 
\[
\Delta\amalg (L(g')\times_{L(g)}L(g')):
L(Y)\amalg (L(Y')\times_{L(X')}L(Y'))
\rightarrow 
L(Y)\times_{L(X)}L(Y)
\]
is an epimorphism. 
\end{enumerate} 
\end{df}

\begin{df}[\cite{Par}]
\label{A.5.14}
A cartesian square of fs log schemes
\begin{equation}
\label{A.5.14.1}
Q
=
\begin{tikzcd}
Y'\arrow[d,"f'"']\arrow[r,"g'"]&Y\arrow[d,"f"]\\
X'\arrow[r,"g"]&X
\end{tikzcd}
\end{equation}
is called a 
\begin{enumerate}
\item[(i)] \emph{Zariski distinguished square} \index{distinguished square!Zariski}if $f$ and $g$ are open immersions.
\item[(ii)] {\it strict Nisnevich distinguished square } \index{distinguished square!Nisnevich} if $f$ is strict \'etale, $g$ is an open immersion,
$g(X')\cup f(Y)=X$,
and $f$ induces an isomorphism 
$$
f^{-1}(\underline{X}-g(\underline{X'}))\stackrel{\sim}\rightarrow \underline{X}-g(\underline{X'})
$$ 
with respect to the reduced scheme structures.
\item[(iii)] {\it dividing distinguished square}\index{distinguished square!dividing} if $Y'=X'=\emptyset$ and $f$ is a surjective proper log \'etale monomorphism.
See \S \ref{subsec::logsmooth} for properties of log \'etale morphisms and \S \ref{subsec::mono} for properties of log \'etale monomorphisms. \index{morphism of log schemes!log \'etale}\index{log modification}
According to Proposition \ref{A.9.75}, this condition is equivalent to asking that $f$ be a log modification in the sense of F.\ Kato.
\end{enumerate}

Associated to the collection of Zariski distinguished squares (resp.\ strict Nisnevich distinguished squares, resp.\ dividing distinguished squares) 
we have the corresponding \emph{Zariski cd-structure} (resp.\ {\it strict Nisnevich cd-structure}, resp.\ the {\it dividing cd-structure}). \index{cd-structure!Zariski} \index{cd-structure!strict Nisnevich} \index{cd-structure!dividing}
These cd-structures give rise to the \emph{Zariski topology}, {\it strict Nisnevich topology}, and {\it dividing topology} on $lSch/S$, 
respectively.
We let $Zar$ (resp.\ $sNis$, resp.\ $div$) be shorthand for the Zariski topology (resp.\ strict Nisnevich topology, resp.\ dividing topology).
\vspace{0.1in}

The \emph{dividing Zariski cd-structure} (resp.\ {\it dividing Nisnevich cd-structure}) \index{cd-structure!dividing Zariski} is the union of the Zariski (resp.\ strict Nisnevich) and the dividing cd-structures. 
We refer to the associated topology as the \emph{dividing Zariski topology} (resp.\ {\it dividing Nisnevich topology}).\index{topology!dividing Nisnevich}
We let $dZar$ (resp. $dNis$) be shorthand for the dividing Zariski topology (resp.\ dividing Nisnevich topology).\index{topology!dividing Zariski}
\end{df}

\begin{df}\label{A.5.65}
The {\it strict \'etale topology} (resp.\ \emph{Kummer \'etale topology}, resp.\ log \'etale topology) on $lSch/S$ is the smallest Grothendieck topology generated by strict \'etale covers \index{topology!strict \'etale}
(resp.\ Kummer \'etale covers, resp.\ log \'etale covers) in $lSch/S$.
The {\it dividing \'etale topology} on $lSch/S$ \index{topology!dividing \'etale}\index{topology!Kummer \'etale} \index{topology!log \'etale}is the smallest Grothendieck topology finer than the strict \'etale topology and the dividing topology.
We let $s\acute{e}t$ (resp.\ $k\acute{e}t$, $l\acute{e}t$, $d\acute{e}t$) be shorthand for the strict \'etale topology 
(resp.\ Kummer \'etale topology, resp.\ log \'etale topology, resp.\ dividing \'etale topology).
\end{df}

We note that none of these \'etale topologies are associated with a cd-structure.
The log \'etale topology is the smallest Grothendieck topology that is finer than the Kummer \'etale topology and the dividing topology, 
see also \cite[\S 9.1]{MR1922832} whose proof appears in \cite[Proposition 3.9]{MR3658728}.

\begin{exm}
Every Nisnevich distinguished square of usual schemes is also a strict Nisnevich distinguished square.
If the square $Q$ in \eqref{A.5.14.1} is a strict Nisnevich distinguished square, then the square
\[
\underline{Q}
=
\begin{tikzcd}
\underline{Y'}\arrow[d,"\underline{f'}"']\arrow[r,"\underline{g'}"]&\underline{Y}\arrow[d,"\underline{f}"]\\
\underline{X'}\arrow[r,"\underline{g}"]&\underline{X}
\end{tikzcd}
\]
is a Nisnevich distinguished square of usual schemes since the morphisms $f$ and $g$ in \eqref{A.5.14.1} are strict.
We also note that $Q$ is a pullback of $\underline{Q}$, 
i.e.,
\[
Y\cong \underline{Y}\times_{\underline{X}}X,\;
Y'\cong \underline{Y'}\times_{\underline{X}}X,\;
X'\cong \underline{X'}\times_{\underline{X}}X. 
\]
\end{exm}

\begin{rmk}
\label{A.9.68}
\begin{enumerate} 
\item[(1)] Suppose $f:Y\rightarrow X$ is a surjective proper log \'etale monomorphism. 
Theorem \ref{KatoStrThm} shows that Zariski locally on $X$ and $Y$, 
the restriction 
\[
f^{-1}(X-\partial X)\rightarrow X-\partial X
\]
has a chart $\theta:P\rightarrow Q$, 
where $P=0$ is trivial and $Q$ is finite.
Thus $f^{-1}(X-\partial X)$ has a trivial log structure, since $\overline{Q}=0$, 
so that $Y-\partial Y=f^{-1}(X-\partial X)$ and the induced morphism $Y-\partial Y\rightarrow X-\partial X$ is surjective. 
The latter morphism is an isomorphism since surjective proper \'etale monomorphisms of schemes are isomorphisms according to \cite[Th\'eor\`eme IV.17.9.1]{EGA}.
\item[(2)] For cdh-distinguished squares of schemes \cite[Section 2]{Vunst-nis-cdh}, 
the schemes $Y'$ and $X'$ in \eqref{A.5.14.1} can be nonempty. 
Thus the dNis-structure is not a direct analog of the cdh-structure \cite{Vcdtop}.
\item[(3)]  We choose the terminology ``dividing'' because morphisms of fs log schemes induced by subdivisions of fans in toric geometry are surjective proper log \'etale monomorphisms.
For a proof we refer to Proposition \ref{A.9.75} and Example \ref{A.9.76}.
\item[(4)] A presheaf $\cF$ is a dividing sheaf if and only if for every log modification $Y\rightarrow X$ the induced homomorphism $\cF(X)\rightarrow \cF(Y)$ is an isomorphism.
\end{enumerate}
\end{rmk}

\subsection{Cohomology on big and small sites}
If the base scheme $S$ has trivial log structure and $\cF$ is an \'etale sheaf of $\Lambda$-modules on $Sm/S$, 
then there are two equivalent ways of defining the cohomology group $H^i(X,\cF)$ for any integer $i\geq 0$, namely 
\[
\hom_{\mathbf{D}(\Shv_{\acute{e}t}(X_{\acute{e}t},\Lambda))}(\Lambda(X),\cF\vert_{X_{\acute{e}t}}[i])
\text{ and }
\hom_{\mathbf{D}(\Shv_{\acute{e}t}(Sm/S,\Lambda))}(\Lambda(X),\cF[i]).
\]
In this section, we generalize the above comparison to the topologies we consider on $lSm/k$.
As an application, we shall discuss compact objects in the corresponding categories of sheaves.

\begin{df}
A morphism of fs log schemes $Y\rightarrow X$ is a \emph{dividing Nisnevich} 
(resp.\ \emph{dividing \'etale}) morphism if it is contained in the smallest class of morphisms containing log modifications and strict Nisnevich (resp.\ strict \'etale) \index{morphism of log schemes!dividing Nisnevich}\index{morphism of log schemes!dividing \'etale}
morphisms and closed under compositions.
\end{df}

\begin{df}
\label{bigsmall.1}\index[notation]{XsNis @ $X_{sNis},X_{s\acute{e}t}, X_{k\acute{e}t}, X_{dNis}, X_{d\acute{e}t}, X_{l\acute{e}t}$}
If $X$ is an fs log scheme, let $X_{sNis}$ (resp.\ $X_{s\acute{e}t}$, $X_{k\acute{e}t}$, $X_{dNis}$, $X_{d\acute{e}t}$, $X_{l\acute{e}t}$) denote the full subcategory of $lSch/X$ consisting 
of fs log schemes strict Nisnevich (resp.\ strict \'etale, Kummer \'etale, dividing Nisnevich, dividing \'etale, log \'etale) over $X$.
\end{df}

Note that these classes of morphisms are closed under compositions and pullbacks.
Due to Lemma \ref{A.9.71} for every log \'etale morphism $Z\rightarrow Y$ of fs log schemes, the diagonal morphism $Z\rightarrow Z\times_Y Z$ is an open immersion.
Thus all these categories have finite limits.

\begin{lem}
If $X$ is an fs log scheme, every morphism in  $X_{sNis}$ (resp.\ $X_{s\acute{e}t}$, $X_{k\acute{e}t}$, $X_{dNis}$, $X_{d\acute{e}t}$, $X_{l\acute{e}t}$) is strict Nisnevich 
(resp.\ strict \'etale, Kummer \'etale, dividing Nisnevich, dividing \'etale, log \'etale).
\end{lem}
\begin{proof}
We focus on the dividing Nisnevich case since the proofs are similar.
Let $\mathscr{P}$ denote the class of dividing Nisnevich morphisms.
Suppose that $g:Z\rightarrow Y$ and $f:Y\rightarrow X$ be morphisms of fs log schemes such that $f,fg\in \mathscr{P}$.
We claim that $g\in \mathscr{P}$.
There exists a commutative diagram:
\[
\begin{tikzcd}
Z\arrow[r,"\Gamma"]\arrow[d,"g"']&
Z\times_X Y\arrow[r,"q"]\arrow[d,"g\times {\rm id}"]&
Y\arrow[d,"{\rm id}"]
\\
Y\arrow[r,"\Delta"]&
Y\times_X Y\arrow[r,"p"]&
Y
\end{tikzcd}
\]
Here $p$ and $q$ are the second projections, $\Delta$ is the diagonal morphism, and the left square is cartesian.
Since $f$ is log \'etale, as noted above $\Delta\in \mathscr{P}$.
It follows that $\Gamma,q\in \mathscr{P}$ because $\mathscr{P}$ is closed under pullbacks.
This shows that $g=q\circ \Gamma\in \mathscr{P}$.
\end{proof}

Let $t$ be one of the following topologies on $lSm/S$:
\[
t=sNis,\; s\acute{e}t,\; dNis,\;d\acute{e}t,\;k\acute{e}t,\textrm{ and }\;l\acute{e}t.
\]
We have an inclusion functor
\[
\eta\colon X_t\rightarrow lSm/S
\]
sending $Y\in X_t$ to $Y$.
We regard $lSm/S$ as a site for the topology $t$.
Since $X_t$ has fiber products and the functor $\eta$ commutes with fiber products,
\cite[Proposition III.1.6]{SGA4} shows that $\eta$ is a continuous functor of sites.
If $Y\in X_t$, then every $t$-covering $\{U_i\rightarrow Y\}_{i\in I}$ in $lSm/S$ has a refinement in $X_t$, so $\eta$ is a cocontinuous functor of sites.
Then \cite[Sections IV.4.7, IV.4.9]{SGA4} gives adjoint functors
\begin{equation}
\label{bigsmall.1.1}
\begin{tikzcd}
\Shv_t(X_t,\Lambda)
\arrow[rr,shift left=1.5ex,"\eta_\sharp"]
\arrow[rr,leftarrow,"\eta^*" description]
\arrow[rr,shift right=1.5ex,"\eta_*"']
&&
\Shv_t(lSm/S,\Lambda),
\end{tikzcd}
\end{equation}
where $\eta^*$ is the restriction functor, $\eta_\sharp$ is left adjoint to $\eta^*$, and $\eta_*$ is right adjoint to $\eta^*$.
Note that $\eta^*$ is exact.
\vspace{0.1in}

There are naturally induced adjoint functors
\begin{equation}
\label{bigsmall.1.2}
\eta_\sharp:\mathbf{C}(\Shv_t(X_t,\Lambda))
\rightleftarrows
\mathbf{C}(\Shv_t(lSm/S,\Lambda)):\eta^*.
\end{equation}
We impose the descent model structures on both categories, see Proposition \ref{A.8.15}.
By virtue of Example \ref{A.8.28} an object $\cF$ in $\mathbf{C}(\Shv_t(X_t,\Lambda))$ (resp.\ $\mathbf{C}(\Shv_t(lSm/S,\Lambda))$) is fibrant if for every $t$-hypercover $\mathscr{Y}\rightarrow Y$ in $X_t$ (resp.\ $lSm/S$) there is a natural quasi-isomorphism
\begin{equation}
\label{bigsmall.1.3}
\cF(Y)\xrightarrow{\cong} {\rm Tot}^\pi \cF(\mathscr{Y}).
\end{equation}
Moreover, 
a morphism $\cF\rightarrow \cG$ in $\mathbf{C}(\Shv_t(X_t,\Lambda))$ (resp.\ $\mathbf{C}(\Shv_t(lSm/S,\Lambda))$) is a fibration if it is a degreewise epimorphism and its kernel is fibrant.
If $Y\in X_t$, then every $t$-hypercovering of $Y$ in $lSm/S$ is again a $t$-hypercovering in $X_t$.
Moreover, since $\eta^*$ is exact, $\eta^*$ commutes with ${\rm Tot}^\pi$.
Hence from the description of fibrant objects \eqref{bigsmall.1.3} we deduce that $\eta^*$ preserves fibrant objects.
Use again the fact that $\eta^*$ is exact to deduce that $\eta^*$ preserves fibrations and trivial fibrations.
Thus \eqref{bigsmall.1.2} is a Quillen adjunction.
By passing to homotopy categories we obtain the adjoint functors
\begin{equation}
\label{bigsmall.1.4}
L\eta_\sharp:
\mathbf{D}(\Shv_t(X_t,\Lambda))
\rightleftarrows
\mathbf{D}(\Shv_t(lSm/S,\Lambda))
:R\eta^*.
\end{equation}
Since $\eta^*$ is exact, $R\eta^*\cong \eta^*$.

\begin{prop}
\label{bigsmall.2}
With $t$ is as above, for every complex of $t$-sheaves $\cF$, there is a canonical isomorphism
\[
\hom_{\mathbf{D}(\Shv_{t}(lSm/S,\Lambda))}(a_t^*\Lambda(X),\eta^*\cF)
\xrightarrow{\cong}
\hom_{\mathbf{D}(\Shv_{t}(X_{t},\Lambda))}(a_t^*\Lambda(X),\cF).
\]
\end{prop}
\begin{proof}
Since $a_t^*\Lambda(X)$ is cofibrant, there is an isomorphism
$$
L\eta_\sharp a_t^*\Lambda(X)\cong a_t^*\Lambda(X).
$$
To conclude, we use the adjunction \eqref{bigsmall.1.4}.
\end{proof}

\begin{prop}
\label{bigsmall.3}
With $t$ is as above, suppose that $X$ is a noetherian fs log scheme.
Then the topos $\Shv_t(X_t)$ is coherent.
\end{prop}
\begin{proof}
Since $X_t$ has fiber products, owing to \cite[Proposition VI.2.1]{SGA4}, it suffices to show that every object of $X_t$ is quasi-compact.
The log \'etale topology is finer than $t$, and $X_t$ is a full subcategory of $X_{l\acute{e}t}$.
Hence it remains to show that every object of $X_{l\acute{e}t}$ is quasi-compact.
For this we refer to \cite[Proposition 3.14(1)]{MR3658728}.
\end{proof}

\begin{df}
\label{bigsmall.4}
Let $\cC$ be a category with a topology $\tau$. \index{cohomological dimension}
We say that $\cC$ has \emph{finite $\tau$-cohomological dimension} for $\Lambda$-linear coefficients if for every fixed $X\in \cC$, 
there exists an integer $N$ such that for every $\cF\in \Shv_\tau(\cC,\Lambda)$ and integer $i>N$ we have the vanishing
\[
H_{\tau}^i(X,\cF)=0.
\]
\end{df}

\begin{prop}
\label{bigsmall.5}
With $t$ as above, suppose that $S$ is noetherian and $lSm/S$ has finite $t$-cohomological dimension for $\Lambda$-linear coefficients.
Then for every $X\in lSm/S$, the functor
\[
\hom_{\mathbf{D}(\Shv_t(lSm/S,\Lambda))}(a_t^*\Lambda(X),-)
\colon
\mathbf{D}(\Shv_t(lSm/S,\Lambda))\rightarrow \Lambda\text{-}\mathbf{Mod}
\]
commutes with small sums.
In particular, the category 
$$
\mathbf{D}(\Shv_t(lSm/S,\Lambda)
$$ 
is compactly generated. \index{compact generators!categories of sheaves}
\end{prop}
\begin{proof}
By virtue of Proposition \ref{bigsmall.2}, it suffices to show that the functor
\[
\hom_{\mathbf{D}(\Shv_t(X_t,\Lambda))}(a_t^*\Lambda(X),-)
\colon
\mathbf{D}(\Shv_t(X_t,\Lambda))\rightarrow \Lambda\text{-}\mathbf{Mod}
\]
commutes with small sums.
By assumption $\Shv_t(X_t,\Lambda)$ has finite $t$-cohomological dimension for $\Lambda$-linear coefficients, and $X_t$ is coherent due to Proposition \ref{bigsmall.3}.
Then apply \cite[Proposition 1.9]{CDEtale} to conclude.
\end{proof}

\subsection{The dividing density structure} \index{density structure}
In addition to completeness and regularity, 
Voevodsky \cite{Vcdtop} also introduced the notion of a bounded cd-structure.
In tandem, 
these properties ensure the associated topology has a finite cohomological dimension.
The standard density structure, 
introduced in \cite{Vunst-nis-cdh}, 
is used to demonstrate that certain cd-structures are bounded, including the Nisnevich cd-structure.
\vspace{0.1in}

We need to deal with certain non-bounded cd-structures for fs log schemes due to the abundance of proper log \'etale monomorphisms (such morphisms are typically not isomorphisms).
In effect, 
we introduce the notion of a quasi-bounded cd-structure, 
see Definition \ref{A.9.65}, 
and modify the standard density structure so that it fits our situation, 
see Definition \ref{A.9.43}.
Moreover, 
we extend B.G.\ functors to simplicial objects in order to deduce a finite cohomological dimension result.

\begin{df}[{\cite[Definition 2.20]{Vcdtop}}]  
\label{A.9.42}
A density structure on a category $\cC$ with an initial object is the data of $D_d(X)$ for $d\in \N$ and $X\in \cC$, 
where $D_d(X)$ is a family of morphisms with codomain $X$ satisfying the following properties.
\begin{enumerate}
\item[(i)] $(\emptyset\rightarrow X)\in D_0(X)$.
\item[(ii)] $D_d(X)$ contains all isomorphisms.
\item[(iii)] $D_{d+1}(X)\subset D_{d}(X)$.
\item[(iv)] If $v:V\rightarrow U$ is in $D_d(U)$ and $u:U\rightarrow X$ is in $D_d(X)$, then $u\circ v$ is in $D_d(X)$.
\end{enumerate}
We often write $U\in D_d(X)$ instead of $(U\rightarrow X)\in D_d(X)$ for short.
\end{df}

\begin{df}\index{density structure!locally of finite dimension}
\label{A.9.62}
A density structure $D_*$ on $\cC$ is {\it locally of finite dimension with respect to a topology $t$} on $\cC$ if for any $X\in \cC$, 
there exists a natural number $n\in \N$ such that for any $f:Y\rightarrow X$ in $D_{n+1}(X)$, there exists a finite $t$-covering sieve $\{f_i:Y_i\rightarrow X\}_{i\in I}$ such that $f=\amalg_{i\in I}f_i$. 
The smallest such $n$ is called the {\it dimension}\index{density structure!dimension} of $X$ with respect to $D_*$.
\end{df}
\begin{rmk}
\label{A.9.63}
A density structure is locally of finite dimension with respect to the trivial topology if and only if it is locally of finite dimension in the sense of \cite[Definition 2.20]{Vcdtop}. 
Indeed, according to \cite[Definition 2.20]{Vcdtop}, a density structure is locally of finite dimension if for any $X$, there exists $n\in \N$ such that any element of $D_{n+1}(X)$ is an isomorphism. 
\end{rmk}

In the following definition, we slightly modify the standard density structure on a finite-dimensional scheme introduced by Voevodsky \cite[p.\ 1401]{Vunst-nis-cdh}. 
Recall that distinct points $x_0,\ldots,x_d$ in $X$ form an {\it increasing sequence} if $x_{i+1}$ is a generalization of $x_i$ for every $i$,
i.e., 
the point $x_i$ is in the closure of ${\{x_{i+1}\}}$.

\begin{df}
\label{A.9.43}
Let $D_d(X)$ denote the family of finite disjoint unions 
\[
U=\amalg_i U_i\rightarrow X
\] 
of open immersions such that for any point $x\in X$ that is not in the image, there is an increasing sequence $x=x_0,\ldots,x_d$ of length $d$ in $X$. 
Note that $D_*$ is a density structure. \index[notation]{Dd @ $D_d(-)$}
\vspace{0.1in}

For a finite-dimensional fs log scheme $X$, as in \cite{Par}, we set
\[
D_d(X):=D_d(\underline{X}).
\]
\end{df}

\begin{rmk}
\label{A.9.53}
In \cite[Section 2, p. 1401]{Vunst-nis-cdh}, only open immersions (and not finite disjoint unions of open immersions) are considered in the standard density structure. 
Our slight weakening of the original conditions allows us to relate the dividing density structure defined in Definition \ref{A.9.57} to the standard density structure.
\end{rmk}

\begin{lem}
\label{A.9.51}
Let $X$ be a finite dimensional scheme, and let $d\in \N$.
\begin{enumerate}
\item[{\rm (1)}] If $U\in D_d(X)$ and $V$ is an open subscheme of $X$, then $U\times_X V\in D_d(V)$.
\item[{\rm (2)}] Let $U\rightarrow X$ be a finite disjoint union of open immersions mapping to $X$, and let $\{V_i\rightarrow X\}_{i\in I}$ be an open cover of $X$. 
Then $U\times_X V_i\in D_d(V_i)$ for every $i$ if and only if $U\in D_d(X)$.
\item[{\rm (3)}] If $U,V\in D_d(X)$, then $U\times_X V\in D_d(X)$.
\end{enumerate}
\end{lem}
\begin{proof}
Parts (1) and (3) are shown in \cite[Lemmas 2.4, 2.5]{Vunst-nis-cdh}. 
Let us prove (2). 
If $U\in D_d(X)$, then $U\times_X V_i\in D_d(V_i)$ for any $i$ by (1). 
Conversely, assume that $U\times_X V_i\in D_d(V_i)$ for every $i$. 
Let $x_0$ be a point of $X$ that is not in the image of $U\rightarrow X$. 
Choose $i$ such that $x_0\in V_i$. 
Then there exists an increasing sequence $x_0,\ldots,x_d$ of $V_i$, which is also an increasing sequence of $X$. 
It follows that $U\in D_d(X)$.
\end{proof}

\begin{df}
\label{A.9.44}
Let $X$ be a finite dimensional noetherian fs log scheme. \index[notation]{Dddiv @ $D_d^{div}(-)$}
For $d\in \N$, we denote by $D_d^{div}(X)$ the family $U=\amalg_i U_i\rightarrow X$ of finite disjoint unions of log \'etale monomorphisms to $X$ with the following property: 
For any log \'etale monomorphism $V\rightarrow X$ such that the projection $U\times_X V\rightarrow V$ is a finite disjoint union of log schemes $\{V_i\}$ with open immersions $V_i\rightarrow V$, 
then $U\times_X V\in D_d(V)$.
\end{df}

\begin{rmk}
\label{A.9.45}
Any finite disjoint union $U\rightarrow X$ of log \'etale monomorphisms is in $D_0^{div}(X)$.
\end{rmk}

In the following series of lemmas, we study fundamental properties of $D_d^{div}$.
In particular, we will show that $D_d^{div}$ is a density structure.

\begin{lem}
\label{A.9.54}
Let $f:Y\rightarrow X$ be a surjective proper log \'etale monomorphism of finite-dimensional noetherian fs log schemes. 
Then $Y\in D_d^{div}(X)$ for all $d\in \N$.
\end{lem}
\begin{proof}
Let $V\rightarrow X$ be a log \'etale monomorphism such that the projection $p:Y\times_X V\rightarrow V$ is a finite disjoint union of open immersions to $V$. 
By Proposition \ref{A.9.22}, $p$ is surjective. 
Thus $Y\times_X V\in D_d(V)$ and $Y\in D_d^{div}(X)$ for all $d\in \N$.
\end{proof}

\begin{rmk}
If $X$ has a trivial log structure, then every proper log \'etale monomorphism is an isomorphism, 
see Remark \ref{A.9.68}(1).
Thus, in this case, 
the density structure $D_d^{div}(X)$ agrees with the standard one.
\end{rmk}

Let us review the structure of affine toric varieties used in the proof of Lemma \ref{A.9.46}.
Suppose $P$ is an fs monoid.
Recall from \cite[\S I.3.3]{Ogu} that for every face $F$ of $P$, there is a naturally induced closed immersion
\[
 \underline{\A_F}=\Spec{\Z[F]}\cong \Spec{\Z[P]/\Z[P-F]}\rightarrow \underline{\A_P}.
\]
There is also a naturally induced open immersion \index[notation]{AF @ $\A_F^*$}
\[
\underline{\A_F^*}:=\Spec{\Z[F^{\rm gp}]}\rightarrow \underline{\A_F}.
\]
The family $\{\underline{\A_F^*}\}$ indexed by all the faces $F$ of $P$ forms a stratification of $\A_P$, 
see \cite[Proposition I.3.3.4]{Ogu}.
Moreover, 
since $\Spec{\Z[P]/\Z[P-F]}$ is naturally identified with $\underline{\A_F}$, 
the points of $\underline{\A_F^*}$ are exactly identified with the points $x$ of $\underline{\A_P}$ such that $\overline{\cM}_{\A_P,x}$ is naturally given by $P/F$.
Suppose that $\theta:P\rightarrow Q$ is a homomorphism of fs monoids.
This induces a natural morphism
\[
\underline{\A_\theta}:\underline{\A_Q}\rightarrow \underline{\A_P}.
\]
For every face $G$ of $Q$, we have that
\begin{equation}
\label{A.9.46.2}
\underline{\A_\theta}(\underline{\A_G})\cong \underline{\A_{\theta^{-1}(G)}}.
\end{equation}

\begin{lem}
\label{A.9.46}
Suppose $f:Y\rightarrow X$ be a log \'etale monomorphism of finite dimensional noetherian fs log schemes. 
For $x\in X$, let $y$ be a generic point of $f^{-1}(x)$, and let $z\in Y$ be a generalization of $y$. 
Then $z$ is a generic point of $f^{-1}(f(z))$.
\end{lem}
\begin{proof}
The question is Zariski local on $Y$ and $X$. 
Thus by Remark \ref{A.9.10}(2), we may assume that $X$ has a neat chart $P$ at $x$.
This means that $P$ is sharp and $\overline{\cM}_{X,x}\cong P$.
Moreover, 
by Lemma \ref{A.9.19}, 
we may assume that $f$ has an fs chart $\theta:P\rightarrow Q$ such that $\theta^{\rm gp}$ is an isomorphism and the induced morphism 
\[
Y\rightarrow X\times_{\A_P}\A_Q
\]
is an open immersion. 
By Remark \ref{A.9.10}(1), we may assume that $Q$ is an exact chart at the point $y$.
This means that $\overline{\cM}_{Y,y}\cong Q$, but we do not require that $Q$ is sharp.
We may assume $Y\cong X\times_{\A_P}\A_Q$ since the question is Zariski local on $Y$.
\vspace{0.1in}

Next, we claim that the induced homomorphism 
\[
\overline{\theta}:P\rightarrow \overline{Q}
\] 
is $\Q$-surjective, i.e., $\overline{\theta}\otimes \Q$ is surjective. Arguing by contradiction, assume that it is not.
Let $G_1,G_2,\ldots,G_r$ be all the faces of $Q$ such that $\theta^{-1}(G_i)=0$.
Then the corresponding faces $\overline{G_1},\ldots,\overline{G_r}$ in $\overline{Q}$ are all the faces of $\overline{Q}$ such that $\overline{\theta^{-1}}(\overline{G_i})=0$.
The implication (6)$\Rightarrow$(1) in \cite[Proposition I.4.3.9]{Ogu} means that $r\geq 2$.
From \eqref{A.9.46.2} we have that
\begin{equation}
\label{A.9.46.1}
\{0\}\times_{\underline{\A_P}}\underline{\A_Q}
\cong
\underline{\A_{G_1}}\cup \cdots \cup \underline{\A_{G_r}}.
\end{equation}
Let $\kappa(-)$ denote the residue field at some point, and set $D_i:=\underline{\A_{G_i}}\times \Spec{\kappa(x)}$ for simplicity.
Then from \eqref{A.9.46.1}, we have 
\[
f^{-1}(x)
\cong
D_1\cup \cdots \cup D_r.
\]
Since $\overline{\cM}_{Y,y}\cong Q$, it follows that $y$ is contained in $D_1$.
The isomorphism
\[
D_2\cong \spec{\kappa(x)[G_2]},
\]
shows that $D_2$ is irreducible.
We have reached a contradiction since $D_1\subsetneq D_2$ and $y$ is a generic point of $f^{-1}(x)$.
\vspace{0.1in}


Let $z'$ be a generalization of $z$ in $f^{-1}(f(z))$. 
Consider the projection 
\[
p:X\times_{\A_P}\A_Q\rightarrow \A_Q, 
\] 
and let $\overline{Q}/G$ (resp.\ $\overline{Q}/G'$) be the log structure on $z$ (resp.\ $z'$) where $G$ (resp.\ $G'$) is a face of $\overline{Q}$. 
Note that $G'\subset G$. 
Since $f(z)=f(z')$ we have 
\[
\overline{\theta}^{-1}(G)=\overline{\theta}^{-1}(G').
\] 
By the implication (1)$\Rightarrow$(2) in \cite[Proposition I.4.3.9]{Ogu} it follows that $G=G'$.
Set $u:=f(z)=f(z')$, and let $w$ be the generic point of $\underline{\A_G}$ with image $t$ in $\underline{\A_P}$. 
Then 
\[
z,z'\in\Spec{\kappa(u)\otimes_{\kappa(t)} \kappa(w)}.
\]
Thus $z=z'$,
so any generalization of $z$ in $f^{-1}(f(z))$ is equal to $z$.
This implies that $z$ is a generic point of $f^{-1}(f(z))$.
\end{proof}

\begin{lem}
\label{A.9.47}
Let $f:Y\rightarrow X$ be a surjective log \'etale monomorphism of finite-dimensional noetherian fs log schemes, and let $j:U\rightarrow X$ be a finite disjoint union of open immersions to $X$. 
If $U\times_X Y\in D_d(Y)$, then $U\in D_d(X)$.
\end{lem}
\begin{proof}
Suppose $x\in X$ is not in the image of $j$. 
Since $f$ is surjective, we can choose a generic point $y_0$ of $f^{-1}(x_0)$ and an increasing sequence $y_0,\ldots,y_d$ of $Y$. 
Then $f(y_i)$ is a generic point of $f^{-1}(f(y_i))$ according to Lemma \ref{A.9.46}.
Thus $f(y_0),\ldots,f(y_d)$ is an increasing sequence of $X$.
It follows that $U\in D_d(X)$.
\end{proof}

\begin{lem}
\label{A.9.56}
Let $X$ be a finite dimensional noetherian fs log scheme, and let $d\in \N$. 
Then for every $U\in D_d^{div}(X)$ and log \'etale monomorphism $V\rightarrow X$, we have $U\times_X V\in D_d^{div}(V)$.
\end{lem}
\begin{proof}
Let $W\rightarrow V$ be a log \'etale monomorphism such that the projection $U\times_X W\rightarrow W$ is a finite disjoint union of open immersions.
Since $U\in D_d^{div}(X)$ and the composition $W\rightarrow X$ is a log \'etale monomorphism, we have
\[
(U\times_X V)\times_V W=U\times_X W\in D_d(W).
\] 
This shows that $U\times_X V\in D_d^{div}(V)$.
\end{proof}

\begin{lem}
\label{A.9.48}
Let $Y\rightarrow X$ be a surjective proper log \'etale monomorphism of finite-dimensional noetherian fs log schemes, and let $U\rightarrow X$ be a finite disjoint union of log \'etale monomorphisms to $X$. 
Then $U\in D_d^{div}(X)$ if and only if $U\times_X Y\in D_d^{div}(Y)$.
\end{lem}
\begin{proof}
If $U\in D_d^{div}(X)$, then $U\times_X Y\in D_d^{div}(Y)$ by Lemma \ref{A.9.56}. 
Conversely, assume that $U\times_X Y\in D_d^{div}(Y)$. 
Let $V\rightarrow X$ be a log \'etale monomorphism such that the projection $U\times_X V\rightarrow V$ is a finite disjoint union of open immersions to $V$. 
Then 
\[
U\times_X V\times_X Y\in D_d(V\times_X Y),
\] 
and the projection $V\times_X Y\rightarrow V$ is a surjective proper log \'etale monomorphism by Proposition \ref{A.9.22}. 
Thus $U\times_X V\in D_d(V)$ by Lemma \ref{A.9.47}, so $U\in D_d^{div}(X)$.
\end{proof}

\begin{lem}
\label{A.9.55}
Let $X$ be a finite dimensional noetherian fs log scheme, and let $d\in \N$.
\begin{enumerate}
\item[{\rm (1)}] Let $U\rightarrow X$ be a finite disjoint union of log \'etale monomorphisms to $X$, 
and let $\{V_i\rightarrow X\}_{i\in I}$ be an open cover of $X$. 
Then $U\times_X V_i\in D_d^{div}(V_i)$ for any $i$ if and only if $U\in D_d^{div}(X)$.
\item[{\rm (2)}] If $U,V\in D_d^{div}(X)$, then $U\times_X V\in D_d^{div}(X)$.
\end{enumerate}
\end{lem}
\begin{proof}
(1) If $U\in D_d^{div}(X)$, then $U\times_X V_i\in D_d^{div}(V_i)$ for any $i$. 
Conversely, assume that $U\times_X V_i\in D_d^{div}(V_i)$ for any $i$. 
Let $W\rightarrow X$ be a log \'etale monomorphism such that the projection $U\times_X W\rightarrow W$ is a finite disjoint union of open immersions to $W$. 
Then 
\[
V_i\times_X U\times_X W\in D_d(V_i\times_X W),
\] 
and hence $U\times_X W\in D_d(W)$ by Lemma \ref{A.9.51}(3). 
This shows that $U\in D_d^{div}(X)$.
 \vspace{0.1in}
  
(2) The question is Zariski local on $X$ by (1), so we may assume that $X$ has an fs chart $P$. 
By Proposition \ref{A.9.21} there is a subdivision of fans $\Sigma\rightarrow \Spec{P}$ such that the pullbacks
\[
U\times_{\A_P}\A_\Sigma\rightarrow X\times_{\A_P}\A_\Sigma,
\,\,
V\times_{\A_P}\A_\Sigma\rightarrow X\times_{\A_P}\A_\Sigma
\]
are disjoint unions of open immersions to $X\times_{\A_P}\A_\Sigma$. 
By Lemma \ref{A.9.48},
\[
U\times_{\A_P}\A_\Sigma,V\times_{\A_P}\A_\Sigma\in D_d^{div}(X\times_{\A_P}\A_\Sigma),
\]
and we are reduced to showing that 
\[
U\times_S V\times_{\A_P}\A_\Sigma\in D_d^{div}(X\times_{\A_P}\A_\Sigma).
\] 
We may replace $X$ by $X\times_{\A_P}\A_\Sigma$, and assume that $U$ and $V$ are a disjoint union of open subschemes of $X$.
Let $W\rightarrow X$ be a log \'etale monomorphism. 
Then 
\[
U\times_X W,V\times_X W\in D_d(W)
\]
and hence $U\times_X V\times_X W\in D_d(W)$ by Lemma \ref{A.9.51}(3). 
This finishes the proof.
\end{proof}

\begin{lem}
\label{A.9.52}
If $v:V\rightarrow U$ is in $D_d^{div}(U)$ and $u:U\rightarrow X$ is in $D_d^{div}(X)$, then the composite $v\circ u:V\rightarrow X$ is in $D_d^{div}(X)$.
\end{lem}
\begin{proof}
The question is Zariski local on $X$ by Lemma \ref{A.9.55}(1), 
so we may assume that $X$ has an fs chart $P$. 
By Proposition \ref{A.9.21}, 
there is a surjective proper log \'etale monomorphism $Y\rightarrow X$ such that the projections $U\times_X Y\rightarrow Y$ and $V\times_X Y\rightarrow Y$ are finite disjoint unions of open immersions to $Y$. 
Then the pullback $V\times_X Y\rightarrow U\times_X Y$ is also a finite disjoint union of open immersions to $U\times_X Y$. 
By Lemma \ref{A.9.48}, 
\[
U\times_X Y\in D_d^{div}(Y) \text{ and } V\times_X Y\in D_d^{div}(U\times_X Y).
\] 
We are reduced to showing that $V\times_X Y\in D_d^{div}(Y)$ again by Lemma \ref{A.9.48}. 
To that end, we may replace $(X,U,V)$ by $(Y,U\times_X Y,V\times_X Y)$ and assume that $u$ and $v$ are disjoint unions of open immersions to $X$.

Let $W\rightarrow X$ be a log \'etale monomorphism such that the projection $U\times_X W\rightarrow W$ is a finite disjoint union of open immersions.
Then $U\times_X W\in D_d(W)$ and $V\times_X W\in D_d(U\times_X W)$, so $V\times_X W\in D_d(W)$ since $D_*$ is a density structure. 
Thus we have $V\in D_d^{div}(X)$.
\end{proof}

Here is a summary of the previous lemmas.

\begin{thm}
Let $X$ be a finite dimensional noetherian fs log scheme.
Then we have the following properties.
\begin{enumerate}
\item[{\rm (1)}] If $U\rightarrow X$ is a surjective proper log \'etale monomorphism, then $U\in D_d^{div}(X)$ for all $d\in \N$.
\item[{\rm (2)}] Let $V\rightarrow X$ be a log \'etale monomorphism.
If $U\in D_d^{div}(X)$, then $U\times_X V\in D_d^{div}(V)$.
\item[{\rm (3)}] Let $Y\rightarrow X$ be a dividing Zariski cover.
Then $U\in D_d^{div}(X)$ if and only if $U\times_X Y\in D_d^{div}(Y)$.
\item[{\rm (4)}] If $U,V\in D_d^{div}(X)$, then $U\times_X V\in D_d^{div}(X)$.
\item[{\rm (5)}] If $v:V\rightarrow U$ is in $D_d^{div}(U)$ and $u:U\rightarrow X$ is in $D_d^{div}(X)$, then $v\circ u$ is in $D_d^{div}(X)$.
\end{enumerate}
\end{thm}

\begin{df}
\label{A.9.57}
The density structure $D_{\ast}^{div}$ on $lSm/S$ and $lSch/S$ established in Lemma \ref{A.9.52} is called the {\it dividing density structure}.
\end{df}\index{density structure!dividing}

\begin{lem}
\label{A.9.58}
Let $X$ be a finite-dimensional noetherian fs log scheme with an fs chart $P$, and let $U\rightarrow X$ be a finite disjoint union of log \'etale monomorphisms to $X$. 
If $U\times_X V\in D_d(V)$ for any surjective proper log \'etale monomorphism $V\rightarrow X$ such that the projection $U\times_X V\rightarrow V$ is a finite disjoint union of open immersions to $V$, 
then $U\in D_d^{div}(X)$.
\end{lem}
\begin{proof}
Let $W\rightarrow X$ be a log \'etale monomorphism such that the projection 
\[
U\times_X W\rightarrow W
\] 
is a finite disjoint union of open immersions to $W$. 
By Proposition \ref{A.9.21}, there is a surjective proper log \'etale monomorphism $V\rightarrow X$ such that
\[
W\times_X V\rightarrow V \text{ and } U\times_X V\rightarrow V
\]
is an open immersion and a finite disjoint union of open immersions to $V$,
respectively. 
\vspace{0.1in}

Consider the commutative diagram with the evident projection morphisms
\[
\begin{tikzcd}
U\times_X W\arrow[d]\arrow[r,leftarrow]&U\times_X W\times_X V\arrow[d]\arrow[r]&U\times_X V\arrow[d]\\
W\arrow[r,leftarrow]&W\times_X V\arrow[r]&V.
\end{tikzcd}\]
By assumption we have $U\times_X V\in D_d(V)$. 
Lemma \ref{A.9.51}(1) implies that 
\[
U\times_X W\times_X V\in D_d(W\times_X V)
\] 
since $W\times_X V$ is an open subscheme of $V$. 
Thus by Lemma \ref{A.9.47}, 
\[
U\times_X W\in D_d(W)
\] 
since the projection $W\times_X V\rightarrow W$ is a surjective proper log \'etale monomorphism by Proposition \ref{A.9.22}. 
This finishes the proof.
\end{proof}

\begin{lem}
\label{A.9.49}
Let $f:Y\rightarrow X$ be a strict Nisnevich morphism of finite-dimensional noetherian fs log schemes. 
Assume that $X$ has an fs chart $P$. 
Then for every $d\in \N$ and $Y_0\in D_d^{div}(Y)$, there exists $X_0\in D_d^{div}(X)$ such that the projection $X_0\times_X Y\rightarrow Y$ factors through $Y_0$.
\end{lem}
\begin{proof}
By definition, $Y_0\rightarrow Y$ is a disjoint union of log \'etale monomorphisms.
Proposition \ref{A.9.21} says that there exists a subdivision of fans $\Sigma\rightarrow \Spec P$ such that the induced morphism
\[
Y_0\times_{\A_P}\A_\Sigma\rightarrow Y\times_{\A_P}\A_\Sigma
\]
is a finite disjoint union of open immersions to $Y\times_{\A_P}\A_\Sigma$. 
Now by Lemma \ref{A.9.48}, $Y_0\times_{\A_P}\A_\Sigma\in D_d^{div}(Y\times_{\A_P}\A_\Sigma)$. 
Hence replacing $Y_0$ by $Y_0\times_{\A_P}\A_\Sigma$, we may assume
\[
Y_0\cong Y_0\times_{\A_P}\A_\Sigma.
\] 
Then the morphism $Y_0\rightarrow Y\times_{\A_P}\A_\Sigma$ is a finite disjoint union of open immersions. 
It suffices to construct $X_0\in D_d^{div}(X\times_{\A_P}\A_\Sigma)$ since then $X_0\in D_d^{div}(X)$ by Lemma \ref{A.9.48}. 
We may replace $X$ by $X\times_{\A_P}\A_\Sigma$ and assume $Y_0$ is a finite disjoint union of open subschemes of $Y$.
\vspace{0.1in}

Set
\[
X_0:=X-cl(f(Y-\im (Y_0\rightarrow Y)),
\]
where $cl(-)$ denotes the closure. 
The construction in the proof of \cite[Lemma 2.9]{Vunst-nis-cdh} shows that $X_0\in D_d(X)$. 
Let $V\rightarrow X$ be a surjective proper log \'etale monomorphism. 
Since $V$ is proper over $X$, we have 
\[
X_0\times_X V=V-cl(f_V(Y\times_X V-\im ((Y_0\times_X V)\rightarrow (Y\times_X V)))),
\]
where $f_V:Y\times_X V\rightarrow V$ is the projection. 
Thus again using the proof of \cite[Lemma 2.9]{Vunst-nis-cdh}, $X_0\times_X V\in D_d(V)$, and $X_0\in D_d^{div}(X)$ by Lemma \ref{A.9.58}.
\end{proof}

\begin{df}[{\cite[Definition 2.21]{Vcdtop}}] \index{distinguished square!reducing}
\label{A.9.64}
Let $P$ be a cd-structure on a category $\cC$ with an initial object, and let $D_*$ be a density structure on $\cC$. 
A distinguished square $Q$ of the form \eqref{A.8.13.1} is called {\it reducing} with respect to $D_*$ if for all $d\in \N^+$, $Y_0\in D_d(Y)$, $X_0'\in D_d(X')$, 
and $Y_0'\in D_{d-1}(Y')$ there is a distinguished square
\[
Q_1
=
\begin{tikzcd}
Y_1'\arrow[d]\arrow[r]&Y_1\arrow[d]\\
X_1'\arrow[r]&X_1
\end{tikzcd}\]
together a morphism $Q_1\rightarrow Q$ satisfying the following properties.
\begin{enumerate}
\item[(i)] The morphism $X_1\rightarrow X\in D_d(X)$.
\item[(ii)] The morphism $Y_1\rightarrow Y$ (resp.\ $X_1'\rightarrow X'$, resp.\ $Y_1'\rightarrow Y'$) factors through $Y_0$ (resp.\ $X_0'$, resp.\ $Y_0'$).
\end{enumerate}

We say that a distinguished square $Q'$ is a {\it refinement}\index{distinguished square!refinement} of $Q$ if there is a morphism $Q'\rightarrow Q$ that is an isomorphism between the lower right corners.
\end{df}

\begin{df}
\label{A.9.65}
Let $P$ be a cd-structure on a category $\cC$ with an initial object, and let $D_*$ be a density structure on $\cC$. Consider the following conditions.
\begin{enumerate}
\item[(i)] Any distinguished square in $P$ has a refinement that is reducing with respect to $D_*$.
\item[(ii)] $D_*$ is locally of finite dimension with respect to the topology associated with $P$.
\end{enumerate}
If $P$ satisfies (i) (resp.\ (i) and (ii)), we say that $P$ is {\it reducing} (\cite[Definition 2.22]{Vcdtop}) (resp.\ {\it quasi-bounded}) with respect to $D_*$.
\end{df}\index{density structure!reducing}\index{density structure!quasi-bounded}

\begin{rmk}
\label{A.9.66}
A cd-structure $P$ is bounded with respect to $D_*$ in the sense of \cite[Definition 2.22]{Vcdtop} if and only if $P$ is reducing with respect to $D_*$ and 
$D_*$ is locally of finite dimension with respect to the trivial topology. \index{density structure!bounded}
The reason why we consider quasi-bounded cd-structures is that $D_*^{div}$ is not bounded. 
Indeed, if $Y\rightarrow X$ is a surjective proper log \'etale monomorphism, then $Y\in D_d^{div}(X)$ for any $d\in \N$ by Lemma \ref{A.9.56}. 
Thus for any $d\in \N$, $D_d^{div}(X)$ can contain morphisms other than isomorphisms in general, so $D_*^{div}$ is not bounded.
\end{rmk}

\begin{lem}
\label{A.9.84}
Let $Y\rightarrow X$ be a log \'etale morphism in $lSm/k$.
Then $\dim Y\leq \dim X$.
\end{lem}
\begin{proof}
Since $X$ and $Y$ are in $lSm/k$, $X-\partial X$ and $Y-\partial Y$ are dense in $X$ and $Y$.
To conclude observe that $\dim (Y-\partial X)\leq \dim(X-\partial X)$.
\end{proof}

\begin{exm}
\label{A.9.85}
Let $\Sigma$ be the star subdivision of the dual of $\N^2$.
If
\[
X:=\{(0,0)\}\times_{\A^2}\A_{\N^2}
\text{ and }
Y=\{(0,0)\}\times_{\A^2}\A_\Sigma,
\]
then the naturally induced morphism $Y\rightarrow X$ is log \'etale, but $\dim Y=2$ and $\dim X=1$.
\end{exm}

\begin{prop}
\label{A.9.61}
The dividing density structures on $lSm/S$ and $lSch/S$ are locally of finite dimension with respect to the dividing Nisnevich topology.
\end{prop}
\begin{proof}
Suppose $X$ is an $n$-dimensional noetherian fs log scheme.
Let $f:U\rightarrow X$ be any finite disjoint union of log \'etale monomorphisms to $X$.
We need to show that there exists an integer $m\geq 0$ such that $U\in D_{m}^{div}(X)$ implies that $f$ is a dividing Nisnevich cover.
Let $r$ be the maximum rank of $\overline{\cM}_{X,x}^{\rm gp}$,
where $x$ ranges over all points of $X$.
Since $X$ is an fs log scheme, the rank of $\overline{\cM}_{X,x}^{\rm gp}$ is a locally finite function on $X$.
Thus $r$ is a finite number since $X$ is quasi-compact.
In order to conclude, 
it suffices to show that $U\in D_{r+d+1}^{div}(X)$ implies that $f$ is a dividing Zariski cover.
The problem is a Zariski local on $X$, 
so that we may assume $X$ has an fs chart $P$ that is exact at a point of $X$,
see Remark \ref{A.9.10}(1). 
By definition we have $\rank P^{\rm gp}\leq r$.
\vspace{0.1in}

Proposition \ref{A.9.21} shows there exists a subdivision of fans $\Sigma\rightarrow \Spec P$ such that the pullback
\[
f':U\times_{\A_P}\A_\Sigma\rightarrow X\times_{\A_P}\A_\Sigma
\] 
is a finite disjoint union of open immersions to $X\times_{\A_P}\A_\Sigma$. 
For every point $x\in \underline{\A_P}$, the fiber of $x$ under $\underline{\A_\Sigma}\rightarrow \underline{\A_P}$ has dimension $\leq r$.
Thus we get that 
\[
\dim (X\times_{\A_P}\A_\Sigma)\leq r+n.
\]
It follows that $f'$ is surjective because 
\[
U\times_{\A_P}\A_\Sigma\in D_{r+n+1}(X\times_{\A_P}\A_\Sigma). 
\]
In \cite[p.\ 1401]{Vunst-nis-cdh}, it is shown that the dimension with respect to the standard density structure is equal to the Krull dimension.
Thus $f'$ is a Zariski cover.
The projection $X\times_{\A_P}\A_\Sigma\rightarrow X$ is a surjective proper log \'etale monomorphism. 
In conclusion, 
$f$ admits a refinement by a dividing Zariski cover, 
i.e., 
$f$ is a dividing Zariski cover.
\end{proof}

\begin{rmk}
The above proof shows that the dimension of $X$ with respect to $D_*^{div}$ is less than or equal to the maximum possible dimension of $X\times_{\A_P}\A_\Sigma$.
If $X$ is log smooth over $k$,
we see that the dimension of $X$ with respect to $D_*^{div}$ is $\leq n$ since $\dim X\times_{\A_P}\A_\Sigma\leq n$ due to Lemma \ref{A.9.84}.
\end{rmk}

\begin{prop}
\label{A.9.50}
The strict Nisnevich cd-structures on $lSm/S$ and $lSch/S$ are quasi-bounded with respect to the dividing density structure.
\end{prop}
\begin{proof}
By Proposition \ref{A.9.61} it remains to show the strict Nisnevich cd-structure is reducing with respect to the dividing density structure. 
Let
\[
Q
=
\begin{tikzcd}
Y'\arrow[r,"g'"]\arrow[d,"f'"']&Y\arrow[d,"f"]\\
X'\arrow[r,"g"]&X
\end{tikzcd}\]
be a strict Nisnevich distinguished square of finite dimensional noetherian fs log schemes, 
where $f$ is strict \'etale, $g$ is an open immersion, 
and $f$ induces an isomorphism 
\[
f^{-1}(\underline{X}-g(\underline{X'}))\stackrel{\sim}\rightarrow \underline{X}-g(\underline{X'})
\] 
with respect to the reduced scheme structures. 
We will show that $Q$ is reducing.
\vspace{0.1in}

For elements $X_0'\in D_d^{div}(X')$, $Y_0\in D_d^{div}(Y)$, and $Y_0'\in D_{d-1}^{div}(Y')$ we need to construct $X_1\in D_d^{div}(X)$ and a strict Nisnevich distinguished square
\[
Q_1
=
\begin{tikzcd}
Y_1'\arrow[r]\arrow[d]&Y_1\arrow[d]\\
X_1'\arrow[r]&X_1
\end{tikzcd}
\]
together with a morphism $Q_1\rightarrow Q$ that coincides with $X_1\rightarrow X$ on the lower right corner such that 
$Y_1\rightarrow Y$ (resp.\ $X_1'\rightarrow X'$, resp.\ $Y_1'\rightarrow Y'$) factors through $Y_0$ (resp.\ $X_0$, resp.\ $Y_0$). 
By appealing to Lemma \ref{A.9.55}(1), it suffices to construct $X_1$ Zariski locally on $X$. Hence we may assume that $X$ has an fs chart $P$.
\vspace{0.1in}

Owing to Lemma \ref{A.9.49} we can choose $V,W\in D_d^{div}(X)$ such that the projections 
\[
V\times_X Y\rightarrow Y \text{ and } W\times_X X'\rightarrow X' 
\]
factors through $Y_0$ and $X_0'$, respectively.
Take $X_0:=V\times_X W$, which is again in $D_d^{div}(X)$ by Lemma \ref{A.9.55}(2).
Then the projections
\[
X_0\times_X Y\rightarrow Y \text{ and } X_0\times_X X'\rightarrow X' 
\]
again factors through $Y_0$ and $X_0'$, respectively.
It suffices to construct $X_1\in D_d^{div}(X_0)$ with the required property since then $X_1\in D_d^{div}(X)$ by Lemma \ref{A.9.52}. 
Hence we may replace $X$ by $X_0$ and assume that $X'=X_0'$ and $Y=Y_0$.
\vspace{0.1in}

By Proposition \ref{A.9.21}, there exists a subdivision of fans $\Sigma\rightarrow \Spec P$ such that the pullback 
\[
Y_0'\times_{\A_P}\A_\Sigma\rightarrow Y'\times_{\A_P}\A_\Sigma
\] 
is a finite disjoint union of open immersions to $Y'\times_{\A_P}\A_\Sigma$. 
By Lemma \ref{A.9.48} we have 
\[
Y_0'\times_{\A_P}\A_\Sigma\in D_{d-1}^{div}(Y'\times_{\A_P}\A_\Sigma).
\]
This reduces the claim to constructing $X_1\in D_d^{div}(X\times_{\A_P}\A_\Sigma)$ with the required property since then $X_1\in D_d^{div}(X)$ again by Lemma \ref{A.9.48}. 
To that end, we may replace $X$ by $X\times_{\A_P}\A_\Sigma$, and hence assume that $Y_0'$ is a finite disjoint union of open subschemes of $Y_0$.
\vspace{0.1in}

The construction in the proof of \cite[Proposition 2.10]{Vunst-nis-cdh} shows that
\[
X_1:=X-((X-X')\cap cl(gf'(Y'-Y_0')))
\in 
D_d(X),
\]
and the induced square
\[
\begin{tikzcd}
Y_0'\arrow[r]\arrow[d]&g'(Y_0')\arrow[d]\\
X'\arrow[r]&X_1
\end{tikzcd}
\]
is a strict Nisnevich distinguished square.
\vspace{0.1in}

Now let $V\rightarrow X$ be a surjective proper log \'etale monomorphism. 
Then 
\[
Y_0'\times_{X}V\in D_{d-1}(V), 
\]
and
\[
X_1\times_X V=V-(V-X'\times_X V)\cap cl(g_Vf_V'(Y'\times_X V-Y_0'\times_X V)),
\]
where $f_v':Y'\times_X V\rightarrow X'\times_X V$ and $g_v:X'\times_X V\rightarrow V$ are the pullbacks since $V$ is proper over $X$. 
Thus $X_1\times_X V\in D_d(X)$ as in the above paragraph, and therefore $X_1\in D_d^{div}(X)$ by Lemma \ref{A.9.58}.
\end{proof}

\begin{prop}
\label{A.9.59}
The dividing cd-structures on $lSm/S$ and $lSch/S$ are quasi-bounded with respect to the dividing density structure.
\end{prop}
\begin{proof}
Let $f:Y\rightarrow X$ be a surjective proper log \'etale monomorphism, and let $Y_0\in D_d^{div}(Y)$. 
Then $Y_0\in D_d^{div}(X)$ by Lemma \ref{A.9.52}, and 
\[
Y_0\times_X Y\cong Y_0\times_Y\times (Y\times_X Y)\cong Y_0
\] 
since $f$ is a monomorphism.
Then
\[
Q_1=\begin{tikzcd}
\emptyset\arrow[d]\arrow[r]&
Y_0\arrow[d,"{\rm id}"]\arrow[d]
\\
\emptyset \arrow[r]&Y_0
\end{tikzcd}
\]
gives the desired distinguished square in Definition \ref{A.9.64}.
\end{proof}

\begin{prop}
\label{A.9.60}
The strict Nisnevich, dividing, and dividing Nisnevich cd-structures on $lSm/S$ and $lSch/S$ are complete, 
regular, and quasi-bounded with respect to the dividing density structure.
\end{prop}
\begin{proof}
These cd-structures are complete by \cite[Lemma 2.5]{Vcdtop}. 
By Propositions \ref{A.9.50} and \ref{A.9.59}, the strict Nisnevich cd-structure and the dividing cd-structure are quasi-bounded with respect to the dividing density structure.
The union of two cd-structures that are quasi-bounded with respect to the same density structure is again a quasi-bounded cd-structure as in \cite[Lemma 2.24]{Vcdtop}.
Thus the dividing Nisnevich cd-structure is also quasi-bounded with respect to the dividing density structure.
\vspace{0.1in}

A square is a strict Nisnevich distinguished square if and only if every morphism in the square is strict and the square of underlying schemes is a Nisnevich distinguished square. 
Therefore the strict Nisnevich cd-structure is regular since the Nisnevich cd-structure is regular by \cite[Lemma 2.14]{Vunst-nis-cdh}. 
The dividing cd-structure is regular since if $f:Y\rightarrow X$ is a log modification, then $f$ is a monomorphism, so the diagonal morphism $Y\rightarrow Y\times_X Y$ is an isomorphism.
Thus the dividing Nisnevich cd-structure is regular by \cite[Lemma 2.12]{Vcdtop}
\end{proof}

\subsection{Flasque simplicial presheaves}
In this section, we generalize Voevodsky's \cite[Lemma 3.5]{Vcdtop} to complete, quasi-bounded, and regular cd-structures.
As an application, we discuss a bounded descent structure on sheaves of abelian groups. \index{descent structure!bounded}

\begin{df}
\label{A.10.1} \index{squareable morphism}
A morphism $f:Y\rightarrow X$ in a category $\cC$ is \emph{squareable}\footnote{Some authors opted for the French expression \emph{quarrable} from SGA 4.} if for every morphism $X'\rightarrow X$ the fiber product $Y\times_X X'$ is representable.
If $f$ is squareable, then let $\check{C}(Y/X)$ (or $\check{C}(Y)$ when no confusion seems likely to arise) denote the corresponding \v{C}ech nerve\index{Cech nerve@\v{C}ech nerve} 
\[
\check{C}(Y/X):=\left(\cdots \substack{\longrightarrow\\[-1em] \longrightarrow \\[-1em] \longrightarrow} Y\times_X Y\rightrightarrows Y\right).
\]
Let $\check{C}(f)$ denote the morphism $\check{C}(Y)\rightarrow X$ induced by the structure morphisms $Y\times_X \cdots \times_X Y\rightarrow Y$.
\end{df}

\begin{df}\index{squarable cd-structure}
A cd-structure $P$ in a category $\cC$ with an initial object is \emph{squareable} if the following conditions are satisfied.
\begin{enumerate}
\item[(i)] Every morphism to the initial object is an isomorphism.
\item[(ii)] For every distinguished square $Q$ of the form \eqref{A.5.14.1} and morphism $V\rightarrow X$, the square $Q\times_X V$ is representable and belongs to $P$.
\end{enumerate} 
\end{df}

Note that any squareable cd-structure is complete by \cite[Lemma 2.5]{Vcdtop}

\begin{df}
\label{A.10.8}
Let $P$ be a cd-structure on a category $\cC$ with an initial object. 
A commutative square
\begin{equation}
\label{A.10.8.1}
\mathscr{Q}
=
\begin{tikzcd}
\mathscr{Y}'\arrow[d]\arrow[r]&\mathscr{Y}\arrow[d]\\
\mathscr{X}'\arrow[r]&\mathscr{X}
\end{tikzcd}\end{equation}
of simplicial objects of $\cC$ is called a {\it level-wise distinguished square}\index{distinguished square!level-wise} if the induced square
\[\mathscr{Q}_n
=
\begin{tikzcd}
\mathscr{Y}_n'\arrow[d]\arrow[r]&\mathscr{Y}_n\arrow[d]\\
\mathscr{X}_n'\arrow[r]&\mathscr{X}_n
\end{tikzcd}\]
is a distinguished square for any $n\in \N$.
\end{df}

In the following, we consider a variant of the notion of a B.G.\ functor introduced in \cite[Definition 3.1]{Vcdtop}
Here B.G.\ is shorthand for Brown-Gersten.

\begin{df}
\label{A.10.2}
Let $P$ be a squareable cd-structure on a category $\cC$ with an initial object.\index{extended B.G. functor}
An {\it extended B.G.\ functor} on $\cC$ with respect to $P$ is a family of contravariant functors $\{T_q\}_{q\in \Z}$ 
from the category of simplicial objects in $\cC$ to the category of pointed sets together with pointed maps 
\[
\partial_{\mathscr{Q}}:T_{q+1}(\mathscr{Y}')\rightarrow T_q(\mathscr{X}) 
\] 
for any level-wise distinguished square $\mathscr{Q}$ of the form \eqref{A.10.8.1} subject to the following conditions.
\begin{enumerate}
\item[(i)] The pointed maps $\partial_{\mathscr{Q}}$ are natural with respect to the morphisms $\mathscr{Q}'\rightarrow \mathscr{Q}$ of level-wise distinguished squares.
\item[(ii)] For any $q\in \Z$ and any level-wise distinguished square $\mathscr{Q}$ of the form \eqref{A.10.8.1}, the induced sequence of pointed sets
\[
T_{q+1}(\mathscr{Y}')\stackrel{\partial_{\mathscr{Q}}}
\rightarrow 
T_q(\mathscr{X})\rightarrow T_q(\mathscr{Y})\times T_q(\mathscr{X}')
\]
is exact.
\item[(iii)] For any $t_P$-covering sieve $\{U_i\rightarrow X\}_{i\in I}$, the induced morphism 
\[
T_q(X)\rightarrow T_q(\check{C}(U))
\] 
is an isomorphism, where $U=\amalg_{i\in I}U_i$.
\end{enumerate}
In \cite[Definition 3.1]{Vcdtop} only the conditions (i) and (ii) are needed in the definition of B.G. functor.
The reason why we introduce the condition (iii) is to take advantage of extending $\{T_q\}_{q\in \Z}$ to simplicial objects.
\end{df}

For any morphism $f:Y\rightarrow X$, let $f^*:T_q(X)\rightarrow T_q(Y)$ denote the induced pullback morphism. 
Let $D$ be a density structure on $\cC$.
For $d\in \N$ and $q\in \Z$, we consider the following condition. \index[notation]{Cov @ ${\rm Cov}_{d,q}$}
\begin{enumerate}
\item[(${\rm Cov}_{d,q}$)] For any $X\in \cC$ and $a\in T_q(X)$, there exists $u:U\rightarrow X$ in $D_d(Y)$ such that the pullback of $a$ to $T_q(\check{C}(U))$ is trivial.
\end{enumerate}

For every $q\in \Z$, let $a_{t_P}T_q$ be the $t_P$-sheaf on $\cC$ associated with the presheaf
\[
(X\in \cC)\mapsto T_q(X).
\]

\begin{lem}
\label{A.10.3}
Let $P$ be a squareable and regular cd-structure reducing with respect to a density structure $D$ on a category $\cC$ with an initial object, let $\{T_q\}_{q\in \Z}$ be an extended B.G.\ functor, and let $d\in \N$ and $q\in \Z$.  
If $a_{t_P}T_{q}$ is trivial, then $({\rm Cov}_{d,q+1})$ implies $({\rm Cov}_{d+1,q})$.
\end{lem}
\begin{proof}
Subject to a few changes, we follow the proof of \cite[Theorem 3.2]{Vcdtop}. 
By forming refinements, 
we may assume that every distinguished square is reducing with respect to a density structure $D$. 
Let $a\in T_q(X)$, 
where $X$ is an object of $\cC$.
\vspace{0.1in}

Consider a distinguished square
\begin{equation}
\label{A.10.3.1}
Q
=
\begin{tikzcd}
Y'\arrow[d,"f'"']\arrow[r,"g'"]&Y\arrow[d,"f"]\\
X'\arrow[r,"g"]&X.
\end{tikzcd}
\end{equation}
First, let us assume there exist morphisms $v:Y_0\rightarrow Y$ in $D_{d+1}(Y)$ and $u':X_0'\rightarrow X'$ in $D_{d+1}(X')$ such that $a$ restricts trivially on $\check{C}(Y_0)$ and $\check{C}(X_0')$ (i.e. becomes trivial after such pullback). 
Since $Q$ is reducing, there exists a morphism $Q_1\rightarrow Q$ of distinguished triangles such that $Y_1\rightarrow Y$ (resp.\ $X_1'\rightarrow X'$) in the upper right (resp.\ lower left) corner factors through 
$Y_0$ (resp.\ $X_0'$), and the morphism $X_1\rightarrow X$ in the lower right corner belongs to $D_{d+1}(X)$. 
Then $a$ restricts trivially on $\check{C}(Y_1)$ and $\check{C}(X_1')$.  
According to condition (ii) in Definition \ref{A.10.2}, there exists an element $b\in T_{q+1}(\check{C}(Y_1'))$ such that $\partial_{\mathscr{Q}}(b)=a$.
By $({\rm Cov}_{d,q+1})$, there exists an object $Y_2'\in D_d(Y_1')$ such that $b$ restricts trivially on $\check{C}(Y_1')$. 
Now $Q_1$ is reducing, so we have a morphism $Q_3\rightarrow Q_1$ of distinguished squares such that $Y_3'\rightarrow Y_1'$ factors through $Y_2'$ and $X_3\in D_{d+1}(X_1)$. 
Note that $b$ restricts trivially on $\check{C}(Y_3')$, 
and hence likewise for $a$ and $\check{C}(X_3)$.
Since $X_1\in D_{d+1}(X)$ and $X_3\in D_{d+1}(X_1)$, we deduce that $X_3\in D_{d+1}(X)$, i.e., $(\mathrm{Cov}_{d+1,q})$ is satisfied, which concludes the proof in this special case.
\vspace{0.1in}

In the general case, since $a_{t_P}T_{q}$ is trivial, there exists a $t_P$-covering sieve 
\[
\cF=\{V_j\rightarrow X\}_{j\in J}
\] 
such that $a$ restricts trivially on $V_j$ for all $j\in J$. 
Since $P$ is complete, we may assume the covering is a simple covering.
By excluding the trivial cover, we may assume that $\cF$ is formed by simple coverings of $Y$ and $X'$ for a distinguished square of the form \eqref{A.10.3.1}.
If there exists morphisms $v:Y_0\rightarrow Y$ in $D_{d+1}(Y)$ and $u':X_0'\rightarrow X'$ in $D_{d+1}(X')$ such that $a$ restricts trivially on $\check{C}(Y_0)$ and $\check{C}(X_0')$,
using the above paragraph, 
there exists a morphism $u:U\rightarrow X$ in $D_d(Y)$ such that $a$ restricts trivially on $\check{C}(U)$.
Repeating this argument, we reduce to the case when $\cF$ is the trivial cover. 
Then clearly $a$ is trivial, and we are done.
\end{proof}

\begin{thm}
\label{A.10.4}
Let $P$ be a squareable and regular cd-structure that is quasi-bounded with respect to a density structure $D_{\ast}$ on a category $\cC$ with an initial object, 
and suppose $\cF$ is a $t_P$-sheaf of abelian groups. 
Then for any $X\in \cC$ and $n>\dim_D X$, the $n$th $t_P$-cohomology group vanishes
\[
H_{t_P}^n(X,\cF)=0.
\]
\end{thm}
\begin{proof}
For any simplicial object $\mathscr{X}$ in $\cC$, set $T_q(\mathscr{X}):=\bH_{t_P}^{-q}(\mathscr{X},\cF)$.
There is a canonical hypercohomology spectral sequence
\[
E_1^{pq}:=H_{t_P}^q(\mathscr{X}_p,\cF)\Rightarrow \bH_{t_P}^{p+q}(\mathscr{X},\cF).
\]
Together with \cite[Lemma 2.19]{Vcdtop} we deduce that $\{T_q\}_{q\in \Z}$ satisfies the conditions (i) and (ii) in Definition \ref{A.10.2}.
For any finite $t_P$-covering sieve $\{U_i\rightarrow X\}_{i\in I}$, the induced morphism $U\rightarrow X$ forms a $t_P$-covering sieve, 
where $X:=\amalg_{i\in I}U_i$.
Let $\mathscr{U}$ denote the \v{C}ech nerve associated with $U\rightarrow X$.
By cohomological descent, we obtain an isomorphism
\[
H_{t_P}^{-q}(X,\cF)\rightarrow H_{t_P}^{-q}(\mathscr{U},\cF).
\]
That is, 
$\{T_q\}_{q\in \Z}$  satisfies condition (iii) in Definition \ref{A.10.2}.
Therefore $\{T_q\}_{q\in \Z}$ is an extended B.G.\ functor. 
Moreover, 
for every $q<0$ we have 
\[
a_{t_P}T_q=a_{t_P}\bH_{t_P}^{-q}=0.
\]
Since $\emptyset\in D_0(X)$, we see that ${(\rm Cov}_{0,0})$ holds.
Then by Lemma \ref{A.10.3}, ${(\rm Cov}_{n,-n})$ holds for any $n\in \N$, 
i.e., 
there exists $j:U\rightarrow X$ in $D_n(Y)$ such that $a\in H^n(X,\cF)$ restricts trivially on $\check{C}(U)$. 
In the range $n>\dim_D X$, the morphism $j$ is a $t_P$-covering sieve. 
Thus $a$ is trivial owing to condition (iii) in Definition \ref{A.10.2}.
\end{proof}

\begin{thm}
\label{A.10.5}
Let $P$ be a squareable, quasi-bounded, and regular cd-structure on a category $\cC$ with an initial object, and let $\{T_q\}_{q\in \Z}$ be an extended B.G.\ functor. 
If $a_{t_P}T_q$ and $T_q(\emptyset)$ are trivial for any $q$, then $T_q$ is trivial for any $q$.
\end{thm}
\begin{proof}
Since $T_q(\emptyset)$ is trivial for any $q$, we see that $({\rm Cov}_{0,q})$ holds for all $q$. 
Applying Lemma \ref{A.10.3} we deduce that $({\rm Cov}_{d,q})$ holds for all $d$ and $q$. 
Thus for all $d$, $q$, $X\in \cC$, and $a\in T_q(X)$, there exists a morphism $u:U\rightarrow X$ in $D_d(Y)$ such that $a$ restricts trivially on $\check{C}(U)$. 
In the range $d>\dim_D(X)$, we have $u=\amalg_{i\in I}u_i$ for some finite $t_P$-covering sieve $\{u_i:U_i\rightarrow X\}_{i\in I}$. 
It follows that $a$ is trivial owing to condition (iii) in Definition \ref{A.10.2}.
\end{proof}

\begin{df}[{\cite[Definition 3.3]{Vcdtop}}]
\label{A.10.6}
Let $P$ be a cd-structure on a category $\cC$ with an initial object. 
A simplicial presheaf $\cF$ on $\cC$ is called {\it flasque}\index{simplicial presheaf!flasque} (with respect to $P$) if $\cF(\emptyset)$ is contractible and for any distinguished square $Q$, 
the square $\cF(Q)$ is homotopy cartesian.
\end{df}

\begin{df}
\label{A.10.9}
Let $\cC$ be a category equipped with a topology $t$.\index{equivalence ! $t$-local}
A morphism of simplicial presheaves $f:\cF\rightarrow \cF'$ is called a {\it $t$-local equivalence} if the following conditions are satisfied.
\begin{enumerate}
\item[(i)] The map $a_t\pi_0(\cF)\rightarrow a_t\pi_0(\cF')$ induced by $f$ is an isomorphism.
\item[(ii)] The map $a_t\pi_n(\cF,x)\rightarrow a_t\pi_n(\cF',f(x))$ induced by $f$ is an isomorphism for any $n\in \N^+$, $X\in \cC$, and $x\in F(X)$.
\end{enumerate}
Note that pointwise weak equivalences are exactly local equivalences for the trivial topology.
\end{df}

\begin{lem}
\label{A.10.7}
Let $P$ be a squareable, quasi-bounded, and regular cd-structure on a category $\cC$ with an initial object. 
Then a morphism $\cF\rightarrow \cG$ of flasque simplicial presheaves is a $t_P$-local equivalence if and only if it is a pointwise weak equivalence.
\end{lem}
\begin{proof}
The if part is trivial. Let us assume that $\cF\rightarrow \cG$ is a $t_P$-local equivalence. 
By the first two paragraphs of the proof of \cite[Lemma 3.5]{Vcdtop}, we reduce to the case when $G={\rm pt}$ and $\cF$ is a Kan simplicial presheaf. 
The fourth paragraph of the proof of \cite[Lemma 3.5]{Vcdtop} shows that $\cF({\rm pt})$ is nonempty. 
Choose an element $x\in \cF({\rm pt})$ and consider the family of functors
\[
T_q(\mathscr{X}):=\pi_q(\cF(\mathscr{X}),x|_\mathscr{X}),
\]
where $q\in \Z$ and $\mathscr{X}$ is a simplicial object in $\cC$ (by convention, $\pi_0(T,x):=\pi_0(T)$ and $\pi_q(T,x):=0$ for any simplicial set $T$ and $x\in \pi_0(T)$ if $q<0$). 
Since $\cF$ is flasque, $T_q$ satisfies the conditions (i) and (ii) of Definition \ref{A.10.2}. 
By \cite[Corollary 5.10]{Vcdtop}, $T_q$ satisfies condition (iii) of Definition \ref{A.10.2}, 
i.e., 
$T_q$ is an extended B.G.\ functor. 
Since $\cF\rightarrow {\rm pt}$ is a $t_P$-local equivalence, $a_{t_P}T_q$ is trivial for any $q$. 
By Theorem \ref{A.10.5} we deduce $T_q$ is trivial for all $q$ since $T_q(\emptyset)$ is trivial for all $q$.
\end{proof}

Let $P$ be a squareable, quasi-bounded, and regular cd-structure on a category $\cC$ with an initial object, and let $t_P$ be the topology associated with $P$.
For an object $T\in \cC$ let $\Lambda(T)$ be the representable presheaf of $\Lambda$-modules given by
\[
Z\mapsto \Lambda(T)(Z):=\Lambda\hom_{\cC}(Z,T).
\]
Let $\cH_{pre}'$ be the family of complexes of the form
\begin{equation}
\label{A.10.10.2}
\Lambda(Y')\rightarrow \Lambda(Y)\oplus \Lambda(X')\rightarrow \Lambda(X)
\end{equation}
induced by the distinguished squares of $P$, as in \eqref{A.8.13.1}.

\begin{cor}
\label{A.10.10}
Let $P$ be a squareable, quasi-bounded, and regular cd-structure on a category $\cC$ with an initial object, and let $t_P$ be the topology associated with $P$.
Then a complex of presheaves of $\Lambda$-modules $\cF$ is $\cH_{pre}'$-flasque in the sense of {\rm Definition \ref{A.8.14}(iv)} if and only if for every $X\in \cC$ and $i\in \Z$ there is an isomorphism
\[
\hom_{\mathbf{K}(\Psh(\cC,\Lambda))}(\Lambda(X)[i],\cF)
\cong
\hom_{\Deri(\Shv_{t}(\cC,\Lambda))}(a_t^*\Lambda(X)[i],a_t^*\cF).
\]
\end{cor}
\begin{proof}
This follows by replacing \cite[Lemma 3.5]{Vcdtop} with Lemma \ref{A.10.7} and adapting \cite[Lemma 4.3]{Vcdtop} in the proof of \cite[Theorem 3.7]{MR2415380}. 
\end{proof}

\begin{lem}
\label{A.10.11}
Let $P$ be a squareable, quasi-bounded, and regular cd-structure on a category $\cC$ with an initial object. 
Then a morphism $f:\cF\rightarrow \cF'$ of $\cH_{pre}'$-flasque complexes of presheaves of $\Lambda$-modules is a quasi-isomorphism if and only if 
$a_t^*f:a_t^*\cF\rightarrow a_t^*\cF'$ is a quasi-isomorphism.
\end{lem}
\begin{proof}
If $f$ is a quasi-isomorphism, then $a_t^*f$ is a quasi-isomorphism since the sheafification functor $a_t^*$ is exact.
Conversely, 
if $a_t^*f$ is a quasi-isomorphism, 
then owing to Corollary \ref{A.10.10} we have an isomorphism
\[
\hom_{\mathbf{K}(\Psh(\cC,\Lambda))}(\Lambda(X)[i],\cF)
\cong
\hom_{\mathbf{K}(\Psh(\cC,\Lambda))}(\Lambda(X)[i],\cF')
\]
for every $X\in \cC$ and $i\in \Z$.
This implies that $f$ is a quasi-isomorphism.
\end{proof}

\begin{exm}
\label{A.8.30}
Let $P$ be a squareable, quasi-bounded, and regular cd-structure on a category $\cC$ with an initial object, and let $t_P$ be the topology associated with $P$.
Recall the collection $\cG$ in Example \ref{A.8.28} and $\cH_{pre}'$ in \eqref{A.10.10.2}. 
Let $\cH'$ be the family of complexes of the form
\[
a_t^*\Lambda(Y')\rightarrow a_t^*\Lambda(Y)\oplus a_t^*\Lambda(X')\rightarrow a_t^*\Lambda(X)
\]
induced by the distinguishes squares of $P$, as in \eqref{A.8.13.1}. 
Suppose that the complex of $t_P$-sheaves $\cF$ is $\cH'$-flasque.
By adjunction this implies that $a_{t*}\cF$ is $\cH_{pre}'$-flasque.
Owing to Corollary \ref{A.10.10} we see that $\cF$ is $\cG$-local.
Thus $(\cG,\cH')$ is a bounded descent structure because $\cH'$ consists of bounded complexes of objects that are finite sums of objects of $\cG$.
This is genuinely different from the descent structure $(\cG,\cH)$ discussed in Example \ref{A.8.28}. 
%
\end{exm}

\newpage

\section{Sheaves with logarithmic transfers}
\label{sec.sheavestransfer}

The following three fundamental results are due to Voevodsky \cite[\S 6]{MVW}.
\begin{enumerate}
\item[(1)] The Nisnevich sheafification of a presheaf with transfers has a unique transfer structure.
\item[(2)] The category of Nisnevich sheaves with transfers $\Shvtrkl$ is a Grothendieck abelian category.
\item[(3)] For $\cF\in \Shvtrkl$ and $X\in Sm/k$ there is a canonical isomorphism (see below for the definition of $\Ztr$)
\[{
\rm Ext}_{\Shvtrkl}(\Ztr(X),F)\cong H_{Nis}^i(X,\cF).
\]
\end{enumerate}

In this section, we develop the log versions of (1)--(3) and show that the following topologies are compatible with log transfers:
\[
sNis,\;dNis,\;s\acute{e}t,\;d\acute{e}t,\;k\acute{e}t,\;\text{and }l\acute{e}t.
\]

\subsection{Category of presheaves with log transfers} 
Recall $\Lambda$ is a commutative unital ring.
The next step after introducing $lCor/k$ is to study presheaves with log transfers in the following sense.

\begin{df}
\label{A.5.37}
A {\it presheaf of $\Lambda$-modules with log transfers} is an additive presheaf of $\Lambda$-modules on the category of finite log correspondences $\LCor$. 
We let $\Pshltrkl$ denote the category of presheaves of $\Lambda$-modules with log transfers. \index{presheaf with log transfers} \index[notation]{Psh @ $\Pshltrkl$}
For $X\in lSm/k$, let $\Zltr(X)$ denote the representable presheaf of $\Lambda$-modules with log transfers given by \index[notation]{Lambdatr @$\Lambda_{\rm ltr}(X)$}
\[
Y\mapsto
\Zltr(X)(Y):=\lCor(Y,X)\otimes \Lambda.
\]
\end{df}

\begin{rmk}
\label{A.5.38}
Let $F$ be a presheaf of $\Lambda$-modules with log transfers. 
For every $X,Y\in lSm/k$, by additivity, there exists an induced paring
\[
\lCor(X,Y)\otimes F(Y)
\rightarrow 
F(X).
\]
This produces the "wrong way" maps parametrized by finite log correspondences.
\end{rmk}

\begin{exm}
Every constant presheaf of $\Lambda$-modules on $lSm/k$ has a canonical log transfer structure.
\end{exm}

\begin{exm}\label{example log transfer Ga}
The presheaves $\cO$ and $\cO^*$ are two examples of presheaves with transfers.
Let us extend these to presheaves with log transfers.
For $X\in lSm/k$, we define
\[
\cO(X):=\cO(\underline{X}),\;\cO^*(X):=\cO^*(\underline{X}).
\]
Owing to Lemma \ref{A.9.18}, $X$ is normal.
For any strict finite surjective morphism $V\rightarrow X$, as observed in \cite[Example 2.4]{MVW}, there is a canonical trace map and a norm map
\[
{\rm Tr}:\cO(V)\rightarrow \cO(X),\;N:\cO^*(V)\rightarrow \cO^*(X).
\]
If $V$ is a finite log correspondence from $X$ to $Y\in lSm/k$, we can define transfer maps via the composite
\[
V^*:\cO(Y)\rightarrow \cO(V)\stackrel{\rm Tr}\rightarrow \cO(X),\;V^*:\cO^*(Y)\rightarrow \cO^*(V)\stackrel{N}\rightarrow \cO^*(X)
\]
Let us check that these transfer maps are compatible with compositions of finite log correspondences. 
Since $X$ is reduced, we have the inclusions
\[
\cO(X)\subset \cO(X-\partial X),\;\cO^*(X)\subset \cO^*(X-\partial X).
\]
Thus for $Z\in lSm/k$ and $W\in \lCor(Y,Z)$ to check that $V^*\circ W^*=(W\circ V)^*$, owing to Lemma \ref{A.5.10}, it suffices to check that
\[
(V-\partial V)^*\circ (W-\partial W)^*=((W-\partial W)\circ (V-\partial V))^*.
\]
This is also observed in \cite[Example 2.4]{MVW}.
Thus $\cO$ and $\cO^*$ are both examples of presheaves with log transfers.
\end{exm}

\begin{exm}
Let $\cM^{\rm gp}$ denote the presheaf of abelian groups on $lSm/k$ given by \index[notation]{Mgr @ $\cM^{\rm gp}$}
\[
\cM^{\rm gp}(X):=\cM_X^{\rm gp}(X).
\]
By \cite[Theorem IV.3.5.1]{Ogu}, $X\in lSm/k$ is log regular.
By \cite[Theorem III.1.11.12]{Ogu}, the log structure on $X$ is the compactifying log structure associated with the inclusion $X-\partial X\rightarrow X$.
In particular, there is a canonical isomorphism
\[
\cM^{\rm gp}(X)\cong \cO^*(X-\partial X).
\]
For $Y\in lSm/k$ and $V\in \lCor(X,Y)$, we can now define a transfer map
\[
V^*\colon \cM^{\rm gp}(X)\rightarrow \cM^{\rm gp}(Y)
\]
because there is a transfer map 
$$
(V-\partial V)^*\colon \cO^*(X-\partial X)\rightarrow \cO^*(Y-\partial Y).
$$
As in the above paragraph, $V^*$ is compatible with compositions of correspondences.
Thus $\cM^{\rm gp}$ is an example of a presheaf with log transfers.
\end{exm}

Next, we give $\Pshltrkl$ the structure of a closed symmetric monoidal category following the work of Day \cite{MR0272852} in a general setting, and \cite{MR2435654} for presheaves with transfers.
\begin{df}
\label{A.5.40}
The tensor product of two representable presheaves with log transfers on $lSm/k$ is defined by 
\[
\Zltr(X)\otimes \Zltr(Y)
:=
\Zltr(X\times_{k}Y).
\]

More generally, \index[notation]{FtimesG @ $F\otimes G$}
we define the tensor product of $F,G\in \Pshltrkl$ by 
\[
F\otimes G
:=
\varinjlim_{X,Y}\Zltr(X)\otimes \Zltr(Y).
\]
Here we write $F$ and $G$ as colimits of representable presheaves
\[
F\cong \varinjlim_X \Zltr(X)
\text{ and }
G\cong \varinjlim_Y \Zltr(Y), 
\]
This gives a symmetric monoidal structure on $\Pshltrkl$.
\vspace{0.1in}
  
The internal hom object $\underline{\hom}_{\Pshltrkl}(F,G)$ is the presheaf with log transfers on $lSm/k$ defined by 
\[
X
\mapsto
\underline{\hom}_{\Pshltrkl}(F,G)(X)
:=
\hom_{\Pshltrkl}(F\otimes \Zltr(X),G).
\]
The functor $\underline{\hom}_{\Pshltrkl}(F,-)$ is right adjoint to $(-)\otimes F$.
\end{df}

Let $f\colon X\rightarrow Y$ be a morphism of fs log schemes log smooth over $k$. 
Then the graph $\gamma_f$ is a finite log correspondence from $X$ to $Y$. 
To show this we may assume $X$ and $Y$ are irreducible. 
Note that $\gamma_f$ is a closed subscheme (see Definition \ref{df::closedimmersion}) of $X\times Y$.
The morphism $\gamma_f\rightarrow X$ is an isomorphism. 
Thus $\gamma_f$ is an elementary log correspondence from $X$ to $Y$ by Example \ref{A.5.31}(4).
If $g\colon Y\rightarrow Z$ is another morphism of fs log schemes log smooth over $k$, 
then by construction $\gamma_g\circ \gamma_f=\gamma_{g\circ f}$. 
Thus there is a faithful functor
\begin{equation}
\label{A.5.40.1}
\gamma
\colon
lSm/k\rightarrow \LCor
\end{equation}
given by 
\[
X\mapsto X,
\,
(f\colon X\rightarrow Y) \mapsto \gamma_f.
\]

\begin{df}
\label{A.5.4}
Let $\Pshlogkl$ be the category of presheaves of $\Lambda$-modules on $lSm/k$. \index[notation]{Pshlog @ $\Pshlogkl$}
For an fs log scheme $X$ of finite type over $k$, 
we let $\Lambda(X)$ be the presheaf of $\Lambda$-modules on $lSm/k$ given by
\[
Y
\mapsto 
\Lambda(X)(Y):=\Lambda\hom_{lSm/k}(Y,X).
\]
The restriction (forgetting log transfer) functor
\[
\gamma^*
\colon
\Pshltrkl
\rightarrow 
\Pshlogkl
\]
is defined by $\gamma^*F(X):=F(\gamma(X))$ for $F\in \Pshltrkl$ and $X\in lSm/k$. 
\end{df}

According to \cite[Proposition I.5.1]{SGA4} there exist adjoint functors
\[
\begin{tikzcd}
\Pshlogkl\arrow[rr,shift left=1.5ex,"\gamma_\sharp "]\arrow[rr,"\gamma^*" description,leftarrow]\arrow[rr,shift right=1.5ex,"\gamma_*"']&&\Pshltrkl.
\end{tikzcd}
\]
By convention $\gamma_\sharp$ is left adjoint to $\gamma^*$ and $\gamma^*$ is left adjoint to $\gamma_*$. 
This sequence of adjoint functors induces a similar sequence of adjoint functors for the categories of chain complexes
\[
\begin{tikzcd}
\Co(\Pshlogkl)\arrow[rr,shift left=1.5ex,"\gamma_\sharp "]\arrow[rr,"\gamma^*" description,leftarrow]\arrow[rr,shift right=1.5ex,"\gamma_*"']&&\Co(\Pshltrkl).
\end{tikzcd}
\]

We shall use similar notations for the category $Cor/k$\index[notation]{Cork@$Cor/k$} of finite correspondences over the field $k$.
Note that the functor $\hat{\gamma}_*$ in \cite[Definition 10.1.1]{CD12} corresponds to our $\gamma^*$.
Let $\Pshtrkl$ denote the category of presheaves of $\Lambda$-modules with transfers.
For any smooth scheme $X$ of finite type over $k$, 
the representable presheaf of $\Lambda$-modules with transfers $\Ztr(X)$ on $Sm/k$ is given by 
\[
Y\mapsto
\Ztr(X)(Y)
:=
\Cor(Y,X)\otimes \Lambda.
\]

We write $\Pshkl$ for the category of presheaves of $\Lambda$-modules on $Sm/k$. 
For any smooth scheme $X$ of finite type over $k$, 
we denote by $\Lambda(X)$ the presheaf of $\Lambda$-modules on $Sm/k$ defined by 
\[
Y\mapsto 
\Lambda(X)(Y)
:=
\Lambda\hom_{Sm/k}(Y,X).
\]
In this setting, there is a faithful functor
\[
\gamma
\colon
Sm/k
\rightarrow 
Cor/k,
\]
and corresponding sequences of adjoint functors
\begin{equation}
\label{A.5.4.1}
\begin{tikzcd}
\Pshkl\arrow[rr,shift left=1.5ex,"\gamma_\sharp "]\arrow[rr,"\gamma^*" description,leftarrow]\arrow[rr,shift right=1.5ex,"\gamma_*"']&&\Pshtrkl,
\end{tikzcd}
\end{equation}
\[
\begin{tikzcd}
\Co(\Pshkl)\arrow[rr,shift left=1.5ex,"\gamma_\sharp "]\arrow[rr,"\gamma^*" description,leftarrow]\arrow[rr,shift right=1.5ex,"\gamma_*"']&&\Co({\Pshtrkl}).
\end{tikzcd}
\]

Next, we define functors between the categories of finite log correspondences and finite correspondences.

\begin{df} \index[notation]{omega @ $\omega$}
\label{A.4.5}
Let  
\[
\omega
\colon 
\LCor
\rightarrow 
Cor/k
\]
be the functor that sends an object $X$ to $X-\partial X$ and a morphism $V\in \lCor(X,Y)$ to $V-\partial V\in \Cor(X-\partial X,Y-\partial Y)$ for all $X,Y\in lSm/k$.
\vspace{0.1in}

Conversely, let 
\[
\lambda
\colon 
Cor/k
\rightarrow 
\LCor
\]
be the functor sending an object $X$ to $X\in lSm/k$ equipped with the trivial log structure and a morphism $V\in\Cor(X,Y)$ to $V\in l\Cor(X,Y)$ for all $X,Y\in Sm/k$, 
see Example \ref{A.5.31}. 
\end{df}

There are induced functors \index[notation]{omegastar @ $\omega^*$}  \index[notation]{lambdastar @ $\lambda^*$} 
\[
\omega^*
\colon 
\Pshtrkl\rightarrow \Pshltrkl 
\text{ and } 
\lambda^*
\colon 
\Pshltrkl\rightarrow \Pshtrkl
\]
given by $\omega^*F(X):=F(\omega(X))$ for $F\in \Pshtrkl$ and $X\in lSm/k$ and $\lambda^*G(Y):=G(\lambda(Y))$ for $G\in \Pshltrkl$ and $Y\in Sm/k$. 
Example \ref{A.5.31}(1) shows that $\lambda^*$ is left adjoint to $\omega^*$. 
\vspace{0.1in}
  
Again by \cite[Proposition I.5.1]{SGA4} there exist adjoint functors
\begin{equation}
\label{A.4.5.1}
\begin{tikzcd}
\Pshtrkl\arrow[rr,shift left=3ex,"\lambda_\sharp "]\arrow[rr,"\lambda^*\cong 
\omega_\sharp" description,leftarrow, shift left=1ex]\arrow[rr,shift right=1ex,"\lambda_*\cong 
\omega^*" description]\arrow[rr,shift right=3ex,"\omega_*"',leftarrow]&&\Pshltrkl.
\end{tikzcd}
\end{equation}
Our convention for such diagrams is that $\lambda_\sharp$ is left adjoint to $\lambda^*\cong \omega_\sharp$, which is left adjoint to $\lambda_*\cong \omega^*$, 
which is left adjoint to $\omega_*$. 
On chain complexes, we use the same notation for the induced adjoint functors \index[notation]{omegasharp @ $\omega_\sharp$}
\[
\begin{tikzcd}
\Co(\Pshtrkl)\arrow[rr,shift left=3ex,"\lambda_\sharp "]\arrow[rr,"\lambda^*\cong 
\omega_\sharp" description,leftarrow, shift left=1ex]\arrow[rr,shift right=1ex,"\lambda_*\cong \omega^*" description]\arrow[rr,shift right=3ex,"\omega_*"',leftarrow]&&\Co(\Pshltrkl).
\end{tikzcd}
\]

\begin{df}
Let 
\[
\omega
\colon 
lSm/k\rightarrow Sm/k
\]
be the functor that sends an object $X$ to $X-\partial X$ and a morphism $f\colon X\rightarrow Y$ to the naturally induced morphism $X-\partial X\rightarrow Y-\partial Y$ for all $X,Y\in lSm/k$.
\vspace{0.1in}

Conversely, let 
\[
\lambda
\colon 
Sm/k\rightarrow lSm/k
\]
be the functor sending $X\in Sm/k$ to $X$ equipped with the trivial log structure and a morphism $f\colon X\rightarrow Y$ to itself for all $X,Y\in Sm/k$. 
\end{df}

With these definitions, there are associated adjoint functors
\[
\begin{tikzcd}
\Pshkl\arrow[rr,shift left=3ex,"\lambda_\sharp "]\arrow[rr,"\lambda^*\cong 
\omega_\sharp" description,leftarrow, shift left=1ex]\arrow[rr,shift right=1ex,"\lambda_*\cong \omega^*" description]\arrow[rr,shift right=3ex,"\omega_*"',leftarrow]&&\Pshlogkl,
\end{tikzcd}\]
\[
\begin{tikzcd}
\Co(\Pshkl)\arrow[rr,shift left=3ex,"\lambda_\sharp "]\arrow[rr,"\lambda^*\cong 
\omega_\sharp" description,leftarrow, shift left=1ex]\arrow[rr,shift right=1ex,"\lambda_*\cong \omega^*" description]\arrow[rr,shift right=3ex,"\omega_*"',leftarrow]&&\Co(\Pshlogkl),
\end{tikzcd}\]
and a commutative diagram
\[
\begin{tikzcd}
lSm/k\arrow[d,"\omega"']\arrow[r,"\gamma"]& lCor/k\arrow[d,"\omega"]\\
Sm/k\arrow[r,"\gamma"]& Cor/k.
\end{tikzcd}
\]

\begin{df}
\label{A.5.24}
Suppose $f\colon Y\rightarrow X$ is a morphism of fs log schemes log smooth over $k$. 
Let $\Zltr(Y\rightarrow X)$ (resp.\ $\Lambda(Y\rightarrow X)$) denote the complex \index[notation]{LambdaCone @ $\Zltr(Y\rightarrow X)$}
\[ 
\Zltr(Y)\stackrel{\Zltr(f)}\longrightarrow \Zltr(X) \;\;(\text{resp.}\ \Lambda(Y)\stackrel{\Lambda(f)}\longrightarrow \Lambda(X)). 
\]

Suppose $\mathscr{X}$ is a simplicial fs log scheme over $k$. 
Let $\Zltr(\mathscr{X})$ (resp.\ $\Lambda(\mathscr{X})$) denote the complex associated with the simplicial object 
\[
i
\mapsto 
\Zltr(\mathscr{X}_i)
\;\;(\text{resp.}\
i\mapsto \Lambda(\mathscr{X}_i))
\] 
under the Dold-Kan correspondence.
\vspace{0.1in}

Similarly, 
when $Y\rightarrow X$ is a morphism of schemes smooth over $k$ and $\mathscr{X}$ is a simplicial scheme over $k$, 
we use the notations $\Lambda(Y\rightarrow X)$, $\Ztr(Y\rightarrow X)$, $\Lambda(\mathscr{X})$, and $\Ztr(\mathscr{X})$.
\end{df}

\subsection{Compatibility with log transfers}
This section discusses several notions of compatibility with log transfers for topologies on log schemes. 
We proceed by adapting to the log setting the analogous notions introduced by Cisinski-Deglise in \cite[Definition 10.3.2]{CD12}.

\begin{df}\label{A.8.1}
Let $t$ be a topology on $lSm/k$ and $\Lambda$ a commutative unital ring.
A presheaf $\cF$ with log transfers is called a {\it $t$-sheaf with log transfers} if $\gamma^*\cF\in \Pshlogkl$ is a $t$-sheaf. 
Here, $\gamma^*\cF(X):=\cF(\gamma(X))$ for $X\in lSm/k$.\index{sheaves with log transfers} \index[notation]{Shv @ $\Shvltrtkl$}
Let $\Shvltrtkl$ denote the category of $t$-sheaves with log transfers.
\vspace{0.1in}

Let $\Shvlogtkl$ be the category of $t$-sheaves on $lSm/k$ (without log transfers).\index[notation]{Shv @ $\Shvlogtkl$} 
The $t$-sheafification functor $a_t^*$ and the forgetful functor $a_{t*}$ form an adjoint functor pair
\[
a_t^*:\Pshlogkl\rightleftarrows \Shvlogtkl:a_{t*}.
\]

By abuse of notation,
let 
\[
a_{t*}:\Shvltrtkl\rightarrow \Pshltrkl
\]
be the inclusion functor, and let
\[
\gamma^*\colon \Shvltrtkl\rightarrow \Shvlogtkl
\]
be the restriction of the functor $\gamma^*\colon \Pshltrkl\rightarrow \Pshlogkl$. 
We note that 
\begin{equation}
\label{A.8.1.1}
\gamma^*a_{t*}\cong a_{t*}\gamma^*.
\end{equation}

For a topology $t$ on $Sm/k$,
let $\Shvtrtkl$ be the category of $t$-sheaves on $Sm/k$ with transfers.
\end{df}

The notion of a topology on $Sm/k$ being compatible with transfers was introduced in \cite[Definition 10.3.2]{CD12}.
We adopt this notion to our log setting.
\begin{df}
\label{A.8.2}
A topology $t$ on $lSm/k$ is \emph{compatible with log transfers} \index{topology!compatible with log transfers}
if for any $t$-hypercover $\mathscr{X}\rightarrow X$, 
the induced morphism
\[
a_t^*\gamma^*\Zltr(\mathscr{X})\rightarrow a_t^*\gamma^*\Zltr(X)
\]
of complexes of presheaves is a quasi-isomorphism (see also Definition \ref{A.5.24}).
\vspace{0.1in}

Following \cite[Definition 10.3.5]{CD12}, we say that $t$ is {\it mildly compatible with log transfers} if $\gamma_*a_{t*}\cF$ is a $t$-sheaf with log transfers for any $t$-sheaf $\cF$.
Recall that $\gamma_*$ is a right adjoint of $\gamma^*$. \index{topology!mildly compatible with log transfers}
There is a description of $\gamma_*$: For every presheaf $\cG$ and $X\in lSm/k$,
\[
\gamma_*\cG(X)=\hom_{\Pshlogkl}(\gamma^*\Zltr(X),\cG).
\]
\end{df}

\begin{rmk}
These notions are equivalent to each other if $t$ is a topology associated with a cd-structure, see Proposition \ref{A.8.13}.
\end{rmk}


\begin{rmk}
According to \cite[10.4.1]{CD12}, the Nisnevich topology on $Sm/k$ is compatible with transfers.
One goal in this section is to generalize this to our log setting,
i.e., we will show that the strict Nisnevich topology on $lSm/k$ is compatible with log transfers.
\end{rmk}

\begin{prop}
\label{A.8.3}
Let $t$ be a topology on $lSm/k$ compatible with log transfers. Then $t$ is mildly compatible with log transfers.
\end{prop}
\begin{proof}
Let $\mathscr{X}\rightarrow X$ be a \v{C}ech $t$-hypercover, and let $\cF$ be a $t$-sheaf. 
To show that $\gamma_*a_{t*}\cF$ is a $t$-sheaf, we need to show that the induced sequence of $\Lambda$-modules
\[
\begin{split}
0\rightarrow \hom_{\Pshltrkl}(\Zltr(X),\gamma_*a_{t*}\cF)\rightarrow &\hom_{\Pshltrkl}(\Zltr(\mathscr{X}_0),\gamma_*a_{t*}\cF)\\
\rightarrow &\hom_{\Pshltrkl}(\Zltr(\mathscr{X}_1),\gamma_*a_{t*}\cF)
\end{split}
\]
is exact. 
This is equivalent to showing exactness of the sequence
\begin{equation}
\label{A.8.3.1}
\begin{split}
0\rightarrow \hom_{\Shvltrtkl}(a_t^*\gamma^*\Zltr(X),\cF)\rightarrow &\hom_{\Shvltrtkl}(a_t^*\gamma^*\Zltr(\mathscr{X}_0),\cF)\\
\rightarrow &\hom_{\Shvltrtkl}(a_t^*\gamma^*\Zltr(\mathscr{X}_1),\cF).
\end{split}
\end{equation}
This follows since by assumption we have a quasi-isomorphism
\[
a_t^*\gamma^*\Zltr(\mathscr{X})\rightarrow a_t^*\gamma^*\Zltr(X).
\]
\end{proof}

\begin{prop}
\label{A.8.12}
Let $t$ be a topology on $lSm/k$ mildly compatible with log transfers. 
Then for any $t$-cover $Y\rightarrow X$, the induced morphism of $t$-sheaves $a_t^*\gamma^*\Zltr(Y)\rightarrow a_t^*\gamma^*\Zltr(X)$ is an epimorphism.
\end{prop}
\begin{proof}
As in the proof of Proposition \ref{A.8.3}, this follows because the sequence \eqref{A.8.3.1} is exact.
\end{proof}

\begin{prop}
\label{A.8.4}
Let $t$ be a topology on $lSm/k$ mildly compatible with log transfers. 
Then the functor
\[
\gamma^*\colon \Shvltrtkl\rightarrow \Shvlogtkl
\]
admits a right adjoint $\gamma_*$ that commutes with $a_{t*}$.
\end{prop}
\begin{proof}
Due to the assumption that $t$ is mildly compatible with log transfers, 
by abuse of notation, we have the functor
\[
\gamma_*\colon \Shvlogtkl\rightarrow \Shvltrtkl.
\]
That is, the restriction of the functor $\gamma_*\colon\Pshlogkl\rightarrow \Pshltrkl$. 
We note that the functors $a_{t*}$ and $\gamma_*$ commute.
\vspace{0.1in}

For $\cF\in \Shvltrtkl$ and $\cG\in \Shvlogtkl$ it remains to note that there are natural isomorphisms of $\Lambda$-modules
\[
\begin{split}
&\hom_{\Shvlogtkl}(\gamma^*\cF,\cG)\\
\cong &\hom_{\Pshlogkl}(a_{t*}\gamma^*\cF,a_{t*}\cG)\quad \text{(since }a_{t*}\text{ is fully faithful)}\\
\cong &\hom_{\Pshlogkl}(\gamma^*a_{t*}\cF,a_{t*}\cG)\quad \text{(since }\gamma^*\text{ commutes with }a_{t*},\text{ see }\eqref{A.8.1.1})\\
\cong &\hom_{\Pshltrkl}(a_{t*}\cF,\gamma_* a_{t*}\cG)\\
\cong &\hom_{\Pshltrkl}(a_{t*}\cF,a_{t*}\gamma_*\cG)\quad \text{(since }\gamma_*\text{ commutes with }a_{t*})\\
\cong &\hom_{\Shvltrtkl}(\cF,\gamma_*\cG)\quad \text{(since }a_{t*}\text{ is fully faithful)}.
\end{split}
\]
\end{proof}

Let us prove an analog of \cite[Lemma 6.16]{MVW}.

\begin{lem}
\label{A.8.23}
Let $t$ be a topology on $lSm/k$ that is mildly compatible with log transfers.
Suppose that $p\colon Y'\rightarrow Y$ is a $t$-cover and $f\in \lCor(X,Y)$ is a finite log correspondence, where $X,Y,Y'\in lSm/k$.
Then there exists a $t$-cover $p'\colon X'\rightarrow X$ and a commutative diagram in $lCor/k$
\begin{equation}
\label{A.8.23.1}
\begin{tikzcd}
X'\arrow[d,"p'"']\arrow[r,"f'"]&Y'\arrow[d,"p"]\arrow[d]\\
X\arrow[r,"f"]&Y.
\end{tikzcd}
\end{equation}
\end{lem}
\begin{proof}
By Proposition \ref{A.8.12}, the induced morphism 
\[
a_t^*\gamma^*\Zltr(Y')\rightarrow a_t^*\gamma^*\Zltr(Y)
\]
in $\Shvltrtkl$ is an epimorphism.
The implication (ii)$\Rightarrow$(i) in \cite[Proposition II.5.1]{SGA4} shows that $X\times_{\Zltr(Y)}\Zltr(Y')\rightarrow X$ is a covering sieve.
From this, we immediately deduce the existence of $p'$ along with \eqref{A.8.23.1}.
\end{proof}

Using this lemma, we prove an analog of \cite[Theorem 6.17]{MVW} as follows.

\begin{lem}
\label{A.8.6}
Let $t$ be a topology on $lSm/k$ that is mildly compatible with log transfers. 
Then for every $\cF\in \Pshltrkl$ there exists an object $\cG\in \Shvltrtkl$ together with a morphism $u\colon \cF\rightarrow a_{t*}\cG$ in $\Pshltrkl$ such that 
$\gamma^*u$ is isomorphic to the sheafification $\gamma^*\cF\rightarrow a_{t*}a_t^*\gamma^*\cF$.
The pair $(\cG,u)$ is unique up to isomorphism.
\end{lem}
\begin{proof}
Let us first construct $\cG\in \Shvltrtkl$.
The unit $\cF\rightarrow \gamma_*a_{t*}a_t^*\gamma^*\cF$ induces a morphism
\[
\eta\colon \gamma^*\cF\rightarrow \gamma^*\gamma_*a_{t*}a_t^*\gamma^*\cF
\]
Here $\gamma^*\gamma_*a_{t*}a_t^*\gamma^*\cF$ is a $t$-sheaf because $t$ is mildly compatible with log transfers.  
Hence according to the universal property of the sheafification functor, the latter morphism admits a factorization
\[
\gamma^*\cF\rightarrow a_{t*}a_t^* \gamma^* \cF\stackrel{\psi}\rightarrow \gamma^*\gamma_*a_{t*}a_t^*\gamma^*\cF
\]
that is unique up to isomorphism. 
Next, we consider the commutative diagram
\[
\begin{tikzcd}
\gamma^*\cF\arrow[r,"ad"']\arrow[d,"ad"]&\gamma^*\gamma_*\gamma^*\cF\arrow[d,"ad"]\arrow[r,"ad'"]&\gamma^*\cF\arrow[d,"ad"]\\
a_{t*}a_t^*\gamma^*\cF\arrow[r,"\psi"]&\gamma^*\gamma_*a_{t*}a_t^*\gamma^*\cF\arrow[r,"ad'"]&a_{t*}a_t^*\gamma^*\cF.
\end{tikzcd}
\]
Here $ad$ (resp.\ $ad'$) is induced by a unit (resp.\ counit).
Let $\psi'$ be the composition of the lower horizontal morphisms.
Since the composition of the upper horizontal morphisms is the identity by the unit-counit identity, there is a commutative diagram
\[
\begin{tikzcd}
&
\gamma^*\cF\arrow[rd,"ad"]\arrow[ld,"ad"']
\\
a_{t*}a_t^*\gamma^*\cF\arrow[rr,"\psi'"]&&a_{t*}a_t^*\gamma^*\cF.
\end{tikzcd}
\]
The diagram still commutes if we replace $\psi'$ by ${\rm id}$.
By the universal property of sheafification, the morphism 
\[
a_{t*}a_t^*\gamma^*\cF\rightarrow a_{t*}a_t^*\gamma^*\cF
\] 
making the above diagram commutative is unique, 
and we deduce $\psi'={\rm id}$.
It follows that $\psi$ is a monomorphism.
\vspace{0.1in}

Let $f\in \lCor(X,Y)$ be a finite log correspondence, where $X,Y\in lSm/k$, and set $\cG=a_t^*\gamma^*\cF$.
Then $f$ and $\psi$ induce a diagram of $\Lambda$-modules
\[
\begin{tikzcd}
\cG(Y)\arrow[r,"{\psi(Y)}"]& \hom_{\Shvlogtkl}(a_t^*\gamma^*\Zltr(Y),\cG)\arrow[d,"\alpha_f"]\\
\cG(X)\arrow[r,"{\psi(X)}"]& \hom_{\Shvlogtkl}(a_t^*\gamma^*\Zltr(X),\cG).
\end{tikzcd}
\]
Since $\psi$ is a monomorphism, to show the log transfer structure on $\gamma_*\alpha_{t*}\cG$ induces a log transfer structure on $\cG$ it suffices to show that
\begin{equation}
\label{A.8.6.2}
\im (\alpha_f\circ \psi(Y)) \subset \im \psi(X).
\end{equation}

For every $a\in \cG(Y)$ there exists an element $b'\in \cF(Y')$ for some $t$-cover $q\colon Y'\rightarrow Y$ such that the images of $a$ and $b'$ in $\cG(Y')$ coincide with some $a'\in \cG(Y')$.
According to Lemma \ref{A.8.23} there exists a $t$-cover $p'\colon X'\rightarrow X$ with a commutative diagram in $lCor/k$
\begin{equation}
\label{A.8.6.1}
\begin{tikzcd}
X'\arrow[d,"p'"']\arrow[r,"f'"]&Y'\arrow[d,"p"]\arrow[d]\\
X\arrow[r,"f"]&Y.
\end{tikzcd}
\end{equation}
Then there is a diagram of $\Lambda$-modules
\[
\begin{tikzcd}
\cG(Y')\arrow[r,"{\psi(Y')}"]& \hom_{\Shvlogtkl}(a_t^*\gamma^*\Zltr(Y'),\cG)\arrow[d,"\alpha_{f'}"]\\
\cG(X')\arrow[r,"{\psi(X')}"]& \hom_{\Shvlogtkl}(a_t^*\gamma^*\Zltr(X'),\cG).
\end{tikzcd}
\]
Suppose that $(\alpha_{f'}\circ \psi(Y'))(a')=\psi(X')(d')$ for some $d'\in \cG(X')$.
The two images of $a'$ in $\cG(Y'\times_Y Y')$ are equal since $\cG$ is a presheaf.
Thus the two images of $\psi(X')(d')$ in $\hom_{\Shvlogtkl}(a_t^*\gamma^*\Zltr(X'\times_X X'),\cG)$ are equal.
Since $\psi$ is a monomorphism, this implies that the two images of $d'$ in $\cG(X'\times_X X')$ are equal.
Thus there exists $d\in \cG(X)$ whose image in $\cG(X')$ is equal to $d'$ since $\cG$ is a sheaf.
Using the exactness of \eqref{A.8.3.1}, we see that
\[
(\alpha_{f}\circ \psi(Y))(a) = \psi(X)(d).
\]
Hence let us show that $(\alpha_{f'}\circ \psi(Y'))(a')=\psi(X')(d')$ for some $d'\in \cG(X')$.
\vspace{0.1in}

The morphism $\eta\colon \gamma^*\cF\rightarrow \gamma^*\gamma_*a_{t*}a_t^*\gamma^*\cF$ induces a commutative diagram of $\Lambda$-modules
\begin{equation}
\label{equation:FGdiagram}
\begin{tikzcd}
\cF(Y')\arrow[r,"\eta(Y')"]\arrow[d,"\cF(f')"']&\hom_{\Shvlogtkl}(a_t^*\gamma^*\Zltr(Y'),\cG)\arrow[d,"\alpha_{f'}"]\\
\cF(X')\arrow[r,"\eta(X')"]&\hom_{\Shvlogtkl}(a_t^*\gamma^*\Zltr(X'),\cG).
\end{tikzcd}
\end{equation}
This implies that $(\alpha_{f'}\circ \psi(Y'))(a')=\psi(X')(d')$ for some $d'\in \cG(X')$.
We deduce the claimed log transfer structure on $\cG$.
\vspace{0.1in}

Since $\psi$ is a monomorphism, 
combining the diagram \eqref{equation:FGdiagram} for $f\colon X\rightarrow Y$ together with \eqref{A.8.6.2} we obtain a canonical commutative diagram of $\Lambda$-modules
\[
\begin{tikzcd}
\cF(Y)\arrow[d,"\cF(f)"']\arrow[r]&
\cG(Y)\arrow[d,"\cG(f)"]\arrow[d]
\\
\cF(X)\arrow[r]&
\cG(X).
\end{tikzcd}
\]
Thus we obtain a morphism $u\colon \cF\rightarrow a_{t*}\cG$ of presheaves with log transfers such that $\gamma^*u$ is isomorphic to the sheafification 
$\gamma^*\cF\rightarrow a_{t*}a_t^*\gamma^*\cF$.
\vspace{0.1in}

For uniqueness, 
suppose $\cG'$ is a $t$-sheaf with log transfers that is equipped with a morphism $u'\colon \cF\rightarrow a_{t*}\cG'$ of presheaves with log transfers such that $\gamma^*u'$ 
is isomorphic to the sheafification $\gamma^*\cF\rightarrow a_{t*}a_t^*\gamma^*\cF$. 
Since $\cG$ and $\cG'$ agree on objects, 
it suffices to check that $\cG(f)=\cG'(f)$ for any $f\in \lCor(X,Y)$ where $X,Y\in lSm/k$. 
Using the commutative diagram \eqref{A.8.6.1}, we see this question is $t$-local on $X$ and $Y$.
Hence after shrinking $Y$ we may assume that for any $a\in \cG(Y)$, there is an element $b\in \cF(Y)$ whose image in $\cG(Y)$ is equal to $a$.
The morphisms $u\colon \cF\rightarrow a_{t*} \cG$ and $u'\colon \cF\rightarrow a_{t*}\cG'$ induce a commutative diagram of $\Lambda$-modules
\[
\begin{tikzcd}
\cG'(Y)\arrow[r,leftarrow]\arrow[d,"{\cG'(f)}"']&\cF(Y)\arrow[d,"\cF(f)"']\arrow[r]&\cG(Y)\arrow[d,"\cG(f)"]\\
\cG'(X)\arrow[r,leftarrow,"{u'(X)}"]&\cF(X)\arrow[r,"u(X)"]&\cG(X).
\end{tikzcd}
\]
With the identification of $\cG$ and $\cG'$ on objects, the homomorphisms $u(X)$ and $u'(X)$ become equal. 
This allows us to conclude 
\[
\cG(f)(a)=u(X)(\cF(f)(b))=u'(X)(\cF(f)(b))=\cG'(f)(a).
\]
\end{proof}

\begin{prop}
\label{A.8.5}
Let $t$ be a topology on $lSm/k$ mildly compatible with log transfers. 
Then the functor 
\[
a_{t*}\colon \Shvltrtkl\rightarrow \Pshltrkl
\] 
admits a left adjoint functor $a_t^*$.
Moreover, the functor $a_t^*$ is exact.
\end{prop}
\begin{proof}
For $\cF\in \Pshltrkl$ we define $a_t^*\cF\colon =\cH$, 
where $\cH\in \Shvltrtkl$ is constructed as in Lemma \ref{A.8.6}. 
The uniqueness part in Lemma \ref{A.8.6} ensures that we obtain a functor 
\[
a_t^*\colon \Pshltrkl\rightarrow \Shvltrtkl.
\]
It remains to show $a_t^*$ is a left adjoint of $a_{t*}$. 
Again, by uniqueness, the morphism $u\colon \cF\rightarrow a_{t*}\cH$ in $\Pshltrkl$ constructed in Lemma \ref{A.8.6} gives rise to a natural transformation
\[
{\rm id}\rightarrow a_{t*}a_t^*.
\]
For $\cG\in \Shvltrtkl$, the uniqueness part also implies that there is an isomorphism
\[
a_t^*a_{t*}\cG\cong \cG.
\]
By functoriality of this isomorphism, we obtain a natural transformation
\[
a_t^*a_{t*}\rightarrow {\rm id}.
\]
Owing to the uniqueness part in Lemma \ref{A.8.6}, 
the unit and counit triangle identities follow from the corresponding identities for $a_t^*\colon \Pshlogkl\rightarrow \Shvlogtkl$ and $a_{t*}\colon \Shvlogtkl\rightarrow \Pshlogkl$.
\vspace{0.1in}

Since $a_t^*$ is a left adjoint, $a_t^*$ is right exact.
It remains to show that $a_t^*$ is left exact.
Let $f\colon \cF\rightarrow \cF'$ be a monomorphism of presheaves with log transfers.
The functor $\gamma^*$ is exact since $\gamma^*$ has both left and right adjoints, and the functor
\[
a_t^*\colon \Pshlogkl\rightarrow \Shvlogtkl
\]
is known to be exact.
It follows that
\[
a_t^*\gamma^*f\colon a_t^*\gamma^*\cF\rightarrow a_t^*\gamma^*\cF'
\]
is a monomorphism of $t$-sheaves.
Since $a_t^*$ commutes with $\gamma^*$ by Proposition \ref{A.8.4},
\[
\gamma^*a_t^*f\colon \gamma^*a_t^*\cF\rightarrow \gamma^*a_t^*\cF'
\]
is a monomorphism of $t$-sheaves.
Let $\cG$ be a kernel of $a_t^*f\colon a_t^*\cF\rightarrow a_t^*\cF'$.
It follows that $\gamma^*\cG=0$.
This means that $\cG(X)=0$ for all $X\in lSm/k$, so $\cG=0$.
Thus $a_t^*f$ is a monomorphism.
\end{proof}

\begin{prop}
\label{A.8.9}
Let $t$ be a topology on $lSm/k$ mildly compatible with log transfers. 
Then the functor
\[
\gamma^*\colon \Shvltrtkl\rightarrow \Shvlogtkl
\]
admits a left adjoint functor $\gamma_\sharp$.
\end{prop}
\begin{proof}
For $\cF\in \Shvlogtkl$ and $\cG\in \Shvltrtkl$, we have isomorphisms
\[
\begin{split}
&\hom_{\Shvltrtkl}(a_t^*\gamma_\sharp a_{t*}\cF,\cG)\\
\cong &\hom_{\Pshltrkl}(a_{t*}\cF,\gamma^*a_{t*}\cG)\\
\cong &\hom_{\Pshltrkl}(a_{t*}\cF,a_{t*}\gamma^*\cG)\quad \text{(since }\gamma^*\text{ and }a_{t*}\text{ commute)}\\
\cong &\hom_{\Shvltrtkl}(\cF,\gamma^*\cG)\quad \text{(since }a_{t*}\text{ is fully faithful}).
\end{split}
\]
By functoriality in $\cF$ and $\cG$, it follows that $a_t^*\gamma_\sharp a_{t*}$ is a left adjoint of $\gamma^*$.
\end{proof}

We recall the definition of Grothendieck abelian category (see e.g., \cite[Chapter V]{MR0389953} for details).
An abelian category $\cA$ is called a {\it Grothendieck abelian category}\index{Grothendieck abelian category} with the following conditions.
\begin{enumerate}
\item[(i)] $\cA$ admits small sums.
\item[(ii)] Filtered colimits of exact sequences in $\cA$ is exact.
\item[(iii)] There is a generator of $\cA$, i.e., there is an element of $G\in \cA$ such that the functor $\hom_{\cA}(G,-)\colon \cA\rightarrow {\bf Set}$ is faithful.
\end{enumerate}

\begin{prop}
\label{A.8.7}
Let $t$ be a topology on $lSm/k$ mildly compatible with log transfers. Then $\Shvltrtkl$ is a Grothendieck abelian category.
\end{prop}
\begin{proof}
Let us check the above axioms of a Grothendieck abelian category.
The categories  $\Pshltrkl$ and $\Shvlogtkl$ are Grothendieck abelian categories, and $\gamma^*\colon \Shvltrtkl\rightarrow \Shvlogtkl$ is exact since $\gamma^*$ has both left and right adjoints.
Thus $\Shvltrtkl$ is an abelian category, and we have (i) and (ii). 
Since by the Yoneda lemma
\[
\hom_{\Shvltrtkl}(a_t^*\Zltr(X),\cF)\cong \hom_{\Pshltrkl}(\Zltr(X),\cF)\cong \cF(X),
\] 
the object $\bigoplus_{X\in lSm/k} a_t^*\Zltr(X)$ is a generator of $\Shvltrtkl$. Thus we have (iii).
\end{proof}
  
In conclusion, 
if $t$ is a topology on $lSm/k$ mildly compatible with log transfers, 
there exist adjunctions between Grothendieck abelian categories
\begin{equation}
\label{A.8.7.1}
\begin{tikzcd}
\Pshlogkl\arrow[d,shift left=0.75ex,leftarrow,"a_{t*}"]\arrow[d,shift right=0.75ex,"a_{t}^*"']
\arrow[r,shift left=1.5ex,"\gamma_\sharp"]\arrow[r,leftarrow,"\gamma^*" description]
\arrow[r,shift right=1.5ex,"\gamma_*"']&\Pshltrkl
\arrow[d,shift left=0.75ex,leftarrow,"a_{t*}"]\arrow[d,shift right=0.75ex,"a_{t}^*"']   \\
\Shvlogtkl\arrow[r,shift left=1.5ex,"\gamma_\sharp"]
\arrow[r,leftarrow,"\gamma^*" description]\arrow[r,shift right=1.5ex,"\gamma_*"']
&\Shvltrtkl
\end{tikzcd}
\end{equation}
such that the following properties hold:
\begin{enumerate}
\item[(i)] $a_t^*$ is left adjoint to the corresponding $a_{t*}$
\item[(ii)] $\gamma^*$ is left (resp.\ right) adjoint to the corresponding $\gamma_*$ (resp.\ $\gamma_\sharp$)
\item[(iii)] $a_{t*}\gamma_*\cong \gamma_*a_{t*}$
\item[(iv)] $a_{t*}\gamma^*\cong \gamma^*a_{t*}$
\end{enumerate}

\begin{prop}
\label{A.8.11}
Let $t$ be a topology on $lSm/k$.
Suppose that for every $t$-cover $f\colon Y\rightarrow X$ the complex
\begin{equation}
\label{A.8.11.1}
\cdots\rightarrow a_t^*\gamma^*\Zltr(Y\times_X Y)\rightarrow a_t^*\gamma^*\Zltr(Y)\rightarrow a_t^*\gamma^*\Zltr(X)\rightarrow 0
\end{equation}
associated with the \v{C}ech nerve of $f$ is acyclic.
Then $t$ is compatible with log transfers.
\end{prop}
\begin{proof}
Let $p\colon \mathscr{X}\rightarrow X$ be a $t$-hypercover.
To check that
\[
a_t^*\gamma^*\Zltr(\mathscr{X})\rightarrow a_t^*\gamma^*\Zltr(X)
\]
is a quasi-isomorphism it suffices to check that
\[
a_t^*\gamma^*\Zltr(\mathscr{X}_i)\rightarrow \cdots \rightarrow a_t^*\gamma^*\Zltr(\mathscr{X}_0)\rightarrow a_t^*\gamma^*\Zltr(X)\rightarrow 0
\]
is exact for all integer $i\geq 0$.
Hence we may assume that $\mathscr{X}$ is a bounded $t$-hypercover of height $n$.
\vspace{0.1in}

To proceed, we use the technique in the proof of \cite[Proposition A4]{MR2034012}.
If $n=0$, then $\mathscr{X}$ is a \v{C}ech nerve, so we are done.
Suppose that $n>0$, and consider the bounded $t$-hypercover 
\[
\mathscr{Y}:={\rm cosk}_{n-1} \mathscr{X}
\]
of height at most $n-1$.
By induction we have that $a_t^*\gamma^*\Zltr(\mathscr{Y})\rightarrow a_t^*\gamma^*\Zltr(X)$ is a quasi-isomorphism.
The \v{C}ech nerve of the canonical morphism $\mathscr{X}\rightarrow \mathscr{Y}$ is a bisimplicial complex;
let $\mathscr{D}$ denote its diagonal.
Since \eqref{A.8.11.1} is acyclic for every $t$-cover, there is a canonically induced quasi-isomorphism
\[
a_t^*\gamma^*\Zltr(\mathscr{D})\rightarrow a_t^*\gamma^*\Zltr(\mathscr{Y}).
\]
In particular, 
$a_t^*\gamma^*\Zltr(\mathscr{D})\rightarrow a_t^*\gamma^*\Zltr(X)$ is a quasi-isomorphism.
Since $\mathscr{X}$ is a retract of $\mathscr{D}$, 
see the proof of \cite[Proposition A4]{MR2034012},
we deduce that
\[
a_t^*\gamma^*\Zltr(\mathscr{X})\rightarrow a_t^*\gamma^*\Zltr(X)
\]
is a quasi-isomorphism.
\end{proof}

\subsection{Derived categories of sheaves with log transfers}\label{Subsection:derivedcategories}
For miscellaneous facts about model structures on chain complexes we refer the reader to Appendix \ref{AppendixB}.
Let $t$ be a topology on $lSm/k$ mildly compatible with log transfers. 
Passing to chain complexes in \eqref{A.8.7.1} we obtain the diagram of adjoint functors 
\begin{equation}
\label{A.8.0.1}
\begin{tikzcd}
\Co(\Pshlogkl)\arrow[d,shift left=0.75ex,leftarrow,"a_{t*}"]\arrow[d,shift right=0.75ex,"a_{t}^*"']
\arrow[r,shift left=0.75ex,"\gamma_\sharp"]\arrow[r,shift right=0.75ex,"\gamma^*"',leftarrow]&\Co(\Pshltrkl)\arrow[d,shift left=0.75ex,leftarrow,"a_{t*}"]\arrow[d,shift right=0.75ex,"a_{t}^*"']   \\
\Co(\Shvlogtkl)\arrow[r,shift left=0.75ex,"\gamma_\sharp"]\arrow[r,leftarrow,shift right=0.75ex,"\gamma^*"']
&\Co(\Shvltrtkl).
\end{tikzcd}
\end{equation}
We refer to Example \ref{A.8.28} for the descent structures on $\Pshlogkl$ and $\Shvlogtkl$. 
The functors $\gamma^*$ in \eqref{A.8.7.1} are exact since they have both left and right adjoints. 
If the topology $t$ is compatible with log transfers, 
then for any $t$-hypercover $\mathscr{X}\rightarrow X$, 
the cone of the morphism $a_t^*\gamma^*\Zltr(\mathscr{X})\rightarrow a_t^*\gamma^*\Zltr(X)$ is quasi-isomorphic to $0$. 
Thus by Proposition \ref{A.8.17}, we also obtain descent structures on $\Pshltrkl$ and $\Shvltrtkl$, and the $(\gamma_\sharp,\gamma^*)$'s are Quillen pairs with respect to the said descent model structures. 
Note also that the $a_t^*$'s satisfy the conditions of Proposition \ref{A.8.16} due to exactness. 
Owing to Propositions \ref{A.8.16} and \ref{A.8.17}, the $(a_t^*,a_{t*})$'s are Quillen pairs with respect to the descent model structures. 
\vspace{0.1in}

Since $\gamma^*$ preserves quasi-isomorphisms and monomorphisms, 
the adjunctions
\[
\gamma^*\colon \Co(\Pshltrkl)\rightleftarrows \Co(\Pshlogkl)\colon \gamma_*,
\]
\[
\gamma^*\colon \Co(\Shvltrtkl)\rightleftarrows \Co(\Shvltrtkl)\colon \gamma_*
\]
are Quillen pairs with respect to the injective model structures; see Definition \ref{A.8.26} for injective model structures.
Thus we have an induced diagram of adjunctions between triangulated categories 
\begin{equation}
\label{A.8.0.4}
\begin{tikzcd}
\Deri(\Pshlogkl)\arrow[d,shift left=0.75ex,leftarrow,"Ra_{t*}"]\arrow[d,shift right=0.75ex,"a_{t}^*"']
\arrow[r,shift left=1.5ex,"L\gamma_\sharp"]\arrow[r,"\gamma^*" description,leftarrow]\arrow[r,shift right=1.5ex,"R\gamma_*"']&\Deri(\Pshltrkl)\arrow[d,shift left=0.75ex,leftarrow,"Ra_{t*}"]\arrow[d,shift right=0.75ex,"a_{t}^*"']   \\
\Deri(\Shvlogtkl)\arrow[r,shift left=1.5ex,"L\gamma_\sharp"]\arrow[r,"\gamma^*" description,leftarrow]\arrow[r,shift right=1.5ex,"R\gamma_*"']
&\Deri(\Shvltrtkl),
\end{tikzcd}\end{equation}
where
\begin{enumerate}
\item[(i)] $L\gamma^*\cong \gamma^*\cong R\gamma^*$ since $\gamma^*$ is exact,
\item[(ii)] $a_t^*\cong La_t^*$ since $a_t^*$ is exact, see Proposition \ref{A.8.5} for the log transfer case.
\end{enumerate}
\vspace{0.1in}

Suppose that $\cW$ is an essentially small class of morphisms in $lSm/k$ that is stable by isomorphisms, products, and compositions. 
We consider $\cW$ as classes of morphisms of representable objects in $\Pshlogkl$, $\Pshltrkl$, $\Shvlogtkl$, and $\Shvltrtkl$.
Note that $a_t^*$ preserves $\cW$ because the object of $\Shvlogtkl$ (resp.\ $\Shvltrtkl$) representable by $X\in lSm/k$ is $a_t^*\Lambda(X)$ (resp.\ $a_t^*\Zltr(X)$).
The four adjunctions in \eqref{A.8.0.1} are Quillen adjunctions with respect to the $\cW$-local descent model structures furnished by Propositions \ref{A.8.20} and \ref{A.8.21}. 
Thus we have the following induced diagram of adjunctions between triangulated categories
\begin{equation}
\label{A.8.0.2}
\begin{tikzcd}
\Deri_{\cW}(\Pshlogkl)\arrow[d,shift left=0.75ex,leftarrow,"Ra_{t*}"]\arrow[d,shift right=0.75ex,"a_{t}^*"']
\arrow[r,shift left=0.75ex,"L\gamma_\sharp"]\arrow[r,shift right=0.75ex,"R\gamma^*"',leftarrow]&\Deri_{\cW}(\Pshltrkl)\arrow[d,shift left=0.75ex,leftarrow,"Ra_{t*}"]\arrow[d,shift right=0.75ex,"a_{t}^*"']   \\
\Deri_{\cW}(\Shvlogtkl)\arrow[r,shift left=0.75ex,"L\gamma_\sharp"]\arrow[r,shift right=0.75ex, "R\gamma^*"',leftarrow]
&\Deri_{\cW}(\Shvltrtkl).
\end{tikzcd}
\end{equation} 
Here $a_t^*\cong La_t^*$ since $a_t^*$ is exact by Proposition \ref{A.8.5} and  $a_t^*$ preserves $\cW$.

\begin{df}
\label{A.8.10}
Let $t$ be a topology on $lSm/k$, and let $\cF$ be a complex of $t$-sheaves with log transfers. 
For $X\in lSm/k$ and $i\in \Z$ we define hypercohomology groups by setting \index[notation]{Ht @ $\bH_t^i$}
\[
\bH_t^i(X,\cF) 
:=
\bH_t^i(X,\gamma^*\cF).
\]
\end{df}

\begin{prop}\label{A.8.8}
Let $t$ be a topology on $lSm/k$ compatible with log transfers. 
Then for all $X\in lSm/k$ and every complex $\cF$ of $t$-sheaves with log transfers, there is a canonical isomorphism
\[
\hom_{\Deri(\Shvltrtkl)}(a_t^*\Zltr(X),\cF[i])
\cong 
\bH_t^i(X,\cF).
\]
\end{prop}
\begin{proof}
The functor 
\[
L\gamma_\sharp\colon \Deri(\Shvlogtkl)\rightarrow \Deri(\Shvltrtkl)
\] 
in \eqref{A.8.0.4} is left adjoint to 
\[
\gamma^*\colon \Deri(\Shvltrtkl)\rightarrow \Deri(\Shvlogtkl).
\] 
Thus the group $\bH_t^i(X,\cF)$ can be identified with 
\[
\hom_{\Deri(\Shvlogtkl)}(a_t^*\Lambda(X),\gamma^*\cF[i])
\cong 
\hom_{\Deri(\Shvltrtkl)}(L\gamma_\sharp a_t^*\Lambda(X),\cF[i]).
\]
The object $a_t^*\Lambda(X)$ is cofibrant, see Definition \ref{A.8.14}.
The functor $a_t^*$ commutes with $\gamma_\sharp$ due to \eqref{A.8.1.1}.
Thus we deduce 
\[
L\gamma_\sharp a_t^*\Lambda(X)\cong \gamma_\sharp a_t^*\Lambda(X)
\cong 
a_t^*\gamma_\sharp\Lambda(X)\cong a_t^*\Zltr(X).
\]
\end{proof}

We will need the following lemma in our discussion of adjunctions.
\begin{lem}
\label{A.8.33}
Suppose $F\colon C\rightarrow D$ is a left and a right Quillen functor between model categories.
Then $F$ preserves cofibrations, fibrations, and weak equivalences.
\end{lem}
\begin{proof}
For the convenience of the reader, we shall outline a proof; see also the discussion in \cite[Remark 1.3.23]{CD12}. 
By assumption, $F$ preserves cofibrations, fibrations, trivial cofibrations, and trivial fibrations.
It remains to show that $F$ preserves weak equivalences.
If $f\colon X\rightarrow Y$ be a weak equivalence in $C$, we can choose a commutative diagram 
\[
\begin{tikzcd}
X'\arrow[d,"g'"']\arrow[r,"f'"]&Y'\arrow[d,"g"]\\
X\arrow[r,"f"]&Y,
\end{tikzcd}\]
where $g$ and $g'$ are trivial fibrations and $f'$ is a weak equivalence between cofibrant objects. 
Then $F(g)$ and $F(g')$ are weak equivalences. 
Ken Brown's lemma implies that $F(f')$ is a weak equivalence. 
It follows that $F(f)$ is also a weak equivalence.
\end{proof}
  
Let $t$ be a topology on $lSm/k$, and let $t'$ be the restriction of $t$ to $Sm/k$. 
Then the functor $\lambda$ from $Sm/k$ with the topology $t'$ to $lSm/k$ with the topology $t$ is a continuous functor of sites. 
Thus we have an induced adjunction
\[
\lambda_\sharp\colon \Shv_{t'}(k,\Lambda)\rightleftarrows \Shvlogtkl\colon \lambda^*.
\]
This induces an adjunction
\[
\lambda_\sharp\colon \Shvtrttkl\rightleftarrows \Shvltrtkl\colon \lambda^*.
\]
\vspace{0.1in}

Assume $t$ is compatible with log transfers and $t'$ is compatible with transfers. 
By appealing to Proposition \ref{A.8.16}, we obtain a Quillen pair for the descent model structures
\[
\lambda_\sharp\colon \Co(\Shvtrttkl)\rightleftarrows \Co(\Shvltrtkl)\colon \lambda^*
\]
since $t'$ is the restriction of $t$, for every $t'$-hypercover $\mathscr{X}\rightarrow X$ we obtain a $t$-hypercover 
\[
\lambda(\mathscr{X})\rightarrow \lambda(X).
\] 
On the level of homotopy categories, there is an induced adjunction
\[
L\lambda_\sharp\colon \Deri(\Shvtrttkl)\rightleftarrows \Deri(\Shvltrtkl)\colon R\lambda^*.
\]
\vspace{0.1in}

Assume that for any $t$-cover $Y\rightarrow X$, the naturally induced morphism 
\[
\omega(Y)
=
Y-\partial Y
\rightarrow 
X-\partial X
=
\omega(X)
\] 
is also a $t$-cover. 
Then $\omega$ is a continuous functor of sites from $lSm/k$ with the topology $t$ to $Sm/k$ with the topology $t'$. 
There are induced adjunctions
\[
\omega_\sharp\colon \Shvlogtkl\rightleftarrows \Shv_t(k,\Lambda)\colon \omega^*
\]
and
\begin{equation}
\label{eqn::omegaadjunction}
\omega_\sharp\colon \Shvltrtkl\rightleftarrows \Shvtrttkl\colon \omega^*.
\end{equation}

By assumption $\omega$ maps any $t$-covers to $t'$-covers, 
and for any $t$-hypercover $\mathscr{X}\rightarrow X$, $\omega(\mathscr{X})\rightarrow \omega(X)$ is a $t'$-hypercover. 
According to Proposition \ref{A.8.16} there is a Quillen pair for the descent model structures
\[
\omega_\sharp\colon \Co(\Shvltrtkl)\rightleftarrows \Co(\Shvtrttkl)\colon \omega^*.
\]
We also note the induced adjunction
\[
L\omega_\sharp\colon \Deri(\Shvltrtkl)\rightleftarrows \Deri(\Shvtrttkl)\colon R\omega^*.
\]
Since $\lambda$ is left adjoint to $\omega$, by Lemma \ref{A.8.33} we deduce
\[R\lambda^*\cong \lambda^*\cong \omega_\sharp \cong L\omega_\sharp.
\]
\vspace{0.1in}

In conclusion, there is a sequence of adjunctions
\begin{equation}
\label{eq::lambdaomegaadj}
\begin{tikzcd}
\Deri(\Shvtrttkl)
\arrow[rr,shift left=1.5ex,"L\lambda_\sharp"]\arrow[rr,"\lambda^*\simeq\omega_\sharp" description,leftarrow]\arrow[rr,shift right=1.5ex,"R\omega^*"']&&{\Deri(\Shvltrtkl)}.
\end{tikzcd}
\end{equation}

As above, 
suppose $\cW$ is an essentially small class of morphisms in $lSm/k$ that is stable by isomorphisms, products, and compositions. 
Then, by Proposition \ref{A.8.20}, we obtain a Quillen pair for the $\cW$-local and $\omega(\cW)$-local descent model structures
\[
\omega_\sharp\colon \Co(\Shvltrtkl)\rightleftarrows \Co(\Shvtrttkl)\colon \omega^*.
\]
Thus we have the induced adjunction
\begin{equation}
\label{A.8.0.3}
L\omega_\sharp\colon \Deri_\cW(\Shvltrtkl)\rightleftarrows \Deri_{\omega(\cW)}(\Shvtrttkl)\colon R\omega^*,
\end{equation}
and similarly for
\[
L\omega_\sharp\colon \Deri_\cW(\Shvlogtkl)\rightleftarrows \Deri_{\omega(\cW)}(\Shv_t(k,\Lambda))\colon R\omega^*.
\]

\begin{lem}
\label{A.8.22}
Let $f\colon X\rightarrow Y$ be a monomorphism of fs log schemes.
Then the naturally induced morphism of presheaves
\[
\Zltr(f)\colon \Zltr(X)\rightarrow \Zltr(Y)
\]
is a monomorphism.
\end{lem}
\begin{proof}
For $T\in lSm/k$, there is a naturally induced commutative diagram
\[
\begin{tikzcd}
\lCor(T,X)\arrow[r]\arrow[d]&
 \lCor(T,Y)\arrow[d]&
\\
\Cor(T-\partial T,X-\partial X)\arrow[r]&
\Cor(T-\partial T,Y-\partial Y).
\end{tikzcd}
\]
We need to show that the upper horizontal homomorphism is injective.
The induced morphism $X-\partial X\rightarrow Y-\partial Y$ is also a monomorphism. 
Thus the lower horizontal homomorphism is injective, see \cite[Exercise 12.22]{MVW}.
The vertical homomorphisms are injective due to Lemma \ref{A.5.10}, which finishes the proof.
\end{proof}

Next we specialize to the setting of cd-structures.
\begin{prop}
\label{A.8.13}
Let $P$ be a complete, quasi-bounded, and regular cd-structure on $lSm/k$, and let $t$ be the associated topology. 
Then the following conditions are equivalent.
\begin{enumerate}
\item[{\rm (i)}] $t$ is compatible with log transfers.
\item[{\rm (ii)}] $t$ is mildly compatible with log transfers.
\item[{\rm (iii)}] For any $P$-distinguished square in $lSm/k$ of the form {\rm \eqref{A.8.13.1}}, the sequence
\begin{equation}
\label{A.8.13.2}
0\rightarrow a_t^*\gamma^*\Zltr(Y')\rightarrow a_t^*\gamma^*\Zltr(Y)\oplus a_t^*\gamma^*\Zltr(X')\rightarrow a_t^*\gamma^*\Zltr(X)\rightarrow 0
\end{equation}
of $t$-sheaves with log transfers is exact.
\end{enumerate}
\end{prop}
\begin{proof}
The implication (i)$\Rightarrow$(ii) is Proposition \ref{A.8.3}.
Assuming condition (ii) the induced sequence  of $t$-sheaves
\begin{equation}
\label{A.8.13.5}
0\rightarrow a_t^*\Lambda(Y')\rightarrow a_t^*\Lambda(Y)\oplus a_t^*\Lambda(X')\rightarrow a_t^*\Lambda(X)\rightarrow 0
\end{equation}
is exact by \cite[Lemma 2.18]{Vcdtop}.
Let us show that the following naturally induced sequence of $t$-sheaves with log transfers is exact:
\begin{equation}
\label{A.8.13.3}
0\rightarrow \gamma_\sharp a_t^*\Lambda(Y')\rightarrow \gamma_\sharp a_t^*\Lambda(Y)\oplus \gamma_\sharp a_t^*\Lambda(X')\rightarrow \gamma_\sharp a_t^*\Lambda(X)\rightarrow 0.
\end{equation}
Since $\gamma_\sharp$ commutes with $a_t^*$, this is equivalent to showing that the sequence
\begin{equation}\label{A.8.13.4}
0\rightarrow a_t^*\Zltr(Y')\rightarrow a_t^*\Zltr(Y)\oplus  a_t^*\Zltr(X')\rightarrow a_t^*\Zltr(X)\rightarrow 0
\end{equation}
of $t$-sheaves with log transfers is exact. 
Since $\gamma_\sharp$ is a left adjoint, $\gamma_\sharp$ is right exact.
Thus 
\eqref{A.8.13.3} is exact at $\gamma_\sharp a_t^*\Lambda(Y)\oplus \gamma_\sharp a_t^*\Lambda(X')$ and $\gamma_\sharp a_t^*\Lambda(X)$. 
Since $P$ is regular, the morphism $Y'\rightarrow Y$ is a monomorphism.
Thus $\Zltr(Y')\rightarrow \Zltr(Y)$ is a monomorphism owing to Lemma \ref{A.8.22}.
This gives the exactness of \eqref{A.8.13.4} at $a_t^*\Zltr(Y')$,
i.e., 
\eqref{A.8.13.3} and \eqref{A.8.13.4} are exact.
We deduce that \eqref{A.8.13.2} is exact since $\gamma^*$ commutes with $a_t^*$ and $\gamma^*$ is exact. 
This implies condition (iii).
\vspace{0.1in}

Let us assume (iii) holds. 
To conclude (ii), we need to show that \eqref{A.8.3.1} is exact when $\mathscr{X}\rightarrow X$ is a \v{C}ech nerve associated with a $t$-cover $T\rightarrow X$.
Every $t$-cover has a simple refinement because $P$ is complete, see \cite[Definition 2.3]{Vcdtop}.
Thus since $t$ is the topology associated with $P$, it suffices to consider the case when $T\rightarrow X$ is of the form $Y\amalg X'\rightarrow X$. 
In this case, our claim follows from the exactness of \eqref{A.8.13.2}.
\vspace{0.1in}

Now assume (ii) and (iii) hold. 
It remains to show (i). 
We work with the descent structure $(\cG,\cH)$ (resp.\ $(\cG,\cH')$) of $\Shvlogtkl$ given in Example \ref{A.8.28} (resp.\ Example \ref{A.8.30}). 
Let $G\in \Deri(\Shvltrtkl)$ be an injective fibrant complex. 
Since the sequence \eqref{A.8.13.3} is exact, we have 
\[ 
\hom_{{\bf K}(\Shvlogtkl)}(F',\gamma^*G)=0
\]
for all $F'\in \cH'$. 
Thus $\gamma^*G$ is $\cH'$-flasque, 
i.e., 
$\gamma^*G$ is $\cG$-fibrant. 
By Propositions \ref{A.8.15}(2) we see that $\gamma^*G$ is $\cH$-flasque. 
Thus for all $F\in \cH$, we have the vanishing
\[ 
\hom_{{\bf K}(\Shvlogtkl)}(F,\gamma^*G)=0.
\]
Since $G$ is an injective fibrant complex, we also have 
\[
\hom_{{\bf K}(\Shvlogtkl)}(F,\gamma^*G)
\cong 
\hom_{{\bf K}(\Shvltrtkl)}(\gamma_\sharp F,G)
\cong 
\hom_{\Deri(\Shvltrtkl)}(\gamma_\sharp F,G).
\]
Here $G$ is an arbitrary injective fibrant complex.
Hence $\gamma_\sharp F$ is quasi-isomorphic to $0$ for all $F\in \cH$. 
This implies that condition (i) holds.
\end{proof}

\subsection{Structure of dividing Nisnevich covers}
\label{subsec::structure_dNis}
In this section, we shall study dividing Nisnevich covers and the dividing Nisnevich sheafification functor.
We also discuss their \'etale versions.
\vspace{0.1in}

Suppose that $X$ is an excellent reduced noetherian scheme. 
In \cite[Theorem 3.1.9]{VSelecta} it is shown that every $h$-cover $f\colon Y\longrightarrow X$ of schemes admits a refinement of the form
\[
Y\stackrel{g_1}\longrightarrow Y_1\stackrel{g_2}\longrightarrow Y_2\stackrel{g_3}\longrightarrow X,
\]
where $g_1$ is a Zariski cover, $g_2$ is a finite surjective morphism, and $g_3$ is a proper surjective birational morphism.
This is one of the fundamental tools in the theory of $h$-coverings.
\vspace{0.1in}

We provide a result with a similar flavor for factorizations of dividing Nisnevich, dividing \'etale, and log \'etale covers.
\begin{prop}
\label{A.5.44}
Let $f\colon Y\rightarrow X$ be a dividing Nisnevich (resp.\ dividing \'etale, resp.\ log \'etale) cover in $lSm/k$.
Then there exists a log modification $h\colon X'\rightarrow X$ such that the pullback
\[
g\colon Y\times_X X'\rightarrow X'
\]
is a strict Nisnevich (resp.\ strict \'etale, resp.\ integral Kummer \'etale) cover.

In particular, $f$ has a refinement of the form
\begin{equation}
\label{A.5.44.1}
Y\times_X X'\stackrel{g}\rightarrow X'\stackrel{h}\rightarrow X,
\end{equation}
where $g$ is a strict Nisnevich (resp.\ strict \'etale, resp.\ integral log \'etale) cover and $h$ is a log modification.
\end{prop}
\begin{proof}
If $f$ is log \'etale, see Theorem \ref{FKatoThm2}.
We focus the attention on the dividing Nisnevich topology since the proofs are similar for the dividing Nisnevich and dividing \'etale cases.
First, we reduce to the case when $f$ has a refinement in the form 
\[
Y_0\stackrel{f_0}\rightarrow Y_0'\stackrel{f_0'}\rightarrow \cdots \stackrel{f_{n-2}'}\rightarrow Y_{n-1}\stackrel{f_{n-1}}\rightarrow Y_{n-1}'\stackrel{f_{n-1}'}\rightarrow Y_n\stackrel{f_n}\rightarrow X,
\]
where each $f_i$ is a strict Nisnevich cover and each $f_i'$ is a log modification.
Owing to Proposition \ref{Fan.12}, there exists a log modification $X'\rightarrow X$ such that the pullback
\[
Y_{n-1}'\times_X X'\rightarrow Y_n\times_X X'
\]
of $f_{n-1}'$ is an isomorphism. In particular, the pullback
\[
Y_{n-1}\times_X X'\rightarrow Y_n\times_X X'
\]
of the composition $Y_{n-1}\rightarrow X$ is a strict Nisnevich cover.
\vspace{0.1in}

Replacing \eqref{A.5.44.1} by
\[
Y_0'\times_X X'\rightarrow \cdots \rightarrow Y_{n-1}\times_X X'\rightarrow X',
\]
we can apply the above process since $X'\in lSm/k$.
In this way, we obtain a log modification $X''\rightarrow X$ such that the pullback
\[
Y_0\times_X X''\rightarrow X''
\]
of the composition $Y_0\rightarrow X$ is a strict Nisnevich cover.
\end{proof}

Recall from Definition \ref{A.9.79} that $X_{div}$ denotes the category of log modifications over $X\in lSm/k$, 
and $X_{div}^{Sm}$ is the full subcategory of $X_{div}$ consisting of log modifications $Y\rightarrow X$ such that $Y\in SmlSm/k$.

\begin{lem}\label{A.5.46}
Let $\cF$ be a strict Nisnevich (resp.\ strict \'etale, resp.\ Kummer \'etale) sheaf on $SmlSm/k$.
There is an isomorphism
\begin{equation}
\label{A.5.46.1}
\begin{split}
a_{dNis}^*\cF(X)\cong &\colimit_{Y\in X_{div}^{Sm}}\cF(Y)
\\
\text{ (resp.\ }
a_{d\acute{e}t}^*\cF(X)\cong \colimit_{Y\in X_{div}^{Sm}}\cF(Y),
&\text{ resp.\ }
a_{l\acute{e}t}^*\cF(X)\cong \colimit_{Y\in X_{div}^{Sm}}\cF(Y)
\text{)}.
\end{split}
\end{equation}
\end{lem}
\begin{proof}
We will only consider the strict Nisnevich case since the proofs are similar.
Let $X_{dNis}$ denote the small dividing Nisnevich site of $X$.
Recall from \cite[Remarque II.3.3]{SGA4} that there is a functor $L_{dNis}$ such that $a_{dNis}^*=L_{dNis}\circ L_{dNis}$ given by 
\[
L_{dNis}\cF(X)\cong \colimit_{X'\in X_{dNis}}\ker(\cF(X')\stackrel{(+,-)}\longrightarrow \cF(X'\times_X X')).
\]
The notation is shorthand for the difference between the two morphisms induced by the projections $X'\times_X X'\longrightarrow X'$.
\vspace{0.1in}

By Proposition \ref{A.5.44}, for any dividing Nisnevich covering $X'\rightarrow X$, there is a refinement
\[
Y'\stackrel{g}\rightarrow Y''\stackrel{h}\rightarrow X
\]
such that $g$ is a strict Nisnevich cover and $h$ is a log modification.
Further, we may assume that $Y''\in SmlSm/k$ by applying Proposition \ref{A.3.19}.
Since
\[
Y'\times_X Y'\cong Y'\times_{Y''}Y''\times_X Y''\times_{Y''}Y'\cong Y'\times_{Y''} Y',
\]
we have 
\[
\ker(\cF(Y')\stackrel{(+,-)}\longrightarrow \cF(Y'\times_X Y'))
\cong
\ker(\cF(Y')\stackrel{(+,-)}\longrightarrow \cF(Y'\times_{Y''} Y'))
\cong
\cF(Y'')
\]
because $\cF$ is assumed to be a strict Nisnevich sheaf.
This gives the identification
\[
L_{dNis}\cF(X)\cong \colimit_{Y\in X_{div}^{Sm}}\cF(Y).
\]

Let us show that $L_{dNis}\cF$ is a strict Nisnevich sheaf.
For every strict Nisnevich cover $f\colon X'\rightarrow X$ the naturally induced functor $f^*\colon X_{div}^{Sm}\rightarrow X_{div}'^{Sm}$ 
is cofinal owing to Propositions \ref{A.9.81} and \ref{Fan.12}.
Thus to show that $L_{dNis}\cF$ is a strict Nisnevich sheaf, it suffices to show there is an exact sequence 
\[
0\rightarrow \colimit_{Y\in X_{div}^{Sm}}\cF(Y)\rightarrow \colimit_{Y\in X_{div}^{Sm}}\cF(Y\times_X X')\rightarrow \colimit_{Y\in X_{div}^{Sm}}\cF(Y\times_X X'\times_X X').
\]
Owing to Proposition \ref{A.9.82} $X_{div}^{Sm}$ is a filtered category.
Since filtered colimits preserve exact sequences, it suffices to show there is an exact sequence 
\[
0\rightarrow \cF(Y)\rightarrow \cF(Y\times_X X')\rightarrow \cF(Y\times_X X'\times_X X')
\]
for all $Y\in X_{div}$. 
This follows from the assumption that $\cF$ is a strict Nisnevich sheaf.
\vspace{0.1in}

Applying the above to $L_{dNis}\cF$ we deduce the isomorphisms
\[
a_{dNis}^*\cF(X)\cong \colimit_{Z\in Y_{div}^{Sm}}\colimit_{Y\in X_{div}^{Sm}}\cF(X)\cong \colimit_{Z\in X_{div}^{Sm}}\cF(X).
\]
\end{proof}

\begin{lem}\label{A.5.45}
Let $\cF$ be a strict Nisnevich (resp.\ strict \'etale, resp.\ Kummer \'etale) sheaf on $lSm/k$.
There is an isomorphim
\begin{equation}
\label{A.5.45.1}
\begin{split}
a_{dNis}^*\cF(X)\cong &\colimit_{Y\in X_{div}}\cF(Y)
\\
\text{ (resp.\ }
a_{d\acute{e}t}^*\cF(X)\cong \colimit_{Y\in X_{div}}\cF(Y),
&\text{ resp.\ }
a_{l\acute{e}t}^*\cF(X)\cong \colimit_{Y\in X_{div}}\cF(Y)
\text{)}.
\end{split}
\end{equation}
\end{lem}
\begin{proof}
The proof is parallel to that of Lemma \ref{A.5.46} if we use Proposition \ref{A.9.80} instead of Proposition \ref{A.9.82}.
\end{proof}

For the definition of a log modification along a smooth center, we refer to Definition \ref{Fan.39}.

\begin{df}
\label{A.5.72}
Let $X_{divsc}^{Sm}$ be the full subcategory of $X_{div}^{Sm}$ given by maps $Y\rightarrow X$ that are isomorphic to a composition of log modifications along a smooth center.
\end{df}

\begin{cor}
\label{A.5.74}
If $X\in lSm/k$, then the category $X_{divsc}^{Sm}$ is cofinal in $X_{div}^{Sm}$.
\begin{proof}
Immediate from Theorem \ref{Fan.16}.
\end{proof}
\end{cor}

\begin{lem}
\label{A.5.75}
Let $\cF$ be a strict Nisnevich (resp.\ strict \'etale, resp.\ Kummer \'etale) sheaf on $SmlSm/k$.
There is an isomorphism
\[
\begin{split}
a_{dNis}^*\cF(X)\cong &\colimit_{Y\in X_{divsc}^{Sm}}\cF(Y)
\\
\text{ (resp.\ }
a_{d\acute{e}t}^*\cF(X)\cong \colimit_{Y\in X_{divsc}^{Sm}}\cF(Y),
&\text{ resp.\ }
a_{l\acute{e}t}^*\cF(X)\cong \colimit_{Y\in X_{divsc}^{Sm}}\cF(Y)
\text{)}.
\end{split}
\]
\end{lem}
\begin{proof}
Immediate from Lemma \ref{A.5.46} and Corollary \ref{A.5.74}.
\end{proof}
\subsection{Examples of topologies compatible with log transfers}

This section shows that the strict Nisnevich, strict \'etale, Kummer \'etale, dividing Nisnevich, dividing \'etale, and Kummer \'etale topologies are compatible with log transfers.
We refer to \S \ref{sec::cdlog} for the definitions of these topologies.
For the strict Nisnevich topology, we use arguments similar to \cite[Lemma 6.2, Proposition 6.12]{MVW} originating in \cite{MR1764202}.

\begin{prop}
\label{A.4.1}
Let $X$ be an fs log scheme log smooth over $k$. Then $\Zltr(X)$ is a strict \'etale sheaf.
\end{prop}
\begin{proof}
For fs log schemes $Y_1$ and $Y_2$ log smooth over $k$ we have 
\[
\Zltr(X)(Y_1\amalg Y_2)
=
\Zltr(X)(Y_1)\oplus \Zltr(X)(Y_2).
\]
Thus it suffices to check that the sequence of $\Lambda$-modules
\begin{equation}
\label{A.4.1.1}
0\rightarrow \Zltr(X)(Y)\rightarrow \Zltr(X)(U)\stackrel{(+,-)}\longrightarrow \Zltr(X)(U\times_Y U)
\end{equation}
is exact for every strict \'etale cover $p\colon U\rightarrow Y$.
\vspace{0.1in}

There is an induced commutative diagram of $\Lambda$-modules
\[
\begin{tikzcd}
\Zltr(X)(Y)\arrow[d]\arrow[r]&\Zltr(X)(U)\arrow[d]\\
\Zltr(X-\partial X)(Y-\partial Y)\arrow[r]&\Zltr(X-\partial X)(U-\partial U).
\end{tikzcd}
\] 
By \cite[Lemma 6.2]{MVW}, $\Zltr(X-\partial X)$ is an \'etale sheaf. 
Thus the lower horizontal morphism is injective. 
By Lemma \ref{A.5.10}, the vertical morphisms are injective. 
Thus the upper horizontal morphism is also injective.
\vspace{0.1in}

It remains to show that the sequence \eqref{A.4.1.1} is exact at $\Zltr(X)(U)$. 
To that end, we consider the cartesian square of fs log schemes
\[
\begin{tikzcd}
(U-\partial U)\times (X-\partial X)\arrow[d]\arrow[r]&U\times X\arrow[d]\\
(Y-\partial Y)\times (X-\partial X)\arrow[r]&Y\times X.
\end{tikzcd}
\]

Suppose $W\in \lCor(U,X)$ is a finite log correspondence with trivial image in $\Zltr(X)(U\times_Y U)$,
and form the finite correspondence 
\[
W-\partial W\in \Cor(U-\partial U,X-\partial X).
\] 
Since $\Zltr(X-\partial X)$ is an \'etale sheaf, there exists a finite correspondence 
\[
V'\in \Cor(Y-\partial Y,X-\partial X)
\]
mapping to $W-\partial W$. 
We let $\underline{V}$ be the closure of $V'$ in $\underline{Y}\times \underline{X}$.
\vspace{0.1in}

Consider the induced morphism $u\colon \underline{W}\rightarrow \underline{U}\times_{\underline{Y}}\underline{V}$ of closed subschemes of $\underline{U}\times \underline{X}$. 
By construction its pullback to $(U-\partial U)\times (X-\partial X)$ is the isomorphism
\[
W-\partial W\rightarrow (U-\partial U)\times_{(Y-\partial Y)}(V-\partial V) 
\]
Since $\underline{p}\colon \underline{U}\rightarrow \underline{Y}$ is strict \'etale and $V-\partial V$ is dense in $\underline{V}$, 
it follows that 
\[
(U-\partial U)\times_{(Y-\partial Y)}(V-\partial V)
\]
is dense in $\underline{U}\times_{\underline{Y}}\underline{V}$. 
Moreover,
$W-\partial W$ is dense in $\underline{W}$ because $u$ is a closed immersion. 
We note that $\underline{V}$ is finite over $\underline{Y}$ by \cite[IV.2.7.1(xv)]{EGA} since $\underline{U}\times_{\underline{Y}}\underline{V}$ is finite over $\underline{U}$ 
and $\underline{p}\colon \underline{U}\rightarrow \underline{Y}$ is an \'etale cover. 
There is a naturally induced morphism $v\colon \underline{W}\rightarrow \underline{V}$. 
Since $u$ is an isomorphism, $v$ is a pullback of $p$, and hence $v$ is \'etale.
\vspace{0.1in}

Let $V^N$ be the fs log scheme whose underlying scheme is the normalization of $\underline{V}$ and whose log structure is induced by $Y$. 
By \cite[11.3.13(ii)]{EGA} the underlying fiber product $\underline{V^N}\times_{\underline{V}}\underline{W}$ is normal since $v$ is \'etale. 
Thus we have the cartesian square of schemes
\[\begin{tikzcd}
\underline{W^N}\arrow[d]\arrow[r]&\underline{W}\arrow[d,"v"]\\
\underline{V^N}\arrow[r]&\underline{V}.
\end{tikzcd}
\]
The log transfer structure gives rise to a morphism $r\colon W^N\rightarrow X$ of fs log schemes over $k$.
By assumption,
we can identify the two composite morphisms in the diagram
\[
W^N\times_{U}(U\times_Y U)
\rightrightarrows 
W^N
\overset{r}{\rightarrow}
X.
\]
Since $p$ is strict \'etale there is a morphism $V^N\rightarrow X$ of fs log schemes over $k$ such that the composition $W^N\rightarrow V^N\rightarrow X$ is equal to $r$, 
see \cite[Corollary III.1.4.5]{Ogu}.
The pair $(\underline{V},V^N\rightarrow X)$ gives a finite log correspondence $V$ from $X$ to $Y$ since $\underline{V}$ is finite over $\underline{Y}$, 
and the pullback of $V$ to $U$ is $W$.
\end{proof}

\begin{prop}
\label{A.5.11}
Let $t$ be one of the following topologies: strict Nisnevich, strict \'etale, and Kummer \'etale topologies.
For every $t$-cover $f\colon U\rightarrow X$ of fs log schemes in $lSm/k$,
the \v{C}ech complex
\begin{equation}
\label{A.5.11.2}
\cdots\rightarrow a_t^*\Zltr(U\times_X U)\rightarrow a_t^*\Zltr(U)\rightarrow a_t^*\Zltr(X)\rightarrow 0
\end{equation}
is exact as a complex of $t$-sheaves.
\end{prop}
\begin{proof}
Arguing as in the proof of \cite[Proposition 6.12]{MVW} we check there is an exact sequence of $\Lambda$-modules
\begin{equation}\
\label{A.5.11.1}
\cdots\rightarrow \Zltr(U\times_X U)(T)\rightarrow \Zltr(U)(T)\rightarrow \Zltr(X)(T)\rightarrow 0,
\end{equation}
where $\underline{T}$ is a hensel local scheme (resp.\ $\underline{T}$ is strictly local, resp.\ $T$ is log strictly local) if $t=sNis$ (resp.\ $t=s\acute{e}t$, resp.\ $t=k\acute{e}t$). 
Here $\Zltr({X})(T)$ denotes the colimit $\varinjlim \Zltr(X)(T_i)$ when $T$ is a limit of fs log schemes $T_i$ log smooth over $k$.
\vspace{0.1in}

Let $Z$ be a saturated log scheme over $T$ such that $\underline{Z}$ quasi-finite over $\underline{T}$. 
We denote by $L(Z/T)$ the free abelian group generated by the pairs 
\[
(\underline{W},W^N\rightarrow Z), 
\]
where $\underline{W}$ is an irreducible component of $\underline{Z}$ that is finite and surjective over a component of $T$, 
$W^N$ is the saturated log scheme whose underlying scheme is the normalization of $\underline{W}$ and whose log structure is induced by $T$, 
and $W^N\rightarrow Z$ is a morphism of saturated log schemes. 
We note that $L(Z/T)$ is covariantly functorial on $Z$. 
Now \eqref{A.5.11.1} is the colimit of the complexes of the form
\[
\cdots\rightarrow L((Z_U)_Z^2/T)\rightarrow L(Z_U/T)\rightarrow L(Z/T)\rightarrow 0,
\]
where
\[
Z_U:=Z\times_X U,
\;
(Z_U)_Z^n:=\underbrace{Z_U\times_Z \cdots \times_Z Z_U}_{n \text{ times}}.
\]
The colimit runs over strictly closed subschemes $Z$ of $T\times X$ such that $Z$ is strict, finite, and surjective over $T$. 
Hence it suffices to check that the latter sequence is exact.
\vspace{0.1in}

Note that $Z$ is henselian (resp.\ strictly henselian, resp.\ log strictly henselian) being finite over the henselian (resp.\ strictly henselian, resp.\ log strictly henselian) saturated log scheme $T$, see \ref{ket.2} for the log strictly henselian case.
Thus the strict Nisnevich (resp.\ strict \'etale, resp.\ Kummer \'etale) cover $Z_U\rightarrow Z$ splits, see Lemma \ref{ket.4} for the Kummer \'etale case.
Letting $s\colon Z\rightarrow Z_U$ be a splitting,
the maps
\[
L(s\times_Z {\rm id})\colon L((Z_U)_Z^n)\rightarrow L((Z_U)_Z^{n+1})
\]
furnish contracting homotopies, and the desired exactness follows.
\end{proof}

\begin{rmk}
Combining Propositions \ref{A.4.1} and \ref{A.5.11} we can remove $a_t^*$ in the equation \eqref{A.5.11.2} if $t=sNis$ or $t=s\acute{e}t$.
\end{rmk}

\begin{prop}
\label{A.5.22}
The strict Nisnevich topology, the strict \'etale topology, and the Kummer \'etale topology on $lSm/k$ are compatible with log transfers.
\end{prop}
\begin{proof}
Follows from Propositions \ref{A.8.11} and \ref{A.5.11}.
\end{proof}

\begin{lem}
\label{A.5.62}
Let $f\colon X'\rightarrow X$ be a log modification in $lSm/k$.
For every $Y\in lSm/k$ and finite log correspondence $V\in \lCor(Y,X)$, there exists a dividing Zariski cover $g\colon Y'\rightarrow Y$ and a finite log correspondence $W\in \lCor(Y',X')$ fitting in a commutative diagram
\[
\begin{tikzcd}
Y'\arrow[d,"g"']\arrow[r,"W"]&X'\arrow[d,"f"]\\
Y\arrow[r,"V"]&X
\end{tikzcd}
\]
in $lCor/k$.
\end{lem}
\begin{proof}
The question is Zariski local on $Y$, so we may assume that $Y$ has an fs chart $P$. 
Let $p\colon V^N\times_X X'\rightarrow V^N$ be the projection, and let $q\colon V^N\rightarrow Y$ be the structure morphism. 
By Lemma \ref{A.9.21}, there is a subdivision of fans $M\rightarrow \Spec P$ such that the projection $V^N\times_{\A_P}\A_M\rightarrow V^N$ admits a factorization
\[
V^N\times_{\A_P}\A_M\rightarrow V^N\times_X X'\rightarrow V^N.
\]
Setting $Y':=Y\times_{\A_P}\A_M$ and $V':=V\circ u$ where $u\colon Y'\rightarrow Y$ denotes the projection, 
the structure morphism $v\colon V'^N\rightarrow X$ factors through $V^N\times_Y Y'\cong V^N\times_{\A_P}\A_M$, so that $v$ factors through $V^N\times_X X'$. 
It follows that $v$ factors through $X'$.
\vspace{0.1in}

Replacing $Y$ by $Y'$ and $V$ by $V'$, we may assume that the structure morphism $V^N\rightarrow X$ factors through $X'$. 
In this case, there is a commutative diagram of fs log schemes
\[
\begin{tikzcd}
V^N\arrow[r,"w"]&Y\times X'\arrow[d]\arrow[r]&X'\arrow[d]\\
&Y\times X\arrow[r]&X.
\end{tikzcd}\]
Since $X'$ is proper over $X$, the morphism $w$ is proper. 
Let $\underline{W}$ be the image of 
\[
\underline{V^N}\rightarrow \underline{Y}\times \underline{X'},
\] 
which we consider as a closed subscheme of $\underline{Y}\times \underline{X'}$ with the reduced scheme structure. 
Then $\underline{V^N}$ is the normalization of $\underline{W}$ since $\underline{V^N}$ is the normalization of the image of 
\[
\underline{V^N}\rightarrow \underline{Y}\times \underline{X}.
\] 
Thus $w$ gives a correspondence $W$ from $Y$ to $X'$ with image $V$ in $\lCor(Y',X)$.
\end{proof}

\begin{lem}
\label{A.5.70}
For every log modification $f\colon Y\rightarrow X$ in $lSm/k$ the induced morphism
\[
a_{dNis}^*\gamma^*\Zltr(Y)\rightarrow a_{dNis}^*\gamma^*\Zltr(X)
\] 
of dividing Nisnevich sheaves is an isomorphism.
\end{lem}
\begin{proof}
Let $T$ be an fs log scheme log smooth over $k$. 
In the commutative diagram
\[
\begin{tikzcd}
\lCor(T,Y)\arrow[r]\arrow[d,hook]&\lCor(T,X)\arrow[d,hook']\\
\Cor(T-\partial T,Y-\partial Y)\arrow[r,"\sim"]&\Cor(T-\partial T,X-\partial X)
\end{tikzcd}
\]
the vertical morphisms are injections by Lemma \ref{A.5.10}, 
and the lower horizontal morphism is an isomorphism since $Y-\partial Y\cong X-\partial X$. 
Thus the upper horizontal morphism is an injection. 
Hence $\Zltr(Y)\rightarrow \Zltr(X)$ is a monomorphism,
and likewise for
\[
a_{dNis}^*\gamma^*\Zltr(Y)\rightarrow a_{dNis}^*\gamma^*\Zltr(X)
\]
since $a_{dNis}^*$ and $\gamma^*$ are exact, see \eqref{A.8.7.1} and Proposition \ref{A.8.5}.
To show that it is an epimorphism, 
let $V\in \lCor(T,X)$ be a finite log correspondence.
Due to Lemma \ref{A.5.62} there exists a dividing Nisnevich cover $g\colon T'\rightarrow T$ and a finite log correspondence $W\in \lCor(T',Y)$ such that $f\circ W=V\circ g$.
This finishes the proof.
\end{proof}

\begin{thm}
\label{A.5.23}
The dividing Nisnevich topology, the dividing \'etale topology, and the log \'etale topology on $lSm/k$ are compatible with log transfers.
\end{thm}
\begin{proof}
We write the proof in the case of the log \'etale topology.
Let $f\colon Y\rightarrow X$ be a log \'etale cover.
We need to show that the \v{C}ech complex
\[
\cdots\rightarrow a_{l\acute{e}t}^*\gamma^*\Zltr(Y\times_X Y)\rightarrow a_{l\acute{e}t}^*\gamma^*\Zltr(Y)\rightarrow a_{l\acute{e}t}^*\gamma^*\Zltr(X)\rightarrow 0
\]
is exact, see Proposition \ref{A.8.11}.
There exists a log modification $X'\rightarrow X$ such that the pullback $f'\colon Y\times_X X'\rightarrow X'$ of $f$ is a Kummer \'etale cover owing to Proposition \ref{A.5.44}.
Set $Y':=Y\times_X X'$.
Then using Lemma \ref{A.5.70}, it suffices to show that the \v{C}ech complex
\[
\cdots \rightarrow a_{l\acute{e}t}^*\gamma^*\Zltr(Y'\times_{X'}Y')\rightarrow a_{l\acute{e}t}^*\gamma^*\Zltr(Y')\rightarrow a_{l\acute{e}t}^*\gamma^*\Zltr(X')\rightarrow 0
\]
is exact.
Since the sheafification functor $a_{l\acute{e}t}^*$ is exact, we are done by Proposition \ref{A.5.11}.
\end{proof}

\subsection{Dividing log correspondences}\label{subsec:divdinglogcor}

For $X,Y\in lSm/k$, we have used $\lCor(X,Y)$ to define dividing Nisnevich sheaves with log transfers.
Since $\Zltr(Y)$ is not a dividing Nisnevich sheaf one cannot in general identify $\lCor(X,Y)$ with 
\[
\hom_{\Shvltrkl}(a_{dNis}^*\Zltr(Y),a_{dNis}^*\Zltr(X)).
\]
However, 
it turns out that algebraic cycles describe $a_{dNis}^*\Zltr(Y)(X)$.
To that end, we introduce the notion of dividing log correspondences.

\begin{df}\label{A.5.47}
For $X,Y\in lSm/k$, a {\it dividing elementary log correspondence} $Z$ from $X$ to $Y$ is an integral closed subscheme $\underline{Z}$ of $\underline{X}\times \underline{Y}$  
that is finite and surjective over a component of $\underline{X}$ together with a log modification $X'\rightarrow X$ and a morphism
\[
u'\colon Z'^N\rightarrow X'\times Y
\]
satisfying the following properties:
\begin{enumerate}
\item[(i)] The image of the composition $\underline{Z'^N}\stackrel{\underline{u'}}\rightarrow \underline{X'}\times \underline{Y}\rightarrow \underline{X}\times \underline{Y}$ is $\underline{Z}$.
\item[(ii)] $\underline{Z'^N}$ is the normalization of $\underline{Z}\times_{\underline{X}}\underline{X'}$.
\item[(iii)] The composition $Z'^N\stackrel{u'}\rightarrow X'\times Y\rightarrow X'$ is strict.
\end{enumerate}

A {\it dividing log correspondence} \index{correspondence!dividing log} from $X$ to $Y$ is a formal sum $\sum n_i Z_i$ of dividing elementary log correspondences $Z_i$ from $X$ to $Y$.
Let $\lCor_k^{div}(X,Y)$ denote the free abelian group of dividing log correspondences from $X$ to $Y$.
When no confusion seems likely to arise, we shall omit the subscript $k$ and write $\lCor^{div}(X,Y)$.
\end{df}

\begin{rmk}\label{A.5.48}
For $X,Y\in lSm/k$, a finite log correspondence $Z\in \lCor(X,Y)$ is determined by $Z-\partial Z\in \lCor(X-\partial X,Y-\partial Y)$ owing to Lemma \ref{A.5.10}.
Thus $Z$ is determined by its underlying scheme $\underline{Z}$.
Hence we arrive at the following equivalent definition:
An elementary log correspondence $Z$ from $X$ to $Y$ is an integral closed subscheme $\underline{Z}$ of $\underline{X}\times \underline{Y}$ that is finite and surjective over a component of 
$\underline{X}$ together with a morphism
\[
u\colon Z^N\rightarrow X\times Y
\]
satisfying the following properties:
\begin{enumerate}
\item[(i)] The image of $\underline{u}\colon \underline{Z^N}\rightarrow \underline{X}\times \underline{Y}$ is $\underline{Z}$.
\item[(ii)] $\underline{Z^N}$ is the normalization of $\underline{Z}$.
\item[(iii)] The composition $Z^N\stackrel{u}\rightarrow X\times Y\rightarrow X$ is strict.
\end{enumerate}
Given the latter description, we can interpret any dividing log correspondence as a finite log correspondence after replacing $X$ with a log modification of $X$.
\end{rmk}

\begin{prop}\label{A.5.49}
For every $X,Y\in lSm/k$ there are isomorphisms
\[
\lCor^{div}(X,Y)\otimes \Lambda\cong a_{dNis}^*\Zltr(Y)(X)\cong a_{d\acute{e}t}^*\Zltr(Y)(X).
\]
\end{prop}
\begin{proof}
With the description of elementary log correspondences in Remark \ref{A.5.48}, this follows from Lemma \ref{A.5.45}.
\end{proof}

\begin{const}\label{A.5.50}
For $X,Y,Z\in lSm/k$, let us construct a composition for dividing log correspondences
\begin{equation}
\label{A.5.50.1}
\circ\colon \lCor^{div}(X,Y)\times \lCor^{div}(Y,Z)\rightarrow \lCor^{div}(X,Z).
\end{equation}
Owing to Proposition \ref{A.5.49} and the Yoneda lemma we have 
\begin{equation}
\label{A.5.50.2}
\lCor^{div}(X,Y)\cong \hom_{\Shvltrkl}(a_{dNis}^*\ZZltr(X),a_{dNis}^*\ZZltr(Y)).
\end{equation}
We obtain \eqref{A.5.50.1} because we have compositions in $\ShvltrkZ$.
Moreover, the composition is associative, and for $g\in \lCor(X,Y)$ we have 
\[
{\rm id}\circ f=f=f\circ {\rm id}.
\]
Thus we can form a category $lCor^{div}/k$ with the same objects as $lSm/k$ and with morphisms given by $\lCor^{div}(X,Y)$.
\end{const}

\begin{df}
A {\it presheaf of $\Lambda$-modules with dividing log transfers} is an additive presheaf of $\Lambda$-modules on the category $lCor^{div}/k$.
We let $\Pshdltrkl$ denote the category of presheaves of $\Lambda$-modules with dividing log transfers.
\vspace{0.1in}

A {\it dividing Nisnevich sheaf of $\Lambda$-modules with dividing log transfers} is a presheaf of $\Lambda$-modules with dividing log transfers such that the restriction to $lSm/k$ is a dividing Nisnevich sheaf.
\vspace{0.1in}

We let $\Shvdltrkl$ (resp.\ $\Shv_{d\acute{e}t}^{\rm dltr}(k,\Lambda)$) denote the category of dividing Nisnevich (resp.\ dividing \'etale) sheaves of $\Lambda$-modules with dividing log transfers.
\end{df}

\begin{prop}\label{A.5.56}
There are equivalences of categories
\begin{gather*}
\Shvdltrkl\cong \Shvltrkl,
\\
\Shv_{d\acute{e}t}^{\rm dltr}(k,\Lambda)\cong \Shv_{d\acute{e}t}^{\rm ltr}(k,\Lambda),
\\
\Shv_{l\acute{e}t}^{\rm dltr}(k,\Lambda)\cong \Shv_{l\acute{e}t}^{\rm ltr}(k,\Lambda).
\end{gather*}
\end{prop}
\begin{proof}
We will only consider the dividing Nisnevich case since the proofs are similar.
It suffices to show that any $\cF\in \Shvltrkl$ admits a unique dividing log transfer structure.
For any $X,Y\in lSm/k$, 
owing to \eqref{A.5.50.2}, 
we have that
\[
\hom_{\Shvltrkl}(a_{dNis}^*\Zltr(X),a_{dNis}^*\Zltr(Y))
\cong \
lCor^{div}(X,Y).
\]
Thus for any $V\in \lCor^{div}(X,Y)$ the Yoneda lemma furnishes a canonical transfer map
\[
\begin{split}
\cF(X)\cong &\hom_{\Shvltrkl}(a_{dNis}^*\Zltr(X),\cF)\\
\rightarrow &\hom_{\Shvltrkl}(a_{dNis}^*\Zltr(Y),\cF)\cong \cF(Y).
\end{split}
\]
Thus there exists a dividing log transfer structure on $\cF$.
\vspace{0.1in}

For uniqueness, 
suppose that $\cG_0$ and $\cG_1$ are two sheaves with dividing log transfers whose restrictions to $\Shvltrkl$ are isomorphic to $\cF$.
Since $\cG_0$ and $\cG_1$ agree on objects, 
it suffices to show that $\cG_0(V)=\cG_1(V)$ for every $V\in \lCor^{div}(X,Y)$.
Owing to Proposition \ref{A.5.49}, 
there exists a dividing Nisnevich cover $p\colon X'\rightarrow X$ such that $V':=V\circ u\in \lCor^{div}(X',Y)$ is an element in $\lCor(X',Y)$.
By the assumption that $\cF$ is the restriction of $\cG_0$ and $\cG_1$, we have an equality of homomorphisms
\[
\cG_0(V')=\cG_1(V')\colon \cF(Y)\rightarrow \cF(X').
\]
Since $\cF$ is a dividing Nisnevich sheaf and $X'\rightarrow X$ is a dividing Nisnevich cover, $\cF(p)\colon \cF(X)\rightarrow \cF(X')$ is injective.
Thus we can conclude $\cG_0(V')=\cG_1(V')$.
\end{proof}

\begin{df}
Let $lCor_{SmlSm}/k$ (resp.\ $lCor_{SmlSm}^{div}/k$) denote the full subcategory of $lCor/k$ (resp.\ $lCor^{div}/k$) consisting of all objects in $SmlSm/k$.
\end{df}

\begin{lem}
\label{A.5.51}
Let $f\colon Y\rightarrow X$ be a log modification in $lSm/k$. 
Then $f$ is an isomorphism in $lCor^{div}/k$.
\end{lem}
\begin{proof}
Consider the composition
\[
\Gamma_f^t\colon Y\stackrel{\Gamma_f}\longrightarrow Y\times X\rightarrow X\times Y,
\]
where $\Gamma_f$ is the graph morphism, and the second morphism switches the factors.
The induced pullback
\[
Y\times_X Y\rightarrow (Y\times_X X)\times  Y, 
\]
is isomorphic to the diagonal morphism $Y\rightarrow Y\times Y$.
Thus $\Gamma_f^t$ is a dividing log correspondence from $X$ to $Y$, 
and it follows that $f$ is invertible in $lCor^{div}/k$.
\end{proof}

\begin{rmk}
Lemma \ref{A.5.51} furnishes a class of non-finite morphisms in $lSm/k$ that become invertible in $lCor^{div}/k$, 
but not in $lCor/k$.
Following the approach to motives with modulus in \cite{KSY2}, 
this suggests that we should invert every proper birational morphism $f:Y\rightarrow X$ for which the naturally induced morphism
\[
Y-\partial Y\rightarrow X-\partial X
\]
is an isomorphism.
However, 
our strategy is different from \cite{KSY2}.
Assuming resolution of singularities, 
in Theorem {\ref{A.3.12}} we show that the naturally induced morphism of log motives
\[
M(f):M(Y)\rightarrow M(X)
\]
is an isomorphism.
\end{rmk}

\begin{lem}\label{A.5.53}
There is an equivalence of categories
\[
lCor_{SmlSm}^{div}/k\simeq lCor^{div}/k.
\]
\end{lem}
\begin{proof}
Since $lCor_{SmlSm}^{div}/k$ is a full subcategory of $lCor^{div}/k$, it remains to show that any object $X$ of $lCor^{div}/k$ is isomorphic to an object of $lCor_{SmlSm}^{div}/k$.
Owing to Proposition \ref{A.3.19}, there is a log modification $f\colon Y\rightarrow X$ such that $Y\in SmlSm/k$.
Then we are done since $f$ is an isomorphism in $lCor^{div}/k$ by Lemma \ref{A.5.51}.
\end{proof}

\subsection{Sheaves with log transfers on $SmlSm/k$}
\label{Subsec::Sheaves.SmlSm}

There are many interesting examples of Nisnevich sheaves with transfers, 
e.g., $\mathbb{G}_m$ and $\Omega^i$.
If $\cF$ is a Nisnevich sheaf with transfers on $Sm/k$, 
one may ask whether it extends to a dividing Nisnevich sheaf with log transfers on $lSm/k$, 
such that $\cF(X)$ makes sense for every $X\in lSm/k$.
In practice, one can divide this problem into two steps.
The first step is to define $\cF(X)$ for every $X\in SmlSm/k$.
And the second step is to define $\cF(X)$ for a general $X\in lSm/k$.
In this subsection, we explain this second step.

\begin{df}\label{A.5.52}
We set
\[
\Pshltrklsm:=\Psh(lCor_{SmlSm}/k),
\]
\[
\Pshdltrklsm:=\Psh(lCor_{SmlSm}^{div}/k).
\]
For $t=dNis$, $d\acute{e}t$, and $l\acute{e}t$, we let $\Shv_t^{\rm ltr}(SmlSm/k,\Lambda)$ (resp.\ $\Shv_t^{\rm dltr}(SmlSm/k,\Lambda)$) denote the full subcategory of $\Pshltrklsm$ 
(resp.\ $\Pshdltrklsm$) consisting of $t$-sheaves on $SmlSm/k$.
\end{df}

\begin{lem}\label{A.5.54}
For $t=dNis$, $d\acute{e}t$, and $l\acute{e}t$, there is an equivalence of categories
\[
\Shv_{t}(SmlSm/k,\Lambda)\simeq \Shv_t^{\rm log}(k,\Lambda).
\]
\end{lem}
\begin{proof}
We will only consider the dividing Nisnevich topology since the proofs are similar.
For any $X\in lSm/k$, there exists a log modification $Y\rightarrow X$ such that $Y\in SmlSm/k$ owing to Proposition \ref{A.3.19}.
We conclude using the implication (i)$\Rightarrow$(ii) in \cite[Th\'eor\`eme III.4.1]{SGA4}.
\end{proof}

Let $\iota$ denote the inclusion functor $lCor_{SmlSm}/k\rightarrow lCor/k$.
For $t=dNis$, $d\acute{e}t$, and $l\acute{e}t$, we have functors
\begin{gather}
\iota^*\colon \Psh^{\rm dltr}(k,\Lambda)
\rightarrow 
\Psh^{\rm dltr}(SmlSm/k,\Lambda)
\\
\label{A.5.58.1}
\iota^*\colon 
\Shv_{t}^{\rm ltr}(k,\Lambda)
\rightarrow
\Shv_{t}^{\rm ltr}(SmlSm/k,\Lambda)
\end{gather}
mapping $\cF$ to $\cF\circ \iota$.
As in \eqref{A.5.40.1} there are faithful functors
\[
\gamma\colon lSm/k\rightarrow lCor^{div}/k,\; \gamma\colon SmlSm/k\rightarrow lCor_{SmlSm}^{div}/k.
\]
These induce functors
\[
\gamma^*\colon \Pshdltrkl\rightarrow \Pshlogkl,
\;
\gamma^*\colon \Pshdltrklsm\rightarrow \Psh(SmlSm/k,\Lambda)
\]
mapping $\cF$ to $\cF\circ \gamma$.

\begin{lem}\label{A.5.55}
For $t=dNis$, $d\acute{e}t$, and $l\acute{e}t$, there is an equivalence of categories
\[
\Shv_{t}^{\rm dltr}(SmlSm/k,\Lambda)\simeq \Shv_{t}^{\rm dltr}(k,\Lambda).
\]
\end{lem}
\begin{proof}
We will only consider the dividing Nisnevich topology since the proofs are similar.
There is a commutative diagram
\[
\begin{tikzcd}
\Pshdltrkl\arrow[r,"\iota^*"]\arrow[d,"\gamma^*"']&\Pshdltrklsm\arrow[d,"\gamma^*"]\\
\Pshlogkl \arrow[r]&\Psh(SmlSm/k,\Lambda).
\end{tikzcd}
\]
The restriction functor $\iota^*$ is an equivalence by Lemma \ref{A.5.53}.
Hence it suffices to show that for every $\cF\in \Pshdltrkl$, $\gamma^*(\cF)$ is a dividing Nisnevich sheaf if and only if its image in $\Psh(SmlSm/k,\Lambda)$ is a dividing Nisnevich sheaf.
This follows from Lemma \ref{A.5.54}.
\end{proof}

\begin{prop}\label{A.5.57}
For $t=dNis$, $d\acute{e}t$, and $l\acute{e}t$, there is an equivalence of categories
\[
\Shv_{t}^{\rm dltr}(SmlSm/k,\Lambda)\simeq \Shv_{t}^{\rm ltr}(SmlSm/k,\Lambda).
\]
\end{prop}
\begin{proof}
Parallel to the proof of Proposition \ref{A.5.56}.
\end{proof}

\begin{prop}\label{A.5.58}
For $t=dNis$, $d\acute{e}t$, and $l\acute{e}t$, the functor in \textrm{\eqref{A.5.58.1}} is an equivalence.
\end{prop}
\begin{proof}
Follows from Propositions \ref{A.5.56}, \ref{A.5.57}, and Lemma \ref{A.5.55}.
\end{proof}

Thus the functor in \textrm{\eqref{A.5.58.1}} admits a left adjoint
\[
\iota_\sharp\colon 
\Shv_{t}^{\rm ltr}(SmlSm/k,\Lambda)
\rightarrow
\Shv_{t}^{\rm ltr}(k,\Lambda).
\]

\begin{lem}
\label{A.5.59}
Let $\cF$ be an object of $\Shv_{t}^{\rm ltr}(SmlSm/k,\Lambda)$, where $t$ is one of $dNis$, $d\acute{e}t$, and $l\acute{e}t$.
For every $X\in lSm/k$ we have that
\[
\iota_\sharp \cF(X)
\cong
\colimit_{Y\rightarrow X}\cF(Y),
\]
where the colimit runs over log modifications $Y\rightarrow X$ such that $Y\in SmlSm/k$.
\end{lem}
\begin{proof}
We only discuss the dividing Nisnvich topology since the proofs are parallel.
Since $\iota^*$ is an equivalence by Proposition \ref{A.5.58}, we have $\iota^*\iota_\sharp \cF\cong \cF$, so that 
\[
\iota_\sharp \cF(Y)\cong \cF(Y)
\]
for every $Y\in SmlSm/k$.
Recall that $X_{div}^{Sm}$ is the full subcategory of $X_{div}$ consisting of $Y\rightarrow X$ such that $Y\in SmlSm/k$.
According to Proposition \ref{A.9.81}, $X_{div}^{Sm}$ is cofinal in $X_{div}$, 
and hence 
\begin{equation}
\label{A.5.59.1}
\colimit_{Y\in X_{div}}\iota_\sharp \cF(Y)
\cong 
\colimit_{Y\in X_{div}^{Sm}}\iota_\sharp \cF(Y)
\cong
\colimit_{Y\in X_{div}^{Sm}}\cF(Y).
\end{equation}
Since $\iota_\sharp \cF(X)$ is a dividing Nisnevich sheaf,
we deduce
\begin{equation}
\label{A.5.59.2}
\iota_\sharp \cF(X)\cong \colimit_{Y\in X_{div}}\iota_\sharp \cF(Y)
\end{equation}
Combining \eqref{A.5.59.1} and \eqref{A.5.59.2} gives the assertion.
\end{proof}

\newpage

\section{Construction of triangulated categories of logarithmic motives}\label{sec.constructldmeff}
This section is dedicated to constructing and studying some basic properties of our categories of log motivic sheaves. 
First, we explain how the dividing Nisnevich cohomology groups are related to the strict Nisnevich cohomology groups.
Next, we construct the triangulated categories $\ldmeff$ and $\ldaeff$ of effective log motives, 
with and without transfers, 
in the Nisnevich and \'etale topologies on log schemes over $k$. 
There are two equivalent ways of constructing $\ldmeff$, 
see Section \ref{ssecequivalence}. 
One way is to use sheaves on $lSm/k$, and the other way is via sheaves on $SmlSm/k$. 
The advantage of having the latter category at our disposal is that it allows us to construct motivic sheaves that are, a priori, 
only defined and functorial on log schemes for which the underlying scheme is smooth over $k$ and boundary a normal crossing divisor.  
In Section \ref{sec:Hodge} we use this trick to prove the representability of Hodge cohomology and cyclic homology.

\subsection{Dividing Nisnevich cohomology groups}

A fundamental invariant of the dividing Nisnevich topology is the corresponding sheaf cohomology theory.
In what follows, we shall express the dividing Nisnevich cohomology groups in terms of strict Nisnevich cohomology groups by promoting Subsection \ref{subsec::structure_dNis} to cohomology.
We also discuss the \'etale version of this result.
As an application, we discuss finite cohomological dimension property for our various topologies.

\begin{lem}
\label{Div.8}
Let $\cF$ be a strict Nisnevich (resp.\ strict \'etale, resp.\ Kummer \'etale) sheaf on $lSm/k$.
If $\cF$ is flabby, then $a_{dNis}^*\cF$ (resp.\ $a_{d\acute{e}t}^*\cF$, resp.\ $a_{l\acute{e}t}^*\cF$) is flabby.
\end{lem}
\begin{proof}
We focus on the strict Nisnevich case since the proofs are similar.
Owing to the implication $(b)\Rightarrow (a)$ in \cite[Proposition III.2.12]{Milneetale} it suffices to show that for every dividing Nisnevich cover $p:U\rightarrow X$ in $lSm/k$ and integer $i>0$, 
the corresponding \v{C}ech cohomology group vanishes 
\[
H^i(a_{dNis}^*\cF(\cU))=0.
\]
Here $\cU$ is the \v{C}ech nerve associated with $U\rightarrow X$.
\vspace{0.1in}

According to Proposition \ref{A.5.44}, there exists a log modification $X'\rightarrow X$ such that the projection $U':=U\times_X X'\rightarrow X'$ is a strict Nisnevich cover.
Since $a_{dNis}^*\cF$ is a complex of dividing Nisnevich sheaves, there is an isomorphism
\[
a_{dNis}^*\cF(\cU)\cong a_{dNis}^*\cF(\cU\times_X X').
\]
Hence we can replace $U\rightarrow X$ with $U'\rightarrow X'$, so we reduce to the case when $p$ is a strict Nisnevich cover.
\vspace{0.1in}

By Lemma \ref{A.5.45} the complex $a_{dNis}^*\cF(\cU)$ is isomorphic to
\[
\colimit_{Y_0\in X_{div}}\cF(Y_0)
\rightarrow
\colimit_{Y_1\in U_{div}}\cF(Y_1)
\rightarrow
\colimit_{Y_2\in (U\times_X U)_{div}}\cF(Y_2)
\rightarrow
\cdots.
\]
Applying Corollary \ref{Fan.14}, we see that the latter complex is isomorphic to
\[
\colimit_{Y\in X_{div}}\cF(Y)
\rightarrow
\colimit_{Y\in X_{div}}\cF(Y\times_X U)
\rightarrow
\colimit_{Y\in X_{div}}\cF(Y\times_X U\times_X U)
\rightarrow
\cdots.
\]
This is a filtered colimit according to Proposition \ref{A.9.80}.
Since every filtered colimit in the category of $\Lambda$-modules is exact, it remains to check that $H^i$ of the complex
\[
\cF(Y)\rightarrow \cF(Y\times_X U)\rightarrow \cF(Y\times_X U\times_X U)\rightarrow \cdots
\]
is trivial for every $Y\in X_{div}$ and $i>0$.
To conclude we use the implication $(a)\Rightarrow (b)$ in \cite[Proposition III.2.12]{Milneetale} and the assumption that $\cF$ is flabby.
\end{proof}

\begin{thm}
\label{Div.3}
Let $\cF$ be a bounded below complex of strict Nisnevich sheaves on $lSm/k$.
Then for every $X\in lSm/k$ and $i\in \Z$ there is an isomorphism
\begin{align*}
\bH_{dNis}^i(X,a_{dNis}^*\cF)&\cong \colimit_{Y\in X_{div}}\bH_{sNis}^i(Y,\cF)
\\
\text{(resp.\ }\bH_{d\acute{e}t}^i(X,a_{d\acute{e}t}^*\cF)&\cong \colimit_{Y\in X_{div}}\bH_{s\acute{e}t}^i(Y,\cF),
\\
\text{resp.\ }\bH_{l\acute{e}t}^i(X,a_{l\acute{e}t}^*\cF)&\cong \colimit_{Y\in X_{div}}\bH_{k\acute{e}t}^i(Y,\cF)\text{)}.
\end{align*}
\end{thm}
\begin{proof}
We only consider the strict Nisnevich case since the proofs are similar.
We can replace $\cF$ by its injective resolution because the sheafification functor $a_{dNis}^*$ is exact.
In particular, $\cF^i$ is flabby for every integer $i\in \Z$.
Owing to Lemma \ref{Div.8} we see that $a_{dNis}^*\cF^i$ is flabby.
It follows that
\[
\bH_{dNis}^i(X,a_{dNis}^*\cF)\cong H^i(a_{dNis}^*\cF(X)).
\]
Due to Lemma \ref{A.5.45} there is an isomorphism
\[
a_{dNis}^*\cF(X)\cong \colimit_{Y\in X_{div}}\cF(Y).
\]
Since this colimit is filtered by Proposition \ref{A.9.80} and every filtered colimit in the category of $\Lambda$-modules is exact, 
we deduce that
\[
\bH_{dNis}^i(X,a_{dNis}^*\cF)\cong \colimit_{Y\in X_{div}}H^i(\cF(Y)).
\]
To conclude the proof, we observe that 
\[
\bH_{sNis}^i(Y,\cF)
\cong 
H^i(\cF(Y)).
\]
\end{proof}

\begin{rmk}
Suppose $\cF$ is a torsion sheaf on the small log \'etale site $X_{l\acute{e}t}$. 
Then for every fs log scheme $X$ with an fs chart, 
the isomorphism
\[
H_{l\acute{e}t}^i(X,\cF)\cong H_{k\acute{e}t}^i(X,\cF)
\]
is proven in \cite[Proposition 5.4(2)]{MR3658728}.
\end{rmk}

\begin{cor}
\label{Div.5}
The category $lSm/k$ has finite $sNis$ and $dNis$-cohomological dimension for $\Lambda$-linear coefficients.
\end{cor}
\begin{proof}
For every $X\in lSm/k$ and $\cF\in \Shv_{sNis}^{\rm log}(k,\Lambda)$,
if $i>\dim X$, then due to \cite[Example 12.2]{MVW} we have the vanishing
\[
H_{sNis}^i(X,\cF)=H_{Nis}^i(\ul{X},\cF)=0.
\]
If $Y$ is a log modification over $X$, then $\dim Y=\dim X$ by Lemma \ref{A.9.84}.
Thus we have the vanishing
\[
H_{sNis}^i(Y,\cF)=0.
\]
Theorem \ref{Div.3} finishes the proof.
\end{proof}

\begin{lem}
\label{ketcomp.13}
Let $X$ be an fs log scheme.
Suppose $\cF$ is a sheaf of $\Q$-modules on the small strict \'etale site $X_{s\acute{e}t}$.
Then, for every integer $i\geq 0$, there is an isomorphism of cohomology groups
\[
H_{sNis}^i(X,\cF)\cong H_{s\acute{e}t}^i(X,\cF).
\]
\end{lem}
\begin{proof}
Equivalently, we need to show there is an isomorphism
\[
H_{Nis}^i(\underline{X},\cF)\cong H_{\acute{e}t}^i(\underline{X},\cF).
\]
This is the content of \cite[Proposition 14.23]{MVW}.
\end{proof}

\begin{lem}
\label{ketcomp.12}
Let $X$ be an fs log scheme.
Suppose $\cF$ is a sheaf of $\Q$-modules on the small Kummer \'etale site $X_{k\acute{e}t}$.
Then, for every integer $i\geq 0$,  there is an isomorphism of cohomology groups
\[
H_{s\acute{e}t}^i(X,\cF)\cong H_{k\acute{e}t}^i(X,\cF).
\]
\end{lem}
\begin{proof}
Let $\epsilon_*\colon \Shv(X_{k\acute{e}t})\rightarrow \Shv(X_{s\acute{e}t})$ be the functor sending $\cF$ to $\cF\vert_{X_{s\acute{e}t}}$.
We need to show that $R^i\epsilon_*\cF=0$ for every integer $i>0$.
Since $R^i\epsilon_*\cF$ is the strict \'etale sheaf associated with the presheaf
\[
(Y\in X_{s\acute{e}t})\mapsto H_{k\acute{e}t}^i(Y,\cF),
\]
it suffices to show the vanishing
$$
H_{k\acute{e}t}^i(X,\cF)=0
$$ 
for every $i>0$ whenever $\underline{X}$ is strictly local.
According to \cite[Proposition 4.1]{MR1457738}, 
see also \cite[Proposition 1.2]{MR3658728}, 
the group $H_{k\acute{e}t}^i(X,\cF)$ is a torsion abelian group by \cite[Corollary 6.11.4]{weibel_1994}.
Since $\cF$ is a sheaf of $\Q$-modules, this implies the desired vanishing.
\end{proof}

\begin{cor}
\label{Div.6}
Let $t$ be one of the following topologies on $lSm/k$:
\[
s\acute{e}t,\; d\acute{e}t,\; k\acute{e}t,\text{ and }l\acute{e}t.
\]
If $k$ has finite $\acute{e}t$-cohomological dimension for $\Lambda$-linear coefficients, then $lSm/k$ has finite $t$-cohomological dimension for $\Lambda$-linear coeffieicnts.
\end{cor}
\begin{proof}
For any $\cF\in \Shv_t^{\rm log}(k,\Lambda)$, the kernel and cokernel of the natural morphism
\[
\cF\rightarrow \cF\otimes \Q
\]
are torsion sheaves.
Hence we only need to deal with the two separate cases: $\Lambda$ is a torsion ring, and $\Q\subset \Lambda$.
\vspace{0.1in}

If $\Lambda$ is a torsion ring, then $lSm/k$ has finite $s\acute{e}t$ and $k\acute{e}t$-cohomological dimension for $\Lambda$-linear coefficients due to \cite[Corollaries X.4.2, X.5.3]{SGA4} and \cite[Theorem 7.2]{MR3658728}.
To show that $lSm/k$ has finite $d\acute{e}t$ and $l\acute{e}t$-cohomological dimension for $\Lambda$-linear coefficients, apply Lemma \ref{A.9.84} and Theorem \ref{Div.3}.
\vspace{0.1in}

If $\Q\subset \Lambda$, then Lemmas \ref{ketcomp.13} and \ref{ketcomp.12} show that $lSm/k$ has finite $s\acute{e}t$ and $k\acute{e}t$-cohomological dimension for $\Lambda$-linear coefficients since $lSm/k$ has finite $sNis$-cohomological dimension due to Corollary \ref{Div.5}.
Then apply Lemma \ref{A.9.84} and Theorem \ref{Div.3} again to finish the proof.
\end{proof}

\begin{thm}
\label{Div.4}
Let $\cF$ be a bounded below complex of strict Nisnevich (resp.\ strict \'etale, resp.\ Kummer \'etale) sheaves on $SmlSm/k$.
Then for every $X\in SmlSm/k$ and integer $i\in \Z$ there is an isomorphism
\begin{align*}
\bH_{dNis}^i(X,a_{dNis}^*\cF)&\cong \colimit_{Y\in X_{div}^{Sm}}\bH_{sNis}^i(Y,\cF)
\\
\text{(resp.\ }\bH_{d\acute{e}t}^i(X,a_{d\acute{e}t}^*\cF)&\cong \colimit_{Y\in X_{div}^{Sm}}\bH_{s\acute{e}t}^i(Y,\cF),
\\
\text{resp.\ }\bH_{l\acute{e}t}^i(X,a_{l\acute{e}t}^*\cF)&\cong \colimit_{Y\in X_{div}^{Sm}}\bH_{k\acute{e}t}^i(Y,\cF)\text{)}.
\end{align*}
The same holds if we replace $X_{div}^{Sm}$ by $X_{divsc}^{Sm}$.
\end{thm}
\begin{proof}
The proof is parallel to that of Theorem \ref{Div.3} if we use Lemma \ref{A.5.46} and Proposition \ref{A.9.82} as replacements for Lemma \ref{A.5.45} and Proposition \ref{A.9.80}.
To show the claim for $X_{divsc}^{Sm}$ use Corollary \ref{A.5.74} and Lemma \ref{A.5.75} instead.
\end{proof}

\subsection{Effective log motives}
\label{subsection:cotcolm}
Voevodsky's category of effective motives $\dmeff$ (resp.\ effective motives without transfers $\daeff$) is the homotopy category with respect to the $\A^1$-local descent model structure on
$\Co(\Shvtrkl)$ (resp.\ $\Co(\Shv_{Nis}(k,\Lambda))$).
The motive $M(X)$ of $X\in Sm/k$ is defined as the object $\Ztr(X)$ (resp.\ $\Lambda(X)$) in $\dmeff$ (resp.\ $\daeff$).
The category $\dmeff$ was initially constructed for bounded above complexes, see e.g., \cite[Lecture 14]{MVW}.
We refer to \cite[Definition 11.1.1]{CD12} for a generalization to unbounded complexes.
\vspace{0.1in}

By abuse of notation, we let $\boxx$ be shorthand for the class of projections $X\times \boxx\rightarrow X$, 
where $X\in lSm/k$.
We also consider $\boxx$ as the corresponding classes of morphisms of representable objects in $\Shvlogkl$ or $\Shvltrkl$.

\begin{df}
The derived category of effective log motives 
\[
\ldmeff
\]
\index[notation]{ldmeff @ $\ldmeff$} 
is the homotopy category of $\Co(\Shvltrkl)$ with respect to the $\boxx$-local descent model structure. 
The motive of $X\in lSm/k$ is the object 
\[
M(X):=a_{dNis}^*\Zltr(X)\in\ldmeff.
\] \index[notation]{M(X) @ $M(X)$} 

Similarly, 
by abuse of notation, 
we write $\ldaeff$ \index[notation]{ldaeff @ $\ldaeff$}  for the homotopy category of $\Co(\Shvlogkl)$ with respect to the $\boxx$-local descent model structure and $M(X)$ for $a_{dNis}^*\Lambda(X)$.
\end{df}

Recall that we set $\boxx:=(\P^1,\infty)$.
By Proposition \ref{A.8.19}, $\ldmeff$ and $\ldaeff$ are symmetric monoidal triangulated categories such that for $X,Y\in lSm/k$, we have
\[
M(X)\otimes M(Y)=M(X\times Y).
\]
\vspace{0.1in}

For any dividing Nisnevich cover $Y\rightarrow X$ in $lSm/k$, we note that 
\[
Y-\partial Y\rightarrow X-\partial X
\] 
is a Nisnevich cover in $Sm/k$.
Moreover, 
we have that \index[notation]{omegasharpstar @ $\omega_\sharp, R\omega^*$}
$$
\omega(\boxx)=\A^1.
$$ 
Thus by \eqref{A.8.0.2} and \eqref{A.8.0.3}, there are naturally induced adjunctions
\[
L\gamma_\sharp:\ldaeff\rightleftarrows \ldmeff:R\gamma^*,
\]
and 
\begin{equation}
\label{eq::omegasharp}
\omega_\sharp:\ldmeff\rightleftarrows \dmeff:R\omega^*.
\end{equation}
\vspace{0.1in}

Similarly, we have the naturally induced adjunctions
\[
L\gamma_\sharp:\daeff\rightleftarrows \dmeff:R\gamma^*,
\]
and
\[
\omega_\sharp:\ldaeff\rightleftarrows \daeff:R\omega^*.
\]
\vspace{0.1in}

Further, there is a diagram of adjunct functors
\[
\begin{tikzcd}
\ldaeff\arrow[d,shift left=0.75ex,leftarrow,"R\omega^*"]\arrow[d,shift right=0.75ex,"\omega_\sharp"']
\arrow[r,shift left=0.75ex,"L\gamma_\sharp"]\arrow[r,shift right=0.75ex,"R\gamma^*"',leftarrow]&\ldmeff\arrow[d,shift left=0.75ex,leftarrow,"R\omega^*"]\arrow[d,shift right=0.75ex,"\omega_\sharp"']   \\
\daeff\arrow[r,shift left=0.75ex,"L\gamma_\sharp"]\arrow[r,shift right=0.75ex,"R\gamma^*"',leftarrow] &\dmeff,
\end{tikzcd}\]
where $L\gamma_\sharp$ commutes with $\omega_\sharp$.

\begin{df}
\label{A.5.12}
A complex of dividing Nisnevich sheaves $\cF$ is {\it strictly $\boxx$-invariant}\index{invariant!strictly box @ strictly $\boxx$} if for every fs log scheme $X$ log smooth over $k$ and integer $i\in \Z$, 
the projection $X\times \boxx\rightarrow X$ induces an isomorphism on sheaf cohomology groups
\[
\bH_{dNis}^i(X,\cF)\rightarrow \bH_{dNis}^i(X\times \boxx,\cF).
\]
\end{df}

By definition of $\boxx$-local model structure, a complex of dividing Nisnevich sheaves $\cF$ is fibrant for the $\boxx$-local descent model structure if and only if $\cF$ is fibrant for the local model structure and $\cF$ is strictly $\boxx$-invariant.

\begin{prop}
\label{A.5.13}
Let $\cF$ be a strictly $\boxx$-invariant complex of dividing Nisnevich sheaves with log transfers. 
Then,
for every fs log scheme $X$ log smooth over $k$ and an integer 
$i\in \Z$, there is an isomorphism
\[
\hom_{\ldmeff}(M(X),\cF[i])
\cong 
\bH_{dNis}^i(X,\cF).
\]
Similar statement holds for strictly $\boxx$-invariant complex of dividing Nisnevich sheaves $\cF$ and $\ldaeff$.
\end{prop}
\begin{proof}
We focus on $\ldmeff$ since the proofs are similar.
We may assume that $\cF$ is a fibrant object in $\Co(\Shvltrkl)$ for the descent model structure given in \S \ref{Subsection:derivedcategories}. 
By Proposition \ref{A.8.19}(1), $\cF$ is a fibrant object in $\Co(\Shvltrkl)$ with respect to the $\boxx$-local descent model structure. 
Thus we have 
\[
\hom_{\ldmeff}(M(X),\cF[i])
\cong 
\hom_{\Deri(\Shvltrkl)}(a_{dNis}^*\Zltr(X),\cF[i]).
\]
Proposition \ref{A.8.8} finishes the proof.
\end{proof}


\begin{lem}
\label{A.5.73}
Let $\{\cF_i\}_{i\in I}$ be a set of strictly $\boxx$-invariant complexes of dividing Nisnevich sheaves.
Then $\bigoplus_{i\in I}\cF_i$ is strictly $\boxx$-invariant too.
The same holds for strictly $\boxx$-invariant complexes of dividing \'etale or log \'etale sheaves if $k$ has finite \'etale cohomological dimension for $\Lambda$-linear coefficients.
\end{lem}
\begin{proof}
We first treat the dividing Nisnevich case.
Due to Corollary \ref{Div.5} and Proposition \ref{bigsmall.5} for every $X\in lSm/k$ we have isomorphisms 
\begin{gather*}
\bH_{dNis}^i(X,\bigoplus_{i\in I}\cF)
\cong
\bigoplus_{i\in I}\bH_{dNis}^i(X,\cF),
\\
\bH_{dNis}^i(X\times \boxx,\bigoplus_{i\in I}\cF)
\cong
\bigoplus_{i\in I}\bH_{dNis}^i(X\times \boxx,\cF).
\end{gather*}
Combining these isomorphisms finishes the proof.
For the dividing \'etale and log \'etale cases, apply Corollary \ref{Div.6} instead.
\end{proof}

\begin{prop}
\label{A.5.33}
The effective log motives $M(X)[n]$ indexed by $X\in lSm/k$ and $n\in \Z$ form a set of compact generators for the categories $\ldaeff$ and $\ldmeff$.
\end{prop}
\begin{proof}
We focus on $\ldmeff$ since the proofs are similar. \index{compact generators!motives}
Let $\{\cF_i\}_{i\in I}$ be a set of strictly $\boxx$-invariant dividing Nisnevich complexes of sheaves with log transfers.
Then $\bigoplus_{i\in I} \cF_i$ is strictly $\boxx$-invariant too by Lemma \ref{A.5.73}.
Owing to Proposition \ref{A.5.13} we have
\begin{gather*}
\hom_{\ldmeff}(M(X),\bigoplus_{i\in I}\cF_i[n])
\cong 
\hom_{\mathbf{D}(\Shvlogkl)}(a_{dNis}^*\Lambda(X),\bigoplus_{i\in I}\cF_i[n]),
\\
\bigoplus_{i\in I}\hom_{\ldmeff}(M(X),\cF_i[n])
\cong 
\bigoplus_{i\in I}\hom_{\mathbf{D}(\Shvlogkl)}(a_{dNis}^*\Lambda(X),\cF_i[n]).
\end{gather*}
Proposition \ref{bigsmall.5} and Corollary \ref{Div.5} finish the proof.
\end{proof}

\begin{rmk}
\label{A.5.34}
Similarly, the objects $M(X)[n]$, where $X\in Sm/k$ and $n\in \Z$, are compact in $\dmeff$ (resp.\ $\daeff$).
The same objects generate $\dmeff$ (resp.\ $\daeff$) by the above or \cite[Theorem 11.1.13]{CD12} (resp.\ \cite[Proposition 5.2.38]{CD12}).
\end{rmk}

\begin{df}
\label{A.5.7}
If $f:Y\rightarrow X$ is a morphism in $lSm/k$,
let 
$$
M(Y\stackrel{f}\rightarrow X)
$$ 
be the cone in $\ldmeff$ associated with the complex 
$$
a_{dNis}^*\Zltr(Y)\rightarrow a_{dNis}^*\Zltr(X)
$$ 
in $\Co(\Shvltrkl)$.
We use the same notation for the object in $\ldaeff$ associated with the complex $a_{dNis}^*\Lambda(Y)\rightarrow a_{dNis}^*\Lambda(X)$ in $\Co(\Shvlogkl)$.
\vspace{0.1in}

For a simplicial object $\mathscr{X}$ in $lSm/k$, let $M(\mathscr{X})$ denote the object in $\ldmeff$ associated with $\Zltr(\mathscr{X})$ in $\Co(\Shvltrkl)$
We use the same notation for the object in $\ldaeff$ associated with $\Lambda(\mathscr{X})$ in $\Co(\Shvlogkl)$.
\end{df}


\begin{df}
We denote by
\[
\ldmeffprop
\]\index[notation]{ldmeffprop @ $\ldmeffprop$}
the smallest triangulated subcategory of $\ldmeff$ that is closed under small sums and shifts, 
and contains $M(X)$, 
where $X\in lSm/k$ and its underlying scheme $\underline{X}$ is proper over $\Spec{k}$. 
\end{df}

\begin{rmk}
We are interested in $\ldmeffprop$ because for a proper log smooth $X$, 
the log motive $M(X)$ is practically an object in $\dmeff$.
Indeed, if $k$ admits resolution of singularities, 
Theorem \ref{thm::dmeff=ldmeffprop} shows there is an equivalence of triangulated categories
\[
\ldmeffprop\simeq \dmeff.
\]
Moreover, 
we note that the corresponding statement is false in the \'etale topology; see Remark \ref{Etalemot.5}.
\end{rmk}

\begin{prop}
\label{A.5.42}
\begin{enumerate}
\item[(1)] For every strict Nisnevich distinguished square 
\[
\begin{tikzcd}
Y'\arrow[d]\arrow[r]&Y\arrow[d]\\
X'\arrow[r]&X
\end{tikzcd}
\]
in $lSm/k$, there is a naturally induced isomorphism 
\[
M(Y'\rightarrow X')\rightarrow M(Y\rightarrow X)
\]
in $\ldmeff$.
The same holds in $\ldaeff$.
\item[(2)]   
For every log modification $f:Y\rightarrow X$, there is a naturally induced isomorphism
\[
M(Y)\rightarrow M(X)
\]
in $\ldmeff$.
The same holds true in $\ldaeff$.
\end{enumerate}
\end{prop}
\begin{proof}
Let us treat first the case of $\ldaeff$. 
To show (1), it suffices to show the homomorphism
\[
\hom_{\ldaeff}(M(Y\rightarrow X),\cF)\rightarrow \hom_{\ldaeff}(M(Y'\rightarrow X'),\cF)
\]
is an isomorphism for any strictly $\boxx$-invariant complex of dividing Nisnevich sheaf $\cF$. 
For this it suffices to show there is an isomorphism
\begin{equation}
\label{A.5.42.1}\hom_{\Deri(\Shvlogtkl)}(a_{dNis}^*\Lambda(Y\rightarrow X),\cF)
\rightarrow 
\hom_{\Deri(\Shvlogtkl)}(a_{dNis}^*\Lambda(Y'\rightarrow X'),\cF).
\end{equation}
This follows from the exactness of \eqref{A.8.13.4}. 
Part (2) follows similarly, using a dividing distinguished square instead of a strict Nisnevich distinguished square.
\vspace{0.1in}

To show (1) in the case of $\ldmeff$, it suffices to show the homomorphism
\[
\hom_{\ldaeff}(M(Y\rightarrow X),\cF)
\rightarrow 
\hom_{\ldaeff}(M(Y'\rightarrow X'),\cF)
\]
is an isomorphism for any strictly $\boxx$-invariant complex of dividing Nisnevich sheaf $\cF$ with log transfers. 
This follows from Proposition \ref{A.5.13} since \eqref{A.5.42.1} is an isomorphism. 
We omit the proof of (2) because it is similar to (1).
\end{proof}

\begin{exm}
\label{A.5.43}
The commutative square
\begin{equation}
\label{A.5.43.1}
\begin{tikzcd}
\A_\N\arrow[r]\arrow[d]&(\P^1,0)\cong \boxx\arrow[d]\\
\A^1\arrow[r]&\P^1
\end{tikzcd}
\end{equation}
is a strict Nisnevich distinguished square, so that the induced morphism
\[
M(\A_\N\rightarrow \A^1)\rightarrow M(\boxx\rightarrow \P^1)\cong \Lambda(1)[2]
\]
is an isomorphism in $\ldmeff$.
This is analogous to the isomorphism 
\[
M(\A^1/\mathbb{G}_m)\rightarrow M(\P^1/\A^1)\cong \Lambda(1)[2]
\]
in $\dmeff$ and $\daeff$.
Indeed, applying $\omega$ to the square \eqref{A.5.43.1} we retain the Nisnevich distinguished square
\[
\begin{tikzcd}
\mathbb{G}_m\arrow[d]\arrow[r]&\P^1-\{0\}\arrow[d]
\\
\A^1\arrow[r]&\P^1
\end{tikzcd}
\]
for the standard covering of the projective line.
However, in the log situation, 
we note that $\Zltr(\A_\N)$ is not a subsheaf of $\Zltr(\A^1)$.
\end{exm}

\begin{df}\label{A.5.66}
Recall that $div$ is the class of morphisms in $lSm/k$ consisting of log modifications $Y\rightarrow X$ in $lSm/k$.
We denote by $(\boxx,div)$ the class of morphisms in $lSm/k$ consisting of all the projections $X\times \boxx\rightarrow X$, where $X\in lSm/k$, and the log modifications $Y\rightarrow X$ in $lSm/k$.
\end{df}

\begin{exm}
\label{A.4.28}
A useful property of $\A^1$-invariant presheaf $\cF$ with transfers is that for any Zariski cover $U_1$ and $U_2$ of an open subscheme $U$ of $\A^1$, 
there is a split exact sequence
\begin{equation}
\label{A.4.28.1}
0\rightarrow \cF(U)\rightarrow \cF(U_1)\oplus \cF(U_2)\rightarrow \cF(U_1\cap U_2)\rightarrow 0, 
\end{equation}
see \cite[Lemma 22.4]{MVW}.
This is no longer true for $\boxx$-invariant presheaves with log transfers.
In what follows, we check that $\Zltr(\A^1)$ provides a counterexample.
\vspace{0.1in}

For $X\in lSm/k$, we have that
\[
\Zltr(\A^1)(X)=\Cor(\underline{X},\A^1).
\]
Let $V$ be an elementary correspondence from $\underline{X}\times \P^1$ to $\A^1$.
For any point $x$ of $\underline{X}$, consider the restriction $V_x$ of $V$ to $\{x\}\times \P^1$.
Since $V_x$ is finite over $\{x\}\times \P^1$, $V_x$ is proper over $k$.
Thus the image of $V'$ in $\A^1$ is again proper, so the image should be a union of points.
Let $y$ be a rational point of $\P^1$, 
let $V'$ (resp.\ $V_x'$) be the restriction of $V$ to $\underline{X}\times \{y\}$ (resp.\ $\{x\}\times \{y\}$) and let 
$p:\underline{X}\times \P^1\rightarrow \underline{X}$ and $p_x:\{x\}\times \P^1\rightarrow \{x\}$ be the projections.
From the above we have $p_x^*V_x'=V_x$ and hence $p^*V'=V$.
Thus the morphism
\[
p^*:\Zltr(\A^1)(X)\rightarrow \Zltr(\A^1)(X\times \P^1)
\]
is surjective.
Since $p$ has a section, it follows that $p^*$ is injective.
So far we have observed that $\Zltr(\A^1)$ is $\boxx$-invariant.
\vspace{0.1in}

Next, we turn to the split exactness of \eqref{A.4.28.1}. 
Set $U:=\A^1=\Spec{k[x]}$, $U_1:=\A^1-\{0\}$, and $U_2:=\A^1-\{1\}$.
Then there is an element
\[
[1/x+1/(x-1)]\in \Zltr(\A^1)(U_1\cap U_2),
\]
where for $f:U_1\cap U_2\rightarrow \A^1$, we let $[f]$ denotes the associated correspondence.
On the other hand,
the closure of $[1/x+1/(x-1)]$ in $U_j\times \A^1$ is not surjective over $U_j$, i.e., it does not define an elementary log correspondence in $\lCor(U_j,\A^1)$ for $j=1,2$.
It follows that the homomorphism
\[
\Zltr(U_1)\oplus \Zltr(U_2)\rightarrow \Zltr(U_1\cap U_2)
\]
is not even surjective.
\end{exm}

\begin{rmk}
Let $\cF$ be an $\A^1$-invariant Nisnevich sheaf $\cF$ with transfers.
The spilt exactness of \eqref{A.4.28.1} is used to prove that also the presheaf $H_{Nis}^n(-,\cF)$ is $\A^1$-invariant, 
see e.g., \cite[Theorem 24.1]{MVW}.
Due to Example \ref{A.4.28},
it is unclear whether for any $\boxx$-invariant strict Nisnevich sheaf $\cF$ with transfers, the presheaf $H_{sNis}^n(-,\cF)$ is $\boxx$-invariant.
\end{rmk}

\subsection{Effective \'etale log motives} 
In this subsection we define the \'etale versions of $\ldmeff$ and $\ldaeff$.

\begin{df}
\label{Etalemot.1}
The derived category of effective dividing \'etale log motives 
\[
\ldmeffet
\] \index[notation]{ldmeffet @ $\ldmeffet$}
is the homotopy category of 
\[
\Co(\Shv_{d\acute{e}t}^{\rm ltr}(k,\Lambda))
\] 
with respect to the $\boxx$-local descent model structure. 
The motive of $X\in lSm/k$ is the object 
\[
M(X):=a_{d\acute{e}t}^*\Zltr(X)\in\ldmeffet.
\] 

Similarly, 
we write $\ldaeffet$ \index[notation]{ldaeffet @ $\ldaeffet$} for the homotopy category of $\Co(\Shv_{d\acute{e}t}^{\rm log}(k,\Lambda))$ with respect to the $\boxx$-local descent model structure and $M(X)$ for $a_{d\acute{e}t}^*\Lambda(X)$.
\vspace{0.1in}

The derived category of effective log \'etale log motives
\[
\ldmefflet
\] \index[notation]{ldmefflet @ $\ldmefflet$}
is the homotopy category of $\Co(\Shv_{l\acute{e}t}^{\rm ltr}(k,\Lambda))$ with respect to the $\boxx$-local descent model structure.
The effective log \'etale motive of $X\in lSm/k$ is the object 
\[
M(X):=a_{l\acute{e}t}^*\Zltr(X)\in\ldmefflet.
\] 

Similarly,
we write $\ldaefflet$ \index[notation]{ldaefflet @ $\ldaefflet$} for the homotopy category of $\Co(\Shv_{l\acute{e}t}^{\rm log}(k,\Lambda))$ with respect to the $\boxx$-local descent model structure and $M(X)$ for $a_{l\acute{e}t}^*\Lambda(X)$.
\end{df}

By Proposition \ref{A.8.19}, $\ldmeffet$, $\ldaeffet$, $\ldmefflet$, and $\ldaefflet$ are symmetric monoidal triangulated categories such that for all $X,Y\in lSm/k$,
\[
M(X)\otimes M(Y)=M(X\times Y).
\]
\vspace{0.1in}

For the descent model structures, 
there are Quillen adjunctions 
\[
a_{d\acute{e}t}^*:\Co(\Shv_{dNis}^{\rm ltr}(k,\Lambda))\rightleftarrows \Co(\Shv_{d\acute{e}t}^{\rm ltr}(k,\Lambda)):a_{d\acute{e}t*}
\]
and 
\[
a_{l\acute{e}t}^*:\Co(\Shv_{d\acute{e}t}^{\rm ltr}(k,\Lambda))\rightleftarrows \Co(\Shv_{l\acute{e}t}^{\rm ltr}(k,\Lambda)):a_{l\acute{e}t*}.
\]
\vspace{0.1in}

Owing to \cite[Theorem 3.3.20(1)]{MR1944041}, this is again a Quillen adjunction for the $\boxx$-local model structures.
By passing to the homotopy categories, we obtain the adjunctions
\[
La_{d\acute{e}t}^*:\ldmeff\rightleftarrows \ldmeffet:Ra_{d\acute{e}t*}
\]
and
\[
La_{l\acute{e}t}^*:\ldmeffet\rightleftarrows \ldmefflet:Ra_{l\acute{e}t*}.
\]

\begin{df}
Let $Y\rightarrow X$ be a morphism in $lSm/k$, and let $\mathscr{X}$ be a simplicial object of $lSm/k$.
As in Definition \ref{A.5.7}, we associate objects 
$$
M(Y\stackrel{f}\rightarrow X)
\text{ and }
M(\mathscr{X})
$$ 
in $\ldmeffet$, $\ldaeffet$, $\ldmefflet$, and $\ldaefflet$.
\end{df}

\begin{df}
\label{Etalemot.2}
A complex of dividing \'etale sheaves $\cF$ is {\it strictly $\boxx$-invariant}\index{invariant!strictly box @ strictly $\boxx$} if for every fs log scheme $X$ log smooth over $k$ and integer $i\in \Z$, 
the projection $X\times \boxx\rightarrow X$ induces an isomorphism on sheaf cohomology groups
\[
\bH_{d\acute{e}t}^i(X,\cF)\rightarrow \bH_{d\acute{e}t}^i(X\times \boxx,\cF).
\]

Similarly, a complex of log \'etale sheaves $\cF$ is {\it strictly $\boxx$-invariant} if for every fs log scheme $X$ log smooth over $k$ and integer $i\in \Z$, 
the projection $X\times \boxx\rightarrow X$ induces an isomorphism on sheaf cohomology groups
\[
\bH_{l\acute{e}t}^i(X,\cF)\rightarrow \bH_{l\acute{e}t}^i(X\times \boxx,\cF).
\]
\end{df}

\begin{prop}
\label{Etalemot.3}
Let $\cF$ be a strictly $\boxx$-invariant complex of dividing \'etale (resp.\ log \'etale) sheaves with log transfers. 
Then,
for every fs log scheme $X$ log smooth over $k$ and an integer 
$i\in \Z$, there is an isomorphism
\[
\hom_{\ldmeffet}(M(X),\cF)
\cong 
\bH_{d\acute{e}t}^i(X,\cF)
\]
\[
\text{(resp.\ }
\hom_{\ldmefflet}(M(X),\cF)
\cong 
\bH_{l\acute{e}t}^i(X,\cF)
\text{)}.
\]
\end{prop}
\begin{proof}
The proof is parallel to that of Proposition \ref{A.5.13}.
\end{proof}

\begin{df}
\label{Etalemot.4}
We denote by
\[
\ldmeffetprop
\;
\text{(resp.\ }
\ldmeffletprop
\text{)}
\] \index[notation]{ldmeffetprop @ $\ldmeffetprop$} \index[notation]{ldmeffletprop @ $\ldmeffletprop$}
the smallest triangulated subcategory of $\ldmeffet$ (resp.\ $\ldmefflet$) that is closed under small sums and shifts, 
and contains $M(X)$, where $X\in lSm/k$ and its underlying scheme is proper over $\Spec{k}$.
\end{df}

\begin{rmk}
\label{Etalemot.5}
Suppose $k$ has positive characteristic $p>0$.
In contrast to Theorem \ref{thm::dmeff=ldmeffprop}, we do {\it not} have an equivalence of triangulated categories between 
\[
\ldmeffetpropZ
\text{ and }
\dmeffetZ.
\]

To wit, note that $\Z/p\cong 0$ in $\dmeffetZ$ by \cite[Example 9.16]{MVW}.
Notice that since $\Z\cong M(k)\in \ldmeffetpropZ$, we have $\Z/p\in \ldmeffetpropZ$.
In Example \ref{Nontriviality} we will see that $\Z/p$ is nontrivial in $\ldmeffetpropZ$.
\end{rmk}

\begin{prop}
Suppose that $k$ has finite $\acute{e}t$-cohomological dimension for $\Lambda$-linear coefficients.
Then the effective log motives $M(X)[n]$ indexed by $X\in lSm/k$ and $n\in \Z$ form a set of compact generators for the categories
\[
\ldaeffet,\;\ldmeffet,\; \ldaefflet,\text{ and }\ldmefflet.
\]
\end{prop}
\begin{proof}
Argue as in Proposition \ref{A.5.33} by appealing to Corollary \ref{Div.6}.
\end{proof}

\subsection{Construction of motives using $SmlSm/k$}
\label{ssecequivalence}
We may use the results of Section \ref{Subsec::Sheaves.SmlSm} to construct categories of effective log motives for objects in $SmlSm/k$.

\begin{df}
\label{SmlSm.1}
Let 
\[
\ldmeffSmlSm
\] \index[notation]{ldmeffSmlSm @ $\ldmeffSmlSm$}
be the homotopy category of 
\[
\Co(\Shvltrklsm)
\] 
with respect to the $\boxx$-local descent model structure.

Similarly, we let 
\[
\ldmeffetSmlSm
\] \index[notation]{ldmeffetSmlSm @ $\ldmeffetSmlSm$}
be the homotopy category of 
\[
\Co(\Shv_{d\acute{e}t}^{\rm ltr}(SmlSm/k,\Lambda))
\] 
with respect to the $\boxx$-local descent model structure. 
\end{df}

\begin{prop}
\label{SmlSm.2}
There are naturally induced equivalences of triangulated categories
\begin{align*}
\ldmeffSmlSm & \simeq \ldmeff, \\
\ldmeffetSmlSm & \simeq \ldmeffet, \\
\ldmeffletSmlSm & \simeq \ldmefflet.
\end{align*}
\end{prop}
\begin{proof}
We will only consider the dividing Nisnevich topology since the proofs are similar.
Owing to Proposition \ref{A.5.58} there is an equivalence of categories
\[
\Shvltrklsm
\simeq 
\Shvltrkl.
\]

Let us consider the descent model structures on the categories
\[
\Co(\Shvltrklsm)
\]
and
\[
\Co(\Shvltrkl).
\]
Since the descent model structures are invariant under equivalences of categories, 
there is a Quillen equivalence for the descent model structures
\[
\Co(\Shvltrklsm)\rightleftarrows \Co(\Shvltrkl).
\]
Owing to \cite[Theorem 3.3.20(1)]{MR1944041} this becomes a Quillen equivalence if we impose the $\boxx$-local descent model structure.
\end{proof}

\begin{const}
\label{SmlSm.3}
Suppose $\cF$ is a complex of dividing Nisnevich sheaves with log transfers on $SmlSm/k$.
Recall that $X_{div}^{Sm}$ denotes the category of log modifications $Y\rightarrow X$, 
where $Y\in SmlSm/k$.
Owing to Lemma \ref{A.5.59}, we can naturally extend $\cF$ to a complex $i_\sharp \cF$ on $lSm/k$ via the formula
\[
\iota_\sharp \cF(X)
:=
\colimit_{Y\in X_{div}^{Sm}}\cF(Y).
\]
By Lemma \ref{A.5.54} there is an equivalence of categories
\[
\Shv_{dNis}(SmlSm/k,\Lambda)\simeq \Shvlogkl.
\]
This implies that for every $X\in SmlSm/k$ and integer $i\in \Z$ we have an isomorphism of cohomology groups
\begin{equation}
\label{SmlSm.3.1}
H_{dNis}^i(X,\cF)\cong H_{dNis}^i(X,\iota_\sharp \cF).
\end{equation}
\vspace{0.1in}

If $\cF$ is strict $\boxx$-invariant, 
i.e., 
for every $X\in SmlSm/k$
\[
H_{dNis}^i(X,\cF)\cong H_{dNis}^i(X\times \boxx,\cF),
\]
\eqref{SmlSm.3.1} shows $\iota_\sharp \cF$ is strictly $\boxx$-invariant too.
Owing to Proposition \ref{A.5.13}, for every $X\in SmlSm/k$ and integer $i\in \Z$, we obtain an isomorphism
\begin{equation}
\label{SmlSm.3.2}
\hom_{\ldmeff}(M(X),\iota_\sharp \cF[i])\cong H_{dNis}^i(X,\cF).
\end{equation}
\vspace{0.1in}

If $\cF$ is a complex of dividing \'etale (resp.\ log \'etale) sheaves with log transfers on $SmlSm/k$, then we can argue similarly to conclude the isomorphisms
\begin{equation}
\label{SmlSm.3.3}
\hom_{\ldmeffet}(M(X),\iota_\sharp \cF[i]) \cong H_{d\acute{e}t}^i(X,\cF)
\end{equation}
and
\begin{equation}
\label{SmlSm.3.4}
\hom_{\ldmefflet}(M(X),\iota_\sharp \cF[i])\cong H_{l\acute{e}t}^i(X,\cF).
\end{equation}
\end{const}
\newpage

\section{Calculus of fractions and homotopy theory of presheaves}
\label{section:calculus}

Throughout this section, $T$ is a small category with finite limits.
Let $\pt$ denote the final object in $T$. 
Suppose $\mathcal{A}$ is a set of morphisms in $T$ that admits a calculus of right fractions.
We let $I$ be an object of $T$ admitting morphisms
\[
i_0,i_1:{\rm pt}\rightarrow I,\;\; p:I\rightarrow {\rm pt}
\]
in $T$ and a morphism
\[
\mu:I\times I\rightarrow I
\]
in $T[\mathcal{A}^{-1}]$ satisfying
\[
\mu(i_0\times {\rm id})=\mu({\rm id}\times i_0)=i_0p,
\]
\[
\mu(i_1\times {\rm id})=\mu({\rm id}\times i_1)={\rm id}.
\]
Moreover, 
we assume 
\[
i_0\amalg i_1:{\rm pt}\amalg {\rm pt}\rightarrow I
\] 
is a monomorphism. 
Note that $I$ is an interval of $T[\mathcal{A}^{-1}]$ in the sense of \cite[\S 2.3]{MV}.
\vspace{0.1in}

Recall that $\mathcal{A}$ admits a \emph{calculus of right fractions}\index{calculus of right fractions} if the following conditions holds \cite[dual of I.2.2]{GZFractions}.

\begin{enumerate}
\item[(i)] The set $\mathcal{A}$ contains the isomorphisms and is closed under composition.
\item[(ii)] (Right Ore condition)\index{right Ore condition} For any diagram $Y\stackrel{f}\rightarrow X\stackrel{g}\leftarrow X'$ in $\cC$ with $f\in \cA$, there is a commutative diagram
\[
\begin{tikzcd}
Y'\arrow[d,"f'"']\arrow[r]&Y\arrow[d,"f"]\\
X'\arrow[r,"g"]&X
\end{tikzcd}
\]
in $\cC$ such that $f'\in \cA$.
\item[(iii)] (Right cancellability condition)\index{right cancellability condition} For any diagram
\[
\begin{tikzcd}
Z\arrow[r,shift left=0.5ex,"f"]\arrow[r,shift right=0.5ex,"g"']&Y\arrow[r,"u"]&X
\end{tikzcd}
\]
in $\cC$ with $u\in \cA$ and $u\circ f=u\circ g$, there is a morphism $v:W\rightarrow Z$ in $\cA$ such that $f\circ v=g\circ v$.
\end{enumerate}
\vspace{0.1in}

In \cite[\S 2.3]{MV}, Morel-Voevodsky constructs the singular functor ${\rm Sing}$ for an interval $I$ of $T$. 
Since our assumption is that $I$ is an interval of $T[\cA^{-1}]$, 
instead of $T$, 
we cannot apply the same arguments verbatim. 
For example, 
in the proof of \cite[Corollary 3.5, p.\ 89]{MV} it is essential that $\mu\in T$ when dealing with the multiplication map $\mu:I\times I\rightarrow I$.
In this section, 
we construct a singular functor ${\rm Sing}$ for $\mu$ in $T[\cA^{-1}]$.
We argue that there is another way of constructing ${\rm Sing}$ in \cite{Conservativity} without assuming the existence of $\mu$. 
The style of this construction is different from that in \cite[\S 2.3]{MV}.
See \cite[Section 4.5]{AdditiveHomotopy} for a way of constructing the Sing functor obtained by 
forming a weak interval object \cite[Definition 2.19]{AdditiveHomotopy} in the category of simplicial presheaves on $T$.

\subsection{Properties of localization functors}

\begin{df}
Let $\sPsh(T)$ denote the category of simplicial presheaves on $T$. 
A morphism $f:F\rightarrow G$ of simplicial presheaves is called 
\begin{enumerate}[(i)]
\item a {\it pointwise weak equivalence}\index{equivalence!pointwise weak} if $F(X)\rightarrow G(X)$ is a weak equivalence of simplicial sets for any object $X$ of $T$.
\item an {\it injective cofibration} if $F(X)\rightarrow G(X)$ is a cofibration of simplicial sets for any object $X$ of $T$.
\item an {\it injective fibration} if it has the right lifting property with respect to injective cofibrations that are pointwise weak equivalences.
\item a {\it projective fibration} if $F(X)\rightarrow G(X)$ is a fibration of simplicial sets for any object $X$ of $T$.
\item a {\it projective cofibration} if it has the left lifting property with respect to projective fibrations that are pointwise weak equivalences.
\end{enumerate}
  
Recall that the model structure on $\sPsh(T)$ with the pointwise weak equivalences, injective cofibrations, and injective fibrations is called the {\it injective model structure}\index{model structure!injective} (see \cite[Theorem 5.8]{JardineLocal}). 
The model structure on $\sPsh(T)$ with the pointwise weak equivalences, projective cofibrations, and projective fibrations is called the {\it projective model structure}\index{model structure!projective} (see \cite[\S 5.5]{JardineLocal}).

\end{df}
\begin{df}
\label{A.1.3}
Associated to the localization functor
\[
v:T\rightarrow T[\mathcal{A}^{-1}]
\]
we let 
\[
v^*:\sPsh(T[\mathcal{A}^{-1}])\rightarrow \sPsh(T)
\]
denote the functor given by 
\[
v^*F(X)
:=
F(X).
\] 
Here $F$ is a simplicial presheaf on $T[\mathcal{A}^{-1}]$ and $X$ is an object of $T$. 
\vspace{0.1in}

By \cite[Proposition I.5.1]{SGA4}, there exist adjoint functors 
\[
\begin{tikzcd}
\sPsh(T)\arrow[rr,shift left=1.5ex,"v_\sharp "]\arrow[rr,"v^*" description,leftarrow]\arrow[rr,shift right=1.5ex,"v_*"']&&\sPsh(T[\mathcal{A}^{-1}]),
\end{tikzcd}
\]
where $v_\sharp$ is left adjoint to $v^*$, and $v^*$ is left adjoint to $v_*$.
\end{df}

\begin{df}
\label{A.5.1}
For any object $X$ of $T$, we denote by $\cA\downarrow X$ the category where an object is a morphism $Y\rightarrow X$ in $\cA$, and a morphism is a commutative diagram
\[
\begin{tikzcd}
Y'\arrow[rr,"g"]\arrow[rd]&&Y\arrow[ld]\\
&X
\end{tikzcd}
\]
in $T$ such that $g\in \cA$. 
\end{df}

Since $\mathcal{A}$ admits a calculus of right fractions, 
by the dual of \cite[I.2.3]{GZFractions}, 
we have the formula
\begin{equation}
\label{A.1.3.1}
\hom_{T[\mathcal{A}^{-1}]}(X,S)
\cong 
\colimit_{Y\in \mathcal{A}\downarrow X}\hom_T(Y,S).
\end{equation}

\begin{prop}
\label{A.1.4}
The functor $v_\sharp $ is given by
\begin{equation}
\label{A.1.4.1}
v_\sharp F(X)\cong \colimit_{Y\in \mathcal{A}\downarrow X} F(Y).
\end{equation}
\end{prop}
\begin{proof}
By \cite[I.5.1.1]{SGA4}, we reduce to the case when $F=Z\times \Delta[n]$, where $Z\in T$. 
In this case \eqref{A.1.4.1} follows from \eqref{A.1.3.1}.
\end{proof}

\begin{prop}
\label{A.2.1}
The adjunction
\[
v_\sharp :\sPsh(T)\rightleftarrows \sPsh(T[\mathcal{A}^{-1}]):v^*
\]
is a Quillen pair with respect to the projective model structures.
\end{prop}
\begin{proof}
The functor $v:T\rightarrow T[\mathcal{A}^{-1}]$ commutes with finite limits by using the dual of \cite[Proposition I.3.1]{GZFractions}. 
Thus $v$ is a morphism of sites for the trivial topologies. 
We are done by \cite[Corollary 8.3]{MR2034012}.
\end{proof}

\begin{prop}
\label{A.2.7}
The functor 
\[
v^*:\sPsh(T[\mathcal{A}^{-1}])\rightarrow \sPsh(T)
\]
preserves pointwise weak equivalences and injective cofibrations. In particular, the adjunction
\[
v^*:\sPsh(T[\mathcal{A}^{-1}])\rightleftarrows \sPsh(T):v_*
\]
is a Quillen pair with respect to the injective model structures.
\end{prop}
\begin{proof}
Pointwise weak equivalences and injective cofibrations are defined objectwise, so the claim follows from the fact that $v:T\rightarrow T[\mathcal{A}^{-1}]$ is essentially surjective (\cite[1.1]{GZFractions}).
\end{proof}

\begin{df}
For simplicial presheaves $F$ and $G$ on $T$, let ${\mathbf S}(F,G)$ denote the simplicial enrichment (or internal hom), which is given by the simplicial set 
\[
{\mathbf S}(F,G)_n:=\hom_{\sPsh(T)}(F\times \Delta[n],G).
\]
\end{df}

\begin{df}
\label{A.2.3}
Let $\mathcal{A}$ be a set of morphisms in $T$. 
\begin{enumerate}[(1)]
\item 
A simplicial presheaf $F$ on $T$ is {\it $\mathcal{A}$-invariant}\index{invariant!A@$\cA$} if the morphism 
\[
F(X\times Z)\rightarrow F(Y\times Z)
\] 
is a pointwise weak equivalence for any morphism $Y\rightarrow X$ in $\mathcal{A}$ and any $Z\in T$.
\item 
A morphism $f:F\rightarrow G$ of simplicial presheaves on $T$ is a {\it $\mathcal{A}$-weak equivalence}\index{equivalence!A-weak@$\cA$-weak} if for any $\mathcal{A}$-invariant simplicial presheaf $H$ on $T$, 
$f$ induces a weak equivalence of simplicial sets 
\[
{\mathbf S}(G,H)\rightarrow {\mathbf S}(F,H).
\]
\item 
A morphism $f:F\rightarrow G$ of simplicial presheaves on $T$ is an {\it injective $\mathcal{A}$-fibration} if it has the right lifting property with respect to injective cofibrations that are $\mathcal{A}$-weak equivalences.
\end{enumerate}
  
Recall that the left Bousfield localization of the injective model structure on $\sPsh(T)$ by $\mathcal{A}$-weak equivalences is called the {\it injective $\mathcal{A}$-local model structure}, 
which exists by \cite[Theorem 2.5, p.\ 71]{MV}.
Note that the classes of $\mathcal{A}$-weak equivalences, injective cofibrations, and injective $\mathcal{A}$-fibrations form the injective $\mathcal{A}$-local model structure.

\end{df}
\begin{prop}
\label{A.2.2}
The functor $v_\sharp :\sPsh(T)\rightarrow \sPsh(T[\mathcal{A}^{-1}])$ maps $\mathcal{A}$-weak equivalences to pointwise weak equivalences.
\end{prop}
\begin{proof}
Let $\mathcal{C}$ be the class of morphisms $f:F\rightarrow G$ of simplicial presheaves on $T$ such that $v_\sharp f$ is a pointwise weak equivalence. 
By \cite[Proposition 4.2.74]{Ayo07}, it suffices to check the following conditions.
\begin{enumerate}[(i)]
\item 
$\mathcal{C}$ contains pointwise weak equivalences and morphisms of the form
\[
u\times {\rm id}:Y\times F\rightarrow X\times F,
\]
where $F$ is a simplicial presheaf on $T$, and $u:Y\rightarrow X$ is a morphism in $\mathcal{A}$.
\item $\mathcal{C}$ has the 2-out-of-3 property.
\item The class of projective cofibrations in $\mathcal{C}$ is stable by pushouts and transfinite compositions.
\end{enumerate}

By \cite[Proposition I.3.1]{GZFractions} the functor $v$ commutes with finite limits. 
Therefore $v_\sharp (u\times {\rm id})$ is an isomorphism since $v_\sharp u$ is an isomorphism. 
If $f:F\rightarrow G$ is a pointwise weak equivalence, then $F(X)\rightarrow G(X)$ is a weak equivalence for any object $X$ of $T$. 
According to \cite[Proposition 7.3]{Dug}, filtered colimits of pointwise weak equivalences are pointwise weak equivalences. 
Since $\mathcal{A}$ admits a calculus of right fractions, we see that $\mathcal{A}\downarrow X$ is cofiltered. 
From the formula \eqref{A.1.4.1} we deduce that $v_\sharp F(X)\rightarrow v_\sharp G(X)$ is a pointwise weak equivalence.
This verifies condtion (i).
\vspace{0.1in}

The conditions (ii) and (iii) follow from Proposition \ref{A.2.1} and the fact that the class of pointwise weak equivalences and class of projective cofibrations satisfy (ii) and (iii).
\end{proof}

\begin{prop}
\label{A.1.5}
Let $F$ be a simplicial presheaf on $T$. 
If $F$ is $\mathcal{A}$-invariant, then the unit map
\[
F\rightarrow v^*v_\sharp F
\]
is a pointwise weak equivalence.
\end{prop}
\begin{proof}
By Proposition \ref{A.1.4}, we need to show that the induced morphism
\begin{equation}
\label{A.1.5.1}
F(X)\rightarrow \colimit_{Y\in \mathcal{A}\downarrow X}F(Y)
\end{equation}
is a pointwise weak equivalence for any $X\in T$. 
Each $F(X)\rightarrow F(Y)$ with $Y\in \mathcal{A}\downarrow X$ is a pointwise weak equivalence since $F$ is $\mathcal{A}$-invariant, 
and the category $\mathcal{A}\downarrow X$ is cofiltered since $\mathcal{A}$ admits a calculus of right fractions. 
Thus \eqref{A.1.5.1} is a pointwise weak equivalence since any filtered colimit of pointwise weak equivalences is again a pointwise weak equivalence by \cite[Proposition 7.3]{Dug}.
\end{proof}

\begin{prop}
\label{A.2.5}
The unit of the adjunction ${\rm id}\rightarrow v^*v_\sharp $ is an $\mathcal{A}$-weak equivalence.
\end{prop}
\begin{proof}
Let $F$ be a simplicial presheaf on $T$ and choose an injective $\mathcal{A}$-fibrant resolution $f:F\rightarrow G$.
Note that $\cG$ is $\cA$-invariant.
Then $v^*v_\sharp f$ is a pointwise weak equivalence according to Propositions \ref{A.2.7} and \ref{A.2.2}. 
Hence to show that $F\rightarrow v^*v_\sharp F$ is an $\mathcal{A}$-weak equivalence, it suffices to show that $G\rightarrow v^*v_\sharp G$ is an $\mathcal{A}$-weak equivalence. 
For the latter, see Proposition \ref{A.1.5}.
\end{proof}
  
\begin{df}
Let $(I,\mathcal{A})$ denote the set of morphisms $\{i_0:\pt\rightarrow I\}\cup \mathcal{A}$ in $T$.
From Definition \ref{A.2.3}(2), we have the notions of $I$-weak equivalences and $(I,\cA)$-weak equivalences.
\end{df}

\begin{prop}
\label{A.2.4}
The functor 
\[
v^*:\sPsh(T[\mathcal{A}^{-1}])\rightarrow \sPsh(T)
\]
maps $I$-weak equivalences to $(I,\mathcal{A})$-weak equivalences.
\end{prop}
\begin{proof}
Let $\mathcal{C}$ be the class of morphisms $f:F\rightarrow G$ of simplicial presheaves on $T[\mathcal{A}^{-1}]$ such that $v^*f$ is an $(I,\mathcal{A})$-weak equivalence. 
By \cite[Proposition 4.2.74]{Ayo07}, it suffices to check the following conditions.
\begin{enumerate}[(i)]
\item 
$\mathcal{C}$ contains pointwise weak equivalences and morphisms of the form
\[
i_0\times {\rm id}:{\rm pt}\times F\rightarrow I\times F,
\]
where $F$ is a simplicial presheaf on $T[\mathcal{A}^{-1}]$.
\item $\mathcal{C}$ has the 2-out-of-3 property.
\item The class of morphisms in $\mathcal{C}$ that are injective cofibrations is stable by pushouts and transfinite compositions.
\end{enumerate}

Since $v^*$ is a right adjoint of $v_\sharp $, it follows that $v^*$ commutes with limits. 
Thus to show $v^*(i_0\times {\rm id})$ is an $(I,\mathcal{A})$-weak equivalence, it suffices to show that $v^*i_0$ is an $(I,\mathcal{A})$-weak equivalence.
Note that $v^*i_0\cong v^*v_\sharp i_0$.
Consider the induced commutative diagram
\[
\begin{tikzcd}
{\rm pt}\arrow[d]\arrow[r,"i_0"]&I\arrow[d]\\
v^*v_\sharp {\rm pt}\arrow[r,"v^*v_\sharp i_0"]&v^*v_\sharp I
\end{tikzcd}\]
of simplicial presheaves on $T$. 
The upper horizontal morphism is an $I$-weak equivalence, and the vertical morphisms are $\mathcal{A}$-weak equivalences by Proposition \ref{A.2.5}. 
Hence $v^*v_\sharp i_0$ is an $(I,\mathcal{A})$-weak equivalence. 
By Proposition \ref{A.2.7}, $\mathcal{C}$ contains the pointwise weak equivalences. 
This verifies condition (i).
\vspace{0.1in}

The conditions (ii) and (iii) follow from Proposition \ref{A.2.7} since the $(I,\mathcal{A})$-weak equivalences and the injective cofibrations satisfy (ii) and (iii).
\end{proof}

\subsection{Construction of the singular functor ${\rm Sing}$}

Since $I$ is an interval in $T[\mathcal{A}^{-1}]$, we are entitled to the cosimplicial presheaf 
\[
\Delta_I^\bullet\in T[\mathcal{A}^{-1}]
\]
defined as follows, see \cite[p.\ 88]{MV}.
On objects we set $\Delta_I^n:=I^n$.
Let $p:I^n\rightarrow {\rm pt}$ denote the canonical morphism to the final object, and let $pr_k:I^k\rightarrow I$ denote the $k$-th projection.
For any morphism $f:\{0,\ldots,n\}\rightarrow \{0,\ldots,m\}$ in $\Delta$, let 
$$
\phi(f):\{1,\ldots,m\}\rightarrow \{0,\ldots,n+1\}
$$ 
denote the function given by
\[
\phi(f)(i):=\left\{
\begin{array}{cl}
\min\{l\in \{0,\ldots,n\}:f(l)\geq i\}&\text{if this set is nonempty}\\
n+1&\text{otherwise}.
\end{array}
\right.
\]
The morphism $\Delta_I(f):I^n\rightarrow I^m$ is given by
\[
pr_k\circ \Delta_I(f):=\left\{
\begin{array}{cl}
pr_{\phi(f)(k)}&\text{if }\phi(f)(k)\in \{1,\ldots,n\}\\
i_0\circ p&\text{if }\phi(f)(k)=n+1\\
i_1\circ p&\text{if }\phi(f)(k)=0.
\end{array}
\right.
\]
\vspace{0.1in}

We also have the simplicial presheaf
\[
{\rm Sing}_\bullet^I(F)\in T[\mathcal{A}^{-1}]
\] 
defined in \cite[p.\ 88]{MV}. \index[notation]{SingI @ ${\rm Sing}_\bullet^I$}
Recall that
${\rm Sing}_\bullet^I(F)$ is the diagonal simplicial presheaf of the bisimplicial presheaf  
\[
(\Delta_I^\bullet,F_\bullet):(m,n)\mapsto \hom_{\Psh(T[\mathcal{A}^{-1}])}(\Delta_I^m,F_n). 
\]

\begin{df}
\label{A.1.7}
For any simplicial presheaf $F$ on $T$, we set\index[notation]{SingIA @ ${\rm Sing}_\bullet^{(I,\mathcal{A})}$}
\[
{\rm Sing}_\bullet^{(I,\mathcal{A})}(F):=v^*{\rm Sing}_\bullet^I(v_\sharp F).
\]
\end{df}

For simplicity we shall often omit the superscript $(I,\mathcal{A})$ in ${\rm Sing}_\bullet^{(I,\mathcal{A})}(F)$.

\begin{df}\label{A.1.10}\index{Singular functor}
Let $f,g:F\rightarrow G$ be two morphisms of simplicial presheaves on $T$. 
An {\it elementary $(I,\mathcal{A})$-homotopy} is a morphism 
\[
H:v_\sharp F\times I\rightarrow v_\sharp G
\]
of simplicial presheaves on $T[\mathcal{A}^{-1}]$ such that 
\[
H\circ i_0=v_\sharp f
\text{ and }
H\circ i_1=v_\sharp g. 
\]
Two morphisms of simplicial presheaves on $T$ are {\it $(I,\mathcal{A})$-homotopic}\index{homotopy} if they can be connected by a finite number of $(I,\mathcal{A})$-homotopies.
\vspace{0.1in}

A morphism $f:F\rightarrow G$ of simplicial presheaves on $T$ is a {\it strict $(I,\mathcal{A})$-homotopy equivalence}\index{homotopy equivalence} if there is a morphism $g:G\rightarrow F$ 
of simplicial presheaves on $T$ such that $f\circ g$ and $g\circ f$ are $(I,\mathcal{A})$-homotopic to the identity morphism.
\end{df}

In the following results, we argue following \cite[Proposition 3.4 - Corollary 3.8, pp.\ 89-90]{MV} and \cite[Proposition 4.23 - Proposition 4.26]{AdditiveHomotopy}.

\begin{prop}
\label{A.1.11}
The canonical morphism 
$$
i_0:{\rm pt}\rightarrow I
$$ 
is a strict $(I,\mathcal{A})$-homotopy equivalence.
\end{prop}
\begin{proof}
The diagram of simplicial presheaves on $\cT[\cA^{-1}]$
\[
\begin{tikzcd}
I\arrow[r,shift left=0.5ex,"i_0\times {\rm id}"]\arrow[r,shift right=0.5ex,"i_1\times {\rm id}"']&I\times I\arrow[r,"\mu"]&I
\end{tikzcd}
\]
displays an elementary $(I,\cA)$-homotopy from ${\rm id}=\mu\circ (i_0\times {\rm id})$ to $i_0p=\mu\circ(i_1\times {\rm id})$. 
Since $pi_0$ is the identity morphism on $\pt$, we are done.
\end{proof}

\begin{prop}
\label{A.1.12}
Suppose $f,g:F\rightarrow G$ are $(I,\mathcal{A})$-homotopic morphisms of simplicial presheaves on $T$. 
Then the induced morphisms 
\[
{\rm Sing}_\bullet(f),{\rm Sing}_\bullet(g):{\rm Sing}_\bullet(F)\rightarrow {\rm Sing}_\bullet(G)
\] 
are simplicially homotopic.
\end{prop}
\begin{proof}
We may assume there is an elementary $(I,\mathcal{A})$-homotopy $H:v_\sharp F\times I\rightarrow v_\sharp G$ from $f$ to $g$. 
Then $H$ is an $I$-homotopy from $v_\sharp f$ to $v_\sharp g$, so ${\rm Sing}_\bullet(v_\sharp f)$ and ${\rm Sing}_\bullet(v_\sharp g)$ are simplicially homotopic by \cite[Proposition 3.4 in p.\ 89]{MV}. 
It remains to show that $v^*$ preserves simplicial homotopies.
\vspace{0.1in}

Since $v^*$ is a right adjoint of $v_\sharp $, it follows that $v^*$ commutes with limits. 
Thus 
\[
v^*(F\times \Delta[1])\cong v^*F\times \Delta[1]
\] 
for any simplicial presheaf $F$ on $T[\mathcal{A}^{-1}]$, 
so $v^*$ preserves simplicial homotopies.
\end{proof}

\begin{cor}
\label{A.1.13}
For any simplicial presheaf $F$ on $T$, the morphism
\[
{\rm Sing}_\bullet(X)\rightarrow {\rm Sing}_\bullet(X\times I)
\]
induced by $i_0:{\rm pt}\rightarrow I$ is a pointwise weak equivalence.
\end{cor}
\begin{proof}
Follows from Propositions \ref{A.1.11} and \ref{A.1.12}.
\end{proof}

\begin{prop}
\label{A.1.14}
For any simplicial presheaf $F$ on $T$, the canonical morphism
\[
F\rightarrow {\rm Sing}_\bullet(F)
\]
is an $(I,\mathcal{A})$-weak equivalence.
\end{prop}
\begin{proof}
We have the factorization
\[
F\rightarrow v^*v_\sharp F\rightarrow v^*{\rm Sing}_\bullet^I(v_\sharp F).
\]
The first morphism is an $\mathcal{A}$-weak equivalence due to Proposition \ref{A.2.5}. 
Thus, 
by Proposition \ref{A.2.4}, 
it suffices to show the canonical morphism
\[
v_\sharp F\rightarrow {\rm Sing}_\bullet^I(v_\sharp F)
\]
is an $I$-weak equivalence. 
This follows from \cite[Corollary 3.8, p.\ 89]{MV}.
\end{proof}

Recall that $\Lambda$ is a commutative unital ring.

\begin{df}
Let $\sPsh(T,\Lambda)$\index[notation]{sPsh @ $\sPsh$} denote the category of simplicial sheaves of $\Lambda$-modules on $T$,
and let $\Co^{\leq 0}(\Psh(T,\Lambda))$ denote the category of complexes $\cF$ of sheaves of $\Lambda$-modules on $T$ such that $H^i(\cF)=0$ for any $i>0$.
\end{df}

By the Dold-Kan correspondence\index{Dold-Kan correspondence} \cite[Theorem 8.4.1]{weibel_1994}, there are equivalences of categories
\[
N:\sPsh(T,\Lambda)\rightleftarrows \Co^{\leq 0}(\Psh(T,\Lambda)):\Gamma
\]
and
\[
N:\sPsh(T[\cA^{-1}],\Lambda)\rightleftarrows \Co^{\leq 0}(\Psh(T[\cA^{-1}],\Lambda)):\Gamma
\]

\begin{df}
An object $F$ of $\Co^{\leq 0}(\Psh(T,\Lambda))$ is \emph{$\cA$-invariant} if $\Gamma(F)$ is $\cA$-invariant.
A morphism $f:F\rightarrow G$ in $\Co^{\leq 0}(\Psh(T,\Lambda))$ (resp.\ $\Co^{\leq 0}(\Psh(T[\cA^{-1}],\Lambda))$) is a $(I,\cA)$-weak equivalence (resp.\ $\cA$-weak equivalence) 
if $\Gamma(f):\Gamma(F)\rightarrow \Gamma(G)$ is a $(I,\cA)$-weak equivalence (resp.\ $\cA$-weak equivalence).
\end{df}

Let $\Co^{-}(\Psh(T,\Lambda))$ denote the category of bounded above complexes of sheaves of $\Lambda$-modules.
Then an object $F\in \Co^{-}(\Psh(T,\Lambda))$ is called $\cA$-\emph{invariant} if there exists an integer $n\in \Z$ such that $F[n]\in\Co^{\leq 0}(\Psh(T,\Lambda))$ is $\cA$-\emph{invariant}.

The $(I,\cA)$-weak equivalences in $\Co^{-}(\Psh(T,\Lambda))$ and the $\cA$-weak equivalences in $\Co^{-}(\Psh(T[\cA^{-1}],\Lambda))$ are defined similarly.
\vspace{0.1in}

Suppose $G$ is a presheaf of $\Lambda$-modules on $T$ and consider the simplicial presheaf $\Lambda$-modules $\Sing_{\bullet}^{(I,\cA)}(G)$ on $T$. 
Taking the alternating sum of the face maps yields the \emph{Suslin complex} \index{Suslin complex} of presheaves 
$$
C_*^{(I,\cA)}G.
$$
\vspace{0.1in}

We let $C_*G$ be shorthand for $C_*^{(I,\cA)}G$ when no confusion seems likely to arise, and define
\[
C_*^{DK}(G):=N({\rm Sing}_\bullet^{(I,\cA)}(G)).
\]
Here $C_*^{DK}(G)$ is canonically a normalized subcomplex of $C_*(G)$. 
The morphism 
\[
C_*^{DK}(G)\rightarrow C_*(G)
\] 
is a quasi-isomorphism, but not an isomorphism in general.
\vspace{0.1in}
 
If $F$ is a bounded above complex of presheaves on $T$, 
we can consider $C_*F$ and $C_*^{DK}F$ as double complexes and form their total complexes ${\rm Tot}(C_*F)$ and ${\rm Tot}(C_*^{DK}F)$.
\vspace{0.1in}

Suppose $H\in \sPsh(T,\Lambda)$.
The Eilenberg-Zilber theorem says that there is a quasi-isomorphism
\begin{equation}
\label{A.1.15.1}
N({\rm Sing}_\bullet^{(I,\cA)}(H))
\cong
{\rm Tot}(C_*^{DK}(N(H))).
\end{equation}
\vspace{0.1in}

As in Definition \ref{A.1.3} there exist adjoint functors
\[
\begin{tikzcd}
\Co^{\leq 0}(\Psh(T,\Lambda))\arrow[rr,shift left=1.5ex,"v_\sharp "]\arrow[rr,"v^*" description,leftarrow]\arrow[rr,shift right=1.5ex,"v_*"']&&\Co^{\leq 0}(\Psh(T[\mathcal{A}^{-1}],\Lambda)),
\end{tikzcd}
\]
where $v_\sharp$ is left adjoint to $v^*$, and $v^*$ is left adjoint to $v_*$.
For $F\in \Co^{\leq 0}(\Psh(T[\mathcal{A}^{-1}],\Lambda))$ and $X\in T$, we have
\[
\Gamma(v^*F)(X)\cong \Gamma(v^*F(X))=\Gamma(F(v(X))=\Gamma(F)(v(X))\cong v^*(\Gamma(F))(X).
\]
Thus $\Gamma$ commutes with $v^*$.
Similarly, we can show that $N$ commutes with $v^*$.
By adjunction we see that both $\Gamma$ and $N$ commute with $v_\sharp$.

\begin{prop}
\label{A.1.15}
Let $F$ be a bounded above complex of presheaves of $\Lambda$-modules on $T$. 
Then the following properties hold.
\begin{enumerate}
\item[{\rm (1)}] 
If $F$ is $\cA$-invariant, then the unit of the adjunction
\[
F\rightarrow v^*v_\sharp F
\]
is a quasi-isomorphism.
\item[{\rm (2)}] 
The unit of the adjunction
\[
F\rightarrow v^*v_\sharp F
\]
is an $\mathcal{A}$-weak equivalence.
\item[{\rm (3)}] 
The morphism
\[
{\rm Tot}(C_*F)\rightarrow {\rm Tot}(C_*(F\times I))
\]
induced by $i_0:{\rm pt}\rightarrow I$ is a quasi-isomorphism.
\item[{\rm (4)}] 
The canonical morphism
\[
F\rightarrow {\rm Tot}(C_*F)
\]
is an $(I,\mathcal{A})$-weak equivalence.
\end{enumerate}
\end{prop}
\begin{proof}
We observe that all the results in this section hold verbatim for simplicial presheaves of $\Lambda$-modules.
It suffices to prove these assertions for the shifted complex $F[n]$, $n\in \Z$. 
Thus we may assume $F\in \Co^{\leq 0}(\Psh(T,\Lambda))$.
As noted above both $\Gamma$ and $N$ commute with both $v_\sharp$ and $v^*$.
Moreover, 
quasi-isomorphisms (resp.\ $\cA$-weak equivalences, resp.\ $(I,\cA)$-weak equivalences) in $\Co^{\leq 0}(\Psh(T,\Lambda))$ corresponds to pointwise weak equivalences 
(resp.\ $\cA$-weak equivalences, resp.\ $(I,\cA)$-weak equivalences) in $\sPsh(T,\Lambda)$ under $\Gamma$ and $N$.
\vspace{0.1in}

Suppose $F$ is $\cA$-invariant.
Then $\Gamma(F)$ is $\cA$-invariant, 
and the $\Lambda$-module version of Proposition \ref{A.1.5} shows the morphism
\[
\Gamma(F)\rightarrow v^*v_\sharp \Gamma(F)\xrightarrow{\cong} \Gamma(v^*v_\sharp F)
\]
is a pointwise weak equivalence.
Thus $F\rightarrow v^*v_\sharp F$ is a quasi-isomorphism, which shows (1).
A similar argument shows (2).
Due to \eqref{A.1.15.1} we have isomorphisms 
\begin{gather*}
{\rm Tot}(C_*F)\simeq {\rm Tot}(C_*^{DK}F)\simeq N({\rm Sing}_\bullet^{(I,\cA)}(\Gamma(F))),
\\
{\rm Tot}(C_*(F\times I))\simeq {\rm Tot}(C_*^{DK}(F\times I))\simeq N({\rm Sing}_\bullet^{(I,\cA)}(\Gamma(F\times I))).
\end{gather*}
Owing to the $\Lambda$-module version of Corollary \ref{A.1.13} the composite morphism
\[
{\rm Sing}_\bullet^{(I,\cA)}(\Gamma(F))\rightarrow {\rm Sing}_\bullet^{(I,\cA)}(\Gamma(F)\times I)\xrightarrow{\cong} {\rm Sing}_\bullet^{(I,\cA)}(\Gamma(F\times I))
\]
is an isomorphism.
Combining the above, we deduce (3).
A similar argument shows (4).
\end{proof}
\newpage

\section{Properties of logarithmic motives}
In this section, we formulate and prove the fundamental properties of our categories of logarithmic motives. 
Most of these properties hold for categories associated with a base scheme $S$, 
not necessarily a field, 
subject to a list of axioms listed below.
This setup affords the case of $\ldmeff$ and its variants, see Proposition \ref{A.5.71}. 
\vspace{0.1in}

It can be quite interesting to observe that certain implications hold among the axioms.
For example, $(\P^n, \P^{n-1})$-invariance can be deduced by assuming Nisnevich descent (Nis-desc), 
$\boxx$-invariance, and dividing-descent, see Proposition \ref{A.6.1}.
On the other hand, 
we can include $(\P^n, \P^{n-1})$-invariance as part of the axioms.
In that case dividing-descent can be deduced from (Nis-desc) and $(\P^n, \P^{n-1})$-invariance for every $n\geq 1$. 
We refer to Section \ref{ssec:invariance-asuming-PnPn-1} for further details.
\vspace{0.1in}

After proving some general properties about $\P^n$-bundles and $\boxx^n$-bundles, 
we begin by showing an invariance property for certain admissible blow-ups. 
This is used for showing that log motives satisfy the $(\P^n, \P^{n-1})$-invariance property. 
We proceed by defining Thom motives, 
a.k.a.\ motivic Thom spaces, 
prove a log version of homotopy purity, 
and construct Gysin maps for closed embeddings. 
\vspace{0.1in}

Assuming that $k$ admits resolution of singularities allow us to prove a more potent invariance property for proper birational morphisms, 
see Theorem \ref{A.3.12}. 
Finally we discuss strictly $(\P^\bullet,\P^{\bullet-1})$-invariant and strictly dividing invariant complexes of sheaves on $SmlSm/k$. 
This provides a convenient tool for constructing objects of $\ldmeff$ without explicitly verifying the dividing-invariance property on the level of sheaf cohomology.
\vspace{0.1in}

Throughout this section, we fix a base scheme $S$ and let $\mathscr{S}/S$ denote either $lSm/S$ or $SmlSm/S$.
Moreover, we fix a symmetric monoidal triangulated category $\cT$,
and we fix a monoidal functor
\[
M\colon \mathbf{Square}_{\mathscr{S}/S}\rightarrow \cT
\]
satisfying (Sq) and (MSq).
For the category $\mathbf{Square}_{\mathscr{S}/S}$, $n$-squares, and the conditions (Sq) and (MSq), we refer to Appendix \ref{appendix:cat_tool}.
We choose to use $n$-squares because this is a convenient way to construct certain cones canonically within triangulated categories without referring to model categories or $\infty$-categories.
For example, $2$-squares are used in \cite{Deg08} to develop the theory of Gysin triangles in $\A^1$-motivic theory.
In most of our argument we only need $n$-squares for $n\leq 2$, but we need $3$-squares in Subsection \ref{ssec:ThomMotives}.
\vspace{0.1in}

We shall often assume some of the following properties.

\begin{enumerate}

\item[($Zar$-sep)] 
\emph{Zariski separation property.} If $\{X_1,\ldots,X_r\}$ is a finite Zariski cover of $X$ and $I=\{i_1,\ldots,i_s\}\subset \{1,\ldots,r\}$, we set 
\[
X_I:=X_{i_1}\times_S \cdots \times_S X_{i_s}.
\] 
Then for any commutative square of the form
\begin{equation}
\label{A.3.0.2}
\begin{tikzcd}
Y'\arrow[d]\arrow[r]&Y\arrow[d]
\\
X'\arrow[r]&X,
\end{tikzcd}
\end{equation}
the morphism
\[
M(Y'\rightarrow X')\rightarrow M(Y\rightarrow X)
\] 
is an isomorphism if for every nonempty $I\subset \{1,\ldots,r\}$ there is an isomorphism 
\[
M(Y_I'\rightarrow X_I')\rightarrow M(Y_I\rightarrow X_I)
\] 
where 
\[
Y_I:=Y\times_X X_I, X_I':=X'\times_X X_I, \text{ and }Y_I':=Y'\times_Y Y_I.
\]   

\item[($\boxx$-inv)] 
\emph{$\boxx$-invariance property.} For any $X\in \mathscr{S}/S$ the projection $X\times  \boxx\rightarrow X$ induces an isomorphism 
\[
M(X\times \boxx)\xrightarrow{\cong} M(X).
\]

\item[($sNis$-des)] 
\emph{$sNis$-descent property.}
Every strict Nisnevich distinguished square of the form \eqref{A.3.0.2} induces an isomorphism 
\[
M(Y'\rightarrow X')\xrightarrow{\cong} M(Y\rightarrow X).
\] 

\item[($div$-des)] 
\emph{$div$-descent property.}
Every log modification $f:Y\rightarrow X$ in $\mathscr{S}/S$ induces an isomorphism
\[
M(Y)\xrightarrow{\cong} M(X).
\]

\item[($(\P^\bullet,\P^{\bullet-1})$-inv)]
\emph{$(\P^\bullet,\P^{\bullet-1})$-invariance property.}
For any $X\in \mathscr{S}/S$ and positive integer $n>0$ the projection $X\times  (\P^n,\P^{n-1})\rightarrow X$ induces an isomorphism 
\[
M(X\times (\P^n,\P^{n-1}))\xrightarrow{\cong} M(X).
\]
Here we consider $\P^{n-1}$ as the hyperplane of $\P^n$ defined by the equation $x_n=0$ with respect to the coordinates $[x_0:\cdots:x_n]$.

\item[(adm-blow)]
\emph{Admissible blow-up property.}
For every proper birational morphism $f\colon Y\rightarrow X$ in $\mathscr{S}/S$ such that the naturally induced morphism $Y-\partial Y\rightarrow X-\partial X$ is an isomorphism, there is a naturally induced isomorphism
\[
M(Y)\xrightarrow{\cong} M(X).
\]
\end{enumerate}

The following figure illustrates how these properties are related.
Here, ROS stands for resolution of singularities.
\[
\begin{tikzpicture}
\node[draw,align=center] at (-5,0)
{($Zar$-sep)
\\
($\boxx$-inv)
\\
($sNis$-des)
\\ 
(adm-blow)
\\
$\mathscr{S}=lSm/k$};
\node[draw,align=center] at (0,0)
{($Zar$-sep)
\\
($\boxx$-inv)
\\
($sNis$-des)
\\ 
($div$-des)
\\
$\mathscr{S}=lSm/S$};
\node[draw,align=center] at (5,0)
{($Zar$-sep)
\\
($sNis$-des)
\\ 
($(\P^\bullet,\P^{\bullet-1})$-inv)
\\ $\mathscr{S}=SmlSm/S$};
\draw (-3.5,0.6)--(-1.5,0.6);
\draw (-3.5,0.5)--(-1.5,0.5);
\node at (-1.5,0.55) {$>$};
\node at (-2.5,0.8) {\footnotesize Trivial if $S=k$};
\draw (-3.5,-0.6)--(-1.5,-0.6);
\draw (-3.5,-0.5)--(-1.5,-0.5);
\node at (-3.5,-0.55) {$<$};
\node at (-2.5,-0.8) {\footnotesize Theorem \ref{A.3.12}};
\node at (-2.5,-1.1) {\footnotesize {(assuming ROS)}};
\draw (1.35,0.6)--(3.35,0.6);
\draw (1.35,0.5)--(3.35,0.5);
\node at (3.35,0.55) {$>$};
\node at (2.35,0.8) {\footnotesize Proposition \ref{A.6.1}};
\draw (1.35,-0.6)--(3.35,-0.6);
\draw (1.35,-0.5)--(3.35,-0.5);
\node at (1.35,-0.55) {$<$};
\node at (2.35,-0.8) {\footnotesize Construction \ref{Pn.4}};
\end{tikzpicture}
\]

\begin{df}
\label{A.5.60}
The Tate object in $\cT$ is defined by
\[M(S)(1)
:=
M(S\stackrel{i_0}\rightarrow \P_S^1)[-2],
\]
where $i_0:S \rightarrow \P_S^1$ is the $0$-section.  

For $n\geq 0$, we set 
\[
M(S)(n)
:=
\underbrace{M(S)(1)\otimes \cdots \otimes M(S)(1)}_{n\text{ times}},
\]
where $M(S)(0):=M(S)$ by convention. 
We define the $n$th Tate twist of $K\in\cT$ as 
\[
K(n)
:=
K\otimes M(S)(n).
\]
\end{df}

\begin{prop}
\label{A.5.71}
The conditions {\rm (Sq)} and {\rm (MSq)} of Appendix \ref{appendix:cat_tool} and the properties {\rm ($Zar$-sep)}, {\rm ($\boxx$-inv)}, {\rm ($sNis$-des)}, 
and {\rm ($div$-des)} hold in the triangulated categories
\begin{gather*}
\ldmeff,\;\ldaeff,\;\ldmeffet,\;\ldaeffet,
\\
\ldmefflet,\;\ldaefflet.
\end{gather*}
\end{prop}
\begin{proof}
We work in $\ldmeff$ since the proofs for the other categories are similar.
For every $n$-square $C$, we set
\[
M(C):={a_{dNis}^*}{\rm Tot}(\Zltr(C)).
\]
For every $n$-square $C$ with $n\geq 1$, there exists a naturally induced sequence
\[
{\rm Tot}(\Zltr(s_{i,0}^*C))\rightarrow {\rm Tot}(\Zltr(s_{i,1}^*C))\rightarrow {\rm Tot}(\Zltr(C))\rightarrow {\rm Tot}(\Zltr(s_{i,0}^*C))[1].
\]
Moreover, for every $n$-square $C$ and $m$-square $D$, there exists a canonical isomorphism (not just a quasi-isomorphism) of complexes
\[
{\rm Tot}(\Zltr(C))\otimes {\rm Tot}(\Zltr(D))\rightarrow {\rm Tot}(\Zltr(C\otimes D)).
\]
The latter is reminiscent of an Eilenberg-Zilber map, but it is, in fact, much simpler since no shuffles are involved. 
We can readily deduce that the conditions (Sq) and (MSq) hold.
\vspace{0.1in}

Property ($\boxx$-inv) holds by construction. 
Proposition \ref{A.5.42} implies ($sNis$-des) and ($div$-des).
To show ($Zar$-sep), 
let $\mathscr{X}$ be the \v{C}ech nerve of the covering $\{X_1,\ldots,X_r\}$, 
and set 
\[
\mathscr{X}':=X'\times_X \mathscr{X}, \mathscr{Y}:=Y\times_X \mathscr{X}, \mathscr{Y}':=Y'\times_X \mathscr{X}. 
\]
Then, by assumption, we have the homotopy cartesian diagram
\[
\begin{tikzcd}
M(\mathscr{Y}')\arrow[d]\arrow[r]&M(\mathscr{Y})\arrow[d]\\
M(\mathscr{X}')\arrow[r]&M(\mathscr{X}).
\end{tikzcd}
\]
This allows us to conclude that ($Zar$-sep) holds owing to the isomorphisms 
\[
M(X)\cong M(\mathscr{X}), M(Y)\cong M(\mathscr{Y}), M(X')\cong M(\mathscr{X}'), \text{ and } M(Y')\cong M(\mathscr{Y}').
\]
\end{proof}

\subsection{Some toric geometry and examples of bundles}
Let $d\geq 0$ and suppose $\sigma$ is a $d$-codimensional cone of a fan $\Sigma$.
We can naturally associate a $d$-dimensional closed subscheme $V(\sigma)$ of $\ul{\A_\Sigma}$, which is the closure of the orbit space ${\rm orb}(\sigma)$ in $\ul{\A_\Sigma}$, see \cite[Proposition 1.6, Corollary 1.7]{TOda}.
Let us explain the structure in the case that $\Sigma=\Spec{P}$ for some fs monoid $P$.
Note that the only maximal cone of $\Sigma$ is $P^\vee$.
\vspace{0.1in}

The cone $\sigma$ is associated with a face $F$ of $P^\vee$,
which corresponds by duality to the face 
\[
G:=F^\bot\cap P=\{x\in P:\langle x,y\rangle =0\text{ for all }y\in F\}.
\]
Then ${\rm orb}(\sigma)$ is identified with
\[
\Spec{\Z[F^\bot \cap P^\gp]}\cong \Spec{\Z[G^\gp]}, 
\]
see \cite[p.\ 52]{Fulton:1436535}.
Owing to \cite[Proposition I.3.3.4]{Ogu} the closed immersion
$$
V(\sigma)\rightarrow \ul{\A_\Sigma}
$$
is induced by the natural epimorphism
$$
\Z[P]\rightarrow \Z[P]/\Z[P-G],
$$
where we regard $\Z[P-G]$ as an ideal of $\Z[P]$ as in \cite[Section I.3.3]{Ogu}.
\vspace{0.1in}

If more specifically $P=\N^p\oplus \N^q$ and $F$ is the face $\N^p\oplus 0$ of $P^\vee$ for some integers $p,q\geq 0$, then $G=0\oplus \N^q$. Thus if we express $\Z[P]$ as $\Z[x_1,\ldots,x_{p+q}]$, then
\[
\Z[P]/\Z[P-G]=\Z[x_1,\ldots,x_{p+q}]/(x_1,\ldots,x_p).
\]

\begin{df}
Let $\Sigma$ be a fan.
Suppose that $\Sigma'$ is a subfan of $\Sigma$, i.e., every cone of $\Sigma'$ is a cone of $\Sigma$.
We denote by
\[
\A_{(\Sigma,\Sigma')}
\]
the fs log scheme whose underlying scheme is $\ul{\A_\Sigma}$ and whose log structure is the compactifying log structure associated with the open immersion
\[
(\ul{\A_{\Sigma}}-\bigcup_{\sigma\in \Sigma'-\{0\}}V(\sigma))\rightarrow \ul{\A_\Sigma}.
\]
For the convenience of the reader we will often draw figures illustrating the fans $\Sigma$ and $\Sigma'$.
In our drawings of $2$-dimensional fans, 
thick (resp.\ thin) lines record the $1$-dimensional cones in $\Sigma'$ (resp.\ not in $\Sigma'$), 
and dark (resp.\ light) grey regions records the $2$-dimensional cones in $\Sigma'$ (resp.\ not in $\Sigma'$).
For example, 
the figure
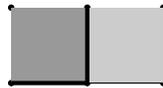
\begin{equation}\label{eq.examplefan}
\begin{tikzpicture}
\foreach \a in {-1,0,1}
\foreach \b in {-1,0,1}
\filldraw (\a,\b) circle (1pt);
\fill[black!20] (0,0)--(1,0)--(1,1)--(0,1);
\fill[black!40] (0,0)--(0,1)--(-1,1)--(-1,0);
\draw (0,0)--(1,0);
\draw[ultra thick] (0,0)--(0,1);
\draw[ultra thick] (0,0)--(-1,0);
\end{tikzpicture}
\end{equation}
depicts the pair of fans
\[
\Sigma=\{{\rm Cone}((1,0),(0,1)),{\rm Cone}((0,1),(-1,0))\},
\;
\Sigma'=\{{\rm Cone}((0,1),(-1,0))\}.
\]
Here, the notation $\Sigma=\{\sigma,\ldots \}$ means that $\Sigma$ consists of all faces of the elements of the set $\{\sigma,\ldots\}$. 
The algebraic variety encoded in the figure \eqref{eq.examplefan} can be described as follows.
We set $\sigma_1 ={\rm Cone}((1,0),(0,1))$, 
$\sigma_2=\Sigma' = {\rm Cone}((0,1),(-1,0))$, 
$M=\Z^2$, 
and $N= {\rm Hom}_\Z(M,\Z)$. 
The toric variety $\ul{\A_\Sigma}$ is obtained by gluing the two affine schemes 
$$
\mathbb{A}_{\sigma_1}=\Spec{\Z[\sigma_1^\vee \cap N]} = \Spec{\Z[x,y]}
$$ 
and 
$$
\mathbb{A}_{\sigma_2} = \Spec{\Z[\sigma_2^\vee \cap N]} = \Spec{\Z[x^{-1},y]}
$$ 
along their common open subset 
$$
\Spec{\Z[\tau^\vee\cap N]} = \Spec{\Z[x,x^{-1}, y]},
$$ 
where $\tau$ is the 1-dimensional cone $(0,1)$. 
The underlying scheme is 
$$
X = \P^1\times \A^1.
$$ 
Here $x$ is the coordinate on $\P^1$, and $y$ is the coordinate on $\A^1$. 
Consider the divisor $\partial X = \infty\times \A^1+\P^1\times 0$ on $X$. 
The corresponding compactifying log structure $(X,\partial X)$ can be identified with the fs log scheme $\A_{(\Sigma,\Sigma')}$.
\vspace{0.1in}

In our drawings of $3$-dimensional fans, 
thick points records $1$-dimensional cones in $\Sigma'$, thick (resp.\ thin) lines record $2$-dimensional cones in $\Sigma'$ (resp.\ not in $\Sigma'$), 
and dark (resp.\ light) grey regions record $3$-dimensional cones in $\Sigma'$ (resp.\ not in $\Sigma'$).
\end{df}

Before discussing the motives of general fs log schemes, we will first investigate some elementary examples of $\P^1$- and $\boxx$-bundles.
\vspace{0.1in}

Let us define the geometric notion of vector bundles over fs log schemes.

\begin{df}
\label{df:logvectorbundle}
Let $X$ be an fs log scheme. 
A {\it vector bundle}\index{bundle!vector} $\xi:\cE\rightarrow X$ is a morphism of fs log schemes such that $\underline{\xi}:\underline{\cE}\rightarrow \underline{X}$ is a usual vector bundle and the diagram
\[
\begin{tikzcd}
\cE\arrow[r,"\xi"]\arrow[d]\arrow[r]&X\arrow[d]\\
\underline{\cE}\arrow[r,"\underline{\xi}"]&\underline{X}
\end{tikzcd}
\]
is cartesian, where the vertical arrows are the canonical morphisms removing the log structures. In other words, the log structure on $\cE$ is the pullback of the log structure on $X$ along $\xi$.
\end{df}

\begin{exm}
The projection $X\times_S \A_S^n\rightarrow X$ is a vector bundle.
This is called the {\it trivial} vector bundle of rank $n$.
\end{exm}

\begin{df}
We say that ``an fs log scheme $X$ satisfies Zariski locally a property $P$'' if for any point $x$ of $X$ and Zariski neighborhood $V$ of $x$, 
there is a Zariski neighborhood $U$ of $x$ inside $V$ such that $U$ satisfies $P$.
\end{df}

\begin{df}
\label{df::projectivebundle}
Let $X$ be an fs log scheme.
A $\P^n$-\emph{bundle}\index{bundle!Pn @ $\P^n$} $\xi:\cP\rightarrow X$ is a morphism of fs log schemes that Zariski locally on $X$ is isomorphic to the projection $X\times \P^n\rightarrow X$.
\vspace{0.1in}

Similarly, a $(\P^n,\P^{n-1})$-\emph{bundle}\index{bundle!PnPn @ $(\P^n,\P^{n-1})$}\index{bundle!boxx @ $\boxx$} $\xi:\cP\rightarrow X$ is a morphism of fs log schemes that Zariski locally on $X$ is isomorphic to the projection 
$X\times (\P^n,\P^{n-1})\rightarrow X$.
Note that we have $(\P^1,\P^{0})=\boxx$.
\end{df}

Recall that we view $\P^{n-1}$ as the hyperplane defined by $x_n=0$ where $[x_0:\cdots:x_n]$ is the coordinate on $\P^n$.

\begin{lem}
\label{A.6.4}
The blow-up $X$ of $\P^n$ along the point $[0:\cdots:0:1]$ is a $\P^1$-bundle over $\P^{n-1}$.
\end{lem}
\begin{proof} 
While this should be a well-known fact, since we need to fix some notation for later use, we give proof using the language of toric geometry. 
Let $e_1,\ldots,e_n$ be the standard coordinates of $\Z^n$. 
The blow-up $X$ in question is the toric variety associated with the fan with maximal cones 
\[
\sigma_i:={\rm Cone}(e_1,\ldots,e_{i-1},e_{i+1},\ldots,e_n,e_1+\cdots+e_n)\quad  (1\leq i\leq n),
\]
\[
\sigma_i':={\rm Cone}(e_1,\ldots,e_{i-1},e_{i+1},\ldots,e_n,-e_1-\cdots-e_n)\quad (1\leq i\leq n).
\]
For a fixed $i$, the toric variety $U_i$ formed by $\sigma_i$ and $\sigma_i'$ is isomorphic to $\A^{n-1}\times \P^1$.
\vspace{0.1in}

Recall that $\P^{n-1}$ is the toric variety associated with the fan with maximal cones 
\[
\tau_i:={\rm Cone}(e_1,\ldots,e_{i-1},e_{i+1},\ldots,e_{n-1},-e_1-\cdots-e_{n-1})\quad (1\leq i\leq n-1),
\]
\[
\tau_n:={\rm Cone}(e_1,\ldots,e_{n-1}).
\]
For the homomorphism of lattices 
\[
\varphi:\Z^n\rightarrow \Z^{n-1}
\]
mapping $(x_1,\ldots,x_n)$ to $(x_1-x_n,\ldots,x_{n-1}-x_n)$, 
we have that 
$$
\varphi(\sigma_i\cup \sigma_i')=\tau_i.
$$ 
Thus the morphism $X\rightarrow \P^{n-1}$ induced by $\varphi$ is a $\P^1$-bundle.
\end{proof}

\begin{lem}
\label{A.6.5}
The blow-up of $(\P^n,\P^{n-1})$ along the point $[0:\cdots:0:1]$ is a $\boxx$-bundle over $\P^{n-1}$.
\end{lem}
\begin{proof}
The cones $\sigma_i$ and $\sigma_i'$ appearing in the proof of Lemma \ref{A.6.4} correspond espectively to
\[
U_i:={\rm Spec}\left( \Z\left[\frac{x_0}{x_{i-1}},\ldots,\frac{x_{i-2}}{x_{i-1}},\frac{x_i}{x_{i-1}},\ldots,\frac{x_{n-1}}{x_{i-1}},\frac{x_{i-1}}{x_n}\right]\right)
\]
and   
\[U_i':={\rm Spec}\left( \Z\left[\frac{x_0}{x_{i-1}},\ldots,\frac{x_{i-2}}{x_{i-1}},\frac{x_i}{x_{i-1}},\ldots,\frac{x_{n-1}}{x_{i-1}},\frac{x_{n}}{x_{i-1}}\right]\right).
\] 
Their log structures induced by $(\P^n,\P^{n-1})$ are given by
\[
\N \left(\frac{x_{n}}{x_{i-1}}\right)\rightarrow \Z\left[\frac{x_0}{x_{i-1}},\ldots,\frac{x_{i-2}}{x_{i-1}},\frac{x_i}{x_{i-1}},\ldots,\frac{x_{n-1}}{x_{i-1}},\frac{x_{n}}{x_{i-1}}\right].
\]
In the proof of Lemma \ref{A.6.4}, we noted that 
$$
U_i\cup U_i'\cong \A^{n-1}\times \P^1.
$$
If we impose the above log structure, then this becomes $\A^{n-1}\times \boxx$, and the claim follows.
\end{proof}

\begin{lem}\label{A.3.28}
For $n\geq 2$ we let $X$ be the blow-up of $\P^n$ at the point $[0:\cdots:0:1]$.
Let $H$ be the divisor of $\P^n$ defined by $(x_0=0)$, 
let $E$ be the exceptional divisor on $X$, and let $H'$ be the strict transform of $H$ in $X$.
Then 
$$
(X,H'+E)
$$ 
is a $\boxx$-bundle over $(\P^{n-1},\P^{n-2})$.
\end{lem}
\begin{proof}
Recall the schemes $U_i$ and $U_i'$ appearing in the proof of Lemma \ref{A.6.5}. 
We equip $U_i$ with the log structures 

\[\N\left(\frac{x_0}{x_{i-1}}\right)\oplus \N\left(\frac{x_{i-1}}{x_{n}}\right)\rightarrow \Z\left[\frac{x_0}{x_{i-1}},\ldots,\frac{x_{i-2}}{x_{i-1}},\frac{x_i}{x_{i-1}},\ldots,\frac{x_{n-1}}{x_{i-1}},\frac{x_{n}}{x_{i-1}}\right]\text{ if }i>1,\]

\[\N\left(\frac{x_0}{x_{n}}\right)\rightarrow \Z\left[\frac{x_0}{x_{i-1}},\ldots,\frac{x_{i-2}}{x_{i-1}},\frac{x_i}{x_{i-1}},\ldots,\frac{x_{n-1}}{x_{i-1}},\frac{x_{n}}{x_{i-1}}\right]\text{ if }i=1,\]
and $U_i'$ with the log structures 

\[\N\left(\frac{x_0}{x_{i-1}}\right)\rightarrow \Z\left[\frac{x_0}{x_{i-1}},\ldots,\frac{x_{i-2}}{x_{i-1}},\frac{x_i}{x_{i-1}},\ldots,\frac{x_{n-1}}{x_{i-1}},\frac{x_{i-1}}{x_{n}}\right]\text{ if }i>1,\]

\[0\rightarrow  \Z\left[\frac{x_0}{x_{i-1}},\ldots,\frac{x_{i-2}}{x_{i-1}},\frac{x_i}{x_{i-1}},\ldots,\frac{x_{n-1}}{x_{i-1}},\frac{x_{i-1}}{x_{n}}\right]\text{ if }i=1.\]
\vspace{0.1in}

Let $X_i$ and $X_i'$ denote the associated fs log schemes, and let $Y_i$ be the fs log scheme associated with
\[
\N\left(\frac{x_0}{x_{i-1}}\right)\rightarrow  \Z\left[\frac{x_0}{x_{i-1}},\ldots,\frac{x_{i-2}}{x_{i-1}},\frac{x_i}{x_{i-1}},\ldots,\frac{x_{n-1}}{x_{i-1}}\right].
\]
With these definitions, we have 
\[
X_i\cup X_i'\cong Y_i\times \boxx.
\] 
This finishes the proof because the $Y_i$'s glue together to form $(\P^{n-1},\P^{n-2})$.  
Here we consider $\P^{n-2}$ as the divisor on $\P^{n-1}$ defined by $(x_0=0)$.
\end{proof}

\begin{lem}
\label{A.3.33}
Let $X$ be the blow-up of $\P^n$ along the point $[0:\cdots:0:1]$, and let $E$ be the corresponding exceptional divisor on $X$.
Then $(X,E)$ is a $\boxx$-bundle over $\P^{n-1}$.
\end{lem}
\begin{proof}
Recall the schemes $U_i$ and $U_i'$ appearing in the proof of Lemma \ref{A.6.5}.
Their log structures induced by $(X,E)$ are given by

\[
\N \left(\frac{x_{i-1}}{x_{n}}\right)\rightarrow \Z\left[\frac{x_0}{x_{i-1}},\ldots,\frac{x_{i-2}}{x_{i-1}},\frac{x_i}{x_{i-1}},\ldots,\frac{x_{n-1}}{x_{i-1}},\frac{x_{i-1}}{x_{n}}\right].
\]
Our claim follows by noting there is an isomorphism 
$$
U_i\cup U_i'\cong \boxx\times \A^{n-1}.
$$
\end{proof}

\begin{prop}
\label{A.3.40}
Suppose $Y$ is a $\boxx$-bundle over $X\in\mathscr{S}/S$.
If $\cT$ satisfies {\rm ($\boxx$-inv)} and {\rm ($sNis$-des)}, then the projection $Y\rightarrow X$ induces an isomorphism
\[
M(Y)\xrightarrow{\cong} M(X).
\]
\end{prop}
\begin{proof}
Use induction on the number of open subsets in a trivialization of $Y$; then from ($sNis$-des), we may assume $Y\cong X\times \boxx$. 
To conclude, we apply ($\boxx$-inv).
\end{proof}

\subsection{Invariance under admissible blow-ups along smooth centers}\label{ssec:invariancesmoothblowups}
In this subsection, 
assuming the properties ($Zar$-sep), ($\boxx$-inv), and ($sNis$-des) we will extend the invariance under log modifications ($div$-des) to admissible blow-ups along smooth centers, 
whose definition is given below.

\begin{df}
\label{A.3.17}
Suppose that $Z$ is an effective Cartier divisor on $X\in Sm/S$, and let $i\colon Z\rightarrow X$ be the corresponding closed immersion.
Then $Z$ is a {\it strict normal crossing divisor}\index{divisor!strict normal crossing} on $X$ over $S$ if Zariski locally on $X$, there is a cartesian square 
\[
\begin{tikzcd}
Z\arrow[d,"i"']\arrow[r]&S\times \Spec {\Z[x_1,\ldots,x_n]/(x_1\cdots x_r)}\arrow[d]\\
X\arrow[r]&    S\times \Spec{ \Z[x_1,\ldots,x_n]}
\end{tikzcd}
\]
of schemes over $S$ for some $r\leq n$, 
where the horizontal morphisms are \'etale.
\vspace{0.1in}

We say the divisors $Z_1,\ldots,Z_r$ on $X\in Sm/S$ form a {\it strict normal crossing divisor} if 
\[
Z_1+\cdots+Z_r
\] 
is a strict normal crossing divisor on $X$ over $S$.
\vspace{0.1in}

Suppose that $Z$ is a strict normal crossing divisor on $X$ over $S$, and let $v:W\rightarrow X$ be a closed immersion of schemes smooth over $S$. 
We say that $W$ has {\it strict normal crossing} with $Z$ over $S$ if Zariski locally on $X$, there is a commutative diagram of cartesian squares of schemes over $S$ 
\[
\begin{tikzcd}
Z\arrow[d,"i"']\arrow[r]&S\times \Spec{ \Z[x_1,\ldots,x_n]/(x_1\cdots x_r)}\arrow[d]\\
X\arrow[d,leftarrow,"v"']\arrow[r]&S\times \Spec{ \Z[x_1,\ldots,x_n]}\arrow[d,leftarrow]\\
W\arrow[r]&S\times \Spec {\Z[x_1,\ldots,x_n]/(x_{i_1},\ldots,x_{i_s})}
\end{tikzcd}\]
for some $r\leq n$ and $1\leq i_1<\cdots<i_s\leq n$, 
where the right vertical morphisms are the evident closed immersions and the horizontal morphisms are \'etale. Note that, a priori, in this formulation $W$ can be contained in a component of $Z$. If $r< i_1< \cdots i_s\leq n$, we get, in particular, that the restriction of $Z$ to $W$ is again a Cartier divisor, and that it forms a strict normal crossing divisor on $W$.
\end{df}

\begin{df}\label{def:associative-notation-log}
For divisors $Z_1,\ldots,Z_r$ forming a strict normal crossing divisor on $X\in Sm/S$, for $1\leq s\leq r$, we use the convenient ``associative'' notation
\[
\begin{split}
((X,Z_{s+1}+\cdots +Z_r),Z_1'+\cdots+Z_s')&:=(X,Z_1+\cdots+Z_r)
\end{split}
\]
where $Z_i'$ is any fs log scheme such that $\underline{Z_i'}\cong Z_i$ (we refer to Definition \ref{A.0.2} for the notation $(X,Z_1+\cdots+Z_r)$).
\end{df}

To illustrate, suppose $X\in SmlSm/S$ has compactifying log structure given by $Z_{s+1}+\cdots +Z_r$, written $(X, Z_{s+1}+\cdots+Z_r)$. 
And let us consider a strict normal crossing divisor $Z_1+\cdots+Z_s$ on $X$ such that $Z_1+\cdots+Z_s + Z_{s+1}+\cdots +Z_r$ is again strict normal crossing. 
This determines another object of $SmlSm/S$, 
namely $(X, Z_{1}+\cdots+Z_r)$, 
equipped with a canonical map 
\[(X, Z_{1}+\cdots+Z_r) \to (X, Z_{s+1}+\cdots+Z_r) \]
obtained by adding the divisor $Z_{1}+\cdots+Z_s$ to the log structure. 
\vspace{0.1in}

For instance, if $Y:=(X,Z_{s+1},\cdots,Z_r)$, then we obtain
\[
(Y,Z_1+\cdots+Z_s)=(X,Z_1+\cdots +Z_r).
\]
If $Y=\P^1\times \boxx=(\P^1\times \P^1,\P^1\times \{\infty\})$ and $Z_1'=\{\infty\}\times \boxx$, then
\[
(\P^1\times \boxx,\{\infty\}\times \boxx)=\boxx^2.
\]

\begin{df}
An \emph{admissible blow-up along a smooth center over $S$} \index{blow-up!along a smooth center}(or along a smooth center for short) is a proper birational morphism
\[
f:X'\rightarrow X
\] 
in $SmlSm/S$ for which the following properties hold. 
\begin{enumerate}
\item[(1)] $\underline{f}$ is the blow-up along a center $Z$ in $\underline{X}$ such that $Z$ is smooth over $S$.
\item[(2)] $\partial X'$ is the union of $f^{-1}(Z)$ and the strict transform of $\partial X$ with respect to $f$.
\item[(3)] $Z$ is contained in $\partial X$ and has strict normal crossing with $\partial X$ over $S$ in the sense of Definition \ref{A.3.17}.
\end{enumerate}

Note that by Lemma \ref{lem::SmlSm}, we may assume that $\partial X$ is a strict normal crossing divisor on $X$, 
so that condition (3) makes sense. 
The open immersion $X-Z\rightarrow X$ lifts to $X'$.
Hence the log structure on $X'$ is precisely the compactifying log structure associated with the composite open immersion
\[
X-\partial X\rightarrow X-Z\rightarrow X'.
\]
If $Z$ has codimension $d$ in $X$, then we say that $f$ is an \emph{admissible blow-up along a codimension $d$ smooth center}.
Let $\cA\cB l_{Sm}/S$ denote the smallest class of morphisms in $\mathscr{S}/S$ containing all admissible blow-ups along a smooth center and closed under composition.
\end{df}

Let us give a fundamental example of admissible blow-ups along a smooth center.

\begin{exm}
\label{A.3.41}
Let $X=\A_\N^p\times \A^q$ and let $e_1,\ldots,e_{p+q}$ denote the standard coordinates on $\Z^n$.
We set
\[
\sigma:={\rm Cone}(e_1,\ldots,e_{p+q}),\;
\sigma':={\rm Cone}(e_1,\ldots,e_p),
\]
and form the associated fans
\[
\Sigma:=\{\sigma\},\;\Sigma':=\{\sigma'\}.
\]
Then $X$ can be written as $\A_{(\Sigma,\Sigma')}$.
\vspace{0.1in}

Let $X'$ be the admissible blow-up of $X$ along the origin.
Then
\[
\underline{X'}=\underline{\A_{\Gamma}},
\]
where $\Gamma$  denotes the star subdivision\index{fan!star subdivision} $\Sigma^*(\sigma)$ of $\Sigma$ relative to $\sigma$.
Let $\Gamma'$ be the fan consisting of all cones $\sigma\in \Gamma$ satisfying
\[
\sigma\subset {\rm Cone}(e_1,\ldots,e_p,e_1+\cdots+e_p).
\]
Note that
\[
\partial X=V(e_1)\cup \cdots \cup V(e_p).
\]
The strict transform of $V(e_i)\subset X$ for $1\leq i\leq p$ with respect to $f$ is again $V(e_i)\subset X'$.
The exceptional divisor on $X'$ is $V(e_1+\cdots+e_p)$.
Using these descriptions of the divisors, we deduce that
\[
X'=\A_{(\Gamma,\Gamma')}.
\]
For $p=q=1$, 
we are considering $\A^2$ with log structure given by the affine line $(x=0)=H \subset \A^2$, and $X'$ is $B_{(0,0)}(\A^2)$ with log structure $\tilde{H} + E$, 
where $\tilde{H}$ is the strict transform of $H$. 
Note that $X'\to X$  \emph{is not} a log modification. 
In fact, since the log structure on $\A^2$ is given by $\N\to \Spec{\Z[x,y]}$, $n\mapsto x^n$, there are no non-trivial log modification.  
\end{exm}

We will first prove a particular case of Theorem \ref{A.3.7} (invariance under admissible blow-ups), 
which is the crucial point of the proof in the general case.
Roughly speaking, 
it says that the cone of an admissible blow-up morphism is insensitive to the addition of a smooth component of the boundary divisor, 
provided that it stays normal crossing with the center. 
This is carried out,  
however, 
without constructing directly any map between the two cones. 
A posteriori, this invariance property could be re-interpreted using the Gysin triangle.

\begin{prop}
\label{A.3.39}
Assume that $\cT$ satisfies {\rm ($\boxx$-inv)} and {\rm ($sNis$-des)}.
Let $H_1$ and $H_2$ be the axes of $Y:=\A^2$, let $H_1'$ and $H_2'$ be the strict transforms in the blow-up $Y':=B_{\{(0,0)\}}\A^2$ at $(0,0)$, 
and let $E$ be the exceptional divisor on $Y'$.
For $X\in \mathscr{S}/S$, the naturally induced morphism
\[
M(X\times(Y',E+H_2'))\rightarrow M(X\times(Y,H_2))
\]
is an isomorphism if and only if the naturally induced morphism
\[
M(X\times(Y',E+H_1'+H_2'))\rightarrow M(X\times(Y,H_1+H_2))
\]
is an isomorphism.
\end{prop}
\begin{proof}
It is enough to prove the claim for $X=\Spec{\Z}$.
We consider the following cones
\begin{gather*}
\sigma_1:={\rm Cone}((1,0),(1,1)),
\;
\sigma_2:={\rm Cone}((1,1),(0,1)),
\;
\sigma_3:={\rm Cone}((0,1),(-1,0)),
\\
\sigma_4:={\rm Cone}((-1,0),(0,-1)),
\;
\sigma_5:={\rm Cone}((0,-1),(1,0)).
\end{gather*}
Form the following fans
\begin{gather*}
\Sigma_1:=\{\sigma_1\cup \sigma_2,\sigma_3,\sigma_4,\sigma_5\},
\;
\Sigma_2:=\{\sigma_1,\sigma_2,\sigma_3,\sigma_4,\sigma_5\},
\;
\Sigma_3:=\{\sigma_1,\sigma_2\cup \sigma_3,\sigma_4,\sigma_5\},
\\
\Sigma_4:=\{\sigma_1\cup \sigma_2\},
\;
\Sigma_5:=\{\sigma_1,\sigma_2\},
\;
\Sigma_6:=\{\sigma_3,\sigma_4,\sigma_5\},
\;
\Sigma_7:=\{\sigma_2\cap \sigma_3,\sigma_1\cap \sigma_5\},
\\
\Sigma_8:=\{\sigma_2,\sigma_3\},
\;
\Sigma_9:=\{\sigma_2\cup \sigma_3\},
\;
\Sigma_{10}:=\{\sigma_1,\sigma_4,\sigma_5\},
\;
\Sigma_{11}:=\{\sigma_1\cap \sigma_2,\sigma_3\cap \sigma_4\},
\\
\Sigma_1'=\Sigma_6':=\{\sigma_3,\sigma_4\},
\;
\Sigma_2':=\{\sigma_2,\sigma_3,\sigma_4\},
\;
\Sigma_3':=\{\sigma_2\cup \sigma_3,\sigma_4\},
\\
\Sigma_4'=\Sigma_7':=\{\sigma_2\cap \sigma_3\},
\;
\Sigma_5':=\{\sigma_2\},
\;
\Sigma_8':=\{\sigma_2,\sigma_3\},
\\
\Sigma_9':=\{\sigma_2\cup\sigma_3\},
\;
\Sigma_{10}':=\{\sigma_1\cap \sigma_2,\sigma_4\},
\;
\Sigma_{11}':=\{\sigma_1\cap \sigma_2,\sigma_3\cap \sigma_4\}.
\end{gather*}

For the convenience of the reader, we include a figure illustrating the above fans.
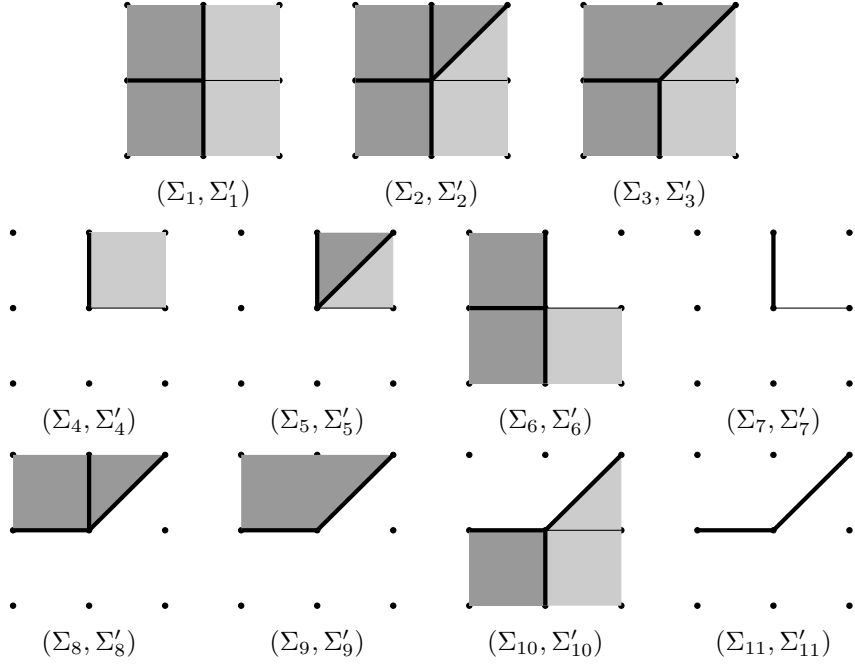
\[
\begin{tikzpicture}[yscale=1, xscale=1]
\foreach \a in {-1,0,1}
\foreach \b in {-1,0,1}
\filldraw (\a,\b) circle (1pt);
\fill[black!20] (0,0)--(1,0)--(1,1)--(0,1);
\fill[black!40] (0,0)--(0,1)--(-1,1)--(-1,0);
\fill[black!40] (0,0)--(-1,0)--(-1,-1)--(0,-1);
\fill[black!20] (0,0)--(0,-1)--(1,-1)--(1,0);
\draw (0,0)--(1,0);
\draw[ultra thick] (0,0)--(0,1);
\draw[ultra thick] (0,0)--(-1,0);
\draw[ultra thick] (0,0)--(0,-1);
\node at (0,-1.5) {$(\Sigma_1,\Sigma_1')$};
\begin{scope}[shift={(3,0)}]
\foreach \a in {-1,0,1}
\foreach \b in {-1,0,1}
\filldraw (\a,\b) circle (1pt);
\fill[black!20] (0,0)--(1,0)--(1,1);
\fill[black!40] (0,0)--(1,1)--(0,1);
\fill[black!40] (0,0)--(0,1)--(-1,1)--(-1,0);
\fill[black!40] (0,0)--(-1,0)--(-1,-1)--(0,-1);
\fill[black!20] (0,0)--(0,-1)--(1,-1)--(1,0);
\draw (0,0)--(1,0);
\draw[ultra thick] (0,0)--(0,1);
\draw[ultra thick] (0,0)--(-1,0);
\draw[ultra thick] (0,0)--(0,-1);
\draw[ultra thick] (0,0)--(1,1);
\node at (0,-1.5) {$(\Sigma_2,\Sigma_2')$};
\end{scope}
\begin{scope}[shift={(6,0)}]
\foreach \a in {-1,0,1}
\foreach \b in {-1,0,1}
\filldraw (\a,\b) circle (1pt);
\fill[black!20] (0,0)--(1,0)--(1,1);
\fill[black!40] (0,0)--(1,1)--(-1,1)--(-1,0);
\fill[black!40] (0,0)--(-1,0)--(-1,-1)--(0,-1);
\fill[black!20] (0,0)--(0,-1)--(1,-1)--(1,0);
\draw (0,0)--(1,0);
\draw[ultra thick] (0,0)--(1,1);
\draw[ultra thick] (0,0)--(-1,0);
\draw[ultra thick] (0,0)--(0,-1);
\node at (0,-1.5) {$(\Sigma_3,\Sigma_3')$};
\end{scope}
\end{tikzpicture}
\]
\[
\begin{tikzpicture}[yscale=1, xscale=1]
\foreach \a in {-1,0,1}
\foreach \b in {-1,0,1}
\filldraw (\a,\b) circle (1pt);
\fill[black!20] (0,0)--(1,0)--(1,1)--(0,1);
\draw (0,0)--(1,0);
\draw[ultra thick] (0,0)--(0,1);
\node at (0,-1.5) {$(\Sigma_4,\Sigma_4')$};
\begin{scope}[shift={(3,0)}]
\foreach \a in {-1,0,1}
\foreach \b in {-1,0,1}
\filldraw (\a,\b) circle (1pt);
\fill[black!20] (0,0)--(1,0)--(1,1);
\fill[black!40] (0,0)--(1,1)--(0,1);
\draw (0,0)--(1,0);
\draw[ultra thick] (0,0)--(1,1);
\draw[ultra thick] (0,0)--(0,1);
\node at (0,-1.5) {$(\Sigma_5,\Sigma_5')$};
\end{scope}
\begin{scope}[shift={(6,0)}]
\foreach \a in {-1,0,1}
\foreach \b in {-1,0,1}
\filldraw (\a,\b) circle (1pt);
\fill[black!20] (0,0)--(0,-1)--(1,-1)--(1,0);
\fill[black!40] (0,1)--(-1,1)--(-1,-1)--(0,-1);
\draw (0,0)--(1,0);
\draw[ultra thick] (0,0)--(0,1);
\draw[ultra thick] (0,0)--(-1,0);
\draw[ultra thick] (0,0)--(0,-1);
\node at (0,-1.5) {$(\Sigma_6,\Sigma_6')$};
\end{scope}
\begin{scope}[shift={(9,0)}]
\foreach \a in {-1,0,1}
\foreach \b in {-1,0,1}
\filldraw (\a,\b) circle (1pt);
\draw (0,0)--(1,0);
\draw[ultra thick] (0,0)--(0,1);
\node at (0,-1.5) {$(\Sigma_7,\Sigma_7')$};
\end{scope}
\end{tikzpicture}
\]
\[
\begin{tikzpicture}[yscale=1, xscale=1]
\foreach \a in {-1,0,1}
\foreach \b in {-1,0,1}
\filldraw (\a,\b) circle (1pt);
\fill[black!40] (0,0)--(1,1)--(-1,1)--(-1,0);
\draw[ultra thick] (0,0)--(1,1);
\draw[ultra thick] (0,0)--(-1,0);
\draw[ultra thick] (0,0)--(0,1);
\node at (0,-1.5) {$(\Sigma_8,\Sigma_8')$};
\begin{scope}[shift={(3,0)}]
\foreach \a in {-1,0,1}
\foreach \b in {-1,0,1}
\filldraw (\a,\b) circle (1pt);
\fill[black!40] (0,0)--(1,1)--(-1,1)--(-1,0);
\draw[ultra thick] (0,0)--(1,1);
\draw[ultra thick] (0,0)--(-1,0);
\node at (0,-1.5) {$(\Sigma_9,\Sigma_9')$};
\end{scope}
\begin{scope}[shift={(6,0)}]
\foreach \a in {-1,0,1}
\foreach \b in {-1,0,1}
\filldraw (\a,\b) circle (1pt);
\fill[black!20] (0,0)--(0,-1)--(1,-1)--(1,1);
\fill[black!40] (0,0)--(-1,0)--(-1,-1)--(0,-1);
\draw (0,0)--(1,0);
\draw[ultra thick] (0,0)--(1,1);
\draw[ultra thick] (0,0)--(-1,0);
\draw[ultra thick] (0,0)--(0,-1);
\node at (0,-1.5) {$(\Sigma_{10},\Sigma_{10}')$};
\end{scope}
\begin{scope}[shift={(9,0)}]
\foreach \a in {-1,0,1}
\foreach \b in {-1,0,1}
\filldraw (\a,\b) circle (1pt);
\draw[ultra thick] (0,0)--(1,1);
\draw[ultra thick] (0,0)--(-1,0);
\node at (0,-1.5) {$(\Sigma_{11},\Sigma_{11}')$};
\end{scope}
\end{tikzpicture}
\]

Here we have $\A_{(\Sigma_1,\Sigma_1')} = \boxx\times (\P^1, 0+\infty)$, 
and $\A_{(\Sigma_2,\Sigma_2')} = B_{(0,0)}(\P^1\times \P^1)$ with divisor the strict transform of the divisor on $\boxx\times (\P^1, 0+\infty)$ together with the exceptional divisor. 
Finally, 
$\A_{(\Sigma_3,\Sigma_3')}$ is (isomorphic to) the blow-up of $\P^2$ at the point $[1:0:0]$, 
with the coordinate lines $X_0=0$, $X_1=0$ and $X_2=0$ as divisors. 
Note that there is also a blow-up morphism  $B_{(0,0)}(\P^1\times \P^1) \to \A_{(\Sigma_3,\Sigma_3')}$.  
We consider the following Zariski open coverings of these log schemes. 
\begin{equation}\label{eq:diag-chart1}
\begin{tikzcd}
\A_{(\Sigma_7,\Sigma_7')}\arrow[d]\arrow[r]&
\A_{(\Sigma_6,\Sigma_6')}\arrow[d]
\\
\A_{(\Sigma_4,\Sigma_4')}\arrow[r]&
\A_{(\Sigma_1,\Sigma_1')},
\end{tikzcd}
\quad
\begin{tikzcd}
\A_{(\Sigma_7,\Sigma_7')}\arrow[d]\arrow[r]&
\A_{(\Sigma_6,\Sigma_6')}\arrow[d]
\\
\A_{(\Sigma_5,\Sigma_5')}\arrow[r]&
\A_{(\Sigma_2,\Sigma_2')},
\end{tikzcd}
\end{equation}
\begin{equation}\label{eq:diag-chart2}
\begin{tikzcd}
\A_{(\Sigma_{11},\Sigma_{11}')}\arrow[d]\arrow[r]&
\A_{(\Sigma_{10},\Sigma_{10}')}\arrow[d]
\\
\A_{(\Sigma_8,\Sigma_8')}\arrow[r]&
\A_{(\Sigma_2,\Sigma_2')},
\end{tikzcd}
\quad
\begin{tikzcd}
\A_{(\Sigma_{11},\Sigma_{11}')}\arrow[d]\arrow[r]&
\A_{(\Sigma_{10},\Sigma_{10}')}\arrow[d]
\\
\A_{(\Sigma_9,\Sigma_9')}\arrow[r]&
\A_{(\Sigma_3,\Sigma_3')}.
\end{tikzcd}
\end{equation}
Here, 
the second diagram in \eqref{eq:diag-chart1} and the first diagram in \eqref{eq:diag-chart2} are two different open charts for the blow-up of $\P^1\times \P^1$ at $(0,0)$ 
(and the toric notation helps keeping track of the different log structures). 
Note that $\A_{(\Sigma_6,\Sigma_6')}$ is a common Zariski open subset of $\P^1\times \P^1$ and of the blow-up $B_{(0,0)}(\P^1\times \P^1)$, namely it is the open subset 
$\P^1\times \P^1- \{(0,0)\}$ (which does not contain the centre of the blow-up). 
Similarly, 
note that $\A_{(\Sigma_{10},\Sigma_{10'})}$ is a common Zariski open subset of $\A_{(\Sigma_3,\Sigma_3')}$ and of the blow-up of $\P^1\times \P^1$ at $(0,0)$ 
(this time containing the center of the blow-up).
\vspace{0.1in}

Using ($sNis$-des) we have a zig-zag of isomorphisms
\begin{gather*}
M(\A_{(\Sigma_5,\Sigma_5')}\rightarrow \A_{(\Sigma_2,\Sigma_2')})
\xrightarrow{\cong}
M(\A_{(\Sigma_4,\Sigma_4')}\rightarrow \A_{(\Sigma_1,\Sigma_1')}),
\\
M(\A_{(\Sigma_8,\Sigma_8')}\rightarrow \A_{(\Sigma_2,\Sigma_2')})
\xrightarrow{\cong}
M(\A_{(\Sigma_9,\Sigma_9')}\rightarrow \A_{(\Sigma_3,\Sigma_3')}).
\end{gather*}

The squares in \eqref{eq:diag-chart1} can be completed to a $3$-dimensional commutative diagram with back side 
\[\begin{tikzcd}
\A_{(\Sigma_5,\Sigma_5')}\arrow[d]\arrow[r]&
\A_{(\Sigma_2,\Sigma_2')}\arrow[d]
\\
\A_{(\Sigma_4,\Sigma_4')}\arrow[r]&
\A_{(\Sigma_1,\Sigma_1')},
\end{tikzcd}
\]
where the morphisms $\A_{(\Sigma_2,\Sigma_2')}\to \A_{(\Sigma_1,\Sigma_1')}$ and $\A_{(\Sigma_5,\Sigma_5')}\to \A_{(\Sigma_4,\Sigma_4')}$ are induced by the blow-up morphisms 
$B_{(0,0)}(\P^1\times\P^1)\to \P^1\times \P^1$ and $B_{(0,0)}(\A^2)\to \A^2$, respectively. 
Taking into account the log structures, 
the morphism $(Y',E+H_2')\rightarrow (Y,H_2)$ is equal to the morphism $\A_{(\Sigma_5,\Sigma_5')}\rightarrow \A_{(\Sigma_4,\Sigma_4')}$.  
Similarly, we can complete the squares in \eqref{eq:diag-chart2} to a 3-dimensional commutative diagram with back side 
\[\begin{tikzcd}
\A_{(\Sigma_8,\Sigma_8')}\arrow[d]\arrow[r]&
\A_{(\Sigma_2,\Sigma_2')}\arrow[d]
\\
\A_{(\Sigma_9,\Sigma_9')}\arrow[r]&
\A_{(\Sigma_3,\Sigma_3')},
\end{tikzcd}
\]
where the blow-ups induce the vertical morphisms. 
More precisely, via the linear transformation
\[
\Z^2\rightarrow \Z^2
\]
mapping $(a,b)$ to $(a-b,a)$, 
the morphism $(Y',E+H_1'+H_2')\rightarrow (Y,H_1+H_2)$ can be identified with $\A_{(\Sigma_8,\Sigma_8')}\rightarrow \A_{(\Sigma_9,\Sigma_9')}$ 
(note that $\A_{\Sigma_9} = \Spec{\Z[x^{-1}, yx^{-1}]}$).
Due to Proposition \ref{A.3.44} there are induced isomorphisms
\begin{gather}
\label{A.3.39.7}
M(\A_{(\Sigma_5,\Sigma_5')}\rightarrow \A_{(\Sigma_4,\Sigma_4')})
\xrightarrow{\cong}
M(\A_{(\Sigma_2,\Sigma_2')}\rightarrow \A_{(\Sigma_1,\Sigma_1')}),
\\
\label{A.3.39.8}
M(\A_{(\Sigma_8,\Sigma_8')}\rightarrow \A_{(\Sigma_9,\Sigma_9')})
\xrightarrow{\cong}
M(\A_{(\Sigma_2,\Sigma_2')}\rightarrow \A_{(\Sigma_3,\Sigma_3')}).
\end{gather}
Thus from \eqref{A.3.39.7} and \eqref{A.3.39.8}, we reduce to showing the equivalence of
\[
M(\A_{(\Sigma_2,\Sigma_2')}\rightarrow \A_{(\Sigma_1,\Sigma_1')})=0
\]
and
\[
M(\A_{(\Sigma_2,\Sigma_2')}\rightarrow \A_{(\Sigma_3,\Sigma_3')})=0.
\]
\vspace{0.1in}

Both $\A_{(\Sigma_1,\Sigma_1')}$ and $\A_{(\Sigma_3,\Sigma_3')}$ are $\boxx$-bundles over $\boxx$, 
so that Proposition \ref{A.3.40} gives isomorphisms
\[
M(\A_{(\Sigma_1,\Sigma_1')})\rightarrow M(\Spec \Z),
\;
M(\A_{(\Sigma_3,\Sigma_3')})\rightarrow M(\Spec \Z).
\]
This finishes the proof.
\end{proof}

Recall now some terminology from \cite{Deg}.
\begin{df}[{\cite[Definition 4.5.3]{Deg}}]
A {\it closed pair}\index{closed pair} $(X,Z)$ {\it smooth} over $S$ is a closed immersion $Z\rightarrow X$ where both $X$ and $Z$ are object in $Sm/S$. 
A {\it cartesian} (resp.\ {\it excisive}) morphism $(f,f')$ of closed pairs $(X,Z)\rightarrow (X',Z')$ smooth over $S$ is a commutative diagram of $S$-schemes
\[
\begin{tikzcd}
Z\arrow[r]\arrow[d,"f'"']&X\arrow[d,"f"]\\
Z'\arrow[r]&X'
\end{tikzcd}
\]
which is cartesian (resp.\ such that $f':Z\rightarrow Z'$ is an isomorphism).
\end{df}

\begin{df}[{\cite[Definitions 4.5.8]{Deg}}]
Let $(X,Z)$ be a closed pair smooth over $S$, and consider the smooth closed pair $(\A_S^{r+s},\A_S^r)$ given by 
\[
a_0\times {\rm id}:\A_S^s\rightarrow \A_S^{r+s},
\] 
where $a_0:S\rightarrow \A_S^r$ is the $0$-section. 
A {\it parametrization}\index{parametrization} of $(X,Z)$ is a cartesian morphism $(f,f'):(X,Z)\rightarrow (\A_S^{r+s},\A_S^s)$ such that $f$ is \'etale.
\end{df}

\begin{const}
\label{A.3.16}
We will use the following technique appearing in the proof of homotopy purity in \cite[Lemma 2.28, p.\ 117]{MV} --- 
see also \cite[\S 4.5.2]{Deg} for another exposition, 
and see \cite[\S 7]{MR3431674} for formulations in related settings.
Let $(X,Z)$ be a closed pair smooth over $S$.  
Then by \cite[IV.17.12.2]{EGA}, Zariski locally on $X$, there are \'etale morphisms $w$ and $w'$ making the diagram
\[
\begin{tikzcd}
Z\arrow[r,"w'"]\arrow[d]&\A_S^r\arrow[d]
\\
X\arrow[r,"w"]&\A_S^{r+s}.
\end{tikzcd}
\]
cartesian. Setting $X_1:=\A_Z^s$ the morphism ${\rm id}\times w':X_1\rightarrow \A_S^{r+s}$ is \'etale. 
Let $\Delta$ be the diagonal of $Z$ over $\A_S^{r+s}$.  
This is a closed subscheme of $Z\times_{\A_S^{r+s}}Z$, 
and we define 
\[
X_2
:=
X\times_{\A_S^{r+s}}X_1-(X\times_{\A_S^{r+s}}Z-\Delta)\cup (Z\times_{\A_S^{r+s}}X_1-\Delta).
\]
By \cite[Lemma 4.5.6]{Deg}, $\Delta$ is a closed subscheme of $X_2$, and $X_2$ is an open subscheme of $X\times_{\A_S^{r+s}}X_1$. 
Let $i_2:Z\cong \Delta\rightarrow X_2$ be the closed immersion. 
There are cartesian excisive morphisms of closed pairs smooth over $S$
\begin{equation}\label{eq:excisivesquares}
(X,Z)\stackrel{(u,{\rm id})}\longleftarrow (X_2,Z)\stackrel{(v,{\rm id})}\longrightarrow (X_1,Z) = (\A^s_Z, Z),
\end{equation}
where $u$ (resp.\ $v$) denotes the \'etale morphism induced by the projection 
$$
X\times_{\A_S^{r+s}}X_1\rightarrow X 
(\text{resp.}\ X\times_{\A_S^{r+s}}X_1\rightarrow X_1).
$$
By working Zariski locally on $X$, 
the above morphisms of closed pairs smooth over $S$ allow us to reduce claims about $X$ and $Z$ to claims about $X_1$ and $Z$.
\end{const}

\begin{exm}
If $Z=S$, then $r=0$ and $X_1\rightarrow \A_S^{r+s} = \A^s_S$ is an isomorphism.
Thus $X_2\cong X$, and the morphism $v:X_2\rightarrow X_1$ is the parametrization map $X\rightarrow \A_S^s$.
\end{exm}

\begin{thm}
\label{A.3.7}
Assume that $\cT$ satisfies {\rm ($Zar$-sep)}, {\rm ($\boxx$-inv)}, {\rm ($sNis$-des)}, and {\rm ($div$-des)}.
Let $Y'\rightarrow Y$ be a morphism in $\cA\cB l_{Sm}/S$. 
Then 
\begin{equation}
\label{A.3.7.1}
M(Y')\rightarrow M(Y)
\end{equation}
is an isomorphism.
\end{thm}
\begin{proof}
We may assume that $Y'$ is an admissible blow-up of $Y$ along a connected codimension $d$ smooth center $Z$ for some integer $d\geq 2$.
We set $X:=\underline{Y}$ and $X':=\underline{Y'}$. Write $Z'$ for the exceptional divisor, $Z':=Z\times_X X'$.
Let $Z_1,\ldots,Z_r$ be divisors on $X$ smooth over $S$ such that $\partial Y=Z_1+\cdots+Z_r$. 
We write $Z_i'$ for the strict transform of $Z_i$ in the blow-up $X'$, 
and set 
\[
Y:=(X,Z_1+\cdots+Z_r),
\;
Y':=(X',Z_1'+\cdots+Z_r'+Z').
\]
Let $W_i:=\A^{i-1}\times \Spec \Z\times \A^{n-i}$ denote the smooth divisor of $\A^n$ given by $y_i=0$ for $1\leq i\leq n$, where $(y_1,\ldots, y_n)$ are coordinates on $\A^n$. 
Write $C_1, \ldots, C_l$ for the irreducible components of the intersections $Z\cap (Z_{i_1}\ldots \cap Z_{i_k})$, where $(i_1, \ldots, i_k)$ is a subset of $(1,\ldots, r)$ of length $k$, 
for $k=1,\ldots, r$. 
Finally we set 
\begin{equation}
\label{eq:definition-s} 
X(s):= {\rm max}\{c_1, \ldots, c_l\},
\end{equation}
where $c_i$ is the codimension of $C_i$ in $X$.
\vspace{0.1in}

Next, observe the question is Zariski local on $X$, 
owing to ($Zar$-sep), 
so that we may assume there exists an \'etale morphism $u:X\rightarrow S\times \A^n$ in $Sm/S$ for which
\begin{enumerate}
\item[(i)] $u^{-1}(S\times W_i)=Z_i$ for $1\leq i\leq r$ ,
\item[(ii)] $u^{-1}(S\times W)=Z$, where 
\[
W:=\A^{r-p}\times (\Spec \Z)^p \times (\Spec \Z)^q\times \A^{n-r-q}
\]
for some $1\leq p\leq r$ and $0\leq  q\leq n-r$. 
\end{enumerate}
Note that  $W$ is a smooth closed subscheme of codimension $p+q$ in $\A^n$ having strict normal crossing with $W_1+\ldots +W_i$ for every $i=1,\ldots, n$.
\vspace{0.1in}

Assume that $S=\Spec{\Z}$ for simplicity of notation (the general case is deduced from this one by base change $S\to \Spec{\Z}$).
For $r+1\leq i\leq r+q$ we set
\[
Z_i:=u^{-1}(W_i), \quad
T:=u^{-1}((\Spec \Z)^{r+q}\times \A^{n-r-q}).
\] 
Note that $T\subseteq Z$ and that $d=p+q$ is the codimension of $Z$ in $X$. 
Moreover, Zariski-locally, the codimension of $T$ in $X$ agrees with $X(s)$ defined in \eqref{eq:definition-s}.
We also observe that \eqref{A.3.7.1} is trivially an isomorphism if $\dim X=q=r=0$. 
Similarly, if $X(s)=1$, then $Z=T$ is a divisor in $X$ and coincides with one of the components of $\partial X$. 
Thus the blow-up is trivially an isomorphism. 
We can now proceed by induction on $X(s)>1$.
\vspace{0.1in}

\textbf{Step 1} {\it Reduction to the case when $X=\A^{p+q}$ and $r=p$.} 
Starting with a smooth pair $(X,T)$ and $(X,T)\rightarrow (\A^n,\A^{n-r-q})$, 
we may use the technique of Construction \ref{A.3.16} with respect to the parametrization $u\colon (X,T)\to (\A^n, \A^{n-r-q})$ to form the schemes $X_1$ and $X_2$. 
In this case $X_1 = \A^{r+q}_T$, 
$X_2$ is open in $X\times_{\A^n} X_1$,
and there exist \'etale morphisms $X_2\to X$, $X_2\to X_1$.

Let $X'_2$ be the pullback $X'\times_X X_2$. 
We equip $X_2$ (resp.~$X_2'$) with the log structure induced by $Y$ (resp.~$Y'$) and set
\[
Y_2:=Y\times_X X_2,\;\;Y_2':=Y'\times_X X_2.
\]
Next, we consider the open complement of $T$ in $X$
\[
U:=(X-T,(Z_1-T)+\cdots+(Z_r-T)),
\]
and the pullbacks
\[
U':=U\times_Y Y',\;\; U_2:=U\times_X X_2,\;\;U_2':=U'\times_X X_2.
\]
 
We obtain strict Nisnevich distinguished squares in $\mathscr{S}/S$
\[
\begin{tikzcd}
U_2\arrow[d]\arrow[r]&Y_2\arrow[d]\\
U\arrow[r]&Y
\end{tikzcd}\quad
\begin{tikzcd}
U_2'\arrow[d]\arrow[r]&Y_2'\arrow[d]\\
U'\arrow[r]&Y'.
\end{tikzcd}
\]
By ($sNis$-des) there are isomorphisms
\begin{equation}
\label{A.3.7.2}
M(U_2\rightarrow U)\xrightarrow{\cong} M(Y_2\rightarrow Y),
\;
M(U_2'\rightarrow U')\xrightarrow{\cong} M(Y_2'\rightarrow Y').
\end{equation}
Observe that by removing $T$, we obtain strict inequalities 
\[
U(s)<X(s), \quad U_2(s)< X(s).
\]
Thus, by induction, there are isomorphisms
\[
M(U_2'\rightarrow U_2)\cong M(U'\rightarrow U)\cong 0,
\]
which by Proposition \ref{A.3.44} imply the isomorphism 
\begin{equation}
\label{A.3.7.3}
M(U_2'\rightarrow U')\xrightarrow{\cong} M(U_2\rightarrow U).
\end{equation}
Combining \eqref{A.3.7.2} and \eqref{A.3.7.3}, we reduce to showing
\[
M(Y_2'\rightarrow Y_2)\cong 0.
\]
Here we can replace $X$ with $X_2$. 
Similarly, we can repeat the argument using the right excisive square in \eqref{eq:excisivesquares} and replace $X_2$ by $X_1$.
We note that $X_1=\A^{r+q}\times T$ and that the subschemes $Z_i$ and $Z$ of $X$ correspond to the subschemes 
$$
\A^{i-1}\times \Spec{\Z}\times \A^{r+q-i}\times T
\text{ and }
\A^{r-p}\times (\Spec{\Z})^p\times (\Spec{\Z})^q \times T.
$$  
Thus we have reduced to the case 
\begin{equation}\label{eq:step1equation}
X=\A^{r+q}\times T,
\;\;
Z_i=\A^{i-1}\times \Spec \Z \times \A^{r+q-i}\times T,
\end{equation}
\[
Z=\A^{r-p}\times (\Spec \Z)^p\times (\Spec \Z)^q\times T.
\]

Note that under \eqref{eq:step1equation}, we have further an identification
\[
Y\cong \A_\N^{r-p}\times \A_\N^p\times \A^q\times T. 
\]
Using the monoidal structure in $\cT$, we may further assume that $T=\Spec{\Z}$ and $r=p$,
i.e., 
\[
X=\A^{p+q},\;\;Z_i=\A^{i-1}\times \Spec \Z \times \A^{p+q-i},\;\;Z=\Spec{\Z},
\]
and 
\[
Y\cong \A_\N^p\times \A^q.
\]
\vspace{0.1in}
  
\textbf{Step 2} {\it Cases when $d=2$.} If $p=2$ and $q=0$, then the morphism
\[
(X',Z_1'+\cdots+Z_r'+Z')\rightarrow (X,Z_1+\cdots+Z_r)
\]
in $\cT$ is a log modification since this is $\A_M\rightarrow \A_\N^p$, 
where $M$ is the star subdivision of $\N^p$ relative to itself as in Definition \ref{Starsubdivision}.
Thus we are done by ($div$-des).
The case when $p=q=1$ is already done in Proposition \ref{A.3.39}.
Together with Step 1, we see that
\[
M(Y')\rightarrow M(Y)
\]
is an isomorphism for general $Y'\rightarrow Y$ if $d=2$.
\vspace{0.1in}
  
\textbf{Step 3} {\it General case.}  Recall that we may assume 
\[
X=\A^{p+q},\;\;Z_i=\A^{i-1}\times \Spec \Z\times \A^{p+q-i},\;\;Z=(\Spec \Z)^p\times (\Spec \Z)^q.
\]
We will construct $X'$ from $X$ by suitable blow-ups and blow-downs along codimension $2$ smooth centers. 
\vspace{0.1in}

Let $\{e_1,\ldots,e_{p+q}\}$ denote the standard basis of $\Z^{p+q}$.
Starting with $\Sigma_0:=\N^{p+q}$, for $i=1,\ldots,p+q-1$ we inductively let $\Sigma_i$ be the star subdivision \index{fan!star subdivision} $\Sigma_{i-1}^*(\eta_{i+1})$, 
where
\[
\eta_i
:=
{\rm Cone}( e_1,\ldots,e_i).
\]
Let $\Sigma_i'$ be the fan consists of cones $\sigma\in \Sigma_i$ such that
\[
\sigma\subset {\rm Cone}(e_1,e_2,\ldots,e_p,e_1,e_1+e_2,\ldots,e_1+\cdots+e_i).
\]
Let $\Gamma_0'$ be shorthand for $\Sigma_0^*(\eta_{p+q})$.
For $i=1,\ldots,p+q-2$, we inductively define $\Gamma_i'$ to be the star subdivision $\Gamma_i'^*(\eta_{i+1})$.
We let $\Gamma_i'$ be the fan consisting of cones $\sigma\in \Gamma_i'$ such that
\[
\sigma\subset {\rm Cone}(e_1,e_2,\ldots,e_p,e_1,e_1+e_2,\ldots,e_1+\cdots+e_i,e_1+\cdots+e_{p+q}).
\]
Note that $(\Sigma_{p+q-1},\Sigma_{p+q-1}')=(\Gamma_{p+q-2},\Gamma_{p+q-2}')$.
We have naturally induced maps
\[
(\Sigma_{p+q-1},\Sigma_{p+q-1}')\rightarrow \cdots \rightarrow (\Sigma_{0},\Sigma_{0}'),
\;
(\Gamma_{p+q-2},\Gamma_{p+q-2}')\rightarrow \cdots \rightarrow (\Gamma_{0},\Gamma_{0}').
\]
\vspace{0.1in}

The following figure illustrates the fans in the case when $p=2$ and $q=1$.
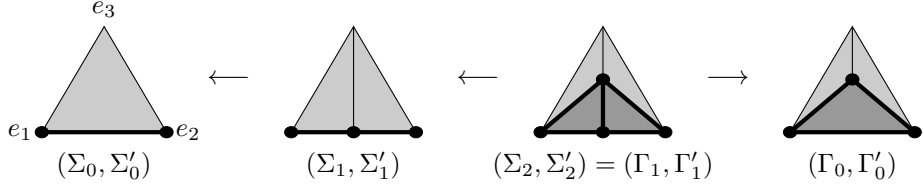
\[\begin{tikzpicture}[yscale=0.7, xscale=0.82]
\fill[black!20] (0,0)--(1,2)--(2,0)--(0,0);
\draw (0,0)--(1,2)--(2,0);
\draw[ultra thick] (0,0)--(2,0);
\fill[black!20] (4,0)--(5,2)--(6,0)--(4,0);
\filldraw (0,0) circle (3pt);
\filldraw (2,0) circle (3pt);
\draw (4,0)--(5,2)--(6,0);
\draw (5,2)--(5,0);
\draw[ultra thick] (4,0)--(6,0);
\fill[black!20] (8,0)--(9,2)--(10,0)--(8,0);
\fill[black!40] (8,0)--(10,0)--(9,1);
\filldraw (4,0) circle (3pt);
\filldraw (5,0) circle (3pt);
\filldraw (6,0) circle (3pt);
\draw (8,0)--(9,2)--(10,0);
\draw[ultra thick] (8,0)--(10,0);
\draw (9,2)--(9,1);
\draw[ultra thick] (9,1)--(9,0);
\draw[ultra thick] (8,0)--(9,1)--(10,0);
\filldraw (8,0) circle (3pt);
\filldraw (9,0) circle (3pt);
\filldraw (10,0) circle (3pt);
\filldraw (9,1) circle (3pt);
\fill[black!20] (12,0)--(13,2)--(14,0)--(12,0);
\fill[black!40] (12,0)--(14,0)--(13,1);
\draw[ultra thick] (12,0)--(14,0);
\draw (12,0)--(13,2)--(14,0)--(12,0);
\draw[ultra thick] (12,0)--(13,1)--(14,0);
\draw (13,1)--(13,2);
\filldraw (12,0) circle (3pt);
\filldraw (13,1) circle (3pt);
\filldraw (14,0) circle (3pt);
\node [below] at (1,-0.2) {$(\Sigma_0,\Sigma_0')$};
\node [below] at (5,-0.2) {$(\Sigma_1,\Sigma_1')$};
\node [below] at (9,-0.2) {$(\Sigma_2,\Sigma_2')=(\Gamma_1,\Gamma_1')$};
\node [below] at (13,-0.2) {$(\Gamma_0,\Gamma_0')$};
\node [left] at (0,0) {$e_1$};
\node [right] at (2,0) {$e_2$};
\node [above] at (1,2) {$e_3$};
\node at (3,1) {$\longleftarrow$};
\node at (7,1)  {$\longleftarrow$};
\node at (11,1) {$\longrightarrow$};
\end{tikzpicture}
\]
\vspace{0.1in}

From the description of admissible blow-ups in Example \ref{A.3.41}, we see that all the induced morphisms
\[
\A_{(\Sigma_{p+q-1},\Sigma_{p+q-1}')}\rightarrow \cdots \rightarrow \A_{(\Sigma_{0},\Sigma_{0}')},
\;
\A_{(\Gamma_{p+q-2},\Gamma_{p+q-2}')}\rightarrow \cdots \rightarrow \A_{(\Gamma_{0},\Gamma_{0}')}
\]
are admissible blow-ups along a codimension two smooth center.
Step 2 shows that the induced morphisms
\begin{gather*}
M(\A_{(\Sigma_{p+q-1},\Sigma_{p+q-1}')})\rightarrow \cdots \rightarrow M(\A_{(\Sigma_{0},\Sigma_{0}')}),
\\
M(\A_{(\Gamma_{p+q-2},\Gamma_{p+q-2}')})\rightarrow \cdots \rightarrow M(\A_{(\Gamma_{0},\Gamma_{0}')})
\end{gather*}
are isomorphisms.
This implies there is a naturally induced isomorphism
\[
M(\A_{(\Gamma_{0},\Gamma_{0}')})
\xrightarrow{\cong} 
M(\A_{(\Sigma_{0},\Sigma_{0}')}),
\]
which can be identified with
$$
M(Y')
\xrightarrow{\cong} 
M(Y).
$$
\end{proof}

\begin{df}
Let $\cA\cB l_{div}/S$\index[notation]{ABldivS @ $\cA\cB l_{div}/S$} denote the smallest class of proper birational morphisms in $lSm/S$
that is closed under composition and contains $\cA\cB l_{Sm}/S$ together with all the dividing coverings.
\end{df}

\begin{prop}
\label{A.3.11}
Assume that $\cT$ satisfies {\rm ($Zar$-sep)}, {\rm ($\boxx$-inv)}, {\rm ($sNis$-des)}, and {\rm ($div$-des)}.
Let $f:Y\rightarrow X$ be a morphism in $\cA\cB l_{div}/S$. 
Then there is a naturally induced isomorphism
\[
M(Y)\xrightarrow{\cong} M(X).
\]
\end{prop}
\begin{proof}
If $f$ is a dividing cover, then apply ($div$-des). 
If $f$ is in $\cA\cB l_{Sm}/S$, then apply Theorem \ref{A.3.7}. 
Since any morphism in $\cA\cB l_{div}/S$ is a composition of dividing covers and morphisms in $\cA\cB l_{Sm}/S$, we are done.
\end{proof}

\begin{prop}
For $X\in SmlSm/S$ and $Y\rightarrow X\in \cA\cB l_{div}/S$, 
there exists a log modification $Z\rightarrow Y$ such that the composite morphism $Z\rightarrow X$ is in $\cA \cB l_{Sm}/S$.
\end{prop}
\begin{proof}
Without loss of generality, we may assume $f$ is a dividing cover. 
Proposition \ref{A.9.75} shows that $f$ is a log modification.
Owing to Theorem \ref{Fan.16} the morphism $f$ is dominated by a sequence of log modifications $X_n\to \ldots \to X_1\to X$ along a smooth center, see Definition \ref{Fan.39}.  
This suffices to conclude the proof.
\end{proof}

\subsection{Blow-up triangles and $(\mathbb{P}^n,\mathbb{P}^{n-1})$-invariance}
Throughout this section we shall assume that the triangulated category $\cT$ satisfies the properties {\rm ($Zar$-sep)}, {\rm ($\boxx$-inv)}, {\rm ($sNis$-des)}, and {\rm ($div$-des)}.
Our aim is to conclude that ($(\P^\bullet,\P^{\bullet-1})$-inv) holds.
As in the formulation of ($(\P^\bullet,\P^{\bullet-1})$-inv), 
recall that we write 
\[
[x_0:\cdots:x_n]
\] 
for the coordinates on $\P^n$ and view $\P^{n-1}=H_n$ as the hyperplane defined by $x_n=0$.
The following result is a direct consequence of Theorem \ref{A.3.7}. 

\begin{prop}
\label{A.6.1}
For every $X\in\mathscr{S}/S$ the projection $X\times (\P^n,\P^{n-1})\rightarrow X$ induces an isomorphism
\[
M(X\times (\P^n,\P^{n-1}))\xrightarrow{\cong} M(X).
\]
\end{prop}
\begin{proof}
Owing to the isomorphism 
\[
M(X\times (\P^n,\P^{n-1}))\cong M(X)\otimes M(\P^n\times S,\P^{n-1}\times S), 
\]
it suffices to consider the case $X=S$. 
We proceed by induction on $n$. 
The case $n=0$ is trivial, and the case $n=1$ is precisely the {\rm ($\boxx$-inv)} property, so we may assume $n>1$.
For the standard coordinates $e_1,\ldots,e_n$ on $\Z^n$, 
recall that $\P^n$ is the toric variety associated to a fan $\Sigma_1$ with the maximal cones 
\[\Cone(e_1,\ldots,e_n),\]
\[\Cone(e_1,\ldots,e_{i-1},e_{i+1},\ldots,e_n,-e_1-\cdots-e_n)\quad (1\leq i\leq n-1),\]
\[\tau:=\Cone(e_1,\ldots,e_{n-1},-e_1-\cdots-e_n).\]
\vspace{0.1in}

We shall compare $\P^n$ with $\P^{n-1}\times \P^1$. 
Both of these schemes are dominated by a common blow-up, which we describe next.
Consider the star subdivision 
$$
\Sigma_3:=(\Sigma_1^*(\tau'))^*(\tau),
$$ 
where 
$$
\tau':=\Cone(e_n,-e_1-\cdots-e_n).
$$
The maximal cones of $\Sigma_3$ are given by 
\[
\sigma_{1}:=\Cone(e_1,\ldots,e_n),
\]
\[
\sigma_{i2}:=\Cone(e_1,\ldots,e_{i-1},e_{i+1},\ldots,e_n,-e_1-\cdots-e_{n-1})\quad (1\leq i\leq n-1),
\]
\[
\sigma_{i3}:=\Cone(e_1,\ldots,e_{i-1},e_{i+1},\ldots,e_{n-1},-e_1-\cdots-e_n,-e_1-\cdots-e_{n-1})\quad (1\leq i\leq n-1),
\]
\[
\sigma_{i4}:=\Cone(e_1,\ldots,e_{i-1},e_{i+1},\ldots,e_n,-e_1-\cdots-e_n,-e_n)\quad (1\leq i\leq n-1),
\]
\[\sigma_{5}:=\Cone(e_1,\ldots,e_{n-1},-e_n).\]
  
Moreover, 
$\P^{n-1}\times \P^1$ is the toric variety corresponding to a fan $\Sigma_2$ with the maximal cones 
\[
\sigma_{1},\;\sigma_{5},
\]
\[
\sigma_{i2}\quad (1\leq i\leq n-1),
\]
\[
\sigma_{i6}:=\Cone(e_1,\ldots,e_{i-1},e_{i+1},\ldots,e_{n-1},-e_n,-e_1-\cdots-e_{n-1})\quad (1\leq i\leq n-1).
\]
Since the maximal cones of the star subdivision $(\sigma_{i6})^*(\tau'')$ are $\sigma_{i3}$ and $\sigma_{i4}$, 
where $\tau'':=\Cone(-e_n,-e_1-\cdots-e_{n-1})$,
we observe that $\Sigma_3=\Sigma_2^*(\tau'')$.
\vspace{0.1in}

Form the following fans
\begin{gather*}
\Sigma_1':=\{{\rm Cone}(-e_1-\cdots-e_n)\},
\;
\Sigma_2':=\{\tau''\},
\\
\Sigma_3':=\{{\rm Cone}(-e_1-\cdots-e_n,-e_1-\cdots-e_{n-1}),{\rm Cone}(-e_1-\cdots-e_n,-e_n)\}.
\end{gather*}
\vspace{0.1in}

When $n=2$, the following figure describes the above fans.
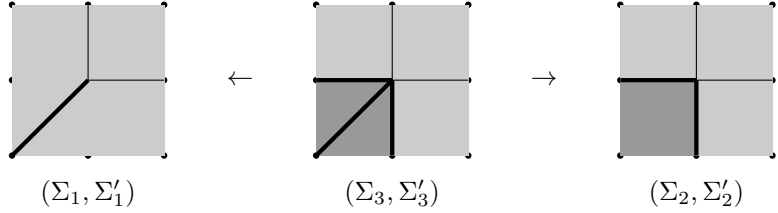
\[
\begin{tikzpicture}[yscale=1, xscale=1]
\foreach \a in {-1,0,1}
\foreach \b in {-1,0,1}
\filldraw (\a,\b) circle (1pt);
\fill[black!20] (1,1)--(1,-1)--(-1,-1)--(-1,1);
\draw (0,0)--(1,0);
\draw (0,0)--(0,1);
\draw[ultra thick] (0,0)--(-1,-1);
\node at (0,-1.5) {$(\Sigma_1,\Sigma_1')$};
\node at (2,0) {$\leftarrow$};
\begin{scope}[shift={(4,0)}]
\foreach \a in {-1,0,1}
\foreach \b in {-1,0,1}
\filldraw (\a,\b) circle (1pt);
\fill[black!20] (1,1)--(1,-1)--(-1,-1)--(-1,1);
\fill[black!40] (0,0)--(0,-1)--(-1,-1)--(-1,0);
\draw (0,0)--(1,0);
\draw (0,0)--(0,1);
\draw[ultra thick] (0,0)--(-1,-1);
\draw[ultra thick] (0,0)--(0,-1);
\draw[ultra thick] (0,0)--(-1,0);
\node at (2,0) {$\rightarrow$};
\node at (0,-1.5) {$(\Sigma_3,\Sigma_3')$};
\end{scope}
\begin{scope}[shift={(8,0)}]
\foreach \a in {-1,0,1}
\foreach \b in {-1,0,1}
\filldraw (\a,\b) circle (1pt);
\fill[black!20] (1,1)--(1,-1)--(-1,-1)--(-1,1);
\fill[black!40] (0,0)--(0,-1)--(-1,-1)--(-1,0);
\draw (0,0)--(1,0);
\draw (0,0)--(0,1);
\draw[ultra thick] (0,0)--(0,-1);
\draw[ultra thick] (0,0)--(-1,0);
\node at (0,-1.5) {$(\Sigma_2,\Sigma_2')$};
\end{scope}
\end{tikzpicture}
\]
The scheme associated with the left-hand drawing is $(\P^2, \P^1)$, where $\P^1$ is the line at infinity $x_2=0$. 
The middle drawing is the blow-up of $\P^1\times \P^1$ at $(\infty, \infty)$ with log structure given by the exceptional divisor, 
indicated by the diagonal line, 
together with the strict transforms of $\infty\times \P^1 + \P^1\times \infty$ (cf.\ the proof of Proposition \ref{A.3.39}). 
Finally, the right-hand drawing is $\boxx\times \boxx$. 
Note that the blow-up dominates both $\boxx\times \boxx$ and $(\P^2, \P^1)$. 
From the description of admissible blow-ups along smooth centers discussed in Example \ref{A.3.41}, 
we observe that the morphisms
\[
\A_{(\Sigma_1,\Sigma_1')}\leftarrow \A_{(\Sigma_3,\Sigma_3')}\rightarrow \A_{(\Sigma_2,\Sigma_2')}
\]
are admissible blow-ups along smooth centers. 
More precisely, the morphism 
$$
\A_{(\Sigma_3,\Sigma_3')} \to  \A_{(\Sigma_2,\Sigma_2')}
$$ 
is a dividing cover, while the morphism
$$
\A_{(\Sigma_3,\Sigma_3')} \to \A_{(\Sigma_1,\Sigma_1')}
$$ 
is an admissible blow-up, albeit not a dividing cover.
\vspace{0.1in}
  
Theorem \ref{A.3.7} shows there are naturally induced isomorphisms
\[
M( \A_{(\Sigma_1,\Sigma_1')}\times S)\xleftarrow{\cong} M(\A_{(\Sigma_3,\Sigma_3')}\times S)\xrightarrow{\cong} M(\A_{(\Sigma_2,\Sigma_2')}\times S).
\]
Since we have
\[
\A_{(\Sigma_1,\Sigma_1')}\cong (\P^n,\P^{n-1}) \text{ and }
\A_{(\Sigma_2,\Sigma_2')}\cong (\P^{n-1},\P^{n-2})\times \boxx,
\]
there is a naturally induced isomorphism
\[
M((\P^n,\P^{n-1})\times S)\cong M((\P^{n-1},\P^{n-2})\times S)\otimes M(\boxx\times S)
\]
Using induction the above implies there is an isomorphism
\[
M(\P^n\times S,\P^{n-1}\times S)\cong M(\boxx\times S)^{\otimes n}.
\]
This finishes the proof by appealing to the property ($\boxx$-inv). 
\end{proof}

\begin{prop}
\label{A.6.3}
Suppose $Y$ is a $(\P^n,\P^{n-1})$-bundle over $X\in\mathscr{S}/S$.
Then the projection $Y\rightarrow X$ induces an isomorphism
\[
M(Y)\xrightarrow{\cong} M(X).
\]
\end{prop}
\begin{proof}
As in the proof of Proposition \ref{A.3.40} by using induction on the number of open subsets in a trivialization of $Y$, we may assume $Y\cong X\times (\P^n,\P^{n-1})$. 
We conclude by reference to Proposition \ref{A.6.1}.
\end{proof}

In the following theorem, 
we compare the relative motive $M(Z\to X)$ for a smooth closed pair $(X,Z)$ with the relative motive of the blow-up of $X$ along $Z$. 
The result would be automatic in the presence of the localization property, which does not hold in our log setting (cf.\ the example in Section \ref{sec:Gysin}).  

\begin{thm}
\label{A.3.13}
Let $(X,Z)$ be a closed pair smooth over $S$, and let $X'$ be the blow-up of $X$ along a smooth center $Z$. 
For $Z':=Z\times_X X'$ there is a naturally induced isomorphism
\[
M(Z'\rightarrow X')\xrightarrow{\cong}  M(Z\rightarrow X).
\]
\end{thm}
\begin{proof}
Note that both $X$ and $Z$ have a trivial log structure. 
By Proposition \ref{A.3.44}, it suffices to show there is a naturally induced isomorphism
\[
M(Z'\rightarrow Z)\xrightarrow{\cong} M(X'\rightarrow X).
\]
This question is Zariski local on  $X$ by ($Zar$-sep), 
so that we may assume there exists an \'etale morphism $u:X\rightarrow \A^{r+s}\times S$ of schemes for which $u^{-1}(W\times S)=Z$, 
where $W$ denotes the subscheme $(\Spec Z)^r\times \A^s$ of $\A^{r+s}$. 
Using $u$ and $Z\hookrightarrow X$, 
we obtain schemes $X_1$ and $X_2$ over $k$ as in Construction \ref{A.3.16}, together with the corresponding excisive squares. 
\vspace{0.1in}

In $\mathscr{S}/S$, there are strict Nisnevich distinguished squares
\[
\begin{tikzcd}
X_2-Z\arrow[d]\arrow[r]&X_2\arrow[d]\\
X-Z\arrow[r]&X
\end{tikzcd}\quad
\begin{tikzcd}
X_2'-Z'\arrow[d]\arrow[r]&X_2'\arrow[d]\\
X'-Z'\arrow[r]&X'
\end{tikzcd}\]
where $X_2':=X_2\times_X X'$ (recall that $Z\cong \Delta \hookrightarrow X_2$). 
Property ($sNis$-des) implies there are isomorphisms
\[
M((X_2-Z)\rightarrow (X-Z))\xrightarrow{\cong} M(X_2\rightarrow X),
\]
and
\[
M((X_2'-Z')\rightarrow (X'-Z'))\xrightarrow{\cong} M(X_2'\rightarrow X').
\]
\vspace{0.1in}

On account of the identifications $X_2-Z\cong X_2'-Z'$ and $X-Z\cong X'-Z'$, there is an isomorphism
\[
M(X_2'\rightarrow X')\xrightarrow{\cong} M(X_2\rightarrow X).
\]
By Proposition \ref{A.3.44} we find the isomorphism
\[
M(X_2'\rightarrow X_2)\xrightarrow{\cong} M(X'\rightarrow X).
\]

Similarly, 
letting $X_1'$ denote the blow-up of $X_1$ along $Z$, 
there is a naturally induced isomorphism
\[
M(X_2'\rightarrow X_2)\xrightarrow{\cong} M(X_1'\rightarrow X_1).
\]
\vspace{0.1in}

By the above, we may replace $X$ with $X_1$ and assume $X=\A^r\times Z$. 
Using the monoidal structure on $\cT$, we can further reduce to the case $X=\A^r$ and $Z=\{0\}$. Then $X'$ is $B_{0}(\A^r)$ and $Z'=E$ is the exceptional divisor.
Let $Y$ be shorthand for $(\P^n\times S,H_n\times S)$, 
where $H_n=\P^{n-1}$ is the hyperplane at infinity. 
Let $Y'$ be the object of $\mathscr{S}/S$ with log structure induced by $Y$ and whose underlying scheme is the blow-up of $\P^n$ at the origin.
\vspace{0.1in}

Using the open immersion $X\hookrightarrow Y$, we obtain strict Nisnevich distinguished squares in $\mathscr{S}/S$
\begin{equation}\label{eq:blowup-reduction-nolog}
\begin{tikzcd}
X-Z\arrow[r]\arrow[d]&X\arrow[d]\\
Y-Z\arrow[r]&Y
\end{tikzcd}
\quad
\begin{tikzcd}
X'-Z'\arrow[r]\arrow[d]&X'\arrow[d]\\
Y'-Z'\arrow[r]&Y'.
\end{tikzcd}
\end{equation}
By ($sNis$-des), there are isomorphisms
\begin{align*}
M((X-Z)\rightarrow (Y-Z))&\xrightarrow{\cong} M(X\rightarrow Y),\\
M((X'-Z')\rightarrow (Y'-Z'))&\xrightarrow{\cong} M(X'\rightarrow Y').
\end{align*}
Since $X'-Z'\cong X-Z$ and $Y'-Z'\cong Y-Z$, we obtain an isomorphism
\[
M(X'\rightarrow Y')\xrightarrow{\cong} M(X\rightarrow Y).
\]
Using Proposition \ref{A.3.44}, and the relevant cube obtained by combining the two squares in \eqref{eq:blowup-reduction-nolog}, we get the isomorphism
\[
M(X'\rightarrow X)\xrightarrow{\cong} M(Y'\rightarrow Y).
\]
Hence it suffices to show there is a naturally induced isomorphism
\[
M(Z'\rightarrow Z)\xrightarrow{\cong} M(Y'\rightarrow Y).
\]

Due to Proposition \ref{A.6.1} there is an isomorphism 
\[
M(Z)\xrightarrow{\cong} M(Y),
\]
given by the inverse of the projection $(\P^n, \P^{n-1})\to S$.
Hence we are reduced to consider $M(Z')\rightarrow M(X')$. 
Lemma \ref{A.6.5} shows that $X'$ is a $\boxx$-bundle over $\P_S^{n-1}$. 
Thus by Proposition \ref{A.6.3}, the projection $X'\rightarrow \P_S^{n-1}$ induces an isomorphism
\[
M(X')\xrightarrow{\cong} M(\P_S^{n-1}).
\]
This concludes the proof since the composition $Z'\rightarrow X'\rightarrow \P_S^{n-1}$ is an isomorphism of schemes.
\end{proof}

\subsection{Thom motives}
\label{ssec:ThomMotives}
Throughout this section we shall assume that $\cT$ satisfies the properties {\rm ($Zar$-sep)}, {\rm ($\boxx$-inv)}, {\rm ($sNis$-des)}, and {\rm ($div$-des)}.
Let $E$ be a vector bundle over $X\in Sm/k$ with $0$-section $Z$.
The Thom motive associated with $E$ in $\dmeff$ is defined as
\[
M((E-Z)\rightarrow  E).
\]
In the log setting, 
taking the open complement $E-Z$ of the $0$-section in $E$ does not give the ``correct'' homotopy type: 
we can see this by taking the Betti realization of the log scheme, as discussed in \cite[V.1]{Ogu}. 
See also Proposition \ref{prop::unitdiskbundle} below.  
In the following, we will discuss an appropriate definition and then prove some properties of our Thom motives. 

\begin{df}
\label{A.3.15}
Suppose $Z_1,\ldots,Z_r$ form a strict normal crossing divisor on $X\in Sm/S$, and that a smooth closed subscheme $Z$ has strict normal crossing with $Z_1+\cdots+Z_r$ over $S$.
For $Y:=(X,Z_1+\cdots+Z_r)$ the blow-up of $Y$ along $Z$ is defined as
\[
B_ZY:=(B_Z X,W_1+\cdots+W_r),
\]
where $B_Z X$ is the blow-up of $X$ along $Z$, and $W_i$ is the strict transform of $Z_i$.
If $E$ is the corresponding exceptional divisor on $X$, then there is a naturally induced morphism of schemes
\[
B_Z X-(E\cup W_1\cup \cdots\cup W_r)\rightarrow X-(Z_1\cup \cdots\cup Z_r).
\]
Since the construction of compactifying log structure is functorial by \cite[\S III.1.6]{Ogu}, the above morphism induces a morphism of fs log schemes
\[
(B_Z Y,E)\rightarrow Y.
\]
See Definition \ref{def:associative-notation-log} for the notation. By definition, we say that $(B_Z Y,E)$ is the blow-up of $Y$ along $Z$.\index{blow-up!of fs log schemes} 
On the other hand, we do \emph{not} have a naturally induced morphism of fs log schemes $B_ZY\rightarrow Y$ in general.
\vspace{0.1in}

Regard $B_Z Y$ as the strict transform of $Y\times \{0\}$ in $B_{Z\times \{0\}}(Y\times \boxx)$. 
The deformation space of a pair $(Y,Z)$ is defined as
\[
D_Z Y:=B_{Z\times \{0\}}(Y\times \boxx)-B_Z Y.
\]
The strict transform of $Z\times \boxx$ in $B_{Z\times \{0\}} (Y\times \boxx)$ does not intersect with $B_ZX$, so we may and will consider $Z\times \boxx$ as a smooth divisor on $D_Z Y$.
\vspace{0.1in}

Consider the normal bundle $p:N_ZX\rightarrow X$ of $Z$ in $X$.
The \emph{normal bundle}\index{normal bundle} of $Z$ in $Y$ is defined as
\[
N_Z Y:=(N_Z X,p^{-1}(Z_1)+\cdots+p^{-1}(Z_r)).
\]
Note that $N_ZY\rightarrow Y$ is a vector bundle.
\end{df}

\begin{prop}
\label{A.3.3}
Suppose $Z_1,\ldots,Z_r$ are divisors forming a strict normal crossing divisor on $X\in Sm/S$ over $S$, and that $Z$ has strict normal crossing with $Z_1+\cdots+Z_r$ over $S$.
Let $f:X'\rightarrow X$ be an \'etale morphism in $Sm/S$ such that $Z':=f^{-1}(Z)\rightarrow Z$ is an isomorphism, 
and let $E$ (resp.\ $E'$) be the exceptional divisor on the blow-up $B_Z X$ (resp.\ $B_{Z'}X'$).
Let 
\[
Y:=(X,Z_1+\cdots+Z_r),
\;
Y':=(X',Z_1'+\cdots+Z_r'),
\] 
where $Z_i'$ denotes the pullback of $Z_i$ along $f$. 
Then $f$ induces an isomorphism
\[
M((B_{Z'}Y',E')\rightarrow Y')\xrightarrow{\cong} M((B_ZY,E)\rightarrow Y).
\]
\end{prop}
\begin{proof}
For notational convenience we set 
\[
U:=(X-Z,(Z_1-Z)+\cdots+(Z_r-Z)),\; U':=U\times_Y Y'.
\] 
Since $X'\amalg X-Z\rightarrow X$ and $B_{Z'}X'\amalg X-Z\rightarrow B_Z X$ are Nisnevich covers by assumption,
there are induced strict Nisnevich distinguished squares
\[
\begin{tikzcd}
U'\arrow[d]\arrow[r]&(B_{Z'}Y',E')\arrow[d]& U'\arrow[d]\arrow[r]&Y'\arrow[d]\\
U\arrow[r]&(B_Z Y, E)&U\arrow[r]&Y.
\end{tikzcd}
\]
By passing to motives, 
property ($sNis$-des) shows there are naturally induced isomorphisms
\[
M(U'\rightarrow U)\xrightarrow{\cong} M((B_{Z'}Y',E')\rightarrow (B_Z Y,E)),
\]
\[
M(U'\rightarrow U)\xrightarrow{\cong} M(Y'\rightarrow Y).
\]
This finishes the proof by appealing to Proposition \ref{A.3.44} and the isomorphism
\[
M((B_{Z'}Y',E')\rightarrow (B_ZY,E))\xrightarrow{\cong} M(Y'\rightarrow Y).
\]
\end{proof}

Associated to vector bundles over fs log schemes, we have the fundamental notion of a Thom motive;
this plays a crucial role in the theory of log motives.

\begin{df}
\label{A.3.22}
Suppose $Z_1,\ldots,Z_r$ form a strict normal crossing divisor on $X\in Sm/S$.
For $Y:=(X,Z_1+\cdots+Z_r)$ and a vector bundle $\xi:\cE\rightarrow Y$ of fixed rank, let $Z$ be the $0$-section of $\cE$.
Then
\[
\xi^{-1}(Z_1),\ldots,\xi^{-1}(Z_r)
\]
form a strict normal crossing divisor on $\underline{\cE}$, and $Z$ has strict normal crossing with $\xi^{-1}(Z_1)+\cdots+\xi^{-1}(Z_r)$ over $S$.
Let $E$ be the exceptional divisor on the blow-up $B_Z(\cE)$.
The {\it Thom motive}\index{Thom motive} of $\cE$ over $X$ is defined as \index[notation]{mth @ $MTh_X$}
\[
MTh_X(\cE)
:=
M((B_Z(\cE),E)\rightarrow \cE).
\]
We often omit the subscript $X$ in the notation when no confusion seems likely to arise.
\end{df}

According to the following proposition, 
the Betti realization of the motivic Thom space is homotopy equivalent to the quotient of the unit disk bundle by the unit sphere bundle, 
which is a formulation of Thom spaces in algebraic topology.

\begin{prop}
\label{prop::unitdiskbundle}
Suppose $X$ is a separated scheme over $\Spec\C$, and let $\xi:\cE\rightarrow X$ be a vector bundle. 
Then the Betti realization of $\cE$ (resp.\ $\widetilde{\cE}:=(B_Z(\cE),E)$) is homotopy equivalent to a unit disk bundle (resp.\ unit sphere bundle).
\end{prop}
\begin{proof}
The question is Zariski local on $X$, so we may assume that $\xi$ is a trivial rank $n$ vector bundle.
Consider the fs log scheme $Y:=(B_{\{0\}}\A_{\C}^n,D)$ where $D$ is the exceptional divisor. 
Then by Proposition \ref{A.9.74}, the Betti realization of $Y$ is homotopy equivalent to 
\[
(Y-\partial Y)_{log}\cong (\A_{\C}^r-\{0\})_{an}\cong S^{2r-1},
\] 
where $S^i$ denotes the $i$-dimensional unit sphere. 
Recall the Betti realization of $\A_{\C}^r$ is homotopy equivalent to the $2r$-dimensional unit disk $D^{2r}$. 
Since $\cE\cong X\times_{\C}\A_\C^n$ and $\widetilde{\cE}\cong X\times_{\C} Y$, we are done.
\end{proof}

We prove an analog of \cite[Proposition 2.17(3), p.\ 112]{MV} as follows. 
Note that our proof, 
which involves the blow-up along the zero section of the vector bundle, 
is different from the proof given by Morel-Voevodsky \cite{MV}.

\begin{prop}
\label{A.3.34}
Let $\cE$ be a rank $n$ vector bundle over $X$ in $SmlSm/S$, 
and view $\P(\cE)$ as the closed subscheme of $\P(\cE\oplus \cO)$ at infinity.
Then there is a canonical isomorphism
\[
MTh(\cE)\xrightarrow{\cong} M(\P(\cE)\rightarrow \P(\cE\oplus \cO)).
\]
In particular, if $n=1$ and $\cE$ is trivial, then
\[
MTh(X\times \A^1)\cong M(X)(1)[2].
\]
\end{prop}
\begin{proof}
Let $Y$ (resp.\ $Y'$) be the blow-up of $\cE$ (resp.\ $\P(\cE\oplus \cO)$) along the zero section $Z_0$, and let $E$ (resp.\ $E'$) be the exceptional divisor on $Y$ (resp.\ $Y'$).
View $\cE$ as the open subscheme $\P(\cE\oplus \cO)-\P(\cE)$ of $\P(\cE)$.
Owing to ($sNis$-des) and Proposition \ref{A.3.44}, the cartesian square
\[
\begin{tikzcd}
(Y,E)\arrow[r]\arrow[d]& (Y',E')\arrow[d]\\
\cE\arrow[r]&\P(\cE\oplus \cO)
\end{tikzcd}
\]
induces an isomorphism
\begin{equation}
\label{A.3.34.1}
MTh(\cE)=M((Y,E)\rightarrow \cE)\xrightarrow{\cong} M((Y',E')\rightarrow \P(\cE\oplus \cO)).
\end{equation}
Since $Z_0\cap \P(\cE)=\emptyset$, we can view $\P(\cE)$ as a divisor $Z$ of $Y'$ not intersecting with $E'$.
\vspace{0.1in}

We claim the closed immersion $Z\rightarrow (Y',E')$ induces an isomorphism
\[
M(Z)\xrightarrow{\cong} M(Y',E').
\]
Due to ($Zar$-sep), we reduce to the case when $\cE$ is trivial.
In this case, owing to Lemma \ref{A.3.40}, $(Y',E')$ is a $\boxx$-bundle over $\P_X^{n-1}$.
The composite morphism $Z\rightarrow (Y',E')\rightarrow \P_X^{n-1}$ is an isomorphism.
Hence the composite morphism
\[
M(Z)\rightarrow M(Y',E')\rightarrow M(\P_X^{n-1})
\]
is also an isomorphism. 
The second morphism is an isomorphism by Proposition \ref{A.6.3}.
Thus the first morphism is an isomorphism, which proves the claim.
\vspace{0.1in}

From the above and Proposition \ref{A.3.48}, we deduce the isomorphism
\begin{equation}
\label{A.3.34.2}
M((Y',E')\rightarrow \cP(\cE\oplus \cO))\cong M(\P(\cE)\rightarrow \P(\cE\oplus \cO)).
\end{equation}
To conclude the proof we combine \eqref{A.3.34.1} and \eqref{A.3.34.2}.
\end{proof}

\begin{df}
For a smooth fan $\Sigma$ with a cone $\sigma$, we have used the notation $\Sigma^*(\sigma)$, which is the star subdivison (see Definition \ref{Starsubdivision}) 
\index{fan!star subdivision} relative to $\sigma$.
If $x$ is the center of $\sigma$, i.e., $x=(a_1+\cdots+a_n)$ whenever $\sigma={\rm Cone}(a_1,\ldots,a_n)$, we set
\[
\Sigma^*\langle x\rangle :=\Sigma^*(\sigma).
\]
\end{df}

In $\A^1$-homotopy theory, 
the construction of the Thom space of a vector bundle $E$ on $X$ is compatible with the monoidal structure.
That is, if $E_i$ is a vector bundle on $X_i$ for $i=0,1$, there is a canonical isomorphism 
$$
Th_{X_1}(E_1)\otimes Th_{X_2}(E_2) \cong Th_{X_1\times X_2}(E_1\times E_2).
$$ 
This follows immediately from the construction, see \cite[Proposition 2.17 (1)]{MV}. 
Alas, in our setting, this property is less evident. 
In the following, we address this problem for double and triple products of vector bundles. 
This will be applied later in the proof of Proposition \ref{A.3.43}, 
see also Remark \ref{rmk:why-3-bundles-necessary}.

\begin{rmk} 
Assume that $k$ admits resolution of singularities as in Definition \ref{A.3.8}. 
In this case, the compatibility between Thom motives and products can be deduced a posteriori from the one in $\dmeff$, 
provided the underlying scheme $\ul{X}$ is proper. 
This is a consequence of Proposition \ref{A.4.31}.
\end{rmk}

\begin{lem}
\label{A.3.52}
Let $X$ be an irreducible separated scheme with two nowhere dense closed subschemes $Z_1$ and $Z_2$.
We set $W_1:=Z_1\times_X B_{Z_2}X$ and $W_2:=B_{Z_2}X\times_X Z_1$.
If $Y$ is the closure of the generic point of $B_{Z_1}X\times_X B_{Z_2}X$ whose image in $X$ is the generic point, then there are isomorphisms
\[
Y\cong B_{W_2}(B_{Z_1}X)\cong B_{W_1}(B_{Z_2}X).
\]
\end{lem}
\begin{proof}
Due to \cite[Section B.6.9]{Fulton}, $B_{W_2}(B_{Z_1}X)$ is a closed subscheme of $B_{Z_1}X\times_X B_{Z_2}X$.
Since $B_{W_2}(B_{Z_1}X)$ is proper and birational over $X$, we have $Y\cong B_{W_2}(B_{Z_1}X)$.
We can similarly show $Y\cong B_{W_1}(B_{Z_2}X)$.
\end{proof}

\begin{lem}
\label{A.3.53}
Let $\Sigma$ be a fan in $N$ with two partial subdivisions $\Sigma_1$ and $\Sigma_2$.
We set $\Sigma_{12}:=\Sigma_1\times_\Sigma \Sigma_2$, i.e.,
\[
\Sigma_{12}=\{\sigma_1\cap \sigma_2:\sigma_1\in \Sigma_1,\sigma_2\in \Sigma_2\}.
\]
If $Y$ is the closure of the generic point $\xi$ of $\ul{\A_{\Sigma_1}}\times_{\ul{\A_\Sigma}}\ul{\A_{\Sigma_2}}$ whose image in $\ul{\A_{\Sigma}}$ is the generic point, then there is an isomorphism
\[
Y\cong \ul{\A_{\Sigma_{12}}}.
\]
\end{lem}
\begin{proof}
We only need to show this for cones of $\Sigma$, $\Sigma_1$, and $\Sigma_2$.
Hence we may assume
\[
\Sigma=\Spec{P},
\;
\Sigma_1=\Spec{P_1},
\;
\Sigma_2=\Spec{P_2},
\]
where $P$, $P_1$, and $P_2$ are fs submonoids of the dual lattice $M$ of $N$ such that
\[
M\cong P^\gp\cong P_1^\gp\cong P_2^\gp.
\]
There is an isomorphism
\[
\ul{\A_{P_1}}\times_{\ul{\A_P}}\ul{\A_{\P_2}}
\cong
\ul{\A_{P_1\oplus_P^{\rm mon}P_2}},
\]
where $P_1\oplus_P^{\rm mon}P_2$ denotes the amalgamated sum in the category of monoids.
The amalgamated sum $P_1\oplus_P^{\rm int}P_2$ in the category of integral monoids is the image of $P_1\oplus P_2$ in $M$, i.e.,
\[
P_1\oplus_P^{\rm int}P_2=P_1+P_2.
\]
There is a naturally induced homomorphism
\[
f\colon \Z[P_1\oplus_P^{\rm mon}P_2]\rightarrow \Z[M],
\]
and since $Y$ is the closure of $\xi$ we observe that $Y=\Spec{\im{f}}$.
Thus
\[
Y\cong \Spec{\Z[P_1+P_2]}
\]
since $\im{f}$ agrees with $\Z[P_1+P_2]$.
The fan $\Sigma_{12}$ has only one maximal cone $P_1^\vee\cap P_2^\vee$, which corresponds to $P_1+P_2$ by duality.
This shows $Y\cong \ul{\A_{\Sigma_{12}}}$.
\end{proof}

\begin{const}
\label{A.3.47}
For $X_1,X_2,X_3\in Sm/S$ and vector bundles $\cE_i\rightarrow X_i$ of fixed rank for $i=1,2,3$, we will for future usage construct 55 blow-ups of the triple product
$$
T:=\cE_1\times_S \cE_2\times_S \cE_3.
$$

Let $Z_i$ be the $0$-section of $\cE_i$ for $i=1,2,3$, and we set
\begin{gather*}
W_1:=Z_1\times_S \cE_2\times_S \cE_3,
\;
W_2:=\cE_1\times_S Z_2\times_S \cE_3,
\;
W_3:=\cE_1\times_S \cE_2\times_S Z_3,
\\
W_4:=Z_1\times_S Z_2\times_S \cE_3,
\;
W_5:=\cE_1\times_S Z_2\times_S Z_3,
\;
W_6:=Z_1\times_S Z_2\times_S Z_3.
\end{gather*}
For every subset $I=\{i_1,\ldots,i_n\}\subset \{1,\ldots,6\}$ with $i_1<\ldots<i_n$, consider the fiber product
\[
B_{W_{i_1}}(T)\times_T \cdots \times_T B_{W_{i_n}}(T)
\]
Collect all generic points whose image in $T$ are also generic, and let $T_I$ denote the union of the closures of these generic points.
Due to Lemma \ref{A.3.52} if $j\notin I$, then $T_{I\amalg \{j\}}$ is the blow-up of $T_I$ along the preimage $W_{I,j}$ of $W_j$ in $T_I$.
\vspace{0.1in}

We claim that $W_{I,j}$ is smooth if
\[
I\cap \{4,5,6\}=\{4,5\}.
\]
This question is Zariski local on $X_1$, $X_2$, and $X_3$, so we may assume that $\cE_i$ is a trivial vector bundle of rank $p_i$ for $1\leq i\leq 3$.
We set $n=p_1+p_2+p_3$.
Let $e_1,\ldots,e_{n}$ be the standard coordinates in $\Z^{n}$, and let $\Sigma$ be the fan whose only maximal cone is ${\rm Cone}(e_1,\ldots,e_{n})$.
We define the following sets
\begin{gather*}
S_0:=\emptyset,
\;
S_1:=\{e_1,\ldots,e_{p_1}\},
\\
S_2:=\{e_{p_1+1},\ldots,e_{p_1+p_2}\},
\;
S_3:=\{e_{p_1+p_2},\ldots,e_{n}\},
\\
S_4:=S_1\cup S_2,
\;
S_5:=S_2\cup S_3,
\;
S_6:=S_1\cup S_2\cup S_3.
\end{gather*}
We set
\begin{gather*}
f_1:=e_1+\cdots+e_{p_1},
\;
f_2:=e_{p_1+1}+\cdots+e_{p_1+p_2},
\;
f_3:=e_{p_1+p_2+1}+\cdots+e_{n},
\\
f_4:=f_1+f_2,
\;
f_5:=f_2+f_3,
\;
f_6:=f_1+f_2+f_3.
\end{gather*}
For every subset $I=\{i_1,\ldots,i_m\}\subset \{1,\ldots,6\}$ with $i_1<\ldots<i_m$ we set
\[
\Sigma_{I}:=\Sigma_{i_1}\times_{\Sigma}\cdots \times_{\Sigma}\Sigma_{i_m},
\]
i.e.,
\[
\Sigma_{I}=\{\sigma_1\cap \cdots \cap \sigma_m:\sigma_1\in \Sigma_{i_1},\ldots,\sigma_m\in \Sigma_m\}.
\]
By convention $\Sigma_{\emptyset}:=\Sigma$. 
Due to Lemmas \ref{A.3.52} and \ref{A.3.53} there is an isomorphism
\[
T_I\cong \underline{\A_{\Sigma_I}}\times (X_1\times_S X_2\times_S X_3).
\]
Hence it suffices to show that $\Sigma_{I\amalg \{j\}}$ is a star subdivision of $\Sigma_I$ relative to a smooth cone.
This will be done in Lemma \ref{A.3.51} after proving several preliminary results.
\vspace{0.1in}

For $u\in S_1$, $v\in S_2$, $w\in S_3$, $p\in \{0,1,4,6\}$, $q\in \{0,2,4,5,6\}$, and $r\in \{0,3,5,6\}$, we set
\[
\sigma_{uvw}^{pqr}:=
\{
(a_1,\ldots,a_n)\in \N^n:
a_u\leq a_i\text{ if } i\in S_p,
\;
a_v\leq a_i\text{ if } i\in S_q,
\;
a_w\leq a_i\text{ if } i\in S_r
\}.
\]
Note that if $p=0$ (resp.\ $q=0$, resp.\ $r=0$), then the condition $a_u\leq a_i$ (resp.\ $a_v\leq a_i$ resp.\ $a_w\leq a_i$) has no meaning.
For notational convenience, in this case we also write $u=*$ (resp.\ $v=*$, resp.\ $w=*$).
\vspace{0.1in}

If inequalities of the form $a_i\leq a_j$ and $a_j\leq a_i$ enter in the condition for some $i\neq j$ when forming $\sigma_{uvw}^{pqr}$, 
then $\dim \sigma_{uvw}^{pqr}<n$.
By using this criterion, we can list all possible values of $(p,q,r)$ such that $\dim \sigma_{uvw}^{pqr}=n$, see Table \ref{pqrtable}.
\vspace{0.1in}

\begin{table}
\begin{center}
\begin{tabular}{ |c|c|c||c|c|c||c|c|c| }
\hline
$p$&$q$&$r$&$p$&$q$&$r$&$p$&$q$&$r$
\\
\hline
$0$ or $1$ & $0$ or $2$ & $0$ or $3$ & $6$ & $0$ or $2$ & $0$ or $3$ & $4$ & $5$ & $0$ or $3$ 
\\
\hline
$4$ & $0$ or $2$ & $0$ or $3$ & $0$ or $1$ & $6$ & $0$ or $3$ & $4$ & $0$ or $2$ & $5$ 
\\
\hline
$0$ or $1$ & $4$ & $0$ or $3$ & $0$ or $1$ & $0$ or $2$ & $6$ & $0$ or $1$ & $4$ & $5$ 
\\
\hline
$0$ or $1$ & $5$ & $0$ or $3$ & $4$ & $0$ or $2$ & $6$ & & &
\\
\hline
$0$ or $1$ & $0$ or $2$ & $5$ & $6$ & $0$ or $2$ & $5$ & & &
\\
\hline
\end{tabular}
\vspace{0.1in}
\caption{List of all concise $(p,q,r)$.} \label{pqrtable}
\end{center}
\end{table}

We say that $(p,q,r)$ is \emph{concise} if it has a value as listed in the table.
A concise triple $(p,q,r)$ is called \emph{standard} if $\{4,5\}\not\subset \{p,q,r\}$.
For $1\leq i\leq n$ and $1\leq t\leq 6$ to simplify notation we set
\[
f_0^i:=e_i,
\;
f_t^i:=f_t.
\]
If $(p,q,r)$ is standard, then we can check by hand that
\[
\sigma_{uvw}^{pqr}={\rm Cone}(e_1,\ldots,e_{u-1},e_{u+1},\ldots,e_{v-1},e_{v+1},\ldots,e_{w-1},e_{w+1},\ldots,e_n,f_p^u,f_q^v,f_r^w).
\]
We regard $\{0,\ldots,6\}$ as a partially ordered set whose order $\prec$ is generated by
\[
(0\prec 1,2,3),\; (1\prec4,5),\; (2\prec 4,5,6),\; (3 \prec 5,6),\; (4 \prec 6),\; (5\prec 6).
\]
We use the following shorthand notation:
\begin{gather*}
p^s:=\sup\{p,s\}\text{ (if }s\in \{1,4,6\}),
\;
q^s:=\sup\{q,s\}\text{ (if }s\in \{2,4,5,6\}),
\\
r^s:=\sup\{r,s\}\text{ (if }s\in \{3,5,6\}),
\end{gather*}
where $\sup$ is taken with respect to the order $\prec$.
Note that
\begin{equation}
\label{A.3.47.7}
p^s=p\text{ or }s,
\;
q^s=q,\;s,\text{ or }6,
\;
r^s=r\text{ or }s.
\end{equation}
If $p\neq 0$ and $s\in \{1,4,6\}$, then we have
\begin{equation}
\label{A.3.47.2}
\begin{split}
&\sigma_{uvw}^{pqr}\cap \sigma_{u**}^{s00}
\\
=&
\{
(a_1,\ldots,a_n)\in \N^n:
a_u\leq a_i\text{ if } i\in S_p\cup S_s,
\;
a_v\leq a_i\text{ if } i\in S_q,
\;
a_w\leq a_i\text{ if } i\in S_r
\}
\\
=&
\sigma_{uvw}^{p^sqr}.
\end{split}
\end{equation}
We similar have
\begin{equation}
\begin{split}
\label{A.3.47.3}
\sigma_{uvw}^{pqr}\cap \sigma_{*v*}^{0s0}=\sigma_{uvw}^{pq^sr}\text{ (if }q\neq 0),&
\;
\sigma_{uvw}^{pqr}\cap \sigma_{**w}^{00s}=\sigma_{uvw}^{pqr^s}\text{ (if }r\neq 0),
\\
\sigma_{*vw}^{0qr}\cap \sigma_{u**}^{s00}=\sigma_{uvw}^{sqr},
\;
\sigma_{u*w}^{p0r}\cap \sigma_{*v*}^{0s0}&=\sigma_{uvw}^{psr},
\;
\sigma_{uv*}^{pq0}\cap \sigma_{**w}^{00s}=\sigma_{uvw}^{pqs}
\end{split}
\end{equation}
for all appropriate $s$.
Moreover, we have
\begin{equation}
\label{A.3.47.4}
\dim\sigma_{uvw}^{pqr}\cap \sigma_{u'**}^{s00}<n\text{ if }u\neq u',\;p\neq 0,\;s\in \{1,4,6\}
\end{equation}
since the inequalities $a_u\leq a_{u'}$ and $a_{u'}\leq a_u$ are appeared.
Similarly, we have
\begin{equation}
\begin{split}
\label{A.3.47.5}
\dim\sigma_{uvw}^{pqr}\cap \sigma_{*v'*}^{0s0}<n&\text{ if }v\neq v',\;q\neq 0,\;s\in \{2,4,5,6\},
\\
\dim\sigma_{uvw}^{pqr}\cap \sigma_{**w'}^{00s}<n&\text{ if }w\neq w',\;r\neq 0,\;s\in \{3,5,6\}.
\end{split}
\end{equation}
\end{const}

\begin{lem}
\label{A.3.50}
With notations as above,
let $I$ be a subset of $\{1,\ldots,6\}$ such that
\[
I\cap \{4,5,6\}\neq \{4,5\}.
\]
Then any $n$-dimensional  cone of $\Sigma_I$ can be written as $\sigma_{uvw}^{pqr}$ for some standard triple $(p,q,r)$.
\end{lem}
\begin{proof}
Suppose $s\in \{1,\ldots,6\}-I$ and $(I\amalg \{s\})\cap \{4,5,6\}\neq \{4,5\}$.
Any $n$-dimensional cone of $\Sigma_{I\amalg \{s\}}$ is of the form $\sigma\cap \sigma'$ for some $\sigma\in \Sigma_I$ and $\sigma'\in \Sigma_s$.
Thus from \eqref{A.3.47.7}, \eqref{A.3.47.2}, \eqref{A.3.47.3}, \eqref{A.3.47.4}, and \eqref{A.3.47.5}, 
we observe that any $n$-dimensional cone of $\Sigma_{I\amalg\{s\}}$ takes the form $\sigma_{uvw}^{pqr}$ with $p,q,r\in I\amalg \{0,s\}$ if any maximal cone of $\Sigma_I$ 
can be written as $\sigma_{uvw}^{pqr}$ with $p,q,r\in I\amalg\{0\}$.
Now use induction on the number of elements of $I$ to deduce this holds for all $I$ satisfying $I\cap \{4,5,6\}\neq \{4,5\}$.
\vspace{0.1in}

Let $\sigma_{uvw}^{pqr}$ be an $n$-dimensional cone of $\Sigma_I$.
If $(p,q,r)$ is not concise, then $\sigma_{uvw}^{pqr}$ is not an $n$-dimensional cone.
Thus $(p,q,r)$ is concise.
If $(p,q,r)$ is not standard, then $\{4,5\}\in \{p,q,r\}$.
We have shown in the above paragraph that this implies $\{4,5\}\in I$.
From the condition $I\cap \{4,5,6\}\neq \{4,5\}$ we have $6\in I$.
Then there are six cases according to Table \ref{pqrtable}:
\[
\sigma_{uvw}^{pqr}=\sigma_{uvw}^{450},
\;
\sigma_{uvw}^{453},
\;
\sigma_{uvw}^{405},
\;
\sigma_{uvw}^{425},
\;
\sigma_{uvw}^{045},
\;
\sigma_{uvw}^{145}.
\]
In the first and second (resp.\ third and fourth, resp.\ fifth and sixth) cases, $a_u,a_v\notin \sigma_{uvw}^{pqr}$ (resp.\ $a_u,a_w\notin \sigma_{uvw}^{pqr}$, resp.\ $a_v,a_w\notin \sigma_{uvw}^{pqr}$).
Thus in all cases, $\sigma_{uvw}^{pqr}$ is not contained in any of
\[
\sigma_{u'**}^{600},
\;
\sigma_{*v'*}^{060},
\;
\sigma_{**w'}^{060}.
\]
This contradicts the fact that $6\in I$.
Thus $(p,q,r)$ is standard.
\end{proof}

\begin{lem}
\label{A.3.51}
With notations as above,
let $I$ be a subset of $\{1,\ldots,6\}$, and let $s$ be an element of $\{1,\ldots,6\}-I$.
If
\[
\{4,5,6\}\cap I \neq \{4,5\}\text{ and }\{4,5,6\}\cap (I\amalg \{s\})\neq \{4,5\},
\]
then
\[
\Sigma_{I\amalg \{s\}}=\Sigma_I^*(\tau),
\]
where $\tau$ is the cone of $\Sigma_I$ generated by $f_s$.
\end{lem}
\begin{proof}
Let $\sigma_{uvw}^{pqr}$ be a maximal cone of $\Sigma_I$.
Lemma \ref{A.3.50} shows that $(p,q,r)$ is standard.
Suppose we have
\[
\{\sigma_{uvw}^{pqr}\}\times_{\Sigma}\Sigma_j
\neq 
\{\sigma_{uvw}^{pqr}\}.
\]
Due to \eqref{A.3.47.2}, \eqref{A.3.47.3}, \eqref{A.3.47.4}, and \eqref{A.3.47.5}, its $n$-dimensional cones are of the following forms:
\begin{gather*}
\sigma_{uvw}^{sqr}\text{ (if }p\neq 0\text{ and }s\neq 2,3,5),
\;
\sigma_{uvw}^{psr}\text{ (if }q\neq 0\text{ and }s\neq 1,3),
\\
\sigma_{uvw}^{pqs}\text{ (if }q\neq 0\text{ and }s\neq 1,2,4),
\;
\sigma_{1vw}^{sqr},\ldots,\sigma_{p_1vw}^{sqr}\text{ (if }p= 0\text{ and }s\neq 2,3,5),
\\
\sigma_{u(p_1+1)w}^{psr},\ldots,\sigma_{u(p_1+p_2)w}^{psr}\text{ (if }q= 0\text{ and }s\neq 1,3),
\\
\sigma_{uv(p_1+p_2+1)}^{pqs}\ldots,\sigma_{uvn}^{pqs}\text{ (if }r=0\text{ and }s\neq 1,2,4).
\end{gather*}
Note that some of these cones can have dimension $<n$.
\vspace{0.1in}

Suppose $\sigma_{uvw}^{sqr}=\sigma_{uvw}^{pqr}\cap \sigma_{u**}^{s00}$ is a cone in $\{\sigma_{uvw}^{pqr}\}\times_\Sigma \Sigma_j$ and has dimension $n$ with $s\in \{1,4,6\}$ and $p\prec s$.
Lemma \ref{A.3.50} shows that $(s,q,r)$ is standard.
Then
\begin{gather*}
\sigma_{uvw}^{pqr}={\rm Cone}(e_1,\ldots,e_{u-1},e_{u+1},\ldots,e_{v-1},e_{v+1},\ldots,e_{w-1},e_{w+1},\ldots,e_n,f_{p}^u,f_q^v,f_r^w),
\\
\sigma_{uvw}^{sqr}={\rm Cone}(e_1,\ldots,e_{u-1},e_{u+1},\ldots,e_{v-1},e_{v+1},\ldots,e_{w-1},e_{w+1},\ldots,e_n,f_{s},f_q^v,f_r^w).
\end{gather*}
As a particular case, we have
\[
\sigma_{*vw}^{0qr}={\rm Cone}(e_1,\ldots,e_{u-1},e_{u+1},\ldots,e_{v-1},e_{v+1},\ldots,e_{w-1},e_{w+1},\ldots,e_n,e_u,f_q^v,f_r^w).
\]
Since $\sigma_{uvw}^{pqr}$ contains $\sigma_{uvw}^{sqr}$, we have $f_s\in \sigma_{uvw}^{pqr}$.
Let $\tau$ be the subcone of $\sigma_{uvw}^{pqr}$ generated by $f_s$, and
let $\{d_1,\ldots,d_l\}$ be the generator of $\tau$.
Since $f_s\in \tau$, it follows that $\{f_s,d_1,\ldots,d_l\}$ is linearly independent.
It follows that if $f_p^u\notin \tau$, then $\dim \sigma_{uvw}^{sqr}<n$.
This contradicts the fact that $(s,q,r)$ is standard.
In conclusion, we have:
\begin{enumerate}
\item[(Rule 1)] To obtain $\sigma_{uvw}^{sqr}$ from $\sigma_{uvw}^{pqr}$ replace $f_p^u$ by $f_{s}$.
To obtain $\sigma_{uvw}^{sqr}$ from $\sigma_{*vw}^{0qr}$ replace $e_u$ by $f_s$.
\item[(Rule 2)] $f_p^u$ should be in $\tau$. 
\end{enumerate}
We have similar rules for $\sigma_{uvw}^{psr}$ and $\sigma_{uvw}^{pqs}$ also under similar assumptions.
\vspace{0.1in}

If ${\rm Cone}(\epsilon_1,\ldots,\epsilon_t)$
is an arbitrary cone with the star subdivision $\Gamma$ relative to a subcone ${\rm Cone}(\epsilon_1,\ldots,\epsilon_{t'})$, then we have:
\begin{enumerate}
\item[(Rule 1')] To obtain a maximal cone of $\Gamma$ replace $\epsilon_i$ by $\epsilon_1+\cdots+\epsilon_{t'}$.
\item[(Rule 2')] $\epsilon_i$ should be in ${\rm Cone}(\epsilon_1,\ldots,\epsilon_{t'})$.
\end{enumerate}
Compare (Rule 1) and (Rule 2) with (Rule 1') and (Rule 2') to deduce that
\begin{equation}
\label{A.3.51.1}
\{\sigma_{uvw}^{pqr}\}\times_{\Sigma}\Sigma_s = \{\sigma_{uvw}^{pqr}\}^*(\tau),
\end{equation}
which finishes the proof.
\end{proof}

\begin{exm}
Let us illustrate \eqref{A.3.51.1} with an example.
If $(p,q,r,s)=(4,2,0,6)$, then
\begin{gather*}
\sigma_{uv*}^{420}={\rm Cone}(e_1,\ldots,e_{u-1},e_{u+1},\ldots,e_{v-1},e_{v+1},\ldots,e_{w-1},e_{w+1},\ldots,e_n,f_{4},f_2,e_w),
\\
\sigma_{uv*}^{620}={\rm Cone}(e_1,\ldots,e_{u-1},e_{u+1},\ldots,e_{v-1},e_{v+1},\ldots,e_{w-1},e_{w+1},\ldots,e_n,f_6,f_2,e_w),
\\
\sigma_{uv*}^{460}={\rm Cone}(e_1,\ldots,e_{u-1},e_{u+1},\ldots,e_{v-1},e_{v+1},\ldots,e_{w-1},e_{w+1},\ldots,e_n,f_{4},f_6,e_w),
\\
\sigma_{uvw}^{426}={\rm Cone}(e_1,\ldots,e_{u-1},e_{u+1},\ldots,e_{v-1},e_{v+1},\ldots,e_{w-1},e_{w+1},\ldots,e_n,f_{4},f_2,f_6).
\end{gather*}
The cone $\sigma_{uv*}^{460}$ has dimension $<n$, and
\[
\sigma_{uv*}^{620},\sigma_{uv(p_1+p_2+1)}^{426},\ldots,\sigma_{uvn}^{426}
\]
are precisely all the maximal cones of $\{\sigma_{uv*}^{420}\}\times_{\Sigma}\Sigma_6$.
In addition, 
\[
\tau={\rm Cone}(f_4,e_{p_1+p_2+1},\ldots,e_n)
\]
is the subcone of $\sigma_{uv*}^{420}$ generated by $f_6$, and
\[
\{\sigma_{uvw}^{420}\}\times_{\Sigma}\Sigma_6 = \{\sigma_{uvw}^{420}\}^*(\tau).
\]
\end{exm}

\begin{const}
\label{A.3.49}
For $X_1,X_2,X_3\in SmlSm/S$ and vector bundles $\cE_i\rightarrow X_i$ of fixed rank for $i=1,2,3$, 
we may apply Construction \ref{A.3.47} to the vector bundles $\ul{\cE_1}\rightarrow \ul{X_1}$, $\ul{\cE_2}\rightarrow \ul{X_2}$, and $\ul{\cE_3}\rightarrow \ul{X_3}$.
In this way, we obtain the blow-up $\ul{T_I}$ of
\[
\ul{T}:=\ul{\cE_1}\times_S \ul{\cE_2}\times_S \ul{\cE_3}
\]
for every subset $I$ of $\{1,\ldots,6\}$ such that
\[
I\cap\{4,5,6\}\neq \{4,5\}.
\]
For $i=1,2,3$ let $Z_i$ be the $0$-section of $\cE_i$, and set
\begin{gather*}
U:=(\cE_1-\partial \cE_1)\times_S (\cE_2-\partial \cE_2)\times_S (\cE_3-\partial \cE_3),
\;
U_1:=U-Z_1\times_S \cE_2\times_S \cE_3,
\\
U_2:=U-\cE_1\times_S Z_2\times_S \cE_3,
\;
U_3:=U-\cE_1\times_S \cE_2\times_S Z_3,
\\
U_4:=U-Z_1\times_S Z_2\times_S \cE_3,
\;
U_5:=U-\cE_1\times_S Z_2\times_S Z_3,
\;
U_6:=U-Z_1\times_S Z_2\times_S Z_3.
\end{gather*}

If $I=\{i_1,\ldots,i_n\}$, then the open immersion
\[
j_I\colon U_I:=U_{i_1}\cap \cdots \cap U_{i_n}\rightarrow \underline{T}
\]
can be lifted to $\ul{T_I}$ since the images of the exceptional divisors on $\ul{T_I}$ in $\ul{T}$ lie in the complement of $j_I$.
Let $T_I$ denote the fs log scheme whose underlying scheme is $\ul{T_I}$ and whose log structure is the compactifying log structure associated with $j_I\colon U_I\rightarrow \underline{T}$.
If $I=\{i_1,\ldots,i_n\}$, for simplicity of notation, we set
\[
T_{i_1\ldots i_n}:=T_I
\]
In particular, we set $T:=T_\emptyset$.
\vspace{0.1in}

We can divide all the $T_I$'s into the following groups such that if $T_I$ and $T_J$ belong to the same group, then $U_I=U_J$.
\begin{equation}
\begin{gathered}
\label{A.3.49.1}
\{T\},
\;
\{T_1, T_{14}, T_{16}, T_{146}\},
\;
\{T_2, T_{24}, T_{25}, T_{26}, T_{246}, T_{256}, T_{2456}\},
\\
\{T_3, T_{35}, T_{36}, T_{356}\},
\;
\{T_{4}, T_{46}\},
\;
\{T_5, T_{56}\},
\;
\{T_6\},
\\
\{T_{12}, T_{124}, T_{125}, T_{126}, T_{1246}, T_{1256}, T_{12456}\},
\\
\{T_{13}, T_{134}, T_{135}, T_{136}, T_{1346}, T_{1356}, T_{13456}\},
\\
\{T_{23}, T_{234}, T_{235}, T_{236}, T_{2346}, T_{2356}, T_{23456}\},
\\
\{T_{15}, T_{156}, T_{1456}\},
\;
\{T_{34}, T_{346}, T_{3456}\},
\\
\{T_{123}, T_{1234}, T_{1235}, T_{1236}, T_{12346}, T_{12356}, T_{123456}\},
\;
\{T_{456}\}.
\end{gathered}
\end{equation}
If $U_I=U_J$ and $\lvert J\rvert=\lvert I\rvert +1$, then $T_J$ is an admissible blow-up of $T_I$ along a smooth center since we have shown that $W_{I,j}$ in Construction \ref{A.3.47} is smooth.
\end{const}

We can now turn to prove compatibility between Thom motives and products.  
As can be expected, since blow-ups are involved in our formulation, this requires some extra arguments than in the $\A^1$-setting.

\begin{prop}
\label{A.3.42}
For $X_1,X_2\in SmlSm/S$ and vector bundles $\cE_1\rightarrow X_1$ and $\cE_2\rightarrow X_2$, there exists a canonical isomorphism
\begin{equation}
\label{A.3.42.10}
MTh_{X_1}(\cE_1)\otimes MTh_{X_2}(\cE_2)\xrightarrow{\cong} MTh_{X_1\times_S X_2}(\cE_1\times_S \cE_2).
\end{equation}
\end{prop}
\begin{proof}
\textbf{Step 1} Apply Construction \ref{A.3.49} to the vector bundles $\cE_1\rightarrow X_1$, $\cE_2\rightarrow X_2$, and $S\rightarrow S$ 
to obtain an fs log scheme $T_I\in SmlSm/S$ for every subset $I$ of $\{1,2,4\}$.
According to Construction \ref{A.3.49}, the morphisms
\[
T_{124}\rightarrow T_{12},
\;
T_{14}\rightarrow T_1,
\;
T_{24}\rightarrow T_2
\]
are admissible blow-ups along smooth centers.
Theorem \ref{A.3.7} shows the naturally induced morphisms
\begin{equation}
\label{A.3.42.2}
M(T_{124})\rightarrow M(T_{12}),
\;
M(T_{14})\rightarrow M(T_1),
\;
M(T_{24})\rightarrow M(T_2)
\end{equation}
are isomorphisms.
\vspace{0.1in}

\textbf{Step 2}
Form the naturally induced $3$-squares
\[
Q:=
\begin{tikzcd}[row sep=tiny, column sep=tiny]
&T_{124}\arrow[dd]\arrow[rr]\arrow[ld]&&T_{14}\arrow[dd]\arrow[ld]\\
T_{24}\arrow[dd]\arrow[rr,crossing over]&&T\\
&T_{4}\arrow[rr]\arrow[ld]&&T_{4}\arrow[ld]\\
T_{4}\arrow[rr]&&T\arrow[uu,crossing over,leftarrow]
\end{tikzcd}
\quad
Q':=
\begin{tikzcd}[row sep=tiny, column sep=tiny]
&T_{124}\arrow[dd]\arrow[rr]\arrow[ld]&&T_{14}\arrow[dd]\arrow[ld]\\
T_{24}\arrow[dd]\arrow[rr,crossing over]&&T\\
&T_{24}\arrow[rr]\arrow[ld]&&T_{4}\arrow[ld]\\
T_{24}\arrow[rr]&&T\arrow[uu,crossing over,leftarrow]
\end{tikzcd}
\]
with $C_1:=\;$the bottom square of $Q$, $C_2:=\;$the top square of $Q$, $D_1:=\;$the front square of $Q$, $D_2:=\;$the back square of $Q$, 
$C_1':=\;$the bottom square of $Q'$, $C_2':=\;$the top square of $Q"$, $D_1':=\;$the front square of $Q'$, and $D_2':=\;$the back square of $Q'$.
By (Sq), there exist canonical distinguished triangles
\begin{gather}
\label{A.3.42.17}
M(D_2')\rightarrow M(D_1')\rightarrow M(Q') \rightarrow M(D_2')[1],
\\
\label{A.3.42.18}
M(T_{24}\rightarrow T_{24})\rightarrow M(T\rightarrow T)\rightarrow M(D_1')\rightarrow M(T_{24}\rightarrow T_{24})[1].
\end{gather}
Due to Proposition \ref{A.3.45} and \eqref{A.3.42.18} we have $M(D_1')=0$.
We will show in Step 3 that $M(D_2')=0$.
Assuming this, from \eqref{A.3.42.17}, we have 
$$
M(Q')=0.
$$
There exists a morphism of distinguished triangles due to (Sq)
\begin{equation}
\label{A.3.42.19}
\begin{tikzcd}
M(C_2')\arrow[d]\arrow[r]&
M(C_1')\arrow[r]\arrow[d]&
M(Q')\arrow[d]\arrow[r]&
M(C_2')[1]\arrow[d]
\\
M(C_2)\arrow[r]&
M(C_1)\arrow[r]&
M(Q)\arrow[r]&
M(C_2)[1].
\end{tikzcd}
\end{equation}
The first and fourth vertical morphisms are the identity, and the second vertical morphism is an isomorphism since we also have a morphism of distinguished triangles
\[
\begin{tikzcd}
M(T_{24}\rightarrow T_{24})\arrow[d]\arrow[r]&
M(T_4\rightarrow T)\arrow[r]\arrow[d]&
M(C_1')\arrow[d]\arrow[r]&
M(T_{24}\rightarrow T_{24})[1]\arrow[d]
\\
M(T_4\rightarrow T_4)\arrow[r]&
M(T_4\rightarrow T)\arrow[r]&
M(C_1)\arrow[r]&
M(T_4\rightarrow T_4)[1].
\end{tikzcd}
\]
Thus the third vertical morphism in \eqref{A.3.42.19} is an isomorphism.
It follows that 
$$
M(Q)=0.
$$
Using (Sq) we deduce that the naturally induced morphism
\begin{equation}
\label{A.3.42.21}
M(C_2)\rightarrow M(C_1)
\end{equation}
is an isomorphism.
\vspace{0.1in}

Let $C_3$ be the square
\[
\begin{tikzcd}
T_{12}\arrow[d]\arrow[r]&T_1\arrow[d]
\\
T_2\arrow[r]&T.
\end{tikzcd}
\]
Since the morphisms in \eqref{A.3.42.2} are isomorphism, by Proposition \ref{A.3.48}, the naturally induced morphism
\begin{equation}
\label{A.3.42.20}
M(C_2)\rightarrow M(C_3)
\end{equation}
is an isomorphism.
By (MSq) there is a canonical isomorphism
\begin{equation}
\label{A.3.42.9}
MTh_{X_1}(\cE_1)\otimes MTh_{X_2}(\cE_2)\cong M(C_3).
\end{equation}
Combine \eqref{A.3.42.21}, \eqref{A.3.42.20}, and \eqref{A.3.42.9} to obtain \eqref{A.3.42.10}.
\vspace{0.1in}

\textbf{Step 3} It remains to verify that $M(D_1')=0$, which is equivalent to showing the naturally induced morphism
\[
M(T_{124}\rightarrow T_{24})\rightarrow M(T_{14}\rightarrow T_4)
\]
is an isomorphism due to (Sq).
By ($Zar$-sep), this question is Zariski local on $X_1$ and $X_2$.
Hence we may assume that $\cE_1$ and $\cE_2$ are trivial, say of rank $p_1$ and $p_2$.
In this case, we can further reduce to the case when $X_1=X_2=S$ by applying (MSq).
\vspace{0.1in}

For simplicity of notation we suppose that $S=\Spec{\Z}$.
Recall the notations $f_1$, $f_2$, and $\Sigma_I$ for $I\subset \{1,2,4\}$ in Construction \ref{A.3.47}.
We set
\begin{gather*}
\Sigma_4':=\{{\rm Cone}(f_1+f_2)\},
\;
\Sigma_{14}':=\{{\rm Cone}(f_1,f_1+f_2)\},
\;
\Sigma_{24}':=\{{\rm Cone}(f_2,f_1+f_2)\},
\\
\Sigma_{124}':=\{{\rm Cone}(f_1,f_1+f_2),{\rm Cone}(f_2,f_1+f_2)\}.
\end{gather*}
With these definitions we have 
\[
T_{124}=\A_{(\Sigma_{124},\Sigma_{124}')},
\;
T_{14}=\A_{(\Sigma_{14},\Sigma_{14}')},
\;
T_{24}=\A_{(\Sigma_{24},\Sigma_{24}')},
\;
T_4=\A_{(\Sigma_4,\Sigma_3')}.
\]
\vspace{0.1in}

When $p=2$ and $q=1$, we can describe the fans as follows.
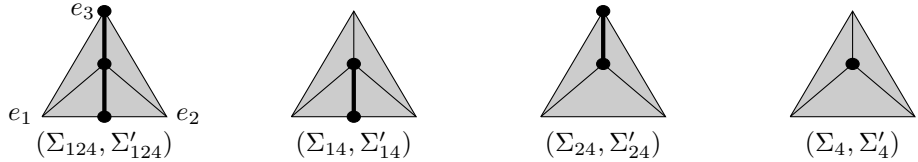
\[
\begin{tikzpicture}[yscale=0.7, xscale=0.82]
\fill[black!20] (0,0)--(1,2)--(2,0)--(0,0);
\draw (0,0)--(1,2)--(2,0)--(0,0);
\draw (0,0)--(1,1)--(2,0);
\draw[ultra thick] (1,0)--(1,1)--(1,2);
\filldraw (1,0) circle (3pt);
\filldraw (1,1) circle (3pt);
\filldraw (1,2) circle (3pt);
\node [left] at (0,0) {$e_1$};
\node [right] at (2,0) {$e_2$};
\node [left] at (1,2) {$e_3$};
\node [below] at (1,-0.1) {$(\Sigma_{124},\Sigma_{124}')$};
\begin{scope}[shift={(4,0)}]
\fill[black!20] (0,0)--(1,2)--(2,0)--(0,0);
\draw (0,0)--(1,2)--(2,0)--(0,0);
\draw (0,0)--(1,1)--(2,0);
\draw[ultra thick] (1,0)--(1,1);
\draw (1,1)--(1,2);
\filldraw (1,0) circle (3pt);
\filldraw (1,1) circle (3pt);
\node [below] at (1,-0.1) {$(\Sigma_{14},\Sigma_{14}')$};
\end{scope}
\begin{scope}[shift={(8,0)}]
\fill[black!20] (0,0)--(1,2)--(2,0)--(0,0);
\draw (0,0)--(1,2)--(2,0)--(0,0);
\draw (0,0)--(1,1)--(2,0);
\draw[ultra thick] (1,1)--(1,2);
\filldraw (1,1) circle (3pt);
\filldraw (1,2) circle (3pt);
\node [below] at (1,-0.1) {$(\Sigma_{24},\Sigma_{24}')$};
\end{scope}
\begin{scope}[shift={(12,0)}]
\fill[black!20] (0,0)--(1,2)--(2,0)--(0,0);
\draw (0,0)--(1,2)--(2,0)--(0,0);
\draw (0,0)--(1,1)--(2,0);
\draw (1,1)--(1,2);
\filldraw (1,1) circle (3pt);
\node [below] at (1,-0.1) {$(\Sigma_{4},\Sigma_{4}')$};
\end{scope}
\end{tikzpicture}
\]

For any cone $\sigma$, we write $\{\sigma\}$ for the fan whose only maximal cone is $\sigma$.
Due to ($Zar$-sep) it suffices to check that the naturally induced morphism
\begin{equation}
\label{A.3.42.11}
M(T_{124}\times_{\underline{\A_{\Sigma_4}}}\underline{\A_{\{\sigma\}}}
\rightarrow 
T_{24}\times_{\underline{\A_{\Sigma_4}}}\underline{\A_{\{\sigma\}}})
\rightarrow
M(T_{14}\times_{\underline{\A_{\Sigma_4}}}\underline{\A_{\{\sigma\}}}
\rightarrow
T_4\times_{\underline{\A_{\Sigma_4}}}\underline{\A_{\{\sigma\}}})
\end{equation}
is an isomorphism for every cone $\sigma$ of $\Sigma_4$.
Equivalently, owing to Proposition \ref{A.3.44}, it suffices to check that the naturally induced morphism
\begin{equation}
\label{A.3.42.12}
M(T_{124}\times_{\underline{\A_{\Sigma_4}}}\underline{\A_{\{\sigma\}}}
\rightarrow 
T_{14}\times_{\underline{\A_{\Sigma_4}}}\underline{\A_{\{\sigma\}}})
\rightarrow
M(T_{24}\times_{\underline{\A_{\Sigma_4}}}\underline{\A_{\{\sigma\}}}
\rightarrow
T_4\times_{\underline{\A_{\Sigma_4}}}\underline{\A_{\{\sigma\}}})
\end{equation}
is an isomorphism.
If
\[
\sigma\subset {\rm Cone}(e_1,\ldots,e_{i-1},e_{i+1},\ldots,e_{p_1+p_2},f_1+f_2)
\]
for some $1\leq i\leq p_1$, then the naturally induced morphisms
\[
T_{124}\times_{\underline{\A_{\Sigma_4}}}\underline{\A_{\{\sigma\}}}
\rightarrow 
T_{24}\times_{\underline{\A_{\Sigma_4}}}\underline{\A_{\{\sigma\}}},
\;
T_{14}\times_{\underline{\A_{\Sigma_4}}}\underline{\A_{\{\sigma\}}}
\rightarrow
T_4\times_{\underline{\A_{\Sigma_4}}}\underline{\A_{\{\sigma\}}}
\]
are isomorphisms.
Thus the morphism \eqref{A.3.42.11} is also an isomorphism.
On the other hand, if
\[
\sigma\subset {\rm Cone}(e_1,\ldots,e_{p_1+j-1},e_{p_1+j+1},\ldots,e_{p_1+p_2},f_1+f_2)
\]
for some $1\leq j\leq p_2$, then the naturally induced morphisms
\[
T_{124}\times_{\underline{\A_{\Sigma_4}}}\underline{\A_{\{\sigma\}}}
\rightarrow 
T_{14}\times_{\underline{\A_{\Sigma_4}}}\underline{\A_{\{\sigma\}}},
\;
T_{24}\times_{\underline{\A_{\Sigma_4}}}\underline{\A_{\{\sigma\}}}
\rightarrow
T_4\times_{\underline{\A_{\Sigma_4}}}\underline{\A_{\{\sigma\}}}
\]
are isomorphisms.
Thus the morphism \eqref{A.3.42.12} is also an isomorphism.
\end{proof}

\begin{prop}
\label{A.3.46}
For $X_1,X_2,X_3\in SmlSm/S$ and vector bundles $\cE_i\rightarrow X_i$ for $i=1,2,3$, 
there are naturally induced isomorphisms rendering the diagram 
\begin{equation}
\label{A.3.46.1}
\begin{tikzcd}[column sep=tiny]
MTh_{X_1}(\cE_1)\otimes MTh_{X_2}(\cE_1)\otimes MTh_{X_3}(\cE_1)\arrow[d, "\simeq"']\arrow[r, "\simeq"]&
MTh_{X_1}(\cE_1)\otimes MTh_{X_2\times_S X_3}(\cE_2\times_S \cE_3)\arrow[d, "\simeq"]
\\
MTh_{X_1\times_S X_2}(\cE_1\times_S \cE_2)\otimes MTh_{X_3}(\cE_3)\arrow[r, "\simeq"]&
MTh_{X_1\times_S X_2\times_S X_3}(\cE_1\times_S \cE_2\times_S \cE_3)
\end{tikzcd}
\end{equation}
commutative.
\end{prop}
\begin{proof}
We note the isomorphisms are obtained from Proposition \ref{A.3.42}.
Applying Construction \ref{A.3.49} to the vector bundles $\cE_1\rightarrow X_1$, $\cE_2\rightarrow X_2$, and $\cE_3\rightarrow X_3$ we obtain an fs log scheme 
$T_I\in SmlSm/S$ for every subset $I$ of $\{1,\ldots,6\}$ such that
\[
I\cap \{4,5,6\}\neq \{4,5\}.
\]
For $i=1,2,3$, 
let $Z_i$ be the $0$-section of $\cE_i$, and let $E_i$ be the exceptional divisor on the blow-up $B_{Z_i}(\cE_i)$.
All of the blow-ups $\ul{T_I}$ contains
\[
(\cE_1-\partial \cE_1)\times_S (\cE_2-\partial \cE_2)\times_S (\cE_3-\partial \cE_3)
\]
as an open subscheme.
Let $T_I$ denote the fs log scheme whose underlying scheme is $\ul{T_I}$ and whose log structure is the compactifying log structure associated with the above open immersion.
\vspace{0.1in}

Form the $3$-squares
\[
Q_1:=
\begin{tikzcd}[row sep=tiny, column sep=tiny]
&T_{123}\arrow[dd]\arrow[rr]\arrow[ld]&&T_{13}\arrow[dd]\arrow[ld]\\
T_{23}\arrow[dd]\arrow[rr,crossing over]&&T_3\\
&T_{12}\arrow[rr]\arrow[ld]&&T_1\arrow[ld]\\
T_2\arrow[rr]&&T\arrow[uu,crossing over,leftarrow]
\end{tikzcd}
\quad
Q_2:=
\begin{tikzcd}[row sep=tiny, column sep=tiny]
&T_{1234}\arrow[dd]\arrow[rr]\arrow[ld]&&T_{134}\arrow[dd]\arrow[ld]\\
T_{234}\arrow[dd]\arrow[rr,crossing over]&&T_{3}\\
&T_{124}\arrow[rr]\arrow[ld]&&T_{14}\arrow[ld]\\
T_{24}\arrow[rr]&&T\arrow[uu,crossing over,leftarrow]
\end{tikzcd}
\]
\[
Q_3:=
\begin{tikzcd}[row sep=tiny, column sep=tiny]
&T_{34}\arrow[dd]\arrow[rr]\arrow[ld]&&T_{34}\arrow[dd]\arrow[ld]\\
T_{34}\arrow[dd]\arrow[rr,crossing over]&&T_{3}\\
&T_{4}\arrow[rr]\arrow[ld]&&T_4\arrow[ld]\\
T_4\arrow[rr]&&T\arrow[uu,crossing over,leftarrow]
\end{tikzcd}
\quad
Q_2':=
\begin{tikzcd}[row sep=tiny, column sep=tiny]
&T_{1235}\arrow[dd]\arrow[rr]\arrow[ld]&&T_{135}\arrow[dd]\arrow[ld]\\
T_{235}\arrow[dd]\arrow[rr,crossing over]&&T_{35}\\
&T_{125}\arrow[rr]\arrow[ld]&&T_{1}\arrow[ld]\\
T_{25}\arrow[rr]&&T\arrow[uu,crossing over,leftarrow]
\end{tikzcd}
\]
\[
Q_3':=
\begin{tikzcd}[row sep=tiny, column sep=tiny]
&T_{15}\arrow[dd]\arrow[rr]\arrow[ld]&&T_{15}\arrow[dd]\arrow[ld]\\
T_{5}\arrow[dd]\arrow[rr,crossing over]&&T_{5}\\
&T_{15}\arrow[rr]\arrow[ld]&&T_{1}\arrow[ld]\\
T_{5}\arrow[rr]&&T\arrow[uu,crossing over,leftarrow]
\end{tikzcd}
\quad
Q:=
\begin{tikzcd}[row sep=tiny, column sep=tiny]
&T_{123456}\arrow[dd]\arrow[rr]\arrow[ld]&&T_{13456}\arrow[dd]\arrow[ld]\\
T_{23456}\arrow[dd]\arrow[rr,crossing over]&&T_{356}\\
&T_{12456}\arrow[rr]\arrow[ld]&&T_{1456}\arrow[ld]\\
T_{2456}\arrow[rr]&&T\arrow[uu,crossing over,leftarrow]
\end{tikzcd}
\]
\[
Q':=
\begin{tikzcd}[row sep=tiny, column sep=tiny]
&T_{6}\arrow[dd]\arrow[rr]\arrow[ld]&&T_{6}\arrow[dd]\arrow[ld]\\
T_{6}\arrow[dd]\arrow[rr,crossing over]&&T_{6}\\
&T_{6}\arrow[rr]\arrow[ld]&&T_{6}\arrow[ld]\\
T_{6}\arrow[rr]&&T.\arrow[uu,crossing over,leftarrow]
\end{tikzcd}
\]
Then form the $2$-squares
\[
C_1:=
\begin{tikzcd}[row sep=tiny, column sep=tiny]
&
T_{34}\arrow[ld]\arrow[dd]
\\
T_3\arrow[dd]
\\
&
T_4\arrow[ld]
\\
T
\end{tikzcd}
\quad
C_2:=
\begin{tikzcd}[row sep=tiny, column sep=tiny]
&
T_{346}\arrow[ld]\arrow[dd]
\\
T_{36}\arrow[dd]
\\
&
T_{46}\arrow[ld]
\\
T
\end{tikzcd}
\quad
C_3:=
\begin{tikzcd}[row sep=tiny, column sep=tiny]
&
T_{6}\arrow[ld]\arrow[dd]
\\
T_6\arrow[dd]
\\
&
T_6\arrow[ld]
\\
T
\end{tikzcd}
\quad
C:=
\begin{tikzcd}[row sep=tiny, column sep=tiny]
&
T_{13456}\arrow[ld]\arrow[dd]
\\
T_{356}\arrow[dd]
\\
&
T_{1456}\arrow[ld]
\\
T
\end{tikzcd}
\]
\[
C_1':=
\begin{tikzcd}[row sep=tiny, column sep=tiny]
&T_{15}\arrow[ld]\arrow[rr]&&T_1\arrow[ld]
\\
T_5\arrow[rr]&&T
\end{tikzcd}
\quad
C_2':=
\begin{tikzcd}[row sep=tiny, column sep=tiny]
&T_{156}\arrow[ld]\arrow[rr]&&T_{16}\arrow[ld]
\\
T_{56}\arrow[rr]&&T
\end{tikzcd}
\]
\[
C_3':=
\begin{tikzcd}[row sep=tiny, column sep=tiny]
&T_{6}\arrow[ld]\arrow[rr]&&T_{6}\arrow[ld]
\\
T_{6}\arrow[rr]&&T
\end{tikzcd}
\quad
C':=
\begin{tikzcd}[row sep=tiny, column sep=tiny]
&T_{12456}\arrow[ld]\arrow[rr]&&T_{1456}\arrow[ld]
\\
T_{2456}\arrow[rr]&&T.
\end{tikzcd}
\]

From the above, there is a naturally induced commutative diagram
\[
\begin{tikzcd}
M(Q_1)\arrow[r,leftarrow,"(a)"]\arrow[d,leftarrow,"(a)"']&
M(Q_2)\arrow[d,leftarrow,"(c)"]\arrow[r,"(b)"]&
M(Q_3)\arrow[r,leftarrow,"(a)"]&
M(C_1)\arrow[d,leftarrow,"(a)"]
\\
M(Q_2')\arrow[d,"(b)"']\arrow[r,leftarrow,"(c)"]&
M(Q)\arrow[d,leftarrow,"(e)"]\arrow[r,leftarrow,"(e)"]\arrow[rd]&
M(C)\arrow[r,"(d)"]&
M(C_2)\arrow[d,"(b)"]
\\
M(Q_2')\arrow[d,leftarrow,"(a)"']&
M(C')\arrow[d,"(d)"]&
M(Q')\arrow[d,leftarrow,"(f)"]\arrow[r,leftarrow,"(f)"]&
M(C_3)\arrow[d,leftarrow,"(a)"]
\\
M(C_1')\arrow[r,leftarrow,"(a)"]&
M(C_2')\arrow[r,"(b)"]&
M(C_3')\arrow[r,leftarrow,"(a)"]&
M(T_6\rightarrow T).
\end{tikzcd}
\]

The morphisms labeled $(a)$ and $(b)$ are isomorphisms, see the proof of Proposition \ref{A.3.42}.
Moreover, 
after inverting all the morphisms with label $(b)$, the outer diagram can be identified with \eqref{A.3.46.1} using (MSq).
The morphisms with labels $(c)$ and $(d)$ are isomorphisms by Proposition \ref{A.3.48}, Theorem \ref{A.3.7}, 
and the discussion about admissible blow-ups along smooth centers \eqref{A.3.49.1}.
It follows that the morphisms labeled $(e)$ are also isomorphisms.
\vspace{0.1in}

Ley $C''$ denote the square
\[
C'':=
\begin{tikzcd}
T_6\arrow[d]\arrow[r]&
T_6\arrow[d]
\\
T_6\arrow[r]&
T_6.
\end{tikzcd}
\]
By (Sq) there are distinguished triangles
\begin{gather*}
M(C'')\rightarrow M(C_3)\rightarrow M(Q')\rightarrow M(C'')[1],
\;
M(C'')\rightarrow M(C_3')\rightarrow M(Q')\rightarrow M(C'')[1],
\\
M(T_6\rightarrow T_6)\rightarrow M(T_6\rightarrow T_6)\rightarrow M(C'')\rightarrow M(T_6\rightarrow T_6)[1].
\end{gather*}
Since $M(T_6\rightarrow T_6)=0$ by Proposition \ref{A.3.45}, it follows that $M(C'')=0$.
Thus the morphisms labeled $(f)$ are isomorphisms.
The diagram remains commutative after inverting all morphisms with labels $(a)$, $(c)$, $(e)$, and $(f)$.
This means that also the diagram \eqref{A.3.46.1} commutes.
\end{proof}

\begin{prop}
\label{A.3.43}
For every $X\in SmlSm/S$ and integer $n\geq 1$, there is a canonical isomorphism
\[
MTh_X(X\times \A^n)\cong M(X)(n)[2n].
\]
\end{prop}
\begin{proof}
Owing to Proposition \ref{A.3.34} there are isomorphisms
\[
MTh_X(X\times \A^1)\cong M(X)(1)[2],
\;
MTh_S(S\times \A^1)\cong M(S)(1)[2].
\]
Then apply Proposition \ref{A.3.42} to obtain an isomorphism
\[
MTh_X(X\times \A^r)\otimes MTh_S(S\times \A^1)\cong MTh_X(X\times \A^{r+1})
\]
for every integer $r\geq 1$.
Induction on $n$ concludes the proof.
\end{proof}

\begin{rmk}\label{rmk:why-3-bundles-necessary}
We can also produce an isomorphism 
$$
MTh_X(X\times \A^n)\xrightarrow{\cong} M(X)(n)[2n]
$$ 
by applying Proposition \ref{A.3.42} in a different order, e.g., as in 
\begin{align*}
& MTh(X\times \A^1)\otimes( MTh_S(S\times \A^1)\otimes MTh_S(S\times \A^1) )
\\
\rightarrow &
MTh(X\times \A^1)\otimes MTh_S(S\times \A^2)
\rightarrow
MTh_X(X\times \A^3).
\end{align*}
Proposition \ref{A.3.46} ensures that the isomorphism $MTh_X(X\times \A^n)\cong M(X)(n)[2n]$ is indeed independent of choosing such an order.
\end{rmk}

\begin{cor}
For every $X\in SmlSm/S$ and integer $n\geq 1$, there is a canonical isomorphism
\[
M(\P_X^{n-1}\rightarrow \P_X^{n})\cong M(X)(n)[2n].
\]
\end{cor}
\begin{proof}
By (MSq) there is a canonical isomorphism
\[
M(X)\otimes M(\P_S^{n-1}\rightarrow \P_S^n)\xrightarrow{\cong} M(\P_X^{n-1}\rightarrow \P_X^n).
\]
Thus we may reduce to the case when $X=S$, and conclude by applying Propositions \ref{A.3.34} and \ref{A.3.43}.
\end{proof}

\begin{rmk}
This result is analogous to the isomorphism 
\[
\P^n/\P^{n-1}\simeq T^n
\]
in $\A^1$-homotopy theory of schemes \cite[Corollary 2.18, p.\ 112]{MV}.
\end{rmk}

\subsection{Gysin triangles}\label{sec:Gysin}
Throughout this section we shall assume $\cT$ satisfies the properties {\rm ($Zar$-sep)}, {\rm ($\boxx$-inv)}, {\rm ($sNis$-des)}, and {\rm ($div$-des)}.
In $\dmeff$, 
recall from \cite[Theorem 15.15]{MVW} that for any smooth scheme $X$ with a smooth closed subscheme $Z$ with codimension $c$, 
there is the Gysin isomorphism 
\[
M((X-Z)\rightarrow X)\cong M(Z)(c)[2c].
\]
This does not hold in $\ldmeff$.
For example, if $X=\P^1$ and $Z=\{\infty\}$,
then the said Gysin isomorphism would imply 
$$
M(\A^1\rightarrow \P^1)\cong \Lambda(1)[2].
$$
It follows that $M(\A^1)\cong \Lambda$, so that $\A^1$ is invertible in $\ldmeff$.
This is evidently a contradiction since the object $\mathbb{G}_a\otimes \Lambda$ is in $\ldmeff$, 
but
\[
\hom_{\ldmeff}(M(\A^1),\mathbb{G}_a\otimes \Lambda)\not\cong \hom_{\ldmeff}(M(k),\mathbb{G}_a\otimes \Lambda),
\]
see Theorem \ref{thmHodge}.
\vspace{0.1in}

Our goal in this section is to formulate a Gysin isomorphism in $\ldmeff$. 
As for the construction of Thom motives performed in the previous section,  
we need to ``compactify'' $X-Z$ without changing the ``homotopy type'' of $X-Z$.
This is done by taking the blow-up $B_Z X$ of $X$ along $Z$ and adding the log structure given by the exceptional divisor $E$ on $B_Z X$.
In this case, our formulation of the Gysin isomorphism should take the form
\[
\mathfrak{g}_{X,Z}:M((B_Z X,E)\rightarrow X)\stackrel{\cong}\rightarrow MTh(N_ZY).
\]
Here $MTh(N_Z Y)$ is the Thom motive associated with the normal bundle $N_ZY$.
We shall formulate the Gysin isomorphism more generally in Construction \ref{A.3.37} allowing nontrivial log structures on $X$.
\vspace{0.1in}

Recall that the technique of deformation to the normal cone is used in the proof of purity for $\A^1$-homotopy theory of schemes \cite[Theorem 2.23, p.\ 115]{MV}. 
We will develop a log version of the said technique based on our unit interval $\boxx$.

\begin{df}
Let $X\in Sm/S$, and let $Z$ be a smooth closed subscheme of $X$. 
We have the commutative diagram of deformation to the normal cone
\begin{equation}
\label{A.3.15.1}
\begin{tikzcd}
Z\arrow[d]\arrow[r]&Z\times \boxx\arrow[d]\arrow[r,leftarrow]&Z\arrow[d]\\
X\arrow[d]\arrow[r]&D_Z X\arrow[d]\arrow[r,leftarrow]&N_ZX\arrow[d]\\
\{1\}\arrow[r,"i_1"]&\boxx\arrow[r,"i_0",leftarrow]&\{0\}
\end{tikzcd}
\end{equation}
where
\begin{enumerate}
\item[(i)] each square is cartesian,
\item[(ii)] $i_0$ is the $0$-section, and $i_1$ is the $1$-section,
\item[(iii)] $Z\rightarrow N_ZX$ is the $0$-section,
\item[(iv)] $D_ZX\rightarrow \boxx$ is the composition $D_ZX\rightarrow B_Z(X\times \boxx)\rightarrow X\times \boxx\rightarrow \boxx$ where the third arrow is the projection.
\end{enumerate}
Suppose $Z_1,\ldots,Z_r$ form a strict normal crossing divisor on $X$, and that $Z$ has strict normal crossing with $Z_1+\cdots+Z_r$ over $S$.
For $Y:=(X,Z_1+\cdots+Z_r)$ we obtain the commutative diagram
\begin{equation}
\label{A.3.15.2}
\begin{tikzcd}
B_Z Y\arrow[d]\arrow[r]&B_{Z\times \boxx}(D_Z Y)\arrow[d]\arrow[r,leftarrow]&B_Z(N_ZY)\arrow[d]\\
Y\arrow[d]\arrow[r]&D_Z Y\arrow[d]\arrow[r,leftarrow]&N_ZY\arrow[d]\\
\{1\}\arrow[r,"i_1"]&\boxx\arrow[r,"i_0",leftarrow]&\{0\}.
\end{tikzcd}
\end{equation}
\end{df}

\begin{rmk}
One difference between the technique of deformation to the normal cone used in the proof of \cite[Theorem 2.23, p.\ 115]{MV} and ours is that we use $X\times \boxx$ instead of $X\times \A^1$.
\end{rmk}

\begin{const}\label{A.3.37}
We are ready to formulate and prove our purity theorem for log motives.
With the above definitions, let $E$ (resp.\ $E^D$, $E^N$) be the exceptional divisor on $B_ZY$ (resp.\ $B_{Z\times \boxx}(D_Z Y)$, resp.\ $B_Z(N_ZY)$).
In the case $Z$ has codimension $1$ in $Y$, we have $E=Z$ (resp.\ $E^D=Z\times \boxx$, resp.\ $E^N=Z$) by convention.
From the diagram \eqref{A.3.15.2}, we have the naturally induced morphisms
\[
M((B_ZY,E)\rightarrow Y)\rightarrow M((B_{Z\times \boxx}(D_Z Y),E^D)\rightarrow D_Z Y),
\]
\[
M((B_Z(N_Z Y),E^N)\rightarrow N_Z Y)\rightarrow M((B_{Z\times \boxx}(D_Z Y),E^D)\rightarrow D_Z Y).
\]
If these morphisms are isomorphisms, then via a zig-zag, we obtain the isomorphism
\[
\mathfrak{g}_{Y,Z}:M((B_Z Y,E)\rightarrow Y)\rightarrow MTh(N_Z Y)
\]
that we refer to as the Gysin \index{Gysin isomorphism} isomorphism for the closed pair $(X,Z)$.  
This will be the log analogue to \cite[Theorem 2.23]{MV}. 
\end{const}

\begin{thm}
\label{A.3.36}
Suppose that $Z_1,Z_2,\ldots,Z_r$ are smooth divisors forming a strict normal crossing divisor on $X\in Sm/S$ 
and that $Z$ is a smooth closed subscheme of $X$ having strict normal crossing with $Z_1+ \cdots + Z_r$ over $S$ such that $Z$ is not contained in any component of $Z_1\cup \cdots\cup Z_r$.
Let $E$ (resp.\ $E^D$, resp.\ $E^N$) be the exceptional divisor on $B_Z X$ (resp.\ $B_{Z\times \boxx}(D_Z X)$, resp.\ $B_Z(N_Z X)$).
For $Y:=(X,Z_1+\cdots+Z_r)$ the morphisms
\[
M((B_Z Y,E)\rightarrow Y)\rightarrow M((B_{Z\times \boxx}(D_ZY),E^D)\rightarrow D_ZY),
\]
\[
M((B_Z(N_ZY),E^N)\rightarrow N_ZY)\rightarrow M((B_{Z\times \boxx}(D_ZY),E^D)\rightarrow D_ZY).
\]
induced by {\rm \eqref{A.3.15.2}} are isomorphisms and functorial in the pair $(Y,Z)$.
\end{thm}
\begin{proof}
{\bf Step 1} We proceed as in the proof of Theorem \ref{A.3.7}, using a local parametrization of $(X,Z)$. 
For $i=1, \ldots, r$, consider the divisor 
$
W_i:= V(t_i)
$
on $\A^{r+s}=\Spec{\Z[t_1,\ldots,t_{r+s}]}$,
and let $W$ denote the subscheme $V(t_{r+1},\ldots,t_{r+p})$.
The question is Zariski local on $X$ by ($Zar$-sep),
and $Z$ is not contained in any component of  $Z_1\cup \cdots\cup Z_r$,
so we may assume there exists an \'etale morphism $u:X\rightarrow \A_S^{r+s}$ of schemes over $S$ for which $u^{-1}(S\times W_i)=Z_i$ for $1\leq i\leq r$ and $u^{-1}(S\times W)=Z$. 
\vspace{0.1in}

{\bf Step 2}
Using $u$ and $Z\rightarrow X$, 
Construction \ref{A.3.16} produces schemes $X_j$ over $S$ for $j=1,2$.
Form $Y_j$ from $X_j$ in the same way as we obtained $Y$ from $X$, and
let $E_j$ (resp.\ $E$) be the exceptional divisor on $B_Z Y_j$ (resp.\ $B_ZY$).
\vspace{0.1in}

{\bf Step 3}
Since blow-ups commute with flat base change, we have the cartesian square
\[
\begin{tikzcd}
D_Z X_2\arrow[d]\arrow[r]&X_2\times\boxx\arrow[d]\\
D_Z X\arrow[r]&X\times \boxx
\end{tikzcd}\]
of fs log schemes over $S$. 
It follows that $D_ZX_2\rightarrow D_ZX$ is \'etale.
Let $E_2^D$ (resp.\ $E_2^N$) be the exceptional divisor on $B_{Z\times \boxx}(D_Z Y_2)$ (resp.\ $B_{Z}(N_Z Y_2)$).
Since $X_2\rightarrow X$ and $N_ZX_2\rightarrow N_ZX$ are also \'etale, 
Proposition \ref{A.3.3} implies there are isomorphisms
\[
M((B_{Z}Y_2,E_2)\rightarrow Y_2)\xrightarrow{\cong} M((B_ZY,E)\rightarrow Y),
\]
\[
M((B_{Z\times \boxx}(D_{Z}Y_2),E_2^D)\rightarrow D_ZY_2)\xrightarrow{\cong} M((B_{Z\times \boxx}(D_ZY),E^D)\rightarrow D_ZY),
\]
and 
\[
M((B_Z(N_{Z}Y_2),E_2^N)\rightarrow N_ZY_2)\xrightarrow{\cong} M((B_Z(N_{Z}Y),E^N)\rightarrow N_ZY). 
\] 
Hence we can replace $X$ with $X_2$. 
Similarly, we can replace $X_2$ with $X_1$.
This reduces to the situation when $X=T\times \A^{r+s},$ $Z_i=T\times W_i$ for $1\leq i\leq r$, and $Z=T\times W$ for some $T\in Sm/S$.
\vspace{0.1in}

{\bf Step 4}
Let $F$ (resp.\ $F^D$, resp.\ $F^N$) be the exceptional divisor on $B_{\{0\}}(\A^p)$ (resp.\ $B_{\{0\}\times \boxx}(D_{\{0\}}(\A_S^p))$, resp.\ $B_{\{0\}\times \boxx}(N_{\{0\}}(\A_S^p))$).
Setting $V:=(Z,Z\cap Z_1+\cdots+Z\cap Z_r)$ we observe that 
\[
V=(T\times W,T\times (W_1\cap W)+\cdots+T\times (W_r\cap W))\cong T\times \A_{\N}^r\times (\Spec {\Z})^{p}\times \A^{s-p},
\]
\[
(B_ZY,E)=(T\times B_{W}(\A^{r+s}),T\times W_1+\cdots +T\times W_{r})\cong V\times (B_{\{0\}}(\A^p),F).
\]
and 
\[
Y=(T \times \A^{r+s},T\times W_1+\cdots+T\times W_r)\cong V\times \A^p.
\]
By (MSq), 
it remains to show that \eqref{A.3.15.1} induces isomorphisms
\[
\begin{split}
&M(V)\otimes M((B_{\{0\}}(\A_S^p),F)\rightarrow \A_S^p)\\
\xrightarrow{\cong} &M(V)\otimes M((B_{\{0\}\times \boxx}(D_{\{0\}}(\A_S^p),F^D)\rightarrow D_{\{0\}}(\A_S^p)),
\end{split}
\]
\[
\begin{split}
& M(V)\otimes M((B_{\{0\}}(N_{\{0\}}(\A_S^p)),F^N)\rightarrow N_{\{0\}}(\A_S^p))\\
\xrightarrow{\cong} & M(V)\otimes M((B_{\{0\}}(D_{\{0\}}(\A_S^p)),F^D)\rightarrow D_{\{0\}}(\A_S^p)).
\end{split}
\]
\vspace{0.1in}

{\bf Step 5}
Due to Step 4 we may assume that $X=\A_S^p$, $r=0$, and $Z=\{0\}\times S$.
We set $S=\Spec{\Z}$ for simplicity.
Let $e_1,\ldots,e_p$ denote the standard coordinates in $\Z^p$.
There are corresponding cones
\begin{gather*}
\tau:={\rm Cone}(e_1,\ldots,e_{p-1}),
\;
\sigma_1:={\rm Cone}(e_1,\ldots,e_{p-1},e_1+\cdots+e_p),
\\
\sigma_2:={\rm Cone}(e_1,\ldots,e_{p-1},-e_p),
\;
\sigma_3:={\rm Cone}(e_1,\ldots,e_{p-1},-e_1-\cdots-e_p),
\end{gather*}
and their fans
\[
\Sigma_1:=\{\sigma_1,\sigma_2\},
\;
\Sigma_2:=\{\sigma_1,\sigma_3\},
\;
\Sigma_1':=\{{\rm Cone}(-e_p)\},
\;
\Sigma_2':=\{{\rm Cone}(-e_1-\cdots-e_p)\}.
\]
Moreover, we shall consider the star subdivisions
\[
\Sigma_3:=\Sigma_2^*(\sigma_3),
\;
\Sigma_4:=\Sigma_1^*(\tau),
\;
\Sigma_5:=\Sigma_2^*(\tau),
\;
\Sigma_6:=\Sigma_3^*(\tau),
\]
and the subfans
\begin{gather*}
\Sigma_3':=\{{\rm Cone}(-e_p,-e_1-\cdots-e_p)\},
\;
\Sigma_4':=\{{\rm Cone}(e_1+\cdots+e_{p-1},-e_p)\},
\\
\Sigma_5':=\{{\rm Cone}(e_1+\cdots+e_{p-1},-e_1-\cdots-e_p)\},
\\
\Sigma_6':=\{{\rm Cone}(e_1+\cdots+e_{p-1},-e_p),{\rm Cone}(-e_p,-e_1-\cdots-e_p)\}.
\end{gather*}

For $1\leq i\leq p-1$, we set
\[
\eta_i:={\rm Cone}(e_1,\ldots,e_{i-1},e_{i+1},\ldots,e_{p-1},-e_p,-e_1-\cdots-e_p),
\]
and form the fans
\[
\Sigma_7:=\{\eta_1,\ldots,\eta_{p-1}\},
\;
\Sigma_7':=\Sigma_3'.
\]
\vspace{0.1in}

When $p=3$, the above fans take the following forms.
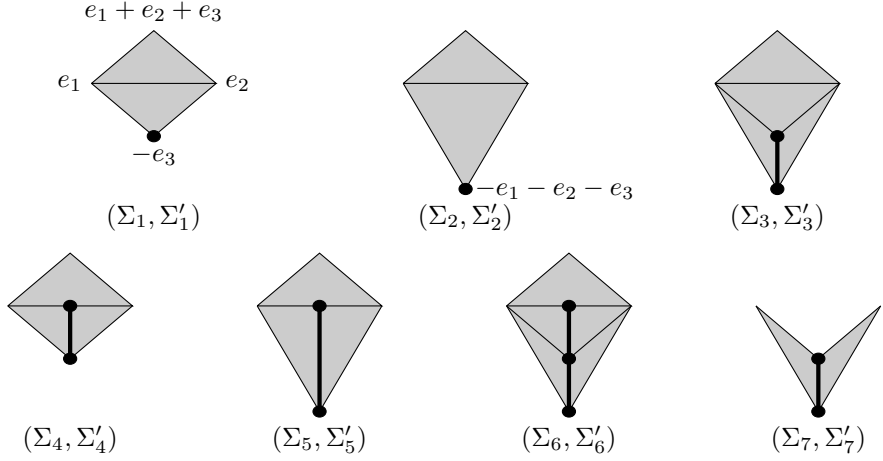
\[
\begin{tikzpicture}[yscale=0.7, xscale=0.82]
\fill[black!20] (0,0)--(1,1)--(2,0)--(0,0);
\fill[black!20] (0,0)--(2,0)--(1,-1)--(0,0);
\draw (0,0)--(1,1)--(2,0)--(0,0);
\draw (0,0)--(1,-1)--(2,0);
\node [left] at (0,0) {$e_1$};
\node [right] at (2,0) {$e_2$};
\node [above] at (1,1) {$e_1+e_2+e_3$};
\node [below] at (1,-1) {$-e_3$};
\node [below] at (1,-2.1) {$(\Sigma_1,\Sigma_1')$};
\filldraw (1,-1) circle (3pt);
\begin{scope}[shift={(5,0)}]
\fill[black!20] (0,0)--(1,1)--(2,0)--(0,0);
\fill[black!20] (0,0)--(2,0)--(1,-2)--(0,0);
\draw (0,0)--(1,1)--(2,0)--(0,0);
\draw (0,0)--(1,-2)--(2,0);
\filldraw (1,-2) circle (3pt);
\node [right] at (1,-2) {$-e_1-e_2-e_3$};
\node [below] at (1,-2.1) {$(\Sigma_2,\Sigma_2')$};
\end{scope}
\begin{scope}[shift={(10,0)}]
\fill[black!20] (0,0)--(1,1)--(2,0)--(0,0);
\fill[black!20] (0,0)--(2,0)--(1,-2)--(0,0);
\draw (0,0)--(1,1)--(2,0)--(0,0);
\draw (0,0)--(1,-1)--(2,0);
\draw[ultra thick] (1,-1)--(1,-2);
\draw (0,0)--(1,-2)--(2,0);
\filldraw (1,-2) circle (3pt);
\filldraw (1,-1) circle (3pt);
\node [below] at (1,-2.1) {$(\Sigma_3,\Sigma_3')$};
\end{scope}
\end{tikzpicture}
\]
\[
\begin{tikzpicture}[yscale=0.7, xscale=0.82]
\fill[black!20] (0,0)--(1,1)--(2,0)--(0,0);
\fill[black!20] (0,0)--(2,0)--(1,-1)--(0,0);
\draw (0,0)--(1,1)--(2,0)--(0,0);
\draw (0,0)--(1,-1)--(2,0);
\draw[ultra thick] (1,0)--(1,-1);
\filldraw (1,0) circle (3pt);
\filldraw (1,-1) circle (3pt);
\node [below] at (1,-2.1) {$(\Sigma_4,\Sigma_4')$};
\begin{scope}[shift={(4,0)}]
\fill[black!20] (0,0)--(1,1)--(2,0)--(0,0);
\fill[black!20] (0,0)--(2,0)--(1,-2)--(0,0);
\draw (0,0)--(1,1)--(2,0)--(0,0);
\draw (0,0)--(1,-2)--(2,0);
\draw[ultra thick] (1,0)--(1,-2);
\filldraw (1,0) circle (3pt);
\filldraw (1,-2) circle (3pt);
\node [below] at (1,-2.1) {$(\Sigma_5,\Sigma_5')$};
\end{scope}
\begin{scope}[shift={(8,0)}]
\fill[black!20] (0,0)--(1,1)--(2,0)--(0,0);
\fill[black!20] (0,0)--(2,0)--(1,-2)--(0,0);
\draw (0,0)--(1,1)--(2,0)--(0,0);
\draw (0,0)--(1,-1)--(2,0);
\draw[ultra thick] (1,0)--(1,-2);
\draw (0,0)--(1,-2)--(2,0);
\filldraw (1,-2) circle (3pt);
\filldraw (1,-1) circle (3pt);
\filldraw (1,0) circle (3pt);
\node [below] at (1,-2.1) {$(\Sigma_6,\Sigma_6')$};
\end{scope}
\begin{scope}[shift={(12,0)}]
\fill[black!20] (0,0)--(1,-1)--(2,0)--(1,-2);
\draw (0,0)--(1,-1)--(2,0);
\draw[ultra thick] (1,-1)--(1,-2);
\draw (0,0)--(1,-2)--(2,0);
\filldraw (1,-2) circle (3pt);
\filldraw (1,-1) circle (3pt);
\node [below] at (1,-2.1) {$(\Sigma_7,\Sigma_7')$};
\end{scope}
\end{tikzpicture}
\]
\vspace{0.1in}

Observe that
\[
D_ZY=\A_{(\Sigma_1,\Sigma_1')},
\;
(B_{Z\times \boxx}(D_Z Y),E^D)=\A_{(\Sigma_4,\Sigma_4')}.
\]
The fan $\{\sigma_1\}$ corresponds to the trivial line bundle $Y\times \A^1$, and the $0$-section (resp.\ $1$-section) of the bundle is $Y$ (resp.\ $N_ZY$).
Hence there are naturally induced commutative diagrams
\[
\begin{tikzcd}
&
\A_{(\Sigma_1,\Sigma_1')}\arrow[d]&
\\
Y\arrow[r]\arrow[ru]\arrow[rd]&
\A_{(\Sigma_3,\Sigma_3')}\arrow[d]&
N_ZY\arrow[l]\arrow[lu]\arrow[ld]
\\
&
\A_{(\Sigma_2,\Sigma_2')},
\end{tikzcd}
\;\;
\begin{tikzcd}
&
\A_{(\Sigma_4,\Sigma_4')}\arrow[d]&
\\
(B_Z Y,E)\arrow[r]\arrow[ru]\arrow[rd]&
\A_{(\Sigma_6,\Sigma_6')}\arrow[d]&
(B_Z(N_Z Y),E^N)\arrow[l]\arrow[lu]\arrow[ld]
\\
&
\A_{(\Sigma_5,\Sigma_5')}.
\end{tikzcd}
\]

Due to these diagrams, we are reduced to showing the naturally induced morphisms
\begin{gather}
\label{A.3.36.1}
M(\A_{(\Sigma_4,\Sigma_4')}\rightarrow \A_{(\Sigma_1,\Sigma_1')})\rightarrow M(\A_{(\Sigma_6,\Sigma_6')}\rightarrow \A_{(\Sigma_3,\Sigma_3')}),
\\
\label{A.3.36.2}
M(\A_{(\Sigma_6,\Sigma_6')}\rightarrow \A_{(\Sigma_3,\Sigma_3')})\rightarrow M(\A_{(\Sigma_5,\Sigma_5')}\rightarrow \A_{(\Sigma_2,\Sigma_2')}),
\\
\label{A.3.36.3}
M((B_Y Z,E)\rightarrow Y)\rightarrow M(\A_{(\Sigma_5,\Sigma_5')}\rightarrow \A_{(\Sigma_2,\Sigma_2')}),
\\
\label{A.3.36.4}
M((B_Z(N_Z Y),E^N)\rightarrow N_ZY)\rightarrow M(\A_{(\Sigma_5,\Sigma_5')}\rightarrow \A_{(\Sigma_2,\Sigma_2')})
\end{gather}
are isomorphisms.
\vspace{0.1in}

The naturally induced morphisms
\[
\A_{(\Sigma_6,\Sigma_6')}\rightarrow \A_{(\Sigma_5,\Sigma_5')},
\;
\A_{(\Sigma_3,\Sigma_3')}\rightarrow \A_{(\Sigma_2,\Sigma_2')}
\]
are admissible blow-ups along
\[
V({\rm Cone}(e_1+\cdots+e_{p-1},-e_1-\cdots-e_p)),
\;
V(\sigma_3).
\]
Theorem \ref{A.3.7} shows that
\[
M(\A_{(\Sigma_3,\Sigma_3')}\rightarrow \A_{(\Sigma_2,\Sigma_2')})=0,
\;
M(\A_{(\Sigma_6,\Sigma_6')}\rightarrow \A_{(\Sigma_5,\Sigma_5')})=0.
\]
Proposition \ref{A.3.44} allows us to deduce that \eqref{A.3.36.2} is an isomorphism.
\vspace{0.1in}

We can identify $\A_{(\Sigma_2,\Sigma_2')}$ and $\A_{(\Sigma_5,\Sigma_5')}$ with
\[
Y\times \boxx,
\;
(B_Z Y,E)\times \boxx.
\]
The fibers at $0\in \boxx$ (resp.\ $1\in \boxx$) correspond to $N_ZY$ and $(B_Z(N_ZY),E^N)$ (resp.\ $Y$ and $(B_ZY,E)$).
Thus ($\boxx$-inv) shows that \eqref{A.3.36.3} and \eqref{A.3.36.4} are isomorphisms.
\vspace{0.1in}

There are strict Zariski distinguished squares
\[
\begin{tikzcd}
\A_{(\Sigma_1,\Sigma_1')}\cap \A_{(\Sigma_7,\Sigma_7')}\arrow[r]\arrow[d]& \A_{(\Sigma_7,\Sigma_7')}\arrow[d]
\\
\A_{(\Sigma_1,\Sigma_1')}\arrow[r]& \A_{(\Sigma_3,\Sigma_3')},
\end{tikzcd}
\quad
\begin{tikzcd}
\A_{(\Sigma_4,\Sigma_4')}\cap \A_{(\Sigma_7,\Sigma_7')}\arrow[r]\arrow[d]& \A_{(\Sigma_7,\Sigma_7')}\arrow[d]
\\
\A_{(\Sigma_4,\Sigma_4')}\arrow[r]& \A_{(\Sigma_6,\Sigma_6')}.
\end{tikzcd}
\]
Since 
$$
\A_{(\Sigma_1,\Sigma_1')}\cap \A_{(\Sigma_7,\Sigma_7')}=\A_{(\Sigma_4,\Sigma_4')}\cap \A_{(\Sigma_7,\Sigma_7')},
$$ 
we can use ($sNis$-des) to obtain the isomorphism
\[
M(\A_{(\Sigma_4,\Sigma_4')}\rightarrow \A_{(\Sigma_6,\Sigma_6')})\xrightarrow{\cong} M(\A_{(\Sigma_1,\Sigma_1')}\rightarrow \A_{(\Sigma_3,\Sigma_3')}).
\]
Combined with Proposition \ref{A.3.44} we deduce that \eqref{A.3.36.1} is an isomorphism, which finishes the proof.
Moreover, 
functoriality follows from the construction.
\end{proof}

\begin{df}\label{A.3.38}
With reference to Theorem \ref{A.3.36}, the {\it Gysin isomorphism} \index{Gysin isomorphism} is the isomorphism 
\[
\mathfrak{g}_{Y,Z}:M((B_ZY,E)\rightarrow Y)\rightarrow MTh(N_ZY),
\]
\end{df}

\subsection{Invariance under admissible blow-ups}\label{ssec:invadmblowrefined}
In this section, 
we assume that $\cT$ satisfies the properties {\rm ($Zar$-sep)}, {\rm ($\boxx$-inv)}, {\rm ($sNis$-des)}, and {\rm ($div$-des)}, $S=\Spec k$, and $k$ admits resolution of singularities. 
Under these hypotheses, we prove that motives are invariant under admissible blow-ups. 
This invariance will be used for comparing $\ldmeff$ with $\dmeff$. We start by introducing the following (weaker) version of the notion of admissible blow-up in Definition \ref{A.3.17}.

\begin{df}
\label{A.3.1}
An {\it admissible blow-up} \index{blow-up!admissible} $Y\rightarrow X$ in $\mathscr{S}/S$ is a proper birational morphism of fs log schemes such that the induced morphism
\[
Y-\partial Y\rightarrow X-\partial X
\]
of schemes is an isomorphism. 
We denote by $\cA\cB l/S$ the class of admissible blow-ups.

Note that $\cA\cB l_{Sm}/S$ and $\cA\cB l_{div}/S$ are subclasses of $\cA\cB l/S$.
\end{df}
\begin{rmk}
The class $\cA\cB l/S$ is equivalent to the class $\cS_b$ in \cite{AdditiveHomotopy} if $S=\Spec k$ and $\mathscr{S}/S=SmlSm/k$.
\end{rmk}

\begin{df}
\label{A.3.8}
The field $k$ admits {\it resolution of singularities}\index{resolution of singularities} if the following conditions hold.
\begin{enumerate}
\item[(1)] For any integral scheme $X$ of finite type over $k$, there is a proper birational morphism $Y\rightarrow X$ of schemes over $k$ such that $Y$ is smooth over $k$.
\item[(2)] Let $f:Y\rightarrow X$ be a proper birational morphism of integral schemes over $k$ such that $X$ is smooth over $k$ and let $Z_1,\ldots, Z_r$ be smooth divisors forming a strict normal crossing divisor on $X$. 
Assume that 
\[
f^{-1}(X-Z_1\cup \cdots\cup Z_r)\rightarrow X-Z_1\cup \cdots\cup Z_r
\]
is an isomorphism. 
Then there is a sequence of blow-ups
\[
X_n\stackrel{f_{n-1}}\rightarrow X_{n-1}\stackrel{f_{n-2}}\rightarrow \cdots\stackrel{f_0}\rightarrow X_0=X
\]
along smooth centers $W_i\subset X_i$ such that
\begin{enumerate}
\item[(i)]  the composition $X_n\rightarrow X$ factors through $f$,
\item[(ii)] $W_i$ is contained in the preimage of $Z_1\cup\cdots\cup Z_r$ in $X_i$,
\item[(iii)] $W_i$ has strict normal crossing with the sum of the reduced strict transforms of 
\[
Z_1,\ldots,Z_r,f_0^{-1}(W_0),\ldots,f_{i-1}^{-1}(W_{i-1})
\]
in $X_i$.
\end{enumerate}
\end{enumerate}
Note that $k$ must be perfect in order to satisfy (1). If $k$ has characteristic $0$, then (1) (resp.\ (2)) is satisfied by \cite[Main theorem I]{Hir} (resp.\ \cite[Main theorem II]{Hir}).
\end{df}

\begin{prop}
\label{A.3.2}
Assume that $k$ admits resolution of singularities. 
Then the class $\cA\cB l_{div}/k$ admits a calculus of right fractions.
\end{prop}
\begin{proof}
By definition, 
the class in question contains isomorphisms and is closed under composition. 
If $f,g:X\rightarrow Y$ and $s:Y\rightarrow Z$ are morphisms of fs log schemes log smooth over $k$ such that $s\circ f=s\circ g$ and $s\in \cA\cB l_{div}/k$, 
then $f$ and $g$ agree on the dense open subscheme $X-\partial X$ of $X$ since the morphism $Y-\partial Y\rightarrow Z-\partial Z$ induced by $s$ is an isomorphism.
Thus $f=g$ by Lemma \ref{A.5.9} and the right cancellabilty condition holds. 
\vspace{0.1in}

It remains to check the right Ore condition. 
Suppose $f:Y\rightarrow X$ and $g:X'\rightarrow X$ are morphisms where $f\in \cA\cB l_{div}/k$. 
By Proposition \ref{A.3.19}, we can choose a dividing cover $X_1'\rightarrow X'$ such that $\partial X_1'$ is a strict normal crossing divisor.
Since $f$ is an admissible blow-up, $Y\times_X X_1'$ contains $X_1'-\partial X_1'$ as an open subscheme.
Let $Y_1'$ be the closure of $X_1'-\partial X_1'$, which is proper and birational over $X_1'$.
By resolution of singularities, there exists a morphism $Y'\to X_1'\in \cA\cB l_{Sm}/k$ that factors through $Y_1'$.
Thus we have the commutative diagram
\[
\begin{tikzcd}
Y'\arrow[d]\arrow[r]&Y\arrow[d,"f"]\\
X'\arrow[r,"g"]&X.
\end{tikzcd}\]
Here the left vertical arrow is in $\cA\cB l_{div}/k$, so the right Ore condition is satisfied.
\end{proof}

\begin{prop}
\label{A.3.9}
Assume that $k$ admits resolution of singularities. 
For any morphism $f:Y\rightarrow X$ in $\cA\cB l/k$, there is a morphism $Y'\rightarrow X$ in $\cA\cB l_{div}/k$ that factors through $f$.
\end{prop}
\begin{proof}
There exists a dividing cover $g:X_1\rightarrow X$ such that $\underline{X_1}$ is smooth over $k$ and $\partial X_1$ is a strict normal crossing divisor, 
see Proposition \ref{A.3.19}.
Applying part (2) of Definition \ref{A.3.8} to the projection $p:X_1\times_X Y\rightarrow X_1$,
we obtain a morphism $h:Y'\rightarrow X_1$ in $\cA \cB l_{Sm}/k$ that factors through $p$.
Then the composition $gh$ is in $\cA \cB l_{div}/k$ and factors through $f$.
\end{proof}

\begin{prop}
\label{A.3.10}
Assume that $k$ admits resolution of singularities. 
Then $\cA\cB l/k$ admits a calculus of right fractions.
\end{prop}
\begin{proof}
The right Ore condition follows from Propositions \ref{A.3.2} and \ref{A.3.9}.  
The proof of the right cancellability condition is similar to that in Proposition \ref{A.3.2}.
\end{proof}


\begin{thm}
\label{A.3.12}
Assume that $k$ admits resolution of singularities. 
Then any morphism $f:Y\rightarrow X$ in $\cA\cB l/k$ induces an isomorphism
\[
M(Y)\xrightarrow{\cong} M(X).
\]
\end{thm}
\begin{proof}
Propositions \ref{A.3.2} and \ref{A.3.10} show that $\cA=\cA\cB l/k$ and $\cB=\cA\cB l_{div}/k$ satisfy a calculus of right fractions. 
All the conditions in Lemma \ref{A.1.16} are met by Propositions \ref{A.3.11} and \ref{A.3.9}, so we are done.
\end{proof}

\begin{rmk}
\label{A.3.20}
By Theorem \ref{A.3.12}, we can apply the results in \ref{section:calculus} to
\[  
\cA=\cA Bl/k;
\,
I=\boxx;
\,
T=lSm/k, Cor/k. 
\]  
  
For example, 
for any $(I,\cA)$-weak equivalence $f:\cF\rightarrow \cG$ of bounded above complexes of presheaves (resp.\ presheaves with transfers) on $lSm/k$, 
the image of $f$ in $\ldaeff$ (resp.\ $\ldmeff$) is an isomorphism.
\end{rmk}

\subsection{Invariance under log modifications}\label{ssec:invariance-asuming-PnPn-1}
In this section we assume that $\mathscr{S}=SmlSm/S$ and $\cT$ satisfies the properties
{\rm ($Zar$-sep)}, {\rm ($\boxx$-inv)}, {\rm ($sNis$-des)}, and ($(\P^\bullet,\P^{\bullet-1})$-inv). 
Under these hypothesis we show that the property ($div$-des) holds.

\begin{lem}
\label{A.3.26}
Let $H_1$ and $H_2$ be the closed subschemes $\{0\}\times \A^1$ and $\A^1\times \{0\}$ of $\A^2$, 
let $H_1'$ and $H_2'$ be the strict transforms of $H_1$ and $H_2$ in ${\rm Bl}_{\{(0,0)\}}(\A^2)$, 
and let $E$ be the exceptional divisor on ${\rm Bl}_{\{(0,0)\}}(\A^2)$.
Then for any $X\in \mathscr{S}/S$, there is a naturally induced isomorphism
\[
M(X\times ({\rm Bl}_{\{(0,0)\}}\A^2,H_1'+H_2'+E))\xrightarrow{\cong} M(X\times (\A^2,H_1+H_2)).
\]
\end{lem}
\begin{proof}
Assume that $S=X=\Spec{\Z}$ for simplicity of notation as usual.
Due to Proposition \ref{A.3.39}, it suffices to show there is a naturally induced isomorphism
\[
M({\rm Bl}_{\{(0,0)\}}\A^2,H_1'+E)\xrightarrow{\cong} M(\A^2,H_1).
\]
Let $\overline{H_1}$ and $\overline{H_2}$ be the closure of $H_1$ and $H_2$ in $\P^2$, 
and let $\overline{H_1'}$ and $\overline{H_2'}$ be the closures of $H_1'$ and $H_2'$ in ${\rm Bl}_{\{O\}}(\P^2)$, 
where $O:=[0:0:1]$.
There are Zariski distinguished squares
\[
\begin{tikzcd}[column sep=small]
(\A^2,H_1)-\{O\}\arrow[d]\arrow[r]&
(\A^2,H_1)\arrow[d]
\\
(\P^2,\overline{H_1})-\{O\}\arrow[r]&
(\P^2,\overline{H_1})
\end{tikzcd}
\quad
\begin{tikzcd}[column sep=small]
(\A^2,H_1)-\{O\}\arrow[d]\arrow[r]&
({\rm Bl}_{\{(0,0)\}}\A^2,H_1'+E)\arrow[d]
\\
(\P^2,\overline{H_1})-\{O\}\arrow[r]&
({\rm Bl}_{\{O\}}\P^2,\overline{H_1'}+E).
\end{tikzcd}
\]
Applying ($sNis$-des) to these squares yields an isomorphism
\[
M(({\rm Bl}_{\{(0,0)\}}\A^2,H_1'+E)\rightarrow ({\rm Bl}_{\{O\}}\P^2,\overline{H_1'}+E))
\xrightarrow{\cong}
M((\A^2,H_1)\rightarrow (\P^2,\overline{H_1}))
\]
By Proposition \ref{A.3.44} we find the isomorphism
\[
M(({\rm Bl}_{\{(0,0)\}}\A^2,H_1'+E)\rightarrow (\A^2,H_1))
\xrightarrow{\cong}
M(({\rm Bl}_{\{O\}}\P^2,\overline{H_1'}+E)\rightarrow (\P^2,\overline{H_1})).
\]
Hence we only need to show that the morphism
\[
M({\rm Bl}_{\{O\}}\P^2,\overline{H_1'}+E)\rightarrow M(\P^2,\overline{H_1})
\]
is an isomorphism.
Since $({\rm Bl}_{\{O\}}\P^2,\overline{H_1'}+E))$ is a $\boxx$-bundle over $\boxx$ by Lemma \ref{A.3.28}, 
($(\P^\bullet,\P^{\bullet-1})$-inv) gives
\[
M({\rm Bl}_{\{O\}}\P^2,\overline{H_1'}+E)\cong M(\Z)\cong M(\P^2,\overline{H_1}).
\]
\end{proof}

\begin{df}\label{A.3.32}
For a number $r\in \N^+$, let ${\rm Fan}_r$\index[notation]{Fanr @ ${\rm Fan}_r$} denote the category where objects are $r$-dimensional fans and morphisms are subdivisions.
For an $r$-dimensional smooth fan $\Sigma$, 
let ${\rm Fan}_r^{sm}/\Sigma$\index[notation]{Fanrsm @ ${\rm Fan}_r^{sm}/\Sigma$} denote the full category of ${\rm Fan}_r/\Sigma$ consisting of smooth $r$-dimensions fans that are subdivisions of $\Sigma$. \index{fan!star subdivision}
Let $\cS \cD/\Sigma$ be the class of morphisms in ${\rm Fan}_r/\Sigma$, 
and let $\cS \cD_2/\Sigma$ denote the class of morphisms in ${\rm Fan}_r^{sm}/\Sigma$ that are obtained by finite successions of star subdivisions relative to two-dimensional cones.
\end{df}

\begin{lem}\label{A.3.29}
Let $\Sigma$ be a fan.
The class $\cS \cD/\Sigma$ admits a calculus of right fractions.
If in addition $\Sigma$ is smooth, then
the class $\cS \cD_2/\Sigma$ also admits a calculus of right fractions.
\end{lem}
\begin{proof}
Let $r$ be the dimension of $\Sigma$.
By definition, $\cS \cD_2/\Sigma$ and $\cS \cD_2/\Sigma$ are closed under compositions.
For any objects $\Sigma_1$ and $\Sigma_2$ of ${\rm Fan}_r/\Sigma$ or ${\rm Fan}_r^{sm}/\Sigma$, there is at most one morphism from $\Sigma_2$ to $\Sigma_1$.
Thus the right cancellabilty condition holds for $\cS \cD/\Sigma$ and $\cS \cD_2/\Sigma$.
It remains to check the right Ore condition.
Suppose that $f:\Sigma_2\rightarrow \Sigma_1$ and $g:\Sigma_1'\rightarrow \Sigma_1$ are in $\cS \cD/\Sigma$.
Let $\Sigma_2'$ be the fan generated by $\Sigma_2$ and $\Sigma_1'$.
We have the naturally induced commutative diagram of fans
\[
\begin{tikzcd}
\Sigma_2'\arrow[d,"f'"]\arrow[r]&\Sigma_2\arrow[d,"f"]\\
\Sigma_1'\arrow[r,"g"]&\Sigma_1.
\end{tikzcd}
\]
This shows the Ore condition for $\cS \cD/\Sigma$.
\vspace{0.1in}

Suppose that $\Sigma$, $\Sigma_1$, $\Sigma_1'$, and $\Sigma_2$ are smooth.
The fan $\Sigma_2'$ is not necessarily smooth, but $\Sigma_2'$ has a subdivision $\Sigma_2''$ that is smooth by toric resolution of singularities, 
see \cite[Theorem 11.1.9]{CLStoric}.
By Lemma \ref{A.3.31}, there exists a morphism $f':\Sigma_2'''\rightarrow \Sigma_1$ in $\cS \cD_2/\Sigma$ that factors through 
$$
\Sigma_2''\rightarrow \Sigma_2'\stackrel{f'}\rightarrow \Sigma_1'.
$$
This shows the right Ore condition for $\cS \cD_2/\Sigma$.
\end{proof}

\begin{thm}\label{A.3.27}
Every log modification $f:Y\rightarrow X$ in $SmlSm/S$ induces an isomorphism 
\[
M(f):M(Y)\xrightarrow{\cong} M(X).
\]
\end{thm}
\begin{proof}
By ($Zar$-sep), the question is Zariski local on $X$. 
Hence by definition, we may assume that there exists a morphism $g:Z\rightarrow Y$ such that both $fg$ and $g$ are log blow-ups.
It suffices to show that $M(fg):M(Z)\rightarrow M(X)$ and $M(g):M(Z)\rightarrow M(Y)$ are isomorphisms, so
we reduce to the case when $f$ is a log blow-up.
\vspace{0.1in}

The question is Zariski local on $X$, so we may assume that $X$ has an fs chart $P$ with $P\cong \N^r$ for some $r$ such that $X\rightarrow \A_P$ is smooth by Lemma \ref{lem::Smchart}.
By \cite[Lemma III.2.6.4]{Ogu}, we may further assume that $f$ is the projection $X\times_{\A_P}\A_M\rightarrow X$ for some subdivision of fans $M\rightarrow \Spec{P}$.
Consider the dual monoid $P^\vee$ as a fan, and consider $M$ as a fan that is a subdivision of $P^\vee$.
\vspace{0.1in}

We have the functors $\cS \cD/P^\vee\rightarrow \cT$ and $\cS \cD_2/P^\vee\rightarrow \cT$ given by
\[
(N\in {\rm Fan}_r/P^{\vee}) \mapsto M(X\times_{\A_P}\A_{N}).
\]
By Lemma \ref{A.3.29}, both $\cS \cD/P^\vee$ and $\cS \cD_2/P^\vee$ admit a calculus of right fractions.
Owing to Lemma \ref{A.1.16}, to have the isomorphism
$$
M(f):M(X\times_{\A_P}\A_M)\xrightarrow{\cong} M(X).
$$ 
Hence we are reduced to showing that 
\[
M(g):M(X\times_{\A_P}\A_{N'})\rightarrow M(X\times_{\A_P}\A_N)
\]
is an isomorphism for any morphism $N'\rightarrow N$ in $\cS \cD_2/P^\vee$.
It suffices to consider the case when $N'$ is the star subdivision of $N$ relative to a $2$-dimensional cone.
Using Zariski locality, we reduce to the case when $N$ is a $2$-dimensional cone, and $N'$ is its star subdivision.
Since $X$ is smooth over $\A_P$, we see that $X\times_{\A_P}\A_N$ is smooth over $\A_N$.
Hence we can replace $X$ by $X\times_{\A_P}\A_N$ and $N'$ by $M$ so that it suffices to show
\[
M(f):M(X\times_{\A_P}\A_M)\rightarrow M(X)
\]
is an isomorphism when $P=\N^2$ and $M$ is the star subdivision of $P^\vee$.
\vspace{0.1in}


Next, we set 
\[
Z:=\underline{X}\times_{\underline{\A_{P}}}\underline{{\rm pt}_{P}}.
\]
Since $X$ is smooth over its chart $\A_{P}$, Zariski locally on $X$,
there is a strict \'etale morphism $X\rightarrow \A^n \times \A_{P}$ such that the composition with the projection $\A^n\times \A_{P}\rightarrow \A_{P}$ is equal to $X\rightarrow \A_P$.
In this case, 
$(\underline{X},Z)$ has a parametrization $(\A^{n+2},\A^n)$ and Construction \ref{A.3.16} produces schemes $\underline{X_j}$ for $j=1,2$.
Give $\underline{X_2}$ the log structure induced by that of $X$, and we obtain an fs log scheme $X_2$.
Set $X_1:=Z\times \A_P$, which gives the log structure on $\underline{X_1}$.
Now define $Y_1:=X_1\times_{\A_P}\A_M$ and $Y_2:=X_2\times_{\A_P}\A_M$.
Let $E$, $E_1$, and $E_2$ be the exceptional divisors on $Y$, $Y_1$, and $Y_2$.
Note that $E_2\rightarrow E$ and $E_2\rightarrow E_1$ are isomorphisms since $Z$ is unchanged in $X_2\rightarrow X$ and $X_2\rightarrow X_1$.
\vspace{0.1in}

Owing to Lemma \ref{A.3.26}, $M(Y_1)\rightarrow M(X_1)$ is an isomorphism, i.e., we have 
\[
M(Y_1\rightarrow X_1)\simeq 0.
\]
We obtain strict Nisnevich distinguished squares in $SmlSm/S$
\[
\begin{tikzcd}
X_2-Z\arrow[r]\arrow[d]&X_2\arrow[d]\\
X-Z\arrow[r]&X
\end{tikzcd}
\begin{tikzcd}
Y_2-E_2\arrow[r]\arrow[d]&Y_2\arrow[d]\\
Y-E\arrow[r]&Y.
\end{tikzcd}
\]

By ($sNis$-des), there are isomorphisms
\[
M((Y_2-E_2)\rightarrow (Y-E))\stackrel{\sim}\rightarrow M(Y_2\rightarrow Y),\;
M((X_2-Z)\rightarrow (X-Z))\stackrel{\sim}\rightarrow M(X_2\rightarrow X).
\]
Since $Y_2-E_2'\cong X_2-Z$ and $Y-E\cong X-Z$, there is an isomorphism
\[
M(Y_2\rightarrow Y)\cong M(X_2\rightarrow X_2)
\]
Using Proposition \ref{A.3.44}, we obtain an isomorphism
\[
M(Y_2\rightarrow X_2)\cong M(Y\rightarrow X).
\]
Similarly, there is an isomorphism
\[
M(Y_2\rightarrow X_2)\cong M(Y_1\rightarrow X_1).
\]
Thus $M(Y\rightarrow X)\cong 0$ since $M(Y_1\rightarrow X_1)\cong 0$,
so $M(Y)\rightarrow M(X)$ is an isomorphism.
\end{proof}

\subsection{Strictly $(\mathbb{P}^\bullet,\mathbb{P}^{\bullet-1})$-invariant complexes}
In this section we introduce the notions of strictly $(\P^\bullet,\P^{\bullet-1})$-invariant and strictly dividing invariant complexes of sheaves on $SmlSm/k$.
We show that these two notions are equivalent.
Consequently, we provide another way of constructing objects of $\ldmeff$ and $\ldmeffet$.

\begin{df}
\label{Pn.1}
We denote by $(\P^{\bullet},\P^{\bullet-1})$ the class of projections
\[
(\P^{n},\P^{n-1})\times X\rightarrow X
\]
for all $X\in SmlSm/k$ and positive integers $n>0$.
\end{df}

\begin{df}
\label{Pn.2}
Let $\cF$ be a complex of $t$-sheaves on $SmlSm/k$, where $t$ is one of $sNis$, $s\acute{e}t$, or $k\acute{e}t$.
We say that $\cF$ is \emph{strictly dividing invariant}\index{invariant!strictly dividing} if for every log modification $Y\rightarrow X$ in $SmlSm/k$ and integer $i\in \Z$ there is an isomorphism
\[
\bH_{t}^i(X,\cF)\cong \bH_{t}^i(Y,\cF).
\]

We say that $\cF$ is \emph{strictly $(\P^{\bullet},\P^{\bullet-1})$-invariant}\index{invariant!strictly PnPn @ strictly $(\P^{\bullet},\P^{\bullet-1})$} if for every $X\in SmlSm/k$, $i\in \Z$, and integer $n>0$ there is an isomorphism
\[
\bH_{t}^i(X,\cF)\cong \bH_{t}^i(X\times (\P^n,\P^{n-1}),\cF).
\]
\end{df}

\begin{thm}
\label{Pn.3}
Let $\cF$ be a complex of $t$-sheaves on $SmlSm/k$, where $t$ is one of $sNis$, $s\acute{e}t$, or $k\acute{e}t$.
Then $\cF$ is strictly dividing invariant if and only if $\cF$ is strictly $(\P^\bullet,\P^{\bullet-1})$-invariant.
\end{thm}
\begin{proof}
We only consider the strict Nisnevich topology since the proofs are similar.
Without loss of generality, we may assume that $\cF$ is strict Nisnevich local.
\vspace{0.1in}

On $\Co(\Shv_{sNis}(SmlSm/k,\Lambda))$ we have the $div$-local descent model structure and the $(\P^\bullet,\P^{\bullet-1})$-local descent model structure.
Let $\mathbf{Ho}_{div}$ and $\mathbf{Ho}_{(\P^\bullet,\P^{\bullet-1})}$ denote the corresponding homotopy categories.
Proposition \ref{A.5.71} shows the properties  {\rm ($Zar$-sep)}, {\rm ($\boxx$-inv)}, {\rm ($sNis$-des)}, and {\rm ($div$-des)} hold in $\ldaeff$.
We can similarly prove that {\rm ($Zar$-sep)}, {\rm ($\boxx$-inv)}, {\rm ($sNis$-des)}, and {\rm ($div$-des)} hold in $\mathbf{Ho}_{div}$, 
while {\rm ($Zar$-sep)}, {\rm ($\boxx$-inv)}, {\rm ($sNis$-des)}, and ($(\P^\bullet,\P^{\bullet-1})$-inv) hold in $\mathbf{Ho}_{(\P^\bullet,\P^{\bullet-1})}$.
Owing to Proposition \ref{A.6.1} and Theorem \ref{A.3.27}, 
we see that {\rm ($Zar$-sep)}, {\rm ($\boxx$-inv)}, {\rm ($sNis$-des)}, {\rm ($div$-des)}, and ($(\P^\bullet,\P^{\bullet-1})$-inv) hold in both $\mathbf{Ho}_{div}$ 
and $\mathbf{Ho}_{(\P^\bullet,\P^{\bullet-1})}$.
\vspace{0.1in}

Suppose now that $\cF$ is strictly dividing invariant.
Then $\cF$ is a fibrant object with respect to the $div$-local descent model structure on $\Co(\Shv_{sNis}(SmlSm/k,\Lambda))$.
Thus for every $X\in SmlSm/k$ and integer $i\in \Z$, there is an isomorphism
\[
\hom_{\mathbf{Ho}_{div}}(\Lambda(X),\cF[i])\cong \bH_{sNis}^i(X,\cF).
\]
Since ($(\P^\bullet,\P^{\bullet-1})$-inv) holds in $\mathbf{Ho}_{div}$ this shows $\cF$ is strictly $(\P^{\bullet},\P^{\bullet-1})$-invariant.
\vspace{0.1in}

Conversely, if $\cF$ is strictly $(\P^\bullet,\P^{\bullet-1})$-invariant,
then $\cF$ is a fibrant object with respect to the $(\P^\bullet,\P^{\bullet-1})$-local descent model structure on $\Co(\Shv_{sNis}(SmlSm/k,\Lambda))$.
Thus for every $X\in SmlSm/k$ and integer $i\in \Z$, there is an isomorphism
\[
\hom_{\mathbf{Ho}_{(\P^\bullet,\P^{\bullet-1})}}(\Lambda(X),\cF[i])\cong \bH_{sNis}^i(X,\cF).
\]
Since ($div$-des) holds in $\mathbf{Ho}_{(\P^\bullet,\P^{\bullet-1})}$, this shows $\cF$ is strictly dividing invariant.
\end{proof}

\begin{const}
\label{Pn.4}
Let us begin with a strictly $(\P^\bullet,\P^{\bullet-1})$-invariant bounded below complex $\cF$ of strict Nisnevich sheaves with log transfers on $SmlSm/k$.
Owing to Theorem \ref{Pn.3} we see that $\cF$ is strictly dividing invariant.
Applying Theorem \ref{Div.4}, this means that for every $X\in SmlSm/k$ and integer $i\in \Z$ there are isomorphisms
\begin{equation}
\label{Pn.4.1}
H_{dNis}^i(X,a_{dNis}^*\cF)\cong \colimit_{Y\in X_{div}^{Sm}} H_{sNis}^i(Y,\cF)\cong H_{sNis}^i(X,\cF).
\end{equation}
This shows that $a_{dNis}^*\cF$ is strictly $\boxx$-invariant.
Then using Construction \ref{SmlSm.3}, we see that $\iota_\sharp a_{dNis}^*\cF$ is strictly $\boxx$-invariant too.
In this way to obtain an object $\iota_\sharp a_{dNis}^*\cF$ of $\ldmeff$ from $\cF$.
Combining \eqref{SmlSm.3.2} and \eqref{Pn.4.1} we have that
\[
\hom_{\ldmeff}(M(X),\iota_\sharp a_{dNis}^*\cF)\cong H_{dNis}^i(X,a_{dNis}^*\cF)\cong H_{sNis}^i(X,\cF).
\]

Similarly, if $\cF$ is a strictly $(\P^\bullet,\P^{\bullet-1})$-invariant bounded below complex of strict \'etale (resp.\ Kummer \'etale) sheaves with log transfers, then there is an isomorphism
\begin{gather}
\label{Pn.4.2}
\hom_{\ldmeffet}(M(X),\iota_\sharp a_{d\acute{e}t}^*\cF)\cong H_{s\acute{e}t}^i(X,\cF),
\\
\label{Pn.4.3}
\text{(resp.\ }\hom_{\ldmefflet}(M(X),\iota_\sharp a_{l\acute{e}t}^*\cF)\cong H_{k\acute{e}t}^i(X,\cF).\text{)}
\end{gather}
\end{const}

\newpage

\section{Comparison with Voevodsky's motives}
\label{section:cvvmotives}

Our main results in this section are the projective bundle theorem and a comparison result showing that Voevodsky's category of derived motives $\dmeff$ 
identifies with the full subcategory of $\A^1$-local objects in our derived category of log motives $\ldmeff$.

\subsection{Presheaves with transfers and direct image functors}
\begin{lem}
\label{A.4.4}
Let $\cF$ be an $\A^1$-invariant Nisnevich sheaf with transfer on $Sm/k$.
Let $X$ be a smooth scheme over $k$ and let $j\colon U\rightarrow X$ be an open immersion whose complement is a strict normal crossing divisor on $X$. 
If $k$ is perfect, then for any $i\in \Z$, there is an isomorphism
\[
H_{Nis}^i(X,j_*j^*\cF)
\cong H_{Nis}^i(U,\cF).
\]
\end{lem}
\begin{proof}
Let $Z_1,\ldots,Z_r$ be the divisors forming the complement of $j$, and set
\[
U_i:=X-(Z_{i+1}\cup \cdots \cup Z_r).
\]
Now let $j_i\colon U_i\rightarrow U_{i+1}$ denote the induced open immersion. 
By the Leray spectral sequence
\[
E_2^{pq}=H^p(U_{i+1},R^qj_{i*}j_i^*\cF) \Rightarrow H^{p+q}(U_i,\cF),
\]
it suffices to show that 
$$
R^qj_{i*}j_i^*\cF=0
$$ 
for any $q>0$. 
Thus we reduce to the case when $X-U$ is a smooth divisor on $X$.
Furthermore, we reduce to the case when $X$ is the Hensel local scheme of a smooth scheme at a point.
Let $c\colon X-U\rightarrow X$ denote the complement of $j\colon U\rightarrow X$.
\vspace{0.1in}

For every $q\in \Z$, let $H_{Nis}^q(\cF)$ denote the presheaf given by 
\[
X\mapsto H_{Nis}^q(X,\cF).
\]
Owing to \cite[Theorem 24.1]{MVW} the presheaf $H_{Nis}^q(\cF)$ is $\A^1$-invariant.
By \cite[Proposition 23.5, Example 23.8, Theorem 23.12]{MVW} we obtain the vanishing 
\[
R^qj_*j^*\cF
\cong 
c_*H_{Nis}^q(\cF)_{(X\times \A^1,Z\times \A^1)}
\cong 
c_*( H_{Nis}^q(\cF)_{-1})_{Nis}
=
0
\]
for $q>0$ as in the proof of \cite[Proposition 24.3]{MVW}.
This allows us to proceed as in the proof of \cite[Proposition 24.3]{MVW}.
\end{proof}

\begin{df}
\label{A.4.19}
Suppose $\cF$ is a presheaf on a category $\cC$ equipped with a topology $t$. 
We denote the $t$-sheafification of $\cF$ by \index[notation]{Ft $\cF_t$}
\[\cF_t
:=
a_t\cF.
\]
\end{df}

\begin{df}[{cf.\ \cite[\S 4.2]{MR1764200}}]
\label{A.4.20} \index[notation]{FXWZ @ $\cF_{(X,W,Z)}$}
Let $X$ be a smooth scheme over $k$ with a strict normal crossing divisor $W$. 
Let $Z$ be an irreducible component of $W$ with corresponding closed immersion $i\colon Z\rightarrow X$.
For a presheaf $\cF$ on $X$, we set
\[
\cF_{(X,W,Z)}:=i^*(\coker(u_*u^*\cF\rightarrow v_*v^*\cF))_{Zar}.
\]
Here $u\colon X-\overline{(W-Z)}\rightarrow X$ and $v\colon X-W\rightarrow X$ are the naturally induced open immersions.
\end{df}

There is an exact sequence
\begin{equation}
\label{A.4.20.1}
(u_*u^*\cF)_{Zar}\rightarrow (v_*v^*\cF)_{Zar}\rightarrow i_*\cF_{(X,W,Z)}\rightarrow 0
\end{equation}
of Zariski sheaves on $X$ as in \cite[Remark 4.10]{MR1764200}.

\begin{lem}
\label{A.4.25}
With the above notations, let $f\colon X'\rightarrow X$ be an \'etale morphism of smooth schemes over $k$ such that $Z':=f^{-1}(Z)\rightarrow Z$ is an isomorphism.
For $W':=f^{-1}(W)$ there is a naturally induced isomorphism 
\[
\cF_{(X,W,Z)}\xrightarrow{\cong} \cF_{(X',W',Z')}
\]
of sheaves on $Z$ if $\cF$ is an $\A^1$-invariant presheaf with transfers.
\end{lem}
\begin{proof}
First note that $W'$ is a strict normal crossing divisor on $X'$ and $Z'$ is an irreducible compoenent of $Z'$ since $f$ is \'etale, so we can use the notation $\cF_{(X',W',Z')}$.
Let 
$$u'\colon X'-\overline{(W'-Z')}\rightarrow X'$$ 
and 
$$v'\colon X'-W'\rightarrow X'$$ 
be the naturally induced open immersions, and let $U$ be an open subscheme of $X$.
For $U':=U\times_X X'$ we have the commutative diagram with exact rows
\[
\begin{tikzcd}[column sep=tiny]
\cF(U\times_X (X-\ol{(W-Z)}))\arrow[d]\arrow[r]&\cF(U\times_X (X-W))\arrow[d]\arrow[r]& \coker(u_*u^*\cF\rightarrow v_*v^*\cF)(U)\arrow[r]\arrow[d,"\varphi"]&0\\
\cF(U'\times_{X'}(X'-\ol{(W'-Z')}))\arrow[r]&\cF(U'\times_{X'}(X'-W'))\arrow[r]&\coker(u_*'u'^*\cF\rightarrow v_*'v'^*\cF)(U')\arrow[r]&0.
\end{tikzcd}
\]
Let $a$ be an element of $\coker(u_*u^*\cF\rightarrow v_*v^*\cF)(U)$.
If $\varphi(a)=0$, then Zariski locally on $U$, we have $a=0$ by \cite[Proposition 4.12(1)]{MR1764200}.
On the other hand, suppose $b$ is an element of $\coker(u_*'u'^*\cF\rightarrow v_*'v'^*\cF)(U')$.
Then, Zariski locally on $U$, $b$ is in the image of $\varphi$ by \cite[Proposition 4.12(2)]{MR1764200}.
\end{proof}

\begin{lem}
\label{A.4.26}
Let $Z$ be a smooth scheme over $k$ with a strict normal crossing divisor $Y$, and let $j\colon Z-Y\rightarrow Z$ be the open immersion.
For any $\A^1$-invariant presheaf $\cF$ with transfers, there is a canonical isomorphism 
\[
(j_*j^*\cF_{-1})_{Zar}\cong \cF_{(Z\times \A^1,(Z\times \{0\})\cup (Y\times \A^1),Z\times \{0\})}
\]
of Zariski sheaves on $Z$.
\end{lem}
\begin{proof}
We shall argue as in the proof of \cite[Proposition 4.11]{MR1764200} (see also the proof of \cite[Proposition 23.10]{MVW}).
Fix a point $z$ of $Z$.
Let $U$ be an affine open neighborhood of $z$ in $Z$, and let $V$ be an affine open neighborhood of $U\times \{0\}$ in $U\times \A^1$.
Following the proof of \cite[Proposition 23.10]{MVW}, 
there exists a neighborhood $U'$ of $z$ in $U$ and an affine open subscheme $T$ of $\P_k^1$ such that $U'\times T$ contains both $\P_{U'}^1-V'$ and 
$U'\times \{0\}$ where $V':=(U'\times \A^1)\times_{U\times \A^1}V$.
We form the fiber product $Y_{U'}:=Y\times_Z U'$.
\vspace{0.1in}

Since a Cartier divisor gives $Y$, we may further assume that 
$U'-Y_{U'}$ is affine,
and hence $(U'-Y_{U'})\times T$ is affine.
Setting $V'':=V'\times_{U'}(U'-Y_{U'})$ we deduce the standard triple (in the sense of \cite[Definition 11.5]{MVW})
\[
T_{U',V'}:=(\P_{U'-Y_{U'}}^1,\P_{U'-Y_{U'}}^1-V'',(U'-Y_{U'})\times \{0\}),
\]
see the second paragraph in the proof of \cite[Proposition 23.10]{MVW}.
Next, we form 
$$
T_{U'}:=(\P_{U'-Y_{U'}}^1,(U'-Y_{U'})\times \{\infty\},(U'-Y_{U'})\times \{0\}) \text{ and } V_0'':=V''-(U'\times \{0\}).
$$

Following the proof of \cite[Proposition 23.10]{MVW} we see the identity on $\P_{U'}^1$ is a finite morphism of standard triples $T_{U',V'}\rightarrow T_{U'}$ 
in the sense of \cite[Definition 21.1]{MVW}.
Using \cite[Theorem 21.6]{MVW}, we obtain a split exact sequence
\[
0 \rightarrow \cF((U'-Y_{U'})\times \A^1)\rightarrow \cF((U'-Y_{U'})\times (\A^1-\{0\}))\oplus \cF(V'')\rightarrow \cF(V_0'')\rightarrow 0.
\]
Since $\cF$ is homotopy invariant, the morphism 
$$
\cF((U'-Y_{U'})\times \A^1)\rightarrow \cF((U'-Y_{U'})\times (\A^1-\{0\}))
$$ 
is injective.
This implies that the morphism $\cF(V'')\rightarrow \cF(V_0'')$ is injective and
\[
(j_*j^*\cF_{-1})(U')\cong \cF(V_0'')/\cF(V'').
\]
The right side is $\coker(u_*u^*\cF\rightarrow v_*v^*\cF)(V')$, 
where
\[
u\colon (Z\times \A^1)-(Y\times \A^1)\rightarrow Z\times \A^1,
\; 
v\colon Z\times \A^1-((Z\times \{0\}\cup (Y\times \A^1))\rightarrow Z\times \A^1
\]
are the naturally induced open immersions.
We are done by passing to the limit over all $U'$ and $V'$.
\end{proof}

\begin{lem}
\label{A.4.27}
Let $X$ be a scheme smooth over $k$ with a strict normal crossing divisor $W$, 
let $Z$ be an irreducible component of $W$, and let 
$$
j\colon (X-\overline{(W-Z)})\times_X Z\rightarrow Z
$$ 
be the naturally induced open immersion.
For any $\A^1$-invariant presheaf $\cF$ with transfers, 
Zariski locally on $X$, 
there is an isomorphism 
\[
(j_*j^*\cF_{-1})_{Zar}\cong \cF_{(X,W,Z)}
\]
of Zariski shaves on $X$.
\end{lem}
\begin{proof}
Since the question is Zariski local on $X$, 
we may assume there exists an \'etale morphism $f\colon X\rightarrow \A_k^n$ of schemes over $k$ such that $W$ is the preimage of the union of axes in $\A_k^n$ and 
$Z=f^{-1}(\{0\}\times \A_k^{n-1})$, 
see Definition \ref{A.3.17}.
\vspace{0.1in}

Consider the morphism of closed pairs 
$$
(Z,X)\leftarrow (Z,X_2)\rightarrow (Z,X_1)
$$ 
in Construction \ref{A.3.16}.
Note that $X_1=Z\times \A^m$ where $m$ is the codimension of $Z$ in $X$.
Using Lemma \ref{A.4.25} twice, we can replace $(Z,X)$ by $(Z,X_1)$.
In this case, 
$W$ becomes the strict normal crossing divisor 
$$
(Z\times \{0\})\cup ((Z\cap (\overline{(W-Z)})\times \A^1)
$$ 
on $X_1=X\times \A^1$.
We are done by Lemma \ref{A.4.26}.
\end{proof}

\begin{rmk}
If $W=Z$, then Lemma \ref{A.4.27} is a special case of \cite[Theorem 4.14]{MR1764200}.
\end{rmk}

\begin{lem}
\label{A.4.17}
Let $\cF$ be an $\A^1$-invariant presheaf with transfer on $Sm/k$.  
Let $X$ be a smooth scheme over $k$ and let $j\colon U\rightarrow X$ be an open immersion whose complement is a strict normal crossing divisor $Z$ on $X$. 
If $k$ is perfect, then the sheafification morphism $\cF\rightarrow \cF_{Nis}$ induces an isomorphism
\begin{equation}
\label{A.4.17.1}
(j_*j^*\cF)_{Nis}\xrightarrow{\cong} (j_*j^*\cF_{Nis})_{Nis}\cong j_*j^*(\cF_{Nis}).
\end{equation}
\end{lem}
\begin{proof}
Assume the sheafification morphism $\cF\rightarrow \cF_{Zar}$ induces an isomorphism
\begin{equation}
\label{A.4.17.2}
(j_*j^*\cF)_{Zar}\xrightarrow{\cong} (j_*j^*\cF_{Zar})_{Zar}\cong j_*j^*(\cF_{Zar}).
\end{equation}
The functors $j^*$ and $j_*$ map Nisnevich sheaves to Nisnevich sheaves, and $\cF_{Zar}\cong \cF_{Nis}$ by \cite[Proposition 5.5]{MR1764200}. 
Thus $j_*j^*(\cF_{Zar})$ is a Nisnevich sheaf on $Sm/X$, so $(j_*j^*\cF)_{Zar}$ is a Nisnevich sheaf on $Sm/X$ due to \eqref{A.4.17.2}.
It follows that \eqref{A.4.17.1} is an isomorphism. 
Therefore it suffices to show that \eqref{A.4.17.2} is an isomorphism.
\vspace{0.1in}

We proceed by induction on the number $r$ of connected components of $Z$.
If $r=0$, then $j$ is an isomorphism, and there is nothing to prove.
Suppose $r\geq 1$, and let $Z_1,\ldots,Z_r$ be the irreducible components of $Z$.
For each $i$, set $S_i:=Z_1\cup \cdots \cup Z_i$, and let
\[
u_i\colon X-S_i\rightarrow X,\;j_i\colon X-S_{i+1}\rightarrow X-S_i,\; z_i\colon Z_i\rightarrow X
\]
be the naturally induced immersions.
\vspace{0.1in}

For any open immersion $v\colon V'\rightarrow V$ of schemes, 
the functor $v^*$ commutes with Zariski sheafification.
Furthermore, we have
$v^*w_*\cong w_*'v'^*$ where $v'\colon V'\times_V W\rightarrow W$ and $w\colon V'\times_V W\rightarrow W$ are the projections.
Hence the problem is Zariski local on $X$.
Moreover, 
by \cite[Corollary 4.19, Proposition 4.34]{MR1764200}, Lemma \ref{A.4.27} and induction, 
we may assume there are isomorphisms 
\[
\cF_{(X,S_{i+1},Z_{i+1})}\cong (j_{i*}j_i^*\cF_{-1})_{Zar}\cong j_{i*}j_i^*((\cF_{Zar})_{-1})
\] 
and an exact sequence
\[
0\rightarrow u_{i*}u_i^*\cF_{Zar}\rightarrow u_{i+1,*}u_{i+1}^*(\cF_{Zar})\rightarrow {z_i}_*(j_{i*}j_i^*(\cF_{Zar})_{-1})\rightarrow 0
\]
of Zariski sheaves on $X$.
Using \eqref{A.4.20.1}, there is an induced commutative diagram with exact rows
\[
\begin{tikzcd}
&(u_{i*}u_i^*\cF)_{Zar}\arrow[d]\arrow[r]&(u_{i+1,*}u_{i+1}^*\cF)_{Zar}\arrow[d]\arrow[r]&z_{i*}((j_{i*}j_i^*\cF_{-1})_{Zar})\arrow[r]\arrow[d]&0\\
0\arrow[r]&u_{i*}u_i^*(\cF_{Zar})\arrow[r]&u_{i+1,*}u_{i+1}^*(\cF_{Zar})\arrow[r]&z_{i*}(j_{i*}j_i^*(\cF_{Zar})_{-1})\arrow[r]&0,
\end{tikzcd}
\]
where the Zariski sheafification functor induces the vertical morphisms. 
Here the left vertical morphism is an isomorphism by induction, 
and the right vertical morphism is an isomorphism since $(j_{i*}j_i^*\cF_{-1})_{Zar}\cong j_{i*}j_i^*((F_{-1})_{Zar})$ by induction on $r$ and $(\cF_{-1})_{Zar}\cong (\cF_{Zar})_{-1}$.
The snake lemma implies that also the middle vertical morphism is an isomorphism.
\vspace{0.1in}

The above paragraph shows that $(u_{i*}u_i^*\cF)_{Zar}\rightarrow u_{i*}u_i^*(\cF_{Zar})$ is an isomorphism by induction on $i$.
When $i=0$, this is trivial since $u_0$ is the identity.
Specializing to the case $i=r$ finishes the proof because $u_r=j$.
\end{proof}

We refer to \eqref{eqn::omegaadjunction} the definition of the functor
\[
\omega^*\colon \Shvtrkl\rightarrow \Shvltrkl.
\]

\begin{lem}
\label{A.4.18}
Let $\cF$ be a strictly $\A^1$-invariant complex of Nisnevich sheaves with transfer on $Sm/k$.  
Let $X$ be a smooth scheme over $k$ and let $j\colon U\rightarrow X$ be an open immersion whose complement is a strict normal crossing divisor on $X$. 
If $k$ is perfect, then for every $i\in \Z$, there is an isomorphism
\[
\bH_{Nis}^i(X,j_*j^*\cF)\cong \bH_{Nis}^i(U,\cF).
\]
In particular, for every $Y\in SmlSm/k$ and $i\in \Z$, there is an isomorphism
\[
\bH_{sNis}^i(Y,\omega^*\cF)\cong \bH_{Nis}^i(Y-\partial Y,\cF).
\]
\end{lem}
\begin{proof}
For any complex $\cG$ of presheaves and $i\in Z$, we let $H^i(\cG)$ denote the associated cohomology presheaf. 
The Nisnevich cohomological dimensions of $X$ and $U$ are finite, 
see e.g., \cite[Example 11.2]{MVW}.
Hence by \cite[5.7.9]{weibel_1994}, there are two strongly convergent hypercohomology spectral sequences $E$ and $E'$ given by
\[
E_2^{pq}=H_{Nis}^p(X,(H^q(j_*j^*\cF))_{Nis})\Rightarrow \bH^{p+q}(X,j_*j^*\cF),
\]
and 
\[
E_2'^{pq}=H_{Nis}^p(U,(H^q(\cF))_{Nis})\Rightarrow \bH^{p+q}(U,\cF).
\]
There is a naturally induced morphism of spectral sequences $E\rightarrow E'$. 
\vspace{0.1in}

Since $j_*$ and $j^*$ are both exact functors between presheaf categories, there is an isomorphism
\[
H^q(j_*j^*\cF)\cong j_*j^*(H^q(\cF)).
\] 
Thus $(H^q(j_*j^*\cF))_{Nis}\cong j_*j^*(H^q(\cF)_{Nis})$ by Lemma \ref{A.4.17}. 
We deduce the maps on $E_2$ terms are isomorphisms by Lemma \ref{A.4.4} since $(H^q(\cF))_{Nis}$ is an $\A^1$-invariant Nisnevich sheaf by \cite[Proposition 14.8]{MVW}. 
Owing to the strong convergence of both the spectral sequences above, the induced map between the filtered target groups is also an isomorphism.
\end{proof}

\begin{lem}
\label{A.4.7}
Let $\cF$ be a strictly $\A^1$-invariant complex of Nisnevich sheaves with transfers on $Sm/k$. 
For every $X\in lSm/k$ and integer $i\in \Z$ there is a naturally induced isomorphism
\begin{equation}
\label{A.4.7.7}
\bH_{dNis}^i(X,\omega^*\cF)\cong \bH_{Nis}^i(X-\partial X,\cF).
\end{equation}
\end{lem}
\begin{proof}
If $Y\rightarrow X$ is a log modification, there is a naturally induced isomorphism 
$$
Y-\partial Y\xrightarrow{\cong} X-\partial X.
$$
Due to Lemma \ref{A.5.70}, we deduce the isomorphism 
$$
a_{dNis}^*\Zltr(Y)\xrightarrow{\cong} a_{dNis}^*\Zltr(X).
$$
Moreover, we conclude there are isomorphisms
\begin{align*}
\bH_{dNis}^i(X,\omega^*\cF)
& \cong
\hom_{\mathbf{D}(\Shvltrkl)}(a_{dNis}^*\Zltr(X),\omega^*\cF) \\
& \cong
\hom_{\mathbf{D}(\Shvltrkl)}(a_{dNis}^*\Zltr(Y),\omega^*\cF) \\
& \cong
\bH_{dNis}^i(Y,\omega^*\cF).
\end{align*}
This shows both sides of \eqref{A.4.7.7} are invariant under log modifications.
By Proposition \ref{A.3.19}, we may assume that $\underline{X}$ is smooth, and $\partial X$ is a strict normal crossing divisor.
\vspace{0.1in}

Choose a fibrant resolution $f\colon \cF\rightarrow \cG$ in the model category $\Co(\Shvtrkl)$ with respect to the descent model structure.
This implies that $f$ is a quasi-isomorphism and
\begin{equation}
\label{A.4.7.4}
\bH_{Nis}^i(X-\partial X,\cF)\cong H^i(\cG(X-\partial X)).
\end{equation}
Let us show that $\omega^*\cG$ is a fibrant object in the model category $\Co(\Shvltrkl)$ for the descent model structure.
By definition of $\omega^*$, for every integer $i\in \Z$, there is an isomorphism 
\begin{equation}
\label{A.4.7.1}
H^i(\omega^*\cG(X))\cong H^i(\cG(X-\partial X)).
\end{equation}
We need to show that for every dividing Nisnevich hypercover $\mathscr{X}\rightarrow X$  we have an isomorphism
\[
H^i(\omega^*\cG(X))\rightarrow H^i({\rm Tot}^\pi(\omega^*\cG(\mathscr{X}))),
\]
see \eqref{A.8.28.1}.
\vspace{0.1in}

Since $\mathscr{X}-\partial \mathscr{X}\rightarrow X-\partial X$ is a Nisnevich hypercover,
the assumption that $\cG$ is fibrant in $\Co(\Shvtrkl)$ implies
\begin{equation}
\label{A.4.7.2}
H^i(\cG(X-\partial X))\cong H^i({\rm Tot}^\pi \cG(\mathscr{X}-\partial \mathscr{X})).
\end{equation}
Using \eqref{A.4.7.1} we have that
\begin{equation}
\label{A.4.7.3}
H^i({\rm Tot}^\pi(\cG(\mathscr{X}-\partial \mathscr{X})))
\cong
H^i({\rm Tot}^\pi(\omega^*\cG(\mathscr{X}))).
\end{equation}
By combining \eqref{A.4.7.1}, \eqref{A.4.7.2}, and \eqref{A.4.7.3} we deduce that $\omega^*\cG$ is fibrant.
Hence there is an isomorphism
\begin{equation}
\label{A.4.7.5}
\bH_{dNis}^i(X,\omega^*\cG)\cong H^i(\omega^*\cG(X)).
\end{equation}

Due to Lemma \ref{A.4.18} there are isomorphisms
\[
\bH_{sNis}^i(X,\omega^*\cF)\cong \bH_{Nis}^i(X-\partial X,\cF)
\]
and
\[
\bH_{sNis}^i(X,\omega^*\cG)\cong \bH_{Nis}^i(X-\partial X,\cG).
\]
Since $f$ is a quasi-isomorphism, we deduce that $\omega^*f$ is a quasi-isomorphism for the strict Nisnevich topology.
Thus $\omega^*f$ is a quasi-isomorphism for the dividing Nisnevich topology.
It follows that
\begin{equation}
\label{A.4.7.6}
\bH_{dNis}^i(X,\omega^*\cF)\cong \bH_{dNis}^i(X,\omega^*\cG).
\end{equation}
To conclude the proof we combine \eqref{A.4.7.4}, \eqref{A.4.7.1}, \eqref{A.4.7.5}, and \eqref{A.4.7.6}.
\end{proof}

\begin{lem}
\label{ketcomp.9}
Let $\cF$ be a Nisnevich (resp.\ an \'etale) sheaf on $Sm/k$.
Then $\omega^*\cF$ is a dividing Nisnevich (resp.\ log \'etale) sheaf on $lSm/k$.
\end{lem}
\begin{proof}
Let $p\colon U\rightarrow X$ be a dividing Nisnevich (resp.\ log \'etale) cover.
We need to show that the naturally induced sequence
\[
\omega^*\cF(X)\rightarrow \omega^*\cF(U)\rightrightarrows \omega^*\cF(U\times_X U)
\]
is exact.
This sequence is isomorphic to
\[
\cF(X-\partial X)\rightarrow \cF(U-\partial U)\rightrightarrows\cF(U-\partial U\times_{X-\partial X} U-\partial U),
\]
which is exact since $U-\partial U\rightarrow X-\partial X$ is a Nisnevich (resp.\ an \'etale) cover and $\cF$ is a Nisnevich (resp.\ an \'etale) sheaf.
\end{proof}

\begin{prop}
\label{A.4.13}
Let $\cF$ be a strictly $\A^1$-invariant complex of Nisnevich sheaves with transfers on $Sm/k$.
Then $\omega^*\cF$ is a both strictly $(\P^\bullet,\P^{\bullet-1})$-invariant and strictly $\A^1$-invariant complex of dividing Nisnevich sheaves with log transfers.
\end{prop}
\begin{proof}
Lemma \ref{ketcomp.9} shows that $\omega^*\cF$ is a complex of dividing Nisnevich sheaves with transfers.
For $X\in lSm/k$, owing to Lemma \ref{A.4.7} there are isomorphisms
\begin{gather*}
\bH_{dNis}^i(X,\omega^*\cF)\cong \bH_{Nis}^i(X-\partial X,\cF),\\
\bH_{dNis}^i(X\times (\P^n,\P^{n-1}),\omega^*\cF)\cong \bH_{Nis}^i((X-\partial X)\times \A^{n},\cF),
\\
\bH_{dNis}^i(X\times \A^1,\omega^*\cF)\cong \bH_{Nis}^i((X-\partial X)\times \A^1,\cF).
\end{gather*}
To conclude the proof, we use the assumption that $\cF$ is strict $\A^1$-invariant.
\end{proof}

\subsection{$\A^1$-local objects in $\ldmeff$}
We denote by 
\[
v\colon \LCor\rightarrow \LCor[(\cA\cB l/k)^{-1}]
\]
the evident localization functor, where $\cA\cB l/k$ is the class of admissible blow-ups in Definition \ref{A.3.1}.
As observed in Definition \ref{A.1.16} there exist adjoint functors
\[
\begin{tikzcd}
\Co^{\leq 0}(\Psh(\LCor,\Lambda))\arrow[rr,shift left=1.5ex,"v_\sharp "]
\arrow[rr,"v^*" description,leftarrow]\arrow[rr,shift right=1.5ex,"v_*"']&&\Co^{\leq 0}(\Psh(\LCor[(\cA\cB l/k)^{-1}],\Lambda)).
\end{tikzcd}
\]
Here $v_\sharp$ is left adjoint to $v^*$ and $v^*$ is left adjoint to $v_*$.

\begin{prop}
\label{A.4.2}
Assume that $k$ admits resolution of singularities. 
Let $X$ be a smooth and proper scheme over $k$, and let $Y$ be an fs log scheme log smooth over $k$. 
Then there is a naturally induced isomorphism 
\[
v^*v_\sharp \Zltr(X)(Y)\cong \Zltr(X)(Y-\partial Y).
\]
\end{prop}
\begin{proof}
Let $Y'\rightarrow Y$ be an admissible blow-up.  
The induced morphism 
\[
Y'-\partial Y'\rightarrow Y-\partial Y
\]
of schemes over $k$ is an isomorphism.
Thus we can form the composite homomorphism of $\Lambda$-modules
\[
\varphi_{Y'}\colon \Zltr(X)(Y')\rightarrow \Zltr(X)(Y'-\partial Y')\stackrel{\sim}\rightarrow \Zltr(X)(Y-\partial Y).
\]
From the formula \eqref{A.1.4.1} we have the isomorphism
\[
v^*v_\sharp \Zltr(X)(Y)\cong \colimit_{Y'\in (\cA\cB l/k) \downarrow Y} \Zltr(X)(Y').
\]
Collecting the $\varphi_{Y'}$'s, we get the induced homomorphism of $\Lambda$-modules
 \[
 \varphi\colon v^*v_\sharp \Zltr(X)(Y)\rightarrow \Zltr(X)(Y-\partial Y).
 \]
Let $V\in \lCor(Y',X)$ be a finite log correspondence. 
If $\varphi_{Y'}(V)=0$, then $V=0$ since $V$ is the closure of $\varphi_{Y'}(V)$ in $Y'\times X$. 
Thus $\varphi$ is injective. 
It remains to show that $\varphi$ is surjective.
\vspace{0.1in}

Let $W\in \Cor(Y-\partial Y,X)$ be an elementary correspondence.
By Raynaud-Gruson \cite[Th\'eor\`eme 5.7.9]{RaynaudGruson}, 
there is an algebraic space $\underline{Y'}$ and a blow-up $\underline{f}\colon\underline{Y'}\rightarrow \underline{Y}$ such that $\underline{f}$ is an isomorphism on $Y-\partial Y$ and the closure $\underline{W'}$ of $\underline{W}$ in $\underline{Y'}\times \underline{X}$ is flat over $\underline{Y'}$.
The algebraic space $\underline{Y'}$ is a scheme because it is a blow-up of a scheme. 
Further, 
$\underline{W'}$ is finite over $\underline{Y'}$ since $\underline{Y'}\times \underline{X}$ is proper over $\underline{Y'}$.
\vspace{0.1in}

By resolution of singularities, 
there is a blow-up $g\colon \ul{Y''}\rightarrow \ul{Y'}$ such that $\ul{Y''}$ is smooth over $k$ and the complement of $(\underline{f}\circ \underline{g})^{-1}(Y-\partial Y)$ in $\underline{Y''}$ consists of strict normal crossing divisors $Z_1',\ldots,Z_r'$.
Setting $Y'':=(\underline{Y''},Z_1'+\cdots+Z_r')$ the induced morphism $Y''\rightarrow Y$ of fs log schemes log smooth over $k$ is an admissible blow-up.
Since the closure $\underline{W''}$ of $\underline{W}$ in $\underline{Y''}\times \underline{X}$ is a closed subscheme of $\underline{W'}\times_{\underline{Y'}}\underline{Y''}$, $\underline{W''}$ is finite over $\underline{Y''}$.
It follows that $W$ can be extended to a finite correspondence from $\underline{Y''}$ to $X$.
Since $X$ has the trivial log structure, 
such a finite correspondence gives a finite log correspondence from $Y''$ to $X$ by Example \ref{A.5.31}(2). 
Thus $\varphi$ is surjective.
\end{proof}

\begin{prop}
\label{A.4.30}
Assume that $k$ admits resolution of singularities.
For every smooth and proper scheme $X$ over $k$, there is an isomorphism 
$$
a_{dNis}^*\Zltr(X)\cong \omega^*\Ztr(X)
$$
in $\ldmeff$.
Likewise, there are isomorphisms 
$$
a_{d\acute{e}t}^*\Zltr(X)\cong \omega^*\Ztr(X)
\text{ and } 
a_{l\acute{e}t}^*\Zltr(X)\cong \omega^*\Ztr(X)
$$ 
in $\ldmeffet$ and $\ldmefflet$, 
respectively.
\end{prop}
\begin{proof}
We only treat the case of dividing Nisnevich coverings since the proofs are similar.
Owing to Lemma \ref{A.5.45} we have an isomorphism
\[
a_{dNis}^*v^*v_\sharp \Zltr(X) \cong v^*v_\sharp a_{dNis}^*\Zltr(X).
\]
Moreover, Proposition \ref{A.4.2} implies
\[
a_{dNis}^*v^*v_\sharp \Zltr(X)\cong a_{dNis}^*\omega^*\Ztr(X), 
\]
and from Remark \ref{A.3.20} we deduce
\[
 v^*v_\sharp a_{dNis}^*\Zltr(X)\cong a_{dNis}^*\Zltr(X).
\]
To conclude, note that $\omega^*\Ztr(X)$ is a dividing Nisnevich sheaf, see Lemma \ref{ketcomp.9}.
\end{proof}

For a bounded above complex $\cF$ of presheaves with log transfers, recall the Suslin double complex $C_*\cF$ from Definition \ref{A.1.16}.
If $\cF$ is a presheaf with log transfers, then for $U\in lSm/k$, we have 
\[
C_*\cF(U):= (\cdots \rightarrow \colimit_{Y\in (\cA \cB l/k)\downarrow (U\times \Delta_{\boxx}^1)}\cF(Y)\rightarrow \colimit_{Y\in (\cA \cB l/k)\downarrow U}\cF(Y)).
\]
For a bounded above complex $\cF$ of presheaves with log transfers, we also have the Suslin double complex $C_*^{\A^1}\cF$. 
If $\cF$ is a presheaf with transfers, then for $U\in Sm/k$, 
\[
C_*^{\A^1}\cF(U):=(\cdots \rightarrow \cF(U\times \Delta_{\A^1}^1)\rightarrow \cF(U)).
\]

\begin{prop}
\label{A.4.12}
Suppose $\cF$ is a bounded above complex of presheaves with log transfers on $Sm/k$. 
Then there is a quasi-isomorphism of total complexes
\[
{\rm Tot}(C_*\omega^*\cF)\cong {\rm Tot}(\omega^* C_*^{\A^1} \cF).
\]
\end{prop}
\begin{proof}
Using the technique in \cite[Lemma 9.12]{MVW}, we reduce to the case when $\cF$ is a presheaf with log transfers.
For every $X\in lSm/k$, admissible blow-up $Y\rightarrow \Delta_{\boxx}^n\times X$, $m\in \Z$, and $n\in \N$,
we have  
\[
\omega^*\cF(Y)=\cF(Y-\partial Y)\cong \cF(\Delta_{\A^1}^n\times (X-\partial X)).
\] 
Thus for the internal hom object 
\[
\underline{\hom}(\Delta_{\A^1}^n,\cF)
\]
mapping $U\in Sm/k$ to $\cF(U\times \Delta_{\A^1}^n)$, we have that
\[
\colimit_{Y\in (\cA\cB l/k) \downarrow(\Delta_{\boxx}^n\times X)}\omega^*\cF(Y)
= 
\cF(\Delta_{\A^1}^n\times (X-\partial X))
\cong 
(\omega^* \underline{\hom}(\Delta_{\A^1}^n,\cF))(X).
\]
Since
\[
C_*\omega^*\cF(U):= (\cdots \rightarrow \colimit_{Y\in (\cA \cB l/k)\downarrow (U\times \Delta_{\boxx}^1)}\omega^*\cF(Y)\rightarrow \colimit_{Y\in (\cA \cB l/k)\downarrow U}\omega^*\cF(Y)),
\]
and
\[
\omega^*C_*^{\A^1}\cF(U):=(\cdots \rightarrow \omega^*\underline{\hom}(\Delta_{\A^1}^1,\cF)(X)\rightarrow \omega^*\underline{\hom}(\Delta_{\A^1}^0,\cF)(X)),
\]
we deduce the desired isomorphism.
\end{proof}

\begin{prop}
\label{A.4.29}
Assume that $k$ admits resolution of singularities. 
Let $X$ be a scheme smooth over $k$.
Then there are isomorphisms 
\[
\omega^*\Ztr(X)\cong 
C_*\omega^*\Ztr(X)\cong 
\omega^*C_*^{\A^1}\Ztr(X)
\]
in $\ldmeff$, $\ldmeffet$, and $\ldmefflet$.
\end{prop}
\begin{proof}
For the first isomorphism, apply Proposition \ref{A.1.15}(4) and Remark \ref{A.3.20}.
For the second isomorphism, apply Proposition \ref{A.4.12}.
\end{proof}

\begin{prop}
\label{A.4.9}
Assume that $k$ admits resolution of singularities. 
Let $X$ be a scheme smooth over $k$, and let $Y$ be an fs log schemes log smooth over $k$. 
Then for every integer $i\in \Z$ there is a naturally induced isomorphism
\[
\hom_{\ldmeff}(M(Y),\omega^* \Ztr(X)[i])
\cong 
\hom_{\dmeff}(M(Y-\partial Y),M(X)[i]).
\]
\end{prop}
\begin{proof}
Due to Proposition \ref{A.4.29}, we have an isomorphism
\[
\hom_{\ldmeff}(M(Y),\omega^*\Ztr(X)[i])\\
\cong \hom_{\ldmeff}(M(Y),\omega^*C_*^{\A^1}\Ztr(X)[i]).
\]
Since the Suslin complex $C_*^{\A^1}\Ztr(X)$ is strictly $\A^1$-invariant for the Nisnevich topology by \cite[Corollary 14.9]{MVW}, 
we deduce $\omega^*C_*^{\A^1}\Ztr(X)$ is strictly $\boxx$-invariant for the dividing Nisnevich topology by Proposition \ref{A.4.13}.
Thus owing to Proposition \ref{A.5.13} there is an isomorphism
\[
\hom_{\ldmeff}(M(Y),\omega^*C_*^{\A^1}\Ztr(X)[i])\cong \bH_{dNis}^i(Y,\omega^*C_*^{\A^1}\Ztr(X)).
\]
From Lemma \ref{A.4.7} we obtain the isomorphism
\[
\bH_{dNis}^i(Y,\omega^*C_*^{\A^1}\Ztr(X))\cong \bH_{Nis}^i(Y-\partial Y,C_*^{\A^1}\Ztr(X)).
\]
We also have an isomorphism due to \cite[Proposition 14.16]{MVW}
\[
\bH_{Nis}^i(Y-\partial Y,C_*^{\A^1}\Ztr(X))\cong \hom_{\dmeff}(M(Y-\partial Y),M(X)[i]).
\]
To conclude, we combine the above isomorphisms.
%
%
\end{proof}

\begin{prop}
\label{A.4.8}
Assume that $k$ admits resolution of singularities. 
Let $X$ be a smooth and proper scheme over $k$, and let $Y$ be an fs log scheme log smooth over $k$. 
Then for every $i\in \Z$, there is a natural isomorphism
\[
\hom_{\ldmeff}(M(Y)[i],M(X))
\cong 
\hom_{\dmeff}(M(Y-\partial Y)[i],M(X)).
\]
\end{prop}
\begin{proof}
Immediate from Propositions \ref{A.4.30} and \ref{A.4.9}.
\end{proof}

\begin{rmk}
We do not expect that Proposition \ref{A.4.8} holds for non-proper schemes.
\end{rmk}

Recall from \eqref{eq::omegasharp} the adjunction \index[notation]{omegasharpstar @ $\omega_\sharp, R\omega^*$}
\[
\omega_\sharp :\ldmeff \rightleftarrows \dmeff:R\omega^*.
\]

\begin{prop}
\label{A.4.15}
Assume that $k$ admits resolution of singularities. 
Let $X$ be a proper and smooth scheme over $k$.
Then the unit of  the adjunction ${\rm id}\rightarrow R\omega^*\omega_\sharp$ induces an isomorphism 
\[
M(X)\rightarrow R\omega^*\omega_\sharp M(X)
\]
in $\ldmeff$.
\end{prop}
\begin{proof}
Appealing to the generators $M(Y)[i]$ of $\ldmeff$,
where $Y\in lSm/k$ and $i\in \Z$,
it suffices to show there is an isomorphism
\[
\hom_{\ldmeff}(M(Y)[i],M(X))
\cong 
\hom_{\ldmeff}(M(Y)[i],R\omega^*M(X)).
\]

This follows from Proposition \ref{A.4.8} and the isomorphisms
\[
\begin{split}&\hom_{\ldmeff}(M(Y)[i],R\omega^*M(X))\\
\cong &\hom_{\dmeff}(\omega_\sharp M(Y)[i],M(X))\\
\cong &\hom_{\dmeff}(M(Y-\partial Y)[i],M(X)).
\end{split}
\]
\end{proof}

\begin{prop}
\label{A.4.11}
Let $Z_1,\ldots,Z_r$ be divisors forming a strict normal crossing divisor on a smooth scheme $X$ over $k$.
Suppose the smooth subscheme $Z$ of $X$ has strict normal crossing and is not contained in $Z_1\cup\cdots\cup Z_r$.
Let $E$ be the exceptional divisor of the blow-up $B_ZX$.
Then for 
\[
Y:=(X,Z_1+\cdots+Z_r)
\]
there is a commutative diagram in $\ldmeff$
\[
\begin{tikzcd}
M((B_Z Y,E)\rightarrow Y)\arrow[d,"\sim"']\arrow[r]&R\omega^*M((Y-\partial Y\up Z\rightarrow Y-\partial Y)\arrow[d,"\sim"]\\
MTh(N_ZY)\arrow[r]&R\omega^*(MTh(N_{Z-\partial Y\cap Z}(Y-\partial Y))).
\end{tikzcd}\]
The vertical morphisms are the Gysin isomorphisms, and the horizontal morphisms are induced by the unit of the adjunction ${\rm id}\rightarrow R\omega^*\omega_\sharp$.
\end{prop}
\begin{proof}
With the same notations in the proof of Theorem \ref{A.3.36},
we set 
\begin{gather*}
U:=(B_Z Y,E),
\;
V:=Y,
\;
U':=(B_{Z\times \boxx}(D_ZY),E^D),
\;
V':=B_{Z\times \boxx}(D_ZY),
\\
U'':=(B_Z(N_ZY),E^N),
\;
V'':=N_Z Y.
\end{gather*}
The diagram
\[
\begin{tikzcd}
M(U\rightarrow V)\arrow[d]\arrow[r]&R\omega^*M((U-\partial U)\rightarrow (V-\partial V))\arrow[d]\\
M(U'\rightarrow V')\arrow[d,leftarrow]\arrow[r]&R\omega^*M((U-\partial U)\rightarrow (V'-\partial V'))\arrow[d,leftarrow]\\
M(U''\rightarrow V'')\arrow[r]&R\omega^*M((U''-\partial U'')\rightarrow (V''-\partial V''))
\end{tikzcd}
\]
commutes because the horizontal morphisms are obtained from an adjunction. 
Our assertion follows from this.
\end{proof}

\begin{prop}
\label{A.4.31}
Let $X$ be an fs log scheme in $SmlSm/k$ with a rank $n$ vector bundle $\cE$.
View $\P(\cE)$ as the closed subscheme of $\P(\cE\oplus \cO)$ at infinity.
Then $\cF:=\cE-\partial \cE$ is a vector bundle over $X-\partial X$, 
and there is a commutative diagram
\begin{equation}
\label{A.4.31.1}
\begin{tikzcd}
MTh(\cE)\arrow[d]\arrow[r]&
R\omega^*MTh(\cF)\arrow[d]
\\
M(\P(\cE)\rightarrow \P(\cE\oplus \cO))\arrow[r]&
R\omega^* M(\P(\cF)\rightarrow \P(\cF\oplus \cO)).
\end{tikzcd}
\end{equation}
Here the vertical morphisms are obtained from \textup{Proposition \ref{A.3.34}} and \cite[Proposition 2.17, p.\ 112]{MV}, 
and the horizontal morphisms are induced by the unit of the adjunction ${\rm id}\rightarrow R\omega^*\omega_\sharp$.
\end{prop}
\begin{proof}
We form the fs log schemes $Y$, $Y'$ and the divisors $E$, $E'$ as in the proof of Proposition \ref{A.3.34}, 
and set
\[V:=(Y,E)-\partial (Y,E),
\;
V':=(Y',E')-\partial (Y',E').
\]
Then there is a commutative diagram
\[
\begin{tikzcd}
M((Y,E)\rightarrow \cE)\arrow[d]\arrow[r]&
M((Y',E')\rightarrow \P(\cE\oplus \cO))\arrow[r,leftarrow]\arrow[d]&
M(\P(\cE)\rightarrow \P(\cE\oplus \cO))\arrow[d]
\\
R\omega^*M(V\rightarrow \cF)\arrow[r]&
R\omega^*M(V'\rightarrow \P(\cF\oplus \cO))\arrow[r,leftarrow]&
R\omega^*M(\P(\cF)\rightarrow \P(\cF\oplus \cO)).
\end{tikzcd}
\]
In the proof of Proposition \ref{A.3.34} and \cite[Proposition 2.17, p.\ 112]{MV} it is shown that the horizontal morphisms in the right-hand side square are isomorphisms.
After inverting these, we obtain the required commutative diagram.
\end{proof}

\begin{thm}
\label{A.4.10}
Assume that $k$ admits resolution of singularities. 
Let $X$ be an fs log scheme log smooth over $k$. If the underlying scheme $\underline{X}$ is proper over $k$, 
then the unit of the adjunction $\id\rightarrow R\omega^*\omega_\sharp$ induces an isomorphism
\[
M(X)\rightarrow R\omega^*\omega_\sharp M(X).
\]
Therefore, for every $Y\in lSm/k$ and an integer $i\in \Z$, there is a natural isomorphism
\[
\hom_{\ldmeff}(M(Y)[i],M(X))
\cong 
\hom_{\dmeff}(M(Y-\partial Y)[i],M(X-\partial X)).
\]
\end{thm}
\begin{proof}
The question is dividing local on $X$, so we may assume that $X\in SmlSm/k$ by Proposition \ref{A.3.19}.
Then $\partial X$ is a strict normal crossing divisior formed by smooth divisors $Z_1,\ldots,Z_r$ on $X$, 
and we can form 
\begin{gather*}
X_1:=(\underline{X},Z_2+\cdots+Z_r),
\;
U:=X-\partial X,
\;
U_1:=X_1-\partial X_1,
\\
V:=Y-\partial Y,
\;
W_1:=Z_1-\partial X\cap Z_1.
\end{gather*}

The proof proceeds by induction on $r$. 
Proposition \ref{A.4.15} implies the claim for $r=0$. 
Assuming $r>0$, by Proposition \ref{A.4.31} there is a commutative diagram
\[
\begin{tikzcd}
MTh(N_{Z_1}X)\arrow[d]\arrow[r]&
R\omega^*MTh(N_{W_1}U)\arrow[d]
\\
M(\P(N_{Z_1}X)\rightarrow \P(N_{Z_1}X\oplus \cO))\arrow[r]&
R\omega^* M(\P(N_{W_1}U)\rightarrow \P(N_{W_1}U\oplus \cO)).
\end{tikzcd}
\]
The vertical morphisms are isomorphisms by Proposition \ref{A.3.34} and \cite[Proposition 2.17, p.\ 112]{MV}.
By induction, there are isomorphisms
\[
M(\P(N_{Z_1}X))\xrightarrow{\cong} R\omega^*M(\P(N_{W_1}U))
\]
and
\[
M(\P(N_{Z_1}X\oplus \cO))\xrightarrow{\cong} R\omega^*M(\P(N_{W_1}U\oplus \cO)).
\]
It follows that the lower horizontal morphism is an isomorphism, and hence the upper horizontal morphism is an isomorphism.
\vspace{0.1in}

Due to Proposition \ref{A.4.11} there is a commutative diagram
\[
\begin{tikzcd}
M(X\rightarrow X_1)\arrow[d]\arrow[r]& R\omega^*M(U\rightarrow U_1)
\arrow[d]\\
MTh(N_{Z_1}X) \arrow[r]&R\omega^*MTh(N_{W_1}U).
\end{tikzcd}
\]
The above shows the lower horizontal morphism is an isomorphism.
The vertical morphisms are isomorphisms by Theorem \ref{A.3.36} and \cite[Theorem 15.15]{MVW}, 
respectively.
This allows us to conclude for the upper horizontal morphism.
\vspace{0.1in}

By induction $M(X_1)\rightarrow R\omega^*M(U_1)$ is an isomorphism.
The five lemma implies $M(X)\rightarrow R\omega^*M(U_1)$ is an isomorphism.
\end{proof}

\begin{prop}
\label{A.4.16}
Assume that $k$ admits resolution of singularities. 
Then there is an adjunction
\[
R\omega^*: \dmeff\rightleftarrows \ldmeff:R\omega_*
\]
\end{prop}
\begin{proof}
By \cite[Proposition 1.3.19]{CD12} we need to show that $\dmeff$ admits a set of objects $\cG$ with the following properties.
\begin{enumerate}
\item For each $K\in \cG$, $K$ (resp.\ $R\omega^*K$) is compact in $\dmeff$ (resp.\ $\ldmeff$).
\item $\dmeff$ (resp.\ $\ldmeff$) is generated by $\cG$ (resp.\ $R\omega^*\cG$).
\item $\ldmeff$ is compactly generated.
\end{enumerate}

By Remark \ref{A.5.34} (resp.\ Proposition \ref{A.5.33}), $\dmeff$ (resp.\ $\ldmeff$) is generated by compact objects $M(X)[i]$ for $X\in Sm/k$ (resp.\ $X\in lSm/k$). 
Hence it remains to check that $R\omega^*M(X)$ is a compact object in $\ldmeff$.
\vspace{0.1in}

By resolution of singularities, 
there is a projective fs log scheme $\overline{X}$ in $SmlSm/k$ with $X\cong \overline{X}-\partial \overline{X}$.
Theorem \ref{A.4.10} shows there is an isomorphism 
$$
M(\overline{X})\xrightarrow{\cong} R\omega^*M(\overline{X}-\partial \overline{X})\cong R\omega^*M(X).
$$ 
Thus $R\omega^*M(X)$ is compact in $\ldmeff$ since the motive $M(\overline{X})$ is compact in $\ldmeff$.
\end{proof}

To summarize the above, we note there are adjunctions 
\[
\begin{tikzcd}
\ldmeff\arrow[rr,shift left=1.5ex,"\omega_\sharp "]\arrow[rr,"R\omega^*" description,leftarrow]\arrow[rr,shift right=1.5ex,"R\omega_*"']&&\dmeff.
\end{tikzcd}
\]

\begin{thm}
\label{A.4.14}
Assume that $k$ admits resolution of singularities. 
Then the functor
\[
R\omega^*\colon \dmeff\rightarrow \ldmeff
\]
is fully faithful.
\end{thm}
\begin{proof}
We need to show that the unit of the adjunction ${\rm id}\rightarrow R\omega_*R\omega^*$ is an isomorphism. 
By invoking the generators $M(X)[i]$ of $\dmeff$, 
where $X\in Sm/k$ and $i\in \Z$, 
we only need to show there is an isomorphism
\[
M(X)\rightarrow R\omega_*R\omega^*M(X).
\]
Again by invoking the generators, we reduce to showing there is an isomorphism
\[
\hom_{\dmeff}(M(Y)[i],M(X))\cong \hom_{\dmeff}(M(Y)[i],R\omega_*R\omega^*M(X))
\]
for any $Y\in Sm/k$ and $i\in \Z$.
\vspace{0.1in}

By resolution of singularities, there are proper fs log schemes $\overline{X}$ and $\overline{Y}$ in $SmlSm/k$ such that $X\cong \overline{X}-\partial \overline{X}$ and $Y\cong \overline{Y}-\partial \overline{Y}$.
Owing to Theorem \ref{A.4.10} we have
\[
R\omega^*M(X)\cong M(\overline{X}),
\;
R\omega^*M(Y)\cong M(\overline{Y}).
\]
Applying Theorem \ref{A.4.10} we deduce the isomorphisms
\begin{align*}
\hom_{\dmeff}(M(Y)[i],M(X))\cong &\hom_{\ldmeff}(M(\overline{Y})[i],M(\overline{X}))
\\
\cong &\hom_{\ldmeff}(R\omega^*M(Y)[i],R\omega^*M(X)).
\end{align*}
By adjunction, this concludes the proof.
\end{proof}

\begin{df}
\label{A.4.24}
An object $\cF$ of $\ldmeff$ is called {\it $\A^1$-local}\index{A1-local@$\A^1$-local} if for any $X\in lSm/k$ and $i\in \Z$ the projection $X\times \A^1\rightarrow X$ induces an isomorphism 
\[
\hom_{\ldmeff}(M(X)[i],\cF)\rightarrow \hom_{\ldmeff}(M(X\times \A^1)[i],\cF).
\]
\end{df}

We note that if a complex $\cF$ of dividing Nisnevich sheaves with log transfers is both strictly $\A^1$-invariant and strictly $\boxx$-invariant, 
then the associated object in $\ldmeff$ is $\A^1$-local.

\begin{prop}
\label{A.4.21}
Assume that $k$ admits resolution of singularities. 
Let $\cF$ be an $\A^1$-local object of $\ldmeff$. 
Then, for every $Y\in lSm/k$ and $i\in \Z$, there is a natural isomorphism
\[
\hom_{\ldmeff}(M(Y)[i],\cF)
\cong 
\hom_{\ldmeff}(M(Y-\partial Y)[i],\cF).
\]
\end{prop}
\begin{proof}
We may assume $\cF$ is a strictly $\boxx$-invariant complex of dividing Nisnevich sheaves with log transfers by taking a fibrant replacement.
Then, 
by Proposition \ref{A.5.13}, 
there is an isomorphism
\[
\hom_{\ldmeff}(M(Y)[i],\cF)
\cong 
\bH_{dNis}^i(Y,\cF).
\]
If $Y$ has the trivial log structure, then the dividing Nisnevich topology on $Y$ agrees with the Nisnevich topology on $Y$, and we obtain
\begin{equation}
\label{A.4.21.1}
\hom_{\ldmeff}(M(Y)[i],\cF)
\cong 
\bH_{Nis}^i(Y,\cF).
\end{equation}
Thus the restriction $\cG$ of $\cF$ on $Sm/k$ is a strictly $\A^1$-invariant complex of Nisnevich sheaves with transfers.
\vspace{0.1in}

The claim is dividing Nisnevich local on $Y$, so we may assume $\partial Y$ is a strict normal crossing divisor formed by smooth divisors $Z_1,\ldots,Z_r$ on $Y$. 
We set 
\[
Y_1:=(\underline{Y},Z_2+\cdots+Z_r),
\]
\[
W=(Z_1,Z_1\cap Z_2+\cdots+Z_1\cap Z_r),
\] 
and proceed by induction on $r$. The claim is evident for $r=0$. 
\vspace{0.1in}

In the following, we assume $r\geq 1$ and set 
\[
Y':=(D_{Z_1}Y_1,Z_1\times \boxx),
\;
Y_1':=D_{Z_1}Y,
\;
Y'':=(N_{Z_1} Y_1,Z_1), 
\;
Y_1'':=N_{Z_1}Y_1.
\] 
There is a naturally induced commutative diagram
\[
\begin{tikzpicture}[baseline= (a).base]
\node[scale=.94] (a) at (0,0)
{
\begin{tikzcd}[column sep=tiny]
\hom_{\ldmeff}(M(Y\rightarrow Y_1)[i],\cF)\arrow[d]\arrow[r]
&\hom_{\dmeff}(M(Y-\partial Y\rightarrow Y_1-\partial Y_1)[i],\cG)\arrow[d]
\\
\hom_{\ldmeff}(M(Y'\rightarrow Y_1')[i],\cF)\arrow[d,leftarrow]\arrow[r]
&\hom_{\dmeff}(M(Y'-\partial Y'\rightarrow Y_1'-\partial Y_1')[i],\cG)\arrow[d,leftarrow]
\\
\hom_{\ldmeff}(M(Y''\rightarrow Y_1'')[i],\cF)\arrow[r]&\hom_{\dmeff}(M(Y-\partial Y\rightarrow Y_1''-\partial Y_1'')[i],\cG).
\end{tikzcd}
};
\end{tikzpicture}
\]
The left vertical morphisms are isomorphisms by Theorem \ref{A.3.36}, and the right vertical morphisms are isomorphisms by \eqref{A.4.21.1} and the proof of \cite[Proposition 3.5.4]{MR1764202}. 
The lower horizontal morphism is an isomorphism by induction, so the upper horizontal morphism is also an isomorphism. 
By induction, there is an isomorphism
\[ 
\hom_{\ldmeff}(M(Y_1)[i],\cF)\rightarrow  \hom_{\dmeff}(M(Y_1-\partial Y_1)[i],\cG)
\]
This finishes the proof on account of the five lemma.
\end{proof}

\begin{thm}
\label{A.4.22}
Assume that $k$ admits resolution of singularities. 
Then the essential image of the functor
\[
R\omega^*\colon \dmeff\rightarrow \ldmeff
\]
is the full subcategory of $\ldmeff$ consisting of $\A^1$-local objects.
\end{thm}
\begin{proof}
Let $\cG$ be an object of $\dmeff$. 
Then, for every $Y\in lSm/k$ and $i\in \Z$,
the isomorphisms
\[
\begin{split}
\hom_{\ldmeff}(M(Y)[i],R\omega^*\cG) & \cong \hom_{\dmeff}(\omega_\sharp M(Y)[i],\cG) \\
& \cong \hom_{\dmeff}(M(Y-\partial Y)[i],\cG) \\
& \cong \hom_{\dmeff}(M((Y-\partial Y)\times \A^1)[i],\cG) \\
& \cong \hom_{\dmeff}(\omega_\sharp M(Y\times \A^1)[i],\cG) \\
& \cong \hom_{\ldmeff}(M(Y\times \A^1)[i],R\omega^*\cG)
\end{split}
\]
show that $R\omega^*\cG$ is $\A^1$-local.
\vspace{0.1in}

Now let $\cF$ be an $\A^1$-local object of $\ldmeff$. 
To show that $\cF$ is in the essential image of $R\omega^*$, it suffices to show that the unit of the adjunction $\cF\rightarrow R\omega^*\omega_\sharp \cF$ is an isomorphism. 
We show this by checking that, for any $Y\in lSm/k$ and $i\in \Z$, there is a naturally induced isomorphism
\[
\hom_{\ldmeff}(M(Y)[i],\cF)\rightarrow \hom_{\dmeff}(\omega_\sharp M(Y)[i],\omega_\sharp \cF).
\]

The open immersion $Y-\partial Y\rightarrow Y$ induces a commutative diagram
\[
\begin{tikzcd}
\hom_{\ldmeff}(M(Y)[i],\cF)\arrow[d]\arrow[r]&
\hom_{\dmeff}(\omega_\sharp M(Y)[i],\omega_\sharp \cF)\arrow[d]
\\
\hom_{\ldmeff}(M(Y-\partial Y)[i],\cF)\arrow[r]&
\hom_{\dmeff}(\omega_\sharp M(Y-\partial Y)[i],\omega_\sharp \cF).
\end{tikzcd}
\]
The left vertical homomorphism is an isomorphism by Proposition \ref{A.4.21}, and the right vertical homomorphism is an isomorphism since
\[
\omega_\sharp M(Y)\cong M(Y-\partial Y)\cong \omega_\sharp M(Y-\partial Y).
\]
Hence to show the upper horizontal homomorphism is an isomorphism, it suffices to show the lower horizontal homomorphism is an isomorphism, 
i.e., we are reduced to the case when $Y\in Sm/k$.
\vspace{0.1in}

We may view $\cF$ as a strictly $\boxx$-invariant object of $\Co(\Shvltrkl)$.
Recall from \eqref{eq::lambdaomegaadj} the adjunction
\[
L\lambda_\sharp\colon \Deri(\Shvtrkl)\rightleftarrows \Deri(\Shvltrkl)\colon \lambda^*\cong R\lambda^*\cong \omega_\sharp.
\]
Thus, for every $Y\in Sm/k$ and $i\in \Z$, there are isomorphisms 
\begin{align*}
\hom_{\Deri(\Shvtrkl)}(\Ztr(Y)[i],\omega_\sharp \cF)
& \cong
\hom_{\Deri(\Shvtrkl)}(\Ztr(Y)[i],\lambda^* \cF) \\
& \cong
\hom_{\Deri(\Shvltrkl)}(L\lambda_\sharp \Ztr(Y)[i], \cF) \\
& \cong
\hom_{\Deri(\Shvltrkl)}(\Zltr(Y)[i],\cF) \\
& \cong
\hom_{\ldmeff}(M(Y)[i],\cF)
\end{align*}
Together with the assumption that $\cF$ is $\A^1$-local, we deduce $\omega_\sharp F$ is $\A^1$-local, 
and 
\[
\hom_{\Deri(\Shvtrkl)}(\Ztr(Y)[i],\omega_\sharp \cF)\cong \hom_{\dmeff}(M(Y)[i],\omega_\sharp \cF).
\]
This shows the isomorphism
\[
\hom_{\ldmeff}(M(Y)[i],\cF)\cong \hom_{\dmeff}(M(Y)[i],\omega_\sharp \cF), 
\]
To conclude the proof, we note that
$$
\omega_\sharp M(Y)\cong M(Y).
$$
\end{proof}

\begin{thm}
\label{thm::dmeff=ldmeffprop}
Assume that $k$ admits resolution of singularities.
Then there is an equivalence of triangulated categories
\[
\ldmeffprop\simeq \dmeff.
\]
\end{thm}
\begin{proof}
Owing to Theorem \ref{A.4.22} it suffices to identify $\ldmeffprop$ with the essential image of the functor
\[
R\omega^*\colon \dmeff\rightarrow \ldmeff.
\]
The essential image of $R\omega^*$ is the smallest triangulated subcategory of $\ldmeff$ that is closed under small sums and contains $R\omega^*M(X)$ for every $X\in Sm/k$.
On the other hand, $\ldmeffprop$ is the smallest triangulated subcategory of $\ldmeff$ that is closed under small sums and contains the motive $M(Y)$ for every $Y\in lSm/k$, 
where $Y$ is proper over $k$.
\vspace{0.1in}

Owing to Theorem \ref{A.4.10}, there is an isomorphism
\[
M(Y)\cong R\omega^*M(Y-\partial Y).
\]
Thus $\ldmeffprop\subset \im R\omega^*$.
By resolution of singularities, there exists an object $X'\in SmlSm/k$ such that $X'-\partial X'\cong X$.
Moreover, by Theorem \ref{A.4.10}, there is an isomorphism
\[
M(X')\cong R\omega^*M(X).
\]
Thus we have the inclusion 
$$
\im R\omega^*\subset \ldmeffprop.
$$
\end{proof}

\begin{exm}
\label{A.4.23}
Assume that $k$ admits resolution of singularities.
Then the functor 
\[
R\omega^*\colon \dmeff\rightarrow \ldmeff
\] 
is {\it not} essentially surjective.
Indeed,
in Theorem \ref{thmHodge} we show that for every $X\in SmlSm/k$ and $i,j\geq 0$ there is an isomorphism
\[
\hom_{\ldmeff}(M(X), \Omega^j_{-/k}[i]) 
\cong
H^i_{Zar}(\ul{X}, \Omega^j_{X/k}).
\]
However, the groups $H^i_{Zar}(\ul{X}, \Omega^j_{X/k})$ and $H^i_{Zar}(\ul{X}\times \A^1, \Omega^j_{X\times \A^1/k})$ are non-isomorphic, 
so that Theorem \ref{A.4.22} implies $\Omega^j_{X/k}$ is not in the essential image of $R\omega^*$.
\end{exm}

\subsection{Projective bundle formula}
In this subsection, we formulate and prove the projective bundle formula under the assumption of resolution of singularities.
We begin by discussing orientations on $\ldmeff$. 

\begin{df}
\label{A.7.1}
We say that $\ldmeff$ admits an {\it orientation}\index{orientation} if for any $X\in lSm/k$, there is a morphism
\[
c_1\colon {\rm Pic}(\underline{X})\rightarrow \hom_{\ldmeff}(M(X),\Lambda(1)[2])
\]
of sets, 
functorial in $X$, 
such that the class of the canonical line bundle in ${\rm Pic}(\P^1)$ corresponds to the canonical morphism $M(\P^1)\rightarrow \Lambda(1)[2]$. 
\end{df}

\begin{exm}
\label{A.7.2}
Assume that $k$ admits resolution of singularities.
Recall that $\Lambda(1)[2]=M(\pt\rightarrow \P^1)$.
Since $\pt$ and $\P^1$ are smooth and proper, by Proposition \ref{A.4.8}, we have
\[
{\rm Pic}(X-\partial X)
\cong 
\hom_{\ldmeff}(M(X),\Lambda(1)[2]).
\]
Thus we are entitled to the composition
\[
{\rm Pic}(\underline{X})\stackrel{u^*}
\rightarrow 
{\rm Pic}(X-\partial X)\stackrel{\sim}\rightarrow \hom_{\ldmeff}(M(X),\Lambda(1)[2]),
\]
where $u\colon X-\partial X\rightarrow \underline{X}$ is the induced open immersion.  
Hence $\ldmeff$ admits an orientation.
\end{exm}

\begin{const}
\label{A.7.3}
Let $X$ be an fs log scheme log smooth over $k$, and let $\cP$ be a $\P^n$-bundle over $X$, 
i.e., Zariski locally on $X$, the morphism $\cP\rightarrow X$ is isomorphic to the projection $\P^n\times X\rightarrow X$.
If $\ldmeff$ admits an orientation, then the canonical line bundle over $\cP$ induces a morphism 
\[
\tau_1\colon M(\cP)\rightarrow \Lambda(1)[2]
\]
in $\ldmeff$. 
More generally, 
for $1\leq i\leq n$, 
there is a morphism 
\begin{equation}
\label{A.7.3.1}
\tau_i\colon M(\cP)\rightarrow \Lambda(i)[2i]
\end{equation}
given by the composition
\[
M(\cP)
\stackrel{\Delta}\longrightarrow
M(\cP)^{\otimes i}
\stackrel{\tau_1^{\otimes i}}\longrightarrow 
(\Lambda(1)[2])^{\otimes i}
\cong
\Lambda(i)[2i].
\]
We also have a morphism
\[
\sigma_i\colon M(\cP)\rightarrow M(X)(i)[2i]
\]
given by the composition
\[
M(\cP)\stackrel{\Delta}\longrightarrow M(\cP)\otimes M(\cP)\stackrel{\id \times \tau_i}\longrightarrow M(\cP)\otimes \Lambda(i)[2i]\longrightarrow M(X)\otimes \Lambda(i)[2i]\cong M(X)(i)[2i].
\]
This defines an induced morphism
\[
\sigma\colon M(\cP)\rightarrow \bigoplus_{i=0}^n M(X)(i)[2i].
\]
\end{const}

\begin{prop}
\label{A.6.6}
Assume that $k$ admits resolution of singularities.
For every $X\in lSm/k$ and integer $n\geq 0$, the morphism
\[
\sigma\colon M(\P_X^n)\rightarrow \bigoplus_{i=0}^n M(X)(i)[2i]
\]
is an isomorphism.
\end{prop}
\begin{proof}
Using the monoidal structure, we have that
$$
M(\cP)\cong M(X)\otimes M(\P^n)\in \ldmeff.
$$ 
Hence we may assume that $X=\Spec k$,
i.e., 
for our purposes, it suffices to show there is an isomorphism
\[
M(\P^n)
\cong 
\bigoplus_{i=0}^n \Lambda(i)[2i].
\]
It suffices to show that for $Y\in lSm/k$ and $i\in \Z$ there is an isomorphism
\[
\hom_{\ldmeff}(M(Y)[i],M(\P^n))\cong \hom_{\ldmeff}(M(Y)[i],\bigoplus_{i=0}^n \Lambda(i)[2i]).
\]
By Proposition \ref{A.4.8}, 
there is a commutative diagram with vertical isomorphisms
\[
\begin{tikzcd}[column sep=tiny]
\hom_{\ldmeff}(M(Y)[i],M(\P^n))\arrow[d,"\sim"']\arrow[r]&\hom_{\ldmeff}(M(Y)[i],\bigoplus_{i=0}^n \Lambda(i)[2i])\arrow[d,"\sim"]\\
\hom_{\dmeff}(M(Y-\partial Y)[i],M(\P^n))\arrow[r]&\hom_{\dmeff}(M(Y-\partial Y)[i],\bigoplus_{i=0}^n \Lambda(i)[2i]).
\end{tikzcd}
\]
Hence it suffices to show there is an isomorphism 
\[
M(\P^n)
\cong 
\bigoplus_{i=0}^n \Lambda(i)[2i]
\]
in $\dmeff$, which is noted in \cite[Corollary 15.5, Exercise 15.11]{MVW}.
\end{proof}

\begin{thm}
\label{A.6.2}
Assume that $k$ admits resolution of singularities.
Let $\mathcal{E}$ be a vector bundle of rank $n+1$ over an fs log scheme $X$ log smooth over $k$. 
Then the naturally induced morphism
\[
\sigma\colon M(\P(\cE))\rightarrow \bigoplus_{i=0}^n M(X)(i)[2i]
\]
is an isomorphism in $\ldmeff$.
\end{thm}
\begin{proof}
We argue as in \cite[Theorem 3.2]{Deg08}. 
For $\cP:=\P(\cE)$ and an open immersion $U\rightarrow X$ of fs log schemes over $k$, consider the induced commutative diagram
\[
\begin{tikzcd}
\cP_U\arrow[d]\arrow[r]& \cP_U\times \cP\arrow[d]\\
\cP\arrow[r]&\cP\times \cP
\end{tikzcd}
\]
of fs log schemes over $k$, 
where $\cP_U:=\cP\times_X U$ and the horizontal morphisms are the graph morphisms. 
There are induced morphisms
\[
M(\cP_U\rightarrow \cP)
\rightarrow 
M(\cP_U\times \cP\rightarrow \cP\times \cP)
\stackrel{\sim}
\rightarrow  M(\cP_U\rightarrow \cP)\otimes M(\cP)\rightarrow M(U\rightarrow X)\otimes M(\cP).
\]
\vspace{0.1in}

Using the morphism $M(\cP)\rightarrow \bigoplus_{i=0}^n \Lambda(i)[2i]$ from Construction \ref{A.7.3}, 
we obtain 
\[
M(\cP_U\rightarrow \cP)\rightarrow \bigoplus_{i=0}^n M(U\rightarrow X)(i)[2i].
\]

We use induction on the finite number $m$ of Zariski open subsets appearing in a trivialization of $\cE$.
The case $m=1$ is already done in Proposition \ref{A.6.6}.
Hence suppose $m>1$.
Then there exists a Zariski cover $\{V_1,\ldots,V_m\}$ such that $\cE$ is trivial on $V_i$ for every $i$.
Set
\[
U_1:=V_1,
\;
U_2:=V_2\cup \cdots\cup V_m,
\;
\text{and }
U_{12}:=U_1\cap U_2.
\]
\vspace{0.1in}

There is a naturally induced commutative diagram in $\ldmeff$
\[
\begin{tikzcd}
M(\cP_{U_{12}}\rightarrow \cP_{U_2})\arrow[r]\arrow[d]&\bigoplus_{i=0}^n M(U_{12}\rightarrow U_{2})(i)[2i]\arrow[d]\\
M(\cP_{U_1}\rightarrow \cP)\arrow[r]&\bigoplus_{i=0}^n M(U_1\rightarrow X)(i)[2i].
\end{tikzcd}
\]
The vertical morphisms are isomorphisms by Mayer-Vietoris since $\{U_1,U_2\}$ is a Zariski cover of $X$. 
By induction the claim holds for $U_1$, $U_2$, and $U_{12}$. 
This implies the upper horizontal morphism is an isomorphism so that the lower horizontal arrow is also an isomorphism. 
\vspace{0.1in}

As in \cite[Lemma 3.1]{Deg08}, there is a commutative diagram
\[
\begin{tikzcd}[column sep=small]
M(\cP_{U_1})\arrow[d]\arrow[r]&M(\cP)\arrow[d]\arrow[r]&M(\cP_{U_1}\rightarrow \cP)\arrow[d]\arrow[r]&M(\cP_{U_1})[1]\arrow[d]\\
\bigoplus_i M(U_1)(i)[2i]\arrow[r]&\bigoplus_i M(X)(i)[2i]\arrow[r]&\bigoplus_i M(U_1\rightarrow X)(i)[2i]\arrow[r]&\bigoplus_i M(U_1)(i)[2i+1]
\end{tikzcd}
\]
The first and third vertical morphisms are isomorphisms so that the second vertical morphism is also an isomorphism.
\end{proof}

\begin{rmk}
\label{A.6.10}
In the proof of \cite[Theorem 15.12]{MVW} an extra argument is needed to justify the induction argument, 
see the proofs of \cite[Lemma 3.1, Theorem 3.2]{Deg08}.
\end{rmk}

\begin{thm}
\label{A.6.7}
Assume that $k$ admits resolution of singularities.
Let $\cE$ be a rank $n$ vector bundle over $X\in SmlSm/k$.
Then there is a canonical isomorphism
\[
MTh(\cE)\cong M(X)(n)[2n].
\]
\end{thm}
\begin{proof}
Recall that we have the morphisms
\[
\tau_i\colon M(\P(\cE))\rightarrow \Lambda(i)[2i]
\text{ and }
\tau_i\colon M(\P(\cE\oplus \cO))\rightarrow \Lambda(i)[2i]
\]
The pullback of the canonical line bundle over $\P(\cE\oplus \cO)$ by the closed immersion $u\colon \P(\cE)\rightarrow \P(\cE\oplus \cO)$ is isomorphic to the canonical line bundle over $\P(\cE)$.
This means that there is a commutative diagram
\[
\begin{tikzcd}
M(\P(\cE))\arrow[r,"\tau_1"]\arrow[d,"M(u)"']&
\Lambda(1)[2]\arrow[d,"{\rm id}"]
\\
M(\P(\cE\oplus \cO))\arrow[r,"\tau_1"]&
\Lambda(1)[2].
\end{tikzcd}
\]
Thus for every integer $i\geq 0$, there is a commutative diagram
\[
\begin{tikzcd}
M(\P(\cE))\arrow[d]\arrow[r,"\Delta"]&
M(\P(\cE))^{\otimes i}\arrow[r,"\tau_1^{\otimes i}"]\arrow[r]\arrow[d,"{\rm id}"]&
\Lambda(1)[2]^{\otimes i}\arrow[d]
\\
M(\P(\cE\oplus \cO))\arrow[r,"\Delta"]&
M(\P(\cE\oplus \cO))^{\otimes i}\arrow[r,"\tau_1^{\otimes i}"]\arrow[r]&
\Lambda(1)[2]^{\otimes i}.
\end{tikzcd}
\]

Applying the above for $1\leq i\leq n-1$, we obtain the commutative diagram
\[
\begin{tikzcd}
M(\P(\cE))\arrow[d,"M(u)"']\arrow[r,"\sigma"]&
\bigoplus_{i=0}^{n-1} M(X)(i)[2i]\arrow[d,"\beta:=\sigma\circ M(u)\circ \sigma^{-1}"']\arrow[rd,"{\rm id}"]
\\
M(\P(\cE\oplus \cO))\arrow[r,"\sigma"]&
\bigoplus_{i=0}^n  M(X)(i)[2i]\arrow[r,"p"]&
\bigoplus_{i=0}^{n-1}  M(X)(i)[2i],
\end{tikzcd}
\]
where $p$ is the canonical projection.
It follows that there is a canonical commutative diagram with horizontal split distinguished triangles and vertical isomorphisms
\[
\begin{tikzcd}[column sep=small]
\bigoplus_{i=0}^{n-1}M(X)(i)[2i]\arrow[d,"\cong"']\arrow[r,"\beta"]&
\bigoplus_{i=0}^{n}M(X)(i)[2i]\arrow[d,"\cong"]\arrow[r]&
M(X)(n)[2n]\arrow[r,"0"]&
\bigoplus_{i=0}^{n-1}M(X)(i)[2i+1]\arrow[d,"\cong"]
\\
M(\P(\cE))\arrow[r]&
M(\P(\cE\oplus \cO))\arrow[r]&
MTh(\cE)\arrow[r,"0"]&
M(\P(\cE))[1].
\end{tikzcd}
\]
The canonical composite morphism
\[
\alpha\colon M(X)(n)[2n]\rightarrow \bigoplus_{i=0}^n M(X)(i)[2i]\xrightarrow{\cong}M(\P(\cE\oplus \cO))\rightarrow MTh(\cE)
\]
makes the above diagram into a morphism of distinguished triangles, where the first morphism is the canonical embedding.
By the five lemma, $\alpha$ is an isomorphism.
\end{proof}

\begin{cor}
\label{A.6.8}
Assume that $k$ admits resolution of singularities.
For $X\in SmlSm$ and $Z$ a codimension $n$ smooth subscheme of $\underline{X}$ such that $Z$ has strict normal crossing 
with $\partial X$ and is not contained in $\partial X$, 
there is a canonical isomorphism
\[
M((B_Z X,E)\rightarrow X)\rightarrow M(Z)(n)[2n].
\]
Here $E$ is the exceptional divisor of the blow-up $B_Z X$.
\end{cor}
\begin{proof}
Immediate from Theorems \ref{A.3.36} and \ref{A.6.7}.
\end{proof}

\subsection{Motives with rational coefficients}
If $\Lambda$ is a $\Q$-algebra, then according to \cite[Theorem 14.30]{MVW}, there is an equivalence of Voevodsky's triangulated categories of Nisnevich motives and \'etale motives
\[
\dmeff\cong \dmeffet.
\]
In what follows, we argue as in \cite[Lecture 14]{MVW} and prove an analog in the log setting: 
If $\Q\subset \Lambda$, then there is an equivalence of triangulated categories
\[
\ldmeff\cong \ldmeffet.
\]

\begin{lem}
\label{MotRat.1}
Let $\cF$ be a strict Nisnevich sheaf of $\Q$-modules with log transfers.
Then $\cF$ is a strict \'etale sheaf.
\end{lem}
\begin{proof}
Owing to Proposition \ref{A.5.22} the strict \'etale topology is compatible with log transfers.
Thus $a_{s\acute{e}t}^*\cF$ is a strict \'etale sheaf of $\Q$-modules with log transfers.
We need to show there is a naturally induced isomorphism
\[
p\colon \cF\rightarrow a_{s\acute{e}t*}a_{s\acute{e}t}^*\cF.
\]
To that end, it suffices to show that 
$$
a_{s\acute{e}t}^*(\ker p)=a_{s\acute{e}t}^*(\coker p)=0.
$$
Hence we are reduced to consider the case when $a_{s\acute{e}t}^*\cF=0$.
\vspace{0.1in}

We may assume that $\cF$ is nontrivial.
Then there exists a henselization $T$ of an fs log scheme in $lSm/k$ such that $\cF(T)\neq 0$, i.e., $\ul{T}$ is a henselization of $\ul{Z}$ for some $Z\in lSm/k$ and $T=\ul{T}\times_{\ul{Z}}Z$.
Choose a nonzero element $c\in \cF(T)$.
Since $a_{s\acute{e}t}^*\cF=0$, 
there exists a strict \'etale cover $f\colon X\rightarrow T$ such that $f^*(c)\in \cF(X)$ is zero.
Owing to the implication $(a)\Rightarrow (c'')$ in the proof of \cite[Th\'eor\`eme IV.18.5.11]{EGA} there exists a factorization $g\colon Y\rightarrow T$ of $f$, 
where $g$ is a finite \'etale morphism.
It follows that $g^*(c)=0\in \cF(Y)$.
\vspace{0.1in}

The transpose of the graph of $g\colon Y\rightarrow T$ gives an elementary log correspondence $h\in \lCor(T,Y)$.
If $d$ is the degree of $g$,
then the composition $g\circ h\in \lCor(T,T)$ is equal to $d$ times the identity correspondence.
Thus the composition
\[
\cF(T)\stackrel{g^*}\rightarrow \cF(Y)\stackrel{h^*}\rightarrow \cF(T)
\]
is an isomorphism since $\cF$ is a sheaf of $\Q$-modules.
This implies $c=0$ since $g^*(c)=0$, which is a contradiction.
\end{proof}

\begin{rmk}
\label{MotRat.2}
The proof of Lemma \ref{MotRat.1} does not work for the Kummer \'etale topology.
Indeed, 
in this case, 
$Y$ is not necessarily strict over $X$, 
and hence the transpose of the graph of $g$ is not an elementary log correspondence from $S$ to $Y$.
\end{rmk}

\begin{lem}
\label{MotRat.3}
Suppose $\cF$ be a dividing Nisnevich sheaf of $\Q$-modules with log transfers.
Then $\cF$ is a dividing \'etale sheaf.
Thus if $\Q\subset \Lambda$, there is an equivalence of categories
\[
\Shv_{dNis}^{\rm ltr}(k,\Lambda)\cong \Shv_{d\acute{e}t}^{\rm ltr}(k,\Lambda).
\]
\end{lem}
\begin{proof}
Owing to Lemma \ref{MotRat.1} $\cF$ is a strict \'etale sheaf.
For any log modification $Y\rightarrow X$ in $lSm/k$ there is an isomorphism $\cF(X)\rightarrow \cF(Y)$ since $\cF$ is a dividing sheaf.
To conclude, we apply Lemma \ref{A.5.45}.
\end{proof}

\begin{prop}
\label{MotRat.4}
If $\Q\subset \Lambda$, there is an equivalence of triangulated categories
\[
\ldmeff\cong \ldmeffet.
\]
\end{prop}
\begin{proof}
Let $\cG$ be the set of objects of the form $\Zltr(X)$ for $X\in lSm/k$.
With respect to this set, we are entitled to the corresponding descent model structures and the equivalent categories
\[
\Co(\Shv_{dNis}^{\rm ltr}(k,\Lambda))\text{ and }\Co(\Shv_{d\acute{e}t}^{\rm ltr}(k,\Lambda)), 
\]
and hence a Quillen equivalence
\[
\Co(\Shv_{dNis}^{\rm ltr}(k,\Lambda))\rightleftarrows \Co(\Shv_{d\acute{e}t}^{\rm ltr}(k,\Lambda)).
\]
Due to \cite[Theorem 3.3.20(1)]{MR1944041} the above is also a Quillen equivalence for the $\boxx$-local descent model structures, 
and we are done by passing to the homotopy categories.
\end{proof}

\subsection{Locally constant log \'etale sheaves}
If $\Lambda$ is a torsion ring coprime to the exponential characteristic of $k$, then \cite[Theorem 9.35]{MVW} shows there is an equivalence of triangulated categories
\[
\mathbf{D}^{-}(k_{\acute{e}t},\Lambda)\rightarrow \dmeffetminus.
\]
The proof relies on Suslin's rigidity theorem \cite[Theorem 7.20]{MVW}.
We refer to \cite{CDEtale} for a way to remove boundedness.
In this subsection, we discuss the log \'etale version of the above equivalence.
\vspace{0.1in}

Let $\eta\colon k_{\acute{e}t}\rightarrow lSm/k$ denote the inclusion functor.
Then $\eta$ is a continuous and cocontinuous functor with respect to the \'etale topology on $k_{\acute{e}t}$ and the log \'etale topology on $lSm/k$.
Since $\eta$ preserves finite limits, $\eta$ is a morphism of sites due to \cite[IV.4.9.2]{SGA4}.
As a consequence there exist adjunct functors
\begin{equation}
\label{equation:ketlog}
\begin{tikzcd}
\Shv(k_{\acute{e}t},\Lambda)
\arrow[rr,shift left=1.5ex,"\eta_\sharp "]
\arrow[rr,"\eta^*" description,leftarrow]
\arrow[rr,shift right=1.5ex,"\eta_*"']&&
\Shv_{l\acute{e}t}^{\rm log}(k,\Lambda),
\end{tikzcd}
\end{equation}
and $\eta_\sharp$ is exact.

\begin{df}
\label{ketcomp.3}
A log \'etale sheaf $\cF$ on $lSm/k$ is \emph{locally constant}\index{locally constant sheaf} if the counit of the adjunction $\eta_\sharp \eta^*\cF\rightarrow \cF$ in \eqref{equation:ketlog} is an isomorphism.
\end{df}

\begin{prop}
\label{ketcomp.11}
The functor $\eta_\sharp\colon \Shv(k_{\acute{e}t},\Lambda)\rightarrow \Shv_{l\acute{e}t}^{\rm log}(k,\Lambda)$ is fully faithful.
\end{prop}
\begin{proof}
The functor $\eta\colon k_{\acute{e}t}\rightarrow lSm/k$ is fully faithful, so for every $X\in k_{\acute{e}t}$ the unit
\[
\Lambda(X)\rightarrow \eta^*\eta_\sharp \Lambda(X)
\]
is an isomorphism.
Since $\eta_\sharp$ and $\eta^*$ commute with colimits, for every $\cF\in \Shv(k_{\acute{e}t},\Lambda)$ we have
\[
\eta^*\eta_\sharp\cF
\cong
\eta^*\eta_\sharp \colimit_{X\rightarrow \cF}\Lambda(X)
\cong
\colimit_{X\rightarrow \cF} \eta^*\eta_\sharp\Lambda(X)
\cong
\colimit_{X\rightarrow \cF} \Lambda(X)
\cong
\cF.
\]
It follows that $\eta_\sharp$ is fully faithful.
\end{proof}

\begin{prop}
\label{ketcomp.14}
Suppose that there is a cartesian square of locally noetherian fs log schemes
\[
\begin{tikzcd}
Y'\arrow[d,"f'"']\arrow[r,"g'"]&Y\arrow[d,"f"]
\\
X'\arrow[r,"g"]&X
\end{tikzcd}
\]
with the following properties.
\begin{enumerate}
\item[{\rm (i)}] $X$ and $X'$ have trivial log structure, and $X'\cong \limit X_i$ for some projective system of \'etale schemes of finite type over $X$ such that every transition morphism $X_j\rightarrow X_i$ is affine.
\item[{\rm (ii)}] $f$ is a morphism of finite type.
\end{enumerate}
Then for every $\cF\in \mathbf{D}(X_{l\acute{e}t},\Lambda)$ there is a base change isomorphism
\[
g^*Rf_* \cF\xrightarrow{\cong} Rf_*'g'^*\cF.
\]
\end{prop}
\begin{proof}
As in the proof of \cite[Theorem 1.1.14]{CDEtale}, this can be reduced to showing that for every log \'etale sheaf $\cF$ of $\Lambda$-modules on $X_{l\acute{e}t}$ there is a base change isomorphism
\begin{equation}
\label{ketcomp.14.1}
g^*f_*\cF\xrightarrow{\cong} f_*'g'^*\cF.
\end{equation}
We can also replace $X'$ by $X_i$ in the projective system and $Y'$ by $Y\times_X X_i$ and then apply the limit.
Hence we reduce to the case when $g$ is \'etale.
In this case, \eqref{ketcomp.14.1} can be proved directly.
\end{proof}

\begin{prop}
\label{ketcomp.1}
Suppose $\Lambda$ is a torsion ring coprime to the exponential characteristic of $k$.
Let $X$ be an fs log scheme in $lSm/k$ and $q\in\Z$.
For every complex $\cF$ of locally constant sheaves of $\Lambda$-modules on $lSm/k$, the projection morphism $p\colon X\times \boxx\rightarrow X$ induces an isomorphism of log \'etale cohomology groups 
\[
p^*\colon \bH_{l\acute{e}t}^q(X,\cF)\xrightarrow{\cong} \bH_{l\acute{e}t}^q(X\times \boxx,\cF).
\]
\end{prop}
\begin{proof}
By applying Proposition  \ref{ketcomp.14} to the case when $g$ is $\Spec{k^s}\rightarrow \Spec{k}$, we reduce to the case when $k$ is separably closed.
Then $X$ and $X\times \boxx$ have finite cohomological dimension by virtue of \cite[Theorem 7.2(2)]{MR3658728}.
Apply \cite[5.7.9]{weibel_1994} to obtain two strongly convergent hypercohomology spectral sequences
\[
E_{pq}^2=H^p(X,a_{l\acute{e}t}^* H^q(\cF))\Rightarrow \bH_{l\acute{e}t}^{p+q}(X,\cF).
\]
and
\[
E_{pq}'^2=H^p(X\times \boxx,a_{l\acute{e}t}^* H^q(\cF))\Rightarrow \bH_{l\acute{e}t}^{p+q}(X\times \boxx,\cF).
\]
Since there is a naturally induced morphism  of spectral sequences $E_{pq}\rightarrow E_{pq}'$, we reduce to case when $\cF$ is a locally constant sheaf of $\Lambda$-modules on $X_{l\acute{e}t}$.
\vspace{0.1in}

Since $k$ has become separably closed, $\cF$ is a constant torsion sheaf.
We can further reduce to the case when $\cF=\Z/n\Z$ for some integer $n>0$ invertible on $k$.
Thanks to \cite[Theorem 5.17]{MR3658728} there are isomorphisms
\[
H_{k\acute{e}t}^q(X,\Z/n\Z)\xrightarrow{\cong} H_{l\acute{e}t}^q(X,\Z/n\Z),
\;
H_{k\acute{e}t}^q(X\times \boxx,\Z/n\Z)\xrightarrow{\cong} H_{l\acute{e}t}^q(X\times \boxx,\Z/n\Z).
\]
For this reason, it suffices to show that there is an isomorphism of Kummer \'etale cohomology groups
\[
p^*\colon H_{k\acute{e}t}^q(X,\Z/n\Z)\xrightarrow{\cong} H_{k\acute{e}t}^q(X\times \boxx,\Z/n\Z).
\]

Owing to \cite[Corollary 7.5]{MR1922832} there are isomorphisms 
\begin{align*}
H_{k\acute{e}t}^q(X,\Z/n\Z)&\xrightarrow{\cong} H_{\acute{e}t}^q(X-\partial X,\Z/n\Z),
\\
H_{k\acute{e}t}^q(X\times \boxx,\Z/n\Z)&\xrightarrow{\cong} H_{\acute{e}t}^q((X-\partial X)\times \A^1,\Z/n\Z).
\end{align*}
By $\A^1$-homotopy invariance of \'etale cohomology \cite[Corollaire XV.2.2]{SGA4} there is an isomorphism 
\[
H_{\acute{e}t}^q(X-\partial X,\Z/n\Z)\xrightarrow{\cong} H_{\acute{e}t}^q((X-\partial X)\times \A^1,\Z/n\Z).
\]
Combining these isomorphisms concludes the proof.
\end{proof}

\begin{exm}
\label{ketcomp.2}
For every integer $n>1$ coprime to the exponential characteristic of $k$,
the sheaf of $n$th roots of unity $\mu_n$ is \emph{not} strictly $\boxx$-invariant in the strict \'etale topology.
Indeed, for every integer $i\in \Z$, we have
\[
H_{s\acute{e}t}^i(\boxx,\mu_n)
\cong 
H_{\acute{e}t}^i(\P^1,\mu_n).
\]
In general, the latter group is non-isomorphic to $H_{\acute{e}t}^i(\Spec{k},\mu_n)$.
In particular,
if $k$ contains all $n$th roots of unity, 
then $\Z/n$ is \emph{not} strictly $\boxx$-invariant in the strict \'etale topology.
\end{exm}

\begin{prop}
\label{ketcomp.4}
With $\Lambda$ as above, every locally constant log \'etale sheaf $\cF$ on $lSm/k$ admits a unique log transfer structure.
\end{prop}
\begin{proof}
Owing to \cite[Lemma 6.11]{MVW} $\cF$ has a unique transfer structure on $Sm/k$.
For every $X\in lSm/k$, apply Proposition \ref{ketcomp.1} to have an isomorphism
$$
\cF(X)\cong \cF(X-\partial X).
$$
Thus for every finite log correspondence $V\in \lCor(X,Y)$, 
where $Y\in lSm/k$, 
the transfer map 
$$
(V-\partial V)^*\colon \cF(X-\partial X)\rightarrow \cF(Y-\partial Y)
$$ 
extends uniquely to a transfer map 
$$
V^*\colon \cF(X)\rightarrow \cF(Y).
$$
\end{proof}

With $\Lambda$ as above, as a consequence we have an adjunction
\begin{equation}
\label{equation:ketloglambda}
\eta_\sharp: \Shv(k_{\acute{e}t},\Lambda)\rightleftarrows \Shv_{l\acute{e}t}^{\rm ltr}(k,\Lambda):\eta^*,
\end{equation}
and $\eta_\sharp$ is exact and fully faithful.
From \eqref{equation:ketloglambda} we deduce the canonical Quillen adjunction 
\[
\eta_\sharp:\Co(\Shv(k_{\acute{e}t},\Lambda))\rightleftarrows \Co(\Shv_{l\acute{e}t}^{\rm ltr}(k,\Lambda)):\eta^*
\]
with respect to the descent model structure on the left side and the $\boxx$-local descent model structure on the right side.
On the corresponding homotopy categories, we obtain the canonical adjunction: 
\[
L\eta_\sharp:\mathbf{D}(\Shv(k_{\acute{e}t},\Lambda))\rightleftarrows \ldmefflet:R\eta^*
\]
Since $\eta_\sharp$ is exact, $L\eta_\sharp\simeq \eta_\sharp$.

\begin{prop}
With $\Lambda$ as above, the functor
\[
\eta_\sharp\colon \mathbf{D}(\Shv(k_{\acute{e}t},\Lambda))\rightarrow \ldmefflet
\]
is fully faithful.
\end{prop}
\begin{proof}
Let $\cF$ be an object of $\Co(\Shv(k_{\acute{e}t},\Lambda))$.
By Proposition \ref{ketcomp.1}, $\eta_\sharp \cF$ is strictly $\boxx$-invariant,
and hence there is an isomorphism
$$
R\eta^*\eta_\sharp \cF\simeq \eta^*\eta_\sharp\cF.
$$
To conclude, we observe that the functor 
$$
\eta_\sharp\colon \mathbf{C}(\Shv(k_{l\acute{e}t},\Lambda))\rightarrow \mathbf{C}(\Shv_{l\acute{e}t}^{\rm ltr}(k,\Lambda))
$$ 
is fully faithful since the functor $\eta_\sharp\colon \Shv(k_{l\acute{e}t},\Lambda)\rightarrow \Shv_{l\acute{e}t}^{\rm log}(k,\Lambda)$ is fully faithful.
\end{proof}

Next, we state a rigidity conjecture which is a log analog of \cite[Theorem 9.35]{MVW}.

\begin{conj}
\label{conjrigidity}
With $\Lambda$ as above, the functor
\[
\eta_\sharp\colon \mathbf{D}(\Shv(k_{l\acute{e}t},\Lambda))\rightarrow \ldmefflet
\]
is an equivalence.
\end{conj}

Related to Conjecture \ref{conjrigidity} we have the following log analog of Suslin's rigidity theorem \cite[Theorem 7.20]{MVW}.

\begin{conj}
\label{ketcomp.5}
Every strictly $\boxx$-invariant log \'etale sheaf $\cF$ equipped with log transfers and torsion coprime to the exponential characteristic of $k$ is locally constant.
\end{conj}

\begin{rmk}
It is unclear whether one can expect that every $\boxx$-invariant log \'etale sheaf with log transfers is locally constant.
The sheaf $a_{l\acute{e}t}^*\Zltr(\A^1)$ is a potential counterexample, see Example \ref{A.4.28}.
\end{rmk}

\subsection{Comparison with Voevodsky's \'etale motives}
In this section we state our expectations regarding an \'etale version of Theorem \ref{thm::dmeff=ldmeffprop}.
Since Lemma \ref{A.4.4} fails in the \'etale topology, the proof of Theorem \ref{thm::dmeff=ldmeffprop} cannot be directly adopted to \'etale log motives.


\begin{conj}
\label{ketcomp.6}
There is an equivalence of triangulated categories
\[
\ldmeffet\cong \ldmefflet.
\]
\end{conj}

\begin{rmk}
Theorem \ref{thmHodge} on Hodge sheaves provides some first evidence for Conjecture \ref{ketcomp.6} in the form of the isomorphism 
\[
\hom_{\ldmeffet}(M(X),\Omega_{/k}^j[i])\xrightarrow{\cong} \hom_{\ldmefflet}(M(X),\Omega_{/k}^j[i]).
\]
On the other hand, 
since $\Z/n$ is not strictly $\boxx$-invariant for the strict \'etale topology, 
it is unclear whether the cohomology groups
\[
\hom_{\ldmeffet}(M(X),\Z/n[i])\text{ and } \hom_{\ldmefflet}(M(X),\Z/n[i])
\]
are isomorphic, see Example \ref{ketcomp.2}.
\end{rmk}

\begin{conj}
\label{ketcomp.7}
Suppose that $\Lambda=\Z/n$, where $n$ is invertible in $k$. 
Then there are equivalences of triangulated categories
\[
\ldmeffet\leftarrow \ldmeffetprop\rightarrow \dmeffet,
\]
and
\[
\ldmefflet\leftarrow \ldmeffletprop\rightarrow \dmeffet,
\]
\end{conj}

\begin{rmk}
The Hodge sheaves $\Omega_{/k}^j\in\ldmeffet$ by Theorem \ref{thmHodge}.
We expect the same holds for the de Rham-Witt sheaves $W_m\Omega_{/k}^j$.
Moreover, 
one may guess that all the objects of $\ldmeffet$ can be built from $W_m\Omega_{/k}^j$ and objects of $\dmeffet$.
Due to the vanishing
\[
W_m\Omega_{X/k}^j(X)\otimes \Z[1/p]=0,
\]
where $p$ denotes the characteristic of the base field $k$, 
the de Rham-Witt sheaves may become trivial in $\ldmeffet$ under our assumption.
This indicates that $\ldmeffet$ and $\dmeffet$ are comprised of the same objects.
\end{rmk}

\newpage

\section{Hodge sheaves}
\label{sec:Hodge}
The main goal of this section is to show that Hodge cohomology for (log) schemes defined over a perfect field $k$ is representable in the category of effective log motives $\ldmeff$.
This demonstrates in clear terms a fundamental difference between our category and Voevodsky's category of motives $\dmeff$. 
Hodge cohomology is not representable in the latter category.
\vspace{0.1in}

In the main step towards representability of Hodge cohomology,
see Theorem \ref{thm:Omega_log_transfer}, 
we show that the presheaf $\Omega^j$ on $lSm/k$ given by the assignment 
$$
X\mapsto \Omega^j_{X/k}(X) 
$$ 
extends to a presheaf with log transfers on $lSm/k$ in the sense of Section \ref{sec:logtransfers}. 
Analogous to an argument in \cite[\S 6.2]{BS},
showing $(\P^{\bullet},\P^{\bullet-1})$-invariance is then fairly straightforward by using a suitable interpretation of the projective bundle formula. 
\vspace{0.1in}

The action of finite correspondences on Hodge sheaves (without log structure) is well-known and follows, for example, from the computations in \cite{lecomtewach} (at least in characteristic zero). 
A more conceptual approach, 
based on coherent duality and cohomology with support, 
was developed by Chatzistamatiou and R\"ulling in \cite{ChatzistamatiouRullingANT} and \cite{ChatzistamatiouRullingDocumenta}.
We will adapt the same script to the logarithmic setting.

\subsection{Logarithmic derivations and differentials}\index[notation]{Omega @ $\Omega^i_{-/k}$}
For every fs log scheme, $X$ over $k$ and $i\geq 0$ the sheaf of logarithmic differentials on $X$ is defined by 
$$
\Omega^i_{X/k} 
:= 
\begin{cases}
\cO_X & i=0 \\
\wedge^i \Omega^1_{X/k} & i>0. 
\end{cases}
$$
We shall often omit the subscript $/k$ and write $\Omega_X^i$ when no confusion seems likely to arise.
Recall that $\Omega^1_{X}$ is the sheaf of $\cO_{X}$-modules generated by universal log derivations. 
We refer to \cite[Chapter IV]{Ogu} and Section \ref{log-differential} for details. 
\vspace{0.1in}

If $X\in SmlSm/k$, and $(t_1,\ldots, t_n)$ are local coordinates, 
we recall that $\Omega^1_{X}$ is the free $\cO_{\underline{X}}$-module generated by the symbols $dt_1, \ldots dt_n$. 
For the open immersion 
$$
j\colon \underline{X}- \partial X \to \underline{X},
$$ 
the sheaf 
$$
\Omega^1_{X}=\Omega^1_{\underline{X}}(\log \partial X)
$$ 
of differentials with log poles along $\partial X$ is the subsheaf of $j_*j^* \Omega^1_{\underline{X}}$ generated by 
$$
d\log t_1, \ldots, d\log t_r, dt_{r+1}, \ldots dt_n.
$$
Here, the ideal sheaf of $\partial X$ in $\underline{X}$ is locally defined by $I=(t_1\cdots t_r)$.
\vspace{0.1in}

A \emph{quasi-coherent} sheaf $\cF$ on an fs log scheme $X$ is a quasi-coherent sheaf on $\underline{X}$.
We similarly define \emph{coherent} sheaves and \emph{locally free} coherent sheaves on $X$.
For example, $\Omega_{X}^i$ is a coherent sheaf for every fs log scheme $X$ of finite type over $k$. 
Indeed, if $i=1$ this is \cite[Lemma IV. 1.2.16]{Ogu} and $\Omega_{X}^i$ is the $i$th exterior power of $\Omega_{X}^1$ for $i>1$.  
Moreover, if $X\in SmlSm/k$, then $\Omega_X^i$ is locally free.
\vspace{0.1in}

By generalizing \cite[Proposition III 3.7, Remark III 3.8]{Milneetale} to the Nisnevich case we obtain isomorphisms 
\begin{equation}
\label{eqn:coherentcohomology}
H_{Zar}^i(X,\cF)\cong H_{sNis}^i(X,\cF)\cong H_{s\acute{e}t}^i(X,\cF)
\end{equation}
for every quasi-coherent sheaf $\cF$ on $X$ and integer $i\geq 0$.

\begin{lem}
\label{Omega_dividing_sheaf}
Let $X$ be an fs log scheme.
Then every quasi-coherent sheaf $\cF$ on $X$ is a log \'etale sheaf.
\end{lem}
\begin{proof}
The question is Zariski local on $X$.
Hence we may assume that there exists a presentation of the form
\[
\bigoplus_{j\in J}\cO_X\rightarrow \bigoplus_{i\in I}\cO_X\rightarrow \cF\rightarrow 0.
\]
Since the sheafification functor is exact, we only need to show that $\cO_X$ is a log \'etale sheaf.
This follows from \cite[\S 3.3]{KatoLog2}.
\end{proof}

%

\subsection{$(\P^n, \P^{n-1})$-invariance of logarithmic differentials}
\label{PnPn-1invariancediff}
For every $n\geq 1$, 
let $(\P^n, \P^{n-1})$ be the log smooth log scheme $\P^n = \P^n_k$ with compactifying log structure induced by the open immersion 
$$
j\colon \A^n =\P^n- H\hookrightarrow \P^n,
$$ 
where $H\cong\P^{n-1} \subset \P^n$ is a $k$-rational hyperplane. 
If $X$ is an fs log scheme in $lSm/k$, 
we form the product $X\times (\P^n, \P^{n-1})$ and the projection 
$$
\pi\colon X\times (\P^n, \P^{n-1}) \to X.
$$

\begin{prop}
\label{prop:boxinvdiff}
Suppose $X\in SmlSm/k$.
For every $i\geq 1$  the naturally induced map 
$$
\pi^*\colon \Omega^i_{X} \to {R}\pi_*\Omega^i_{X\times (\P^n, \P^{n-1})}
$$ 
is an isomorphism in the bounded derived category of coherent sheaves ${\bf D}^b(\underline{X}_{\rm Zar})$.
\begin{proof}(Cf.\ \cite[\S 6.2]{BS}). We begin by setting our notation. 
Let $X = (\underline{X}, \partial X)$,
where $\partial X$ is a strict normal crossing divisor in $\underline{X}$. 
Let $\iota_H\colon H\to \P^n$ be the closed embedding of the rational hyperplane $H\cong \P^{n-1}$ in $\P^n$. 
We need to prove there is an isomorphism 
\begin{equation}\label{eq:eqboxinvdiff} 
\pi^* \colon \Omega^i_{\underline{X}}(\log \partial X) \to R\pi_*\Omega^i_{\underline{X} \times  \P^n}(\log H+\partial X)
\end{equation}
in $D^b(\underline{X}_{\rm Zar})$. 
On the right hand side of \eqref{eq:eqboxinvdiff}, 
we write $H+\partial X$ for the strict normal crossing divisor on $\underline{X}\times \P^n$ given by $\underline{X} \times H + \partial X \times \P^n$. 
We shall prove the claim by induction on the number of components of $\partial X$.
\vspace{0.1in}

Let us write 
$$
\partial X = \partial X_1 + \ldots + \partial X_m,
$$ 
where the $\partial X_i$'s are the irreducible components of $\partial X$. 
When $\partial X = \emptyset$, 
the statement follows from the projective bundle formula for the sheaves of logarithmic differential forms.        
This should be well known, but for the lack of a reference, we give the argument.
\vspace{0.1in}

Write $Res_H$ for the residue map along $\ul{X}\times H \subset \ul{X}\times \P^n$ that fits into a short exact sequence (see e.g., \cite[\S 2.3]{EVbuch})
\begin{equation}\label{eq:eqboxinvdiff2}
0\to \Omega^i_{\ul{X}\times \P^n} \to \Omega^i_{\ul{X}\times \P^n}(\log H) \xrightarrow{Res_H} \iota_{H,*}\Omega^{i-1}_{\ul{X}\times H}\to 0
\end{equation}
of $\cO_{\underline{X}\times \P^n}$-modules. 
Pushing forward to $\ul{X}$ we get a distinguished triangle
\begin{equation}\label{eq:eqboxinvdiff3}
R\pi_{H,*} \Omega^{i-1}_{\ul{X}\times H}[-1] \xrightarrow{f} R\pi_{ *}\Omega^i_{\ul{X}\times \P^n} \to R\pi_*  \Omega^i_{\ul{X}\times \P^n}(\log H) 
\rightarrow R\pi_{H,*} \Omega^{i-1}_{\ul{X}\times H}
\end{equation}
in ${\bf D}^b(\ul{X}_{\rm Zar})$, 
where $\pi_H\colon \ul{X}\times H\to \ul{X}$ is the projection. 
Taking cup products with the powers of the first Chern class 
$$
\xi = c_1(\cO_{\P^n}(H)) \in H_{Zar}^1(\P^n, \Omega^1_{\P^n})
$$ 
of $H$ determines an isomorphism in the derived category of bounded complexes of $\cO_{\ul{X}}$-modules
\begin{equation}\label{eq:eqboxinvdiff4}
\bigoplus_{0\leq j\leq n} \Omega^{i-j}_{\ul{X}} [-j] \xrightarrow{\simeq} R\pi_* \Omega^i_{\ul{X}\times \P^n}, \quad (a_0, \ldots, a_n)\mapsto \sum_{0\leq j\leq n} \pi^*(a_j) \cup \xi^j.
\end{equation}
This implies the projective bundle formula in Hodge cohomology for the trivial bundle.
The existence of this map is clear once one interprets $\xi$ as a morphism $\xi\colon \Z[-1] \to \Omega_{\P^1}^1$ in the derived category of coherent $\cO_{\P^1}$-modules ${\bf D}^b(\ul{X}_{\rm Zar})$. The quasi-isomorphism is well-known to follow from a direct local computation. 
See for example \cite[Lemma 3.2]{ArapuraKang} but the result is much older (see \cite{MR265370}). 
\vspace{0.1in}

Similarly, there is an isomorphism
\begin{equation}\label{eq:eqboxinvdiff5}
\bigoplus_{0\leq j\leq n-1} \Omega^{i-1-j}_{\ul{X}}[-j] \xrightarrow{\simeq} R (\pi_H)_* \Omega^{i-1}_{\ul{X}\times H}, \quad (a_0, \ldots, a_{n-1})\mapsto \sum_{0\leq j\leq n-1} \pi_H^*(a_j) \cup \xi^j_H, 
\end{equation}
where 
$$
\xi_H\in H^1(H, \Omega^1_{H/k})
$$ 
is the restriction $\iota_H^*(\xi)$ of $\xi$ to $H$. 
For every local section $a\in \Omega^{i-1-j}_{\ul{X}}$ a direct local computation and the definition of the residue map imply 
\begin{equation}
\label{eq:eqboxinvdiff9}
f( \pi_H^*(a)\cup \xi_H^j) = \pi^*(a)\cup \xi^{j+1}. 
\end{equation}
 
Combining \eqref{eq:eqboxinvdiff3}, \eqref{eq:eqboxinvdiff4},  \eqref{eq:eqboxinvdiff5}, and \eqref{eq:eqboxinvdiff9} we obtain the commutative diagram
\begin{equation} \label{eq:eqboxinvdiff6}
\begin{tikzcd} \bigoplus_{0\leq j\leq n-1} \Omega^{i-1-j}_{\ul{X}}[- j-1] \arrow[r, "\simeq"] \arrow[d, swap, "\nu"]    &  R\pi_{H,*} \Omega^{i-1}_{\ul{X}\times H}[-1] \arrow[d, "f"]\\
\bigoplus_{0\leq j\leq n} \Omega^{i-j}_{\ul{X}} [-j]  \arrow[r, "\simeq"] & R\pi_{ *}\Omega^i_{\ul{X}\times \P^n} 
\end{tikzcd}
\end{equation}
where the left vertical morphism $\nu$ is given by the (split) inclusion into the direct sum 
$$
\bigoplus_{0\leq j\leq n} \Omega^{i-j}_{\ul{X}} [-j]
$$ 
and the horizontal morphisms are quasi-isomorphisms. 
Using \eqref{eq:eqboxinvdiff3} and the equality $\pi^*(-) = \pi^*(-)\cup \xi^0$. 
We deduce the isomorphism \eqref{eq:eqboxinvdiff} by noting that the cokernel of $\nu$ is exactly $\Omega^i_{\ul{X}}$. 
\vspace{0.1in}

Next, we assume that $\partial X \neq \emptyset$, so that $m\geq 1$.  
We proceed following the script of Sato \cite[\S 2]{SatoLogHW} (a sketch for relative differentials can be found  in \cite[\S 6.2]{BS}). 
\vspace{0.1in}

Write $I = \{1, \ldots, m \}$ and for $1\leq a\leq m$, we define the disjoint union
\[ 
\partial X^{[a]} = \coprod_{1\leq i_1<\cdots< i_a\leq m} \partial X_{i_1}\cap \ldots \cap \partial X_{i_a}.
\]
Note that each $\partial X^{[a]}$ is regular (since $\partial X$ is a strict normal crossing divisor on $\ul{X}$), 
and that there is a canonical morphism 
$$
\nu_a\colon \partial X^{[a]}\to \ul{X}
$$ 
for each $a\geq 1$.  
The log structure on $X$ induces the compactifying log structure on each $\partial X^{[a]}$. 
More explicitely, define for each $a\geq 1$ the divisor $E_a$ on $\partial X^{[a]}$ by the formula
\[
E_a 
= 
\coprod_{1\leq i_1< \ldots < i_a\leq m} \big(  (\partial X_{i_1}\cap \ldots \cap \partial X_{i_a}) \cap \sum_{j\notin \{i_1, \ldots, i_a\} } \partial X_j  \big).
\]
This is a normal crossing divisor on $\partial X^{[a]}$ (with strictly less than $m$ components on each irreducible component of $\partial X^{[a]}$). 
The $cdh$-covering of $\partial X$ given by the $\partial X^{[a]}$'s gives rise to the exact sequence of sheaves 
\begin{align}
\label{eq:exboxinvdiff7} 
0\to \Omega^{i}_{\ul{X}} \xrightarrow{\epsilon_X} \Omega^i_{\ul{X}}(\log \partial X) \to 
& \nu_{1, *} \Omega^{i-1}_{\partial X^{[1]}} \xrightarrow{\rho_2} \nu_{2, *}\Omega^{i-2}_{\partial X^{[2]}} \to\ldots \\ \nonumber
\ldots& \to \nu_{a, *} \Omega^{i-a}_{\partial X^{[a]}} \xrightarrow{\rho_a} \nu_{a+1, *} \Omega^{i-a-1}_{\partial X^{[a+1]}}\to \ldots,
\end{align}
on $\ul{X}$, 
where each morphism $\rho_{a}$ is induced by the alternating sum on residues, see e.g., \cite[Proposition 2.2.1]{SatoLogHW}. 
\vspace{0.1in}

We may repeat the construction above for $X\times (\P^n, \P^{n-1})$. 
First, note that there is an evident canonical morphism of log schemes 
$$
X\times (\P^n, \P^{n-1}) \to \ul{X}\times (\P^n, \P^{n-1}).
$$
Letting $\nu_a$ denote the natural map $\partial X^{[a]} \times \P^n \to \ul{X}\times \P^n$, 
we get the following exact sequence of sheaves on $\ul{X}\times \P^n$ (using the above convention that $H$ stands for $\ul{X}\times H$ as a divisor on $\ul{X}\times \P^n$)
\begin{align}
\label{eq:exboxinvdiff8} 
0\to \Omega^{i}_{\ul{X}\times \P^n}(\log H)& \xrightarrow{\epsilon_{X\times \P^n}} \Omega^i_{\underline{X} \times  \P^n}(\log H+\partial X) \to  
\nu_{1, *} \Omega^{i-1}_{\partial X^{[1]}\times \P^n} \xrightarrow{\rho_2} \ldots \\ 
\nonumber \ldots & \to \nu_{a, *} \Omega^{i-a}_{\partial X^{[a]}\times \P^n}  \xrightarrow{\rho_a} 
\nu_{a+1, *} \Omega^{i-a-1}_{\partial X^{[a+1]}\times \P^n} \to \ldots 
\end{align}
Here, on each scheme $\partial X^{[a]}\times \P^n$ we consider the compactifying log structure given by the inclusion of the complement of the normal crossing divisor 
$E^{[a]}\times \P^n + \partial X^{[a]} \times H$. 
To simplify the notation, 
we set 
$$
A_a:= \Omega^{i-a}_{\partial X^{[a]}\times \P^n} 
\text{ and } 
B_a:=\Omega^{i-a}_{\partial{X}^{[a]}}.
$$ 
Let $F_A$ be the cokernel of $\epsilon_{X\times \P^n}$ and let $F_B$ be the cokernel of $\epsilon_X$. 
Then using \eqref{eq:exboxinvdiff8} we obtain the spectral sequence 
\[ 
R^{a+q}\pi_* (\nu_{a+1, *} A_{a+1}) \Rightarrow R^{a+q}\pi_* F_A.
\]

Now, by induction, there is an isomorphism 
$$
\pi^*\colon \nu_{a+1, *} B_{a+1}\xrightarrow{\simeq}  \pi_* (\nu_{a+1, *} A_{a+1}),
$$ 
and the term $R^{a+q}\pi_* (\nu_{a+1, *} A_{a+1})$ vanishes for $q\neq -a$, thus we have 
$$
F_B \xrightarrow\simeq R \pi_* F_A
$$ 
using the above spectral sequence. 
The proposition follows now from the 5-lemma and the case $\partial X = \emptyset$ by comparing the two distinguished triangles
\[
\Omega^i_{\ul{X}} \xrightarrow{\epsilon_X} \Omega^i_{\ul{X}}(\log \partial X) \to F_B\to \Omega^i_{\ul{X}}[1]
\]
and 
\[
R\pi_* \Omega^{i}_{\ul{X}\times \P^n}(\log H) 
\xrightarrow{\epsilon_{X\times \P^n}} R\pi_*\Omega^i_{\underline{X} \times  \P^n}(\log H+\partial X) \to R\pi_* F_A \to R\pi_* \Omega^{i}_{\ul{X}\times \P^n}(\log H) [1]
\]
in the bounded derived category ${\bf D}^b(\ul{X}_{\rm Zar})$. 
\end{proof}
\end{prop}

\begin{cor}
\label{Pn-invariance of Log differentials-Zar}
Let $X$ be an fs log scheme in $SmlSm/k$. 
Then for every $i$, $j$, and $n$, there is an isomorphism
\[
H_{Zar}^i(X\times (\P^n,\P^{n-1}),\Omega^j)\cong H_{Zar}^i(X,\Omega^j).
\]
\begin{proof}
This is immediate from Proposition \ref{prop:boxinvdiff}.
\end{proof}
\end{cor}

\begin{prop}
\label{prop::set=ket}
Let $X$ be an fs log scheme, and let $\cF$ be a quasi-coherent sheaf on $X$.
Then for every $i\geq 0$, there is an isomorphism
\[
H_{s\acute{e}t}^i(X,\cF)\xrightarrow{\cong} H_{k\acute{e}t}^i(X,\cF).
\]
\end{prop}
\begin{proof}
The question is strict \'etale local on $X$.
Hence we may assume that $X$ is affine.
Our claim follows now from Kato's \cite[Proposition 6.5]{KatoLog2} with a proof given by Niziol in \cite[Proposition 3.27]{MR2452875}.
\end{proof}

\begin{prop}
\label{Invariance of coherent cohomology under log blow-ups}
Let $f\colon Y\rightarrow X$ be a log modification in $lSm/k$, and let $\cF$ be a locally free coherent sheaf on $X$.
Then for every $i\geq 0$ and $j\geq 0$ there is an isomorphism
\[
H_{s\acute{e}t}^i(X,\cF)\xrightarrow{\cong} H_{s\acute{e}t}^i(Y,f^*\cF).
\]
\end{prop}
\begin{proof}
The question is strict \'etale local on $X$, so we may assume that $X$ has an fs chart $P$.
Owing to Proposition \ref{A.9.21} there is a subdivision $\Sigma$ of $\Spec{P}$ such that there is an isomorphism
\[
Y\times_{\A_P}\A_\Sigma\xrightarrow{\cong} X\times_{\A_P}\A_\Sigma.
\]
Let $p:X\times_{\A_P}\A_\Sigma\rightarrow X$ and $q:Y\times_{\A_P}\A_\Sigma\rightarrow X$ be the projections.
There is a naturally induced commutative diagram
\[
\begin{tikzcd}
H_{s\acute{e}t}^i(X,\cF)\arrow[d]\arrow[r]&H_{s\acute{e}t}^i(X\times_{\A_P}\A_\Sigma,p^*\cF)\arrow[d,"\cong"]
\\
H_{s\acute{e}t}^i(Y,f^*\cF)\arrow[r]&X_{s\acute{e}t}^i(Y\times_{\A_P}\A_\Sigma,q^*\cF).
\end{tikzcd}
\]
The horizontal morphisms are isomorphisms due to \cite[Remark 5.19]{Vetere}, 
which is a consequence of \cite[Theorem 11.3]{MR1296725} and the projection formula \cite[Tag 01E8]{stacks-project}.
Thus also the left vertical morphism is an isomorphism.
\end{proof}

\begin{cor}
\label{Invariance of Log differentials under log blow-ups}
Let $f\colon Y\rightarrow X$ be a log modification in $lSm/k$.
Then for every $i\geq 0$ and $j\geq 0$ there is an isomorphism 
\[
f^*\colon H_{s\acute{e}t}^i(X,\Omega^j)\xrightarrow{\cong} H_{s\acute{e}t}^i(Y,\Omega^j).
\]
\end{cor}
\begin{proof}
Since $f$ is log \'etale we have $f^*\Omega_X^j\cong \Omega_Y$.
Proposition \ref{Invariance of coherent cohomology under log blow-ups} allows us to conclude the proof.
\end{proof}

\begin{prop}
\label{prop::hodgelet}
Let $X$ be an fs log scheme in $lSm/k$.
For every $i\geq 0$ and $j\geq 0$ there exist functorial isomorphisms 
\begin{equation}
\begin{split}
\label{eqn::hodgelet1}
H_{Zar}^i(X,\Omega^j)
\cong^{(1)} &
H_{sNis}^i(X,\Omega^j)
\cong^{(3)}
H_{dNis}^i(X,\Omega^j)
\\
\cong^{(1)} &
H_{s\acute{e}t}^i(X,\Omega^j)
\cong^{(3)}
H_{d\acute{e}t}^i(X,\Omega^j)
\\
\cong^{(2)} &
H_{k\acute{e}t}^i(X,\Omega^j)
\cong^{(3)}
H_{l\acute{e}t}^i(X,\Omega^j)
\end{split}
\end{equation}
\end{prop}
\begin{proof}
The isomorphisms labeled (1) follow from the fact that the cohomology of coherent sheaves with respect to the Zariski, Nisnevich, and \'etale topologies coincides, 
see \eqref{eqn:coherentcohomology}. 
Proposition \ref{prop::set=ket} implies (2).
Finally, 
the isomorphisms labeled (3) follow from Theorem \ref{Div.3} and Corollary \ref{Invariance of Log differentials under log blow-ups}.
\end{proof}

\begin{rmk}
The comparison of strict \'etale, Kummer \'etale, and log \'etale Hodge cohomology groups are also carried out in \cite[Theorems 6.13, 6.14]{Vetere}.
\end{rmk}

\begin{cor}
\label{Pn-invariance of Log differentials}
Let $X$ be an fs log scheme in $lSm/k$.
Suppose that $t$ is one of following topologies: $Zar$, $sNis$, $s\acute{e}t$, $k\acute{e}t$, $dNis$, $d\acute{e}t$, and $l\acute{e}t$.
Then for every $i$, $j$, and $n$, there is an isomorphism
\[
H_{t}^i(X\times (\P^n,\P^{n-1}),\Omega^j)\cong H_{t}^i(X,\Omega^j).
\]
\begin{proof}
Owing to Corollary \ref{Invariance of Log differentials under log blow-ups} and Propositions \ref{prop::hodgelet} and \ref{A.3.19},
we reduce to the case when $X\in SmlSm/k$.
Then apply Corollary \ref{Pn-invariance of Log differentials-Zar}.
\end{proof}
\end{cor}

\subsection{A recollection on Grothendieck duality}\label{ssec:Grothduality}
We recall some material from Grothendieck-Verdier-Neeman duality theory for coherent sheaves \cite{HartshorneRD}. 
If $\ul{X}$ is a separated scheme of finite type over $k$, 
and $\pi_{\ul{X}}\colon \ul{X}\to \Spec{k}$ is the structural morphism, we have a pair of adjoint functors $(R \pi_{\underline{X}, *}, \pi_{\ul{X}}^!)$ between 
${\bf D}^b(k)$ and ${\bf D}^b(\ul{X}_{\rm Zar}) = {\bf D}^b(\ul{X})$ (since we only consider the Zariski topology we often omit the subscript).
More generally, 
if $f\colon \ul{X}\to \ul{Y}$ is any morphism of schemes, 
there is an adjoint pair
\begin{equation} \label{eq:duality_extended}
Rf_*\colon {\bf D}_{qc}(\ul{X})\leftrightarrows  {\bf D}_{qc}(\ul{Y}) \colon f^\times 
\end{equation}
between the derived categories of quasi-coherent sheaves on $\ul{X}$ and $\ul{Y}$ (see \cite[Reminder 3.1.1]{Neeman_simple}). When $f$ is proper, the functor $f^\times$ agrees with the familiar functor $f^!$ classically considered \cite{HartshorneRD}. When moreover $f$ is proper and of finite Tor dimension (this is the case which is going to be interesting for us), the adjunction \eqref{eq:duality_extended} restricts to an adjunction between the bounded derived categories of coherent sheaves (see \cite[Application 6.1.3]{Neeman_simple})
\begin{equation} \label{eq:duality_classic}
Rf_*\colon {\bf D}^b(\ul{X})\leftrightarrows  {\bf D}^b(\ul{Y}) \colon f^!. 
\end{equation}
In this situation, it induces  a natural transformation of functors
\[
Rf_* R \Hom_{\ul{X}}(-, f^{!}(-)) \xrightarrow{\simeq} R \Hom_{\ul{Y}}(Rf_*(-), -)
\]
called the standard Grothendieck duality isomorphism ($\Hom$ refers to the inner Hom of sheaves on $\ul{X}_{\rm Zar}$ to avoid confusion with the inner Hom in the categories of motives). See \cite[(3.4.10), p.150]{Conrad_duality}.
We denote by $D_{\ul{X}}$ the dualizing functor
\[ 
D_{\ul{X}} 
= 
R \Hom_{\ul{X}}(-, \pi_{\ul{X}}^{!}(k)) 
\colon 
{\bf D}^b(\ul{X})\to {\bf D}^{b}(\ul{X}).
\]
For any separated morphism of finite type, $f\colon \ul{X}\to \ul{Y}$, the functor $f^!$ can be defined by appropriately choosing a factorization of $f$ as an open embedding followed by a proper map (but unless it agrees with $f^{\times}$, it is not the right adjoint to $Rf_*$). However, it can be computed explicitly in some cases of interest.  If $\ul{X}$ is smooth of dimension $d$ over $k$, then there is a canonical isomorphism 
$$
\pi_{\ul{X}}^! k \cong \Omega^d_{\ul{X}}[d].
$$ 
Moreover, 
for every $j\geq 0$ and $n\in \Z$, 
there is an isomorphism (canonical up to the choice of a sign) 
$$
\Omega_{\ul{X}}^j[n]\xrightarrow{\simeq} D_{\ul{X}}(\Omega^{d-j}_{\ul{X}})[n-d].
$$ 
The latter isomorphism arises from exterior products of forms
\[
\Omega^j_{\ul{X}} 
\xrightarrow{\simeq} 
\Hom_{\ul{X}}(\Omega^{d-j}_{\ul{X}}, \Omega^d_{\ul{X}}), \quad \omega
\mapsto 
(\eta\mapsto \omega\wedge \eta). 
\]
\vspace{0.01in}

Our next goal is to extend the above results to the log setting.
For the analytic setting, 
see Esnualt-Viehweg \cite{EsnaultViehweglogDR}, 
which in turn was inspired by analogous results in the theory of $\mathcal{D}$-modules (see e.g., Bernstein \cite{BBDModules}). 
We begin with the following definition.

\begin{df}
Let $X$ be an fs log scheme in $SmlSm/k$. 
By Lemma \ref{lem::SmlSm}, we have that $X = (\ul{X}, \partial X)$, 
where $\ul{X}$ is a smooth $k$-scheme and $\partial X =\partial X_1+\ldots \partial X_r$ is a strict normal crossing divisor on $\ul{X}$. 
We let $I_{\partial X}$ be the invertible sheaf of ideals defining $\partial X$. 
\end{df}


\begin{lem}
\label{lem:dualitylog}
Let $X$ be an integral fs log scheme in $SmlSm/k$ of dimension $d_X$.  
Then, for any $j\geq 0$ and every $n$, there is an isomorphism 
\begin{equation}\label{eq:dualitylog} 
\theta_X\colon \Omega^j_X[n] = \Omega^j_{\ul{X}}(\log \partial X)[n] 
\xrightarrow{\simeq} 
D_{\ul{X}} (\Omega^{d_X-j}_{X} \tensor_{\cO_{\ul{X}}} I_{\partial X})[n-d_X]
\end{equation} 
in ${\bf D}^b(\ul{X})$.
\end{lem}
\begin{proof} 
To prove \eqref{eq:dualitylog}, we note there is an isomorphism of sheaves
\begin{equation}\label{eq:dualitylogbis} 
\Omega^j_{\ul{X}}(\log \partial X) 
\xrightarrow{\simeq} 
\Hom_{\ul{X}}(\Omega^{d_X-j}_{\ul{X}}(\log \partial X) \tensor_{\cO_{\ul{X}}} I_{\partial X}, \Omega^{d_X}_{\ul{X}} ) 
\end{equation}
induced by the exterior product of forms 
$$
\omega\mapsto (\beta \tensor u \mapsto \omega\wedge (\beta \tensor u)),
$$ 
where $u$ is a local section of $I_{\partial X}$. 
Here we are implicitly making a choice of sign since the assignment $\beta\tensor u \mapsto \beta\wedge (\omega \tensor u)$ is equally valid. 
We note the inclusion
$$
\Omega^{d_X-j}_{\ul{X}}(\log \partial X) \tensor_{\cO_{\ul{X}}} I_{\partial X} \subset \Omega^{d_X-j}_{\ul{X}}
$$ 
(every section of the sheaf is a regular differential).
\vspace{0.1in}

To show \eqref{eq:dualitylogbis} is an isomorphism, we resort to computing in local coordinates.
We have 
$$
\Omega^{d_X}_{\ul{X}}(\log \partial X) \tensor I_{\partial X} \cong \Omega^{d_X}_{\ul{X}},
$$ 
since
\[ 
(d\log x_1 \wedge \ldots d\log x_r \wedge dx_{r+1} \wedge \ldots dx_{d_X})\, x_1\cdots x_r = dx_1\wedge \ldots dx_{d_X}, 
\]
where $(x_1\cdots x_r) = I_{\partial X}$ locally around any point of $\partial X$. 
Next, note that for every integer $n\in \Z$, we have
\begin{align} 
\label{eq:dualitylog2} 
&\Hom_{\ul{X}}(\Omega^{d_X-j}_{\ul{X}}(\log \partial X) \tensor_{\cO_{\ul{X}}} I_{\partial X}, \Omega^{d_X}_{\ul{X}} )[n] && \\ 
\nonumber&= \Hom^{\bullet}_{\ul{X}}(\Omega^{d_X-j}_{\ul{X}}(\log \partial X) \tensor_{\cO_{\ul{X}}} I_{\partial X}, \Omega^{d_X}_{\ul{X}}[d_X])[n-d_X]  && 
\text{(sheaf ${\rm Hom}$ complex)}\\\nonumber
& =R {\Hom}_{\ul{X}}( \Omega^{d_X-j}_{\ul{X}}(\log \partial X) \tensor_{\cO_{\ul{X}}} I_{\partial X} , \pi_X^{!}(k))[n-d_X] && 
\text{($\Omega^{d_X-j}_{\ul{X}}(\log \partial X)$ is locally free) }
\end{align}
Here, 
we have used the notation $\Hom_{\ul{X}}^\bullet$ for the inner Hom of complexes of sheaves on $\ul{X}$, and $R\Hom$ for the derived functor of $\Hom_{\ul{X}}$. 
By combining \eqref{eq:dualitylogbis} with \eqref{eq:dualitylog2} we get the desired isomorphism. 
\end{proof}

\begin{rmk} 
\begin{enumerate}
\item[(1)] One can use the twisted dualizing functor 
$$
D_X = D_{\ul{X}}( - \tensor I_{\partial X})
$$ 
to prove that \eqref{eq:dualitylog} holds more generally for the logarithmic de Rham complex $DR_{\partial X} (\mathcal{V})$ associated to a coherent 
$\cO_{\ul{X}}$-module $\mathcal{V}$ equipped with a meromorphic connection with regular singularities along $\partial X$. 
For details see \cite[(A.2)]{EsnaultViehweglogDR}.
\item[(2)]  Let $j\colon U=\ul{X}- \partial X\hookrightarrow \ul{X}$ be the open immersion of the complement of the boundary divisor on $\ul{X}$. 
Then, for every $j\in \Z$, we have
\[ 
\Omega^j_U  
=    
(\Omega^j_X)_{| U} \cong D_{\ul{X}}(\Omega^{d_X-j}_X\tensor I_{\partial X})_{|U}[-d_X] = D_U(\Omega^{d_X-j}_U)[-d_X].
\]
\end{enumerate}
\end{rmk}

\subsection{Cohomology with support and cycle classes} 
Following \cite{ChatzistamatiouRullingANT}, in order to define the action of finite log correspondences on logarithmic differentials, 
we will need to construct a cycle class in Hodge cohomology associated with a closed subscheme. 
A priori, this is potentially a subtle point since we might need to deal with closed immersions of schemes that are not strict as morphisms of log schemes 
(this is the case for an arbitrary finite log correspondence). 
As discussed later in this section, this can be avoided,  
and the following discussion of cohomology with support in the log setting suffices for our purposes.

\begin{const}
Let $X$ be an integral fs log scheme of finite type over $k$,
and let $v\colon Z\hookrightarrow X$ be a strictly closed immersion of log schemes. 
For $\cF$ a Zariski sheaf of abelian groups on $X_{Zar}$ we set 
$$
\Gamma_Z(X,\cF):=\ker (\cF(X) \to \cF(X-Z)).
$$ 
It is elementary to see that the functor $\Gamma_Z(X,-)$ is left exact, and there exists an induced derived functor
\[
R\Gamma_Z(X,-):\Deri^b(X_{Zar})\rightarrow \Deri^b({\bf Ab}).
\]

We define cohomology with support by
$$
H_Z^i(X,\cF):=R\Gamma_Z^i(X,\cF).
$$ 
Note that we have $R\Gamma_X(X,-)=R\Gamma(X, -)$. 
For $Z\subset X$ strict as above, we define the \emph{Hodge cohomology of $X$ with support in $Z$} (cf.\ \cite[\S 2]{ChatzistamatiouRullingANT}) as the sum 
$$
H_{\rm Hodge}^*(X, Z) = \bigoplus_{i,j \geq 0}H^i_{Z}(X, \Omega^j_X).
$$ 
If $Z$ is empty, we write $H_{\rm Hodge}^*(X)$ for $H_{\rm Hodge}^*(X, \emptyset)$.
\end{const}


\begin{const}
\label{constr:pullback}
For $f\colon X\to Y$ a morphism and $Z$ a strictly closed subscheme of $Y$, 
we write $f^{-1}(Z)$ for $X\times_Y Z$ seen as a strictly closed subscheme of $X$. 
There is an isomorphism of functors 
$$
R\Gamma_{f^{-1}(Z)} \xrightarrow{\cong} R\Gamma_{Z} Rf_*.
$$ 

We say that a strict closed subscheme $W$ of $X$ contains $f^{-1}(Z)$ if the strict closed immersion $f^{-1}(Z)\hookrightarrow X$ admits a factorization $f^{-1}(Z)\subset W \hookrightarrow X$.
In this situation the latter isomorphism furnishes a natural transformation
\begin{equation}
\label{eq:pullbacksupport}
f^*\colon 
R\Gamma_Z \to R\Gamma_{Z} Rf_* Lf^* 
\xrightarrow{\cong} 
R\Gamma_{f^{-1}(Z)}Lf^* 
\to 
R\Gamma_W Lf^*.
\end{equation}
See \cite[\S 2.1.5]{ChatzistamatiouRullingANT} for more details. 
The displayed natural transformations are morphisms in the category of functors from $\Deri_{qc}({Y}_{Zar})$ to $\Deri({\bf Ab})$, and restricts to a morphism in the category of functors from $\Deri^b({Y}_{Zar})$ to $\Deri^b({\bf Ab})$ when $f$ has finite Tor dimension.

Note that here the log structures on $X$ and $Y$ are irrelevant since \eqref{eq:pullbacksupport} only involves standard operations on sheaves. 
The pullback  $f^*$ in (\ref{eq:pullbacksupport})  is compatible with composition, see \cite[(2.1.8)]{ChatzistamatiouRullingANT}.
\end{const}

\begin{rmk} 
\label{rmk:pullback}
\begin{enumerate}
\item[(1)] 
Every object and every morphism involved in the natural transformation \eqref{eq:pullbacksupport} belong to the bounded derived category of coherent sheaves on $\ul{X}$ or $\ul{Y}$.  
In particular, 
the sheaf 
$$
\Omega^j_{X} = \Omega^j_{\ul{X}}(\log \partial X)
$$ 
is considered as a coherent sheaf on $\ul{X}$ (and likewise for $\Omega^j_{Y}$ and $\ul{Y}$). 
In other words, 
the log structures on $X$ and $Y$ do not play any role in defining the derived categories that we consider.
Moreover, the sheaves of log differentials completely capture the log geometric information. 

Since $\Omega^j_Y$ is locally free when $Y\in SmlSm/k$, 
we have $Lf^* \Omega^j_Y=f^*\Omega^j_Y$ and composing with the morphism $f^*\Omega^j_Y\to \Omega^j_X$ from \cite[Proposition IV.1.2.15]{Ogu} we obtain
\[ 
f^*\colon R\Gamma_Z \Omega^j_Y \to R\Gamma_W \Omega^j_X. 
\]
In particular, this yields the pullback morphism in Hodge cohomology 
\[
f^*\colon H^*_{\rm Hodge}(Y, Z)\to H^*_{\rm Hodge}(X, W).
\]
In the above $Y\in SmlSm/k$ for simplicity, we do not impose the same assumption on $X$.
    
\item[(2)]
If $X\in SmlSm/k$, there is a canonical morphism to the underlying scheme $p_X\colon X\to \ul{X}$.
For every (irreducible) closed subscheme $\ul{Z}$ of $\ul{X}$, the morphism
$$
Z=\ul{Z}\times_{\ul{X}} X \to X
$$ 
is a strictly closed immersion ($Z$'s log structure is given by restricting $X$'s log structure). 
Applying Construction \ref{constr:pullback} to $p_X$ yields the morphism
\[
p_X^*\colon H^*_{\rm Hodge}(\ul{X}, \ul{Z}) 
= 
\bigoplus_{i,j}H^i_{\ul{Z}}(\ul{X}, \Omega_{\ul{X}}^j) 
\to 
\bigoplus_{i,j} H^i_{Z}(X, \Omega^j_X) = H^{*}_{\rm Hodge}(X, Z).
\]
Here $H^i_{Z}(X, \Omega^j_X)$ is the cohomology group with support $H^i_{\ul{Z}}(\ul{X}, \Omega^j_{\ul{X}}(\log \partial X))$.
The morphism $p_X^*$ is induced by the inclusion of locally free sheaves
$$
\Omega^j_{\ul{X}}\to \Omega^j_{\ul{X}}(\log \partial X)
$$ 
on $\ul{X}$ or, intrinsically using log geometry, is the canonical morphism discussed in \cite[Example IV 1.2.17]{Ogu}.
\item[(3)] We can further compose the morphism in (2) with the morphism induced by the open immersion $j\colon X- \partial X \to X$ to obtain
\[
H^*_{\rm Hodge}(\ul{X}, \ul{Z}) \xrightarrow{p_X^*} H^*_{\rm Hodge}( X, Z )\xrightarrow{j^*} H^*_{\rm Hodge}(X- \partial X, Z- \partial Z), 
\]
where $Z- \partial Z = Z\cap (X- \partial X)$. 
The composite morphism is induced by 
$$
\Omega^j_{\ul{X}} \to j_* \Omega^j_{X- \partial X}, 
$$
which clearly factors through the sheaf $\Omega^j_{\ul{X}}(\log \partial X)$.
\end{enumerate}
\end{rmk}
\vspace{0.1in}

Next, we discuss a derived pushforward functor following the script in \cite[\S 2.2]{ChatzistamatiouRullingANT}. 
Owing to Remark \ref{rmk:pullback}(3) it suffices to perform our constructions in the bounded derived category $\Deri^b(\ul{X})$.

\begin{const}
\label{const:push}
Let $f\colon X\to Y$ be a proper map between integral fs log schemes.
We will assume $Y\in SmlSm/k$ and that $X$ is normal with compactifying Deligne-Faltings log structure $\partial X\to X$, 
where $\partial X$ is an effective Cartier divisor on $X$.
A particular case is the boundary of $X\in lSm/k$.
Recall that $D_{\ul{X}}(-)$ is the dualizing functor on $\Deri^b(\ul{X})$.  
Using standard coherent duality theory \cite{HartshorneRD} we obtain a morphism in $\Deri^b(\ul{Y})$
\begin{equation}
\label{eq:push1}
f_*\colon Rf_* D_{\ul{X}}(\Omega^q_{X}) \to D_{\ul{Y}}(\Omega^q_Y), 
\end{equation}
defined as the composition
\begin{align*}
Rf_* R\Hom_{\ul{X}}(\Omega^q_X, \pi_{\ul{X}}^!(k)) &\xrightarrow{\cong} Rf_* R\Hom_{\ul{X}}(\Omega^q_X, f^!\pi_{\ul{Y}}^!(k)) \\ 
& \to R\Hom_{\ul{Y}}( Rf_*\Omega^q_X, Rf_* f^! \pi_{\ul{Y}}^!(k)) \\
& \xrightarrow{Tr_f} R\Hom_{\ul{Y}}( Rf_*\Omega^q_X, \pi_{\ul{Y}}^!(k)) \\ 
& \xrightarrow{f^*} R\Hom_{\ul{Y}}(  \Omega^q_Y, \pi_{\ul{Y}}^!(k)).
\end{align*}
Here $Tr_f$ is gotten from the trace morphism $Rf_* f^!\to \id$, and $f^*\colon \Omega_Y^q\to Rf_* \Omega^q_X$  is the pullback of forms with log poles. 
If we further assume that $f$ is a strict morphism of log schemes, 
then we have 
$$
\Omega^q_X \tensor I_{\partial X} 
= 
\Omega^q_X\tensor f^* I_{\partial Y}
$$ 
where $f^*$ denotes the pullback of coherent sheaves (locally free, in this case). 
Further the strictness assumption implies $f^{-1}(\partial Y) = \partial X$ as effective Cartier divisors on $X$.
Thus the construction of \eqref{eq:push1} furnishes a morphism
\begin{equation}
\label{eq:push2}
f_*\colon Rf_* D_{\ul{X}}(\Omega^q_X\tensor I_{\partial X}) \to D_{\ul{Y}}(\Omega^q_Y\tensor I_{\partial Y}).
\end{equation} 
\end{const}

\begin{rmk} 
The assumption that $Y\in SmlSm/k$ and $X$ has a compactifying Deligne-Faltings log structure is needed to incorporate 
the ideal sheaves $I_{\partial X}$ and $I_{\partial Y}$ in Construction \ref{const:push}, 
but it is not used to define \eqref{eq:push1}.
In Construction \ref{Construction of pushforward with support} we use the assumption on $X$ to define push forward with 
proper support.
\end{rmk}

For the proof of the following result, we refer the reader to \cite[Proposition 2.2.7]{ChatzistamatiouRullingANT}.

\begin{lem}
\label{lem:obviousproperties} 
For the pushforward morphism in (\ref{eq:push1}) we have $(\id)_*=\id$. 
If $f\colon X\to Y$ and $g\colon Y\to Z$ are proper maps between fs log schemes in $SmlSm/k$, then 
$$
(g\circ f)_* = g_* \circ Rg_*(f_*)\colon Rg_*Rf_*D_{\ul{X}}(\Omega^q_X)\to D_{\ul{Z}}(\Omega^q_Z) 
$$
for every $q\geq 0$. 
Moreover, if $f$ and $g$ are strict, the same holds for the pushforward morphism in \eqref{eq:push2}.
\end{lem}

Let $i\colon X\hookrightarrow Y$ be a strict closed immersion of fs log schemes in $SmlSm/k$, 
and assume that $\ul{X}$ has pure codimension $c$ in $\ul{Y}$, 
so that $\ul{X}\hookrightarrow \ul{Y}$ is a closed immersion of smooth schemes of pure codimension $c$. 
If we combine \eqref{eq:push2} with Lemma \ref{lem:dualitylog} then for every $g\geq0$ we have the composite
\begin{equation}
\label{eq:push3closedimm}
i_* \Omega^q_X  
\xrightarrow{i_*\theta_X} i_*D_{\ul{X}} (\Omega_{X}^{d_X-q}\tensor I_{\partial X})[-d_X] 
\xrightarrow{i_*} D_{\ul{Y}}(\Omega_Y^{d_X-q}\tensor I_{\partial Y})[-d_X]
\xrightarrow{\theta_Y} \Omega^{c+q}_{Y}[c].
\end{equation}
Here $\theta_X$ is the duality isomorphism \eqref{eq:dualitylog} for $X$,  
$\theta_Y$ is the duality isomorphism for $Y$, 
and $d_X$ is the dimension of $X$.

\begin{lem}
\label{functorialitypush} 
Let $i\colon X\hookrightarrow Y$ be a strictly closed immersion of fs log schemes in $SmlSm/k$. 
Let $f\colon Y'\to Y$ be a  morphism in $SmlSm/k$ and assume there is a cartesian square
\[
\begin{tikzcd}
X'\arrow[d,"f'"']\arrow[r,"i'"]&Y'\arrow[d,"f"]\\
X\arrow[r,"i"]&Y
\end{tikzcd}
\]
where $X'\in SmlSm/k$.  
Assume that $\ul{X}$ has pure codimension $c=d_Y-d_X=d_{Y'}-d_{X'}$, where $d_{(-)}$ denotes the dimension of a scheme.
\vspace{0.1in}

Then for all $q\geq 0$, there is a commutative diagram 
\[
\begin{tikzcd}
i_* Rf'_* \Omega_X'^q \arrow[r,"\eqref{eq:push3closedimm}'"]& Rf_* \Omega^{c+q}_{Y'}[c] \\
i_* \Omega^q_X \arrow[u,"(f')^*"] \arrow[r,"\eqref{eq:push3closedimm}"]& \Omega^{c+q}_{Y}[c]. \arrow[u, "f^* "']
\end{tikzcd}
\]
\end{lem}
\begin{proof} 
The proof requires an explicit description of the map \eqref{eq:push3closedimm} using local cohomology sheaves. 
Since the log structure does not cause any complications, 
the argument in \cite[Corollary 2.2.22]{ChatzistamatiouRullingANT} applies, \emph{mutatis mutandis}, to our situation.
\end{proof}

Assume that $\ul{X}$ is integral and smooth. 
Applying $R\Gamma_X(Y,-)$ to \eqref{eq:push3closedimm}, we get a morphism
\[ 
R\Gamma(X,\Omega^q_X) = R\Gamma_X(Y,Ri_* \Omega^q_X) = R\Gamma_X(Y, i_* \Omega^q_X )
\to  
R\Gamma_X(Y,\Omega^{c+q}_{Y})[c] \]
in $\Deri^b(\mathbf{Ab})$.
Furthermore, on cohomology, we obtain
\[
H(i_*)\colon H^*(X, \Omega^q_X) \to H^{*+c}_X(Y, \Omega^{q+c}_Y).
\]
In particular, for $q=*=0$, there is a morphism
\begin{equation}
\label{eq:cycleclasssmooth}
H(i_*)\colon 
H^0(X, \cO_X) 
\to 
H^c_X(Y, \Omega_Y^c).
\end{equation}
The image of $1\in H^0(X, \cO_X)$ along the map $H(i_*)$ of \eqref{eq:cycleclasssmooth} is called the \emph{cycle class} of the strict closed subscheme $X$ in $Y$. 
In order to extend the definition to an arbitrary strict closed subscheme, we first establish some auxiliary results. 

\begin{lem}
\label{vanishing of hodge cohomology with support}
Let $\ul{X}$ be a smooth and separated scheme over $k$, and let $\ul{Z}$ be a closed subscheme of $\ul{X}$ of pure codimension $c$. 
Let $D$ be a strict normal crossing divisor on $\ul{X}$. 
Then, for any $i<c$ and $j$, we have the vanishing
\[
H_{\ul{Z}}^i(\ul{X},\Omega_{\ul{X}}^j(\log D))=0.
\]
\end{lem}
\begin{proof}
By the localization sequence, one reduces to the case where $X$ is local, and $Z$ is the closed point. 
In this case, the assertion follows from the vanishing of local cohomology, using that $\Omega_{\ul{X}}^j(\log D)$ is locally free and that a regular local ring is Gorenstein, 
see e.g., \cite[Theorem 6.3]{localcohomology}.
\end{proof}

Lemma \ref{vanishing of hodge cohomology with support} immediately implies the following result.
\begin{lem}
\label{injectivity of hodge cohomology}
Let $\ul{X}$ be a smooth and separated scheme over $k$, 
and let 
$$
\ul{W}\stackrel{\mu}\rightarrow \ul{Z}\stackrel{\nu}\rightarrow \ul{X}
$$ 
be closed immersions such that the codimension of $\nu$ (resp.\ $\mu$) is $c$ (resp.\ $\geq 1$).
Then the naturally induced homomorphism
\[
H_{\ul{Z}}^c(\ul{X},\Omega_{\ul{X}}^c)
\rightarrow
H_{\ul{Z}-\ul{W}}^c(\ul{X}-\ul{W},\Omega_{\ul{X}-\ul{W}}^c)
\]
is injective. 
\end{lem}

We can now recall the definition of Grothendieck's ``fundamental class'' for a closed subscheme of a smooth $k$-scheme.

\begin{lem}
\label{lem:cycleclassGroth}
Let $\ul{X}$ be a smooth and separated scheme over $k$, and let $\ul{Z}$ be a closed subscheme of $\ul{X}$ of pure codimension $c$.
Consider the  open immersion $j\colon U=\ul{X}-\ul{Z}_{sing}\rightarrow \ul{X}$, and consider the induced homomorphism \eqref{eq:cycleclasssmooth}
\[
H(i_*)
\colon 
H_{\ul{Z}_{\rm reg}}^0(\ul{Z}_{\rm reg},\cO_{\ul{Z}_{\rm reg}})
\rightarrow H_{\ul{Z}_{\rm reg}}^c(U,\Omega_{U}^c),
\]
where $\ul{Z}_{\rm reg}$ is the regular (equivalenty, smooth, since $k$ is perfect) locus of $\ul{Z}$.
Then the following statements hold.
\begin{enumerate}
\item[{\rm (1)}] The naturally induced homomorphism
\[
j^*:H_{\ul{Z}}^c(\ul{X},\Omega_{\ul{X}}^c)\rightarrow H_{\ul{Z}_{reg}}^c(U,\Omega_{U}^c).
\]
is injective.
\item[{\rm (2)}] There exists a unique element $cl(\ul{Z})$ in $H_{\ul{Z}}^c(\ul{X},\Omega_{\ul{X}}^c)$ such that
\[
j^*(cl(\ul{Z}))=H(i_*)(1).
\]
\end{enumerate}
\end{lem}
\begin{proof}
Assertion (1) follows from Lemma \ref{injectivity of hodge cohomology} and assertion (2) is explained in \cite[Proposition 3.1.1]{ChatzistamatiouRullingANT}
\end{proof}

Putting the above together, we arrive at the following construction of a cycle class.

\begin{df}
\label{Construction of cycle class}
Let $X$ be an integral fs log scheme in $SmlSm/k$,
and let $\nu:Z\rightarrow X$ be a strict closed immersion of fs log schemes such that $\ul{Z}$ has pure codimension $c$ in $\ul{X}$.
In this setting, we define the \emph{cycle class}\index{cycle class} of $Z$
\[
cl(Z)\in H_Z^c(X,\Omega_X^c)
\]
as follows: 
Let $j\colon X-\partial X\rightarrow X$ and $p_X\colon X\rightarrow \underline{X}$ denote the canonical morphisms to the underlying scheme of $X$. 
In degree $c$ we have the homomorphisms $p_X^*$ and $j^*$ 
\[
H_{\underline{Z}}^c(\underline{X},\Omega_{\underline{X}}^c)\stackrel{p_X^*}
\rightarrow H_Z^c(X,\Omega_X^c)\stackrel{j^*}
\rightarrow H_{Z-\partial Z}^c(X-\partial X,\Omega_{X-\partial X}^c)
\]
from Remark \ref{rmk:pullback}(3).
Since pullbacks are compatible with compositions, 
the composition in question is the restriction map $\underline{j}^*$ for cohomology with support on the underlying schemes (both $\ul{X}$ and $X-\partial X$ have trivial log structures). 
Lemma \ref{injectivity of hodge cohomology} shows $\underline{j}^*$ is injective.
It follows that also $p_X^*$ is injective,
and we set 
\[
cl(Z):=p^*_X(cl(\underline{Z})).
\]
Note that $j^*cl(Z)\cong cl(Z-\partial Z)$ by Lemma \ref{lem:cycleclassGroth} and Lemma \ref{functorialitypush}. 
By the same token, when $Z\in SmlSm/k$, the class $cl(Z)$ agrees with $H(i_*)(1)$ of \eqref{eq:cycleclasssmooth}. 
\end{df}

\subsection{Pushforward with proper support}
Let $f\colon X\to Y$ be a strict morphism of integral fs log schemes in $SmlSm/k$, 
and let $u\colon Z\hookrightarrow X$ be a strictly closed immersion of fs log schemes. 
We will assume that the restriction of $f$ to $Z$ is proper and surjective onto $Y$. 

\begin{df}
\label{df:goodcomp} 
A good compactification\index{good compactification} of $f$ as above is a factorization 
\[
f=\ol{f}\circ j\colon X\hookrightarrow \ol{X}\xrightarrow{\ol{f}} Y,
\]
where $\ol{X}\to Y$ is a proper and strict map, 
$X\hookrightarrow \ol{X}$ is a dense open immersion, and
$\ol{X}$ is an integral normal fs log scheme with compactifying Deligne-Faltings log structure $\partial X\to X$, 
where $\partial X$ an effective Cartier divisor on $\ul{X}$.
\end{df}
\begin{rmk}
\label{rmk:goodcomp2}
A good compactification exists for every $f$ as above.
Indeed, 
Nagata's compactification theorem \cite{ConradNagata} shows that every separated morphism of finite type $\ul{f}\colon \ul{X}\to \ul{Y}$ between Noetherian schemes admits a compactification, 
i.e., 
there exists a factorization into an open immersion followed by a proper map $\ol{f}\colon \ul{X}'\to Y$. 
Letting $\ol{X}$ denote the fs log scheme with log structure $(\ol{f}^*\cM_Y)_{\rm log}$ on $\ul{X}'$ we obtain a morphism $\ol{f}\colon\ol{X}\to Y$. 
By a standard argument, 
we may assume that $\ul{X}'$ is normal (simply take the normalization of $\ul{X}$ which does not alter the interior) and that the subscheme $\partial X$, 
where the log structure is nontrivial, 
is the support of an effective Cartier divisor. 
If we further assume that resolution of singularities holds over $k$, we may assume $\ol{X}\in SmlSm/k$.    
\end{rmk}

\begin{const}
\label{Construction of pushforward with support}
Let $f:X\rightarrow Y$ be a strict morphism of integral schemes in $SmlSm/k$. 
Let $u:Z\rightarrow X$ be a strictly closed immersion of fs log schemes and assume that the restriction of $f$ to $Z$ is proper and surjective.  
Set $d_X:=\dim X$, $d_Y:=\dim Y$, and $r:=d_X-d_Y$.
In this setting, 
for every $i$ and $j$, we will construct the pushforward morphism
\begin{equation}
\label{pushforward for differentials with supports}
H(f_*):H_Z^{i+r}(X,\Omega_X^{j+r})\rightarrow H^i(Y,\Omega_Y^j).
\end{equation}
First recall from  Lemma \ref{lem:dualitylog} the duality isomorphisms 
\begin{align}
\label{eq:properpush1}  
\theta_X
\colon 
H_Z^{i+r}(X,\Omega_X^{j+r}) 
& \xrightarrow{\cong} H_Z^{i+r-d_X}(X,D_{\underline{X}}(\Omega_X^{d_X-j-r}\otimes I_{\partial X}))\\ 
\nonumber
& =H_Z^{i-d_Y}(X,D_{\underline{X}}(\Omega_X^{d_Y-j}\otimes I_{\partial X})),
\end{align}
and 
\begin{equation}\label{eq:properpush1bis} 
\theta_Y\colon H^{i+r}(Y,\Omega_Y^{j+r})\xrightarrow{\cong} H^{i-d_Y}(Y,D_{\underline{Y}}(\Omega_Y^{d_Y-j}\otimes I_{\partial Y})).
\end{equation}
Here $\theta_X$ is obtained from \eqref{eq:dualitylog} by applying $R\Gamma_Z(X,-)$. 
Thus we are left to construct 
\[
H_Z^{i-d_Y}(X,D_{\underline{X}}(\Omega_X^{d_Y-j}\otimes I_{\partial X}))
\rightarrow 
H^{i-d_Y}(Y,D_{\underline{Y}}(\Omega_Y^{d_Y-j}\otimes I_{\partial Y})).
\]

If $f$ is proper, 
this is precisely the content of \eqref{eq:push1} after applying the cohomology with support functor 
(here, we use the surjectivity of the map from $Z$ to $Y$; otherwise, we would have to introduce support in the target). 
When $f$ is not proper, choose a good compactification $\ol{f}$ in the sense of Definition \ref{df:goodcomp}, which is possible by Remark \ref{rmk:goodcomp2}.
Then we have a natural isomorphism
\begin{equation}
\label{eq:properpush2} 
H_Z^{i-d_Y}(X,D_{\underline{X}}(\Omega_X^{d_Y-j}\otimes I_{\partial X})) 
\xrightarrow{\cong} 
H_Z^{i-d_Y}(\ol{X},D_{\underline{\ol{X}}}(\Omega_{\ol{X}}^{d_Y-j}\otimes I_{\partial {\ol{X}}})) 
\end{equation}
by excision since $Z\subset X\hookrightarrow\ol{X}$ is a strict closed immersion. 
Applying \eqref{eq:push1} to $\ol{f}$ we get
\begin{equation}
\label{eq:properpush3}
H_Z^{i-d_Y}(\ol{X},D_{\underline{\ol{X}}}(\Omega_{\ol{X}}^{d_Y-j}\otimes I_{\partial {\ol{X}}}))
\to  
H^{i-d_Y}(Y,D_{\underline{Y}}(\Omega_Y^{d_Y-j}\otimes I_{\partial Y})).
\end{equation}
Composing \eqref{eq:properpush1}, \eqref{eq:properpush2}, \eqref{eq:properpush3} and the inverse of \eqref{eq:properpush1bis}, we obtain \eqref{pushforward for differentials with supports}.
\end{const}

The following lemma is a routine check, see \cite[Definition 2.3.2, Proposition 2.3.3]{ChatzistamatiouRullingANT}.

\begin{lem} 
The map \eqref{pushforward for differentials with supports} is independent of the choice of compactification. 
Let $f\colon X\to Y$ and $g\colon Y\to V$ be two strict morphisms of integral fs log schemes in $SmlSm/k$ such that the restrictions $f\circ u$ and $g\circ f\circ u$ are proper and surjective.
If we set $l=\dim X-\dim V$, 
then 
$$
H((g\circ f)_*) = H(g_*)\circ H(f_*) \colon H^{i+l}_Z(X, \Omega^{j+l}_X)\to H^i(V, \Omega^j_V).
$$
\end{lem}

\subsection{The action of log correspondences}\label{ssec:coractomega}
Suppose $Z$ is an elementary log correspondence from $X$ to $Y$ in the sense of Definition \ref{A.5.2}, where $X,Y\in SmlSm/k$. 
Our aim in this subsection is to construct a pullback morphism
\[
Z^*
\colon
\Omega^j(Y) 
\rightarrow 
\Omega^j(X).
\]
If we extend by linearity to all finite log correspondences and verify compatibility with the composition, 
this equips the logarithmic differential $\Omega^j$ with the structure of a presheaf with log transfers.

\begin{lem}
\label{Exactness and strictness}
Let
\[
\begin{tikzcd}
Y'\arrow[d,"f'"']\arrow[r,"g'"]&Y\arrow[d,"f"]\\
X'\arrow[r,"g"]&X
\end{tikzcd}
\]
be a cartesian square of fs log schemes.
Suppose that $g'$ is exact.
If there exists a morphism $h:X\rightarrow T$ such that $hg$ is strict, then $g'$ is strict.
\end{lem}
\begin{proof}
Owing to \cite[Definition III.1.2.1]{Ogu}  there are naturally induced morphisms
\[
g_{log}^*\cM_X\rightarrow \cM_{X'},\;h_{log}^*\cM_T\rightarrow \cM_X,\;(hg)_{log}^*\cM_T\rightarrow \cM_{X'}.
\]
Since $hg$ is strict, the third morphism is an isomorphism.
This implies that the first morphism is an epimorphism, 
so that 
\[
\cM_{X'/X}:=\coker(g_{log}^*\cM_X\rightarrow \cM_{X'})=0.
\]
The question is Zariski local on $X'$ and $X$, so owing to \cite[Theorem III.1.2.7]{Ogu}, we may assume that $g$ has a neat chart $\theta:P\rightarrow P'$ (\cite[Definition II.2.4.4]{Ogu}).
This shows $\theta^{\rm gp}$ is an isomorphism since $\cM_{X'/X}=0$.
The question is also Zariski local on $Y$ so that we may assume $f$ admits an fs chart $P\rightarrow Q$.
\vspace{0.1in}

By the construction of fiber products in the category of fs log schemes, $Y'$ has an fs chart 
$$
Q':=P'\oplus_P Q, 
$$
which is an amalgamated sum in the category of fs monoids.
By \cite[Theorem I.1.3.4]{Ogu} we have an isomorphism 
\[
Q'^{\rm gp}\cong P'^{\rm gp}\oplus_{P^{\rm gp}}Q^{\rm gp}.
\]
Since $\theta^{\rm gp}:P^{\rm gp}\rightarrow P'^{\rm gp}$ is an isomorphism, we see that $Q^{\rm gp}\rightarrow Q'^{\rm gp}$ is an isomorphism.
This implies that $Q\rightarrow Q'$ is an isomorphism since $g'$ is exact.
Thus $g'$ is strict.
\end{proof}

\begin{const}
\label{constructionAction}
Let $Z$ be an elementary log correspondence from $X$ to $Y$ where $X,Y\in SmlSm/k$.
We will first construct 
$$
Z^*:\Omega^j(Y)\rightarrow \Omega^j(X)
$$ 
in the case when $X$ is a quasi-projective scheme and has an fs chart $P$.
By definition, there is a strict finite surjective morphism $Z^N\rightarrow X$ and a morphism $Z^N\rightarrow Y$.
By Theorem \ref{FKatoThm2}, 
there exists a log blow-up $V\rightarrow X\times Y$ such that the pullback $f':Z^N\times_{(X\times Y)}V\rightarrow V$ of $f\colon Z^N\to X\times Y$ is exact.
Using Lemma \ref{Exactness and strictness}, we deduce that $f'$ is strict.
\vspace{0.1in}

By Proposition \ref{Fan.12}, there is a log modification $X'\rightarrow X$ such that the pullback
\[(Z^N\times_{X\times Y}V)\times_X X'\rightarrow Z^N\times_X X'\]
is an isomorphism.
Next,  Corollary \ref{Invariance of Log differentials under log blow-ups} gives an isomorphism on sections
\[
\Omega^j(X)\cong \Omega^j(X').
\]
Let $Z'$ be the correspondence from $X'$ to $Y$ defined by $Z^N\times_X X'$.
If we can construct a morphism 
$$
Z'^*:\Omega^j(Y)\rightarrow \Omega^j(X'),
$$
then we obtain $Z^*:\Omega^j(Y)\rightarrow \Omega^j(X)$ via the composite
\[
\Omega^j(Y)\stackrel{Z'^*}
\rightarrow 
\Omega^j(X')
\stackrel{\sim}\rightarrow 
\Omega^j(X).
\]

Hence we may replace $(X,Y,Z,V)$ by $(X',Y',Z,V')$ where $V':=V\times_X X'$.
In this case, 
the projection 
$$
Z^N\times_{(X\times Y)}V\rightarrow Z^N
$$ 
is an isomorphism, 
and the projection $Z^N\times_{(X\times Y)}V\rightarrow V$ gives a morphism 
$$
\nu:Z^N\rightarrow V.
$$
Let $U$ be the maximal open subscheme of $V$ that is strict over $X$, which exists by Lemma \ref{lem::maximal strict open}. 
Since $Z^N$ is strict over $X$, $\nu$ factors through $\mu:Z^N\rightarrow U$, and there is a commutative diagram
\[
\begin{tikzcd}
Z^N\arrow[r,"\mu"]\arrow[rd]\arrow[rdd]&U\arrow[rd,"p_{U,Y}"]\arrow[d]\\
&X\times Y\arrow[d]\arrow[r]&Y\\
&X.
\end{tikzcd}
\]

Since $Z^N$ is quasi-projective, $\mu$ admits a factorization
\[
Z^N\rightarrow U'\rightarrow \P_U^m\rightarrow U
\]
into a closed immersion followed by an open immersion and the projection morphism for some $m\geq 0$.
We let $p_{U'}:U'\rightarrow U$ (resp.\ $q_{U'}$) be the composition $U'\rightarrow \P_U^m\rightarrow U$ (resp.\ $U'\stackrel{p_{U'}}\rightarrow U\rightarrow X\times Y\rightarrow X$).
Now using Constructions \ref{Construction of cycle class} and \ref{Construction of pushforward with support}, we have the homomorphisms
\[
\begin{split}
&H^i(Y,\Omega^j)\stackrel{p_{U,Y}^*}\rightarrow H^i(U,\Omega^j)
\stackrel{p_U^*}\rightarrow H^i(U',\Omega^j)\\
\stackrel{\cup cl(Z^N,U')}\longrightarrow &H_Z^{i+m+d_Y}(U',\Omega^{j+m+d_Y})
\stackrel{q_{U*}}\rightarrow H^i(X,\Omega^j).
\end{split}
\]
When $i=0$ we deduce the desired correspondence action $Z^*:\Omega^j(Y)\rightarrow \Omega^j(X)$.
\vspace{0.1in}

For a general $X$, we choose a Zariski cover $\{X_a\}_{a\in I}$ of $X$ such that each $X_a$ is quasi-projective and has an fs chart.
Set $X_{ab}:=X_a\times_X X_b$, 
and for each $a$ and $b$ we choose a Zariski cover $\{X_{abc}\}_{c\in J_{ab}}$ of $X_{ab}$ such that each $X_{abc}$ is quasi-projective and has an fs chart.
Since $\Omega^j_X$ is a Zariski sheaf, there is an exact sequence
\[
0\rightarrow \Omega^j(X)\rightarrow \bigoplus_{a\in I}\Omega^j(X_a)\rightarrow \bigoplus_{a,b\in I,c\in J_{ab}} \Omega^j(X_{abc}).
\]
This induces an exact sequence
\[
\begin{split}
0\rightarrow &\hom(\Omega^j(Y),\Omega^j(X))\rightarrow \hom(\Omega^j(Y),\oplus_{a\in I}\Omega^j(X_a))\\
\rightarrow &\hom(\Omega^j(Y),\oplus_{a,b\in I,c\in J_{ab}}\Omega^j(X_{abc})).
\end{split}
\]

Letting $Z_a\in \lCor(X_a,Y)$ and $Z_{abc}\in \lCor(X_{abc},Y)$ denote the restrictions of $Z\in \lCor(X,Y)$, 
we have constructed 
$$
Z_a^*\in \hom(\Omega^j(Y),\Omega^j(X_a)), 
\;
Z_{abc}^*\in \hom(\Omega^j(Y),\Omega^j(X_{abc})).
$$
By Lemma \ref{Correspondence, open immersion, and differentials} below, we obtain the desired element 
\begin{equation}
\label{eq:corconst}
Z^*\in \hom(\Omega^j(Y),\Omega^j(X))
\end{equation}
using the above exact sequence. 
Our construction extends to the log setting the correspondence action given by R\"ulling in \cite[\S A.4]{KSY}.
\end{const}

\begin{lem}
\label{Correspondence, open immersion, and differentials}
Let $Z\in lCor(X,Y)$ be a log correspondence where $X,Y\in SmlSm/k$, let $j:U\rightarrow X$ be an open immersion, and let $Z_U\in lCor(U,Y)$ be the restriction of $Z$.
Consider the naturally induced map 
$$
j^*:\Omega^j(X)\rightarrow \Omega^j(X).
$$
If $X$ is quasi-projective and has an fs chart $P$, 
then we have 
\[
j^*\circ Z^*=Z_U^*.
\]
\end{lem}
\begin{proof}
This is immediate from the construction.
\end{proof}

\begin{thm}
\label{thm:Omega_log_transfer}
The presheaf $\Omega^j$ on $lSm/k$ admits a log transfer structure.
\end{thm}
\begin{proof}
We will first show that $\Omega^j$ on $SmlSm/k$ admits a log transfer structure.
Let $V\in lCor(X,Y)$ and $W\in lCor(Y,Z)$ be log correspondences.
From Construction \ref{constructionAction} we have
\[
V^*\in \hom(\Omega^j(Y),\Omega^j(X)),\;
W^*\in \hom(\Omega^j(Z),\Omega^j(Y)),
\]
and
\[
(W\circ V)^*\in \hom(\Omega^j(Z),\Omega^j(X)).
\]
It remains to check that
\[
(W\circ V)^* = V^*\circ W^*.
\]

Due to Construction \ref{constructionAction}, we may assume that $X$, $Y$, and $Z$ are quasi-projective schemes equipped with fs charts. 
Consider the diagram
\[
\begin{tikzcd}
\Omega^j(Z) \arrow[r, "W^*"] \arrow[d, hook] & \Omega^j(Y) \arrow[r, "V^*"]\arrow[d, hook] & \Omega^j(X)\arrow[d, hook] \\
\Omega^j(Z- \partial Z) \arrow[r, "(W^\circ)^*"]  & \Omega^j(Y- \partial Y)  \arrow[r, "(V^\circ)^*"]  & \Omega^j(X- \partial X).
\end{tikzcd}
\]
The correspondences $V^\circ$ and $W^\circ$ are as in Lemma \ref{A.5.10}.  
\vspace{0.1in}

The vertical morphisms are injective by definition of the differential forms with log poles as subsheaf of the pushforward of the sheaf of differential forms 
from the open subset where the log structure is trivial. 
By Lemma \ref{Correspondence, open immersion, and differentials} and by the construction \cite[\S A.4]{KSY}, 
the pullbacks $(V^\circ)^*$ and $(W^\circ)^*$ render the two squares commutative.  
By \cite[Theorem A.4.1]{KSY}, 
the lower horizontal composition $(W^\circ \circ V^\circ)^*$ agrees with $(V^\circ)^*\circ (W^\circ)^*$. 
Since the vertical maps are injective and $(W\circ V)^\circ = W^\circ \circ V^\circ$ by Lemma \ref{lem:composition}, 
we conclude that  $(W\circ V)^*= V^*\circ W^*$ by a diagram chase.
This completes the proof that $\Omega^j$ admits a log transfer structure on $SmlSm/k$.
\vspace{0.1in}

Thanks to Lemma \ref{A.5.59}, the presheaf on $lSm/k$ given by
\[
X\mapsto \colimit_{Y\in X_{div}^{Sm}}\Omega^j(Y)
\]
admits a log transfer structure.
To conclude, observe that this presheaf is isomorphic to $\Omega^j$ since $\Omega^j$ is a dividing sheaf.
\end{proof}

\begin{rmk} 
Let $X$ be a smooth $k$-scheme. 
By \cite[Theorem 3.6]{HuberJorder} the sections of the $h$-sheafification $a_h (\Omega^j)$ of $\Omega^j$ on $X$ agree with $\Omega^j(X)$. 
This implies that the sheaves $\Omega^j$ on $Sm/k$ admit a unique transfer structure. 
Thus the transfer structure defined by R\"ulling and extended in this section to log schemes agrees with the transfer structure defined in \cite{lecomtewach}, 
at least with rational coefficients and in characteristic zero.
\end{rmk}

\begin{prop}
\label{prop::commutativity_differential}
For every integer $j\geq 0$, the differential
\[
d:\Omega^j\rightarrow \Omega^{j+1}
\]
is a morphism of presheaves with log transfers.
\end{prop}
\begin{proof}
We need to show that for every finite log correspondence $V\in \lCor(X,Y)$, with $X,Y\in lSm/k$, the back square in the diagram
\[
\begin{tikzcd}[row sep=tiny, column sep=tiny]
\Omega^j(Y)\arrow[dd]\arrow[rr]\arrow[rd]&
&
\Omega^{j+1}(Y)\arrow[dd]\arrow[rd]
\\
&
\Omega^{j}(Y-\partial Y)\arrow[rr,crossing over]&
&
\Omega^{j+1}(Y-\partial Y)\arrow[dd]
\\
\Omega^j(X)\arrow[rr]\arrow[rd]&
&
\Omega^{j+1}(X)\arrow[rd]
\\
&
\Omega^j(X-\partial X)\arrow[rr]\arrow[uu,crossing over,leftarrow]&
&
\Omega^{j+1}(X-\partial X).
\end{tikzcd}
\]
commutes.
The top and bottom squares commute since $d:\Omega^j\rightarrow \Omega^{j+1}$ is a morphism of presheaves.
The left and right squares commute since $\Omega^j$ is a presheaf with log transfers by Theorem \ref{thm:Omega_log_transfer}.
In the proof of Theorem \ref{thm:Omega_log_transfer}, we have observed that the homomorphisms
\[
\Omega^j(X)\rightarrow \Omega^j(X-\partial X)
\text{ and }
\Omega^j(Y)\rightarrow \Omega^j(Y-\partial Y)
\]
are injective.
Thus we are reduced to showing that the front square commutes, i.e., we may assume that $X,Y\in Sm/k$.
This case is checked in \cite[Theorem B.2.1]{KSY}.
\end{proof}
\begin{rmk}
In Sections \ref{ssec:Grothduality}-\ref{ssec:coractomega} we have made a detour through the work of Chatzistamatiou-R\"ulling \cite{ChatzistamatiouRullingANT} (adapted to the log setting) in order to define an action of the log correspondences on the sheaves of differentials. In fact (as in \cite{ChatzistamatiouRullingANT}), the formalism of Grothendieck duality allows one to construct an action of log correspondences on the cohomology groups of $\Omega^j$, as implied by the argument used in Construction \ref{constructionAction}.

The referee and Kay R\"ulling independently pointed out to us that if one is simply willing to establish an action of correspondences on $\Omega^j$ (and not on the cohomology), a shortcut is possible.

The aim is to define for every $X,Y\in SmlSm/k$, and every elementary log correspondence $Z$ from $X$ to $Y$, a dashed arrow which makes the following diagram commutative
\[ \begin{tikzcd}
\Omega^j(Y) \arrow[r, hook] \arrow[d, dashrightarrow] & \Omega^j(Y-\partial Y) \arrow[d, "(Z^\circ)^*"] \\
\Omega^j(X) \arrow[r, hook] & \Omega^j(X-\partial X)
\end{tikzcd}\]
where  $(Z^\circ)^*$ is the correspondence action on differentials defined in \cite{ChatzistamatiouRullingANT}. Since the horizontal maps are injective (both $X$ and $Y$ are in $SmlSm/k$), it is then enough to show that for every section $a\in \Omega^j(Y)$, the pullback $(Z^\circ)^*(a) \in H^0(\ul{X}, j_*\Omega^j_{X-\partial X})$ belongs to $H^0(\ul{X}, \Omega^j_X)$, where $j\colon \ul{X}-\partial X \hookrightarrow \ul{X}$ is the open immersion. Since this property can be checked around each generic point of $\partial X$ (the differentials are locally free), one can reduce to the case where $X$ is local and 1-dimensional. The elementary log correspondence $Z$ induces by restriction a diagram
\[\begin{tikzcd}
Z^N \arrow[d, "p"] \arrow[r, "q"] & Y\\ 
X
\end{tikzcd}\]
where $\underline{Z}^N$ is the normalization of $\underline{Z}$ and $p$ is a strict, finite and surjective morphism. It is then enough to show that the pushforward $p_*^\circ \colon \Omega^j(Z^N-\partial Z^N) \to \Omega^j(X-\partial X)$ (induced by the transpose of the graph of $p^\circ$) restricts to a morphism between the subsheaves of log forms. This can be checked locally in coordinates.

We have decided to keep the original argument since it gives a bit more, as explained above. Nevertheless, we will come back to this approach in future work, in which we will construct the action of log correspondences on the sheaves of Hodge-Witt differentials. As a byproduct, this implies  that  logarithmic crystalline cohomology is representable in $\ldmeff$.
\end{rmk}
\subsection{Hodge cohomology and cyclic homology}
\label{ssec-Hodgecyclic} 
After having established the action of log correspondences on the sheaves of logarithmic differentials, 
we are ready to establish the representability of Hodge cohomology and cyclic homology in $\ldmeff$.


\begin{thm}
\label{thmHodge} 
For $t=dNis$, $d\acute{e}t$, and $l\acute{e}t$, the sheaves $\Omega^j$ are strictly $\boxx$-invariant $t$-sheaves with log transfers. 
Thus for every $X\in lSm/k$ and $i\in \Z$, there is a natural isomorphism
\begin{align*} 
\hom_{\ldmeff}(M(X), \Omega^j[i]) 
\cong
H^i_{Zar}(X, \Omega^j).
\end{align*}
\end{thm}
\begin{proof}
Lemma \ref{Omega_dividing_sheaf}, 
Proposition \ref{prop::hodgelet}, 
Corollary \ref{Pn-invariance of Log differentials}, 
and Theorem \ref{thm:Omega_log_transfer} imply that $\Omega^j$ is a strictly $\boxx$-invariant $t$-sheaves with log transfers on $SmlSm/k$.
For the isomorphisms
\begin{align*} 
\hom_{\ldmeff}(M(X), \Omega^j[i]) 
&\cong 
\hom_{\ldmeffet}(M(X), \Omega^j[i]) \\
&\cong
\hom_{\ldmefflet}(M(X),\Omega^j[i]) \\
&\cong
H^i_{Zar}(X, \Omega^j), 
\end{align*}
see \eqref{SmlSm.3.2} and \eqref{SmlSm.3.3}.
\end{proof}

\begin{cor}
With the above notations, there is a natural isomorphism
\[
\hom_{\ldmeff}(M(X), \Omega^\bullet[i]) 
\cong
\bH^i_{Zar}(X, \Omega^\bullet).
\]
Hence the de Rham cohomology, which is not $\A^1$-invariant in positive characteristic, is representable in $\ldmeff$.
\end{cor}
\begin{proof}
Suppose $t$ is $dNis$ or $Zar$.
The Hodge to de Rham spectral sequence
\[
E_1^{pq}=H^q_{t}(X,\Omega^p)
\Rightarrow
\bH^{p+q}_{t}(X,\Omega^\bullet)
\]
is strongly convergent, so Theorem \ref{thmHodge} implies that $\bH^i_{Zar}(X, \Omega^\bullet)$ is $\boxx$-invariant for every integer $i\geq 0$.
Furthermore, Proposition \ref{prop::hodgelet} implies $\bH^i_{dNis}(X,\Omega^\bullet) \simeq \bH^i_{Zar}(X,\Omega^\bullet)$.
Hence we obtain the desired isomorphism.
\end{proof}

\begin{exm}
\label{Nontriviality}
Suppose $k$ is a field of positive characteristic $p>0$ and let $\Lambda$ be an $\mathbb{F}_p$-algebra.
Then the Artin-Schreier sequence
\[
0\rightarrow \Z/p \rightarrow \mathbb{G}_a\rightarrow \mathbb{G}_a\rightarrow 0
\]
is an exact sequence of \'etale sheaves on $SmlSm/k$, where by definition we set $\Gamma(X, \mathbb{G}_a) = \Gamma(\underline{X}, \mathbb{G}_a)$.
By Theorem \ref{thmHodge} the sheaf $\mathbb{G}_a$ on $SmlSm/k$ is a strictly $\boxx$-invariant dividing \'etale sheaf with log transfers.
Consequently, $\Z/p$ is a strictly $\boxx$-invariant dividing \'etale sheaf with log transfers.
Applying \eqref{SmlSm.3.3} we see that for every $X\in SmlSm/k$, there is an isomorphism
\[
\hom_{\ldmeffetZ}(M(X), \Z/p) 
\cong
H^0_{Zar}(\ul{X}, \Z/p).
\]
In particular, if $X$ is nonempty, this is a nontrivial group.
It follows that also $\Z/p$ is nontrivial in $\ldmeffetZ$.
Similarly, $\Z/p$ is nontrivial in $\ldmeffletZ$.
\end{exm}

Assume now that $k$ has characteristic zero and let $\Lambda=\mathbb{Q}$. 
Let $HC_n(X)$ denote the cyclic homology groups of a smooth $k$-scheme $X$. 
The Zariski descent method of \cite{weibelcyclic} and the Hochschild-Kostant-Rosenberg Theorem \cite[Theorem 3.4.12]{Loday} provide a natural isomorphism
\begin{equation}
\label{equation:eqcyclic}
HC_n(X) 
\cong 
\bigoplus_{p\in \Z}\bH_{Zar}^{2p-n}(X, \Omega^{\leq 2p}).
\end{equation}
Proposition \ref{prop::commutativity_differential} shows that the differentials in the complex $\Omega^\bullet$ on $SmlSm/k$ are compatible with the log transfer structure, 
so that $\Omega^{\leq 2p}$ is a complex of Nisnevich sheaves with (log) transfers. 
Theorem \ref{thmHodge} and \eqref{equation:eqcyclic} imply representability of cyclic homology in ${\mathbf{logDM}^{\rm eff}(k, \mathbb{Q})}$.

\begin{thm}
\label{thmHC} 
Suppose $k$ has characteristic zero and let $\Lambda=\mathbb{Q}$.
For every ${X}\in Sm/k$ and $n\geq 0$, there is a natural isomorphism
\begin{equation}
\label{equation:cyclichomology}
\hom_{{\mathbf{logDM}^{\rm eff}(k, \mathbb{Q})}}(M({X}), \bigoplus_{p\in \Z} \Omega^{\leq 2p}[2p-n])
\cong
HC_n({X}).
\end{equation}
\end{thm}
Note that for $X\in lSm/k$, 
the left hand side of \eqref{equation:cyclichomology} can be taken as definition of cyclic homology for 
fs log smooth log schemes over a field of characteristic zero.
In a future work, we plan to compare the above with log Hochschild and log cyclic homology.
\newpage

\appendix

\section{Logarithmic geometry}

In this section, 
we provide a minimal reminder of the basic definitions regarding logarithmic geometry and state results 
that are used in the main body of our work.
One can intuitively view the category of log schemes as an extension of the category of schemes given by 
essentially allowing ``schemes with boundary'' meaning pairs $(X,\partial X)$,
where $X$ is a scheme, and $\partial X\subset X$ is a normal crossing divisor, 
In many ways, the log geometry of $(X, \partial X)$ models the geometry on the complement $X\setminus \partial X$: 
for example, 
they have the same de Rham cohomology in characteristic 0 and \'etale sites.
Many classical invariants and properties of schemes extend naturally to log schemes. 
There are notions of smooth and \'etale maps, pullbacks, forms and differentials, and coherent sheaves. 
The Kato-Nakayama or Betti analytification functor provides an important bridge between log schemes and topological spaces;
in our motivic setting, 
this is clearly expressed in the construction of Thom spaces. 
\vspace{0.1in}

The structure sheaf $\mathcal{O}_X$ of $X$ is a sheaf of rings in the Zariski topology.
Its subsheaf of invertible elements is a Zariski sheaf of commutative monoids under multiplication $\mathcal{O}_X^{\times}$.
Moreover, 
it is related to the sheaf of differentials via the logarithmic differential 
\begin{equation}
\label{equation:differential}
\text{dlog}
\colon
\mathcal{O}_{X}^{\times}\to \Omega_{X}
\end{equation}
defined locally by 
\[
f\mapsto \frac{df}{f}.
\] 
A foundational insight of log geometry is that for a scheme with boundary, 
\eqref{equation:differential} extends to a morphism
\begin{equation}
\label{equation:logdifferential}
\text{dlog}
\colon
\mathcal{O}_{X}^{\times}(\partial X)\to \Omega_{X}(\partial X)
\end{equation}
with values in the sheaf of logarithmic differentials with respect to the normal crossing divisor $\partial X$. 
Here, 
$\mathcal{O}_{X}^{\times}(\partial X)$ denotes the monoid of functions which are invertible outside of $\partial X$ (and may have arbitrary order of vanishing at $\partial X$).  
Much of the geometry of the pair $(X, \partial X)$ is captured by the enhanced logarithmic differential map in \eqref{equation:logdifferential}. 
In log geometry, one generalizes the above to the ``logarithmic boundary data'' of a scheme together with a submonoid sheaf. 

A textbook reference for logarithmic geometry is \cite{Ogu}. 
Some of the first accounts of the subjects include \cite{MR1463703}, and \cite{MR1922832}.

\subsection{Monoids}
\begin{df}
\label{monoid}
A monoid is a commutative semigroup equipped with a unit.  
A homomorphism of monoids is required to preserve the unit element. 
Let $\Mon$ denote the category of monoids.
\end{df}

Given a monoid $P$, we can associate the group completion \index{monoid}
\begin{equation*}
P^\gp
:=
\{(a,b) |(a,b)\sim(c,d) \ \mbox{if} \ \exists \ p \in P \ \mbox{such that} \ p+a+d=p+b+c\}.
\end{equation*}
Note that the canonical morphism $P\to P^\gp$ is universal among the homomorphisms of monoids $P\to G$ such that the target $G$ is an abelian group.
\vspace{0.1in}

Let $P^\ast$ \index[notation]{P@ $P^*$}denote the set of all $p\in P$ such that $p+q=0$ for some $q\in P$. It is called the \emph{unit group} of $P$.

\begin{df}
\label{df:integral}
For a monoid $P$ the image of the homomorphism $P\rightarrow P^\gp$ sending $a$ to $(a,0)$ is denoted by $P^{\rm int}$. \index{monoid!integral}
A monoid $P$ is said to be \emph{integral} if $P\rightarrow P^{\rm int}$ is an isomorphism and finitely generated if there exists a surjective homomorphism $\N^{k}\to P$ for some $k$.
It is called \emph{fine} if it is integral and finitely generated.  \index{monoid!fine}
For an integral monoid $P$ we denote by $P^{\rm sat}$ the submonoid of $P^\gp$ such that 
whenever $p\in P^\gp$ and $n$ is a positive integer such that $np \in P$ then $p \in P^{\rm sat}$.
An integral monoid $P$ is said to be \emph{saturated} if $P\rightarrow P^{\rm sat}$ is an isomorphism.\index{monoid!saturated}
A monoid $P$ is \emph{sharp} if its unit group $P^*$ is trivial.\index{monoid!sharp}
\end{df}
We abbreviate the combined condition ``fine and saturated"  to  ``fs."

\begin{df}
For homomorphisms of monoids (resp.\ integral monoids, resp.\ saturated monoids) $P\rightarrow Q$ and $P\rightarrow P'$, the amalgamated sum $Q\oplus_P^{\rm mon} P'$ (resp.\ $Q\oplus_P^{\rm int} P'$, resp.\ $Q\oplus_P P'$) in the category of monoids (resp.\ integral monoids, resp.\ saturated monoids) is the colimit of the diagram $Q\leftarrow P\rightarrow P'$.
\end{df}

In general, the monoids $Q\oplus_P^{\rm mon} P'$, $Q\oplus_P^{\rm int} P'$, and $Q\oplus_P P'$ are not isomorphic to each other even if $P$, $P'$, and $Q$ are saturated. 
The monoid $Q\oplus_P^{\rm int} P'$ is isomorphic to the submonoid of the abelian group $Q^\gp\oplus_{P^\gp}P'^\gp$ generated by the images of $Q$ and $P'$.
There is an isomorphism
\[
Q\oplus_P P'\cong (Q\oplus_P^{\rm int}P')^{\rm sat}.
\]
For an explicit description of $Q\oplus_P^{\rm mon} P'$ we refer to \cite[Proposition I.1.1.5]{Ogu}.

\begin{df}
For a monoid $P$, we set
\[
\overline{P}:=P\oplus_{P^\ast}^{\rm mon}0.
\]
Note that $\overline{P}$ is always sharp.
\end{df}

\begin{df}
\label{KummerDef} \index{homomorphism of monoids!Kummer}
A homomorphism of integral monoids $\theta:P\rightarrow Q$ is called Kummer if it is injective and $\theta\otimes \Q:P\otimes \Q\rightarrow Q\otimes \Q$ is surjective.
\end{df}

\begin{df}
A \emph{face} $F$ of a monoid $P$ is a submonoid of $P$ such that for every $p,q\in P$ \index{monoid!face}
\[
p+q\in F\Rightarrow p,q\in F.
\]
An \emph{ideal} $I$ of a monoid $P$ is a subset of $P$ satisfying the condition \index{ideal (of a monoid)}
\[
p\in I,q\in P\Rightarrow p+q\in I.
\]
For an element $p$ of $P$, let $\langle p\rangle$ denote the smallest face containing $p$, and let $(p)$ denote the smallest ideal containing $p$.
\end{df}

\begin{df}
For a face $F$ of a monoid $P$ we set 
\[
P_F
:=
\{(a,b)\in P\times F|(a,b)\sim(c,d) \ \mbox{if} \ \exists \ p \in F \ \mbox{such that} \ p+a+d=p+b+c\}.
\]
With this definition we have $P_P=P^\gp$.
We also set
\[
P/F:=P\oplus_F^{\rm mon} 0.
\]
There is an isomorphism $\overline{P_F}\cong P/F$.
Finally, for an element $f\in P$, we set
\[
P_f:=P_{\langle f\rangle}.
\]
\end{df}

\subsection{Logarithmic structures} For a scheme $X$, let $\cO_X$ be the structure sheaf. When we refer to it as a sheaf of monoids, we always consider its multiplicative structure. 
\begin{df}
\label{log-str}
Let $X$ be a scheme. 
A pre-logarithmic structure on $X$ is a sheaf of monoids $\cM_{X}$ on the small Zariski site of $X$ combined with a morphism of sheaves of monoids
\[
\alpha\colon \cM_{X} \longrightarrow \mathcal{O}_{X}
\]
called the structure morphism.
A pre-logarithmic structure is called a logarithmic structure if $\alpha$ induces an isomorphism
\[
\alpha^{-1}(\mathcal{O}_{X}^{*})\cong \mathcal{O}^{*}_{X}.
\] 

The pair $(X,\cM_{X})$ is called a log scheme.
For notational convenience, whenever we abbreviate $(X,\cM_X)$ to $Y$ we set $\underline{Y}:=X$ and $\cM_Y:=\cM_X$.
The scheme $\underline{Y}$ is called the underlying scheme of $Y$. 
\end{df}

Note that the last condition in Definition \ref{log-str} gives us a canonical embedding $\mathcal{O}_{X}^{*}\hookrightarrow \cM_{X}$ as a subsheaf of groups.
The quotient sheaf of monoids 
\[
\overline{\cM}_{X}
\colon=
\cM_{X}/\mathcal{O}_{X}^{*}
\]
is called the characteristic monoid of the log structure $\cM_{X}$.

\begin{rmk}
The morphism $\alpha$ is sometimes denoted $\exp$ for ``exponential"  and an inverse $\mathcal{O}_X^\ast \rightarrow \cM_{X}$ is denoted $\log$ for ``logarithmic.''
Any scheme $X$ is a log scheme with $\cM_{X}=\mathcal{O}_{X}^{\ast}$ and $\alpha$ the inclusion. 
This is the trivial log structure on $X$.
\end{rmk}

\begin{df}[{\cite[Definition III.1.1.1]{Ogu}}]
\label{A.9.15}
Let ${\bf Log}_X$ denote the category of log structures on a scheme $X$. \index[notation]{Log @ ${\bf Log}_X$}
\end{df}

\begin{df}[{\cite[Definition III.1.1.5]{Ogu}}]
\label{A.9.13}
Let $f:X\rightarrow Y$ be a morphism of schemes. 
For a log structure $\cM_{X}\rightarrow \cO_X$ on $X$ we set
\[
f_*^{log}\cM_{X}
:=
f_*\cM_{X}\times_{f_*\cO_X}\cO_Y.
\]
For a log structure $\cM_{Y}\rightarrow \cO_Y$ on $Y$, 
let $f_{log}^*\cM_{Y}$ be the log structure induced by the prelog structure
\[
f^{-1}(\cM_{Y})\rightarrow f^{-1}(\cO_Y)\rightarrow \cO_X.
\]

With these definitions, there exists an adjoint pair of functors
\[f_{log}^*
:
{\bf Log}_X
\rightleftarrows 
{\bf Log}_Y:f_*^{log}.
\]
\end{df}

\begin{df}
A morphism of log schemes $(X,\cM_{X})\to (Y,\cM_{Y})$ is a pair $(f,\phi)$, 
where $f\colon X\to Y$ is a morphism of schemes and $\phi\colon f^{-1}\cM_{Y}\to \cM_{X}$ is a homomorphism of sheaves of monoids such that the diagram
\begin{equation}
\label{diagram:normalcrossing}
\vcenter{\xymatrix{
f^{-1}\cM_{Y} \ar[r]^-{\phi} \ar[d] & \cM_{X} \ar[d]\\
f^{-1}\mathcal{O}_{Y}  \ar[r]^-{f^{\ast}} & \mathcal{O}_{X} 
}}
\end{equation}
commutes.
Note that giving a morphism $\varphi:f^{-1}\cM_{Y}\rightarrow \cM_{X}$ is equivalent to giving a morphism $f_{log}^*\cM_{Y}\rightarrow \cM_{X}$.
\vspace{0.1in}

We write $\LogSch$ for the category of log schemes. \index[notation]{LogSch @ $\LogSch$}
\end{df}

\begin{rmk}
A log morphism between schemes equipped with their trivial log structures is just a morphism of schemes.
\end{rmk}

\begin{exm}
\label{ex:logpoints}
Let $X$ be a smooth scheme with an effective divisor $D\subset X$. 
Then we have a standard logarithmic structure on $X$ associated with the pair $(X, D)$ or the divisorial log structure induced by $D$, 
where
\[
\cM_{X} := \{f \in \mathcal{O}_{X} \ | \ f|_{X\setminus D} \in \mathcal{O}^*_{X}\}
\]
with the structure morphism $\cM_{X} \to \mathcal{O}_{X}$ given by the canonical inclusion. 
Note that any section of $\mathcal{O}^*_{X}$ is already in $\cM_{X}$.
In this case, the log structure encodes the intuition that $D$ is a sort of ``boundary'' in $X$.
With reference to Section \ref{subsection:cols},
picking local equations for the branches of $D$ passing through a point $x\in X$ gives local charts for the log structure with monoid $\N^{n}$, 
where $n$ is the number of branches. 
\end{exm}

\begin{exm}
\label{ex:affine-toric-log}
To a monoid $P$ one associates the monoid ring $\Z[P]$ of $P$ over the integers and the affine toric variety $\underline{\A_{P}} := \text{Spec}(\Z[P])$. 
The standard logarithmic structure $\cM_{\A_P}$ on $\underline{\A_P}$ is induced from the pre-logarithmic structure \index[notation]{SpecZP @ $\text{Spec}(\Z[P])$}
\[
P \to \Z[P] = \Gamma(\mathcal{O}_{\A_{P}})
\]
defined by the obvious inclusion. 
This is called the canonical log structure on $\underline{\A_{P}}$, and we set $\A_P:=(\underline{\A_P},\cM_{\A_P})$. \index[notation]{AP @ $\A_P, \A_{\mathbb{N}}$}
For example, 
we can consider $\A_{\N}$ as $\Spec{\Z[\N]}$ with the log structure associated to the origin.
More generally,
we may replace $\Z$ with a ring $R$.
In this way obtain a contravariant functor from $\Mon$ to $\LogSch$.
\end{exm}

Of particular interest is the case when $P$ is a toric monoid.

\begin{exm}
Any open subscheme of a log scheme admits a restricted log structure.
\end{exm}

\begin{exm}
\label{ex:logpoints2}
Let $Q$ be a sharp monoid. 
Define the homomorphism $\alpha\colon Q\to k$ by $\alpha(0)=1$ and  $\alpha(q)=0$ if $q\neq 0$.
The associated log scheme $(\Spec{k},Q)$ is a called a log point in general and a standard log point when $Q=\N$.
This log structure is equivalent to that of $\alpha \colon Q\oplus k^{\ast}\to k$ given by $\alpha(0,x)=x$ and $\alpha(q,x)=0$ if $q\neq 0$.
When $Q=\{0\}$ we obtain the trivial log structure.
If $k$ is algebraically closed, every log structure on $\Spec k$ is equivalent to a log point.
\end{exm}

A log structure isolates the ``singular'' information carried by a pre-log structure, 
i.e., 
it provides logarithmic behavior for functions that are not already invertible.
Any pre-log structure can be turned into a log structure.
The following provides a left adjoint of the inclusion functor from pre-log structures into log structures.

\begin{df}
Let $\alpha\colon \cM_{X} \rightarrow \mathcal{O}_{X}$ be a pre-logarithmic structure. 
Then the associated logarithmic structure $M^{a}_{X}$ is the pushout of the diagram
\[
\xymatrix{
\mathcal{O}_{X}^{*} & \alpha^{-1}(\mathcal{O}_{X}^{*}) \ar[l] \ar[r] & \cM_{X} 
}
\]
in the category of sheaves of monoids on the small Zariski site of $X$, endowed with 
\[
M^{a}_{X} \rightarrow \mathcal{O}_{X}  \qquad (a,b)\mapsto \alpha(a)b \qquad\qquad (a \in M, b \in \mathcal{O}_{X}^{*}).
\]
\end{df}

\begin{exm}
Any morphism of schemes $f\colon X\to \Spec {\Z[P]}$ induces a log structure on $X$ as follows: 
$f$ gives a morphism of monoids $P\to \mathcal{O}_X(X)$ that induces $\widetilde{\alpha}\colon\underline{P}\to \mathcal{O}_X$, 
where $\underline{P}$ is the constant sheaf. 
It is typically not true that $\widetilde{\alpha}$ induces an isomorphism between $\widetilde{\alpha}^{-1}\mathcal{O}^\times_X$ and $\mathcal{O}^\times_X$, 
but the procedure above fixes the behaviour of the units, and this produces a log structure $\alpha\colon M^{a}_{X}\to \mathcal{O}_X$. 
The quotient $M^{a}_{X}/\mathcal{O}_X^\times$ is obtained from $\underline{P}$ by locally killing the sections of $\underline{P}$ that become invertible in $\mathcal{O}_X$.
In particular, 
all the stalks of $M^{a}_{X}/\mathcal{O}_X^\times$ are quotients of the monoid $P$.
\end{exm}

\subsection{Monoschemes and fans}
In this section, we review the definition of fans and the relation between monoschemes and fans.

\begin{df}\index{monoidal space}
A \emph{monoidal space} is a topological space $X$ equipped with a sheaf of monoids $\cM_{X}$.
\vspace{0.1in}

A monoidal space $X$ is \emph{quasi-compact} (resp.\ \emph{quasi-separated}, resp.\ \emph{connected}) if the underlying topological space is quasi-compact 
(resp.\ quasi-separated, resp.\ connected).
\end{df}

\begin{df} \index{monoidal space!of a monoid $P$}
Let $P$ be a monoid.
Its monoidal space $\Spec{P}$ is defined as follows.
The underlying set of $\Spec{P}$ consists of the faces of $P$.
For an element $f\in P$, we set
\[
D(f):=\{F\in \Spec{P}:f\in F\}.
\]
The \emph{Zariski} topology on $\Spec{P}$ is the topology generated by
\[
\mathscr{B}:=\{D(f):f\in P\},
\]
i.e., $\mathscr{B}$ is a base for $\Spec{P}$.
The assignment 
$$
D(f)\mapsto P_f
$$ 
defines a presheaf on $\mathscr{B}$, and its sheafification defines a sheaf $\cM_{\Spec{P}}$ on $\Spec{P}$.
The pair 
$$
\Spec{P}:=(\Spec{P},\cM_{\Spec{P}})
$$ 
is a monoidal space.
\end{df}

\begin{rmk}
The meaning of the notation $\Spec{\Z}$, 
either as the affine scheme associated with the ring $\Z$ or as the monoidal space associated with the monoid $\Z$, 
should always be clear from the context.
\end{rmk}

\begin{df}
A monoidal space $X$ is an fs \emph{monoscheme} \index{monoscheme} if there is an open cover $\{X_i\}$ of $X$ such that for some fs monoid $P_i$ there is an isomorphism
\[
X_i\cong \Spec {P_i}.
\]
If $\cM_{X,x}^{\rm gp}$ is torsion free for every point $x\in X$, we say that $X$ is \emph{toric}.
If $X\cong \Spec{P}$ for some fs monoid $P$, we say that $X$ is \emph{affine}.
\end{df}

The category of fs monoschemes has fiber products, see \cite[Proposition II.1.3.5]{Ogu}.

\begin{df}[{\cite[Definition II.1.6.1]{Ogu}}]
A quasi-compact and quasi-separated fs monoscheme $X$ is \emph{separated} (resp.\ \emph{proper}) if the naturally induced homomorphism
\[
\hom(\Spec \N, X)\rightarrow \hom(\Spec \Z, X)
\]
is injective (resp.\ bijective).
\end{df}

\begin{df}
\label{df:toricvarieties}
We follow standard notation for toric varieties and set 
\[
N:=\Z^{d}, 
\;
M:=\hom_{\Z}(N,\Z), 
\;
N_{\R}:=N\otimes_{\Z}\R, 
\;
M_{\R}:=M\otimes_{\Z}\R.
\]
We also often use the notation $N^\vee:=M$.
If $\sigma\subset N_{\R}$ is a strictly convex rational polyhedral cone, 
we denote its dual cone by 
\[
\sigma^{\vee}
:=
\{
m\in M_{\R} \vert \langle m,n \rangle = 0, \ \forall \ n\in\sigma
\}.
\]
Gordan's lemma tells us that $M\cap \sigma^{\vee}$ is a finitely generated monoid.
\vspace{0.1in}

A \emph{fan}\index{fan} $\Sigma$ in $N$ is a finite set of cones in $N$ satisfying the following conditions:
\begin{enumerate}
\item[(i)] If $\sigma\in \Sigma$, then every face of $\sigma$ is in $\Sigma$.
\item[(ii)] If $\sigma,\tau\in \Sigma$, then $\sigma\cap \tau\in \Sigma$.
\end{enumerate}
\vspace{0.1in}

Let $\underline{\A_{\Sigma}}$ be the toric variety associated to a fan $\Sigma$ in $M$.
By gluing the log structures associated with the homomorphisms 
\[
M\cap \sigma^{\vee}  
\to 
\Z[M\cap \sigma^{\vee}]
\]
for every cone $\sigma$ in $\Sigma$, 
we obtain an fs log scheme $\A_{\Sigma}$ over $\Spec{\Z}$.
This is in fact the divisorial log structure on $\A_{\Sigma}$ coming form the toric boundary $\partial \A_{\Sigma}\subseteq \A_{\Sigma}$, 
i.e., 
the complement of the torus orbit in $X$.
\end{df}


\begin{df}
For a fan $\Sigma$ in a lattice $N$, the \emph{support}\index{fan!support} $\lvert \Sigma \rvert$ of $\Sigma$ is the union of all cones of $\Sigma$.
\end{df}

\begin{df}
For fans $\Sigma$ and $\Sigma'$ in lattices $N$ and $N'$, 
a morphism $\Sigma'\rightarrow \Sigma$ of fans is a homomorphism of lattices $f:N'\rightarrow N$ such that for every cone $\sigma'$ of $N'$, 
$f(\sigma')\subset \sigma$ for some cone $\sigma$ of $N$.
A morphism $\Sigma'\rightarrow \Sigma$ of fans is a \emph{partial subdivision}\index{fan!partial subdivision} if the associated homomorphism $f:N'\rightarrow N$ of lattices is an isomorphism.
A partial subdivision $\Sigma'\rightarrow \Sigma$ is a \emph{subdivision}\index{fan!subdivision} if $\lvert \Sigma'\rvert=\lvert\Sigma\rvert$.
If $\Sigma'\rightarrow \Sigma$ is a morphism of fans, then we can naturally associate a morphism $\A_{\Sigma'}\rightarrow \A_{\Sigma}$ of fs log schemes.
\end{df}

Let $\mathbf{Fan}$ denote the category of fans, and let $\mathbf{MSch}^{fs}$ denote the category of fs monoschemes.\index[notation]{Fan @ $\mathbf{Fan}$}\index[notation]{Mon @ $\mathbf{MSch}^{fs}$}

\begin{thm}
\label{fan=monoscheme2}
There is a fully faithful functor
\begin{equation}
\label{fan=monoscheme2.1}
\mathbf{Fan}\rightarrow \mathbf{MSch}^{fs}.
\end{equation}
The essential image consists of all separated, connected, and toric fs monoschemes whose underlying topological spaces are finite sets.
\end{thm}
\begin{proof}
We refer to \cite[Theorem II.1.9.3]{Ogu}.
\end{proof}

\begin{const}
\label{fan=monoscheme}
Here we give an explicit description of the functor \eqref{fan=monoscheme2.1}.
Let $\Sigma$ be a fan with the cones $\sigma_1,\ldots,\sigma_r$ in a lattice $N$.
Consider the dual cones $\sigma_1^\vee,\ldots,\sigma_r^\vee$.
We can glue $\Spec{\sigma_1^\vee},\ldots,\Spec{\sigma_r^\vee}$ altogether to canonically obtain an fs monoscheme.
Thus we can regard fans as fs monoschemes and morphisms of fans as morphisms of monoschemes.
\vspace{0.1in}

If $P$ is an fs monoid such that $P^\gp$ is a lattice $N$, then $\Spec{P}$ corresponds to the fan with a single maximal cone $P^\vee$ in a lattice $N^\vee$.
\end{const}

\begin{lem}
\label{Fiberproduct_fans}
Let $f\colon\Theta\rightarrow \Sigma$ be a subdivision (resp.\ partial subdivision) of fans in a lattice $N$, 
and let $g\colon \Sigma'\rightarrow \Sigma$ be a morphism of fans with a lattice homomorphism $\theta\colon N'\rightarrow N$.
Then the fiber product $\Theta':=\Theta\times_\Sigma \Sigma'$ in the category of fs monoschemes is a fan in $N'$, 
and the projection $f'\colon \Theta'\rightarrow \Sigma'$ is a subdivision (resp.\ partial subdivision).
\end{lem}
\begin{proof}
Let $\sigma=\Spec{P}$ be a cone of $\Sigma$, 
where $P$ is an fs monoid, 
and let $\tau=\Spec{Q}$ and $\sigma'=\Spec{P'}$ be cones of $\Theta$ and $\Sigma'$ where $P'$ and $Q$ are fs monoids.
Suppose that $f(\tau)=\sigma$ and $g(\sigma')=\sigma$.
Then $Q':=Q\oplus_P P'$ is the submonoid of $N'^\vee$ generated by $\theta^\vee(Q)$ and $P'$, 
i.e., 
$Q'=\theta^\vee(Q)+P'$.
Thus we have $Q'^\vee=(\theta^\vee(Q))^\vee\cap P'^\vee$.
From this local description of $\Theta'$, we see that $\Theta'$ is the fan in $N'$ whose cones are of the form
\begin{equation}
\label{Fiberproduct_fans.1}
\{\theta^\vee(\tau) \cap \sigma':\sigma'\in \Sigma'\text{ and }\tau\in \Theta\}.
\end{equation}

Since $\Sigma'$ and $\Theta'$ are in the same lattice $N'$, $f'\colon\Theta'\rightarrow \Sigma'$ is a partial subdivision.
If $f$ is a subdivision, then $\lvert \Theta\rvert = \lvert \Sigma \rvert$.
Using \eqref{Fiberproduct_fans.1} we deduce that $\lvert \Theta'\rvert=\lvert \Sigma'\rvert$.
\end{proof}

\begin{lem}
\label{monomorphismoffans}
Let $f:\Sigma'\rightarrow \Sigma$ be a subdivision of fans in a lattice $N$.
Then $f$ is a monomorphism in the category of fans.
\end{lem}
\begin{proof}
Let $\xi$ and $\xi'$ are the generic points of the monoschemes $\Sigma$ and $\Sigma'$.
There are isomorphisms
\[
N\cong \cM_{\Sigma,\xi}\cong \cM_{\Sigma',\xi'}.
\]
Here $\Sigma$ and $\Sigma'$ can be viewed as separated fs monoschemes, 
and we may apply \cite[Corollary II.1.6.4]{Ogu} to conclude that $f$ is a monomorphism in the category of fs monoschemes.
In particular, $f$ is a monomorphism in the category of fans.
\end{proof}

%

\begin{prop}
\label{Propersubdivision}
Let $\Sigma$ be a fan in a lattice $N$.
Then there is a canonical bijection
\[
\hom(\Spec{\N},\Sigma)\cong \lvert \Sigma \rvert.
\]
\end{prop}
\begin{proof}
A homomorphism $f:\Z\rightarrow N$ gives a morphism $\Spec{\N}\rightarrow \Sigma$ precisely when $f(\N)$ lies in a cone of $\Sigma$, 
i.e., 
when $f(\N)$ lies in $\lvert \Sigma \rvert$.
This implies the bijection.
\end{proof}

\begin{df}[{\cite[Definition 3.3.17]{CLStoric}}]
\label{Starsubdivision}
Suppose $\Sigma$ is a fan in $\Z^n$,
and let $\tau$ be a cone of $\Sigma$ such that for any cone $\sigma$ containing $\tau$, $\sigma$ is smooth.
Express $\tau$ as an $r$-dimensional cone ${\rm Cone}(a_1,\ldots,a_r)$.
Then for each maximal cone $\sigma={\rm Cone}(a_1,\ldots,a_r,a_{r+1},\ldots,a_n)$ of $\Sigma$ containing $\tau$,
replace $\sigma$ by the fan with maximal cones 
\[
\begin{split}
&{\rm Cone}(a_1,\ldots,a_{r-1},a_{r+1},\ldots,a_n,a_1+\cdots+a_r),\\
&{\rm Cone}(a_1,\ldots,a_{r-2},a_r,\ldots,a_n,a_1+\cdots+a_r),\cdots,{\rm Cone}(a_2,\ldots,a_n,a_1+\cdots+a_r).
\end{split}
\]
The resulting new fan $\Sigma^*(\tau)$ is called the {\it star subdivision} \index{fan!star subdivision} of $\Sigma$ relative to $\tau$.
\end{df}

\begin{lem}
\label{A.3.31}
Let $\Sigma$ and $\Sigma'$ be two smooth fans with $|\Sigma|=|\Sigma'|$. 
Then there is a subdivision $\Sigma''$ of $\Sigma$ obtained by a finite succession of star subdivisions relative to two-dimensional cones such that $\Sigma''$ is a subdivision of $\Sigma'$.
\end{lem}
\begin{proof}
See \cite{MR803344} and \cite[pp.\ 39-40]{TOda}.
\end{proof}

\begin{lem}
\label{Compactification_fan}
Let $\Sigma$ be a fan in a lattice $N$.
Then there exists a fan $\Sigma'$ in $N$ containing $\Sigma$ such that $\lvert \Sigma' \rvert=N$.
\end{lem}
\begin{proof}
See \cite[p.\ 18]{TOda}.
\end{proof}

\subsection{Strict morphisms}

Let $f:X \rightarrow Y$ be a morphism of schemes. 
Given a logarithmic structure $\cM_{Y}$ on $Y$, 
we have defined a logarithmic structure
$f_{log}^*\cM_{Y}$ on $X$, called the inverse image of $\cM_{Y}$ or the pullback log structure on $X$.
Using the inverse image of logarithmic structures, we arrive at the following definition incorporated in our notion of finite log correspondences. 

\begin{df}
A morphism of log schemes $(X,\cM_{X}) \to (Y,\cM_{Y})$ is called strict if the induced homomorphism $f_{log}^{*}\cM_{Y}\to \cM_{X}$ is an isomorphism, 
and it is said to be a strict closed immersion if it is strict and $X\to Y$ is a closed immersion.
\end{df}

\begin{exm}
In Example \ref{ex:affine-toric-log} the log structure on $\Spec{R[P]}$ can be viewed as the pullback log structure on $\Spec{\Z[P]}$ along the morphism induced by $\Z\to R$.
\end{exm}

\begin{exm}
Let $f\colon X\to \A^{1}$ be a smooth morphism.
Then $f$ is a strict morphism of log schemes with respect to the divisorial log structures given by $0\in\A^{1}$ and $f^{-1}(0)\subseteq X$.
\end{exm}

\begin{lem}
\label{lem::maximal strict open}
Let $f:X\rightarrow Y$ be a morphism of fs log schemes.
Then there exists a maximal open subscheme $U$ of $X$ such that the restriction $f_{\vert U}$ of $f$ to $U$ is strict over $Y$.
\end{lem}
\begin{proof}
In \cite[Lemma 1.5]{MR1754621} it is shown that if $X$ is strict at a point $x$, then $f$ restricts to a strict morphism over an \'etale neighborhood.
The same argument applies to fs log schemes equipped with Zariski log structures,
\end{proof}

\subsection{Charts of logarithmic structures}
\label{subsection:cols}
\begin{df}
Let $X$ be a log scheme and $P$ be a monoid. 
A {\em chart}\index{chart} for $\cM_{X}$ is a morphism $P\rightarrow \Gamma(X,\cM_{X})$ 
such that the induced map of logarithmic structures $P^{a}\to \cM_{X}$ is an isomorphism.
Here, 
$P^{a}$ is the logarithmic structure associated to the pre-logarithmic structure given by
\[
P\rightarrow \Gamma(X,\cM_{X})\rightarrow \Gamma(X,\mathcal{O}_{X}).
\]
If $P$ is an fs monoid, the chart is said to be an fs chart.
\end{df}

A chart for $\cM_{X}$ is equivalent to a strict morphism 
\[
f
\colon 
X 
\to 
\A_P.
\]
In fact, by \cite[III.1.2.9]{Ogu},
there is a bijection 
\[
\hom_{\LogSch}(X,\A_P)
\rightarrow 
\hom_{\Mon}(P,\Gamma(X,\cM_{X}))
\]
associating to $f$ the composition 
\[
\xymatrix{
P \ar[r]&\Gamma(X, P_X)\ar[r] & \Gamma(X, \cM_{X}).
}
\]
The morphism $f\colon X\rightarrow \A_P$ also gives a morphism of monoidal spaces
\[
X=(X,\cM_X)\rightarrow \Spec{P}.
\]

\begin{df}
A chart\index{chart!of a morphism} for a morphism of log schemes $(X,\cM_{X})\to (Y,\cM_{Y})$ is a triple $(P\to \cM_{X},Q\to \cM_{Y},Q\to P)$, 
where $P\to \cM_{X}$ and $Q\to \cM_{Y}$ are charts, 
and $Q\to P$ is is a homomorphism such that the diagram 
\begin{equation}
\label{diagram:chartmaps}
\vcenter{\xymatrix{
Q \ar[r]  \ar[d] & P \ar[d]\\
f_{log}^{\ast}\cM_{Y}  \ar[r] & \cM_{X} 
}}
\end{equation}
commutes.
\end{df}

\begin{df}
A log scheme $X$ is said to be \emph{coherent}\index{log scheme! coherent}
if \'etale locally there is a chart $P\rightarrow \cM_{X}$ with $P$ a monoid.
If moreover $P$ can be chosen to be integral (resp.\ fine, resp.\ saturated, resp.\ fs), then $X$ is called an \emph{integral} (resp.\ \emph{fine}, resp.\ \emph{saturated}, resp.\ \emph{fs}) log scheme.
If such a chart exists Zariski locally, we say that $X$ has a \emph{Zariski log structure}.
Finally, if $P$ can be chosen isomorphic to $\N^k$, the logarithmic structure is called locally free.
\end{df}

We are particularly interested in Zariski log structures, which locally arise from pullbacks of monoid algebras.
The property of being fine and saturated is a logarithmic analog of the property of being normal and of finite type  \cite[Chapter IV.1.2]{Ogu}. 

%

\begin{df}\index{log scheme!fiber product}
Suppose that $Y\rightarrow X$ and $X'\rightarrow X$ are morphisms of coherent (resp.\ integral, resp.\ saturated) log schemes.
The fiber product in the category of coherent (resp.\ integral, resp.\ saturated) log schemes is denoted by
\[
Y\times_X^{\rm coh}X'
\text{ (resp.\ }Y\times_X^{\rm int}X',
\text{ resp.\ }Y\times_X X'\text{)}.
\]
\end{df}

\begin{exm}
Suppose that $Q\rightarrow P$ and $P'\rightarrow P$ are homomorphisms of monoids (resp.\ integral monoids, resp.\ saturated monoids).
Then there exists an isomorphism
\begin{gather*}
\A_Q\times_{\A_P}^{\rm coh}\A_{P'}\cong \A_{Q\oplus_P^{\rm mon} P'}
\\
\text{ (resp.\ } \A_Q\times_{\A_P}^{\rm int}\A_{P'}\cong \A_{Q\oplus_P^{\rm int} P'},
\text{ resp.\ } \A_Q\times_{\A_P}\A_{P'}\cong \A_{Q\oplus_P P'}\text{)}.
\end{gather*}
\end{exm}

\begin{lem}
\label{Fiber.product.strict.morphism}
Let $f:X\rightarrow S$ be a strict morphism of fs log schemes.
Then for any morphism $S'\rightarrow S$ of fs log schemes, the induced morphism
\[
\underline{X\times_S S'}\rightarrow \underline{X}\times_{\underline{S}}\underline{S'}
\]
is an isomorphism.
\end{lem}
\begin{proof}
From the construction of fiber products in the category of fs log schemes, 
see \cite[Corollary III.2.1.6]{Ogu},
it suffices to show that the fiber product $X\times_S^{\rm coh}S'$ in the category of coherent log schemes is an fs log scheme.
This question is Zariski local on $S'$,
so we may assume that $S'$ has an fs chart $S'\rightarrow \A_P$.
Since $f$ is strict,
the projection $X\times_S^{\rm coh}S'\rightarrow X$ is also strict.
Then $X\times_S^{\rm coh}S'$ has an fs chart.
Thus $X\times_S^{\rm coh}S'$ is an fs log scheme.
\end{proof}
\begin{df}[{\cite[Definition II.2.3.1]{Ogu}}]
\label{A.9.8}
Suppose $X$ is a coherent log scheme with a chart $\beta:P\rightarrow \cM_{X}$, and let $x\in X$ be a point. 
Then $\beta$ is called exact\index{chart!exact} at $x$ if 
\[
\overline{\beta_x}:\overline{P}\rightarrow \overline{\cM}_{X,x}
\]
is an isomorphism,
and $\beta$ is called neat at $x$ if it is exact\index{chart!neat} at $x$ and $P$ is sharp.
\end{df}

\begin{rmk}
\label{A.9.10}
(1) Let $X$ be a fine log scheme with a fine chart $\beta:P\rightarrow \cM_{X}$, and let $x\in X$ be a point. 
Then, Zariski locally on $X$, $\beta$ factors through an exact chart at $x$ by \cite[Remark II.2.3.2]{Ogu}.
\vspace{0.1in}
  
(2) Let $X$ is an fs log scheme, and let $x\in X$ be a point. Then Zariski locally on $X$ there exists a neat chart at $x$ by \cite[Proposition II.2.3.7]{Ogu}.
\end{rmk}

\begin{lem}
\label{lem::Smchart}
Suppose $X$ is an fs log scheme log smooth over a scheme $S$.
Then Zariski locally on $X$, there is a neat fs chart $P$ such that the induced morphism $X\rightarrow S\times \A_P$ is strictly smooth.
If $\underline{X}$ is smooth over $S$, then $P\cong \N^r$ for some integer $r\in \N$.
\end{lem}
\begin{proof}
Since the question is Zariski local on $X$, we may assume that $X$ admits a neat chart $P$ at a point $x\in X$, 
see Remark \ref{A.9.10}(2).
In particular, since $P$ is sharp,
$P^{\rm gp}$ is torsion-free by \cite[Proposition I.1.3.5(2)]{Ogu}.
Owing to \cite[Proposition II.2.3.7]{Ogu} there is an isomorphism
\[
\cM_{X,x}^{\rm gp}\cong \cO_{X,x}^*\oplus P.
\]
This implies that the order of the torsion subgroup of $\cM_{X,x}^{\rm gp}$ is invertible in $k$.
Hence we can apply \cite[Theorem IV.3.3.1(3)]{Ogu} to conclude that $X\rightarrow S\times \A_P$ is strictly smooth.
\vspace{0.1in}
  
If $\underline{X}$ is smooth over $S$, then $S\times \underline{\A_P}$ is smooth over $S$ by \cite[Proposition IV.17.7.7]{EGA}. 
This shows that $P\cong \N^r$ for some $r$ by \cite[Theorem 1.3.12]{CLStoric} because $P$ is sharp. 
\end{proof}

We refer to {\rm Definition \ref{A.3.17}} for the notion of a strict normal crossing divisor on a scheme.

\begin{lem}\index{divisor!strict normal crossing}
\label{lem::SmlSm}
Suppose that $X\in SmlSm/S$.
Then there exists a strict normal crossing divisor $Z$ on $\ul{X}$ and an isomorphism
\[
X\cong (\underline{X},Z).
\]
\end{lem}
\begin{proof}
The question is Zariski local on $\ul{X}$.
Owing to Lemma \ref{lem::Smchart} there is an fs chart $P\cong \N^r$ of $X$ such that the induced morphism $f:X\rightarrow S\times \A_P$ is strict smooth.
That is, 
$\A_P\cong (\underline{\A_P},W)$,
where $W$ is the strict normal crossing divisor induced by the log structure on $\A_P$. 
Using the strict smooth morphism $f:X\rightarrow S\times \A_P$, 
we see the claimed isomorphism follows for the strict normal crossing divisor $Z:=f^*W$.
\end{proof}

\begin{df}
\label{df::closedimmersion}
A morphism of fs log schemes $f:X\rightarrow Y$ is a \emph{closed immersion}\index{morphism of log schemes!closed immersion} if the following two properties are satisfied:
\begin{enumerate}
\item[(i)] The underlying morphism of schemes $\underline{f}:\underline{Z}\rightarrow \underline{X}$ is a closed immersion.
\item[(ii)] The homomorphism $f^\flat:f_{log}^*(\cM_Y)\rightarrow \cM_X$ is surjective.
\end{enumerate} 
\end{df}

\subsection{Logarithmic smoothness}
\label{subsec::logsmooth}
Log smoothness is one of the central concepts in log geometry.
Kato \cite{MR1463703} showed that every log smooth morphism factors \'etale locally as a usual smooth morphism composed with a morphism induced by a homomorphism of monoids, 
which essentially determines the log structure.
\vspace{0.1in}

Consider the following commutative diagram of log schemes illustrated with solid arrows:

\begin{equation}
\label{diagram:smooth}
\vcenter{\xymatrix{
T_{0} \ar[r]^{\phi}  \ar[d]_{i}& X \ar[d]^{f}\\
T_{1} \ar[r]^{\psi} \ar@{-->}[ur] & Y
}}
\end{equation}
Here $i$ is a strictly closed immersion defined by a square zero ideal (note that $T_{0}$ and $T_{1}$ have the same underlying topological spaces).  
An infinitesimal lifting property defines logarithmic smoothness.

\begin{df}
\label{df:logsmooth}
A morphism $f:X\rightarrow Y$ of fine log schemes is logarithmically smooth (resp.\ logarithmically \'etale)\index{morphism of log schemes!log smooth}\index{morphism of log schemes!log \'etale} if the underlying morphism $\ul{X}\rightarrow \ul{Y}$ 
is locally of finite presentation and for any commutative diagram (\ref{diagram:smooth}), 
\'etale locally on $T_{1}$  there exists a (resp.\ there exists a unique) morphism $g:T_{1}\rightarrow X$ such that $\phi=g\circ i$ and $\psi=f\circ g$.
\end{df}

 We have the following helpful criterion for smoothness from \cite[Theorem 3.5]{MR1463703}.

\begin{thm}[K.\ Kato]\label{KatoStrThm} 
Let $f:X\rightarrow Y$ be a morphism of fine log schemes. 
Assume there is a chart $Q\rightarrow \cM_{Y}$, where $Q$ is a finitely generated integral monoid. 
Then the following are equivalent:
\begin{enumerate}
\item $f$ is logarithmically smooth (resp. logarithmically \'etale).
\item \'Etale locally on $X$, 
there is a chart $(P\rightarrow \cM_{X},Q\rightarrow \cM_{Y},Q\rightarrow P)$ extending $Q\rightarrow \cM_{Y}$ and satisfying the following properties.
\begin{enumerate}
\item The kernel and the torsion part of the cokernel (resp. the kernel and the cokernel) of $Q^{\gp}\rightarrow P^{\gp}$ are finite groups of orders invertible on $X$.
\item The induced morphism $X\rightarrow Y\times_{\text{Spec}(\Z[Q])} \text{Spec}(\Z[P])$ is \'etale in the usual sense.
\end{enumerate}
\end{enumerate}
\end{thm}

\begin{rmk}
In the second part of Theorem \ref{KatoStrThm}, the \'etale locality condition is used to ensure that $X$ admits a chart.
Thus, since we are working with Zariski log structures, 
we may assume the required chart exists Zariski locally on $X$.
\end{rmk}

\begin{exm}
\label{ex:toriclogsmooth}
Suppose $P$ is an fs monoid such that $P^\gp$ is torsion-free.
Let $X=\A_P$ and $Y=\Spec{\Z}$ be equipped with the trivial logarithmic structure.
Then, 
according to the theorem, 
$X$ is logarithmically smooth relative to $Y$, though the underlying toric variety might be singular. 
This is Kato's ``magic of log."  
\end{exm}

\begin{exm}
\label{ex:normalcrossinglogsmooth}
Let $\sigma$ be the cone of $M_{\R}=\R^{d}$ generated by the standard basis vectors $e_{i}$, $1\leq i\leq d$.
For integers $a_{j}>0$, $1\leq i\leq d$, 
let $\tau$ be the subcone of $\sigma$ generated by $a_{1}e_{1}+\dots+a_{d}e_{d}$.
We set $P:=M\cap\sigma$ and $Q:=M\cap\tau$.
Assuming the greatest common divisor $N:=\text{gcd}(a_{1},\dots,a_{d})$ is invertible in the field $k$,
we can identify the fiber product 
\[
\Spec{k}\times_{\Spec{\Z[1/N][Q]}}\Spec{\Z[1/N][P]}
\]
with the normal crossing variety
\[
X
:=
\Spec{
k[x_{1},\dots,x_{d}]/(x_{1}^{a_{1}}\cdots x_{d}^{a_{d}})
}.
\]
There are morphisms 
\begin{equation}
\label{diagram:normalcrossing2}
\vcenter{\xymatrix{
\N \ar[r]  \ar[d]_{\phi} & k \ar[d]\\
\N^{d} \ar[r]^-{\psi} & k[x_{1},\dots,x_{d}]/(x_{1}^{a_{1}}\cdots x_{d}^{a_{d}}),
}}
\end{equation}
where $\phi(1)=a_{1}e_{1}+\dots+a_{d}e_{d}$ and $\psi(e_{i})=x_{i}$.
The induced morphism 
\[
(X,\N^{d})
\to (\Spec{k},\N)
\]
is log smooth.
\end{exm}

\begin{exm}
Suppose $(X,\cM_{X})$ is an fs log scheme which is log smooth over $\Spec{k}$.
Then $X$ is covered \'etale locally by schemes which are smooth over toric varieties,
and there exists a divisor $D$ of $X$ such that the log structure is equivalent to that of $j^{\ast}\mathcal{O}_{X\setminus D}\cap \mathcal{O}_{X}$, 
where $j\colon X\setminus D\to X$ is the inclusion.
\end{exm}

\begin{exm}
The endomorphism of $\A_\N$ induced by $t\mapsto t^{n}$ is a Kummer morphism of log schemes, 
and a log \'etale morphism in case $n$ and $\text{char}(k)$ are coprime.
\end{exm}


\subsection{Logarithmic differentials}

One basic idea in log geometry is that the sheaves of monoids specify where $d\log$ makes sense.
To form sheaves of logarithmic differentials, 
one adds to the sheaf of differentials symbols of the form $d\log \alpha(m)$ for all elements $m\in \cM_{X}$.
\begin{df}
\label{log-differential}\index[notation]{Omega @ $\Omega^i_{-/k}$} \index{log differential}
Let $f: X \to Y$ be a morphism of fine log schemes. 
The sheaf of relative logarithmic differentials $\Omega_{X/Y}^{1}$ is given by 
\[
\Omega_{X/Y}^{1}
\colon=
\Bigl(\Omega_{X/Y}\oplus(\mathcal{O}_{X}\otimes_{\Z} M^{\gp}_{X}) \Bigr)\big/\mathcal{K},
\]
where $\mathcal{K}$ is the $\mathcal{O}_{X}$-module generated by local sections of the following forms:
\begin{enumerate}
\item $(d\alpha(a),0)-(0,\alpha(a)\otimes a)$ with $a\in \cM_{X}$;
\item $(0,1\otimes a)$ with $a\in \im(f^{-1}\cM_{Y}\rightarrow \cM_{X})$.
\end{enumerate}
The universal derivation $(\partial,D)$ is given by 
\[
\partial 
\colon
 \mathcal{O}_{X}\stackrel{d}\rightarrow\Omega_{X/Y}\rightarrow\Omega_{X/Y}^{1}
\] 
and 
\[
D
\colon 
\cM_{X}\rightarrow \mathcal{O}_{X}\otimes_{\Z} \cM_{X}^{\gp}\rightarrow\Omega_{X/Y}^{1}.
\]
\end{df}

\begin{exm}
Let $h: Q \rightarrow P$ be a morphism of fine monoids. 
For the induced morphism $f:\A_P\rightarrow\A_Q$ we have 
\[
\Omega_{f}^{1}=\mathcal{O}_{\A_P}\otimes_{\Z} \coker{h^{\gp}}.
\] 
The generators correspond to the logarithmic differentials $d\log\alpha(p)$ for $p\in P$, 
which are regular on the torus $\text{Spec}(\Z[P^{\gp}])$, modulo those coming from $Q$. 
\end{exm}

Logarithmic differentials behave somewhat analogously to differentials.
Let 
\[
X\stackrel{f}{\rightarrow}Y\stackrel{g}{\rightarrow}Z
\] 
be a sequence of morphisms of fine log schemes. 
\begin{enumerate}
\item There is a natural exact sequence 
\[
f^{*}\Omega_{Y/Z}^{1} \rightarrow \Omega_{X/Z}^{1}\rightarrow \Omega_{Z/Y}^{1}\rightarrow 0.
\]
\item If $f$ is log smooth, 
then $\Omega_{f}^{1}$ is a locally free $\mathcal{O}_{X}$-module of finite type, and there is an exact sequence
\[
0\rightarrow f^{*}\Omega_{Y/Z}^{1} \rightarrow \Omega_{X/Z}^{1}\rightarrow \Omega_{X/Y}^{1}\rightarrow 0.
\]
In particular, if $f$ is log \'etale, then $\Omega_{f}^1=0$.
\item If $g f$ is log smooth and the sequence in (2) is exact and splits locally, then $f$ is log smooth.
\end{enumerate}

We refer to \cite[Chapter IV]{Ogu} for the proof.

\begin{exm}
Let $\A_\Sigma$ be the log scheme introduced in Definition \ref{df:toricvarieties}.
We have 
\[
\Omega_{\A_{\Sigma}}^{1}
\cong
\mathcal{O}_{\A_{\Sigma}}\otimes_{\Z} N.
\]
\end{exm}

\begin{exm}
The normal crossing variety 
\[
X
:=
\Spec{
k[x_{1},\dots,x_{n}]/(x_{1}\cdots x_{d}})
\]
defines a log structure of ``semistable type''  over the standard log point $(\Spec{k},\N)$.
This is a special case of Example \ref{ex:logpoints} with $X=\Spec{k[\N^{n}]}$ and $D=V(x_{1}\cdots x_{d})$, 
corresponding to the submonoid of $\mathcal{O}_{X}$ generated by $\mathcal{O}_{X}^{\times}$ and $\{x_{1},\dots,x_{d}\}$.
We have that $\Omega_{X}^{1}$ is a free $\mathcal{O}_{X}$-module generated by the logarithmic differentials 
\[
\frac{dx_{1}}{x_{1}}
=
d\log x_{1},
\cdots,
\frac{dx_{d}}{x_{d}}
=
d\log x_{d},
dx_{d+1},
\cdots,
dx_{n}
\]
subject to the relation 
\[
d\log x_{1}
+
\cdots
+
d\log x_{d}
=
0.
\]
\end{exm}

\subsection{Exact morphisms}

\begin{df}
A homomorphism $\theta:P\rightarrow Q$ of monoids is \emph{exact}\index{homomorphism of monoids!exact} if the commutative diagram
\[
\begin{tikzcd}
P\arrow[d,"\theta"']\arrow[r]&
P^\gp\arrow[d,"\theta^\gp"]
\\
Q\arrow[r]&
Q^\gp
\end{tikzcd}
\]
is cartesian.
\end{df}

\begin{df}
A morphism $f:Y\rightarrow X$ of log schemes is \emph{exact}\index{morphism of log schemes@exact} if for every point $y\in Y$ the naturally induced homomorphism
\[
\cM_{X,f(x)}\rightarrow \cM_{Y,y}
\]
is exact.
\end{df}

\begin{prop}
Let $f$ be an exact morphism of integral (resp.\ saturated) log schemes.
Then every pullback of $f$ in the category of integral (resp.\ saturated) log schemes is also exact.
\end{prop}
\begin{proof}
See \cite[Proposition III.2.2.1(2)]{Ogu}.
This reference states only the integral case.
However, the saturated case can be proved similarly, see \cite[Proposition I.4.2.1(6)]{Ogu}.
\end{proof}

\begin{prop}\label{kummer-is-logetandstrict}
Let $f:X\rightarrow Y$ be a log \'etale morphism of fs log schemes.
Then $f$ is Kummer if and only if $f$ is exact.
\end{prop}
\begin{proof}
We refer to \cite[1.6]{MR1922832}.
\end{proof}

\subsection{Kato-Nakayama spaces of fs log schemes}
If $X$ is a fine and saturated log scheme that is locally of finite type over $\text{Spec}({\C})$ we are entitled to the log analytic space $X_{\text{an}}$. 
The idea of Kato-Nakayama is to associate to $X_{\text{an}}$ a topological space $X_{\log}$ that embodies the log structure of $X$ is a topological way by using points instead of sheaves of monoids.
By construction there exists a continuous map $\tau\colon X_{\log}\to X_{\text{an}}$ that is proper and surjective. 
Let $U\subseteq X_{\text{an}}$ be the trivial locus of the log structure, 
i.e., 
the largest open subset over which $\mathcal{O}_{X_{\text{an}}}^{\times}\hookrightarrow M_{X_{\text{an}}}$ is an isomorphism.
Then the open immersion $i\colon U\to X_{\text{an}}$ factors through $\tau$,
so that $X_{\log}$ can be considered as a ``relative compactification'' of $i$.
We give a short review of Kato--Nakayama spaces and refer to \cite[Section V.1]{Ogu} for more details.
\vspace{0.1in}

Let $\mathrm{pt}^{\dagger}$ be the log analytic point $\mathrm{pt}=\text{Spec}(\C)_{\text{an}}$ with monoid $M_{\mathrm{pt}}=\R_{\geq 0}\times S^1$, 
where each factor is given the monoidal multiplicative structure so that $\R_{\geq 0}^{\gp}=0$,  
and map $\alpha\colon M_{\mathrm{pt}}\to \C$ given by $(r,a)\mapsto ra$. 
As a set, 
\[
X_{\log}:=
\hom(\mathrm{pt}^{\dagger},X_{\text{an}})
\] 
is comprised of morphisms of log analytic spaces from $\mathrm{pt}^{\dagger}$ to $X_{\text{an}}$. 
One identifies $X_{\log}$ with the set of pairs $(x,\phi)$, 
where $x\in X_{\text{an}}$ and $\phi\colon M_{X_{\text{an}},x}^{\gp} \rightarrow S^{1}$ is a homomorphism of abelian groups such that $\phi(f)={f(x)}/{|f(x)|}\in S^{1}$ 
for every $f\in \mathcal{O}_{X_{\text{an}},x}^\times$. 
With this description, 
$\tau\colon X_{\log}\to X_{\text{an}}$ is the projection to the first coordinate $(x,\phi)\mapsto x$.
For every open subset $U\subseteq X_{\text{an}}$ and section $m\in M_{X_{\text{an}}}(U)$, 
we obtain a function $f_{U,m}\colon U_{\log}\rightarrow S^{1}$, defined by $(x,\phi)\mapsto \phi(m_x)\in S^{1}$. 
Here $U_{\log}$ denotes the subset of $X_{\log}$ of points $(x,\phi)$ with $x\in U$,
i.e., 
the Kato--Nakayama space of the log analytic space $(U,M_{X_{\text{an}}}|_U)$. 
The topology on $X_{\log}$ is the coarsest topology that makes the projection $\tau\colon X_{\log}\to X_{\text{an}}$ and all the functions $f_{U,m}\colon U_{\log}\rightarrow S^1$ continuous.
For every $x\in X_{\text{an}}$, 
the fiber $\tau^{-1}(x)$ is a torsor under the space of homomorphisms of abelian groups $\hom(M_{X_{\text{an}},x}^{\gp},S^{1})$.
If $X_{\text{an}}$ is fs (or more generally if $M_{X_{\text{an}},x}$ is torsion-free), 
then $\tau^{-1}(x)$ is non-canonically isomorphic to a real torus $(S^{1})^r$ where $r$ is the rank of the free abelian group $\overline{M}^{\gp}_{X_{\text{an}},x}=M^{\gp}_{X_{\text{an}},x}/\mathcal{O}_{X,x}$.
\vspace{0.1in}

The construction of Kato-Nakayama spaces is functorial and compatible with strict morphisms of fine saturated log analytic spaces $X_{\text{an}}\rightarrow Y_{\text{an}}$, 
i.e., 
there is a naturally induced cartesian square of topological spaces:
$$
\xymatrix{
X_{\log}\ar[r] \ar[d]& Y_{\log}\ar[d]\\
X_{\text{an}}\ar[r] & Y_{\text{an}}
}
$$

\begin{exm}
\label{example:kn}\mbox{}
\begin{itemize}
\item[1.] For $\A^1$ with its toric log structure $\N\to \Spec{\N}=\A^1$ we have $(\A^1)_{\log}=\R_{\geq 0}\times S^{1}$.
The projection map $\tau\colon \R_{\geq 0}\times S^{1}\rightarrow \C$ sends $(r,a)$ to $ra$.
More generally,
let $P$ be a fine monoid and $X$ the affine toric scheme $\Spec{\C[P]}$ with its toric log structure. 
Then 
\[
X_{\log}
=
\hom(P,\R_{\geq 0}\times S^{1})
=
\hom(P,\R_{\geq 0})\times \hom(P,S^{1})
\]
is locally compact with topology induced by the one on $\R_{\geq 0}\times S^{1}$. 
\item[2.] If $X$ is smooth and $M_D$ is the compactifying log structure coming from a normal crossing divisor $D\subseteq X$, 
then $(X,M_D)_{\log}$ can be identified with the real oriented blow-up of $X$ along $D$.  
This is a ``smooth manifold with corners''.
\end{itemize}
\end{exm}

The last example suggests that $X_{\log}$ should be thought of as the complement of an ``open tubular neighborhood of the log structure.''
Our next observation is used in Proposition \ref{prop::unitdiskbundle}.

\begin{prop}
\label{A.9.74}
Let $X$ be an fs log scheme log smooth over $\C$. Then $X_{\log}$ is homotopy equivalent to $(X-\partial X)_{\text{an}}$.
\end{prop}
\begin{proof}
As observed in \cite[Remark 1.5.1]{MR1700591}, $X$ is a manifold with boundary. 
Our assertion follows readily from this. 
\end{proof}

\begin{exm} 
Let $D\subseteq X_{\text{an}}$ be a simple normal crossings divisor. 
Then $(X,M_D)_{\log}$ is a manifold with boundary; its interior is $X_{\text{an}}-D$ \cite[Theorem V.1.3.1]{Ogu}.
The inclusion $X_{\text{an}}-D \rightarrow (X_{\text{an}},M_D)_{\rm log}$ is a homotopy equivalence. 
\end{exm}

\subsection{Log blow-ups}
In this section, we briefly recall Niziol's work on log blow-ups \cite{Niz}.
Recall that an ideal of a monoid $P$ is a subset $I\subseteq P$ such that $P+I\subseteq I$. 

\begin{df}[{\cite[Definition 3.1]{Niz}}]
A subsheaf $\cI\subseteq \cM_{X}$ of ideals of the monoid $\cM_{X}$ is coherent if locally for the \'etale topology of $X$ there is a chart $P\rightarrow \cM_{X}$ for the log structure and a finitely generated ideal $I\subseteq P$ of the monoid $P$, such that $\cI$ coincides with the subsheaf of $\cM_{X}$ generated by the image of $I$.
\end{df}

A subsheaf $\cI \subseteq \cM_{X}$ of ideals is invertible if a single element can locally generate it, or equivalently if it is induced locally by a principal ideal of a chart $P\rightarrow \cM_{X}$.
If $\cI$ is coherent, then it is invertible if and only if $\cI_x \subseteq M_{X,x}$ is a principal ideal for every $x\in X$.
For a coherent sheaf of ideals $\cI$ of the log structure on $X$, 
the log blow-up of $X$ along $\cI$ is a log scheme $(X_\cI,M_\cI)$ equipped with a map $f_\cI\colon X_\cI \rightarrow X$ satisfying a universal property: 
the ideal generated by $f_\cI^{-1}\cI$ in $M_{\cI}$ is invertible and every morphism of log schemes $f\colon Y\rightarrow X$ such that the ideal generated by $f^{-1}\cI$ on $Y$ is invertible factors 
uniquely through $f_\cI$.   
Locally on $X$, for a chart $P\rightarrow \cM_{X}$ and a finitely generated ideal $I\subseteq P$ inducing $\cI$, 
the log blow-up is given by the pullback along $X\rightarrow \text{Spec}(k[P])$ of the blow-up $\text{Proj}( \bigoplus_n I^n)\rightarrow \text{Spec}(k[P])$.
\vspace{0.1in}

The following result due to Niziol strengthens Kato's work \cite[(10.4)]{MR1296725}.

\begin{prop}
\label{A.3.19}
For $X\in lSm/k$, there is a log blow-up $Y\rightarrow X$ such that $Y\in SmlSm/k$.
\end{prop}
\begin{proof}
This is a particular case of \cite[Theorem 5.10]{Niz}.
\end{proof}

\begin{rmk}
When forming triangulated categories of logarithmic motives, for every log blow-ups $Y\rightarrow X$ in $lSm/k$, we invert the naturally induced morphism of motives
\[
M(Y)\rightarrow M(X).
\]
Since every fs log scheme $X$ admits a log blow-up equipped with a Zariski log structure, see \cite[Theorem 5.4]{Niz}, 
we may assume that the log structure of $X$ --- and thus of every log blow-up \cite[Proposition 4.5]{Niz} --- is Zariski.
\end{rmk}

\begin{rmk}
Log blow-ups performed in the category of fine log schemes or the category of fs log schemes differ by a ``normalization'' or ``saturation'' on the level of monoids.
See \cite[Section 4]{Niz} for details. 
\end{rmk}

\begin{rmk}
Let $X$ be a (normal) toric variety equipped with its natural log structure.
Every (saturated) log blow-up of $X$ turns out to be a ``toric blow-up'' given by subdivisions of the fan $\Sigma$ of $X$.
Indeed, 
by \cite[Proposition 4.3]{Niz}, 
the log blow-up along a coherent sheaf of ideals $\cI\subseteq \cM_{X}$ is the same as the normalization of the blow-up of $X$ along the coherent sheaf of ideals 
$\langle \alpha(\cI)\rangle\subseteq \mathcal{O}_X$ for $\alpha\colon \cM_{X}\rightarrow \mathcal{O}_X$. 
When $X=\text{Spec}(k[P])$, 
the global sections of this image form a homogeneous ideal of $k[P]$, 
and the blow-up along this ideal affords a description via a subdivision of the corresponding cone. 
Patching together the subdivisions on the affine pieces yields a subdivision of $\Sigma$ realizing the log blow-up.
\end{rmk}

\begin{exm}
Set $X=\A_{\N\oplus \N}$, and let $K$ be the zero ideal of $\N\oplus \N$. 
Now let $\cK$ be the coherent sheaf of ideals in $\cM_{X}$ associated with $K$, 
and let $P_1$ (resp.\ $P_2$) be the submonoid of $\Z\oplus \Z$ generated by $(-1,1)$ and $(1,0)$ (resp.\ $(1,-1)$ and $(0,1)$). 
From the construction of log blow-ups in \cite[Lemma II.1.7.3, Remark II.1.7.4]{Ogu}, 
we see that the log blow-up $Y$ of $X$ along $\cK$ is the gluing of $\A_{P_1}$ and $\A_{P_2}$ along $\A_{\Z\oplus \Z}$.  
On the underlying schemes, we note that $\underline{Y}$ is the usual blow-up of $\underline{X}=\A^2_{k}$ along the origin.
\end{exm}

\subsection{Log \'etale monomorphisms}
\label{subsec::mono}
A morphism of schemes is a closed (resp.\ an open) immersion if and only if it is a proper (resp.\ an \'etale) monomorphism by \cite[Corollaire IV.18.12.6, Th\'eor\`eme IV.17.9.1]{EGA}.
Thus a morphism of schemes is an isomorphism if and only if it is a surjective proper \'etale monomorphism.
\vspace{0.1in}

In contrast to this, 
a proper (resp.\ a log \'etale) monomorphism need not be a closed (resp.\ an open) immersion, 
and a surjective proper log \'etale monomorphism does not need to be an isomorphism.
However, on the level of log motives, we will turn surjective proper log \'etale monomorphisms into isomorphisms.
For this purpose, we shall study log \'etale monomorphisms.

\begin{exm}
\label{A.5.15}
By \cite[Proposition III.2.6.3(3)]{Ogu}, log blow-ups are universally surjective proper monomorphisms.
Zariski locally on $X$ and $Y$, 
$f$ has a chart $P\rightarrow Q$ such that $P^{\rm gp}\rightarrow Q^{\rm gp}$ is an isomorphism, which can be checked by appealing to \cite[Lemma II.1.7.3, Remark II.1.7.4]{Ogu}. 
Thus $f$ is log \'etale by \cite[Theorem IV.3.3.1]{Ogu}. In conclusion, $f$ is a universally surjective proper log \'etale monomorphism.
\end{exm}

\begin{lem}[{\cite{ParThesis}}]
\label{A.9.71}
Let $f:X\rightarrow S$ be a log \'etale morphism of fs log schemes, and let $i:S\rightarrow X$ be a section of $f$. Then $i$ is an open immersion.
\end{lem}
\begin{proof}
According to the assumptions, we can form the commutative diagram of fs log schemes
\[
\begin{tikzcd}
S\arrow[r,"i"]\arrow[d,"i"]&X\arrow[d,"i'"]\arrow[r,"f"]&S\arrow[d,"i"]\\
X\arrow[r,"a"]&X\times_S X\arrow[r,"p_2"]&X,
\end{tikzcd}\]
where
\begin{enumerate}[(i)]
\item $a$ denotes the diagonal morphism.
\item $p_2$ denotes the projection on the second factor.
\item Each square is cartesian.
\end{enumerate}
\vspace{0.1in}

It suffices to show that $a:X\rightarrow X\times_S X$ is an open immersion. 
Since the diagonal morphism
\[
\underline{X}
\rightarrow
\underline{X}\times_{\underline{S}}\underline{X}
\]
is radicial, we reduce to showing that $a$ is strict \'etale by \cite[Th\'eor\`eme IV.17.9.1]{EGA}.  
As in \cite[Corollaire IV.17.3.5]{EGA}, the morphism $a$ is log \'etale.  
Thus it remains to show $a$ is strict. 
We divide the proof into three steps.
\vspace{0.1in}

(I) {\it Locality on $S$.} 
Let $g:S'\rightarrow S$ be a strict \'etale cover of fs log schemes, and set $X'=X\times_S S'$. 
Then the commutative cartesian diagram of fs log schemes
\[
\begin{tikzcd}
X'\arrow[r]\arrow[d]&X'\times_{S'}X'\arrow[d]\\
X\arrow[r]&X\times_S X
\end{tikzcd}\]
shows the question is strict \'etale local on $S$.
\vspace{0.1in}
 
(II) {\it Locality on $X$.} Let $h:X'\rightarrow X$ be a strict \'etale cover of fs log schemes. 
Then we have the commutative diagram of fs log schemes
\[
\begin{tikzcd}
&X'\times_S X'\arrow[d,"h''"]\\
X'\arrow[d,"h"']\arrow[r,"a'"]\arrow[ru,"a''"]&X\times_S X'\arrow[d,"h'"]\\
X\arrow[r,"a"]&X\times_S X,
\end{tikzcd}\]
where
\begin{enumerate}[(i)]
\item The lower square is cartesian.
\item $a''$ denotes the diagonal morphism.
\item $h':={\rm id}\times h$ and $h'':=h\times {\rm id}$. 
\end{enumerate}
\vspace{0.1in}
 
Assume that $a''$ is a strict morphism. 
Then $a'$ is also strict since $h''$ is strict, while $a$ is also strict since $h$ is a strict \'etale cover. 
Conversely, if $a$ is a strict morphism, then $a'$ is strict and therefore also $a''$ is strict. 
Thus the question is strict \'etale local on $X$.
\vspace{0.1in}

(III) {\it Final step of the proof.}  By \cite[Theorem IV.3.3.1]{Ogu}, Zariski locally on $X$ and $S$ we have an fs chart $\theta:P\rightarrow Q$ of $f$ such that the following holds.
\begin{enumerate}[(i)]
\item $\theta$ is injective, and the cokernel of $\theta^{\rm gp}$ is finite,
\item The induced morphism $X\rightarrow S\times_{\mathbb{A}_P}\mathbb{A}_Q$ is strict \'etale.
\end{enumerate}
Hence by parts (I) and (II), 
we may assume that $(X,S)=(\mathbb{A}_Q,\mathbb{A}_P)$. 
Then it suffices to show that the diagonal homomorphism
\[
\mathbb{A}_Q
\rightarrow 
\mathbb{A}_{Q}\oplus_{\mathbb{A}_P}\mathbb{A}_Q
\]
is strict. 
In effect, it suffices to show $a\oplus (-a)\in (Q\oplus_P Q)^*$ for every $a\in Q$. 
Choose $n\in \mathbb{N}^+$ such that $na\in P^{\rm gp}$. 
Because the summation homomorphism
\[
P^{\rm gp}\oplus_{P^{\rm gp}}P^{\rm gp}
\rightarrow 
P^{\rm gp}
\]
is an isomorphism, 
the two elements $(na)\oplus 0$ and $0\oplus (na)$ of $Q\oplus_P Q$ are equal. 
Thus we have $n(a\oplus (-a))=0$. 
Since $Q\oplus_P Q$ is an fs monoid, it follows that $a\oplus (-a)\in Q\oplus_P Q$. 
This means that $a\oplus (-a)\in (Q\oplus_P Q)^*$ since $n(a\oplus (-a))=0$.
\end{proof}

\begin{lem}[{\cite{ParThesis}}]
\label{A.9.19}
Let $f:X\rightarrow S$ be a log \'etale monomorphism of fs log schemes, and let $P$ be an fs chart of $S$. 
Then Zariski locally on $X$, there exists a chart $\theta:P\rightarrow Q$ of $f$ with the following properties.
\begin{enumerate}
\item[{\rm (i)}] The induced morphism $X\rightarrow S\times_{\mathbb{A}_P} \mathbb{A}_Q$ is an open immersion,
\item[{\rm (ii)}] $\theta^{\rm gp}:P^{\rm gp}\rightarrow Q^{\rm gp}$ is an isomorphism.
\end{enumerate}
\end{lem}
\begin{proof}
The question is Zariski local on $X$. 
Let $x$ be a point on $X$. 
By appealing to \cite[Theorem IV.3.3.1]{Ogu}, we may assume there exists an fs chart $\theta':P\rightarrow Q'$ with the following properties.
\begin{enumerate}
\item[(i)'] The induced morphism $X\rightarrow S\times_{\A_P}\A_{Q'}$ is strict \'etale.
\item[(ii)'] $\theta'$ is injective, and the cokernel of $\theta'^{\rm gp}$ is finite of order invertibile on $\cO_S$.
\item[(iii)'] The chart $Q'$ of $X$ is exact at $x$.
\end{enumerate}

Note that $Q:=P^{\rm gp}\cap Q'$ is an fs monoid by Gordon's lemma \cite[Theorem I.2.3.19]{Ogu}.
Then the induced homomorphism $P^{\rm gp}\rightarrow Q^{\rm gp}$ is an isomorphism, i.e., we have (ii).
Thus the projection 
\[
S\times_{\A_P}\A_{Q}\rightarrow S
\]
is a log \'etale monomorphism.
Consider $f$ as the composition $X\rightarrow S\times_{\A_P}\A_Q\rightarrow S$.
Then the first arrow is a monomorphism since $f$ is a monomorphism. 
According to \cite[Remark IV.3.1.2]{Ogu},
the first arrow is also log \'etale.
Since the inclusion $Q\rightarrow Q'$ is Kummer, see Definition \ref{KummerDef}, 
replacing $X\rightarrow S$ by $X\rightarrow S\times_{\mathbb{A}_P}\mathbb{A}_Q$, 
we may also assume that $\theta'$ is Kummer.
\vspace{0.1in}

Let $X_1$ (resp.\ $S_1$) be the strict closed subscheme of $X$ (resp.\ $S$) whose underlying scheme is $\{x\}$ (resp.\ $\{f(x)\}$) with the reduced scheme structure. 
Since $f$ is a monomorphism, the induced pullback morphism $X_1\rightarrow S_1$ is also a monomorphism. 
Since this is a morphism of log points, it can be written as
\[
\Spec{Q'\oplus k''^*\rightarrow k''}\rightarrow \Spec{P\oplus k'^*\rightarrow k'},
\] 
where $\Spec {k''} =\underline{X_1}$ and $\Spec {k'}=\underline{S_1}$.
\vspace{0.1in}

Suppose for contradiction that $\theta'$ is not an isomorphism.
Then the fiber product $X_1\times_{S_1}X_1$ has more than one point, 
and hence the diagonal $X_1\rightarrow X_1\times_{S_1}X_1$ is not an isomorphism.
This contradicts the fact that $X_1\rightarrow S_1$ is a monomorphism.
Thus $\theta'$ is an isomorphism. 
\vspace{0.1in}

Then by (i)',
$f$ is strict \'etale.
Since $f$ is a monomorphism,
$f$ is an open immersion by \cite[Th\'eor\`eme IV.17.9.1]{EGA}.
Thus we have (i).
\end{proof}

\begin{lem}\label{A.9.67}
Suppose $P$ is an fs monoid such that $P^\gp$ is torsion free and $u:\Sigma\rightarrow \Spec P$ is subdivision of fans.
Then the morphism $\A_u:\A_\Sigma\rightarrow \A_P$ is universally surjective.
\end{lem}
\begin{proof}
Let $p:S\rightarrow \A_P$ be a morphism of fs log schemes, and set $X:=S\times_{\A_P}\A_\Sigma$. 
We need to show that the projection morphism $f:X\rightarrow S$ is surjective. 
This question is Zariski local on $S$, so we may assume that $p$ has a factorization 
\[
S\rightarrow \A_Q\stackrel{\theta}\rightarrow \A_P,
\] 
where the first morphism is strict and $\theta:P\rightarrow Q$ is a homomorphism of fs monoids. 
The projection 
$$
\Sigma\times_{\Spec P}\Spec Q\rightarrow \Spec Q
$$ 
is also a subdivision of fans due to Lemma \ref{Fiberproduct_fans}.
Hence, by replacing $P$ with $Q$, we may assume that $p$ is strict.
In this case $\underline{f}$ is a pullback of $\underline{\A_u}$ by Lemma \ref{Fiber.product.strict.morphism}.
Since $\underline{\A_u}$ is surjective, and any surjective morphism in the category of schemes is universally surjective, $\underline{f}$ is also surjective.
\end{proof}

\begin{prop}[{cf.\ \cite{ParThesis}}]
\label{A.9.21}
Let $(f_i:X_i\rightarrow S)_{i\in I}$ be a finite family of log \'etale monomorphisms where $X_i$ is quasi-compact for every $i\in I$.
Suppose that there is a strict morphism $S\rightarrow \A_\Sigma$ where $\Sigma$ is a fan.
Then there exists a subdivision of fans 
$$
u:\Sigma'\rightarrow \Sigma
$$ 
such that for each $i\in I$, 
the pullback
\[
f_i'
\colon 
X_i\times_{\mathbb{A}_\Sigma}\mathbb{A}_{\Sigma'}\rightarrow S\times_{\mathbb{A}_\Sigma}\mathbb{A}_{\Sigma'}
\]
is an open immersion.  
If $f_i$ is surjective and proper, then $f_i'$ is an isomorphism.
\end{prop}
\begin{proof}
Let us start with a log \'etale monomorphism $f:X\rightarrow S$, where $X$ is quasi-compact.
Suppose that $\Sigma$ is in a lattice $N$.
Choose fs submonoids $\{P_j\}_{j\in J}$ of $\Z^r$ such that the dual monoids $P_j^{\vee}$ correspond to the maximal cones of $\Sigma$.
Then 
$$
S_j:=X\times_{\A_\Sigma}\A_{P_j}
$$
affords an fs chart $P_j$, and $\{S_j\rightarrow S\}_{j\in J}$ is a finite Zariski cover.
By Lemma \ref{A.9.19}, there exists a finite Zariski cover $\{U_{jk}\rightarrow S_j\}_{k\in K_j}$ such that for each $k$, 
there is a homomorphism $\theta_k:P_j\rightarrow Q_{jk}$ of fs monoids with the properties that $\theta_k^{\rm gp}$ is an isomorphism and the induced morphism 
\[
U_{jk}\rightarrow S_j\times_{\mathbb{A}_{P_j}}\mathbb{A}_{Q_{jk}}
\]
is an open immersion.
\vspace{0.1in}

Lemma \ref{Compactification_fan} yields a fan $\Theta_{jk}$ in $N$ such that $Q_{jk}^\vee\in \Theta_{jk}$ and $\lvert \Theta_{jk}\rvert=N$.
Let $\Sigma_0$ be the fan in $N$ whose only cone is $N$, 
and let $\Sigma'$ be the fiber product of $\Sigma$ and all $\Theta_{jk}$ over $\Sigma_0$.
Owing to Lemma \ref{Fiberproduct_fans} we get that $\Sigma'$ is a subdivision of $\Sigma$.
Moreover, for each cone $\sigma$ of $\Sigma'$ there is $j\in J$ and $k\in K_j$ such that $\sigma\subset Q_{jk}^{\vee}$.
This implies that the induced morphism 
\[
u_j':U_{jk}\times_{\mathbb{A}_\Sigma}\mathbb{A}_{\Sigma'}\rightarrow S_j\times_{\mathbb{A}_\Sigma}\mathbb{A}_{\Sigma'}
\] 
is an open immersion. 
In particular, $u_j'$ is strict \'etale. 
Thus the pullback 
\[ 
f':X\times_{\A_\Sigma}\A_{\Sigma'}\rightarrow S\times_{\A_\Sigma}\A_{\Sigma'}
\]
is strict \'etale. 
Since $f'$ is a monomorphism, it follows that $f'$ is an open immersion by \cite[Th\'eor\`eme IV.17.9.1]{EGA}.
\vspace{0.1in}

Assuming that $f$ is surjective and proper, we consider the cartesian square
\begin{equation}
\begin{tikzcd}\label{A.9.21.1}
X\times_{\A_\Sigma}\A_{\Sigma'}\arrow[r,"f'"]\arrow[d,"g'"']&S\times_{\A_\Sigma}\A_{\Sigma'}\arrow[d,"g"]\\
X\arrow[r,"f"]&S.
\end{tikzcd}\end{equation}
As a result of Lemma \ref{A.9.67}, we get that $\A_u$ is universally surjective, and hence $g$, and $g'$ are universally surjective. 
Thus $(fg')^{-1}(s)$ is nonempty for every point $s\in S$. 
The induced open immersion $(fg')^{-1}(s)\rightarrow g^{-1}(s)$ is proper, 
and it is an isomorphism according to Zariski's main theorem since $g^{-1}(s)$ is connected. 
Thus $f'$ is surjective, and hence an isomorphism.
\vspace{0.1in}

In the general case, 
choose a subdivision of fans $\Sigma_i'\rightarrow \Sigma$ for each $i\in I$ such that the pullback 
\[
X_i\times_{\A_\Sigma}\A_{\Sigma_i'}\rightarrow S\times_{\A_\Sigma}\A_{\Sigma_i'}
\] 
is an open immersion. 
Letting $\Sigma'$ be the fiber product $\Sigma_{i_1}'\times_{\Sigma}\cdots \times_{\Sigma}\Sigma_{i_r}'$
where $I=\{i_1,\ldots,i_r\}$,
we have that the pullback 
$$
X_i\times_{\A_\Sigma}\A_{\Sigma'}\rightarrow S\times_{\A_\Sigma}\A_{\Sigma'}
$$ 
is an open immersion for any $i$.
To conclude we note that the induced morphism $\Sigma'\rightarrow \Sigma$ is a subdivision of fans owing to Lemma \ref{Fiberproduct_fans}.
\end{proof}

\begin{prop}
\label{A.9.22}
Suppose $f:X\rightarrow S$ is a surjective proper log \'etale monomorphism of fs log schemes. 
Then any pullback along $f$ is a surjective proper log \'etale monomorphism. 
In particular, $f$ is universally surjective.
\end{prop}
\begin{proof}
The question is Zariski local on $S$, so we may assume that $S$ is quasi-compact and has an fs chart $P$. 
Then $X$ is quasi-compact since $f$ is proper. 
By Proposition \ref{A.9.21}, there exists a subdivision of fans $u:M\rightarrow \Spec P$ such that the morphism $f'$ in \eqref{A.9.21.1} is an isomorphism.
In the proof of Proposition \ref{A.9.21}, we have checked that the morphism $g$ in \eqref{A.9.21.1} is universally surjective. 
Thus $f$ is universally surjective. 
To conclude, we use that a pullback along a proper log \'etale monomorphism is a proper log \'etale monomorphism.
\end{proof}

\begin{df}
\label{A.9.70}
Recall from \cite{FKato} that a morphism $f:X\rightarrow S$ of fs log schemes is called a {\it log modification} if Zariski locally on $S$, 
there exists a log blow-up $g:Y\rightarrow X$ such that the composition $fg$ is also a log blow-up.
\footnote{In \cite{FKato}, the condition is \'etale local on $S$. Since we are only interested in fs log schemes with Zariski log structures, we may restrict to the Zariski case.}
\end{df}
\begin{exm}
\label{A.9.76}
As observed in \cite{FKato}, a morphism of fs log schemes induced by a subdivision of toric fans in the sense of \cite[\S I.1.9]{Ogu} is an example of a log modification, 
see also \cite[pp.\ 39-40]{TOda}.
\end{exm}

\begin{prop}
\label{A.9.75}
Let $f:X\rightarrow S$ be a morphism of fs log schemes. 
Then $f$ is a surjective proper log \'etale monomorphism if and only if $f$ is a log modification. \index{morphism of log schemes!log \'etale} \index{log modification}
\end{prop}
\begin{proof}
Assume that $f$ is a log modification. 
The question is Zariski local on $S$, so we may assume that there exists a log blow-up $g\colon Y\rightarrow X$ such that $fg$ is also a log blow-up. 
Since $g$ and $fg$ are surjective proper log \'etale monomorphisms as observed in Example \ref{A.5.15}, $f$ is a surjective proper log \'etale morphism. 
Consider the induced commutative diagram  of fs log schemes
\[
\begin{tikzcd}
Y\times_X Y\arrow[d]\arrow[r]&Y\times_S Y\arrow[d]\\
X\arrow[r,"\Delta"]&X\times_S X.
\end{tikzcd}\]
Since $g$ and $fg$ are monomorphisms, the diagonal morphisms $Y\rightarrow Y\times_S Y$ and $Y\rightarrow Y\times_X Y$ are isomorphisms. 
Thus the upper horizontal morphism is an isomorphism. 
The right vertical morphism is surjective since $g$ is universally surjective. 
Thus $\Delta$ is surjective. 
By Lemma \ref{A.9.71}, $\Delta$ is an open immersion, so $\Delta$ is an isomorphism. 
Thus $f$ is a monomorphism, and hence a surjective proper log \'etale monomorphism.
\vspace{0.1in}

Conversely, assume that $f$ is a surjective proper log \'etale monomorphism. 
The question is Zariski local on $S$, so we may assume that $S$ is quasi-compact and has a neat fs chart $P$. 
Then $X$ is quasi-compact since $f$ is proper. 
By Proposition \ref{A.9.21}, there exists a subdivision of fans $\Sigma\rightarrow \Spec P$ such that the pullback 
\[
X\times_{\A_P}\A_\Sigma\rightarrow S\times_{\A_P}\A_\Sigma
\] 
is an isomorphism. 
As observed in Example \ref{A.9.76}, there exists a log blow-up $T\rightarrow \A_\Sigma$ such that the composition $T\rightarrow \A_P$ is also a log blow-up. 
Then the morphisms
\[
X\times_{\A_P}T\rightarrow X \text{ and } X\times_{\A_P}T\rightarrow S
\] 
are log blow-ups, so $f$ is a log modification.
\end{proof}

\begin{prop}
\label{A.9.77}
Any composition of log modifications is a log modification.
Any pullback of a log modification is a log modification.
\end{prop}
\begin{proof}
See \cite[Lemma 3.15(2),(3)]{FKato}.
Alternatively, one can use Proposition \ref{A.9.75}.
\end{proof}

\begin{lem}
\label{A.9.78}
Let $f:X\rightarrow S$ and $g:Y\rightarrow S$ be morphisms of fs log schemes.
Suppose that $g$ is a log modification.
Then there are at most one morphism $h:X\rightarrow Y$ such that $gh=f$.
Moreover, if $f$ is a log modification, then $h$ is a log modification.
\end{lem}
\begin{proof}
Suppose that such an $h$ exists.
The graph morphism $u:X\rightarrow X\times_S Y$ is a pullback of the diagonal morphism $d:Y\rightarrow Y\times_S Y$, and $d$ is an isomorphism since $Y$ is a log modification.
Thus $u$ is an isomorphism.
This means that $h$ is isomorphic to the projection $X\times_S Y\rightarrow Y$.
If $f$ is a log modification, then the projection is also a modification by Proposition \ref{A.9.77}.
\end{proof}

\begin{df}
\label{A.9.79}
Let $X$ be an fs log scheme.
We denote by $X_{div}$ the category of log modifications over $X$, and we denote by $X_{div}^{Sm}$ the full subcategory of $X_{div}$ consisting of log modifications $Y\rightarrow X$ such that $Y\in SmlSm/k$.
\end{df}

\begin{prop}
\label{A.9.83}
Every exact log \'etale monomorphism $f:X\rightarrow S$ of fs log schemes is an open immersion.
\end{prop}
\begin{proof}
The question is Zariski local on $S$ and $X$, so we may assume that $S$ has s neat fs chart $P$ at $s\in S$.
Moreover, due to Lemma \ref{A.9.19}, 
we may assume there exists a chart $\theta:P\rightarrow Q$ of $f$ such that $X\rightarrow S\times_{\A_P}\A_Q$ is an open immersion and 
$\theta^{\rm gp}:P^{\rm gp}\rightarrow Q^{\rm gp}$ is an isomorphism.
\vspace{0.1in}

Let $x$ be a point of $X$ mapping to $s$.
Then $\overline{\cM}_{X,x}\cong Q/F$ for some face $F$ of $Q$.
By the assumption that $f$ is exact, we see that
\[
P\cong \overline{\cM}_{S,s}\rightarrow\overline{\cM}_{X,x}\cong Q/F
\]
is exact.
Thus $P\rightarrow Q_F$ is exact owing to \cite[Proposition I.4.2.1(3)]{Ogu}, i.e., 
$$
P\cong P^{\rm gp}\times_{Q^{\rm gp}}Q_F.
$$
Since $\theta^{\rm gp}$ is an isomorphism, this implies $P\cong Q_F$.
This shows that the morphism $f$ is strict.
Thus $f$ is a strict \'etale monomorphism, and hence an isomorphism according to \cite[Th\'eor\`eme IV.17.9.1]{EGA}.
\end{proof}

\begin{prop}
\label{A.9.80}
Suppose $X$ is an fs log scheme.
Then $X_{div}$ is a filtered category.
\end{prop}
\begin{proof}
For any $Y,Y'\in X_{div}$, we have that $Y\times_X X'\in X_{div}$ owing to Proposition \ref{A.9.77}.
Thus $X_{div}$ is connected.
Owing to Lemma \ref{A.9.78} there exists at most one morphism $Y\rightarrow Y'$ of fs log schemes over $X$.
Thus $X_{div}$ is filtered.
\end{proof}

\begin{prop}
\label{A.9.81}
For $X\in lSm/k$, the category $X_{div}^{Sm}$ is cofinal in $X_{div}$.
\end{prop}
\begin{proof}
Follows from Proposition \ref{A.3.19}.
\end{proof}

\begin{prop}
\label{A.9.82}
For $X\in lSm/k$, the category $X_{div}^{Sm}$ is a filtered category.
\end{prop}
\begin{proof}
Follows from Propositions \ref{A.9.80} and \ref{A.9.81}.
\end{proof}

\subsection{Small Kummer \'etale and small log \'etale sites}
In this subsection we review the small Kummer \'etale site $X_{k\acute{e}t}$ and the small log \'etale site $X_{l\acute{e}t}$ for a saturated log scheme $X$.
We also prove some basic properties.

\begin{df}
\label{ket.1}
A saturated log scheme $X$ is \emph{log strictly local}\index{log scheme!log strictly local} (resp.\ a \emph{log geometric point})\index{log geometric point} if $\underline{X}$ is the spectrum of a strictly henselian local ring 
(resp.\ a separably closed field), 
and for every integer $n>0$ prime to the characteristic of $k$ the multiplication by $n$ morphism on $\overline{\cM}(X)$ is bijective.
\vspace{0.1in}

A saturated log scheme $X$ is \emph{log strictly henselian} if $X$ is a finite disjoint union of log strictly local schemes.
\end{df}

\begin{df}
For a saturated log scheme $X$, let $X_{k\acute{e}t}$ 
(resp.\ $X_{l\acute{e}t}$) denote the category of Kummer \'etale (resp.\ log \'etale) morphisms of saturated log schemes whose target is $X$.
The topology on $X_{k\acute{e}t}$ (resp.\ $X_{l\acute{e}t}$) is generated by the pretopology with coverings comprised of Kummer \'etale (resp.\ log \'etale) morphisms 
$\{U_i\rightarrow U\}_{i\in I}$ for which $\amalg_{i\in I} U_i\rightarrow U$ is universally surjective in the category of \emph{saturated} log schemes.
\end{df}

\begin{rmk}
Suppose that $X$ is an fs log scheme.
For $X_{k\acute{e}t}$, a family of morphisms $\{U_i\rightarrow U\}_{i\in I}$ is a covering if and only if $\amalg_{i\in I} U_i\rightarrow U$ is surjective, see e.g., \cite[Proposition 3.2]{MR3658728}.
In particular, a family of morphisms $\{U_i\rightarrow U\}_{i\in I}$ is a covering if and only if $\amalg_{i\in I} U_i\rightarrow U$ is universally surjective in the category of \emph{fs} log schemes.
\vspace{0.1in}

However, we need to take the category of \emph{saturated} log schemes in the definition of coverings in $X_{l\acute{e}t}$.
For example, 
in the proof \cite[Proposition 3.14(1)]{MR3658728}, 
a surjectivity of $(\amalg_{i\in I}U_i)\times_U V\rightarrow V$ is used for some saturated but non fs log scheme over $U$.
\end{rmk}

\begin{lem}
\label{ket.2}
Let $X$ be a log strictly henselian scheme.
Then every saturated log scheme $Y$ that is finite and strict over $X$ is also strictly henselian.
\end{lem}
\begin{proof}
Since $\underline{X}$ is strictly henselian, $\underline{Y}$ is strictly henselian too.
Let $y$ be a closed point of $Y$ with image $x$ in $X$.
Then $\overline{\cM}_{Y,y}\cong \overline{\cM}_{X,x}$ since $Y$ is strict over $X$.
Thus for every integer $n>0$ prime to the characteristic of $k$ the multiplication by $n$ on $\overline{\cM}_{Y,y}$ is bijective.
\end{proof}

For an integral monoid $P$ and $n\in \N^+$, we let $P^{1/n}$ denote the submonoid of $P^{\rm gp}\otimes \Q$ given by
\begin{equation}
\label{ket.0.1}
P^{1/n}:=\{p\in P^{\rm gp}\otimes \Q:np\in P\}.
\end{equation}

\begin{lem}
\label{ket.3}
Let $S$ be an fs log scheme with an fs chart $P$ whose underlying scheme $\underline{S}$ is strictly local.
Then every Kummer \'etale cover $f:X\rightarrow S$ admits a refinement of the form of a projection
\[
S\times_{\A_P} \A_{P^{1/n}}\rightarrow S
\]
where $n\in N^+$ is invertible in $T$.
\end{lem}
\begin{proof}
Any open subscheme of $X$ is again a Kummer \'etale cover since $\underline{S}$ is strictly local.
Hence, 
owing to Theorem \ref{KatoStrThm}, 
we may assume that $f$ admits an fs chart $\theta:P\rightarrow Q$ with the following properties.
\begin{enumerate}
\item[(i)] A naturally induced strict \'etale morphism
\[
p:X\rightarrow S\times_{\A_P}\A_Q.
\]
\item[(ii)] The kernel and cokernel of $\theta^{\rm gp}:P^{\rm gp}\rightarrow Q^{\rm gp}$ is finite, and the orders are invertible in $S$.
\end{enumerate}
\vspace{0.1in}

Since $f$ is Kummer, (ii) implies that $Q\subset P^{1/n}$ for some $n$ invertible in $S$.
Replacing $X$ by $X\times_{\A_Q}\A_{P^{1/n}}$, we may assume that $Q=P^{1/n}$.
Then $k[Q]$ is a locally free $k[P]$-module so that $\A_Q$ is finite over $\A_P$.
Thus $S\times_{\A_P}\A_Q$ is finite over $\A_P$.
This shows that the underlying scheme of $S\times_{\A_P}\A_Q$ is strictly henselian since $\underline{S}$ is strictly henselian, and $p$ admits a section.
It follows that $S\times_{\A_P}\A_Q$ is a refinement of $X$.
\end{proof}

\begin{lem}
\label{ket.4}
Suppose that $S$ is an fs log scheme.
Let $f:X\rightarrow S$ be a morphism of saturated log schemes and let $g:S'\rightarrow S$ be a Kummer \'etale morphism of fs log schemes.
If $X$ is log strictly henselian, then the projection $X\times_S S'\rightarrow X$ has a section.
\end{lem}
\begin{proof}
We may assume that $X$ is local.
Replacing $\underline{S}$ by its strict henselization at the image of the closed point of $X$, we may assume that $\underline{S}$ is strictly local.
Then since $S$ is local, $f$ admits an fs chart $\theta:P\rightarrow Q$.
Owing to Lemma \ref{ket.3} we may assume that $S'=S\times_{\A_P}\A_{P^{1/n}}$ where $n>0$ is an integer invertible in $k$.
There is a naturally induced commutative diagram of saturated monoids
\[
\begin{tikzcd}
P^{1/n}\arrow[r]\arrow[d]&P\arrow[d,"\theta"]
\\
Q^{1/n}\arrow[r]&Q.
\end{tikzcd}
\]
By the assumption on $Q$ the upper horizontal morphism is an isomorphism.
This means that the homomorphism $Q\rightarrow Q\oplus_P P^{1/n}$ has a retraction, so the projection $X\times_S S'\rightarrow X$ has a section.
\end{proof}

\subsection{Sharpened fans}
In this subsection, we review the notion of Kato's fan in \cite{MR1296725}, which we call a sharpened fan, and we study basic properties.
A sharpened fan contains less information and is more flexible than an fs monoscheme.

\begin{df}
For a monoidal space $\Sigma=(\Sigma,\cM_\Sigma)$, we set
\[
\overline{\Sigma}:=(\Sigma,\overline{\cM}_\Sigma).
\]
For a morphism of monoidal spaces $f:\Sigma'\rightarrow \Sigma$, let
\[
\overline{f}:\overline{\Sigma'}\rightarrow \overline{\Sigma}
\]
denote the naturally induced morphism of sharpened fans.
\vspace{0.1in}

For an fs monoid $P$, we set 
\[
\overline{\rm Spec}(P):=\overline{\Spec{P}}.
\]
Note that $\overline{\rm Spec}(P)\cong \overline{\rm Spec}(\overline{P})$.
\end{df}

\begin{df}
\label{Fan.1}
A \emph{sharpened fan}\index{sharpened fan} $\Delta$ is a monoidal space such that there exists an open cover $\{U_i\}$ with an isomorphism
\[
U_i\cong \overline{\rm Spec} (P_i),
\]
where $P_i$ is a sharp fs monoid.
A morphism of sharpened fans is a morphism of monoidal spaces.
A \emph{cone} of $\Delta$ is an open subset of $\Delta$ that is isomorphic to $\overline{\rm Spec} (P)$ for some sharp fs monoid $P$.
Note that there is a one-to-one correspondence between cones of $\Delta$ and points of $\Delta$.
\vspace{0.1in}

We say that $\Delta$ is \emph{affine}\index{sharpened fan!affine} if there is an isomorphism $\Delta\cong \uSpec(P)$ for some sharp fs monoid $P$.
We say that $\Delta$ is \emph{smooth}\index{sharpened fan!smooth} if every cone of $\Delta$ is isomorphic to $\overline{\rm Spec} (\N^r)$ for some $r$.
\vspace{0.1in}

For any property $\bP$ of maps of topological spaces, we say that a morphism of sharpened fan satisfies $\bP$ if the underlying map of topological spaces satisfies $\bP$.
\end{df}

\begin{rmk}
\label{Fan.7}
The reader may want to compare the above with the related notions of a conical polyhedral complex with an integral structure, a fan satisfying ($S_{fan}$), and an s-fan, 
see \cite[Definitions II.5 and II.6]{MR0335518}, \cite[9.4]{MR1296725}, and \cite[Remark II.1.9.4]{Ogu}.
\end{rmk}

In the following, we establish some basic properties for sharpened fans.
We refer to \cite[Sections II.1.2, II.1.3]{Ogu} for the monoscheme versions.

\begin{prop}
\label{Fan.21}
Let $P$ be a sharp fs monoid, and let $\Delta$ be a sharpened fan.
Then there is a canonical bijection
\[
\alpha:\hom(\Delta,\uSpec(P))\xrightarrow{\cong} \hom(P,\Gamma(\Delta,\cM_{\Delta})).
\]
\end{prop}
\begin{proof}
We define $\alpha$ by the morphism sending $f:\Delta\rightarrow \uSpec(P)$ to its associated homomorphism on the global sections
\[
P\cong \Gamma(\uSpec(P),\cM_{\uSpec(P)})\rightarrow \Gamma(\Delta,\cM_{\Delta}).
\]

Suppose that $\Delta=\uSpec(Q)$ for some sharp monoid $Q$.
Then $\alpha$ has an inverse
\[
\hom(P,Q)\rightarrow \hom(\uSpec(Q),\uSpec(P))
\]
sending $g:P\rightarrow Q$ to $\uSpec(g):\uSpec(Q)\rightarrow \uSpec(P)$.
Thus $\alpha$ is bijective in this case.
\vspace{0.1in}

In the general case, let $\{\sigma_i\}_{i\in I}$ be the set of cones of $\Delta$, and let $\{\sigma_{ijk}\}_{k\in I_{ij}}$ be the set of cones of $\sigma_i\cap \sigma_j$ where $i,j\in I$.
There is a naturally induced diagram
\[
\begin{tikzcd}[column sep=small]
\hom(\Delta,\uSpec(P))\arrow[d]\arrow[r]&
\amalg_{i}\hom(\sigma_i,\uSpec(P))\arrow[d]\arrow[r,shift left=0.5ex]\arrow[r,shift right=0.5ex]&
\amalg_{i,j,k}\hom(\sigma_{ijk},\uSpec(P))\arrow[d]
\\
\hom(P,\Gamma(\Delta,\cM_{\Delta}))\arrow[r]&
\amalg_{i}\hom(P,\Gamma(\sigma_i,\cM_{\sigma_i}))\arrow[r,shift left=0.5ex]\arrow[r,shift right=0.5ex]&
\amalg_{i,j,k}\hom(P,\Gamma(\sigma_i,\cM_{\sigma_i})).
\end{tikzcd}
\]
Here the horizontal diagrams are equalizers.
The middle and right vertical morphisms are bijections by the above.
Thus the left vertical morphism is a bijection.
\end{proof}

On the category of sharpened fans, the \emph{Zariski topology} is the smallest Grothendieck topology containing all open covers.
The Yoneda lemma allows us to identify any sharpened fan with its representable sheaf.
\vspace{0.1in}

\begin{lem}
\label{Fan.23}
Let $\cF$ be a sheaf of sets on the category of sharpened fans, 
and suppose that $\{\Delta_i\rightarrow \cF\}_{i\in I}$ is a set of morphisms from sharpened fans $\Delta_i$ satisfying the following conditions for every morphism 
$\Psi\rightarrow \cF$, where $\Psi$ is a sharpened fan.
\begin{enumerate}
\item[{\rm (i)}] For every $i\in I$, the sheaf $\Delta_i\times_{\cF}\Psi$ is representable, and the projection
\[
\Delta_i\times_{\cF}\Psi\rightarrow \Psi
\]
is an open immersion.
\item[{\rm (ii)}] The set $\{\Delta_i\times_{\cF}\Psi\rightarrow \Psi\}_{i\in I}$ is a Zariski cover.
\end{enumerate}
Then $\cF$ is representable.
\end{lem}
\begin{proof}
For $\Delta_{ij}:=\Delta_i\times_{\cF}\Delta_j$ for $i,j\in I$, 
the projections $\Delta_{ij}\rightarrow \Delta_i,\Delta_j$ are open immersions by the condition (i).
Hence we can glue the morphisms $\Delta_i\rightarrow \cF$ for $i\in I$ to obtain a morphism
\[
f\colon\Delta\rightarrow \cF
\]
from a sharpened fan $\Delta$.
\vspace{0.1in}

It remains to show that $f$ is an isomorphism.
For this purpose, we need to show that for every morphism $\Psi\rightarrow \cF$ from a sharpened fan, the projection
\[
f_\Psi\colon\Delta\times_{\cF} \Psi \rightarrow \Psi
\]
is an isomorphism.
The sharpened fan $\Delta\times_\cF\Psi$ is again the gluing of $\Delta_i\times_\cF\Psi$ for $i\in I$ along $\Delta_{ij}\times_{\cF}\Psi$ for $i,j\in I$.
The condition (ii) shows that $f_\Psi$ is an isomorphism.
\end{proof}

\begin{lem}
\label{Fan.24}
Let $P\rightarrow Q$ and $P\rightarrow P'$ be homomorphisms of sharp, saturated monoids.
Then the fiber product
\[
\uSpec(Q)\times_{\uSpec(P)}\uSpec(P')
\]
is representable by $\uSpec(\overline{Q'})$, where $Q'$ is the amalgamated sum $Q\oplus_P P'$ in the category of saturated monoids.
\end{lem}
\begin{proof}
Since $Q'$ is the amalgamated sum, for every saturated monoid $M$, there is a bijection
\[
\hom(Q',M)\xrightarrow{\cong} \hom(Q,M)\times_{\hom(P,M)}\hom(Q,M).
\]
If $M$ is sharp, 
then $\hom(Q',M)\cong \hom(\overline{Q'},M)$, and hence there is a bijection
\[
\hom(\overline{Q'},M)\xrightarrow{\cong} \hom(Q,M)\times_{\hom(P,M)}\hom(Q,M).
\]
Specializing the above to $M=\Gamma(\Delta,\cM_{\Delta})$, where $\Delta$ is any sharpened fan, 
and applying Proposition \ref{Fan.21} concludes the proof.
\end{proof}

\begin{prop}
\label{Fan.25}
The category of sharpened fans has fiber products.\index{sharpened fan!fiber product}
\end{prop}
\begin{proof}
Let $f:\Psi\rightarrow \Delta$ and $g:\Delta'\rightarrow \Delta$ be morphisms of sharpened fans.
Take the fiber product $\Psi':=\Psi\times_{\Delta}\Delta'$ in the category of sheaves.
We need to show that $\Psi'$ is representable.
Let $\{\sigma_i\}_{i\in I}$ be the set of cones of $\Delta$.
For every $i\in I$, let
\[
\{\tau_{ij}\}_{j\in J_i}\text{ and }\{\sigma_{ik}'\}_{k\in K_i}
\]
be the set of cones of $\Psi$ and $\Delta'$.
Owing to Lemma \ref{Fan.24}, we see that 
$$
\tau_{ijk}':=\tau_{ij}\times_{\sigma_i}\sigma_{ik}'
$$ 
is representable for every $i$, $j$, and $k$.
Moreover, for every morphism $\Theta\rightarrow \Psi'$ of sheaves, where $\Theta$ is a sharpened fan, there is an isomorphism
\[
(\amalg_{i,j,k}\tau_{ijk}')\times_{\Psi'}\Theta \cong \amalg_{i,j,k} (\tau_{ij}\times_\Psi \Theta)\times_{\sigma_i\times_\Delta \Theta}(\sigma_{ik}'\times_{\Delta'}\Theta).
\]
The latter is representable and also a cover of $\Theta$.
Thus Lemma \ref{Fan.23} shows that $\Psi'$ is representable.
\end{proof}

\begin{lem}
\label{Fan.29}
Let $\alpha$ be the functor from the category of fs monoschemes to the category of sharpened fans sending any monoscheme $\Sigma$ to $\overline{\Sigma}$.
Then $\alpha$ preserves fiber products.
\end{lem}
\begin{proof}
Suppose that $\theta:P\rightarrow Q$ and $\varphi:P\rightarrow P'$ are homomorphisms of fs monoids.
Let $Q''$ the amalgamated sum $\ol{P'}\oplus_{\ol{P}}\ol{Q}$ in the category of saturated monoids.
Observe that
\[
\Spec{P'}\times_{\Spec{P}}\Spec{Q}\cong \Spec{P'\oplus_P Q},
\]
and
\[
\uSpec(\overline{P})\times_{\uSpec(\overline{P})}\uSpec(\overline{Q})\cong \uSpec(\overline{Q''}).
\]
\vspace{0.1in}

There is a naturally induced commutative diagram
\[
\begin{tikzcd}
P'\oplus_P Q\arrow[r]\arrow[rd]&\overline{P'}\oplus_{\overline{P}}\overline{Q}\arrow[d]
\\
&\overline{P'\oplus_P Q},
\end{tikzcd}
\]
which induces a commutative diagram
\[
\begin{tikzcd}
\overline{P'\oplus_P Q}\arrow[r,"\mu"]\arrow[rd,"{\rm id}"']&\overline{Q''}\arrow[d]
\\
&\overline{P'\oplus_P Q}.
\end{tikzcd}
\]
The homomorphism $\overline{P'\oplus_P Q}\rightarrow \overline{Q''}$ is surjective since the homomorphism $P'\oplus_P Q\rightarrow Q''$ is surjective.
It follows that the homomorphism $\overline{P'\oplus_P Q}\rightarrow \overline{Q''}$ is an isomorphism.
We deduce that there is a canonical isomorphism
\begin{equation}
\label{Fan.29.1}
\alpha(\Spec{P'}\times_{\Spec{P}}\Spec{Q})\rightarrow \alpha(\Spec{P'})\times_{\alpha(\Spec{P})}\alpha(\Spec{Q}).
\end{equation}

Suppose that $\Theta\rightarrow \Sigma$ and $\Sigma'\rightarrow \Sigma$ be morphisms of fs monoschemes.
There is an open cover $\{\sigma_i\}_{i\in I}$ of $\Sigma$, an open cover $\{\tau_{ij}\}_{j\in J_i}$ for every $i\in I$, and an open cover $\{\sigma_{ik}'\}_{k\in K_i}$ for every $i\in I$ such that $\sigma_i$, $\tau_{ij}$, and $\sigma_{ik}'$ are affine for every $i$, $j$, and $k$.
Due to \eqref{Fan.29.1}, we have isomorphisms
\[
\alpha(\tau_{ij}\times_{\sigma_i}\sigma_{ik}')
\cong
\alpha(\tau_{ij})\times_{\alpha(\sigma_i)}\alpha(\sigma_{ik}').
\]
Gluing these isomorphisms concludes the proof.
\end{proof}

Next, we explain how to construct an fs monoscheme from the data of a sharpened fan.
We begin by proving the following lemma.

\begin{lem}
\label{Fan.19}
Let $\Sigma$ be an fs monoscheme.
Then the map
\begin{equation}
\label{Fan.19.1}
\hom(\Spec{\N},\Sigma)\rightarrow \hom(\uSpec(\N),\overline{\Sigma})
\end{equation}
sending $f$ to $\overline{f}$ is bijective.
\end{lem}
\begin{proof}
If $g:\uSpec(\N)\rightarrow \overline{\Sigma}$ is a morphism, then the image of $g$ lies in a cone of $\overline{\Sigma}$ containing the image of the closed point of $\uSpec(\N)$.
Hence we are reduced to the case $\Sigma=\Spec{P}$, where $P$ is an fs monoid.
\vspace{0.1in}

There is a commutative diagram of sets with vertical bijections
\[
\begin{tikzcd}
\hom(\Spec{\N},\Spec{P})\arrow[d]\arrow[r]&
\hom(\uSpec(\N),\uSpec(\overline{P}))\arrow[d]
\\
\hom(P,\N)\arrow[r,"\mu"]&
\hom(\overline{P},\N),
\end{tikzcd}
\]
where $\mu$ sends $g:P\rightarrow \N$ to $\overline{g}:\overline{P}\rightarrow \N$.
Every homomorphism $g:P\rightarrow \N$ maps the unit group $P^*$ maps to $0$.
This shows that $\mu$ is bijective.
\end{proof}

\begin{const}
\label{Fan.10}
Let $\Delta=(\Delta,\cM_{\Delta})$ be a sharpened fan.
Suppose that $N^\vee\rightarrow \cM_{\Delta}^{\rm gp}$ is a surjective morphism of sheaves, 
where $N$ is a lattice with dual $N^\vee$.
Then 
\[
\Delta_N:=(\Delta,N^\vee\times_{\cM_{\Delta}^{\rm gp}}\cM_{\Delta})
\]
is a monoidal space.
Moreover, $\Delta_N$ is toric, and there is an isomorphism $\overline{\Delta_N}\cong \Delta$.
\vspace{0.1in}

Let us check that $\Delta_N$ is an fs monoscheme.
For this purpose, we may assume that 
$$
\Delta=\overline{\rm Spec}(P),
$$ 
where $P$ is a sharp fs monoid.
Let $x$ be a point of $\Delta$ corresponding to a face $F$ of $P$.
Since $\cM_{\Delta,x}\cong P/F$, there is an isomorphism
\[
\cM_{\Delta_N,x}\cong N^\vee\times_{(P/F)^{\rm gp}}(P/F).
\]
Applying \cite[Proposition I.4.2.1(1)]{Ogu} to $P_F$ we obtain the cartesian square
\[
\begin{tikzcd}
P_F\arrow[d]\arrow[r]&
P/F\arrow[d]
\\
(P_F)^\gp\arrow[r]&
(P/F)^\gp.
\end{tikzcd}
\]
It follows that
\[
\cM_{\Delta_N,x}\cong N^\vee\times_{(P_F)^\gp}P_F\cong N^\vee\times_{P^\gp}P_F,
\]
which is a localization of the monoid 
$$
P_N:=N^\vee\times_{P^{\rm gp}}P
$$ 
with respect to the face $F_N:=N^\vee\times_{P^{\rm gp}}F$.
Thus there is an isomorphism
\[
\Delta_N\cong \uSpec(P_N).
\]
This proves the claim.
\vspace{0.1in}

Suppose that $\Delta$ is connected.
Then $\Delta_N$ is also connected so that it has a unique generic point $\xi$, see \cite[p.\ 198]{Ogu}.
It follows that
\[
\hom(\Spec{\Z},\Delta_N)\cong \hom(\cM_{\Delta_N,\xi}^\gp,\Z)\cong \hom(N^\vee,\Z)\cong N.
\]
We obtain the composite map
\begin{equation}
\label{Fan.20.1}
\varphi_{\Delta,N}:\hom(\uSpec(\N),\Delta)\rightarrow \hom(\Spec{\N},\Delta_N)\rightarrow \hom(\Spec{\Z},\Delta_N)\cong N,
\end{equation}
where the first map is the inverse of \eqref{Fan.19.1}, and the second map is induced by the homomorphism $\N\rightarrow \Z$.
\end{const}

\begin{prop}[{\cite[Section 9.5]{MR1296725}}]
\label{Fan.20}
Let $\Delta$ be a connected sharpened fan whose underlying topological space is a finite set.
Suppose that $N^\vee\rightarrow \cM_{\Delta}^\gp$ is a surjective morphism of sheaves, where $N$ is a lattice with dual $N^\vee$.
If the morphism
\[
\varphi_{\Delta,N}:\hom(\uSpec(\N),\Delta)\rightarrow N
\]
is injective, then the associated fs monoscheme $\Delta_N$ is a fan.
\end{prop}
\begin{proof}
Owing to Theorem \ref{fan=monoscheme2} we need to show that $\Delta_N$ is toric and separated.
In Construction \ref{Fan.10} we noted that $\Delta_N$ is toric.
If $\varphi_{\Delta,N}$ is injective, then the homomorphism
\[
\hom(\Spec{\N},\Delta_N)\rightarrow \hom(\Spec{\Z},\Delta_N)
\]
is injective.
This means that $\Delta_N$ is separated.
\end{proof}

\begin{exm}
\label{Fan.18}
Suppose that $\Sigma$ is an fs monoscheme.
Then $\Delta:=\overline{\Sigma}$ is a sharpened fan.
If $\Sigma$ is a fan in a lattice $N$, then $\cM_{\Sigma}^\gp\cong N^\vee$.
Thus the naturally induced morphism of sheaves $\cM_{\Sigma}^\gp\rightarrow \cM_{\Delta}^\gp$ gives a surjective morphism of sheaves
\[
N^\vee\rightarrow \cM_{\Delta}^\gp.
\]
Applying \cite[Proposition I.4.2.1]{Ogu} to $\cM_\Sigma$ we deduce the isomorphism
\[
\cM_{\Sigma}\cong N^\vee\times_{\cM_\Delta^\gp}\cM_\Delta.
\]
This shows $\Sigma\cong \Delta_N$, where $\Delta_N$ is the fs monoscheme in Construction \ref{Fan.10}.
\end{exm}

As in the case of fans, there are notions of subdivision and partial subdivision of sharpened fans as follows.

\begin{df}
\label{Fan.8}
Let $f:\Delta'\rightarrow \Delta$ be a morphism of sharpened fans.
We say that $f$ is a \emph{subdivision}\index{sharpened fan!subdivision}\index{sharpened fan!partial subdivision} (resp.\ \emph{partial subdivision}) if the following conditions are satisfied.
\begin{enumerate}
\item[(i)] For any $x'\in \Delta'$, $\cM_{\Delta,f(x')}^{\rm gp}\rightarrow \cM_{\Delta',x'}^{\rm gp}$ is surjective.
\item[(ii)] For any $x\in \Delta$, $f^{-1}(x)$ is a finite set.
\item[(iii)] The map
\[
\hom(\uSpec(\N),\Delta')\rightarrow \hom(\uSpec(\N),\Delta)
\]
is bijective (resp.\ injective). 
\end{enumerate}
\end{df}

\begin{rmk}
In \cite[Definition 9.7]{MR1296725} a subdivision in our sense is called a proper subdivision.
\end{rmk}

We note that compositions of subdivisions (resp.\ proper subdivisions) are again subdivisions (resp.\ proper subdivisions).

\begin{lem}
\label{Fan.31}
Any subdivision $f:\Delta'\rightarrow \Delta$ of sharpened fans is surjective.
\end{lem}
\begin{proof}
Let $x$ be a point of $\Delta$, which is the generic point of a cone $\uSpec(P)$ of $\Delta$, 
where $P$ is a sharp fs monoid.
By appeal to \cite[Proposition I.2.2.1]{Ogu} there exists a homomorphism $p:P\rightarrow \N$ such that $p^{-1}(0)=0$.
Let $y$ be the generic point of $\uSpec(\N)$, and set $h:=\uSpec(p)$.
Then we have that $h(y)=x$.
By condition (iii) in Definition \ref{Fan.8} there is a morphism $g:\uSpec(\N)\rightarrow \Delta'$ such that $f\circ g=h$.
It follows that $f(g(y))=h(y)=x$, and hence $f$ is surjective.
\end{proof}

\begin{lem}
\label{Fan.32}
Let $f:\Delta'\rightarrow \Delta$ and $g:\Delta''\rightarrow \Delta'$ be morphisms of sharpened fans.
If $g$ and $f\circ g$ are subdivisions, then $f$ is a subdivision.
\end{lem}
\begin{proof}
Since $g$ is surjective by Lemma \ref{Fan.31}, the condition (i) (resp.\ (ii)) in Definition \ref{Fan.8} for $f\circ g$ implies the condition (i) (resp. (ii)) for $f$.
Condition (iii) for $g$ and $f\circ g$ implies the condition (iii) for $f$.
\end{proof}

\begin{lem}
\label{Fan.34}
If $f\colon\Sigma'\rightarrow \Sigma$ is a subdivision (resp.\ partial subdivision) of fans, 
then $\overline{f}:\overline{\Sigma'}\rightarrow \overline{\Sigma}$ is a subdivision (resp.\ partial subdivision) of sharpened fans.
\end{lem}
\begin{proof}
Suppose that $f$ is a $\Sigma$ is in a lattice $N$.
Let $x$ be a point of $\Sigma$ corresponding to a cone $\sigma$.
If $x'$ is a point of $\Sigma'$ such that $f(x')=x$, then there are isomorphisms
\[
\cM_{\Sigma',x'}^\gp\cong \cM_{\Sigma,x}^\gp\cong N^\vee.
\]
It follows that $\overline{\cM}_{\Sigma',x'}^\gp\rightarrow \overline{\cM}_{\Sigma,x}^\gp$ is surjective, which is condition (i) in Definition \ref{Fan.8}.
There are only finitely many cones in $\Sigma'$.
This shows that condition (ii) holds.
Since $\lvert \Delta' \rvert\rightarrow \lvert \Delta \rvert$ is bijective (resp.\ injective), we deduce condition (iii).
\end{proof}

\begin{prop}
\label{Fan.9}
Let $\Sigma$ be a fan.
Then for every subdivision (resp.\ partial subdivision) $f:\Delta'\rightarrow \Delta:=\overline{\Sigma}$ of sharpened fans, 
there exists a subdivision (resp.\ partial subdivision) of fans $g:\Sigma'\rightarrow \Sigma$ such that $\overline{g}\cong f$.
\end{prop}
\begin{proof}
Suppose that $f$ is a partial subdivision.
We may suppose that $\Sigma$ lies in a lattice $N$.
As observed in Example \ref{Fan.18}, there is a surjective morphism of sheaves $N^\vee\rightarrow \cM_\Delta^\gp$ such that $\Sigma\cong \Delta_N$.
Then we have a composite morphism of sheaves
\[
N^\vee\cong f^*N^\vee\rightarrow f^*\cM_{\Delta}^{\rm gp}\rightarrow \cM_{\Delta'}^{\rm gp},
\]
which is surjective because of condition (i).
This gives a morphism of fs monoschemes $g:\Delta_N'\rightarrow \Delta_N$ such that $\overline{g}\cong f$.
\vspace{0.1in}

Let us check that $\Delta_{N}'$ is a fan.
Condition (ii) implies that the underlying topological space of $\Delta'$ is a finite set.
Moreover, condition (iii) implies that the map in \eqref{Fan.20.1}
\[
\varphi_{\Delta',N}:\hom(\uSpec(\N),\Delta')\rightarrow N
\]
is injective.
Proposition \ref{Fan.20} shows that $\Delta_N'$ is a fan.
Since $\Delta$ and $\Delta'$ are in the same lattice $N$, we deduce that $g$ is a partial subdivision.
\vspace{0.1in}

If $f$ is a subdivision, then condition (iii) implies that the map
\[
\hom(\Spec{\N},\Delta_N')\rightarrow \hom(\Spec{\N},\Delta_N)
\]
is bijective by Lemma \ref{Fan.19}.
Proposition \ref{Propersubdivision} shows that $g$ is a subdivision.
\end{proof}

\begin{df}
\label{Fan.6}
There is a notion of star subdivision for smooth sharpened fans as in the case of toric fans, see Definition \ref{Starsubdivision}.\index{sharpened fan!star subdivision}
Let $\tau=\uSpec(Q)$ be a cone of a smooth sharpened fan $\Delta$ where $Q$ is a sharp fs monoid.
For any cone $\sigma=\uSpec(P)$ of $\Delta$ containing $\tau$ where $P$ is a sharp fs monoid, consider the star subdivision $\Theta$ of $\Spec{P}$ relative to $\Spec{Q}$.
Replace $\sigma$ by $\overline{\Theta}$ for every $\sigma$, then we obtain a new smooth sharpened fan $\Delta^*(\tau)$.
This is called the \emph{star subdivision} of $\Delta$ relative to $\tau$.
\end{df}

\begin{lem}
\label{Fan.27}
Every partial subdivision $f\colon\Delta'\rightarrow \Delta$ of sharpened fans is a monomorphism in the category of sharpened fans.
\end{lem}
\begin{proof}
We show that the diagonal morphism $\Delta'\rightarrow \Delta'\times_\Delta \Delta'$ is an isomorphism.
To that end, it suffices to check that for every cone $\sigma$ of $\Delta$, the induced morphism
\[
\Delta'\times_\Delta \sigma\rightarrow \Delta'\times_\Delta \Delta'\times_\Delta \sigma
\]
is an isomorphism.
Hence, we reduce to the case when $\Delta=\uSpec(P)$ for some sharp fs monoid $P$.
In this case, due to Proposition \ref{Fan.9}, there exists a partial subdivision $g\colon \Sigma'\rightarrow \Sigma:=\Spec{P}$ of fans together with an isomorphism $\overline{g}\cong f$.
Lemma \ref{monomorphismoffans} implies $g$ is a monomorphism.
By appeal to Lemma \ref{Fan.29} we deduce that $f$ is a monomorphism.
\end{proof}

\begin{lem}
\label{Fan.33}
Let $f\colon\Psi\rightarrow \Delta$ be a subdivision (resp.\ partial subdivision) of sharpened fans, and $g\colon \Delta'\rightarrow \Delta$ a partial subdivision of sharpened fans.
Then the projection $f'\colon \Psi':=\Psi\times_{\Delta}\Delta'\rightarrow \Delta'$ is a subdivision (resp.\ partial subdivision).
\end{lem}
\begin{proof}
The question is local on $\Delta$, so we may assume that $\Delta$ is affine.
Owing to Proposition \ref{Fan.9} there exists a subdivision (resp.\ partial subdivision) of sharpened fans $u\colon\Theta\rightarrow\Sigma$ and a partial subdivision 
$v\colon \Sigma'\rightarrow \Sigma$ of sharpened fans such that $\overline{u}\cong f$ and $\overline{v}\cong g$.
Let $u'\colon \Theta':=\Theta\times_\Sigma \Sigma'\rightarrow \Sigma'$ be the projection.
Lemma \ref{Fiberproduct_fans} shows that $u'$ is a subdivision (resp.\ partial subdivision).
By Lemma \ref{Fan.29} we have $\overline{u'}\cong f'$.
Due to Lemma \ref{Fan.34} we can conclude that $f'$ is a subdivision (resp.\ partial subdivision).
\end{proof}

\begin{lem}
\label{Fan.26}
Let $f\colon Y\rightarrow X$ and $g\colon Z\rightarrow Y$ be monomorphisms in a category $\cC$ with fiber products.
Then the projection
\[
f'\colon Y\times_X Z\rightarrow Z
\]
is an isomorphism in $\cC$.
\end{lem}
\begin{proof}
There is a naturally induced commutative diagram with cartesian squares
\[
\begin{tikzcd}
Z\times_X Z\arrow[d]\arrow[r,"g'"]&
Y\times_X Z\arrow[d]\arrow[r,"f'"]&
Z\arrow[d,"fg"]
\\
Z\arrow[r,"g"]&
Y\arrow[r,"f"]&
X.
\end{tikzcd}
\]
Since the composite $fg$ is a monomorphism, it follows that $f'g'$ is an isomorphism.
This shows that $f'$ is an epimorphism.
Similarly, again since the pullback of a monomorphism is a monomorphism, it follows that $f'$ is a monomorphism.
Thus $f'$ is an isomorphism.
\end{proof}

\begin{lem}
\label{Fan.17}
Let $p\colon\Delta'\rightarrow \Delta$ be a subdivision of smooth sharpened fans.
If the underlying topological space of $\Delta$ is a finite set, then there exists a sharpened fan $\Delta''$ obtained by a finite succession of star subdivisions from $\Delta$ such that the morphism $\Delta''\rightarrow \Delta$ factors through $p$.
\end{lem}
\begin{proof}
Let $\{\sigma_1=\uSpec(P_1),\ldots,\sigma_n=\uSpec(P_n)\}$ be all the cones of $\Delta$, where $P_1,\ldots,P_n$ are sharp fs monoids.
We will inductively construct subdivisions
\[
\Delta_n\rightarrow \cdots \rightarrow \Delta_0:=\Delta
\]
such that for every $1\leq i\leq n$ the projection
\begin{equation}
\label{Fan.17.1}
\Delta'\times_\Delta \Delta_i\times_\Delta \sigma_i\rightarrow \Delta_i\times_\Delta \sigma_i
\end{equation}
is an isomorphism.
\vspace{0.1in}

Suppose that we have constructed $\Delta_{i-1}$.
By Lemma \ref{Fan.33} the projection 
$$
\Delta'\times_\Delta \Delta_{i-1}\times_\Delta \sigma_i\rightarrow \Delta_{i-1}\times_\Delta \sigma_i
$$ 
is a subdivision of sharpened fans.
Proposition \ref{Fan.20} shows there exist subdivisions $g:\Sigma\rightarrow \Spec{P_i}$ and $h:\Sigma'\rightarrow \Sigma$ of fans such that $\overline{g}$ and $\overline{h}$ 
are isomorphic to the projections
\[
\Delta_{i-1}\times_{\Delta}\sigma_i\rightarrow \sigma_i
\text{ and }
\Delta'\times_\Delta \Delta_{i-1}\times_{\Delta}\sigma_i\rightarrow \Delta_{i-1}\times_\Delta \sigma_i.
\]
Due to Lemma \ref{A.3.31} there is a subdivision $\Sigma''$ of $\Sigma$ obtained by a finite succession of star subdivisions from $\Sigma$ such that $\Sigma''\rightarrow \Sigma$ 
factors through $\Sigma'$.
Since the underlying topological space of $\Sigma$ is an open subset of the underlying topological space of $\Delta_{i-1}$, 
we can take the corresponding finite succession of star subdivisions from $\Delta_{i-1}$ to obtain a subdivision $\Delta_i$ of $\Delta_{i-1}$.
Then the morphism
\[
\Delta_i\times_{\Delta}\sigma_i\rightarrow \Delta_{i-1}\times_{\Delta}\sigma_i
\]
is a subdivision and factors through $\Delta'\times_\Delta\Delta_{i-1}\times_\Delta \sigma_i$.
The projection 
$$
\Delta'\times_\Delta \Delta_{i-1}\times_\Delta \sigma_i\rightarrow \Delta_{i-1}\times_\Delta \sigma_i
$$ 
is a subdivision by Lemma \ref{Fan.33}.
Now use Lemmas \ref{Fan.27} and \ref{Fan.26} to deduce that the morphism \eqref{Fan.17.1} is an isomorphism.
This completes the construction of $\Delta_i$ and hence, by induction, $\Delta_n$.
\vspace{0.1in}

From \eqref{Fan.17.1} we see that for every $1\leq i\leq n$ the projection
\[
\Delta_n\times_\Delta \Delta'\times_\Delta \sigma_i\rightarrow \Delta_n\times_\Delta \sigma_i
\]
is an isomorphism.
It follows that the projection $\Delta_n\times_\Delta \Delta'\rightarrow \Delta_n$ is an isomorphism since $\{\sigma_1,\ldots,\sigma_n\}$ comprises all the cones of $\Delta$.
Hence the morphism $\Delta_n\rightarrow \Delta$ factors through $\Delta'$.
\end{proof}

\subsection{Frames of fs log schemes}
Let $X$ be an fs log scheme. 
If there exists a fan $\Sigma$ with a strict morphism
\[
X\rightarrow \A_\Sigma,
\]
then any partial subdivision $\Sigma'$ of $\Sigma$ produces a new fs log scheme $X\times_{\A_\Sigma}\A_{\Sigma'}$.
Alas, it is unclear whether such a fan $\Sigma$ always exists in the category of fs log schemes.
For fs log schemes in $lSm/k$, we shall proceed with employing sharpened fans.
As an application, we will show that every log modification $Y\rightarrow X$ in $SmlSm/k$ admits a factorization into log modifications
\[
Y_n\rightarrow \cdots \rightarrow Y_1\rightarrow Y_0:=X
\]
such that for every $1\leq i\leq n$, the morphism 
$$
\underline{Y_i}\rightarrow \ul{Y_{i-1}}
$$ 
is a blow-up along a smooth center.

\begin{df}
\label{Fan.2}
Let $X$ be an fs log scheme.
A \emph{preframe}\index{preframe} of $X$ is a sharpened fan $\Delta$ with a morphism of monoidal spaces $s\colon(X,\overline{\cM}_X)\rightarrow \Delta$.
\vspace{0.1in}

A preframe $\Delta$ of $X$ is called a \emph{frame} of $X$ if the following condition is satisfied:
Locally on $X$ and $\Delta$, 
there exists a chart $t\colon (X,\cM_X)\rightarrow \Spec{P}$ with an fs monoid $P$ such that $\ol{t}$ is isomorphic to $s\colon(X,\overline{\cM}_X)\rightarrow \Delta$.
Such a frame is an $s$-frame in the sense of \cite[Definition III.1.12.1]{Ogu}.
\end{df}

\begin{prop}
\label{Fan.3}
For every $X\in lSm/k$ we can associate a frame 
$$
s_X:(X,\overline{\cM}_X)\rightarrow \Delta_X
$$ 
satisfying the following properties.
\begin{enumerate}
\item[{\rm (i)}] Every preframe of $X$ uniquely factors through $s_X$.
\item[{\rm (ii)}] The morphism of monoidal spaces $s_X$ is open and surjective.
\item[{\rm (iii)}] For every morphism $f:X\rightarrow Y$ in $lSm/k$, 
there exists a unique morphism of sharpened fans $\Delta_f:\Delta_X\rightarrow \Delta_Y$ such that the diagram
\[
\begin{tikzcd}
(X,\overline{\cM}_X)\arrow[r,"f"]\arrow[d,"s_X"']&
(Y,\overline{\cM}_Y)\arrow[d,"s_Y"]
\\
\Delta_X\arrow[r,"\Delta_f"]&
\Delta_Y
\end{tikzcd}
\]
commutes.
\end{enumerate}
\end{prop}
\begin{proof}
Owing to \cite[Theorems III.1.11.1, IV.3.5.1]{Ogu}, $X$ is very solid in the sense of \cite[Definition III.1.10.1]{Ogu}.
Then apply \cite[Proposition III.1.12.3]{Ogu} and its following paragraph.
\end{proof}

\begin{lem}
\label{Fan.15}
Let $X$ be a quasi-compact fs log scheme in $lSm/k$.
Then the underlying space of the associated sharpened fan $\Delta_X$ is a finite set.
\end{lem}
\begin{proof}
There exists a finite Zariski cover $\{U_i\}$ of $X$ with the property that there exists a chart $t_i\colon(U_i,\cM_{U_i})\rightarrow \Spec{P_i}$ such that the composite morphism
\[
(U_i,\ul{\cM}_{U_i})\rightarrow (X,\ul{\cM}_X)\stackrel{s_X}\rightarrow \Delta_X
\]
factors through $\ol{t_i}$.
Since the underlying topological space of $\Spec{P_i}$ is a finite set, we can use property (ii) of Proposition \ref{Fan.3} to conclude.
\end{proof}

\begin{const}
\label{Fan.5}
Suppose $X$ is an fs log scheme with a frame 
$$
\pi:(X,\overline{\cM}_X)\rightarrow \Delta,
$$ 
and let $p:\Delta'\rightarrow \Delta$ be a partial subdivision of sharpened fans.
Let $\cC_{X,\Delta',\Delta}$ be the category of pairs $(f,\pi')$, where $f\colon X'\rightarrow X$ is a morphism of fs log schemes and 
\[
\pi':(X',\overline{\cM}_{X'})\rightarrow \Delta'
\]
is a morphism of monoidal spaces such that the diagram
\[
\begin{tikzcd}
(X',\overline{\cM}_{X'})\arrow[d]\arrow[r,"\pi'"]&
\Delta'\arrow[d,"p"]
\\
(X,\overline{\cM}_{X})\arrow[r,"\pi"]&
\Delta
\end{tikzcd}
\]
commutes.
In \cite[Proposition 9.9]{MR1296725}, it is shown that $\cC_{X,\Delta',\Delta}$ admits a final object
\[
X\times_{\Delta}\Delta'.
\]
The fiber product notation is used for convenience; technically, it is not a fiber product.
Moreover, 
the morphism $\pi'$ gives a framing of $X'$.
\vspace{0.1in}

Let us explain the local description of $X'$ in \cite[Proposition 9.9]{MR1296725}.
Suppose that $P$ is an fs chart of $X$ with a homomorphism $\theta:P\rightarrow Q$ of fs monoids such that $\theta^{\rm gp}$ is an isomorphism, 
$\Delta=\uSpec(\overline{P})$, 
and $\Delta'=\uSpec(\overline{Q})$.
In this case, we have
\[
X\times_{\Delta}\Delta'\cong X\times_{\A_P}\A_{Q}.
\]
More generally, 
suppose that $\Delta=\uSpec(\overline{P})$ and $\Delta'=\overline{\Sigma}$, 
where $P$ is an fs chart of $X$ and $\Sigma$ is a partial subdivision of $\Spec{P}$.
Then the local description gives an isomorphism
\begin{equation}
\label{Fan.5.1}
X\times_{\Delta}\Delta'\cong X\times_{\A_P}\A_\Sigma.
\end{equation}
\end{const}

\begin{lem}
\label{Fan.36}
Let $X$ be an fs log scheme with a frame $\Delta$, and let $\Delta'\rightarrow \Delta$ and $\Delta''\rightarrow \Delta'$ be partial subdivisions of sharpened fans.
Then there is a canonical isomorphism
\[
(X\times_\Delta \Delta')\times_{\Delta'}\Delta''
\cong
X\times_\Delta \Delta''.
\]
\end{lem}
\begin{proof}
We need to show that $(X\times_\Delta \Delta')\times_{\Delta'}\Delta''$ is a final object in the category $\cC_{X,\Delta'',\Delta}$.
Suppose that $f:X'\rightarrow X$ is a morphism of fs log schemes with a morphism of monoidal spaces $(X',\overline{\cM}_{X'})\rightarrow \Delta''$ such that the diagram
\[
\begin{tikzcd}
(X',\overline{\cM}_{X'})\arrow[d]\arrow[r]&
\Delta''\arrow[d]
\\
(X,\overline{\cM}_{X})\arrow[r]&
\Delta
\end{tikzcd}
\]
commutes.
The universal property of $X\times_\Delta \Delta'$ means that there is a unique morphism $X'\rightarrow X\times_\Delta \Delta'$ such that the diagram
\[
\begin{tikzcd}
(X',\overline{\cM}_{X'})\arrow[d]\arrow[r]&
\Delta''\arrow[d]
\\
(X\times_\Delta \Delta',\overline{\cM}_{X\times_\Delta \Delta'})\arrow[r]&
\Delta'
\end{tikzcd}
\]
commutes.
Similarly, 
the universal property of $(X\times_\Delta \Delta')\times_{\Delta'}\Delta''$ means that there is a unique morphism 
$$
X'\rightarrow (X\times_\Delta \Delta')\times_{\Delta'}\Delta''
$$ 
in $\cC_{X,\Delta'',\Delta}$.
This shows that $(X\times_\Delta \Delta')\times_{\Delta'}\Delta''$ is a final object in $\cC_{X,\Delta'',\Delta}$.
\end{proof}

\begin{lem}
\label{Fan.11}
Let $X$ be an fs log scheme in $lSm/k$.
Then for every subdivision  of fans $p:\Delta'\rightarrow \Delta_X$,
the natural morphism
\[
f:X\times_{\Delta_X}\Delta'\rightarrow X
\]
is a log modification.
\end{lem}
\begin{proof}
The question is Zariski local on $X$, so we may assume that $X$ admits an fs chart $t\colon X\rightarrow \Spec{P}$ such that $\ol{t}$ is isomorphic to $s_X\colon (X,\ol{\cM}_X)\rightarrow \Delta_X$.
Owing to Proposition \ref{Fan.9} there exists a subdivision of fans $q:\Sigma'\rightarrow \Spec{P}$ such that $\ol{q}\cong p$.
From the construction of $X\times_{\Delta_X}\Delta'$ we see that
\[
X\times_{\Delta_X}\Delta'\cong X\times_{\A_{\Spec{P}}}\A_{\Sigma'}= X\times_{\A_{P}}\A_{\Sigma'}.
\]
This concludes the proof since the projection $X\times_{\A_{P}}\A_{\Sigma'}\rightarrow X$ is a log modification due to Example \ref{A.9.76}.
\end{proof}

\begin{lem}
\label{Fan.35}
Suppose $X$ is an fs log scheme in $lSm/k$ with an fs chart $t\colon X\rightarrow \Spec{P}$ where $P$ is an fs monoid.
For every subdivision $\Sigma\rightarrow \Spec{P}$ of fans, the associated morphism of sharpened fans
\[
\Delta_{X\times_{\A_P} \A_\Sigma}\rightarrow \Delta_X
\]
is a subdivision.
\end{lem}
\begin{proof}
The question is Zariski local on $X$, so we may assume that $\overline{t}$ is isomorphic to $s_X\colon (X,\overline{\cM}_X)\rightarrow \Delta_X$.
Due to \eqref{Fan.5.1} there is an isomorphism
\[
X\times_{\A_P}\A_{\Sigma}\cong X\times_{\Delta_X}\overline{\Sigma}.
\]
It follows that $\Delta_{X\times_{\A_P}\A_\Sigma}\cong \overline{\Sigma}$.
Since $\Sigma\rightarrow \Spec{P}$ is a subdivision, $\overline{\Sigma}\rightarrow \Delta_X$ is a subdivision by Lemma \ref{Fan.34}.
\end{proof}

\begin{lem}
\label{Fan.30}
For any log modification $f\colon Y\rightarrow X$ in $lSm/k$,
the associated morphism $u\colon\Delta_Y\rightarrow \Delta_X$ of sharpened fans is a subdivision.
\end{lem}
\begin{proof}
The question is Zariski local on $X$, so we may assume that $X$ is quasi-compact and admits an fs chart $t\colon X\rightarrow \Spec{P}$ such that $\ol{t}$ is isomorphic to 
$s_X\colon (X,\ol{\cM}_X)\rightarrow \Delta_X$.
Owing to Proposition \ref{A.9.21} there is a subdivision of fans $p:\Sigma\rightarrow \Spec{P}$ such that the pullback
\[
Y\times_{\A_P}\A_\Sigma\rightarrow X\times_{\A_P}\A_\Sigma =:X'
\]
is an isomorphism.
Therefore we have associated morphisms $\Delta_{X'}\stackrel{v}\rightarrow \Delta_{Y}\stackrel{u}\rightarrow \Delta_X$ of sharpened fans.
\vspace{0.1in}

We claim that $v$ and $u\circ v$ are subdivisions.
Owing to Lemma \ref{A.9.19} there is a Zariski cover $\{U_i\}$ of $Y$ such that for every $i$ the morphism $U_i\rightarrow X$ admits an fs chart $\theta\colon P\rightarrow Q_i$ 
such that $\theta^\gp\colon P^\gp\rightarrow Q_i^\gp$ is an isomorphism.
This means that $\Spec{Q_i}$ is a partial subdivision of $\Spec{P}$.
Lemma \ref{monomorphismoffans} shows that the projection 
$$
\Sigma_i:=\Spec{Q_i}\times_{\Spec{P}}\Sigma\rightarrow \Spec{Q_i}
$$ 
is a subdivision of fans.
Thus by Lemma \ref{Fan.35} the associated morphism of sharpened fans
\[
\Delta_{U_i\times_{\A_P}\A_{\Sigma}}
\cong
\Delta_{U_i\times_{\A_{Q_i}}\A_{\Sigma_i}}
\rightarrow
\Delta_{U_i}
\]
is a subdivision.
It follows that $v$ is a subdivision.
Moreover, Lemma \ref{Fan.35} also shows that $u\circ v$ is a subdivision.
As a consequence, $u$ is a subdivision by Lemma \ref{Fan.32}.
\end{proof}

\begin{lem}
\label{Fan.38}
Let $X$ be an fs log scheme in $SmlSm/k$.
Then the associated sharpened fan $\Delta_X$ is smooth.
\end{lem}
\begin{proof}
The question is Zariski local on $X$, 
so we may assume that $X$ admits an fs chart $t\colon X\rightarrow \Spec{P}$ such that $\ol{t}$ is isomorphic to $s_X\colon (X,\ol{\cM}_X)\rightarrow \Delta_X$.
Lemma \ref{lem::Smchart} shows that $P\cong \N^r$ for some integer $r\geq 0$.
Thus $\Delta_X\cong \uSpec(\N^r)$ is smooth.
\end{proof}

\begin{df}
\label{Fan.39}
For $X\in SmlSm/k$, 
a \emph{log modification along a smooth center}\index{log modification!along a smooth center} is a log modification $p:Y\rightarrow X$ such that Zariski locally on $X$ there exists a chart $X\rightarrow \A_{\N^r}$ 
such that $p$ is isomorphic to the projection
\[
X\times_{\A_{\N^r}}\A_{\Sigma} 
\to
X.
\]
Here $\Sigma$ is the star subdivision of the dual of $\N^r$.
\end{df}

\begin{lem}
\label{Fan.37}
Let $X$ be an fs log scheme in $SmlSm/k$.
Then for every star subdivision of fans $p:\Delta'\rightarrow \Delta_X$, the natural morphism
\[
f\colon X\times_{\Delta_X}\Delta'\rightarrow X
\]
is a log modification along a smooth center.
\end{lem}
\begin{proof}
The question is Zariski local on $X$, 
so we may assume that $X$ admits an fs chart $t\colon X\rightarrow \Spec{P}$ such that $\ol{t}$ is isomorphic to $s_X\colon (X,\ol{\cM}_X)\rightarrow \Delta_X$.
Due to Lemma \ref{Fan.38}, $P\cong \N^r$ for some integer $r\geq 0$.
By the definition of star subdivisions of sharpened fans, there exists a face $Q$ of $P$ such that $\Delta'\cong \overline{\Sigma}$, 
where $\Sigma$ is the star subdivision of $\Spec{P}$ relative to $\Spec{Q}$.
Then \eqref{Fan.5.1} gives an isomorphism
\[
X\times_{\Delta_X}\Delta'\cong X\times_{\A_P}\A_{\Sigma}.
\]
This shows that $f$ is a log modification along a smooth center.
\end{proof}

\begin{thm}
\label{Fan.16}
Let $f:Y\rightarrow X$ be a log modification of quasi-compact fs log schemes in $lSm/k$. If $X\in SmlSm/k$,
then there exists a sequence of log modifications along a smooth center
\[
X_n\rightarrow \cdots \rightarrow X_1\rightarrow X
\]
such that the projection $X_n\times_X Y\rightarrow X_n$ is an isomorphism.
\end{thm}
\begin{proof}
Lemma \ref{Fan.30} shows that the associated morphism $\Delta_Y\rightarrow \Delta_X$ is a subdivision of sharpened fans.
By Lemma \ref{Fan.37} the associated sharpened fan $\Sigma_X$ is smooth.
Lemma \ref{Fan.17} yields star subdivisions of smooth sharpened fans
\[
\Delta_n\rightarrow \cdots \rightarrow \Delta_1\rightarrow \Delta_X
\]
such that the subdivision $\Delta_n\rightarrow \Delta_X$ factors through $\Delta_Y\rightarrow \Delta_X$.
We set 
$$
X_i:=X\times_{\Delta_X}\Delta_i
$$ 
for every $1\leq i\leq n$.
By definition $X_n\rightarrow X$ factors through $f$, 
and by Lemma \ref{Fan.36} there is an isomorphism
\[
X_n\cong X_{n-1}\times_{\Delta_{n-1}}\Delta_n.
\]
Lemma \ref{Fan.37} shows that the morphism $X_i\rightarrow X_{i-1}$ is a log modification along a smooth center for every $1\leq i\leq n$.
To show that the projection $X_n\times_X Y\rightarrow X_n$ is an isomorphism, use Lemma \ref{Fan.26}.
\end{proof}

This is used in the proof of the following result.

\begin{thm}
\label{FKatoThm2}
Let $f:Y\rightarrow X$ be a morphism in $lSm/k$.
Then there exists a log modification $X'\rightarrow X$ such that the projection
\[
f':Y\times_X X'\rightarrow X'
\]
is integral. In particular, $f'$ is exact.
\end{thm}
\begin{proof}
Due to Proposition \ref{A.3.19}, we reduce to the case when $X$ is in $SmlSm/k$.
There exists a Zariski cover $\{U_1,\ldots,U_n\}$ of $X$ and log modifications $U_i'\to U_i$ such that the projection $Y\times_X U_i'\to U_i'$ is integral for every $i$, see \cite{Katointegral}.
Thanks to Theorem \ref{Fan.16}, we may assume that $U_i'$ is obtained by taking a sequence of star subdivisions of $\Delta_{U_i}$.
The sequence of star subdivions can be naturally extended to a sequence of star subdivisions of $\Delta_X$.
It follows that there exists a log modification $X_i'\to X$ such that $U_i'\simeq U_i\times_X X_i'$ for every $i$.

Take $X':=X_1'\times_X \cdots \times_X X_n'$.
Since the projection $Y\times_X U_i'\to U_i'$ is integral for every $i$, the projection $f'\colon Y\times_X X'\to X'$ is integral Zariski locally on $X'$.
This implies that $f'$ is integral.
\end{proof}

\begin{prop}
\label{Fan.12}
Let $f\colon X\rightarrow S$ be an exact morphism in $lSm/k$.
Then for every log modification $g\colon Y\rightarrow X$, there is a log modification $S'\rightarrow S$ such that the pullback
\[
r\colon Y\times_S S'\longrightarrow X\times_S S'
\]
is an isomorphism.
\end{prop}
\begin{proof}
There exists a log modification $S'\rightarrow S$ such that the projection
\[
q\colon Y\times_S S'\longrightarrow S'
\]
is exact owing to Theorem \ref{FKatoThm2}.
Since $f$ is exact, the projection $p\colon X\times_S S'\rightarrow S'$ is exact.
It follows that 
$$
r:Y\times_S S'\rightarrow X\times_S S'
$$ 
is exact owing to \cite[Proposition III.3.2.1(1)]{Ogu}.
We deduce that $r$ is an open immersion owing to Proposition \ref{A.9.83}.
Since $r$ is surjective and proper, it is an isomorphism.
\end{proof}

\begin{cor}
\label{Fan.14}
Let $f:X\rightarrow S$ be an exact morphism in $lSm/k$.
Then the functor
\[
f^*:S_{div}\rightarrow X_{div}
\]
mapping $S'\in S_{div}$ to $S'\times_S X$ is cofinal.
\end{cor}
\begin{proof}
Immediate from Proposition \ref{Fan.12}.
\end{proof}

\newpage

\section{Model structures on the category of complexes}
\label{AppendixB}
\subsection{Descent structures}
  
In this appendix, we recall the construction of model structures on chain complexes.
We follow the approach in \cite{CD09} using descent structures.

\begin{df}
Let $\cT$ be a triangulated category, and let $\cA$ be an abelian category. 
\begin{enumerate}
\item[(i)] 
An object $X$ of $\cT$ is called {\it compact}\index{compact object} if the functor
\[
\hom_{\cT}(X,-)
\]
commutes with small sums.
\item[(ii)]
Let $\cF$ be a family of objects of $\cA$. 
We say that $\cF$ {\it generates} $\cA$ if, for any object $F$ of $\cA$, there is an epimorphism $\oplus X\rightarrow F$ from a small sum of objects of $\cF$.
\item[(iii)] 
Let $\cF$ be a family of objects of $\cT$. We say that $\cF$ {\it generates} $\cT$ if the family of functors
  \[\hom_{\cT}(X[n],-)\]
  for $X\in \cF$ and $n\in \Z$ is conservative.
  \item[(iv)] We say that $\cT$ is {\it compactly generated} if there is a family that generates $\cT$ and consists of compact objects.
  \item[(v)] A {\it localizing subcategory}\index{localizing subcategory} of $\cT$ is a triangulated subcategory of $\cT$ that is stable by small sums. For a family $\cF$ of objects of $\cA$, we write $\langle \cF\rangle$ for the smallest localizing subcategory of $\cT$ containing $\cF$.
  \end{enumerate}
\end{df}

\begin{df}[{\cite[Definition 2.2]{CD09}}]
\label{A.8.14}
Let $\cA$ be a Grothendieck abelian category. 
For $\cG$ an essentially small set of objects of $\cA$, and $\cH$ an essentially small set of complexes of $\cA$ we make the following definitions.
\begin{enumerate}
\item[(i)] 
The class of $\cG$-{\it cofibrations} (or simply {\it cofibrations}) is the smallest class of morphisms in $\Co(\cA)$ \index{Gcof @ $\mathscr{G}$-cofibrations}
that contains the morphisms
\[
\begin{tikzcd}
{[}0\arrow[r]\arrow[d]&F{]}\arrow[d,"{\rm id}"]\\
{[}F\arrow[r,"{\rm id}"]&F{]},
\end{tikzcd}
\]
where $F\in \cG$ and $0$ sits in degree $-1$, and is closed under shifts, pushouts, retracts, and transfinite compositions.
\item[(ii)] 
A complex $K$ of $\cA$ is called $\cG$-{\it cofibrant} (or simply {\it cofibrant}) if $0\rightarrow K$ is a $\cG$-cofibration. 
\item[(iii)] 
A complex $K$ of $\cA$ is called $\cG$-{\it local} (or simply {\it local}) if for any $\cG$-cofibrant complex $L$ and $n\in \Z$ the induced homomorphism  \index{Glocal @ $\mathscr{G}$-local complex}
\[
\hom_{{\bf K}(\cA)}(L[n],K)\rightarrow \hom_{\Deri(\cA)}(L[n],K)
\]
is an isomorphism of abelian groups.
\item[(iv)] 
A complex $K$ of $\cA$ is called $\cH$-{\it flasque} (or simply {\it flasque}) if for any $n\in \Z$ and $L\in \cH$ we have \index{Hflasque @ $\cH$-flasque complex}
\[
\hom_{{\bf K}(\cA)}(L[n],K)=0.
\]
\end{enumerate}
 We say that $(\cG,\cH)$ is a {\it descent structure} \index{descent structure} for $\cA$ if $\cG$ generates $\cA$ and any $\cH$-flasque complex is $\cG$-local.

Assume that $\cA$ has a closed symmetric monoidal structure.
Then $(\cG,\cH)$ is called {\it weakly flat} if $\cG$ is closed under tensor products up to isomorphism, the unit object of $\cA$ is in $\cG$, 
and $K\otimes L$ is acyclic for any $K\in \cG$ and $L\in \cH$. \index{descent structure!weakly flat}
\end{df}

Next we recall the {\it injective model structure} on chain complexes in Grothendieck abelian categories. 
\begin{df}
\label{A.8.26}
Let $\cA$ be a Grothendieck abelian category. 
A monomorphism in $\Co(\cA)$ is called an {\it injective cofibration}. 
By \cite[Theorem 2.1]{CD09}, $\Co(\cA)$ is a proper cellular model category with quasi-isomorphisms as weak equivalences and injective cofibrations as cofibrations. 
A fibration with respect to the injective model structure is called an {\it injective fibration}.
\end{df}

We shall refer to the following model structure as the \emph{descent model structure}.
\begin{prop}
\label{A.8.15}
Let $\cA$ be a Grothendieck abelian category with a descent structure $(\cG,\cH)$. 
Then the following properties hold.
\begin{enumerate}
\item[{\rm (1)}] 
$\Co(\cA)$ is a proper cellular model category with quasi-isomorphisms as weak equivalences and $\cG$-cofibrations as cofibrations.
\item[{\rm (2)}] 
A complex $K$ is $\cG$-local if and only if it is $\cH$-flasque.
\item[{\rm (3)}] 
A morphism of complexes is a fibration if and only if it is a degreewise surjection with a $\cG$-local kernel.
\item[{\rm (4)}] 
Assume that $\cA$ has a closed symmetric monoidal structure. If $(\cG,\cH)$ is weakly flat, then $\Co(\cA)$ is a symmetric monoidal model category.
\end{enumerate}
\end{prop}
\begin{proof}
By \cite[Theorem 2.5]{CD09} (resp.\ \cite[Corollary 5.5]{CD09}, resp.\ \cite[Proposition 3.2]{CD09}), we have (1) and (2) (resp.\ (3), resp.\ (4)).
\end{proof}

\begin{prop}
\label{A.8.32}
Let $\cA$ be a Grothendieck abelian category, and let $\cG$ be an essentially small set of objects of $\cA$ generating $\cA$. 
Then the following properties hold.
\begin{enumerate}
\item[{\rm (1)}] 
$\cG$ generates $\Deri(\cA)$.
\item[{\rm (2)}] 
$\Deri(\cA)=\langle \cG\rangle$.
\end{enumerate}
\end{prop}
\begin{proof}
(1) Let $f:F\rightarrow G$ be a morphism in $\Deri(\cA)$. Assume that the induced map
\[
\hom_{\Deri(\cA)}(X[n],F)\rightarrow \hom_{\Deri(\cA)}(X[n],G)
\]
is an isomorphism for any $X\in \cG$ and $n\in Z$. 
For a cone $H$ of $f$, we have
\[
\hom_{\Deri(\cA)}(X[n],H)=0.
\]
Choose an injective fibrant resolution $H\rightarrow I$. 
If $I$ is not quasi-isomorpic to $0$, then $H^n(I)\neq 0$ for some $n\in \Z$. 
Choose an epimorphism $J\rightarrow \ker d^n$ from a small sum $J$ of objects of $\cG$, where $d^i$ is the differential map $I^i\rightarrow I^{i+1}$. 
Then the $J\rightarrow \ker d^n$ does not factor through $\im d^{n-1}$, and we reach the contradiction
\[\hom_{{\bf D}(\cA)}(J[-n],H)\cong \hom_{{\bf K}(\cA)}(J[-n],I) \neq 0.
\]
Thus $I$ is quasi-isomorphic to $0$, and we conclude that $f$ is an isomorphism.
  
(2) Since $\cG$ generates $\cA$, for any object $F$ of $\cA$, there is a quasi-isomorphism $P\rightarrow F$ where $P$ is a bounded above complex of objects that are small sums of objects of $\cG$. 
Thus we have an inclusion $\cA\subset \langle \cG\rangle$, 
which readily implies $\Deri(\cA)=\langle \cG\rangle$.
\end{proof}

\begin{df} \index{Wlocal @ $\cW$-local descent structure}
\label{A.8.18}
Let $\cA$ be a Grothendieck abelian category with a descent structure $(\cG,\cH)$, and let $\cW$ be an essentially small set of $\cG$-cofibrant complexes of $\cA$. 
Recall the following definitions from \cite[\S 4]{CD09}.
\begin{enumerate}
\item[(i)] 
A complex $G$ is called {\it $\cW$-local} if for any $n\in \Z$ and $F\in \cW$ we have
\[
\hom_{\Deri(\cA)}(F[n],G)=0.
\]
\item[(ii)] 
A morphism $F\rightarrow F'$ of complexes is called a {\it $\cW$-equivalence} if for any $\cW$-local complex $G$ the induced homomorphism
\[
\hom_{\Deri(\cA)}(F'[n],G)\rightarrow \hom_{\Deri(\cA)}(F[n],G)
\]
is an isomorphism of abelian groups.
    \item[(iii)] A morphism of complexes is called a {\it $\cW$-fibration} if it is a fibration with respect to the descent model structure with a $\cW$-local kernel.
  \end{enumerate}
\end{df}

\begin{exm}[{\cite[Example 2.3]{CD09}}]
\label{A.8.28}
Let $\cC$ be a category with a topology $t$ and sheafification functor $a_t^*$.
The category of sheaves $\Shvckl$ of $\Lambda$-modules has a descent structure $(\cG,\cH)$ given as follows. 
Let $\Lambda(X)$ denote the free presheaf of $\Lambda$-modules generated by $X\in \cC$.
For a hypercover $\mathscr{X}\rightarrow X$ we let $\Lambda(\mathscr{X})$ denote the associated complex of sheaves.
Now let $\cG$ be the family of objects $a_t^*\Lambda(X)$ for all $X\in \cC$, 
and let $\cH$ be the family of cones of the morphisms $a_t^*\Lambda(\mathscr{X})\rightarrow a_t^*\Lambda(X)$ for all hypercovers $\mathscr{X}\rightarrow X$. 
With these definitions $(\cG,\cH)$ gives a descent structure of $\Shvckl$.

Note that a complex $\cF\in \Co(\Shvckl)$ is $\cH$-flasque if and only if for every hypercover $\mathscr{X}\rightarrow X$ we have that
\begin{equation}
\label{A.8.28.1}
\cF(X)\simeq {\rm Tot}^\pi\cF(\mathscr{X}).
\end{equation}
This is equivalent to saying that for every hypercover $\mathscr{X}\rightarrow X$ and integer $i\in \Z$ we have that
\begin{equation}
\label{A.8.28.2}
\hom_{\mathbf{K}(\Shvckl)}(a_t^*\Lambda(X),\cF[i])
\cong 
\hom_{\mathbf{K}(\Shvckl)}(a_t^*\Lambda(\mathscr{X}),\cF[i]).
\end{equation}
\end{exm}

\begin{prop}
\label{A.8.19}
Let $\cA$ be a Grothendieck abelian category with a descent structure $(\cG,\cH)$, and let $\cW$ be an essentially small set of $\cG$-cofibrant complexes of $\cA$. 
\begin{enumerate}
\item[{\rm (1)}] 
The category $\Co(\cA)$ is a proper cellular model category with $\cW$-equivalences as weak equivalences, 
$\cG$-cofibrations as cofibrations, 
and $\cW$-fibrations as fibrations. 
This model structure is the left Bousfield localization of the descent model structure by the maps $0\rightarrow T[n]$ for $T\in \cW$ and $n\in \Z$. 
Its homotopy category $\Deri_\cW(\cA)$ is equivalent to the Verdier quotient $\Deri(\cA)/\langle\cW\rangle$, 
where $\langle\cW\rangle$ is the smallest localizing subcategory of $\Deri(\cA)$ containing $\cW$.
\item[{\rm (2)}] 
Assume that $\cA$ has a closed symmetric monoidal structure and that $(\cG,\cH)$ is weakly flat. 
If $T\otimes G\in \cW$ for any $T\in \cW$ and $G\in \cG$, 
then $\Co(\cA)$ is a symmetric monoidal model category and the localization functor $\Deri(\cA)\rightarrow \Deri_\cW(\cA)$ is a symmetric monoidal functor.
This model structure is called the {\rm $\cW$-local descent model structure}.
\end{enumerate}
\end{prop}
\begin{proof}
According to \cite[Propositions 4.3, 4.7]{CD09} (resp.\ \cite[Proposition 4.11]{CD09}), we have (1) (resp.\ (2)).
\end{proof}

\begin{prop}
\label{A.8.16}
Let $\cA$ (resp.\ $\cA'$) be a Grothendieck abelian category with a descent structure $(\cG,\cH)$ (resp.\ $(\cG',\cH')$), and let
\[
f^*:\cA\rightleftarrows \cA':f_*
\]
be a pair of adjoint additive functors. 
Assume that
\begin{enumerate}
\item[{\rm (i)}] 
for any $K$ in $\cG$, $f^*K$ is a direct sum of elements of $\cG'$,
\item[{\rm (ii)}] 
for any $K$ in $\cH$, $f^*K$ is in $\cH'$.
\end{enumerate}

Then the induced pair of adjunctions
\[
f^*:\Co(\cA)\rightleftarrows \Co(\cA'):f_*
\]
is a Quillen pair with respect to the descent model structures.
\end{prop}
\begin{proof}
This is the content of \cite[Proposition 2.14]{CD09}.
\end{proof}

\begin{prop}
\label{A.8.17}
Let $\cA$ be a Grothendieck abelian category with a descent structure $(\cG,\cH)$, let $\cA'$ be another Grothendieck abelian category, and let
\[
f^*:\cA\rightleftarrows \cA':f_*
\]
be a pair of adjoint additive functors. 
Assume that
\begin{enumerate}
\item[{\rm (i)}] $f_*$ is exact,
\item[{\rm (ii)}] for any complex $H$ in $\cH$, $f^*H$ is quasi-isomorphic to $0$.
\end{enumerate}
Then $(f^*\cG,f^*\cH)$ is a descent structure for $\cA'$ and the induced pair of adjunctions
\[
f^*:\Co(\cA)\rightleftarrows \Co(\cA'):f_*
\]
is a Quillen pair with respect to the descent model structures.
\end{prop}
\begin{proof}
Since $f^*$ and $f_*$ are additive functors, there is an induced adjoint functor pair
\[
f^*:{\bf K}(\cA)\rightleftarrows {\bf K}(\cA'):f_*
\]

Let us consider the injective model structures on $\cA$ and $\cA'$.
For a $f^*\cH$-flasque complex $L$ we choose an injective fibrant replacement $p:L\rightarrow L'$ in $\Co(\cA')$. 
For any $\cG$-cofibrant object $K$ in $\cA$, we have the induced commutative diagram of abelian groups:
\[
\begin{tikzcd}
\hom_{\Deri(\cA')}(f^*K,L)\arrow[d,leftarrow,"u"']\arrow[r,"w"]&\hom_{\Deri(\cA')}(f^*K,L')\arrow[d,leftarrow,"u'"]\\
\hom_{{\bf K}(\cA')}(f^*K,L)\arrow[d,"\sim"']\arrow[r]&\hom_{{\bf K}(\cA')}(f^*K,L')\arrow[d,"\sim"]\\
\hom_{{\bf K}(\cA)}(K,f_*L)\arrow[d,"v"']\arrow[r]&\hom_{{\bf K}(\cA)}(K,f_*L')\arrow[d,"v'"]\\
\hom_{\Deri(\cA)}(K,f_*L)\arrow[r,"w'"]&\hom_{\Deri(\cA)}(K,f_*L')\
\end{tikzcd}
\]

In the following, we show that $u$ is an isomorphism.  
It follows that $L$ is a $f^*\cG$-local complex. 
  
Since $p$ is a quasi-isomorphism, $w$ is an isomorphism. 
Moreover, 
$f_*p$ is also an isomorphism since $f_*$ is exact. 
It follows that $w'$ is an isomorphism. 
Since $L'$ is an injectively fibrant object, we deduce $u'$ is an isomorphism.
For any $H\in \cH$, $f^*H$ is quasi-isomorphic to $0$, so $\hom_{\Deri(\cA')}(f^*H,L')=0$. 
Using that $u'$ is an isomorphism, we obtain 
\[
\hom_{{\bf K}(\cA')}(f^*H,L')=0.
\]
Thus we have 
\[
\hom_{{\bf K}(\cA')}(H,f_*L')=0.
\]
That is, 
$f_*L'$ is $\cH$-flasque and hence $f_*L'$ is $\cG$-local since $(\cG,\cH)$ is a descent structure for $\cA$. 
This implies $v'$ is an isomorphism. 
Since $L$ is a $f^*\cH$-flasque complex we have
\[
\hom_{{\bf K}(\cA')}(f^*H,L)=0.
\] 
That is, 
$\hom_{{\bf K}(\cA)}(H,f_*L)=0$ and hence $f_*L$ is a $\cH$-flasque complex. 
This implies $v$ is an isomorphism.

To summarize, we have shown that $u'$, $v$, $v'$, $w$, and $w'$ are isomorphisms.
This implies $u$ is an isomorphism and that $(f^*\cG,f^*\cH)$ is a descent structure for $\cA'$. 
The remaining assertion follows from Proposition \ref{A.8.16}.
\end{proof}

\begin{prop}\label{A.8.20}
Let $\cA$ (resp.\ $\cA'$) be a Grothendieck abelian category with a descent structure $(\cG,\cH)$ (resp.\ $(\cG',\cH')$),  
let $\cW$ (resp.\ $\cW'$) be an essentially small set of $\cG$-cofibrant (resp.\ $\cG'$-cofibrant) complexes of $\cA$ (resp.\ $\cA'$), 
and let
\[
f^*:\cA\rightleftarrows \cA':f_*
\]
be a pair of adjoint additive functors. 
Assume that
\begin{enumerate}
\item[{\rm (i)}]  
for any $K$ in $\cG$, $f^*K$ is a direct sum of elements of $\cG'$,
\item[{\rm (ii)}] 
for any $K$ in $\cH$, $f^*K$ is in $\cH'$,
\item[{\rm (iii)}] 
for any $K$ in $\cW$, $f^*$ is in $\cW'$.
\end{enumerate}
Then the induced pair of adjunctions
\[
f^*:\Co(\cA)\rightleftarrows \Co(\cA'):f_*
\]
is a Quillen pair with respect to the $\cW$-local descent model structures.
\end{prop}
\begin{proof}
This is the content of \cite[Theorem 4.9]{CD09}.
\end{proof}

\begin{prop}
\label{A.8.21}
Let $\cA$ be a Grothendieck abelian category with a descent structure $(\cG,\cH)$ and let $\cA'$ be a Grothendieck abelian category.
Suppose $\cW$ is an essentially small set of $\cG$-cofibrant complexes of $\cA$, 
and let
\[
f^*:\cA\rightleftarrows \cA':f_*
\]
be a pair of adjoint additive functors. 
Assume that
\begin{enumerate}
\item[{\rm (i)}] 
$f_*$ is exact,
\item[{\rm (ii)}] 
for any complex $H$ in $\cH$, $f^*H$ is quasi-isomorphic to $0$.
\end{enumerate}
Then the induced pair of adjunctions
\[
f^*:\Co(\cA)\rightleftarrows \Co(\cA'):f_*
\]
is a Quillen pair with respect to the $\cW$-local descent model structure on $\Co(\cA)$ and the $f^*\cW$-local model structure on $\cO(\cA')$.
\end{prop}
\begin{proof}
This follows from Propositions \ref{A.8.17} and \ref{A.8.20}.
\end{proof}

\begin{df}
A descent structure $(\cG,\cH)$ on a Grothendieck abelian category is called {\it bounded} if every object of $\cH$ is a bounded complexes of objects that are finite sums of objects of $\cG$.
\end{df}
\begin{prop}\label{A.8.29}
Let $\cA$ be a Grothendieck abelian category with a bounded descent structure $(\cG,\cH)$. 
If $\hom_{\cA}(X,-)$ commutes with small sums for any $X\in \cG$, then $X[n]$ is a compact object in $\Deri(\cA)$  for any $n\in \Z$ and $X\in \cG$.
\end{prop}
\begin{proof}
 See \cite[Theorem 6.2]{CD09}.
\end{proof}

\begin{prop}
\label{A.8.31}
Let $\cA$ be a Grothendieck abelian category with a bounded descent structure $(\cG,\cH)$, and let $\cW$ be an essentially small set of bounded complexes of objects that are finite sums of objects of $\cG$. 
If $\hom_{\cA}(X,-)$ commutes with small sums for any $X\in \cG$, then $\cG$ generates $\Deri_\cW(\cA)$ and $X[n]$ is a compact object in $\Deri_\cW(\cA)$ for any $n\in \Z$ and $X\in \cG$.
\end{prop}
\begin{proof}
For any family $\cF$ of objects of $\Deri(\cA)$, let $\langle \cF \rangle$ denote the smallest localizing subcategory of $\Deri(\cA)$ containing $\cF$. 
Note that $\langle \cG\rangle = \Deri(\cA)$ by Proposition \ref{A.8.32}(2). 
Owing to Proposition \ref{A.8.19}(1), we have $\Deri_\cW(\cA)\cong \langle \cG\rangle / \langle \cW\rangle$. 
Proposition \ref{A.8.29} shows $X[n]$ is a compact object in $\Deri_\cW(\cA)$ for any $n\in \cZ$ and $X\in \cG$.
Thus any bounded complex of objects that are finite sums of objects of $\cG$ is compact in $\Deri(\cA)$. 
In particular, 
any object of $\cW$ is compact in $\Deri(\cA)$. 
To conclude we apply Thomason's localization theorem \cite[Theorem 4.4.9]{Neeman} --- 
in loc.~cit.,  
$\cT^{\aleph_0}$ denotes the class of compact objects in a triangulated category $\cT$.
\end{proof}

\newpage

\section{Categorical toolbox}\label{appendix:cat_tool}
In this appendix we collect some useful yoga of triangulated categories.
\subsection{Category of squares}
We fix a category $\cA$ and a symmetric monoidal triangulated category $\cT$.
\vspace{0.1in}

Non-canonicity of the cone construction is well-known in triangulated categories.
To impose canonicality for cones of morphisms in $\cA$, we first introduce the notion of $n$-squares, which helps us deal with iterated cones.

\begin{df}
For $n\geq 0$, let $[n]$ denote the $n$th cubical category, 
i.e., 
the (nonfull) subcategory of the category $\mathsf{Fin}$ of finite sets whose objects are $n+1$-uples $\ul{a} = (a_1, \ldots, a_n)\in \{0,1\}^n$. 
The set of morphisms ${\rm hom}_{[n]}(\ul{a}, \ul{b})$ for $\ul{a}\neq\ul{b}$ consists of a single element if $a_i\leq b_i$ for every index $i$,
and it is empty otherwise.
We note that all morphisms in $[n]$ are generated by morphisms of the form $a\rightarrow b$ for which there exists an index $1\leq i\leq n$ such that
\[
b_j=\left\{
\begin{array}{ll}
a_j+1&\text{if }j=i
\\
a_j&\text{if }j\neq i.
\end{array}
\right.
\]
We define the face functors 
\[s_{i,0}, s_{i,1}\colon [n-1]\to [n], \quad i=1, \ldots, n\]
by inserting $0$ (resp.\  $1$) in the $i$-th coordinate. 
\end{df}

\begin{df}
An $n$-square\index{square} in a category $\mathcal{A}$ is a functor $C\colon [n]\to \cA$, and a morphism of $n$-squares is a natural transformation of functors $C_1\Rightarrow C_2$ between $n$-squares. 
Note that a $0$-square is simply an object of $\cA$, 
a $1$-square is an arrow $C_0\to C_1$ where $C_0$ and $C_1$ are objects of $\cA$, and so on. 
We write $C_{(a_1,\ldots,a_n)} \in {\rm Ob}(\cA)$ for $C((a_1,\ldots,a_n))$. 
We denote by
\[
\mathbf{Square}_\cA
\]
the coproduct of the categories of $n$-squares for all integer $n\geq 0$.
\end{df}

\begin{df}
For an $n$-square $C$ with $n\geq 1$, let $s_{i,0}^*C$ (resp.\ $s_{i,1}^*$) denote the $(n-1)$-square that is the restriction of $C$ to all the coordinates of the form
\begin{gather*}
(a_1,\ldots,a_{i-1},0,a_{i+1},\ldots,a_n)\\
\text{(resp.\ }(a_1,\ldots,a_{i-1},1,a_{i+1},\ldots,a_n)).
\end{gather*}
or, equivalently, the composition $[n-1]\xrightarrow{s_{i, \epsilon}}[n]\xrightarrow{C}\cA$ for $\epsilon=0,1$.
Note that if $C$ is an $(n+1)$-square, then there is a natural morphism of $n$-squares
\[
s_{i,0}^*C\rightarrow s_{i,1}^*C.
\]
\end{df}

\begin{df}
If $\cA$ is monoidal, for an $n$-square $C$ and $m$-square $D$ we denote by
\[
C\otimes D
\]
the $m+n$ square such that the object at the coordinate $(a_1,\ldots,a_n,b_1,\ldots,b_m)$ is given by the tensor product of the objects at the coordinates 
$(a_1,\ldots,a_n)$ in $C$ and $(b_1,\ldots,b_m)$ in $D$.
The morphisms in $C\otimes D$ are naturally induced by the morphisms in $C$ and $D$.
The above defines a monoidal structure on $\mathbf{Square}_\cA$.
\end{df}

We shall often assume that there exists a monoidal functor
\[
M\colon \mathbf{Square}_{\cA}\rightarrow \cT
\]
satisfying the following conditions.

\begin{enumerate}
\item[(Sq)] \textit{Square condition.}
%
For every integer $n\geq 1$ and $n$-square $C$, there exist functorial morphisms
\[
d_{i,C}\colon M(s_{i,1}^*C)\rightarrow M(C),
\;
\partial_{i,C}\colon M(C)\rightarrow M(s_{i,0}^*C)[1]
\]
such that
\[
M(s_{i,0}^*C)\stackrel{M(\theta_{C,i})}\longrightarrow M(s_{i,1}^*C)\stackrel{d_{i,C}}\longrightarrow M(C)\stackrel{\partial_{i,C}}\longrightarrow M(s_{i,0}^*C)[1]
\]
is a distinguished triangle, where $\theta_{C,i}\colon s_{i,0}^*C\rightarrow s_{i,1}^*C$ is the canonical morphism induced by $C$.
\vspace{0.1in}

For every integer $n\geq 2$ with $1\leq i<j\leq n$ and $n$-square $C$, we require that the diagram
\[
\begin{tikzcd}[column sep=large]
M(s_{i,1}^*s_{j,1}^*C)\arrow[d,"d_{i,s_{j,1}^*C}"']\arrow[r,"d_{j,s_{i,1}^*C}"]&
M(s_{i,1}^*C)\arrow[d,"d_{i,C}"]\arrow[r,"\partial_{j,s_{i,1}^*C}"]&
M({s_{i,1}^*s_{j,0}^*C[1]})\arrow[d,"{d_{i,s_{j,0}^*C}[1]}"]
\\
M(s_{j,1}^*C)\arrow[r,"d_{j,C}"]\arrow[d,"\partial_{i,s_{j,1}^*C}"']&
M(C)\arrow[d,"\partial_{i,C}"]\arrow[r,"\partial_{j,C}"]&
M({s_{j,0}^*C[1]})\arrow[d,"-\partial_{i,s_{j,0}^*C}{[1]}"]
\\
M({s_{i,0}^*s_{j,1}^*C[1]})\arrow[r,"d_{j,s_{i,1}^*C}{[1]}"]&
M({s_{i,0}^*C[1]})\arrow[r,"-\partial_{j,s_{i,1}^*C}{[1]}"]&
M(s_{i,0}^*s_{j,0}^*C[2])
\end{tikzcd}
\]
commutes, where we use $s_{i,1}^*s_{j,1}^*=s_{j,1}^*s_{i,1}^*$ etc.

\item[(MSq)] \textit{Monoidal square condition.} \index[notation]{MSq} \index{monoidal square condition}For every $n$-square $C$ and $m$-square $D$, 
let
\[
T_{C,D}:M(C)\otimes M(D)\rightarrow M(C\otimes D)
\]
be the canonical morphism obtained by the monoidality of $M$.
If $n\geq 1$, then for every $1\leq i\leq n$ we require that the diagram
\[
\begin{tikzpicture}[baseline= (a).base]
\node[scale=0.738] (a) at (0,0){
\begin{tikzcd}[column sep=huge]
M(s_{i,0}^*C)\otimes M(D)\arrow[d,"T_{s_{i,0}^*C,D}"']\arrow[r,"M(\theta_{i,C})\otimes M(D)"]&
M(s_{i,1}^*C)\otimes M(D)\arrow[r,"d_{i,C}\otimes M(D)"]\arrow[d,"T_{s_{i,1}^*C,D}"]&
M(C)\otimes M(D)\arrow[r,"\partial_{i,C}\otimes M(D)"]\arrow[d,"T_{C,D}"]&
M(s_{i,0}^*C)\otimes M(D)[1]\arrow[d,"T_{s_{i,0}^*C,D}{[1]}"]
\\
M(s_{i,0}^*C\otimes D)\arrow[r,"M(\theta_{i,C\otimes D})"]&
M(s_{i,1}^*C\otimes D)\arrow[r,"d_{i,C\otimes D}"]&
M(C\otimes D)\arrow[r,"\partial_{i,C\otimes D}"]&
M(s_{i,0}^*C\otimes D)[1]
\end{tikzcd}};
\end{tikzpicture}
\]
commutes.
Similarly, if $m\geq 1$, then for every $1\leq j\leq m$, we require that the diagram
\[
\begin{tikzpicture}[baseline= (a).base]
\node[scale=0.738] (a) at (0,0){
\begin{tikzcd}[column sep=huge]
M(C)\otimes M(s_{j,0}^*D)\arrow[d,"T_{C,s_{j,0}^*D}"']\arrow[r,"M(C)\otimes M(\theta_{j,D})"]&
M(C)\otimes M(s_{j,1}^*D)\arrow[r,"M(C)\otimes d_{j,D}"]\arrow[d,"T_{C,s_{j,1}^*D}"]&
M(C)\otimes M(D)\arrow[r,"M(C)\otimes \partial_{j,D}"]\arrow[d,"T_{C,D}"]&
M(C)\otimes M(s_{j,0}^*D)[1]\arrow[d,"T_{C,s_{j,0}^*D}{[1]}"]
\\
M(C\otimes s_{j,0}^*D)\arrow[r,"M(\theta_{n+j,C\otimes D})"]&
M(C\otimes s_{j,1}^*D)\arrow[r,"d_{n+j,C\otimes D}"]&
M(C\otimes D)\arrow[r,"\partial_{n+j,C\otimes D}"]&
M(C\otimes s_{j,0}^*D)[1]
\end{tikzcd}};
\end{tikzpicture}
\]
commutes.
\end{enumerate}

The conditions (Sq) and (MSq) are often automatically satisfied if $\cT$ arises from a simplicial, cofibrantly generated, and locally presentable monoidal model category (such as the category of sheaves of complexes that we consider in this text).
\vspace{0.1in}

Let us give some basic consequences of (Sq).

\begin{prop}
\label{A.3.45}
Let $f:Y\rightarrow X$ be a morphism in $\mathscr{S}/S$.
Then the morphism
\[
M(f):M(Y)\rightarrow M(X)
\]
is an isomorphism if and only if $M(Y\rightarrow X)=0$.
\end{prop}
\begin{proof}
By (Sq), there is a canonical distinguished triangle
\[
M(Y)\rightarrow M(X)\rightarrow M(Y\rightarrow X)\rightarrow M(Y)[1], 
\]
which implies the claim.
\end{proof}

\begin{prop}
\label{A.3.48}
Let $C\rightarrow D$ be a morphism of $n$-squares.
If the naturally induced morphism
\[
M(C_{(a_1,\ldots,a_n)})\rightarrow M(D_{(a_1,\ldots,a_n)})
\]
is an isomorphism for every coordinate $(a_1,\ldots,a_n)\in \{0,1\}^n$, then the naturally induced morphism
\[
M(C)\rightarrow M(D)
\]
is an isomorphism.
\end{prop}
\begin{proof}
We use induction on $n$.
The case $n=1$ is clear, so suppose that $n>1$.
By (Sq), there is a morphism of distinguished triangles
\[
\begin{tikzcd}
M(s_{1,0}^*C)\arrow[d]\arrow[r]&
M(s_{1,1}^*C)\arrow[d]\arrow[r]&
M(C)\arrow[d]\arrow[r]&
M(s_{1,0}^*C)[1]\arrow[d]
\\
M(s_{1,0}^*D)\arrow[r]&
M(s_{1,1}^*D)\arrow[r]&
M(D)\arrow[r]&
M(s_{1,0}^*D)[1].
\end{tikzcd}
\]
By induction, all the vertical morphisms except for $M(C)\rightarrow M(D)$ are isomorphisms. 
We conclude using the five lemma.
\end{proof}

\begin{prop}
\label{A.3.44}
For any commutative square in $\mathscr{S}/S$
\[
C=\begin{tikzcd}
Y'\arrow[d]\arrow[r]&Y\arrow[d]\\
X'\arrow[r]&X
\end{tikzcd}
\]
such that 
\[
M(Y'\rightarrow X')\rightarrow M(Y\rightarrow X)
\]
is an isomorphism, also
\[
M(Y'\rightarrow Y)\rightarrow M(X'\rightarrow X)
\] 
is an isomorphism.
\end{prop}
\begin{proof}
The condition (Sq) implies that there are canonical distinguished triangles
\begin{gather}
\label{A.3.44.1}
M(Y'\rightarrow X')\rightarrow M(Y\rightarrow X)\rightarrow M(C)\rightarrow M(Y'\rightarrow X')[1],
\\
\label{A.3.44.2}
M(Y'\rightarrow Y)\rightarrow M(X'\rightarrow X)\rightarrow M(C)\rightarrow M(Y'\rightarrow Y)[1].
\end{gather}
By assumption, \eqref{A.3.44.1} gives $M(C)=0$.
Now use \eqref{A.3.44.2} to conclude.
\end{proof}

\subsection{Calculus of fractions}
We record a result on the calculus of fractions which is used repeatedly in this text.
In applications, we are usually interested in the case when $S$ is a base scheme, $\cC=lSm/S$ or $SmlSm/S$, $\cD=\cT$, and $F=M$.

\begin{lem}
\label{A.1.16}
Let $\cB\subset \cA$ be classes of morphisms in a category $\cC$ such that both classes admit a calculus of right fractions in $\cC$,
and let $F:\cC\rightarrow \cD$ be a functor. 
Assume that for any morphism $f:Y\rightarrow X$ in $\cA$, there is a commutative diagram
\[
\begin{tikzcd}
&Y'\arrow[d,"g"]\arrow[ld,"h"']\\
Y\arrow[r,"f"']&X
\end{tikzcd}
\]
in $\cC$, where $g\in \cB$. 
If
\[
F(g):F(Y')\rightarrow F(X)
\]
is an isomorphism for every $g$ in $\cB$, then 
\[
F(f):F(Y)\rightarrow F(X)
\] 
is an isomorphism for every $f$ in $\cA$. 
\end{lem}
\begin{proof}
The right Ore condition (dual to \cite[I.2.2.c]{GZFractions}) for $\cB$ implies there is a commutative diagram in $\cC$
\[
\begin{tikzcd}
Y''\arrow[d,"g'"']\arrow[r,"f'"]&Y'\arrow[d,"g"]\\
Y\arrow[r,"f"']&X,
\end{tikzcd}\]
where $g'\in \cB$. 
Since $fhf'=gf'=fg'$ and $f\in \cA$, 
there is a morphism $r:Z\rightarrow Y''$ in $\mathcal{A}$ such that $hf'r=g'r$ by the right cancellation condition for $\cA$. 
By assumption, there is a morphism $s:Z'\rightarrow Y''$ in $\cB$ that factors through $r$. 
Thus $hf's=g's$, and there is a commutative diagram in $\cC$
\[
\begin{tikzcd}
Z'\arrow[d,"g's"']\arrow[r,"f's"]&Y'\arrow[ld,"h"']\arrow[d,"g"]\\
Y\arrow[r,"f"']&X.
\end{tikzcd}\]
Now by assumption, $F(g)$ and $F(g's)$ are isomorphisms since $g$ and $g's$ are in $\cB$. 
Since $F(g)F(f's)=F(f)F(g's)$, we have that
\[
F(f's)F(g's)^{-1}=F(g)^{-1}F(f), 
\]
and therefore
\[
{\rm id}=F(g's)F(g's)^{-1}=F(h)F(f's)F(g's)^{-1}=F(h)F(g)^{-1}F(f),
\]
\[
{\rm id}=F(g)F(g)^{-1}=F(f)F(h)F(g)^{-1}.
\]
This shows that $F(f)$ is an isomorphism.
\end{proof}

\newpage
\bibliography{biblogmotives}
\bibliographystyle{siam}
\newpage
\printindex
\printindex[notation]
\newpage
\end{document}